\documentclass[a4paper,fleqn,leqno]{article}
\usepackage{latexsym,amssymb,amsthm,amsmath,epsfig,color,epsf}
\usepackage[notcite,notref]{}
\usepackage[english]{babel}
\usepackage[utf8]{inputenc}
\usepackage[normalem]{ulem}
\usepackage{hyperref}
\usepackage[utf8]{inputenc}
\usepackage[normalem]{ulem} 
\usepackage{graphicx,tikz,pgfplots}
\usepackage[mode=buildnew]{standalone}
\usepackage[symbols,nogroupskip,sort=use]{glossaries-extra}

\makeatletter
\@addtoreset{equation}{section}
\makeatother

\newcounter{extralabel}[section]

\newtheorem{ittheorem}{Theorem}
\newtheorem{itlemma}{Lemma}
\newtheorem{itproposition}{Proposition}
\newtheorem{itdefinition}{Definition}
\newtheorem{itcorollary}{Corollary}
\newtheorem{itconjecture}{Conjecture}
\newtheorem{itremark}{Remark}

\newenvironment{theorem}{\addtocounter{extralabel}{1}
	\begin{ittheorem}}{\end{ittheorem}}

\newenvironment{lemma}{\addtocounter{extralabel}{1}
	\begin{itlemma}}{\end{itlemma}}

\newenvironment{proposition}{\addtocounter{extralabel}{1}
	\begin{itproposition}}{\end{itproposition}}

\newenvironment{definition}{\addtocounter{extralabel}{1}
	\begin{itdefinition}}{\end{itdefinition}}

\newenvironment{corollary}{\addtocounter{extralabel}{1}
	\begin{itcorollary}}{\end{itcorollary}}

\newenvironment{remark}{\addtocounter{extralabel}{1}
	\begin{itremark}}{\end{itremark}}

\newcommand{\be}{\begin{equation}}
	
	\newcommand{\bea}[1]{\Rand{\vspace{0,7cm}\tt #1\vspace{-0,7cm}}\begin{eqnarray}\label{#1}}

		\renewcommand{\d}{{\rm d}}
		\newcommand{\e}{{\rm e} }

		\newcommand{\sign}{{\rm sign}}
		\newcommand{\R}{\mathbb{R}}
		\newcommand{\N}{\mathbb{N}}
		\newcommand{\Z}{\mathbb{Z}}

		\newcommand{\E}{\mathbb{E}}
		\renewcommand{\P}{\mathbb{P}}
		\def\1{{\mathchoice {1\mskip-4mu\mathrm l} 
				{1\mskip-4mu\mathrm l}
				{1\mskip-4.5mu\mathrm l} {1\mskip-5mu\mathrm l}}}

		\newcommand{\bt}{\bar{t}}

		\def\CB{\mathcal{B}}
		\def\CC{\mathcal{C}}
		\def\CD{\mathcal{D}}
		
		\def\CF{\mathcal{F}}
		\def\CG{\mathcal{G}}
		
		\def\CM{\mathcal{M}}
		
		\def\CL{\mathcal{L}}

		\def\CO{\mathcal{O}}
		\def\CP{\mathcal{P}}

		\def\CX{\mathcal{X}}

		\def\E{\mathbb{E}}
		
		\def\N{\mathbb{N}}
		\def\P{\mathbb{P}}
		\def\Q{\mathbb{Q}}
		\def\R{\mathbb{R}}
		\def\S{\mathbb{S}}

		\def\Z{\mathbb{Z}}

		\DeclareMathSymbol{\varNu}{\mathord}{letters}{78}
		
		\newcommand{\ee}{\end{equation}}
	\newcommand{\eea}{\end{eqnarray}}
\newcommand{\bean}{\begin{eqnarray*}}
	\newcommand{\eean}{\end{eqnarray*}}
\newcommand{\noi}{\noindent}

\newcommand{\suml}{\sum\limits}

\newcommand{\vt}{\vartheta}

\newcommand{\eff}{\mathrm{eff}}
\newcommand{\aux}{\mathrm{aux}}

\usepackage{color}
\usepackage{colortbl}

\newtheorem{xx}{\bf xxx}

\newtheorem{zz}{\bf zzz}

\newtheorem{yy}{\bf yyy}

\makenoidxglossaries
\glsxtrnewsymbol[description={state space}]{S}{\ensuremath{S}}
\glsxtrnewsymbol[description={space of continuous functions from $[0,\infty)$ to $S$}]{cf}{\ensuremath{C([0,\infty),S)}}
\glsxtrnewsymbol[description={space of bounded continuous functions from $[0,\infty)$ to $S$}]{cfb}{\ensuremath{C_b([0,\infty),S)}}
\glsxtrnewsymbol[description={generator}]{G}{\ensuremath{G}}
\glsxtrnewsymbol[description={hierarchical group of order $N$}]{hiern}{\ensuremath{\Omega_N}}
\glsxtrnewsymbol[description={hierarchical distance on $\Omega_N$}]{dhier}{\ensuremath{d_{\Omega_{N}}}}
\glsxtrnewsymbol[description={sequence of migration rates $\underline{c}=(c_k)_{k\in\N_0}$}]{mig}{\ensuremath{\underline{c}}}
\glsxtrnewsymbol[description={migration rate within distance $k$}]{migk}{\ensuremath{c_k}}
\glsxtrnewsymbol[description={migration kernel on $\Omega_N$}]{migker}{\ensuremath{a^{\Omega_{N}}(\cdot,\cdot)}}
\glsxtrnewsymbol[description={ time-$t$ migration kernel on $\Omega_N$}]{migkert}{\ensuremath{a_t^{\Omega_{N}}(\cdot,\cdot)}}
\glsxtrnewsymbol[description={degree of  random walk}]{degrw}{\ensuremath{\delta}}
\glsxtrnewsymbol[description={sequence of relative seed-bank sizes $\underline{K}=(K_m)_{m\in\N_0}$}]{ssbs}{\ensuremath{\underline{K}}}
\glsxtrnewsymbol[description={sequence of exchange parameters $\underline{e}=(e_m)_{m\in\N_0}$}]{sexp}{\ensuremath{\underline{e}}}
\glsxtrnewsymbol[description={exchange parameter for colour $m$ }]{expar}{\ensuremath{e_m}}
\glsxtrnewsymbol[description={relative size of $m$-dormant population }]{sbsm}{\ensuremath{K_m}}
\glsxtrnewsymbol[description={total exchange rate with dormant population }]{chi}{\ensuremath{\chi}}
\glsxtrnewsymbol[description={relative size of dormant population }]{rho}{\ensuremath{\rho}}
\glsxtrnewsymbol[description={diffusion function }]{g}{\ensuremath{g}}
\glsxtrnewsymbol[description={class of diffusion functions }]{cldf}{\ensuremath{\CG}}
\glsxtrnewsymbol[description={standard Brownian motion at site $\xi\in\Omega_{N}$ }]{bm}{\ensuremath{(w_\xi(t))_{t \geq0}}}
\glsxtrnewsymbol[description={fraction of active individuals of type $\heartsuit$ at colony $\xi$ at time $t$ }]{x}{\ensuremath{x_{\xi}(t)}}
\glsxtrnewsymbol[description={fraction of $m$-dormant individuals of type $\heartsuit$ at colony $\xi$ at time $t$ }]{y}{\ensuremath{y_{\xi,m}(t)}}
\glsxtrnewsymbol[description={full process $\big(x_\xi(t),(y_{\xi,m}(t))_{m\in\N_0}\big)_{\xi\in\Omega_N,\, t\geq 0}$}]{fp}{\ensuremath{(X^{\Omega_N}(t),Y^{\Omega_N}(t))_{t\geq 0}}}
\glsxtrnewsymbol[description={initial density parameter active population }]{thetax}{\ensuremath{\theta_x}}
\glsxtrnewsymbol[description={initial density parameter  $m$-dormant population }]{thetay}{\ensuremath{\theta_{y_m}}}
\glsxtrnewsymbol[description={active $k$-block average at time $t$ }]{xk}{\ensuremath{x_k^{\Omega_N}(t)}}
\glsxtrnewsymbol[description={$m$-dormant $k$-block average at time $t$ }]{yk}{\ensuremath{y_{m,k}^{\Omega_N}(t)}}
\glsxtrnewsymbol[description={domain of generator }]{DF}{\ensuremath{\mathbf{F}}}
\glsxtrnewsymbol[description={Fisher-Wright diffusion function }]{FW}{\ensuremath{g_{\mathrm{FW}}}}
\glsxtrnewsymbol[description={seed-bank coalescent }]{SC}{\ensuremath{(\CC^{(n)}(t))_{t \geqslant 0}}}
\glsxtrnewsymbol[description={transition kernel corresponding to lineage in dual }]{tkerd}{\ensuremath{b(\cdot,\cdot)}}
\glsxtrnewsymbol[description={duality function }]{H}{\ensuremath{H}}
\glsxtrnewsymbol[description={block counting process }]{L}{\ensuremath{(L(t))_{t\geq 0}}}
\glsxtrnewsymbol[description={full process $(X^{\Omega_N}(t),Y^{\Omega_N}(t))_{t\geq 0}$ }]{fpz}{\ensuremath{(Z^{\Omega_N}(t))_{t\geq 0}}}
\pgfplotsset{compat=1.16}
\glsxtrnewsymbol[description={time lapse lineage in dual spends in active state after becoming active for $k$-th time }]{sigmak}{\ensuremath{\sigma_k}}
\glsxtrnewsymbol[description={time lapse lineage in dual spends in active state after becoming active for $k$-th time }]{tauk}{\ensuremath{\tau_k}}
\glsxtrnewsymbol[description={typical wake-up time of lineage in dual }]{tau}{\ensuremath{\tau}}
\glsxtrnewsymbol[description={function that is slowly varying at infinity }]{phi}{\ensuremath{\varphi}}
\glsxtrnewsymbol[description={tail exponent of typical wake-up time }]{gamma}{\ensuremath{\gamma}}
\glsxtrnewsymbol[description={fraction of individuals of type $\heartsuit$ in state $R_\xi$ in colony $\xi$ at time $t$ }]{z}{\ensuremath{z_{(\xi,R_\xi)}(t)}}
\glsxtrnewsymbol[description={space of c\`adl\`ag functions from $[0,\infty)$ to Polish space $E$}]{cadlag}{\ensuremath{D([0,\infty),E)}}
\glsxtrnewsymbol[description={pseudopath corresponding to measurable map $ v: s\to v(s)$ from $ [0,\infty)$ to $E$}]{pseudo}{\ensuremath{\psi_v}}
\glsxtrnewsymbol[description={space of all pseudopaths}]{pseudospace}{\ensuremath{\Psi}}
\glsxtrnewsymbol[description={initial density of $\heartsuit$ in active population and first $k+1$ dormant populations}]{varthetak}{\ensuremath{\vartheta_k}}
\glsxtrnewsymbol[description={slowing-down constants}]{Ek}{\ensuremath{E_k}}
\glsxtrnewsymbol[description={full process on level $l$ with drift towards $\theta$ }]{zl}{\ensuremath{\Big(z_{l,(\theta,(y_{m,l})_{m\in\N_0})}(t)\Big)_{t\geq 0}}}
\glsxtrnewsymbol[description={effective process on level $l$ with drift towards $\theta$}]{zlf}{\ensuremath{(z^{\eff}_{l,\theta}(t))_{t\geq 0}}}
\glsxtrnewsymbol[description={empirical average in active $l$-block at time $t$ }]{thetalx}{\ensuremath{\bar{\Theta}_x^{(l),\Omega_N}(t)}}
\glsxtrnewsymbol[description={empirical average in $m$-dormant $l$-block at time $t$ }]{thetaly}{\ensuremath{\bar{\Theta}_{y_m}^{(l),\Omega_N}(t)}}
\glsxtrnewsymbol[description={empirical average in $l$-block at time $t$ }]{thetal}{\ensuremath{\bar{\Theta}^{(l),\Omega_N}(t)}}
\glsxtrnewsymbol[description={full estimator of $l$-block at time $t$ }]{thetalfull}{\ensuremath{\boldsymbol{\Theta}^{(l),\Omega_N}(t)}}
\glsxtrnewsymbol[description={ effective estimator of the $l$-block at time $t$ }]{thetaleff}{\ensuremath{\boldsymbol{\Theta}^{\eff,(l),\Omega_N}(t)}}
\glsxtrnewsymbol[description={ sequence with $\lim_{N\to\infty}L(N)=\infty$ and $\lim_{N \to \infty} L(N)/N=0$ }]{LN}{\ensuremath{(L(N))_{N\in\N}}}
\glsxtrnewsymbol[description={invariant measure for full $l$-block average process with density $\theta$ }]{Gammal}{\ensuremath{\Gamma_{(\theta,y_l)}^{(l)}}}
\glsxtrnewsymbol[description={invariant measure for effective $l$-block average process with density $\theta$ }]{Gammaleff}{\ensuremath{\Gamma_{\theta}^{\eff,(l)}}}
\glsxtrnewsymbol[description={renormalisation transformation on level $l$ }]{Fl}{\ensuremath{\CF^{E_l,c_l,K_l,e_l}}}
\glsxtrnewsymbol[description={$k$-th iterate of renormalisation transformation }]{Flit}{\ensuremath{\CF^{(k)}}}
\glsxtrnewsymbol[description={interaction chain }]{intchain}{\ensuremath{	(M^k_{-l})_{-l = -(k+1),-k,\ldots,0}}}
\glsxtrnewsymbol[description={transition kernel of interaction chain on level $l$}]{Q}{\ensuremath{Q^{[l]}}}
\glsxtrnewsymbol[description={effective interaction chain }]{intchaineff}{\ensuremath{	(M^{\eff,k}_{-l})_{-l = -(k+1),-k,\ldots,0}}}
\glsxtrnewsymbol[description={effective transition kernel of interaction chain on level $l$}]{Qeff}{\ensuremath{Q^{\eff,[l]}}}
\glsxtrnewsymbol[description={first component of effective $k$-block process}]{effx}{\ensuremath{\bar{x}_{k}^{\Omega_N}(t)}}
\glsxtrnewsymbol[description={$n$-th composition of transition kernel of interaction chain}]{Qn}{\ensuremath{Q^{(n)}}}
\glsxtrnewsymbol[description={initial state of level-$n$ full process}]{thetan}{\ensuremath{\bar{\vt}^{(n)}}}
\glsxtrnewsymbol[description={$n$-th clustering coefficient}]{An}{\ensuremath{A_n}}
\glsxtrnewsymbol[description={$n,m$-th clustering coefficient}]{Amn}{\ensuremath{A_m^n}}
\glsxtrnewsymbol[description={finite geographic space $\{0,1,\cdots,N-1\}$}]{spaceN}{\ensuremath{[N]}}
\glsxtrnewsymbol[description={difference between active and dormant estimator at time $t$}]{difad}{\ensuremath{{\Delta}^{[N]}_{\bar{\Theta}}(t)}}
\glsxtrnewsymbol[description={initial measure for infinite system obtained by periodic continuation of finite system}]{mun}{\ensuremath{\mu_N}}
\glsxtrnewsymbol[description={full infinite process $\big(x_\xi(t),(y_{\xi,m}(t))_{m\in\N_0}\big)_{\xi\in\Omega_N,\, t\geq 0}$ starting from initial law $\mu_N$ }]{ifp}{\ensuremath{(X^{\mu_N}(t),Y^{\mu_N}(t))_{t\geq 0}}}
\glsxtrnewsymbol[description={local time at $0$ at time $t$}]{localtime}{\ensuremath{L_t^0}}
\glsxtrnewsymbol[description={specific sequence of times depending on $N$}]{barln}{\ensuremath{\bar{L}(N)}}
\glsxtrnewsymbol[description={$\CD$-semi martingale operator}]{dsmop}{\ensuremath{G_\dagger^{[N]}}}
\glsxtrnewsymbol[description={auxiliary $l$-block estimator process at time $t$ on level $k+1$}]{aux}{\ensuremath{\boldsymbol{\Theta}^{\aux,(l),\Omega_N^{k+1}}(t)}}
\glsxtrnewsymbol[description={hierarchical group truncated at level $k+1$}]{omegatrunc}{\ensuremath{\Omega_N^{k+1}}}
\glsxtrnewsymbol[description={equilibrium measure of auxiliary process with density $\theta$ at level $1$}]{gammaaux}{\ensuremath{\Gamma_\theta^{\aux,(1)}}}
\begin{document}
	
	
	\title{Spatial populations with seed-bank:\\
		renormalisation on the hierarchical group}
	
	\author{Andreas Greven$^1$, Frank den Hollander$^2$, Margriet Oomen$^3$}
	
	\date{6 October 2021}
	
	\maketitle
	
	\begin{abstract}
		We consider a system of interacting diffusions labeled by a geographic space that is given by the {\em hierarchical group} $\Omega_N$ of order $N\in\N$. Individuals live in colonies and are subject to resampling and migration as long as they are {\em active}. Each colony has a seed-bank into which individuals can retreat to become {\em dormant}, suspending their resampling and migration until they become active again. The migration kernel has a {\em hierarchical structure}: individuals hop between colonies at a rate that depends on the hierarchical distance between the colonies. The seed-bank has a {\em layered structure}: when individuals become dormant they acquire a colour that determines the rate at which they become active again. The latter allows us to model seed-banks whose wake-up times have a fat tail. We analyse a system of coupled stochastic differential equations that describes the population in the large-colony-size limit. 
		
		For fixed $N\in\N$, the system exhibits a dichotomy between \emph{coexistence} (= locally multi-type equilibria) and \emph{clustering} (= locally mono-type equilibria). To identify the \emph{clustering regime}, i.e., the range of parameters controlling the migration and the seed-bank for which clustering prevails, we apply a necessary and sufficient criterion derived in \cite{GdHOpr1} that is valid for any geographic space given by a countable Abelian group endowed with the discrete topology. The diffusion function controlling the resampling does not play a role in this criterion. 
		
		We carry out a \emph{multi-scale renormalisation analysis} in the \emph{hierarchical mean-field limit} $N\to\infty$. In particular, we show that block averages on hierarchical space-time scale $k \in \N$ perform a diffusion with a renormalised diffusion function that depends on $k$. In the clustering regime, after an appropriate scaling with $k$, this diffusion function converges to the Fisher-Wright diffusion function as $k\to\infty$, irrespective of the diffusion function controlling the resampling. Thus, the system exhibits \emph{full universality} on large space-time scales in terms of the scaling limit. For several subclasses of parameters we identify the speed at which the scaled renormalised diffusion function converges to the Fisher-Wright diffusion as $k\to\infty$. This speed in turn determines the speed at which \emph{mono-type clusters grow in space and time}. We show that the seed-bank reduces the speed compared to the model without seed-bank. This reduction is the result of a delicate interplay between migration and seed-bank, which in the limit as $N\to\infty$ can be worked out explicitly.  
		
		
		
		\medskip\noindent 
		\emph{Acknowledgements:} 
		AG was supported by the Deutsche Forschungsgemeinschaft (grant DFG-GR 876/6-1,2). FdH and MO were supported by the Netherlands Organisation for Scientific Research (NWO Gravitation Grant NETWORKS-024.002.003). FdH visited Bonn and Erlangen in the Fall of 2019 and 2020, supported by a Research Award of the Alexander von Humboldt Foundation. AG visited Leiden in February 2020, as Kloosterman Chair. 
	\end{abstract}
	
	\bigskip
	
	\footnoterule
	\noi
	\hspace*{0.3cm} {\footnotesize $^{1)}$ 
		Department Mathematik, Universit\"at Erlangen-N\"urnberg, Cauerstr.\ 11,
		D-91058 Erlangen, Germany\\
		greven@mi.uni-erlangen.de}\\
	\hspace*{0.3cm} {\footnotesize $^{2)}$ 
		Mathematisch Instituut, Universiteit Leiden, Niels Bohrweg 1, 2333 CA  Leiden, NL\\
		denholla@math.leidenuniv.nl}\\
	\hspace*{0.3cm} {\footnotesize $^{3)}$ 
		Mathematisch Instituut, Universiteit Leiden, Niels Bohrweg 1, 2333 CA  Leiden, NL\\
		m.oomen@math.leidenuniv.nl}
	
	
	\newpage
	
	\small
	\tableofcontents
	\normalsize
	\newpage
	\printnoidxglossary[type=symbols,style=long,title={List of Symbols}]
	\newpage
	\section{Background, goals and outline}
	\label{s.background}
	
	\subsection{Background}
	\label{ss.back}
	
	\paragraph{Single colony with seed-bank.}
	
	In populations with a \emph{seed-bank}, individuals can temporarily become dormant and refrain from reproduction, until they can become active again. In \cite{BCEK15} and \cite{BCKW16} the evolution of a population evolving according to the two-type Fisher-Wright model with seed-bank was studied. Individuals move in and out of the seed-bank at prescribed rates. Outside the seed-bank individuals are subject to \emph{resampling}, while inside the seed-bank their resampling is \emph{suspended}. Both the long-time behaviour and the genealogy of the population were analysed in detail. 
	
	Seed-banks are observed in many taxa, including plants, bacteria and other micro-organisms. Typically, they arise as a response to unfavourable environmental conditions. The dormant state of an individual is characterised by low metabolic activity and interruption of phenotypic development (see e.g.~\cite{LJ11}). After a varying and possibly large number of generations, a dormant individual can be resuscitated under more favourable conditions and reprise reproduction after having become active again. This strategy is known to have important implications for population persistence, maintenance of genetic variability and stability of ecosystems. It acts as a \emph{buffer} against evolutionary forces such as genetic drift, selection and environmental variability. 
	
	\paragraph{Multiple colonies with seed-bank.}
	
	In \cite{GdHOpr1} we considered a \emph{spatial} version of the two-type Fisher-Wright model with seed-bank in which individuals can \emph{migrate} between colonies, organised into a \emph{geographic space}, each having a seed-bank consisting of \emph{multiple layers}, each with their own rate of moving in (becoming dormant) and moving out (waking up). We found that the presence of the seed-bank \emph{enhances genetic diversity} compared to the spatial model without seed-bank. Interestingly, we found that the seed-bank can affect the longtime behaviour of the system both qualitatively and quantitatively.
	
	In \cite{GdHOpr1} we settled existence and uniqueness of the spatial model when the geographic space is $\Z^d$, $d\in\N$. We proved convergence to equilibrium, showed that there is a dichotomy between \emph{coexistence} (= locally multi-type equilibria) and \emph{clustering} (= locally mono-type equilibria), and identified the parameter regime for both. We found a change of the dichotomy due to the presence of the seed-bank. Without seed-bank, for migration in the domain of attraction of Brownian motion, clustering occurs in $d=1,2$ and coexistence in $d \geq 3$, i.e., the \emph{critical dimension} for the dichotomy is $d=2$. With seed-bank, however, clustering becomes more difficult and occurs in $d=2$ only when the wake-up time of a typical individual in the seed-bank has finite mean, and in $d=1$ only when the wake-up time has a sufficiently thin tail. In other words, the seed-bank has a tendency to \emph{lower} the critical dimension. 
	
	In fact, in \cite{GdHOpr1} we found that our technique of proof works for geographic spaces that are \emph{arbitrary} countable Abelian groups endowed with the discrete topology. The reason is that the dichotomy can be formulated in terms of how the \emph{degree} of the random walk that underlies the migration balances with the exponent of the \emph{tail} of the typical wake-up time. This raises the question how we can better understand the behaviour of spatial models with seed-bank close to criticality. 
	
	In \cite{GdHOpr2} we established the so-called \emph{finite-systems scheme}, i.e., we identified how a finite truncation of the system (both in the geographic space and in the seed-bank) behaves as both the time and the truncation level tend to infinity, properly tuned together. We found that if the wake-up time has finite mean, then the scaling time is proportional to the volume of the system and there is a \emph{single universality class} for the scaling limit, namely, the system moves through a succession of equilibria of the infinite system with a density that evolves according to a Fisher-Wright diffusion. On the other hand, we found that if the wake-up time has infinite mean, then the scaling time grows faster than the volume of the system, and there are \emph{two universality classes} depending on how fast the truncation level of the seed-bank grows compared to the truncation level of the geographic space.
	
	\subsection{Goals}
	
	In the present paper we take as geographic space the \emph{hierarchical group $\Omega_N$} of order $N$. The reason for this choice is that $\Omega_N$ allows for more detailed computations. At the same time, migration on $\Omega_N$ can be used to \emph{approximate} migration on $\Z^d$ in the \emph{hierarchical mean-field limit} $N\to\infty$. In particular, by playing with the migration kernel we can approximate two-dimensional migration in the sense of potential theory.  We consider migration kernels that in the limit as $N\to\infty$ are \emph{critically recurrent}, i.e., the degree of the class of hierarchical migrations that we consider in the present paper converges to $0$, 
	either from above or from below. 
	
	The present paper has three goals:
	\begin{itemize}
		\item[(1)]
		We apply the results obtained in \cite{GdHOpr1} to $\Omega_N$ with $N<\infty$ fixed. We again find that part of the coexistence regime without seed-bank shifts into the clustering regime with seed-bank when the average wake-up time of a typical individual is infinite.
		\item[(2)]
		We analyse a \emph{space-time renormalised} system in the limit as $N\to\infty$. Namely, we show that the block averages on successive space-time scales each perform a diffusion with a \emph{renormalised diffusion function}. In other words, we establish a \emph{multi-scale version of the finite-systems scheme}. Also, we compare the behaviour of the space-time renormalised system with seed-bank to the one analysed in \cite{DG93a} and \cite{DG93b} without seed-bank. 
		\item[(3)]
		We exhibit \emph{universal behaviour} in the clustering regime close to criticality. To do so, we analyse the attracting orbits of the renormalisation transformation, acting on the space of diffusion functions, that connects successive hierarchical levels. We show that, in the \emph{clustering regime} and after appropriate scaling, the renormalised diffusion function  converges to the Fisher-Wright diffusion function as we move up in the hierarchy, irrespective of the diffusion function controlling the resampling. This convergence shows that the hierarchical system exhibits \emph{universality} on large space-time scales in terms of the scaling limit. For several subclasses of parameters we identify the scaling of the renormalised diffusion function, which reveals a delicate interplay between the parameters controlling the migration and the seed-bank. This rate in turn determines the speed at which \emph{mono-type clusters grow in space and time}.
	\end{itemize}
	
	In the \emph{coexistence regime}, universality does \emph{not} hold and the equilibrium depends on the diffusion function. Since the seed-bank enhances genetic diversity, it may be expected that equilibrium correlations between far away colonies decay faster with seed-bank than without seed-bank, an issue that will not be addressed. 
	
	\begin{remark}{\bf[More general types]}
		{\rm Throughout the paper we consider the \emph{two-type} Fisher-Wright model with seed-bank, in the \emph{continuum limit} where the number of individuals per colony tends to infinity. The extension to a general type space, called the Fleming-Viot model (see \cite{DGV95}), requires only standard adaptations and will not be considered here. In what follows, we \emph{only} work with continuum models. However, we motivate these models by viewing them as the \emph{large-colony-size limit} of individual-based models. For earlier work on hierarchically interacting Fisher-Wright diffusions without seed-bank we refer the reader to \cite{DG93b,DG93a,DGV95,DG96} and \cite{BCGH95,BCGH97,H01}.} \hfill $\blacksquare$
	\end{remark}

	\subsection{Outline}
	
	The present paper consist of two parts:
	\begin{itemize}
		\item {\bf Part I: Model and main results.}
		Sections~\ref{s.introduct}--\ref{s.orbit} collect the main propositions and theorems.  In Section~\ref{s.introduct} we define the hierarchical model and state some basic properties: the \emph{well-posedness} of the associated martingale problem (Proposition~\ref{P.wellp}), the \emph{duality relation} (Proposition~\ref{P.dual1}), and the \emph{clustering criterion} via duality (Proposition~\ref{T.dichcrit}). These properties were all derived in \cite{GdHOpr1}. In Section~\ref{s.mainfinite} we state our main results for $N<\infty$. In particular, we compute the scaling of the wake-up time and the migration kernel (Theorem~\ref{T.scalcoeff}) and identify the \emph{clustering regime} in terms of the coefficients controlling the migration and the seed-bank under the assumption that these are asymptotically polynomial or pure exponential (Theorem~\ref{T.cluscritreg})). In Section~\ref{s.intromultscallim} we state our main results for $N\to\infty$, the hierarchical mean-field limit. In particular, we introduce \emph{block averages on successive hierarchical space-time scales}, analyse their limiting dynamics (Theorems~\ref{T.multiscalehiereff} and \ref{T.multiscalehier}), offer a heuristic explanation how this limiting dynamics arises, introduce a path topology called the Meyer-Zheng topology that is needed for a proper formulation, and introduce an object called the interaction chain, which describes how the diffent hierarchical levels interact with each other. In Section~\ref{s.orbit} we identify the \emph{orbit of the renormalisation transformation} in the clustering regime (Theorem~\ref{T.scalren}), identify the rate of scaling for the renormalised diffusion function (Theorem~\ref{T.dichotomy}), and link this scaling to the rate of growth of mono-type clusters.   
		\item {\bf Part II: Preparations and proofs.}
		Sections~\ref{s.clusreg}--\ref{s.renormasym} provide the proofs of the theorems stated in Part I. These proofs consist of a long series of propositions and lemmas needed to build up the argument. In Section~\ref{s.clusreg} we prove our main results for $N<\infty$. In Sections~\ref{ss.IntroMeanfield}--\ref{ss.pabstracts} we focus on the \emph{mean-field model} (consisting of a single hierarchy) and, respectively, state and prove a number of results that serve as preparation. In Sections~\ref{s.finlevel}--\ref{s.multilevel} we consider extensions of the mean-field model (consisting of finitely hierarchies), which serve as further preparation. In Section~\ref{s.multilevel} we use the results in Sections~\ref{ss.IntroMeanfield}--\ref{s.multilevel} to deal with the full hierarchical model (consisting of infinitely many hierarchies), and prove our main results for $N\to\infty$. In Section~\ref{s.renormasym} we analyse the orbit of the renormalisation transformation controlling the multi-scaling. Appendix~\ref{app.comp} contains a technical computation needed for the identification of the clustering regime.  Appendix~\ref{apb} contains a basic introduction to convergence of paths in the Meyer-Zheng topology, which is needed for the main theorems. 
	\end{itemize}
	
	\noindent
	Part I contains all the main results and their interpretations, and can be read without reference to Part II.

	\part{MODEL AND MAIN RESULTS}

	\section{Introduction of model and basic properties}
	\label{s.introduct}
	
	Section~\ref{ss.model} introduces the model ingredients, Section~\ref{ss.evoleqs} gives the evolution equations, Section~\ref{ss.wellpos} states the well-posedness, Section~\ref{ss.duality} introduces the dual and states the duality relation, while Section~\ref{ss.clcr} formulates the dichotomy between clustering versus coexistence in terms of the dual.
	
	\subsection{Model: geographic space $\Omega_N$, hierachical group of order $N$}
	\label{ss.model}
	
	\paragraph{Single colony.}
	Our building block is the single-colony Fisher-Wright model with seed-bank defined in \cite{BCKW16}. In that model, each individual in the population carries one of two types, $\heartsuit$ or $\diamondsuit$, and each individual can be either active or dormant. Active individuals \emph{resample} until they become dormant. Dormant individuals suspend resampling until they become actieve again. The repository for the dormant individuals is called the \emph{seed-bank}. When an active individual resamples, it randomly chooses another active individual and \emph{adopts its type}. When an active individual becomes dormant, it randomly chooses a dormant individual and forces it to becomes active, i.e., the active and the dormant population \emph{exchange} individuals (see Fig.~\ref{fig:singlecolony}). This exchange guarantees that the sizes of the active and the dormant population stay fixed over time. During the swap both the active and the dormant individual \emph{retain their type}. 
	
	The types of the active population evolve through resampling and through exchange with the dormant population. The types of the the dormant population evolve only through exchange with the active population. It was shown in \cite{BCKW16} that in the large-colony-size limit, i.e., as the number of individuals per colony tends to infinity and time is speeded up by the size of the colony, the two quantities 
	\begin{itemize}
		\item 
		$x(t)$ = the fraction of active individuals of type $\heartsuit$ at time $t$,
		\item 
		$y(t)$ = the fraction of dormant individuals of type $\heartsuit$ at time $t$,  
	\end{itemize}
	satisfy the following system of coupled SDEs:
	\begin{equation}
		\label{seedbankwithoutmigration}
		\begin{aligned}
			\d x(t)  &= Ke\,[y(t)-x(t)]\,\d t+\sqrt{x(t)(1-x(t))}\,\d w(t),\\
			\d y(t) &= e\,[x(t)-y(t)]\,\d t.
		\end{aligned}
	\end{equation}
	Here, $e$ denotes the rate at which an active individual \emph{exchanges} with a dormant individual from the seed-bank, $K$ denotes the relative size of the dormant population with respect to the active population, and $(w(t))_{t \geq 0}$ is a Brownian motion. The first term in the first equation describes the flow from the dormant population to the active population, the term in the second equation describes the flow from the active population to the dormant population, while the second term in the first equation describes the effect of resampling on the active population (see Fig.~\ref{fig:singlecolony}). Active individuals resample at rate $1$. Since dormant individuals do not resample, we do not see such a term in the second equation. The formal derivation of the continuum equations can be found in \cite{BCKW16} and in \cite[Appendix A]{GdHOpr1}. 
	
	\medskip\noindent
	\vspace{-0.5cm}
	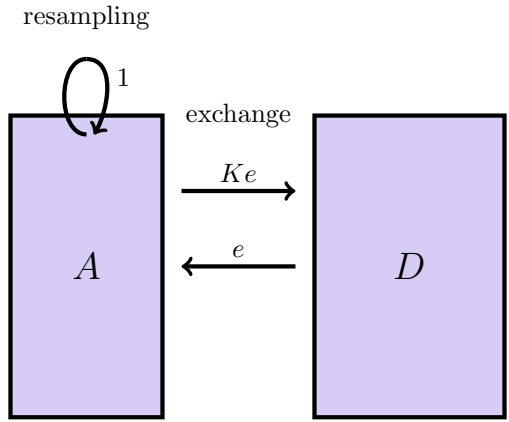
\begin{figure}[htbp]
		\begin{center}
			\begin{tikzpicture}
				\draw [fill=red!20!blue!20!,ultra thick] (0,0) rectangle (2,4);
				\draw [fill=red!20!blue!20!,ultra thick] (4,0) rectangle (6.5,4);
				\node  at (1,2) {\text{\Large $A$}};
				\node  at (5.25,2) {\text{\Large $D$}};
				\draw[ultra thick,<-](2.25,2)--(3.75,2);
				\draw[ultra thick,->](2.25,3)--(3.75,3);
				\node at (3,4) {\text{exchange}};
				\node at (1,5.3) {\text{resampling}};
				\draw [ultra thick] (1,3.75)to [out=180,in=180](1, 4.75) ;
				\draw [ultra thick,->] (1,4.75)to [out=0,in=60](1.1, 3.75) ;
				\node[above]  at (3,3) {$Ke$};
				\node[above]  at (3,2) {$e$};
				\node[above]  at (-0.75,3) {};
				\node[above]  at (-0.75,2) {};
				\node[left]  at (1.7,4.5) {1};
			\end{tikzpicture}
		\end{center}
		\caption{\small Active individuals resample at rate $1$. Active and dormant individuals exchange at rate $e$. The extra factor $K$ arises from the fact that the dormant population is $K$ times as large as the active population. Dormant individuals suspend resampling.}
		\label{fig:singlecolony}
		\vspace{-.5cm}
	\end{figure}
	
	\paragraph{Multi-colony.}
	The present paper focuses on a \emph{multi-colony} setting of the model described above, where the underlying geographic space is the hierarchical lattice of order $N$, given by ($\N_0 = \N \cup \{0\}$)
	\begin{equation}
		\label{hiergroup} 
		\gls{hiern} = \left\{\xi=(\xi_k)_{k\in\mathbb{N}_0}\colon\, 
		\xi_k\in\{0,1,\ldots,N-1\},\ \sum_{k\in\mathbb{N}_0}\xi_k<\infty\right\},
	\end{equation}
	which with addition modulo $N$ becomes the hierarchical group of order $N$ (see Fig.~\ref{fig-hierargr}). The \emph{hierarchical distance} on $\Omega_N$ is defined by
	\begin{equation}
		\gls{dhier}(\xi,\eta)=d_{\Omega_N}(0,\xi-\eta)
		=\min\left\{k\in\mathbb{N}_0\colon\,\xi_l=\eta_l\ \forall\, l \geq k\right\},
		\qquad \xi,\eta \in \Omega_N,
	\end{equation} 
	and is an ultra-metric, i.e., 
	\begin{equation}
		\label{ultra}
		d_{\Omega_N}(\xi,\eta) \leq \max\big\{d_{\Omega_N}(\xi,\zeta),d_{\Omega_N}(\eta,\zeta)\big\}
		\qquad \forall\,\xi,\eta,\zeta \in \Omega_N.
	\end{equation}
	
	The choice of $\Omega_N$ as geographic space plays an important role for population models, and was first exploited in \cite{SF83} in an attempt to formalise ideas coming from ecology. One interpretation is that the sequence $(\xi_k)_{k\in\N_0}$ encodes the `address' of colony $\xi$: $\xi_0$ is the `house', $\xi_1$ is the `street', $\xi_2$ is the `village', $\xi_3$ is the `province', $\xi_4$ is the `country', and so on. To describe the system on the hierarchical group we need three ingredients: 
	\begin{enumerate}
		\item[1.] Hierarchical migration (Section \ref{ss.2.2}).
		\item[2.] Layered seed-bank (Section \ref{ss.2.3}).
		\item[3.] Resampling rate (Section \ref{ss.2.4}).
	\end{enumerate}
	
	\begin{figure}[htbp]
		\centering
		\includegraphics[width=\textwidth]{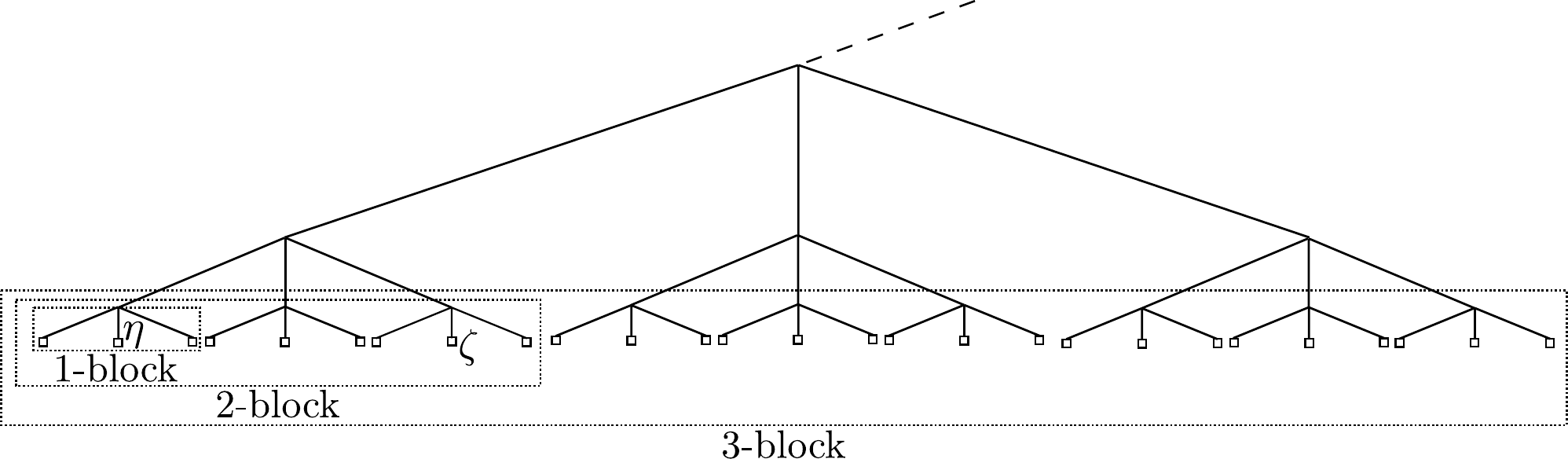}
		\caption{\small Close-ups of a 1-block, a 2-block and a 3-block in the hierarchical group 
			of order $N=3$. The elements of the group are the leaves of the tree (indicated by $\Box$'s). 
			The hierarchical distance between two elements in the group is the graph distance to the most 
			recent common ancestor in the tree: $d_{\Omega_3} (\eta,\zeta) = 2$ for $\eta$ and $\zeta$ 
			in the picture.}
		\label{fig-hierargr}
	\end{figure}
	
	\subsubsection{Hierarchical migration}
	\label{ss.2.2}
	
	We construct a migration kernel $a^{\Omega_N}(\cdot,\cdot)$ on the hierarchical group $\Omega_N$ built from a sequence of migration rates
	\begin{equation}
		\label{738}
		\gls{mig} = (\gls{migk})_{k \in \N_0} \in (0,\infty)^{\N_0}
	\end{equation}
	that do not depend on $N$. Individuals migrate as follows: 
	\begin{itemize}
		\item
		For all $k\in\N$, each individual chooses at rate $c_{k-1}/N^{k-1}$ the block of radius $k$ around its present location and selects a colony uniformly at random from that block. Subsequently it selects an individual in this colony uniformly at random and adopts its type. 
	\end{itemize}
	Note that the block of radius $k$ contains $N^k$ colonies, and that the migration kernel \gls{migker}  is therefore given by
	\begin{equation}
		\label{739}
		a^{\Omega_N} (\eta,\xi) = \suml_{k \geq d_{\Omega_N}(\eta,\xi)} \frac{c_{k-1}}{N^{k-1}}\frac{1}{N^k}, 
		\quad  \eta, \xi \in \Omega_N, \eta \neq \xi , \qquad a^{\Omega_N}(\eta,\eta)=0, \quad \eta \in \Omega_N.
	\end{equation}
	Throughout the paper, we assume that
	\begin{equation}\label{740}
		\limsup_{k \to \infty} \frac{1}{k}\,\log c_k < \log N.
	\end{equation}
	This guarantees that the total migration rate per individual is finite. Indeed, note that for every $\eta \in \Omega_N$,
	\be{}
	\sum_{\xi\in\Omega_N}a^{\Omega_N}(\eta,\xi) 
	= \sum_{\xi\in\Omega_N} \sum_{k\geq d_{\Omega_N}(\eta,\xi)} \frac{c_{k-1}}{N^{2k-1}} 
	= \sum_{k\in\mathbb{N}} \left[\sum_{\xi\in\Omega_N} \textbf{1}_{\{d_{\Omega_N}(\eta,\xi)\leq k\}}\right] 
	\frac{c_{k-1}}{N^{2k-1}} = \sum_{k\in\mathbb{N}} \frac{c_{k-1}}{N^{k-1}},
	\ee 
	which is finite because of \eqref{740}.
	
	\begin{remark}{\bf [Degree of random walk]}
		\label{rem:degree}
		{\rm For a random walk on an Abelian group with time-$t$ transition kernel \gls{migkert}, the \emph{degree} is defined as (see \cite{DGW05})
			\begin{equation}
				\label{e1434}
				\gls{degrw} = \sup\left\{\zeta \in (-1,\infty) \colon\,\int^\infty_0  \d t\,t^\zeta a^{\Omega_N}_t(0,0)< \infty\right\}.
			\end{equation}
			The degree is said to be $\delta^+$, respectively, $\delta^-$ when the integral is finite, respectively, infinite \emph{at} the degree. If $\delta>0$, then $\delta$ is called the \emph{degree of transience}. If $\delta \in (-1,0)$, then $-\delta$ is called the \emph{degree of recurrence}. If the degree is $0^-$, then the random walk is called \emph{critically recurrent}. (It would be interesting to have a version of \eqref{e1434} that includes a slowly varying function in front of the power $t^\zeta$. However, such an extension appears not to have been explored in the literature.)} \hfill $\blacksquare$
	\end{remark}
	
	By playing with $\underline{c}$ and letting $N\to\infty$, we can approximate migration for which the corresponding random walk is critically recurrent, i.e., $\delta^{-}=0$. In that case both the potential theory and the Green function for the hierarchical random walk have the same asymptotics as the potential theory and the Green function for a critically recurrent random walk on $\Z^2$ in the domain of attraction of Brownian motion. Therefore, by tuning $\underline{c}$ properly, we can mimic migration on the geographic space $\Z^2$ (for which $\delta^-=0$), an idea that was exploited in \cite{DGV95}, \cite{DG96}, \cite{DGsel14}, \cite{GHKK14}, \cite{GHK18}.

	\subsubsection{Layered seed-bank}
	\label{ss.2.3}
	
	To create a layered seed-bank, dormant individuals are labeled with a colour $m\in\N_0$. An active individual that becomes dormant is assigned a colour $m\in\N_0$. When an active individual becomes dormant with colour $m$, it exchanges with a dormant individual of colour $m$. This dormant individual becomes active, \emph{loses} its colour, but \emph{retains} its type. To describe the layered seed-bank we need two sequences
	\begin{equation}
		\begin{aligned}
			\label{defKem}
			&\gls{ssbs} = (K_m)_{m\in\N_0}\in (0,\infty)^{\N_0},\\
			&\gls{sexp} = (\gls{expar})_{m\in\N_0}\in(0,\infty)^{\N_0},
		\end{aligned}
	\end{equation}
	both not depending on $N$, which we interpret as follows: 
	\begin{itemize}
		\item 
		\gls{sbsm} is the relative size of the dormant population of colour $m$ with respect to the active 
		population, i.e.,
		\begin{equation}\label{ratio}
			K_m=\frac{\text{size $m$-dormant population}}{\text{size active population}}.
		\end{equation}
		\item 
		At rate $K_m\frac{e_m}{N^m}$ an active individual becomes dormant, is assigned colour $m$, and retains its type. At the same time a dormant individual with colour $m$ becomes active, loses its colour, and retains its type. By defining the rates in this way, the layered structure of the seed-bank is tuned to the hierarchical structure of the geographic space.
	\end{itemize} 
	
	\noindent
	By giving the seed-bank a layered structure, we are able to tune the distribution of the wake-up time, i.e., the time an individual spends in the seed-bank before waking up. In particular, we will see that a layered seed-bank enables us to model wake up times with a \emph{fat tail}, while at the same time preserving the Markov property of the evolution.
	
	Since active and dormant individuals {\em exchange}, $K_m$ remains constant over time for all $m\in\N_0$. Throughout the paper we assume that
	\begin{equation}
		\label{740alt}
		\limsup_{m \to \infty} \frac{1}{m}\,\log (K_m e_m) < \log N.
	\end{equation}
	This guarantees that the total rate of exchange per individual, given by 
	\begin{equation}
		\label{chidef}
		\gls{chi} = \sum_{m\in\mathbb{N}_0} K_m\frac{e_m}{N^m},
		\end {equation}
		is finite. On the other hand, the relative size of the dormant population with respect to the active population 
		\begin{equation}
			\label{rhodef}
			\gls{rho}=\sum_{m\in\N_0} K_m
		\end{equation}
		can be either finite or infinite. We will see that $\rho<\infty$ and $\rho=\infty$ represent two \emph{different regimes}.

		\subsubsection{Resampling rate}
		\label{ss.2.4}
		
		To describe the resampling we use a diffusion function \gls{g} that is taken from the set
		\begin{equation}
			\label{setG}
			\gls{cldf} = \Big\{g(x)\colon\,[0,1] \to [0,\infty)\colon\,g(0) = g(1) =0, \, g(x)>0 \,\,\forall\,x \in (0,1),
			\, g \mbox{ Lipschitz}\Big\},
		\end{equation}
		and think of $h(x)=g(x)/x(1-x)$ as the rate of resampling at type frequency $x$. The choice $g=dg_{\mathrm{FW}}$, $d \in (0,\infty)$, with $g_{\mathrm{FW}}(x)=x(1-x)$, $x \in [0,1]$, corresponds to Fisher-Wright resampling at rate $d$. We use a collection of independent Brownian motions 
		\begin{equation}
			W=\big(\gls{bm}\big)_{\xi\in\Omega_N}
		\end{equation}
		to describe the fluctuations of the type frequencies caused by the resampling in each colony.

		\subsection{Evolution equations}
		\label{ss.evoleqs}

		\subsubsection{Evolution of single colonies}
		
		With the above three ingredients, we can now describe the evolution of the system. For $\xi\in\Omega_N$, define
		\begin{equation}
			\begin{aligned} 
				\gls{x} &= \text{ the fraction of active individuals of type $\heartsuit$ at colony $\xi$ at time $t$},\\
				\gls{y}(t) &= \text{ the fraction of $m$-dormant individuals of type $\heartsuit$ at colony $\xi$ and time $t$}.  
			\end{aligned}
		\end{equation}
		Note that $x_\xi(t)\in [0,1]$ and $y_{\xi,m}(t)\in [0,1]$ for all $\xi\in\Omega_N$, $m\in\N_0$, $t \geq 0$.  Therefore the state space of a single colony is $\mathfrak{s} = [0,1]\times [0,1]^{\N_0}$, and the state space of the system is 
		\begin{equation}
			\gls{S} = \mathfrak{s}^{\Omega_N}.
		\end{equation} 
		Our object of interest is the random process taking values in $S$, written
		\begin{equation}
			\label{XYdef}
			\gls{fpz}=\gls{fp}, \qquad 
			(X^{\Omega_N}(t),Y^{\Omega_N}(t)) = \big(x_\xi(t),(y_{\xi,m}(t))_{m\in\N_0}\big)_{\xi\in\Omega_N},
		\end{equation}
		whose components evolve according to the following SSDE (= system of stochastic differential equations):
		\begin{equation}
			\label{moSDE}
			\begin{aligned}
				\d x_\xi(t) &= \sum_{\eta \in \Omega_N} a^{\Omega_N}(\xi,\eta)[x_\eta(t)-x_\xi(t)]\,\d t
				+\sqrt{g(x_\xi(t))}\,\d w_\xi(t)\\
				&\qquad +\sum_{m\in\N_0} \frac{K_me_m}{N^m}\, [y_{\xi,m}(t)-x_\xi(t)]\,\d t,\\[0.2cm]
				\d y_{\xi,m}(t) &= \frac{e_m}{N^m}\, [x_\xi(t)-y_{\xi,m} (t)]\,\d t,  
				\quad m\in\mathbb{N}_0, \qquad \xi \in \Omega_N.
			\end{aligned}
		\end{equation} 
		The first term in the first equation describes the evolution of the active population at colony $\xi$ due to migration, the second term due to the resampling. The third term in the first equation and the term in the second equation describe the exchange between the active and the dormant population at colony $\xi$ (see Fig.~\ref{fig:multicolony}). Since dormant individuals are not subject to resampling or migration, the dynamics of the dormant population is completely determined by the exchange with the active population. For the \emph{initial state} we assume that
		\begin{equation}
			\label{initialvalue}
			\begin{aligned}
				&\CL(X^{\Omega_N}(0),Y^{\Omega_N}(0)) = \mu^{\otimes\Omega_N}\\ 
				&\text{with } \E^\mu[x_\xi(0)]=\gls{thetax},\, \E^\mu[y_{\xi,m}(0)]=\gls{thetay} \text{ with } 
				\lim_{m\to\infty}\theta_{y_m}=\theta \text{ for some } \theta\in[0,1].
			\end{aligned}
		\end{equation}
		The last assumption in \eqref{initialvalue}, which in \cite{GdHOpr1} was referred to as $\mu$ being \emph{colour regular}, guarantees that for finite $N$ the system in \eqref{moSDE} converges to an ergodic equilibrium. 
		
		\begin{remark}{\bf [Notation]}
			\label{rem:not1}
			{\rm Throughout the sequel we use lower case letters for \emph{single components} and upper case letters for \emph{systems of single components}. We exhibit the geographic space for the system, but suppress it from the components.} \hfill $\blacksquare$ 
		\end{remark}
		
		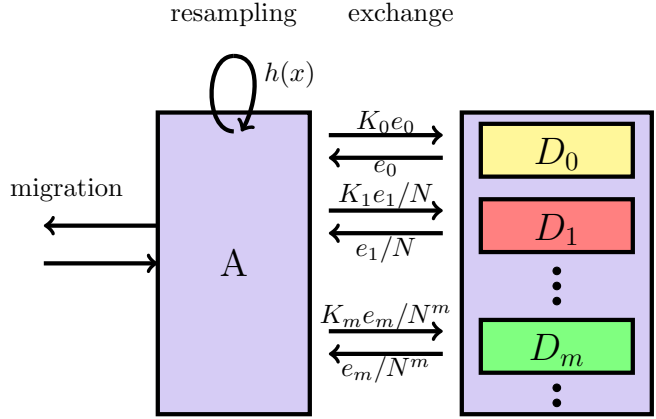
\begin{figure}[htbp]
			\begin{center}
				\begin{tikzpicture}
					\draw [fill=red!20!blue!20!,ultra thick] (0,0) rectangle (2,4);
					\draw [fill=red!20!blue!20!,ultra thick] (4,0) rectangle (6.5,4);
					\draw [fill=yellow!50!, ultra thick] (4.25,3.15) rectangle (6.25,3.85); 
					\draw [fill=red!50!, ultra thick] (4.25,2.15) rectangle (6.25,2.85); 
					\draw [fill=green!50!, ultra thick] (4.25,0.55) rectangle (6.25,1.25); 
					\node  at (1,2) {\text{\Large A}};
					\node  at (5.25,3.45) {\text{\Large $D_0$}};
					\node  at (5.25,2.45) {\text{\Large $D_1$}};
					\node  at (5.25,0.85) {\text{\Large $D_m$}};
					\draw[ultra thick,<-](2.25,3.4)--(3.75,3.4);
					\draw[ultra thick,->](2.25,3.7)--(3.75,3.7);
					\draw[ultra thick,<-](2.25,2.4)--(3.75,2.4);
					\draw[ultra thick,->](2.25,2.7)--(3.75,2.7);
					\draw[ultra thick,<-](2.25,0.8)--(3.75,0.8);
					\draw[ultra thick,->](2.25,1.1)--(3.75,1.1);
					\node at (3.2,5.3) {\text{exchange}};
					\node at (1,5.3) {\text{resampling}};
					\draw [ultra thick] (1,3.75) to [out=180,in=180](1, 4.75) ;
					\draw [ultra thick,->] (1,4.75) to [out=0,in=60](1.1, 3.75) ;
					\foreach \x in {1.9,1.7,1.5,0.35,0.15}
					\draw[fill] (5.25,\x) circle [radius=0.05];
					\draw[ultra thick, ->](0,2.5)--(-1.5,2.5);
					\draw[ultra thick, <-](0,2)--(-1.5,2);
					\node[above] at (-1.2,2.7){\text{migration}};
					\node[above]  at (3,3.6) {$K_0e_0$};
					\node[below]  at (3,3.49) {$e_0$};
					\node[above]  at (3,2.6) {$K_1e_1/N$};
					\node[below]  at (3,2.49) {$e_1/N$};
					\node[above]  at (3,1.0) {$K_me_m/N^m$};
					\node[below]  at (3,0.89) {$e_m/N^m$};
					\node[left]  at (2.2,4.5) {$h(x)$};	
				\end{tikzpicture}
			\end{center}
			\caption{\small Active individuals ($A$) are subject to migration, resampling and exchange 
				with dormant individuals ($D$). When active individuals become dormant they are assigned a 
				colour ($D_m$, $m\in\N_0$), which they lose when they become active again. The resampling rate 
				in the active state at type-$\heartsuit$ frequency $x$ equals $h(x)=g(x)/x(1-x)$ with $g \in \CG$ (e.g.\ for the 
				standard Fisher-Wright diffusion the resampling rate is $1$).}
			\label{fig:multicolony}
		\end{figure}
		
		\subsubsection{Evolution of block averages }\label{ss.blav}
		
		The choice of the migration kernel in \eqref{739} implies that, for every $k\in\N$,  at rate $\asymp\frac{1}{N^k}$ individuals choose a space horizon of distance $k+1$ and subsequently choose a random colony from that space horizon. Therefore, in order to see interactions over a distance $k+1$, we need to speed up time by a factor $N^k$. A similar observation applies to the interaction with the seed-bank. Dormant individuals with colour $k$ become active at rate $\asymp\frac{1}{N^k}$. Therefore, in order to see interactions with the $k$-dormant population, we need to speed up time by a factor $N^k$. To analyse the effective interaction on time scale $N^k$, we introduce \emph{successive block averages} labelled by $k\in\N_0$. 
		
		\begin{definition}{\bf [Block averages]}
			\label{defblockav}
			{\rm For $k\in\mathbb{N}_0$, let $B_k(0)=\{\eta\in\Omega_N\colon\,d_{\Omega_N}(0,\eta\leq k\}$ denote the $k$-block around $0$. Define the $k$-block average around $0$ at time $N^kt$ by 
				\begin{equation}
					\label{blockav}
					\begin{aligned}
						\gls{xk} &= \frac{1}{N^k} \sum_{\eta \in B_k(0)} x_\eta(N^kt),\\
						\gls{yk} &= \frac{1}{N^k} \sum_{\eta \in B_k(0)} y_{\eta,m}(N^kt),
						\qquad  m\in\mathbb{N}_0.
					\end{aligned}
				\end{equation}
				The $k$-block average represents the dynamics of the system on space-time scale $k$.} \hfill $\blacksquare$
		\end{definition}
		
		\noindent
		By translation invariance of the SSDE in  \eqref{moSDE}, each $\xi\in\Omega_N$ can serve as the origin. In the remainder of the paper we consider without loss of generality the $k$-block average around $\xi=0$, and suppress the center $0$ from the notation. 
		
		\begin{remark}{\bf [Notation]}
			\label{rem:not2}
			{\rm We use lower case letters for the block averages because they live in the space of components $\mathfrak{s}=[0,1]\times[0,1]^{\N_0}$. At the same time we exhibit the geographic space $\Omega_N$ for the block averages because they are functionals of the system of components (recall Remark~\ref{rem:not1}).} \hfill$\blacksquare$
		\end{remark}
		
		Using Definition~\ref{defblockav} and inserting the specific choice of the migration kernel defined in \eqref{739}, we can rewrite \eqref{moSDE} for $\xi=0$ as follows ($0$-blocks are single components): 
		\begin{equation}
			\label{blockavx*z*alt}
			\begin{aligned}
				\d x_{0}^{\Omega_N}(t) &= \sum_{l\in\mathbb{N}} \frac{c_{l-1}}{N^{l-1}}
				\big[x_{l}^{\Omega_N}(N^{-l}t)-x_{0}^{\Omega_N}(t)\big]\,\d t+\sqrt{g\big(x_{0}^{\Omega_N}(t)\big)}\,\d w(t)\\
				&\qquad + \sum_{m\in\mathbb{N}_0} \frac{K_me_m}{N^m}
				\big[y_{m,0}^{\Omega_N}(t)-x_{0}^{\Omega_N}(t)\big]\,\d t,\\[0.2cm]
				\d y_{m,0}^{\Omega_N}(t)
				&= \frac{e_m}{N^m} \big[x_{0}^{\Omega_N}(t)-y_{m,0}^{\Omega_N}(t)\big]\,\d t,\qquad m\in\N_0.
			\end{aligned}
		\end{equation}
		From \eqref{blockavx*z*alt} we see that migration between colonies can be expressed as a drift towards block averages at a higher hierarchical level.
		
		The SSDE for the $k$-block average on time scale $N^k$ reads as follows (recall \eqref{blockav}): 
		\begin{equation}
			\label{rblockavxzmulti}
			\begin{aligned}
				\d{x}_{k}^{\Omega_N}(t)
				&= \sum_{l\in\N} 
				\frac{c_{k+l-1}}{N^{l-1}}\big[{x}_{k+l}^{\Omega_N}(N^{-l}t)-{x}_{k}^{\Omega_N}(t)\big]\,\d t
				+\sqrt{\frac{1}{N^k} \sum_{i\in B_k(0)} g\big(x_i(N^kt)\big)}\,\,\d w_k(t)\\
				&\qquad +\sum_{m\in\N_0} N^k \frac{K_m e_m}{N^m}
				\big[{y}_{m,k}^{\Omega_N}(t)-{x}_{k}^{\Omega_N}(t)\big]\,\d t,\\[0.2cm]
				\d{y}_{m,k}^{\Omega_N}(t)
				&=N^k \frac{e_m}{N^{m}}\big[{x}_{k}^{\Omega_N}(t)-{y}_{m,k}^{\Omega_N}(t)\big]\,\d t,
				\qquad m\in\mathbb{N}_0.
			\end{aligned}
		\end{equation}
		To deduce these equations from \eqref{moSDE}, we sum over $\xi\in B_k$, speed up time by a factor $N^k$, insert the specific choice of the migration kernel in \eqref{739}, and use the standard scaling properties of Brownian motion: $w(ct) =^d \sqrt{c}\,w(t)$ and $\sqrt{a}w(t)+\sqrt{b}w^\prime(t)=^d \sqrt{a+b}\,w^{\prime\prime}(t)$, with $w(t)$ and $w^\prime(t)$ independent Brownian motions, and with $=^d$ denoting equality in distribution. This computation is spelled out in Section~\ref{s.multilevel}. 
		
		\begin{remark}
			{\rm The \emph{block averages} and their evolution equations in \eqref{rblockavxzmulti} will be key objects in the analysis of the hierarchical mean-field limit $N\to\infty$. We will see that the limit $N\to\infty$ brings about considerable simplifications. In Section \ref{s.intromultscallim} we discuss these simplifications in detail. In particular, a \emph{complete separation of space-time scales} takes places, in which each block average lives on its own time scale, effectively interacts with only one seed-bank, and effectively feels a drift towards the block average one hierarchical level up.} \hfill$\blacksquare$
		\end{remark}
		
		\subsection{Well-posedness}
		\label{ss.wellpos}
		Let \gls{cf} denote the space of continuous function form $[0,\infty)$ to the state space $S$ and let \gls{cfb} be the space of bounded continuous functions form $[0,\infty)$ to $S$. 
		The generator of the system in \eqref{moSDE} is given by
		\begin{eqnarray}
			\gls{G}&=&\sum_{\xi\in\Omega_N}\Bigg(\sum_{\eta\in\Omega_N}a^{\Omega_N}(\xi,\eta)[x_\eta(t)-x_\xi(t)]
			\frac{\partial }{\partial x_\xi}+\frac{1}{2}g(x_\xi(t))\frac{\partial^2 }{\partial x_\xi^2}\\ \nonumber
			&&\quad +\sum_{m\in\N_0} \left[\frac{K_m e_m}{N^m}[y_{\xi,m}(t)-x_\xi(t)]
			\frac{\partial }{\partial x_\xi}+\frac{ e_m}{N^m}[x_{\xi}(t)-y_{\xi,m}(t)]
			\frac{\partial }{\partial y_{\xi,m}}\right]\Bigg).
		\end{eqnarray}
		Let 
		\begin{equation}
			\label{eq401}
			\begin{aligned}
				\gls{DF}= \Big\{&f \in C_b([0,\infty),S) \colon\, f \mbox{ depends on finitely many components}\\[-0.2cm] 
				&\mbox{and is twice continuously differentiable in each component}\Big\}.
			\end{aligned}
		\end{equation}
		
		\begin{proposition}{{\bf [Well-posedness]}}
			\label{P.wellp}\begin{itemize}
				\item[\rm (a)] 
				The SSDE in \eqref{moSDE} has a unique strong solution in \gls{cf}, whose law is the unique solution of the $(G,\mathbf{F},\delta_u)$-martingale problem for all $u\in S$.
				\item[\rm (b)] 
				The process starting from $u \in S$ is Feller and strong Markov. Consequently, the SSDE in \eqref{moSDE} defines a unique Borel Markov process starting from any initial law on $S$. 
			\end{itemize} 
		\end{proposition}
		
		\begin{proof}
			Comparing with what is called model 2 in \cite{GdHOpr1}, we see that the Abelian group is chosen as in \eqref{hiergroup}, the transition kernel is chosen as in \eqref{739}, and the rates in and out of the seed-bank are $\frac{e_m}{N^m}$ and $\frac{K_me_m}{N^m}$ for colour $m$. Hence the claim follows from \cite{SS80}, in the same way as shown in the proof of \cite[Theorem 
			2.1]{GdHOpr1}.
		\end{proof}
		
		Henceforth we write $\P$ and $\E$ to denote probability and expectation with respect to the random process in \eqref{XYdef}.
		
		\subsection{Duality}
		\label{ss.duality}
		
		If $g= d\gls{FW}$, then our model has a tractable dual, which turns out to play a crucial role in the analysis of the long-time behaviour. In this section we introduce the dual process following the same line of argument as in \cite[Section 2.4]{GdHOpr1}. There it was shown that the spatial Fisher-Wright diffusion with seed-bank is dual to a so-called \emph{block-counting process} of a seed-bank coalescent. The latter describes the ancestral lines of $n\in\N$ individuals sampled from the current population backwards in time in terms of partition elements. At time zero the ancestral line of each individual is represented by a partition element. Traveling backwards in time, two partition elements merge as soon as their ancestral lines coalesce, i.e., two individuals have the same ancestor from that time onwards. Hence the seed-bank coalescent divides the ancestral lines of the $n\in\N$ individuals into subgroups of individuals with the same ancestor (i.e., individuals that are identical by descent). Therefore the seed-bank coalescent generates the ancestral lineages of the individuals evolving according to a Fisher-Wright diffusion with seed-bank, i.e., generates their full genealogy. The corresponding block-counting process counts the number of partition elements that are left when we travel backwards in time.    
		
		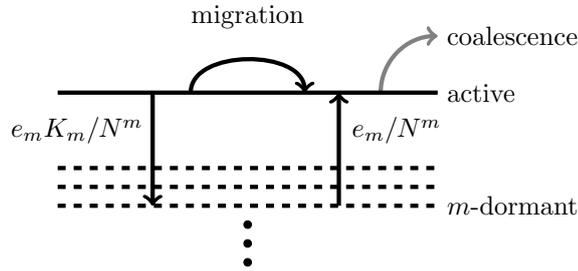
\begin{figure}[htbp]
			\begin{center}
				\begin{tikzpicture}[scale=0.5]
					\draw [ultra thick,->] (-1.5,1)to [out=90,in=90](1.5,1);
					\draw [ultra thick,black!50!,->] (3.5,1)to [out=90,in=180](5,2.5); 
					\node[right] at (5,2.5){coalescence};
					\node[] at (0,3){migration};
					\node[left] at (-2.5,0) {$e_mK_m/N^m$};
					\node[right] at (2.5,0) {$e_m/N^m$};
					\node[right] at (5,-2) {$m$-dormant};
					\node[right] at (5,1) {active};
					\draw[ultra thick, ->] (-2.5,1)--(-2.5,-2);
					\draw[ultra thick, <-] (2.4,1)--(2.4,-2);
					\draw[ultra thick] (-5,1)--(5,1);
					\draw[dashed, ultra thick](-5,-1)--(5,-1);
					\draw[dashed, ultra thick](-5,-1.5)--(5,-1.5);
					\draw[dashed, ultra thick](-5,-2)--(5,-2);
					\draw[fill] (0,-2.5) circle [radius=0.1];
					\draw[fill] (0,-3) circle [radius=0.1];
					\draw[fill] (0,-3.5) circle [radius=0.1];
				\end{tikzpicture}
			\end{center}
			\vspace{-.2cm}
			\caption{\small Transition scheme for an ancestral lineage in the dual, which moves according to the transition kernel $b(\cdot,\cdot)$ defined in \eqref{mrw}. Two active ancestral lineages that are at the same colony coalesce at rate $d$.}
			\label{fig:dualrw}	
		\end{figure}
		
		Formally, the spatial seed-bank coalescent is described as follows. Let $\S=\Omega_N\times\{A,(D_m)_{m\in\N_0}\}$ be the \emph{effective geographic space}. For $n\in\N$ the state space of the \emph{$n$-spatial seed-bank coalescent} is the set of partitions of $\{1,\ldots,n\}$, where the partition elements are marked with a position vector giving their locations. A state is written as $\pi$, where
		\begin{equation}
			\label{e402}
			\begin{aligned}
				&\pi = ((\pi_1,\eta_1), \ldots, (\pi_{\bar{n}},\eta_{\bar{n}})), \quad \bar{n} = |\pi|,\\
				&\pi_\ell\subset \{1,\ldots, n\},  \quad \{\pi_1,\cdots \pi_{\bar{n}}\}\ \mbox{  is a partition of } \{1,\ldots, n\},\\ 
				&\eta_\ell \in \S, \quad \ell \in \{1,\ldots, \bar{n}\}, \quad 1\leq \bar{n}\leq n.
			\end{aligned}
		\end{equation}
		A marked partition element $(\pi_\ell,\eta_\ell)$ is called active if $\eta_\ell=(\xi,A)$ and $m$-dormant if $\eta_\ell=(\xi,D_m)$ for some $\xi\in\Omega_N$. The $n$-spatial seed-bank coalescent is denoted by
		\begin{equation}
			\label{e403}
			\gls{SC},
		\end{equation}
		and starts from
		\begin{equation}
			\label{e403alt}
			\CC^{(n)}(0) = \pi(0), \qquad  \pi(0) = \{(\{1\},\eta_{\ell_1}),\ldots,(\{n\},\eta_{\ell_n})\}, \qquad
			\eta_{\ell_1},\ldots,\eta_{\ell_n} \in \S.
		\end{equation}
		
		\begin{figure}[htbp]
			\begin{center}
				\vspace{.2cm}
				\setlength{\unitlength}{.5cm}
				\begin{tikzpicture}[scale=0.4]
					\draw [fill=blue!10!] (0,-0.5) rectangle (1,4.5);
					\draw [fill=blue!10!] (2,-0.5) rectangle (3,4.5);
					\draw [fill=blue!10!] (4,-0.5) rectangle (5,4.5);
					\draw [fill=blue!10!] (6,-0.5) rectangle (7,4.5);
					\draw [fill=blue!10!] (8,-0.5) rectangle (9,4.5);
					\draw [fill=black, ultra thick] (10,2) circle [radius=0.05] ;
					\draw [fill=black, ultra thick] (11,2) circle [radius=0.05] ;
					\draw [fill=black, ultra thick] (12,2) circle [radius=0.05] ;
					\draw [fill=blue!10!] (13,-0.5) rectangle (14,4.5);
					\draw [ultra thick, ->] (-1,-.5) -- (-1,5);
					\node[left] at (-1.5,2.5) {$t$};
					\draw [fill=black] (.25,-.5) circle [radius=0.1] ;
					\draw [fill=black] (.75,-.5) circle [radius=0.1] ;
					\draw[black,ultra thick] (.25,-.5)--(.25,2);	
					\draw[black,ultra thick] (.75,-.5)--(.75,2)--(.25,2)--(.25,4.5);
					\draw [fill=black] (2.25,-.5) circle [radius=0.1] ;
					\draw[black,ultra thick] (2.25,-.5)--(2.25,1.5);
					\draw[green,ultra thick] (2.25,1.5)--(2.25,3.5);
					\draw[black,ultra thick] (2.25,3.5)--(2.25,4)--(.75,4)--(.75, 4.5);
					\draw [fill=red] (4.25,-.5) circle [radius=0.1] ;
					\draw[red,ultra thick] (4.25,-.5)--(4.25,4.5);	
					\draw [fill=black] (6.25,-.5) circle [radius=0.1] ;
					\draw[black,ultra thick] (6.25,-.5)--(6.25,4.5);
					\draw [fill=yellow] (6.75,-.5) circle [radius=0.1] ;
					\draw[yellow,ultra thick] (6.75,-.5)--(6.75,1);
					\draw[black,ultra thick] (6.75,1)--(6.75,3)--(8.75,3)--(8.75,4.5);
					\draw [fill=black] (13.25,-.5) circle [radius=0.1] ;
					\draw[black,ultra thick] (13.25,-.5)--(13.25,4.5);	
				\end{tikzpicture}
			\end{center}
			\vspace{-.2cm}
			\caption{\small Picture of the evolution of lineages in the spatial coalescent. The purple blocks depict the colonies, the black lines the active lineages, and the coloured lines the dormant lineages. Blue lineages can migrate. Two black lineages can coalesce when they are at the same colony. Red dormant lineages first have to become black and active before they can migrate or coalesce with other black and active lineages. Note that the dual runs \emph{backwards in time}.}
			\label{fig:Duality}
		\end{figure}
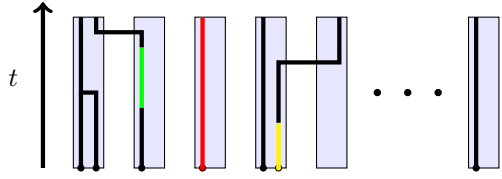

		The $n$-spatial seed-bank coalescent is the Markov process that evolves according to the following two rules (see Figs.~\ref{fig:dualrw}--\ref{fig:Duality}).
		\begin{enumerate}
			\item 
			Each partition element moves independently of all other partition elements according to the transition kernel \gls{tkerd}
			\begin{equation}
				\label{mrw}
				\hspace{-0.2cm}
				b^{}((\xi,R_\xi), (\eta, R_\eta)) = \left\{ \begin{array}{ll}
					a^{\Omega_N} (\xi, \eta), &\text{ if } R_\xi = R_\eta =A,\\
					K_m\frac{e_m}{N^m}, &\text{ if } \xi = \eta,\ R_\xi=A,\ R_\eta = D_m, \text{ for some } m\in\N_0, \\
					\frac{e_m}{N^m}, &\text{ if } \xi = \eta,\ R_\xi=D_m,\ R_\eta = A, \text{ for some } m\in\N_0,\\
					0, &\mbox{ otherwise},
				\end{array}
				\right.
			\end{equation}
			where $a^{\Omega_N}(\cdot,\cdot)$ is the migration kernel defined in \eqref{739}, $K_m$, $m\in\N_0$ are the relative sizes of the $m$-dormant population and the active population defined in \eqref{ratio}, and $e_m$, $m\in\N_0$ are the coefficients controlling the exchange between the active and the dormant population defined in \eqref{defKem}. Thus, an active partition element migrates according to the transition kernel $a^{\Omega_N}(\cdot,\cdot)$ and becomes $m$-dormant at rate $K_m\frac{e_m}{N^m}$, while an $m$-dormant partition element can only become active and does so at rate $\frac{e_m}{N^m}$. 
			\item 
			Independently of all other partition elements, two partition elements that are at the same colony and are both active coalesce with rate $d$, i.e., the two partition elements merge into one partition element.
		\end{enumerate}
		Fig.~\ref{fig:dualrw} gives a schematic overview of the possible transitions of a single lineage, while Fig.~\ref{fig:Duality} gives an example of the evolution in the dual. The \emph{spatial seed-bank coalescent} $(\CC(t))_{t\geq 0}$ is defined as the projective limit of the $n$-spatial seed-bank coalescents $(\CC^{(n)}(t))_{t\geq 0}$ as $n\to\infty$. This object is well-defined by Kolmogorov's extension theorem (see \cite[Section 3]{BCKW16}).
		
		For $n\in\N$ we define the block-counting process \gls{L} corresponding to the $n$-spatial seed-bank coalescent as the process that counts at each site $(\xi,R_\xi)\in\Omega_N\times\{A,(D_m)_{m\in\N_0}\}$ the number of partition elements of $\CC^{(n)}(t)$,  i.e., 
		\begin{equation}
			\label{blctpr}
			\begin{aligned}
				&L(t)=\big(L_{(\xi,A)}(t),\left(L_{(\xi,D_m)}(t)\right)_{m\in\N_0}\big)_{\xi\in\Omega_N},\\
				&L_{(\xi,A)}(t)=L_{(\xi,A)}(\CC^{(n)}(t))
				=\sum_{\ell=1}^{\bar{n}}1_{\{\eta_\ell(t)=(\xi,A)\}},\\ 
				&L_{(\xi,D_m)}(t)=L_{(\xi,D_m)}(\CC^{(n)}(t))
				=\sum_{\ell=1}^{\bar{n}}1_{\{\eta_\ell(t)=(\xi,D_m)\}}, \qquad m \in \N_0.
			\end{aligned}
		\end{equation}
		The state space of $(L(t))_{t\geq 0}$ is $S'=(\N_0\times\N_0^{\N_0})^{\Omega_N}$. We denote the elements of $S^\prime$ by sequences $(m_\xi,(n_{\xi,D_m})_{m\in\N_0})_{\xi\in\Omega_N}$, and define $\delta_{(\eta,R_\eta)}\in S^\prime$ to be the element of $S^\prime$ that is $0$ at all sites $(\xi,R_\xi)\in\Omega_N\times\{A,(D_m)_{m\in\N_0}\}\backslash (\eta,R_\eta)$, and $1$ at the site $(\eta,R_\eta)$. From the evolution of $\CC^{(n)}(t)$ described below \eqref{e403} we see that the block-counting process has the following transition kernel:
		\begin{equation}\label{blockproc}
			\begin{aligned}
				&(m_\xi,(n_{\xi,D_m})_{m\in\N_0})_{\xi\in\Omega_N} \rightarrow\\
				&\qquad\begin{cases}
					({m}_\xi,(n_{\xi,D_m})_{m\in\N_0})_{\xi\in\Omega_N}-\delta_{(\eta,A)}+\delta_{(\zeta,A)},\ 
					&\text{at rate } m_\eta a(\eta,\zeta) \text{ for } \eta,\zeta\in\Omega_N,\\
					({m}_\xi,(n_{\xi,D_m})_{m\in\N_0})_{\xi\in\Omega_N}-\delta_{(\eta,A)},\ 
					&\text{at rate } d{m_\eta \choose 2} \text{ for } \eta\in\Omega_N,\\
					({m}_\xi,(n_{\xi,D_m})_{m\in\N_0})_{\xi\in\Omega_N}-\delta_{(\eta,A)}+\delta_{(\eta,D_m)},\ 
					&\text{at rate } m_\eta K_m\frac{e_m}{N^m} \text{ for } \eta\in\Omega_N, \\
					({m}_\xi,(n_{\xi,D_m})_{m\in\N_0})_{\xi\in\Omega_N}+\delta_{(\eta,A)}-\delta_{(\eta,D_m)},\ 
					&\text{at rate } n_{\eta,m}\frac{e_m}{N^m} \text{ for } \eta\in\Omega_N.\\
				\end{cases}
			\end{aligned}
		\end{equation}
		
		The process $(Z(t))_{t\geq 0}$ defined in \eqref{moSDE} is dual to the block-counting process $(L(t))_{t\geq 0}$ with \emph{duality function} $\gls{H}\colon\,S \times S^\prime\to \R$ defined by
		\begin{equation}
			\label{hpoldef}
			H\Big(\big(x_\xi,(y_{\xi,m})_{m\in\N_0}\big)_{\xi\in\Omega_N},\big(m_\xi,(n_{\xi,D_m})_{m\in\N_0}\big)_{\xi\in\Omega_N}\Big) 
			= \prod_{\xi\in\Omega_N} x_\xi^{m_\xi}\prod_{m\in \N_0}y_{\xi,m}^{n_{\xi,D_m}}.
		\end{equation}
		
		\begin{proposition}{{\bf [Duality relation]}}. 
			\label{P.dual1}
			Let $H$ be as in \eqref{hpoldef}. Then, for all $(x_\xi,(y_{\xi,m})_{m\in\N_0})_{\xi\in\Omega_N}$ $\in S$ and $(m_\xi,(n_{\xi,D_m})_{m\in\N_0})_{\xi\in\Omega_N}\in S^\prime$,
			\begin{equation}
				\label{e401}
				\begin{aligned}
					&\E_{\big(x_\xi,(y_{\xi,m})_{m\in\N_0}\big)_{\xi\in\Omega_N}}
					\Big[H\Big(\big(x_\xi(t),(y_{\xi,m}(t))_{m\in\N_0}\big)_{\xi\in\Omega_N},(m_\xi,n_\xi)_{\xi\in\Omega_N}\Big)\Big]\\
					&=\E_{\big(m_\xi,(n_{\xi,D_m})_{m\in\N_0}\big)_{\xi\in\Omega_N}}
					\Big[H\Big(\big(x_\xi,(y_{\xi,m})_{m\in\N_0 }\big)_{\xi\in\Omega_N},
					\big(L_{(\xi,A)}(t),(L_{(\xi,D_m)}(t))_{m\in\N_0}\big)_{\xi\in\Omega_N}\Big)\Big]
				\end{aligned}
			\end{equation}
			with $\E$ the generic symbol for expectation (on the left over the original process, on the right over the dual process).
		\end{proposition}
		
		\noindent
		Proposition \ref{P.dual1} was proved in \cite[Section 2.4]{GdHOpr1}. Since the duality function $H$ captures all the mixed moments of $(Z(t))_{t\geq 0}$, the duality relation is that of a \emph{moment dual}.
		
		\begin{remark}{\bf [Duality relation in terms of the effective geographic space]}
			\label{dualeff}
			{\rm Interpreting $(Z(t))_{t\geq 0}$ as a process on the effective geographic space $\S=\Omega_N\times\{A,(D_m)_{m\in\N_0}\}$, we can rewrite \eqref{moSDE} as
				\begin{equation}
					\label{SDE2}
					\begin{aligned}
						\d \gls{z}&=\sum_{(\xi,R_\xi)\in\S}b((\xi,R_\xi),(\eta,R_\eta))[z_{(\eta,R_\eta)}(t)-z_{(\xi,R_\xi)}(t)]\,\d t\\
						&\qquad+1_{\{R_\xi=A\}}\sqrt{g(z_{(\xi,R_\xi)}(t))}\,\d w_\xi(t),\qquad (\xi,R_\xi)\in\S,
					\end{aligned}
				\end{equation}
				where $b(\cdot,\cdot)$ is the transition kernel defined in \eqref{mrw}. If $g=dg_{\mathrm{FW}}$, then we can write its dual process as follows. Let $(L(t))_{t\geq 0}=(L(\CC(t))_{t \geq 0}$ be the block-counting process that at each site $(\xi,\R_\xi)\in\S$ counts the number of partition elements of $\CC(t)$, i.e.,
				\begin{equation}
					\begin{aligned}
						L(t)&=(L_{(\xi,R_\xi)}(t))_{{(\xi,R_\xi)}\in\S},\\
						L_{(\xi,R_\xi)}(t)&=L_{(\xi,R_\xi)}(\CC(t)) =\sum_{\ell=1}^{\bar{n}}1_{\{\eta_\ell(t)={(\xi,R_\xi)}\}}.
					\end{aligned}
				\end{equation}
				Rewrite the duality function $H$ in \eqref{hpoldef} as
				\begin{equation}
					H\Big((z_{(\xi,R_\xi)},l_{(\xi,R_\xi)})_{{(\xi,R_\xi)}\in\S}\Big) = \prod_{{(\xi,R_\xi)}\in\S}z_{(\xi,R_\xi)}^{l_{(\xi,R_\xi)}}.
				\end{equation}
				Then, for $z\in S$ and $l\in S^\prime$, the duality relation in \eqref{e401} reads 
				\begin{equation}
					\label{e401b}
					\E_{z(\xi,R_\xi)}\big[H(z_{(\xi,R_\xi)}(t),l_{(\xi,R_\xi)})\big] = \E_{l(\xi,R_\xi)}\big[H(z_{(\xi,R_\xi)},L_{(\xi,R_\xi)}(t))\big]. 
				\end{equation}
				Interpreting the duality relation in terms of the effective geographic space $\S$, we see that each ancestral lineage in the dual is a Markov process moving according to the transition kernel $b^{}(\cdot,\cdot)$ defined in \eqref{mrw}. Interpreting the duality relation in terms of the geographic space $\Omega_N$, we see that an ancestral lineage is a random walk on $\Omega_N$, with internal states $A$ and $(D_m)_{m\in\N_0}$. Both interpretations turn out to be useful when we analyse the long-time behaviour of the system.} \hfill$\blacksquare$
		\end{remark}

		\vspace{-0.5cm}
		\begin{figure}[htbp]
			\begin{center}
				\setlength{\unitlength}{.5cm}
				\begin{picture}(20,2)(0,0)
					\put(0,0){\line(20,0){20}}
					\put(0,0){\circle*{.25}}
					\put(2,0){\circle*{.25}}
					\put(5,0){\circle*{.25}}
					\put(7,0){\circle*{.25}}
					\put(11,0){\circle*{.25}}
					\put(13,0){\circle*{.25}}
					\put(18,0){\circle*{.25}}
					\put(.7,.4){$\sigma_1$}
					\put(3.3,.4){$\tau_1$}
					\put(5.7,.4){$\sigma_2$}
					\put(8.6,.4){$\tau_2$}
					\put(11.7,.4){$\sigma_3$}
					\put(15,.4){$\tau_3$}
				\end{picture}
			\end{center}		
			\caption{\small Renewal process induced by a lineage in the dual moving according to the transition kernel $b(\cdot,\cdot)$. For $k\in\N$, $\sigma_k$ denotes the $k$th active period and $\tau_k$ the $k$th dormant period.}
			\label{fig:per}
		\end{figure}
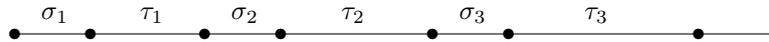
		
		\begin{remark}{\bf [The renewal process induced by the dual process]}
			\label{wakeup} 	
			{\rm The partition elements describing the dual process give rise to a renewal process on the active state $A$ and the dormant state $D=\bigcup_{m\in\N_0}D_m$. Since the only transition a dormant lineage can make is to become active, irrespectively of its colour, each dual lineage induces a sequence of active and dormant time lapses. Let $(\sigma_k)_{k\in\N}$ denote the successive active time periods and $(\tau_k)_{k\in\N}$ the successive dormant time periods (see Fig.~\ref{fig:per}). Then $(\gls{sigmak})_{k\in\N}$ and $(\gls{tauk})_{k\in\N}$ are sequences of i.i.d.\ random variables with marginal laws (recall \eqref{chidef})
				\begin{equation}
					\label{sigtau}
					\P(\sigma_1>t) =\e^{-\chi t}, \quad 
					\P(\tau_1>t) =\sum_{m\in\N_0} \frac{K_m\frac{e_m}{N^m}}{\chi}\,\e^{-\frac{e_m}{N^m} t}, \qquad t \geq 0.
				\end{equation}
				
				\begin{remark}{\bf[Wake up times]}
					{\rm The renewal process in Fig.~\ref{fig:per} is key to understanding the long-time behaviour of the model (as we will see in Section \ref{s.mainfinite}). Note that
						\begin{equation}
							\label{deftau}
							\gls{tau} = \tau_1
						\end{equation} 
						represents \emph{the typical wake-up time of a lineage in the dual}. By choosing specific sequences $(K_m)_{m\in\N_0}$ and $(e_m)_{m\in\N_0}$ we can mimic different wake-up time distributions. In particular, if we allow $\rho=\sum_{m\in\N_0}K_m=\infty$ (recall \eqref{rhodef}), then $\tau$ may have a fat-tail (examples are given in Section \ref{s.mainfinite}). In other words, \emph{the internal structure of the seed-bank allows us to model fat-tailed wake-up times without loosing the Markov property of the evolution.}} \hfill$\Box$
				\end{remark}
				
				Note that even when there is no dual, i.e., $g\in\CG$ with $g\neq dg_{\mathrm{FW}}$, we can still define $\tau$ by \eqref{deftau}, since $\tau_1$ in \eqref{sigtau} is a random variable that depends only on the sequences $(K_m)_{m\in\N_0}$ and $(e_m)_{m\in\N_0}$, and we can still interpret $\tau$ as the typical wake-up time of an individual in the population.} \hfill$\blacksquare$
		\end{remark}
		
		\subsection{Clustering criterion}
		\label{ss.clcr} 
		
		\subsubsection{Clustering criterion for Fisher-Wright diffusion function}
		
		In \cite{GdHOpr1} we showed that the system exhibits a dichotomy between \emph{coexistence} (= locally multi-type equilibria) and \emph{clustering} (= locally mono-type equilibria). The clustering criterion is based on the dual and requires the notion of colour regularity.
		We call a law translation invariant when it is \emph{invariant under the group action}. 
		
		\begin{definition}{\bf [Colour regular initial measures]}
			\label{D.regular}
			{\rm We say that a translation invariant initial measure $\mu(0)$ is \emph{colour regular} when
				\begin{equation}
					\label{covcond1}
					\lim_{m\to\infty} \E_{\mu(0)}[y_{0,m}] \quad \text{ exists}.
				\end{equation}	
				This condition is needed because, as time progresses, lineages starting from slower and slower seed-banks become active and bring new types into the active population. Without control on the initial states of the slow seed-banks, there may be no convergence to equilibrium.} \hfill$\blacksquare$
		\end{definition}
		
		The key clustering criterion is the following.
		
		\begin{proposition}{{\bf [Clustering criterion]}}
			\label{T.dichcrit} 
			Suppose that $\mu(0)$ is translation invariant. If $\rho=\infty$ (recall \eqref{rhodef}), then additionally suppose that $\mu(0)$ is colour regular. Let $d\in(0,\infty)$. Then the system with $g=dg_{\mathrm{FW}}$ clusters if and only if in the dual two partition elements coalesce with probability $1$.
		\end{proposition}
		
		The idea behind Theorem \ref{T.dichcrit} is as follows. If in the dual two partition elements coalesce with probability 1, then a random sample of $n$ individuals drawn from the current population has a common ancestor some finite time backwards in time. Since individuals inherit their type from their parent individuals, this means that all $n$ individuals have the same type. A formal proof was given in \cite[Section 4.3]{GdHOpr1}. The proof is valid for any geographic space given by a countable Abelian group endowed with the discrete topology, of which $\Omega_N$ is an example.
		
		\subsubsection{Clustering criterion for general diffusion function}
		\label{ss.comparisonarg}
		
		For $g\in\CG$ with $g\neq dg_{\mathrm{FW}}$ no dual is available and hence we cannot use the clustering criterion in Proposition \ref{T.dichcrit}. However, as shown in \cite{GdHOpr1}, we can argue by duality comparison arguments (see \cite[Lemma 5.5 and Lemma 6.3]{GdHOpr1}) that the system evolving according to \eqref{moSDE} with $g \in \CG$ clusters if and only if the system with $g=dg_{\mathrm{FW}}$ for some $d\in (0,\infty)$ clusters. In particular, for $g=d g_{\mathrm{FW}}$, $d\in(0,\infty)$, whether or not the system clusters does not depend on the resampling rate $d$.

		\section{Main results: $N<\infty$, identification of clustering regime}
		\label{s.mainfinite}
		
		In this section we identify the \emph{clustering regime}, i.e., the range of \emph{parameters} for which the clustering criterion in Proposition~\ref{T.dichcrit} is met. In \cite[Section 3.2, Theorem 3.3]{GdHOpr1} we derived a necessary and sufficient condition for when clustering prevails, for any geometric space given by a countable Abelian group endowed with the discrete topology. Recall $\chi$ in \eqref{chidef}, $\rho$ in \eqref{rhodef} and $\tau$ in \eqref{deftau},. From \eqref{sigtau} it follows that
		\begin{equation}
			\E[\tau]=\sum_{m\in\N_0}\frac{K_m}{\chi}=\frac{\rho}{\chi},
		\end{equation}
		and hence the mean wake-up time is finite if $\rho<\infty$ and infinite if $\rho=\infty$. In Section~\ref{ss.finsb} we look at $\rho<\infty$ and in Section~\ref{ss.infsb} at $\rho=\infty$. In Section~\ref{ss.clusreg} we summarise our findings and identify the clustering regime. 
		
		\subsection{Finite mean wake-up time}
		\label{ss.finsb}
		
		Suppose that the system evolving according to \eqref{moSDE} has a translation invariant initial measure $\mu(0)$ with density $\theta \in (0,1)$. Then \cite[Theorem 3.3]{GdHOpr1} says that for $\rho<\infty$ clustering occurs if and only if
		\begin{equation}
			\label{cluscritnoseed-a}
			\int_1^\infty  \d t\,a^{\Omega_N}_t(0,0) = \infty.
		\end{equation} 
		It is known that \eqref{cluscritnoseed-a} holds for the hierarchical migration defined in \eqref{739} if and only if \cite[Section 3]{DGW05}
		\begin{equation}
			\label{cluscritnoseed-cals}
			\sum_{k\in\N_0} \frac{1}{c_k} = \infty.
		\end{equation}
		Hence, for $\rho<\infty$, the clustering criterion depends on the migration kernel \emph{only} and the seed-bank has no effect.
		
		\begin{figure}[htbp]
			\vspace{-.3cm}
			\begin{center}
				\setlength{\unitlength}{.5cm}
				\begin{picture}(20,2)(0,0)
					\put(0,0){\line(20,0){20}}
					\put(0,-1.98){\line(20,0){20}}
					\put(0,0){\circle*{.25}}
					\put(2,0){\circle*{.25}}
					\put(5,0){\circle*{.25}}
					\put(7,0){\circle*{.25}}
					\put(11,0){\circle*{.25}}
					\put(13,0){\circle*{.25}}
					\put(19,0){\circle*{.25}}
					\put(0,-2){\circle*{.25}}
					\put(3,-2){\circle*{.25}}
					\put(4,-2){\circle*{.25}}
					\put(8,-2){\circle*{.25}}
					\put(12,-2){\circle*{.25}}
					\put(16,-2){\circle*{.25}}
					\put(18,-2){\circle*{.25}}
					\put(.7,.4){$\sigma_1$}
					\put(3.3,.4){$\tau_1$}
					\put(5.7,.4){$\sigma_2$}
					\put(8.6,.4){$\tau_2$}
					\put(11.7,.4){$\sigma_3$}
					\put(15,.4){$\tau_3$}
					\put(1.2,-1.6){$\sigma'_1$}
					\put(3.2,-1.6){$\tau'_1$}
					\put(5.7,-1.6){$\sigma'_2$}
					\put(9.7,-1.6){$\tau'_2$}
					\put(13.7,-1.6){$\sigma'_3$}
					\put(16.7,-1.6){$\tau'_3$}
					\put(20,-1){$\ldots$}
					\qbezier[15](0,-2)(0,-1)(0,0)
					\qbezier[15](2,-2)(2,-1)(2,0)
					\qbezier[15](5,-2)(5,-1)(5,0)
					\qbezier[15](7,-2)(7,-1)(7,0)
					\qbezier[15](12,-2)(12,-1)(12,0)
					\qbezier[15](13,-2)(13,-1)(13,0)
					\qbezier[15](19,-2)(19,-1)(19,0)
				\end{picture}
				\vspace{0.8cm}
			\end{center}
			\caption{\small Successive periods during which the two random walks are active and dormant (recall \eqref{sigtau} and Fig.~\ref{fig:per}). The time lapses between successive pairs of dotted lines represent periods of \emph{joint activity}.}
			\label{fig:periods}
		\end{figure}
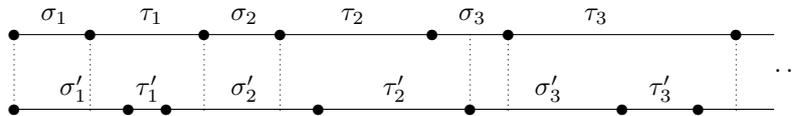 
		
		In view of Proposition \ref{T.dichcrit}, if $g=dg_{\mathrm{FW}}$, then clustering prevails if and only if two lineages in the dual coalesce with probability 1. Recall that two lineages in the dual can only coalesce when they are at the same site and are both active. Since the rate of coalescence is $d\in(0,\infty)$, each time this happens the two lineages have a positive probability to coalesce before moving or becoming dormant. Therefore, clustering prevails if and only if two lineages meet infinitely often while being active. This happens exactly when \eqref{cluscritnoseed-a} holds. The fact that the seed-bank plays no role can be seen from the dual. Each lineage in the dual moves according to the transition kernel $b(\cdot,\cdot)$ (recall \eqref{mrw}). Looking at the renewal process induced by the dual process (recall Remark \ref{wakeup} and Fig.~\ref{fig:periods}), we see that for $\rho<\infty$ the probability that a lineage in the dual is active at time $t$ is approximately $\frac{1}{1+\rho}$ for large $t$. The total activity time of a lineage up to time $t$ is therefore approximately $\frac{1}{1+\rho} t$ for large $t$. Hence the total time the two lineages in the dual are at the same site and are both active is approximately
		\begin{equation}
			\label{mo1}
			\int_1^\infty \d t\left(\frac{1}{1+\rho}\right)^2 a_{2\frac{1}{1+\rho} t}(0,0),
		\end{equation}
		By Polya's argument, if the integral in \eqref{mo1} is infinite, then two lineages in the dual meet infinitely often while being active. After a variable transformation, \eqref{mo1} becomes the integral in \eqref{cluscritnoseed-a} up to a constant. (For a formal proof of the criterion in \eqref{cluscritnoseed-a}, we refer to \cite{GdHOpr1}.) By the above argument, we can think of the integral in \eqref{cluscritnoseed-a} as the \emph{total hazard of coalescence} of two dual lineages. To get the result for general $g\in\CG$ we must invoke the duality comparison arguments mentioned in Section \ref{ss.comparisonarg}.
		
		In terms of the degree of the random walk (recall Remark~\ref{rem:degree}), \eqref{cluscritnoseed-a} corresponds to hierarchical migration with degree $0^-$. The same criterion as in \eqref{cluscritnoseed-a} was found in \cite{FG96} for interacting Fisher-Wright diffusions on the hierarchical lattice without seed-bank ($\rho=0$). Hence we conclude that for $\rho<\infty$ the seed-bank does not affect the dichotomy.
		
		\subsection{Infinite mean wake-up time} 
		\label{ss.infsb}
		
		If $\rho=\infty$, then the seed-bank does affect the dichotomy. To apply the criterion in \cite[Theorem 3.3]{GdHOpr1}, we assume that the system evolving according to \eqref{moSDE} has a translation-invariant initial measure $\mu(0)$ with density $\theta \in (0,1)$ that is colour regular. 
		
		The criterion for clustering that was derived in \cite{GdHOpr1} for $\rho=\infty$ applies to wake-up times $\tau$ (recall \eqref{deftau}) of the form    
		\begin{equation}
			\label{mod}
			\frac{\P(\tau \in \d t)}{\d t}  \sim \varphi(t)\, t^{-(1+\gamma)}, \quad t \to\infty, \quad \gamma\in (0,1],  
		\end{equation}
		with $\gls{phi}$ \emph{slowly varying at infinity}. Define
		\begin{equation}
			\label{hatphidefplus}
			\hat\varphi(t) = \left\{\begin{array}{ll}
				\varphi(t), &\gamma \in (0,1),\\[0.2cm]
				\E\left[\tau\wedge t\right] &\gamma=1. 
			\end{array}
			\right.
		\end{equation} 
		As shown in \cite[Section 1.3]{BGT87}, every slowly varying function $\varphi$ may be assumed to be infinitely differentiable and to be represented by the integral  
		\begin{equation}
			\label{hatphirepr}
			\varphi(t) = \exp\left[\int_{(\cdot)}^t \frac{\d u}{u}\,\psi(u)\right]
		\end{equation}
		for some $\psi\colon [0,\infty) \to \R$ such that $\lim_{u\to\infty} |\psi(u)| = 0$. From \eqref{mod} we see that $\hat{\varphi}(t)$ is also slowly varying. If we assume that $|\psi(u)| \leq C/\log u$ for some $C<\infty$, then the system clusters if and only if (see \cite[Section 3.2]{GdHOpr1})
		\begin{equation}
			\label{cluscritseed-b}
			\int_{(\cdot)}^\infty  \d t\,\hat\varphi(t)^{-1/\gamma}\,
			t^{-(1-\gamma)/\gamma}\,\hat{a}_t(0,0)=\infty.
		\end{equation} 
		Note that $\gls{gamma}$ in \eqref{mod} is the \emph{tail exponent} of the typical wake-up time $\tau$ (recall Remark \ref{wakeup}) and depends on the sequences $\underline{e},\underline{K}$ in \eqref{defKem} governing the exchange with the seed-bank. If $g=dg_{\mathrm{FW}}$, then in view of Theorem~\ref{T.dichcrit} the criterion in \eqref{cluscritseed-b} determines whether two lineages in the dual coalesce with probability 1. 
		
		In Section~\ref{s.clusreg} we will use the renewal process induced by the dual (recall Remark~\ref{wakeup}) to show that \eqref{cluscritseed-b} indeed gives the \emph{total hazard of coalescence} of two dual lineages. Therefore the integral in \eqref{cluscritseed-b} is the counterpart of \eqref{cluscritnoseed-a}. The rate of coalescence again does not affect the dichotomy: in \cite{GdHOpr1} a duality comparison argument was used to show that \eqref{cluscritseed-b} gives the clustering criterion also for $g\in\CG$ with $g\neq dg_{\mathrm{FW}}$. The effect of the seed-bank on the dichotomy is embodied by the term $\hat\varphi(t)^{-1/\gamma}\,t^{-(1-\gamma)/\gamma}$ in \eqref{cluscritseed-b}.  The criterion in \eqref{cluscritseed-b} shows that there is a competition between migration and exchange with the seed-bank. 
		
		For the special case where $\hat\varphi(t) \asymp 1$, the criterion in \eqref{cluscritseed-b} says that (recall Remark~\ref{rem:degree})
		\begin{equation}
			\label{cluscrit}
			\text{clustering} \quad \Longleftrightarrow \quad  \text{ either } \delta^- \leq -\frac{1-\gamma}{\gamma} \text{ or } 
			\delta^+ < -\frac{1-\gamma}{\gamma}.
		\end{equation} 
		Condition \eqref{cluscritseed-b} implies that for $\gamma \in (0,\tfrac12)$ no clustering is possible: the typical wake-up time has such a heavy tail that with a positive probability two dual lineages do not meet, irrespective of the migration.
		
		\begin{definition}
			{\rm In what follows we will focus on the following two specific parameter regimes:
				\begin{enumerate}
					\item[$\bullet$] 
					\emph{Asymptotically polynomial}, i.e.,
					\begin{equation}
						\label{regvar}
						K_k \sim Ak^{-\alpha}, \quad e_k \sim Bk^{-\beta}, \quad c_k \sim Fk^{-\phi},
						\quad k \to \infty,\, A,B,F \in (0,\infty),\,\alpha,\beta,\phi \in \R.
					\end{equation}
					\item[$\bullet$] 
					\emph{Pure exponential}, i.e.,
					\begin{equation}
						\label{pureexp}
						K_k = K^k, \quad e_k = e^k, \quad c_k = c^k,
						\quad k \in\N_0,\, K,e,c \in (0,\infty).
					\end{equation}
				\end{enumerate}
				Note that both \eqref{740} and \eqref{740alt} are satisfied for $N\to\infty$. Also note that an infinite seed-bank corresponds to $\alpha \in (-\infty,1]$, respectively, $K \in [1,\infty)$.} \hfill $\Box$
		\end{definition} 
		
		The scaling of the wake-up time and the migration kernel in these parameter regimes are as follows.
		
		\begin{theorem}{{\bf [Infinite seed-bank: Scaling of wake-up time and migration kernel]}}
			\label{T.scalcoeff}
			Suppose that $\rho=\infty$. Then
			\begin{itemize}
				\item[(a)] Subject to \eqref{regvar}, 
				\begin{equation}
					\label{mo3}
					\begin{aligned}
						&\gamma = 1, \quad \varphi(t) \asymp (\log t)^{-\alpha}, \quad \hat\varphi(t) \asymp 
						\left\{\begin{array}{ll}
							(\log t)^{1-\alpha}, &\alpha \in (-\infty,1),\\[0.2cm]
							\log\log t, &\alpha=1,
						\end{array}
						\right.
						\\
						&a_t^{\Omega_N}(0,0) \asymp t^{-1}\log^\phi t.
					\end{aligned}
				\end{equation}
				\item[(b)] Subject to \eqref{pureexp}, 
				\begin{equation}\label{mo2}
					\begin{aligned}
						&\gamma = \gamma_{N,K,e} = \frac{\log(N/Ke)}{\log(N/e)}, \quad \varphi(t) \asymp 1, \quad \hat\varphi(t) \asymp 
						\left\{\begin{array}{ll}
							1, &K \in (1,\infty),\\[0.2cm]
							\log t, &K = 1,
						\end{array}
						\right. 
						\\
						&a_t^{\Omega_N}(0,0) \asymp t^{-1-\delta_{N,c}},
					\end{aligned}
				\end{equation} 
				where
				\begin{equation}
					\label{deltaNc}
					\delta_{N,c} = \frac{\log c}{\log (N/c)}.
				\end{equation}
			\end{itemize}
		\end{theorem}
		
		\noindent
		Theorem \ref{T.scalcoeff} will be proved in Section \ref{s.clusreg}. 
		
		Note that, by \eqref{mo2}, $\gamma_{N,K,e} =1$ for all $N$ when $K=1$, while $\gamma_{N,K,e}<1$ for all $N$ when $K>1$, but with $\gamma_{N,K,e} \uparrow 1$ as $N\to\infty$. Also note that, subject to \eqref{regvar}, \eqref{mo3} says that the degree of the random walk is $0^-$ for $\phi \geq -1$ and $0^+$ for $\phi < -1$, while subject to \eqref{pureexp}, by \eqref{mo2}, the degree of the random walk is $\delta_{N,c}^-$, which is $0$ for all $N$ when $c=1$, and tends to $0$ as $N\to\infty$ from above when $c>1$ and from below as $c<1$. Thus, both \eqref{regvar} and \eqref{pureexp} with $N\to\infty$ correspond to a \emph{critically recurrent migration} and a \emph{critically infinite seed-bank}.
		
		\subsection{Clustering regime}
		\label{ss.clusreg}
		
		Summarising the above discussion, we can now identify the clustering regime for both finite and infinite seed-banks.
		
		\begin{theorem}{{\bf [Clustering regime]}}
			\label{T.cluscritreg}
			\begin{enumerate}
				\item[\rm (1)] If $\rho<\infty$, then clustering prevails if and only if
				\begin{equation}
					\label{cluscritnoseed-c}
					\sum_{k\in\N_0} \frac{1}{c_k} = \infty.
				\end{equation}
				\item[\rm (2)] If $\rho=\infty$, then clustering prevails for $N$ large enough   
				\begin{itemize}
					\item[(a)] Subject to \eqref{regvar}  if and only if
					\begin{equation}
						\label{clusregregvar}
						- \phi \leq \alpha \leq 1.
					\end{equation}
					\item[(b)] Subject to \eqref{pureexp} if and only if
					\begin{equation}
						\label{clusregpureexp}
						Kc \leq 1 \leq K.
					\end{equation}
				\end{itemize}
			\end{enumerate}
		\end{theorem}
		
		\noindent
		Also Theorem \ref{T.cluscritreg} will be proved in Section \ref{s.clusreg}. 
		
		Note that for $\rho<\infty$ the clustering regime follows by combining \eqref{cluscritnoseed-a} and \eqref{cluscritnoseed-cals}, while for $\rho=\infty$ the clustering regime follows by substituting into \eqref{cluscritseed-b} the scaling of the wake-up times and the migration kernel stated in Theorem \ref{T.scalcoeff}. 
		
		\begin{remark}
			{\rm Note that subject to \eqref{regvar}, respectively, \eqref{pureexp}, $\rho<\infty$ implies that $\alpha >1$, respectively, $K<1$, and so the clustering regime is $-\phi \leq 1$, respectively, $c \leq 1$ (recall \eqref{cluscritnoseed-c}), which are less stringent than \eqref{clusregregvar}, respectively, \eqref{clusregpureexp}.} \hfill$\blacksquare$  
		\end{remark}

		\section{Main results: $N\to\infty$, renormalisation and multi-scale limit}
		\label{s.intromultscallim}
		
		This section contains our multi-scale hierarchical limit theorems. The multi-scale hierarchical limit theorems analyse the evolution of the block averages defined in Definition~ \ref{defblockav}. In Section~\ref{MeyerZheng} we recall a path topology referred to as the \emph{Meyer-Zheng topology}, which we will need in part of our multi-scale hierarchical limit theorems. In Section~\ref{ss.ing} we present the conceptual ingredients needed for our theorems. In Section~\ref{ss.ht} we state two versions of the hierarchical multi-scale limit (Theorems~\ref{T.multiscalehiereff} and \ref{T.multiscalehier}), and comment on how they are related to each other. In Section \ref{ss.hml} we explain how they arise from a heuristic analysis of the SSDE in \eqref{moSDE}.

		\subsection{Intermezzo: Meyer-Zheng topology}
		\label{MeyerZheng}
		
		Recall the block averages defined in \eqref{defblockav} and their evolution equations in \eqref{rblockavxzmulti}. In the limit as $N\to\infty$, some of the pre-factors in \eqref{rblockavxzmulti} diverge as a result of the speeding up of time. This makes the processes increasingly more volatile: paths becomes rougher and rougher during rarer and rarer times. Therefore we cannot work with weak convergence on path space $C([0,\infty),E)$ w.r.t.\ the topology generated by the \emph{sup-norm} on compacts, or on path space $\gls{cadlag}$ w.r.t.\ the Skorohod metric on compacts. Rather we must follow the methodology used in \cite[pp.\ 792--794]{DGsel14} and employ the so-called \emph{Meyer-Zheng topology} on \emph{pseudopaths}, (see \cite{MZ84}), which is based on the following idea. Consider functions $f\colon\,[0,\infty) \to E$, with $(E,d)$ a Polish space, and sequences of functions $(f_n)_{n \in \N}$ that are c\`adl\`ag paths, i.e., functions in the Skorohod space $D([0,\infty),E)$. Then  $(f_n)_{n \in \N}$ converges to $f$ in the Meyer-Zheng topology if and only if 
		\begin{equation}
			\label{e1514}
			\lim_{n\to\infty} \int_a^b \d t\,\big[1 \wedge d(f(t),f_n(t))\big] = 0 \qquad \forall\,\,0 \leq a < b< \infty.
		\end{equation}
		However, the topology induced by the metric in \eqref{e1514} does not turn $D([0,\infty),E)$ into a \emph{closed} space (while in order to apply the classical theory of weak convergence of probability laws on path space we need the path space to be Polish). 
		
		To turn the idea from \eqref{e1514} into a manageable topology, we proceed by defining a \emph{space of pseudopaths} equipped with the \emph{Meyer-Zheng topology}. If $(E,d)$ is a Polish space and $s \mapsto v(s)$ is a measurable map from $[0,\infty)$ to $E$, then the pseudopath  $\gls{pseudo}$   is the \emph{probability measure} $\rho$ on $[0,\infty) \times E$, defined by
		\begin{equation}
			\label{e1519}
			\rho((a,b) \times B) = \int_a^b \d s\,\e^{-s}\,1_B(v(s)), \qquad B \in \CB(E),
		\end{equation}
		where $\CB(E)$ is the Borel $\sigma$-algebra over $E$.
		Hence $\psi_v$  is the image measure of $\e^{-t} \d t$ under the mapping $t\to (t,v(t))$. In other words, we consider the \emph{weighted occupation measure} of the path in $E$ in order to describe paths that are \emph{regular representatives} in the space of functions once we take into account \eqref{e1514}. Note that a piece-wise constant c\`adl\`ag path is uniquely determined by its occupation measure. So is a continuous path with continuous local times. The space of all pseudopaths is denoted by $\gls{pseudospace}$.
		
		Since pseudopaths are measures on $[0,\infty]\times E$, convergence of pseudopaths is defined as \emph{weak convergence} of \emph{probability measures on $[0,\infty]\times E$}. A sequence $(v_n)_{n\in\N}$ of measurable maps from $[0,\infty)\times E$ is said to converge in the Meyer-Zheng topology to a measurable map $v$ if $\lim_{n\to \infty} \psi_{v_n}=\psi_v$, i.e., $\lim_{n\to \infty} \psi_{v_n}f=\psi_vf$ for all continuous bounded functions $f$ on $[0,\infty] \times E$.	
		
		\begin{remark}{\bf [Pseudopaths]}
			\label{r.1556}
			{\rm The space $\Psi$ of pseudopaths endowed with the Meyer-Zheng topology is Polish, but the space $D([0,\infty),E)$ endowed with the Meyer-Zheng topology is not Polish (see \cite[p.\ 372]{MZ84}).}
			\hfill $\blacksquare$
		\end{remark}
		
		In what follows, each time convergence holds in the Meyer-Zheng topology we will say so explicitly. If no topology is mentioned, then we mean convergence in $\CC_b([0,\infty),[0,1])$. In Appendix~\ref{apb} we collect some basic facts about the Meyer-Zheng topology taken from \cite{MZ84} and \cite{K91}.

		\subsection{Main ingredients for the hierarchical multi-scale limit}
		\label{ss.ing}
		
		Recall the definition of $\theta_x$ and $\theta_{y_m}$ in \eqref{initialvalue}. Define
		\begin{equation}
			\label{deftheta}
			\gls{varthetak} = \frac{\theta_x + \sum_{m=0}^k K_m \theta_{y_m}}{1+\sum_{m=0}^k K_m}, \qquad k \in \N_0.
		\end{equation}
		For $\rho<\infty$, and for $\rho=\infty$ under the additional assumption of colour regularity (recall Proposition~\ref{T.dichcrit}), we have
		\begin{equation}
			\label{thetantotheta}
			\lim_{k \to \infty}\vartheta_{k}=\theta \quad \text{ for some } \theta\in[0,1].
		\end{equation}
		Define the \emph{slowing-down constants} ($E_0=1$)
		\begin{equation}
			\label{Ekdef}
			\gls{Ek} = \frac{1}{1+\sum_{m=0}^{k-1} K_m}, \qquad k \in \N_0.
		\end{equation}
		For $l\in\N_0$, let
		\begin{equation}
			\big(\theta,(y_{m,l})_{m\in\N_0}\big)
		\end{equation}
		be a sequence of random variables taking values in $[0,1]$, and let
		\begin{equation}
			\label{z2a}
			\gls{zl} = \Big(x_l(t),(y_{m,l}(t))_{m \in \N_0}\Big)_{t\geq 0}
		\end{equation}
		be the \emph{full process} on level $l$ evolving according to
		\begin{equation}
			\label{z11a}
			\begin{aligned}
				&\d x_l(t) =  E_l\Bigg[c_l [\theta - x_l(t)]\, \d t 
				+ \sqrt{(\CF^{(l)} g)(x_l(t))}\, \d w (t) +  K_l e_l\, [y_{l,l}(t)-x_{l}(t)]\,\d t\Bigg],\\
				&\begin{array}{lll} 
					&y_{m,l}(t) = x_l(t), &0\leq m<l,\\
					&\d y_{l,l}(t) = e_l\,[x_l(t)-y_{l,l}(t)]\, \d t, &m=l,\\
					& y_{m,l}(t) =y_{m,l}, &m > l.
				\end{array}
			\end{aligned}
		\end{equation}
		where $\CF^{(l)}g$ is an element of $\CG$, (recall \eqref{setG}), that will be defined in \eqref{frenormit} below. By \cite{YW71} the above SSDE has a unique solution for every initial measure. For $l \in\N_0$, let  
		\begin{equation}
			\label{92}
			\gls{zlf}=\left(x^\eff_l(t),y^\eff_{l,l}(t)\right)_{t\geq 0}
		\end{equation}
		be the \emph{effective process} evolving according to 
		\begin{equation}
			\label{927}
			\begin{aligned}
				&\d x^{\eff}_l(t) = E_l\left[c_l\,[\theta - x^{\eff}_l(t)]\, \d t + \sqrt{(\CF^{(l)}g)(x^{\eff}_l(t))}\, \d w (t) 
				+ K_le_l\,[y^{\eff}_{l,l}(t)-x^{\eff}_l(t)]\,\d t\right],\\
				&\d y^{\eff}_{l,l}(t) = e_l\, [x^{\eff}_l(t)-y^{\eff}_{l,l}(t)]\, \d t.
			\end{aligned}
		\end{equation}
		Comparing \eqref{z11a} with \eqref{927}, we see that the effective process looks at the non-trivial components of the full process. 
		
		Apart from \eqref{z2a} and \eqref{92}, we need the following list of \emph{ingredients} to formally state the multi-scale limit:
		\begin{enumerate}
			\item 
			For $l\in\N_0$ and $t>0$, define the \emph{estimators} for the finite system by
			\begin{equation}
				\label{tam6alt}
				\begin{aligned}
					\gls{thetal}
					&=\frac{1}{N^l}\sum_{\xi\in B_l}\frac{ x^{\Omega_N}_\xi(t)
						+\sum_{m=0}^{l-1}K_m y^{\Omega_N}_{\xi,m}(t)}{1+\sum_{m=0}^{l-1}K_m},\\
					\gls{thetalx}
					&=\frac{1}{N^l}\sum_{\xi\in B_l} x^{\Omega_N}_{\xi}(t),\\
					\gls{thetaly}
					&=\frac{1}{N^l}\sum_{\xi\in B_l} y^{\Omega_N}_{\xi,m}(t),\qquad m\in\N_0,
				\end{aligned}
			\end{equation}
			and put
			\begin{equation}\label{412}
				\begin{aligned}
					\gls{thetalfull}
					&=\big(\Theta^{(l) ,\Omega_N}_x(t),\big(\Theta^{(l) ,\Omega_N}_{y_m}(t)\big)_{m\in\N_0}\big),\\
					\gls{thetaleff}
					&=\big(\bar{\Theta}^{(l) ,\Omega_N}(t),\Theta^{(l) ,\Omega_N}_{y_{l}}(t)\big).
				\end{aligned}
			\end{equation}
			We call $(\boldsymbol{\Theta}^{(l),\Omega_N}(t))_{t\geq 0}$ the \emph{full estimator process} and $(\boldsymbol{\Theta}^{\eff,(l),\Omega_N}(t))_{t \geq 0}$ the \emph{effective process}. Note that $\Theta_x^{(l),\Omega_N}(t)$ is the empirical average of the active components in the $l$-block, while $\Theta_{y_m}^{(l),\Omega_N}(t)$, is the empirical average of the $m$-dormant components in the $l$-block, both without scaling of time. Note that $\Theta^{(l),\Omega_N}_x(N^lt)=x^{\Omega_N}_l(t)$. The \emph{level-$l$ estimator} $\bar{\Theta}^{(l) ,\Omega_N}(t)$ will play an important role in our analysis. Using \eqref{rblockavxzmulti}, we can derive the evolution equations of $\bar{\Theta}^{(l),\Omega_N}(N^lt)$. We see that in the evolution of $\bar{\Theta}^{(l),\Omega_N}(N^lt)$ no rates appear that tend to infinity as $N\to\infty$. However, in the evolution of $\Theta^{(l) ,\Omega_N}_x(N^lt)$ and $\Theta^{(l) ,\Omega_N}_{y_m}(N^lt)$ for $m<l$ the rates describing the interaction between the active and the dormant population tend to infinity as $N\to\infty$. 
			\item  
			For $l\in\N_0$, consider \emph{time scales} $N^lt_l$ such that 
			\begin{equation}
				\CL\big[\bar{\Theta}^{(l),\Omega_N} (N^{l}t_l-L(N)N^{l-1})-\bar\Theta^{(l),\Omega_N}( N^{l}t_l)\big]=\delta_0
			\end{equation}
			for all sequences $\gls{LN}$ satisfying $\lim_{N\to\infty} L(N)=\infty$ and $\lim_{N\to\infty} L(N)/N=0$, but not for $L(N)=N$. In words, $N^lt_l$ is the time scale on which $(\bar{\Theta}^{(l),\Omega_N}(N^lt_l))_{t_l>0}$ is no longer a fixed process. 
			\item 
			For $l \in \N_0$ the \emph{invariant measure} for the limiting evolution of the full process on level $l$ in \eqref{z11a} is denoted by
			\begin{equation}
				\label{singcoleq322alt}
				\gls{Gammal}, \qquad y_l = (y_{m,l})_{m\in\N_0}.
			\end{equation}
			(The existence of and convergence to this equilibrium will be proved in Section~\ref{ss.flmfss}.) Note that $\Gamma_{(\theta,y_l)}^{(l)}$ depends on choice of the rates  $E_l,c_l,K_l,e_l$ in \eqref{z11a}.   The \emph{invariant measure} of the limiting evolution for the effective $l$-block process in \eqref{927} is denoted by
			\begin{align}
				\label{inveff}
				\gls{Gammaleff}.
			\end{align}
			Also $\Gamma_{\theta}^{\eff,(l)}$ depends on the choice of the rates $E_l,c_l,K_l,e_l$.
			\item 
			For $l\in\N_0$, let $\gls{Fl}$ denote the \emph{renormalisation transformation} acting on $\CG$ defined by
			\begin{equation}
				\label{renor}
				(\CF^{E_l,c_l,K_l,e_l}g)(\theta) = \int_{[0,1]^2} g(x)\,\Gamma_{\theta}^{\eff,(l)}(\d x), \qquad \theta \in [0,1].
			\end{equation}
			(In Section~\ref{ss.pabstracts} we show that $\CF g\in \CG$.) For $k\in\N_0$, define the \emph{iterate} of the renormalisation transformation as the composition
			\begin{equation}
				\label{frenormit}
				\gls{Flit} = \CF^{E_{k-1},c_{k-1},K_{k-1},e_{k-1}} \circ \cdots \circ \CF^{E_0,c_0,K_0,e_0}.
			\end{equation}
			\item 
			For $k\in\N_0$, define the \emph{interaction chain} \cite{DGV95}
			\begin{equation}
				\label{fintchain}
				\gls{intchain}
			\end{equation}
			as the \emph{time-inhomogeneous} Markov chain on $[0,1]\times[0,1]^{\N_0}$ with initial state 
			\begin{equation}\label{intch}
				M^k_{-(k+1)} =(\vartheta_k,\overbrace{\vartheta_k,\cdots, \vartheta_k}^{k+1\text{ times }},
				\theta_{y_{k+1}},\theta_{y_{k+2}},\cdots)
			\end{equation} 
			that evolves from time $-(l+1)$ to time $-l$ according to the transition kernel $\gls{Q}$ on $[0,1]\times[0,1]^{\N_0}$ given by
			\begin{equation}\label{ker}
				Q^{[l]}(u,\d v) =  \Gamma_{u}^{(l)}(\d v).
			\end{equation}
			(See Fig.~\ref{fig:int}.) For $k\in\N_0$, define the \emph{effective interaction chain}
			\begin{equation}
				\label{fintchaineff}
				\gls{intchaineff}
			\end{equation}
			as the \emph{time-inhomogeneous} Markov chain on $[0,1]\times[0,1]$ with initial state
			\begin{equation}
				M^{\eff,k}_{-(k+1)} =(\vartheta_k,\theta_{y_{k+1}})
			\end{equation} 
			that evolves  from time $-(l+1)$ to time $-l$ according to the transition kernel $\gls{Qeff}$ on $[0,1]\times[0,1]$ given by 
			\begin{equation}\label{kereff}
				Q^{\eff,[l]}(u,\d v) =  \Gamma^{\eff,(l)}_{u_x}(\d v),
			\end{equation}
			where $u_x$ denotes the first component of $u=(u_x,u_y)$.(See Fig.~\ref{fig:effint}.) We denote the components of $\left(M^{\eff,k}_{-l}\right)$ by 
			\begin{equation}
				M^{\eff,k}_{-l} =\big(M^{\eff,k}_{-l,x},M^{\eff,k}_{-l,y}\big).
			\end{equation}
		\end{enumerate}
		
		\subsection{Hierarchical multi-scale limit theorems}
		\label{ss.ht}
		
		First we present and discuss the scaling of the effective process. Afterwards we do the same for the full process.  
		
		\begin{figure}[htbp]
			\begin{center}
				\begin{tikzpicture}
\foreach \x in {9,11}
\draw [ultra thick,->] (0,\x) to (0,\x-1);
\node[centered] at (0,11.5) {$M^{k,\eff}_{-(k+1)}=(\textcolor{red}{\vartheta_k},\theta_{k+1})$};
\node[centered] at (0,9.5) {$M^{k,\eff}_{-k}=(\textcolor{green}{x_k(t_k)},y_{k,k}(t_k))$};
\node[centered] at (0,7.5) {$M^{k,\eff}_{-(k-1)}=\left(\textcolor{purple}{x_{k-1}(t_{k-1})},y_{k-1,k-1}(t_{k-1})\right)$};
\node[right] at (0.5,10.5){$Q^{[k]}\left(M^{k,\eff}_{-(k+1)},(du, dv)\right)=\Gamma^\eff_{\textcolor{red}{\vartheta_k}}(d u,d v)$};
\node[right] at (0.5,8.5){$Q^{[k-1]}\left(M^{k,\eff}_{-k},(du, dv)\right)=\Gamma^\eff_{\textcolor{green}{x_k(t_k)}}(d u,d v)$};
\node[right] at (0.5,6.75){$Q^{[k-2]}\left(M^{k,\eff}_{-(k-1)},(du, dv)\right)=\Gamma^\eff_{\textcolor{purple}{x_{k-1}(t_{k-1})}}(d u,d v)$};
\node[right] at (0.5,4){$Q^{[0]}\left(M^{k,\eff}_{-1},(du, dv)\right)=\Gamma^\eff_{\textcolor{blue}{x_1(t_1)}}(d u,d v)$};
\foreach \x in {6.15,6,6.30}
\draw[fill] (0,\x) circle [radius=0.05];
\foreach \x in {7,5.8}
\draw [ultra thick,->] (0,\x) to (0,\x-.5);
\node[centered] at (0,5) {$M^{k,\eff}_{-1}=\left(\textcolor{blue}{x_1(t_1)}, y_{1,1}(t_1)\right)$};
\draw [ultra thick,->] (0,4.5) to (0,3.5);
\node[centered] at (0,3) {$M^{k,\eff}_{0}=\left(x_0(t_0), y_{0,0}(t_0)\right)$};
\end{tikzpicture}
			\end{center}
			\caption{Effective interaction chain.}
			\label{fig:effint}
		\end{figure}
		
		\subsubsection{Effective process}
		
		We present one of our main theorems, the hierarchical mean-field limit for the effective process. We will use the process and notation introduced in Section~\ref{ss.ing}. 
		
		\begin{theorem}{{\bf [Hierarchical mean-field: the effective process]}}
			\label{T.multiscalehiereff}  
			Suppose that the initial state of the hierarchical system is given by \eqref{initialvalue}. Let $L(N)$ be such that $\lim_{N\to\infty} L(N)=\infty$ and $\lim_{N\to\infty} L(N)/N=0$. For $k \in \N$ and $t_k,\ldots,t_0 \in (0,\infty)$, set $\bar{t}=N^k L(N) + \sum_{n=0}^k N^n t_n$.
			\begin{itemize}
				\item[(a)]
				For $k \in \N$,
				\begin{equation}
					\lim_{N\to\infty} \CL\left[\big(\boldsymbol{\Theta}^{\eff,(l),\Omega_N}(\bar{t}\,)\big)_{l=k+1,k,\ldots,0}\right]
					= \CL\left[(M^{\eff,k}_{-l})_{-l = -(k+1),-k,\ldots,0}\right].
				\end{equation}
				\item[(b)]
				For $k\in\N$, 
				\begin{equation}
					\label{multiconv}
					\begin{aligned}
						&l>k\colon \lim\limits_{N\to\infty} 
						\CL\left[\left(\boldsymbol{\Theta}^{\eff,(l),\Omega_N}(\bar{t}+N^kt)\right)_{t > 0}\right] 
						= \delta_{(\vt_l, \theta_{y_l})},\\[0.3cm]
						&l=k\colon \lim\limits_{N\to\infty} 
						\CL\left[\left(\boldsymbol{\Theta}^{\eff,(l),\Omega_N}(\bar{t}+N^lt)\right)_{t > 0}\right]
						= \CL\left[\left(z^\eff_{k,M^{\eff,k}_{-(k+1),x}}(t)\right)_{t > 0}\right],\\[0.3cm]
						&l<k\colon \lim\limits_{N\to\infty} 
						\CL\left[\left(\boldsymbol{\Theta}^{\eff,(l),\Omega_N}(\bar{t}+N^lt)\right)_{t > 0}\right] 
						= \CL\left[\left(z^\eff_{l,M^{\eff,k}_{-(l+1),x}}(t)\right)_{t > 0}\right],
					\end{aligned}
				\end{equation} 
				where the initial laws of the limiting processes are given by (see Fig.~\ref{fig:effint})
				\begin{equation}
					\begin{aligned}
						&\CL\left[z^\eff_{k,M^k_{-(k+1),x}}(0)\right]=\Gamma^{\eff,(k)}_{M^k_{-(k+1),x}},\\
						&\CL\left[z^\eff_{k,M^k_{-(k+1),x}}(0)\right]=\Gamma^{\eff,(k)}_{M^k_{-(l+1),x}},\\
						&\Gamma^{\eff,(l)}_{M^k_{-(l+1)}}=\int_{[0,1]^2}\cdots\int_{[0,1]^2}\Gamma^{\eff,(k)}_{M^k_{-(k+1)}}(\d u_k)\cdots\Gamma^{\eff,(l+1)}_{u_{l+2}}(\d u_{l+1}) \Gamma^{\eff,(l)}_{u_{l+1}}
					\end{aligned}
				\end{equation}
			\end{itemize}
		\end{theorem}
		
		Theorem~\ref{T.multiscalehiereff} can be interpreted as follows. The statement in (a) shows that if we look at the effective process on multiple space-time scales simultaneously, then the joint distribution of the different block averages is the law of the two-dimensional interaction chain defined in \eqref{fintchaineff} and depicted in Fig.~\ref{fig:effint}. Note that the process $(\boldsymbol{\Theta}^{\eff,(l),\Omega_N}(\bar{t}+N^lt))_{t > 0}$ has at each level a different colour seed-bank average as second component, which is called \emph{the effective seed-bank}. The statement in (b) describes the law of the path on different time scales. 
		\begin{itemize}
			\item 
			On time scale $\bar{t}+N^kt$ the $l$-block averages with $l>k$ are not moving, i.e., $(\boldsymbol{\Theta}^{\eff,(l),\Omega_N}(\bar{t}+N^kt))_{t>0}$ converges to the constant process taking the value $(\vt_l,\theta_{y_l})=\boldsymbol{\Theta}^{\eff,(l),\Omega_N}(0)$.   
			\item 
			On time scale $\bar{t}+N^kt$ the $k$-block averages have reached equilibrium. The full $k$-block average feels a drift towards the full $(k+1)$-block average, which is still in its initial state $\vt_k$. Therefore migration between the $k$-blocks in the hierarchical mean-field limit is replaced by a drift towards $\vt_k$, and the $k$-blocks become independent. This phenomenon is called \emph{decoupling} (or `propagation of chaos'). The resampling function for the full estimator converges to $\CF^{(k)}g$ (see \eqref{renor}), the average diffusion function of the $k$-blocks. Finally, the full $k$-block exchanges individuals with the $k$-dormant population. Hence the $k$-dormant population is the effective seed-bank on space-timescale $k$  Both the migration and the renormalisation are qualitatively similar to that of the hierarchical system without seed-bank \cite{DG93a}. However, the seed-bank still quantitatively affects the migration and the resampling through the slowing-down factor $E_k$. (In Section~\ref{ss.hml} we will see how the latter arises.) 
			\item 
			On time scale $\bar{t}+N^lt$ the $l$-block averages with $l<k$ are in a \emph{quasi-equilibrium}. The full $l$-block averages feel a drift towards the instantaneous value of the $(l+1)$-block average at time $\bar{t}$. Therefore also the $l$-block averages decouple. The $(l+1)$-block average is not moving on time scale $\bar{t}+N^lt$, and so for $t=L(N)$ we see that the $l$-block averages equilibrate faster than the $(l+1)$-block averages moves. The resampling function is given by $\CF^{(l)}g$, which is to be interpreted as the average diffusion function of the $l$-blocks. The full average interacts with the $l$-blocks of the $l$-dormant population, which is the effective seed-bank on level $l$. Again the full $l$-block average feels a slowing-down factor $E_l$.  
		\end{itemize}
		Note that Theorem~\ref{T.multiscalehiereff} only describes the limiting process of the combined block average $\bar{ \Theta}^{(l),\Omega_N}$  and the effective seed-bank $\Theta_{y_l}^{(l), \Omega_N}$. It does not provide a full description of the system, which is in Theorem~\ref{T.multiscalehier} below. We will see later that Theorem~\ref{T.multiscalehiereff} does describe all the \emph{non-trivial} components of the system. 
		
		\begin{remark}{\bf [Quasi equilibria]} 
			\label{remarkeff}
			{\rm Note that Theorem~\ref{T.multiscalehiereff} does not depend on the choice of $t_k,\ldots,t_0 \in (0,\infty)$. Since at each level $0\leq l\leq k$ we start from time $\bar{t}$, the $l$-block averages have already reached a quasi-equilibrium. } \hfill$\blacksquare$
		\end{remark}   
		
		\begin{figure}[htbp]
			\begin{center}
				\begin{tikzpicture}
\foreach \x in {9,11}
\draw [ultra thick,->] (0,\x+0.25) to (0,\x+0.25-1);
\node[centered] at (0,11.5) {$M^{k}_{-(k+1)}=(\textcolor{red}{\vartheta_k},\vartheta_k,\cdots, \vartheta_k,\theta_{y_{k+1}},\theta_{y_{k+2}},\cdots)$};
\draw [thick, decoration={brace}, decorate] (-0.75,11.75) -- (0.5,11.75);
\node[centered] at (-.125,12) {\small $k+1$ times};
\node[centered] at (0,9.5) {$M^{k}_{-k}=(\textcolor{green}{x_k(t_k)},x_k(t_k),\cdots,x_k(t_k),\textcolor{green}{y_{k,k}(t_k)},\theta_{y_{k+1}},\theta_{y_{k+2}},\cdots)$};
\node[centered] at (0,7.5) {$M^{k}_{-(k-1)}=\left(\textcolor{purple}{x_{k-1}(t_{k-1})},x_{k-1}(t_{k-1}),\cdots,x_{k-1}(t_{k-1}), \textcolor{purple}{y_{k-1,k-1}(t_{k-1})},\textcolor{green}{y_{k,k}(t_k)},\theta_{y_{k+1}},\theta_{y_{k+2}},\cdots\right)$};
\node[right] at (0.5,10.75){$Q^{[k]}\left(M^{k}_{-(k+1)},\text{d}u\right)=\Gamma_{M^{k}_{-(k+1)}}(\text{d}u)$};
\draw [thick, decoration={brace}, decorate] (-1.75,9.75) -- (0.25,9.75);
\node[centered] at (-0.75,10) {\small $k$ times};
\draw [thick, decoration={brace}, decorate] (-3,7.75) -- (0.5,7.75);
\node[centered] at (-1.25,8) {\small $k-1$ times};
\node[right] at (0.5,8.75){$Q^{[k-1]}\left(M^{k}_{-k},\text{d}u\right)=\Gamma_{M^k_{-k}}(\text{d}u)$};
\node[right] at (0.5,6.75){$Q^{[k-2]}\left(M^{k}_{-(k-1)},\text{d}u\right)=\Gamma_{M^{k}_{-(k-1)}}(\text{d}u)$};
\node[right] at (0.5,4){$Q^{[0]}\left(M^{k}_{-1},\text{d}u\right)=\Gamma_{M^k_{-1}}(\text{d}u)$};
\foreach \x in {6.15,6,6.30}
\draw[fill] (0,\x) circle [radius=0.05];
\foreach \x in {7,5.8}
\draw [ultra thick,->] (0,\x) to (0,\x-.5);
\node[centered] at (0,5) {$M^{k}_{-1}=\left(\textcolor{blue}{x_1(t_1)},x_1(t_1), \textcolor{blue}{y_{1,1}(t_1)},y_{2,2}(t_2),\cdots,\textcolor{purple}{y_{k-1,k-1}(t_{k-1})},\textcolor{green}{y_{k,k}(t_k)},\theta_{y_{k+1}},\theta_{y_{k+2}},\cdots\right)$};
\draw [ultra thick,->] (0,4.5) to (0,3.5);
\node[centered] at (0,3) {$M^{k}_{0}=\left(\textcolor{orange}{x_0(t_0), y_{0,0}(t_0)}, \textcolor{blue}{y_{1,1}(t_1)},y_{2,2}(t_2),\cdots,\textcolor{purple}{y_{k-1,k-1}(t_{k-1})},\textcolor{green}{y_{k,k}(t_k)},\theta_{y_{k+1}},\theta_{y_{k+2}},\cdots\right)$};
\end{tikzpicture}
			\end{center}
			\caption{Full interaction chain.}
			\label{fig:int}
		\end{figure}

		\subsubsection{Full process}
		
		To state our second main theorem, we will again use the process and the notation as defined in Section~\ref{ss.ing}.
		
		\begin{theorem}{{\bf [Hierarchical mean-field: full process]}}
			\label{T.multiscalehier} 
			Suppose that the initial state is given by \eqref{initialvalue}. Let $L(N)$ be such that $\lim_{N\to\infty} L(N)=\infty$ and $\lim_{N\to\infty} L(N)/N=0$. For $k \in \N$ and $t_k,\ldots,t_0 \in (0,\infty)$, set $\bar{t}=N^k L(N) + \sum_{n=0}^k N^n t_n$. 
			\begin{itemize}
				\item[(a)]
				For $k \in \N$, 
				\begin{equation}
					\begin{aligned}
						&\lim_{N\to\infty} \CL\left[\left(\boldsymbol{\Theta}^{(l),\Omega_N}\left(\bar{t}\,\right)\right)_{l=k+1,k,\ldots,0}\right]
						= \CL\left[(M^k_{-l})_{-l = -(k+1),-k,\ldots,0}\right].
					\end{aligned}
				\end{equation}
				\item[(b)]
				For $k\in\N$, 
				\begin{equation}
					\label{multiconvalt}
					\begin{aligned}
						&l>k\colon \lim\limits_{N\to\infty} 
						\CL\left[\left(\boldsymbol{\Theta}^{(l),\Omega_N}(\bar{t}+N^kt)\right)_{t > 0}\right] 
						= \delta_{(M^k_{-(k+1)})},\\[0.3cm]
						&l=k\colon \lim\limits_{N\to\infty} 
						\CL\left[\left(\boldsymbol{\Theta}^{(l),\Omega_N}(\bar{t}+N^lt)\right)_{t > 0}\right]
						= \CL\left[\left(z_{k,M^k_{-(k+1)}}(t)\right)_{t > 0}\right],\\[0.3cm]
						&l<k\colon \lim\limits_{N\to\infty} 
						\CL\left[\left(\boldsymbol{\Theta}^{(l),\Omega_N}(\bar{t}+N^lt)\right)_{t > 0}\right] 
						= \CL\left[\left(z_{l,M^k_{-(l+1)}}(t)\right)_{t > 0}\right],\\[0.3cm]
						&\text{in the Meyer-Zheng topology},
					\end{aligned}
				\end{equation}
				where the initial laws of the limiting processes are given by (see Fig.~\ref{fig:int})
				\begin{equation}
					\begin{aligned}
						&\CL\left[z_{k,M^k_{-(k+1),x}}(0)\right]=\Gamma^{(k)}_{M^k_{-(k+1)}},\\
						&\CL\left[z_{k,M^k_{-(k+1),x}}(0)\right]=\Gamma^{(l)}_{M^k_{-(l+1)}},\\
						&\Gamma^{(l)}_{M^k_{-(l+1)}}=\int_{\mathfrak{s}^{}}\cdots\int_{\mathfrak{s}^{}}\int_{\mathfrak{s}^{}}\Gamma^{(k)}_{M^k_{-(k+1)}}(\d u_k)\Gamma^{(k-1)}_{u_k}(\d u_{k-1})\cdots\Gamma^{(l+1)}_{l+2}(\d u_1) \Gamma^{(l)}_{u_{l+1}}
					\end{aligned}
				\end{equation}
			\end{itemize}
		\end{theorem}
		
		\begin{remark}{\bf[convergence in the Meyer-Zheng topology]}
			{\rm Note that Theorem~\ref{T.multiscalehier}(b) is stated in the Meyer-Zheng topology. This  topology is needed because at time-scales $N^lt$ rates occur in\eqref{rblockavxzmulti} that tend to infinity as $N\to\infty$. In Section~\ref{ss.hml} we define the Meyer-Zheng topology and explain why it is neeeded.}\hfill$\blacksquare$
		\end{remark}
		
		The statement in (a) shows that if we look at multiple space-time scales simultaneously, then the joint distribution of the different block averages behaves like the infinite-dimensional interaction chain defined in \eqref{fintchain}. The statement in (b) describes the law of the path on different times scales.
		\begin{itemize}
			\item 
			On time scale $N^kt$, the $l$-block averages with $l>k$ are not moving, i.e., $(\boldsymbol{\Theta}^{(l),\Omega_N}(\bar{t}+N^kt))_{t > 0}$ is a constant process. However, there is a difference between seed-banks with colour $m>k$ and seed-banks with colour $0\leq m\leq k$ in the way they interact with the active population. For $m>k$, even the $m$-dormant single colonies have not yet moved at time $\bar{t}+N^kt$, and hence are still in their initial states, with expectations $(\theta_{y_m})_{m=l+1}^\infty$. Therefore, also the $l$-block averages of $m$-dormant populations are still in their initial states, with expectations $(\theta_{y_m})_{m=l+1}^\infty$. For $l\leq k$ the $m$-dormant single colonies with $m\leq k$ at time $\bar{t}+N^kt$ have already interacted with the active population. Due to this interaction, for $l>k$ the $l$-block averages of $m$-dormant populations with $m\leq k$ are in state $\vt_k$ instead of their initial state $\theta_{y_m}$. However, on time scale $\bar{t}+N^kt$ $l$-block averages of $m$-dormant populations are not moving. (In Section~\ref{ss.hml} we explain how the shift from $\theta_{y_m}$ to $\vt_k$ occurs.) Also the $l$-block averages of the active population are in state $\vt_k$.
			\item 
			On time scale $\bar{t}+N^k t$, the $k$-block averages have reached equilibrium. We see that the active $k$-block average and the $k$-dormant $k$-block average evolve together like the effective $k$-block process in Theorem~\ref{T.multiscalehiereff}. Therefore the evolution of the active $k$-block average is slowed down by a factor $E_k$, the active $k$-block feels a drift towards $\vt_k$ (the value of the active $(k+1)$-block average at time $\bar{t}$), resamples with diffusion function $\CF^{(k)}g$, and exchanges individuals with the $k$-dormant $k$-block. The $k$-dormant $k$-block average evolves only via interaction with the active $k$-block. The $m$-dormant $k$-block averages with colour $0\leq m<k$ equal the active $k$-block average and hence follow their evolution. The $m$-dormant $k$-blocks with colour $m>k$ are still in their initial states, since on time scale $\bar{t}+N^kt$ even single colony seed-banks with colour $m>k$ have not yet started to interact with the active population. 
			\item 
			On time scale $\bar{t}+N^lt$, for $0\leq l<k$, the $l$-block averages are in a quasi-equilibrium. Again, the active $l$-block and the $l$-dormant $l$-block average, which is the effective seed-bank, behave as the effective process in \ref{T.multiscalehiereff}. Hence, the active $l$-block average feels a drift towards the instantaneous value of the active $(l+1)$-block average, which is given by the first component of the interaction chain $M^k_{-(l+1)}$, resamples according to the renormalised diffusion function $\CF^{(l)}g$, and exchanges with the $l$-block of the $l$-dormant population. The evolution of the active $l$-block average is slowed down by a factor $E_l$. The $l$-block of the $l$-dormant population exchanges individuals with the active population. The $l$-blocks of the $m$-dormant population with colours $0\leq m<l$ follow the active population. The states of the $m$-dormant population with colour $m>l$ are given by the corresponding components in the interaction chain $M^k_{-(l+1)}$. Hence the $l$-block averages with colours $m>k$ are still in their initial states $\theta_{y_m}$, because on time scale $\bar{t}+N^lt$ even the single dormant colonies with colour $m>k$ have not yet interacted with the active population. However, something interesting is happening with the colours $l<m\leq k$: they are in a state obtained on the time scale in which they where effective, i.e., for $l<m\leq k$ the $m$-dormant $l$-block average is in state $y_{m,m}(\bar{t}\,)$. This happens because at time $\bar{t}$ the single colonies have already interacted with the active population, but on time scale $N^lt$ they do not interact anymore with the active population. (Also this effect will be further explained in Section~\ref{ss.hml}.) 
		\end{itemize}
		
		\begin{remark}{\bf [Comparison to the hierachical multi-scale limit without seed-bank]}
			{\rm Comparing Theorem~\ref{T.multiscalehier} with the multi-scale limit theorems derived for the hierarchical system without seed-bank \cite{DG93a}, \cite{DG93b}, \cite{DGV95}, we see that the seed-bank affects the system both quantitatively and qualitatively. First, the active population is \emph{slowed down} by the total size of the seed-banks it has interacted with, represented by the slowing-down factors $(E_l)_{l\in\N_0}$. Second, the interaction with the \emph{effective seed-bank} on each space-time scale is special to the system with seed-bank. Still, the \emph{decoupling} of the active component and the \emph{renormalisation transformation} for the diffusion function are similar as in the system without seed-bank.}
		\end{remark}
		
		\begin{remark}{\bf [$k\to\infty$ limit of the interaction chain]}
			\label{remark1}
			{\rm The result in \eqref{multiconv} raises the question how the hierarchical multi-scale limit behaves for large $k$. We find the following dichotomy:
				\begin{equation}
					\lim_{k\to\infty} \CL\left[(M^k_{-(k+1),-k,\cdots,0})\right] = \CL\left[(M^\infty_{k})_{k\in\Z^{-}}\right], 
				\end{equation}
				where in the \emph{clustering regime}
				\begin{equation}
					\CL\left[(M^\infty_k)_{k\in\Z^-}\right]=\theta\delta_{(1,1^{\N_0})^{\Z^-}}+(1-\theta)\delta_{(0,0^{\N_0})^{\Z^-}}
				\end{equation}
				and in the \emph{coexistence regime} $M^\infty=(M_k^\infty)_{k\in\Z^-}$ is a realisation of the unique entrance law of the interaction chain at time $-\infty$ with
				\begin{equation}
					\lim_{l\to\infty} M_{-l}^\infty = (\theta,\theta).
				\end{equation}
				In the latter case, $M^\infty$ corresponds to the equilibrium vector of block averages around site $0$, whose law agrees with that of the equilibrium block averages for the mean-field model after we take the limit $N\to\infty$ (see \cite[Proposition 6.2 and 6.3]{DGV95}).} \hfill $\blacksquare$
		\end{remark}
		
		\begin{remark}{\bf [Interaction field]} 
			\label{remark2}
			{\rm Theorem~\ref{T.multiscalehier} looks at the tower of block averages over a fixed site, namely, $0$. In order to study the cluster formation in the clustering regime or the equilibria in the coexistence regime, we must analyse the dependence structure between the towers of block averages over \emph{different} sites.  We can show that, in the limit as $N\to\infty$, an interacting random field emerges, indexed by a tree with countably many edges coming out of every site at every level. This random field has the property that the averages over any two points $\eta,\eta^\prime$ with $d(\eta,\eta^\prime)=l$, follow a single interaction chain in equilibrium from $k+1$ until $l$ (or from $-\infty$ until $l$ in the entrance law) and, conditional on the state in $l$, evolve independently as the interaction chain beyond $l$. This corresponds to what is called \emph{propagation of chaos} of the $(l-1)$-block averages given the $l$-block average. For the model without seed-bank such results are described in \cite[Section 0(e)]{DGV95}. Using our results for the model with seed-bank above, we can in principle follow an analogous line of argument. We refrain from spelling out the details.} \hfill$\blacksquare$
		\end{remark}

		\subsection{Heuristics behind the multi-scale limit}
		\label{ss.hml}
		
		The proofs of Theorems~\ref{T.multiscalehiereff} and \ref{T.multiscalehier} written out in Sections~\ref{ss.IntroMeanfield}--\ref{s.multilevel}, are long and technical. In order to help the reader appreciate these proofs, we provide the heuristics in Sections~\ref{sss.blockavh}--\ref{sss.formin}.
		
		\subsubsection{Evolution of the block-averages}
		\label{sss.blockavh}
		
		Recall the block averages introduced in Definition \ref{defblockav} and their evolution defined in \eqref{rblockavxzmulti}. In the limit as $N\to\infty$, we heuristically obtain from the SSDE in \eqref{rblockavxzmulti} the following results for the $k$-block process 
		\begin{equation}
			({x}_{k}^{\Omega_N}(t),({y}_{m,k}^{\Omega_N}(t))_{m\in\N_0})_{t \geq 0}.
		\end{equation}
		\begin{itemize}
			\item 
			{\bf Migration.} Recall that the migration is captured by the first term of the first equation in \eqref{rblockavxzmulti}, i.e., the first term of the evolution of the active part of the population. Letting $N\to\infty$, we see that in the sum over $l$ only the term $l=1$ contributes. Hence we expect that the effective migration felt by the active $k$-block average is towards ${x}_{k+1}^{\Omega_N}(0)$, the initial state of the active $(k+1)$-block average. Note that the migration term in \eqref{rblockavxzmulti} can be written as
			\begin{equation}
				\label{mig}
				\sum_{l\in\N} \frac{c_{k+l-1}}{N^{l-1}}\big[{x}_{k+l}^{\Omega_N}(N^{-l}t)-{x}_{k}^{\Omega_N}(t)\big]\,\d t=\sum_{l\in\N} 
				\frac{c_{k+l-1}}{N^{l-1}}\left[\frac{1}{N^l}\sum_{k=0}^{N^l-1}{x}_{k}^{\Omega_N}(t)-{x}_{k}^{\Omega_N}(t)\right]\,\d t.
			\end{equation}
			The drift towards the $(k+1)$-block average is therefore also a drift towards the current average of the $k$-blocks in the $(k+1)$-block. In the limit as $N\to\infty$, the latter can be approximated by $\E[x_k(t)]$. Effectively, as $N\to\infty$, the $k$-blocks become independent given the the value of ${x}_{k+1}^{\Omega_N}(0)$, i.e., there is \emph{decoupling}.
			\item 
			{\bf Resampling.} Recall that the diffusion term in the evolution equation of the active population represents the resampling. Therefore we see that the active $k$-block resamples at a rate that is the average resampling rate over the $k$-block. For $k=1$, the resampling rate of the $1$-block is the average of the resampling ratse of the single colonies.  Therefore, in the limit $N\to\infty$, due to the decoupling described above, we expect that the resampling rate for the $1$ block is given by $\E[g]$, where the expectation is w.r.t.\ the quasi-equilibrium of the single colonies. This expectation is exaclty the renormalised diffusion function $\CF g$ (see \eqref{renor}). For the $k$-block, we may interpret the diffusion function to be the average of the diffusion function for the $(k-1)$-blocks. By ``induction" we assume that the $(k-1)$ blocks resample at rate $\CF^{(k-1)}g$. Hence, due to the decoupling of the $(k-1)$-blocks as $N\to\infty$, we expect the resampling rate for the $k$ blocks to equal $\E[\CF^{(k-1)}g]$, where the expectation is w.r.t.\ the quasi-equlibrium of the $(k-1)$-blocks. This yields another iteration of the renormalisation transformation (see \eqref{frenormit}). Hence, we expect the diffusion function for the $k$-blocks to converge to $\CF^{(k)}g$.
			\item 
			{\bf Exchange with the seed-bank.} Recall that the last term of the first equation in \eqref{rblockavxzmulti} and the second equation in \eqref{rblockavxzmulti} together describe the exchange of the active $k$-block with the $m$-dormant $k$-block. To describe the limiting behaviour as $N\to\infty$, we distinguish three cases: $0\leq m<k$, $m=k$, $m>k$.
			\begin{itemize}		
				\item	
				If $0\leq m<k$, then we see that the rate of exchange between the active $k$-block and the $m$-dormant $k$-block tends to infinity as $N\to\infty$. We therefore expect them to equalise, i.e.,
				\begin{equation}
					\label{14}
					\lim_{N \to \infty} \CL\left[\left({x}_{k}^{\Omega_N}(t)-{y}_{m,k}^{\Omega_N}(t)\right)_{t>0}\right] = \delta_0,
				\end{equation}
				where $0$ should be interpreted as the process equal to $0$, $(0)_{t>0}$.
				Hence we see that $m$-dormant $k$-block follows the active $k$-block immediately. (To formalise this fact, we need the \emph{Meyer-Zheng topology} \cite{MZ84}.) 
				\item		
				If $m=k$, then there is a non-trivial exchange between the active $k$-block and the $k$-dormant $k$-block. 
				\item		
				If $m>k$, then the exchange rate between the active $k$-block and the $m$-dormant $k$-block tends to zero as $N\to\infty$. 
			\end{itemize}		
			Thus, only the $k$-dormant $k$-block has a non-trivial interaction with the active $k$-block. We express this by saying that on space-time scale $k$ the $k$-dormant population plays the role of the \emph{effective seed-bank}.   
		\end{itemize}
		
		\subsubsection{Limiting evolution of the block-averages}
		\label{sss.efproc}
		
		To determine the limiting evolution of the full block-averages process, we first have a look at the limiting evolution of the effective process.
		
		\paragraph{The effective process.}
		
		To determine the limit as $N\to\infty$ of \eqref{rblockavxzmulti}, we need to get rid of the diverging rates. Instead of only looking at the $k$-block process $({x}_{k}^{\Omega_N}(t),({y}_{m,k}^{\Omega_N}(t))_{m\in\N_0})_{t \geq 0}$, which evolves according to \eqref{rblockavxzmulti}, we look at the \emph{effective $k$-block process} defined as
		\begin{equation}
			\label{419}
			\left(\bar{x}_{k}^{\Omega_N}(t),y_{k,k}^{\Omega_N}(t)\right)_{t \geq 0},
		\end{equation}
		where we abbreviate 
		\begin{equation}
			\gls{effx}=\frac{x_{k}^{\Omega_N}(t)+\sum_{m=0}^{k-1}K_my_{m,k}^{\Omega_N}(t)}
			{1+\sum_{m=0}^{k-1}K_m}.
		\end{equation}
		By \eqref{14} and the heuristic discussion given above, the process in \eqref{419} equals $(x_{k}^{\Omega_N}(t),y_{k,k}^{\Omega_N}(t))_{t\geq 0}$ in the limit as $N\to\infty$, i.e., it describes the joint distribution of the active $k$-block and the effective dormant $k$-block, which is the $k$-dormant $k$-block. Using \eqref{rblockavxzmulti}, we see that the process in \eqref{419} evolves according to the SSDE
		\begin{equation}
			\label{rblockavxzmultieff}
			\begin{aligned}
				\d\bar{x}_{k}^{\Omega_N}(t)
				&= E_k \sum_{l\in\N} \frac{c_{k+l-1}}{N^{l-1}}\big[{x}_{k+l}^{\Omega_N}(N^{-l}t)-{x}_{k}^{\Omega_N}(t)\big]\,\d t\\
				&\qquad + E_k \sqrt{\frac{1}{N^k} \sum_{\xi\in B_k(0)} g(x_\xi(N^kt))}\,\,\d w_k(t)\\
				&\qquad + E_k \sum_{m=k}^\infty N^k \frac{K_m e_m}{N^m}
				\big[{y}_{m,k}^{\Omega_N}(t)-{x}_{k}^{\Omega_N}(t)\big]\,\d t,\\[0.2cm]
				\d{y}_{k,k}^{\Omega_N}(t)
				&=e_k\big[{x}_{k}^{\Omega_N}(t)-{y}_{k,k}^{\Omega_N}(t)\big]\,\d t.
			\end{aligned}
		\end{equation}  
		In \eqref{rblockavxzmultieff} no infinite rates appear anymore. In the limit as $N\to\infty$, by \eqref{14} we can approximate
		\begin{equation}
			\label{ap1}
			x_{k}^{\Omega_N}(t)\approx y_{m,k}^{\Omega_N}(t), \qquad 0\leq m<k,
		\end{equation}
		such that
		\begin{equation}
			\label{ap2}
			x_{k}^{\Omega_N}(t) \approx \bar{x}_{k}^{\Omega_N}(t).
		\end{equation}
		We can therefore approximate \eqref{rblockavxzmultieff} by
		\begin{equation}
			\label{eveff}
			\begin{aligned}
				\d \bar{x}_{k}^{\Omega_N}(t)
				&= E_k \sum_{l\in\N} \frac{c_{k+l-1}}{N^{l-1}}\big[\bar{x}_{k+l}^{\Omega_N}(N^{-l}t)-\bar{x}_{k}^{\Omega_N}(t)\big]\,\d t\\
				&\qquad + E_k \sqrt{\frac{1}{N^k} \sum_{\xi\in B_k(0)} g\big(\bar{x}_\xi(N^kt)\big)}\,\,\d w_k(t)\\
				&\qquad + E_k \sum_{m=k}^\infty N^k \frac{K_m e_m}{N^m}
				\big[{y}_{m,k}^{\Omega_N}(t)-\bar{x}_{k}^{\Omega_N}(t)\big]\,\d t,\\[0.2cm]
				\d{y}_{k,k}^{\Omega_N}(t)
				&=e_k\big[\bar{x}_{k}^{\Omega_N}(t)-{y}_{k,k}^{\Omega_N}(t)\big]\,\d t.
			\end{aligned}
		\end{equation}  
		Hence, in the limit as $N\to\infty$, the process in \eqref{419} becomes autonomous. Moreover, assuming that $\lim_{N \to \infty}\bar{x}_{l+1}^{\Omega_N}(0)=\vt_l$, we see that \eqref{eveff} approaches the effective process defined in \eqref{927}, with 
		\begin{equation}
			\theta=\vt_l, \quad E=E_k, \quad c=c_k, \quad e=e_k, \quad K=K_k, \quad g=\CF^{(k)}g. 
		\end{equation}
		In particular, we see that the \emph{slowing-down constant} $E_k$ arises because the active population is the only part of the first component of \eqref{419}. Note that therefore only a part, the active part,  from the first component migrates, resamples and exchanges with the seed-bank. Due to the infinite rates, the active population ``drags along all the fast seed-banks with total size $\sum_{m=0}^{k-1}K_m$". This causes the slowing down factors $E_k$.
		
		Since there are no infinite rates in the evolution of the effective process, we can use the classical path space topology. This allows us in the proof in Sections~\ref{ss.IntroMeanfield}--\ref{s.multilevel} to build on techniques developed for the hierarchical mean-field model without seed-bank in \cite{DG93b}, \cite{DGV95}. It turns out that the \emph{effective process} is very useful in our analysis.
		
		\paragraph{From the effective process to the full process.}
		
		For large $N$, by \eqref{ap1} and \eqref{ap2}, the evolution of our original process $({x}_{k}^{\Omega_N}(t),({y}_{m,k}^{\Omega_N}(t))_{m\in\N_0})_{t \geq 0}$ can be approximated by 
		\begin{equation}
			\label{ev}
			\begin{aligned}
				\d \bar{x}_{k}^{\Omega_N}(t)
				&\approx E_k \sum_{l\in\N} \frac{c_{k+l-1}}{N^{l-1}}\big[\bar{x}_{k+l}^{\Omega_N}(N^{-l}t)-\bar{x}_{k}^{\Omega_N}(t)\big]\,\d t\\
				&\qquad + E_k \sqrt{\frac{1}{N^k} \sum_{\xi\in B_k(0)} g(\bar{x}_\xi(N^kt))}\,\,\d w_k(t)\\
				&\qquad + E_k \sum_{m=k}^\infty N^k \frac{K_m e_m}{N^m}
				\big[{y}_{m,k}^{\Omega_N}(t)-\bar{x}_{k}^{\Omega_N}(t)\big]\,\d t,\\
				{y}_{k,k}^{\Omega_N}(t)&=\bar{x}_{k}^{\Omega_N}(t),\\
				\d{y}_{k,k}^{\Omega_N}(t)
				&=e_k\big[\bar{x}_{k}^{\Omega_N}(t)-{y}_{k,k}^{\Omega_N}(t)\big]\,\d t,\\[0.2cm]
				{y}_{k,k}^{\Omega_N}(t)&={y}_{m,k}^{\Omega_N}(0).
			\end{aligned}
		\end{equation}  
		By the ergodic theorem for exchangeable measures, we can assume that $\lim_{N \to \infty}\bar{x}_{k+1}^{\Omega_N}(0)=\vt_k\, a.s.$. We expect that \eqref{ev} approaches \eqref{z11a} with 
		\begin{equation}
			\theta=\vt_k, \quad E=E_k, \quad c=c_k, \quad e=e_k, \quad K=K_k, \quad g=\CF^{(k)}g.
		\end{equation} 
		To prove that
		\begin{equation}
			\lim_{N \to \infty}\CL[{y}_{m,k}^{\Omega_N}(t)]=\lim_{N\to\infty}\CL[\bar{x}_{k}^{\Omega_N}(t)], \qquad 0\leq m<k-1,
		\end{equation} 
		we need \emph{the Meyer-Zheng topology} explained in Section~\ref{MeyerZheng}. In the proof in Sections~\ref{ss.ProofsMeanfield}--\ref{s.multilevel} we show how the above approximations can be made rigorous. 
		
		\paragraph{Conserved quantities.}
		
		Note that, by \eqref{rblockavxzmulti} and \eqref{mig}, for each $k\in\N_0$
		\begin{equation}
			\E\left[\frac{x^{\Omega_N}_k(t)+\sum_{m\in\N_0}K_my^{\Omega_N}_k(t)}{1+\sum_{m\in\N_0}K_m}\right] = \theta,
			\qquad t \geq 0,
		\end{equation}
		is a conserved quantity. For each $k\in\N$ we obtain that for $l\geq k$ 
		\begin{equation}
			\lim_{N \to \infty}\E\left[\frac{x^{\Omega_N}_k(t)+\sum_{m=0}^{l}K_my^{\Omega_N}_k(t)}{1+\sum_{m=0}^l K_m}\right]
			=\vt_l,
		\end{equation}
		is a conserved quantity. For the effective process, \eqref{eveff} implies that
		\begin{equation}
			\label{359}
			\begin{aligned}
				\frac{\d}{\d t}\E[\bar{x}_k^{\Omega_N}(t)] 
				&=E_{k} K_ke_k\,\Big(\E[y_{k,k}^{\Omega_N}(t)]-\E[\bar{x}_k^{\Omega_N}(t)]\Big),\\
				\frac{\d}{\d t}\E[y_{k,k}^{\Omega_N}(t)] &= e_k\,\Big(\E[\bar{x}_k^{\Omega_N}(t)]-\E[y_{k,k}^{\Omega_N}(t)]\Big).
			\end{aligned}
		\end{equation}
		Recall that 
		\begin{equation}
			\E[\bar{x}_k^{\Omega_N}(0)] = \E\left[\frac{x_k^{\Omega_N}(0)
				+\sum_{m=0}^{k-1}K_my_{m,k}^{\Omega_N}(0)}{1+\sum_{m=0}^{k-1}K_m}\right]
			=\vt_{k-1},\qquad \E[y_{k,k}^{\Omega_N}(0)]=\theta_{y_k}.
		\end{equation}
		Therefore we can solve \eqref{359} explicitly as
		\begin{equation}
			\label{expz}
			\begin{aligned}
				\E[\bar{x}_k^{\Omega_N}(t)] &= \vt_k + \frac{E_kK_k}{1+E_kK_k} (\vt_{k-1}-\theta_{y_k})\, \e^{-(E_kK_k+1)e_kt},\\
				\E[y_{k,k}^{\Omega_N}(t)]&= \vt_k - \frac{1}{1+E_kK_k} (\vt_{k-1}-\theta_{y_k})\, \e^{-(E_kK_k+1)e_kt}.
			\end{aligned}
		\end{equation}
		
		The above computation shows what happens to $\E[\bar{x}_k^{\Omega_N}(t)]$  if we move one space-time scale up in the hierarchy, namely, a new seed-bank starts interacting with the active population. This causes that $\vt_{k-1}$ is pulled a bit towards $\theta_{y_k}$, so that also $\E[\bar{x}_k^{\Omega_N}(t)]$ changes a bit. Each new seed-bank that opens up changes the expectation of the active population, which results in the sequence $(\theta_x,\vt_0, \vt_1,\vt_2,\ldots)$ for the expectation of the active population on space-time scales $\{0,1,2,3,\cdots\}$. From \eqref{mig} we see that the drift of $\bar{x}_k^{\Omega_N}(t)$ is towards
		\begin{equation}
			\bar{x}_{k+1}^{\Omega_N}(N^{-1}t)=\frac{1}{N^k}\sum_{k=0}^{N^k-1}{x}_{k}^{\Omega_N}(t)\approx\E[\bar{x}_k^{\Omega_N}(t)],
		\end{equation}
		where the last approximation can be made because the $k$-blocks decouple. Hence, in the limit as $N\to\infty$, once the $k$-blocks are in a quasi-equilibrium we can replace the drift towards $\bar{x}_{k+1}^{\Omega_N}(N^{-1}t)$ by a drift towards $\E[\bar{x}_k^{\Omega_N}(t)]=\vt_k$.
		
		\paragraph{Shifting averages.}
		
		Recall the full estimator process $(\mathbf{\Theta}^{(l) ,\Omega_N}(t))_{t >0}$ defined in \eqref{412}. Equation~\ref{blockavx*z*alt} implies that the evolution of the estimator process is given by
		\begin{equation}
			\label{463}
			\begin{aligned}
				\d \Theta^{(l) ,\Omega_N}_x(t)
				&=\sum_{n=l+1}^\infty\frac{c_{n-1}}{N^{n-1}}[\Theta^{(n) ,\Omega_N}_x (t)-\Theta^{(l)}_x(t)]\,\d t\\
				&\qquad+\sqrt{\frac{1}{N^{2l}}\sum_{\xi\in B_l}g\big(x_\xi(t)\big)}\,\d w(t)\\
				&\qquad+\sum_{m\in\N_0}\frac{K_me_m}{N^m}[\Theta^{(l) ,\Omega_N}_{y_m}(t)-\Theta^{(1) ,\Omega_N}_x(t)]\,\d t,\\[0.2cm]
				\d \Theta^{(l) ,\Omega_N}_{y_m}(t)&=\frac{e_m}{N^m}[\Theta^{(l) ,\Omega_N}_x(t)-\Theta^{(l) ,\Omega_N}_{y_m}(t)]\,\d t,\, m\in\N_0.
			\end{aligned}
		\end{equation}
		Looking at the estimator process $(\mathbf{\Theta}^{(l),\Omega_N}(t))_{t> 0}$ on time scale $N^kt$, we see that 
		\begin{equation}
			\label{464}
			\begin{aligned}
				\d \Theta_x^{(l) ,\Omega_N}(N^kt)&=\sum_{n=l+1}^\infty\frac{c_{n-1}}{N^{n-1-k}}
				\left[\Theta_x^{(n),\Omega_N}(N^kt)-\Theta_x^{(l),\Omega_N}(N^kt)\right]\,\d t\\
				&\qquad+\sqrt{\frac{N^k}{N^{2l}}\sum_{\xi\in B_l}g\big(x_\xi(N^kt)\big)}\,\d w(t)\\
				&\qquad+\sum_{m\in\N_0} \frac{K_me_m}{N^{m-k}}
				\left[\Theta^{(l) ,\Omega_N}_{y_m}(N^{k}t)-\Theta^{(l) ,\Omega_N}_{x}(N^{k}t)\right]\,\d t,\\[0.2cm]
				\d \Theta^{(l) ,\Omega_N}_{y_m}(N^{k}t)&=\frac{e_m}{N^{m-k}}
				\left[\Theta^{(l) ,\Omega_N}_{x}(N^{k}t)-\Theta^{(l) ,\Omega_N}_{y_m}(N^{k}t)\right]\,\d t, \qquad m\in\N_0.
			\end{aligned}
		\end{equation}
		From \eqref{464} we get that, on time scale $N^{k}t$, for all $l\geq k+1$,
		\begin{equation}
			\Theta^{(l) ,\Omega_N}(N^kt)=\frac{\Theta^{(l) ,\Omega_N}_x(N^kt)
				+\sum_{m=0}^{l-1}K_m\Theta^{(l) ,\Omega_N}_{y_m}(N^kt)}{1+\sum_{m=0}^{l-1}K_m}, \qquad t \geq 0, 
		\end{equation}
		is a conserved quantity in the limit as $N\to\infty$, and for $t\geq 0$,
		\begin{equation}
			\label{vyt}
			\lim_{N \to \infty}\Theta^{(l) ,\Omega_N}(N^kt)=\frac{\theta_x
				+\sum_{m=0}^{l-1}K_m \theta_{y_m}}{1+\sum_{m=0}^{l-1}K_m}=\vt_l, \text{ in probability}.
		\end{equation}
		For $m>k$, also $\Theta^{(l) ,\Omega_N}_{y_m}(N^{k}t)$ is a conserved quantity, and
		\begin{equation}
			\lim_{N \to \infty}\Theta^{(l) ,\Omega_N}_{y_m}(N^{k}t)=\theta_{y_m}, \qquad t \geq 0.
		\end{equation}
		However, for $l\geq k+1$, $\Theta^{(l) ,\Omega_N}_{x}(N^{k}t)$ and $(\Theta^{(l) ,\Omega_N}_{y_m}(N^{k}t))_{m=0}^k$ are not conserved quantities in the limit as $N\to\infty$. Note that from \ref{rblockavxzmulti} we heuristically see that the full $l$-block estimator process $({\bf \Theta}^{(l),\Omega_N}(N^kt))_{t\geq 0}$ with $l>k$ converges to the process
		\begin{equation}
			\left(\Theta^{(l)}_x(t),\left(\Theta^{(l)}_{y_m}(t)\right)_{m\in\N_0}\right)_{t>0},
		\end{equation}
		which evolves according to 
		\begin{equation}
			\begin{aligned}
				&\d \Theta^{(l) }_x(t) = E_{k-1}K_{k}e_{k}[\Theta^{(l) }_{y_k}(t)-\Theta^{(l),\Omega_N }_x(t)]\,\d t,\\
				&\begin{array}{llll}
					\Theta^{(l) }_{y_{m}}(t) &=& \Theta^{(l) }_x(t), &m<k,\\
					\d \Theta^{(l) }_{y_{k}}(t) &=& e_{k}[\Theta^{(l) }_x(t)-\Theta^{(l) }_{y_{k}}(t)]\,\d t, &m=k,\\
					\Theta^{(l) }_{y_m}(t) &=& \theta_{y_m}, &m>k.
				\end{array}
			\end{aligned}
		\end{equation}   
		This system can be explicitly solved as
		\begin{equation}
			\label{234}
			\begin{aligned}
				&\Theta^{(l) }_x(t) = \frac{\Theta^{(l) }_x(0)+E_{k-1}K_k\Theta^{(l) }_{y_{k}}(0)}{1+E_{k-1}K_k}\\
				&\qquad\qquad\qquad +\frac{E_{k-1}K_k}{1+E_{k-1}K_k}
				[\Theta^{(l),\Omega_N }_x(0)-\Theta^{(l) }_{y_k}(0)]\,\e^{-(E_{k-1}K_ke_k+e_k)t},\\
				&\begin{array}{llll}
					\Theta^{(l) }_{y_{m}}(t) &=& \Theta^{(l) }_x(t), &m<k,\\[0.2cm]
					\Theta^{(l) }_{y_{k}}(t) &=& \frac{\Theta^{(l) }_x(0)+E_{k-1}K_k\Theta^{(l) }_{y_{k}}(0)}{1+E_{k-1}K_k}\\[0.2cm]
					&&\qquad -\frac{1}{1+E_{k-1}K_k}[\Theta^{(l),\Omega_N }_x(0)
					-\Theta^{(l) }_{y_k}(0)]\,\e^{-(E_{k-1}K_ke_k+e_k)t}, &m=k,\\
					\Theta^{(l) }_{y_m}(t) &=& \theta_{y_m}, &m>k.
				\end{array}
			\end{aligned}
		\end{equation} 
		The latter shows that, each time we enter a new space-time scale, all the large active blocks interact with the large effective dormant blocks until they equalise. Thus, on each space-time scale, all the active $l$-blocks and the dormant $l$-blocks of colour $m\leq l$ move for a short period of time. As a consequence, the value of $\Theta_x^{(l)}(0)$ depends on the scaling we choose. To illustrate this, we note that
		\begin{equation}
			\begin{aligned}
				\lim_{N\to\infty}\Theta_{x}^{(l),\Omega_N}(0) &=\theta_x,\\
				\lim_{N \to \infty}\Theta_x^{(l),\Omega_N}(L(N)+t) &=\vt_0,\\
				\lim_{N \to \infty}\Theta_x^{(l),\Omega_N}((L(N)N^n+N^nt)) &=\vt_n, \quad 0\leq n\leq k.
			\end{aligned}
		\end{equation}
		However, the scaling $N^kt$ would imply that
		\begin{equation}
			\begin{aligned}
				\lim_{N\to\infty}\Theta_{x}^{(l),\Omega_N}(0) &=\theta_x,\\
				\lim_{N \to \infty}\Theta_x^{(l),\Omega_N}\left(N^k\frac{L(N)+t}{N^k}\right) &=\vt_0,\\
				\lim_{N \to \infty}\Theta_x^{(l),\Omega_N}\left(N^k\frac{L(N)N^n+N^nt}{N^k}\right) &=\vt_n,\quad 0\leq n\leq k.
			\end{aligned}
		\end{equation}
		Hence, if $(t(N))_{N\in\N}$ is a sequence such that $\lim_{N \to \infty}N^kt(N)=0$, then the value of the limit
		\begin{equation}
			\lim_{N\to\infty} \Theta^{(l),\Omega_N}_x(N^k t(N))
		\end{equation}
		depends on $(t(N))_{N\in\N}$. Moreover, to obtain a limiting process for $({\bf \Theta}^{(l),\Omega_N}(N^kt))_{t\geq 0}$ 
		we need convergence also at time $0$, while it is not clear what $N^kt \downarrow  0$ means. To circumvent these subtleties, we look at the process at times $t>0$ and use as starting time $\bar{t}$ defined in Theorem~\ref{T.multiscalehiereff}. 
		
		From \eqref{234} it follows that, for $l\geq k$,
		\begin{equation}
			\label{acin}
			\begin{aligned}
				\lim_{N \to \infty}\Theta^{(l) ,\Omega_N}_{x}(\bar{t}\,) &=\vt_k, \text{ in probability},\\
				\lim_{N \to \infty}\Theta^{(l) ,\Omega_N}_{y_m}(\bar{t}\,) &=\vt_k, \quad m\leq k, \text{ in probability},\\
				\lim_{N \to \infty}\Theta^{(l) ,\Omega_N}_{y_m}(\bar{t}\,) &=\theta_{y_m}, \quad m > k\text{ in probability}. 
			\end{aligned}
		\end{equation}    
		Note that the shifting of averages mentioned earlier is closely related to the conserved quantities discussed in Section~\ref{sss.efproc} because, for large $N$,
		\begin{equation}
			\lim_{N \to \infty}\Theta^{(l) ,\Omega_N}_{x}(\bar{t}\,)\approx\E[x_k^{\Omega_N}],
		\end{equation}
		where the expectation is taken in the quasi-equilibrium the $k$-blocks have attained after scaling with time $\bar{t}$. 
		
		\subsubsection{Formation of the interaction chain}
		\label{sss.formin}
		
		In Section~\ref{sss.blockavh} we saw how subsequent space-time scales are connected via the migration term. In this section we show how the interaction chain arises from this connection. We first show how the effective interaction chain arises for the effective process. Then we show how the full interaction chain is formed, by studying the slow seed-banks.
		
		\paragraph{Connections between different space-time scales}
		
		Let $\bar{t}$ be as in Theorem~\ref{T.multiscalehiereff}. From \eqref{463} it follows that the process $({\bf\Theta}^{(l) ,\Omega_N}(\bar{t}+N^lt))_{t>0}$ evolves according to  
		\begin{equation}
			\label{463b}
			\begin{aligned}
				\d \Theta^{(l) ,\Omega_N}_x(\bt+N^lt)
				&=\sum_{n=l+1}^\infty\frac{c_{n-1}}{N^{n-1-l}}[\Theta^{(n) ,\Omega_N}_x (\bt+N^l t)-\Theta^{(l)}_x(\bt+N^lt)  ]\,\d t\\
				&\qquad+\sqrt{\frac{1}{N^{l}}\sum_{\xi\in B_l}g\big(x_\xi(\bt+N^lt)\big)}\,\d w(t)\\
				&\qquad+\sum_{m\in\N_0}\frac{K_me_m}{N^{m-l}}[\Theta^{(l) ,\Omega_N}_{y_m}(\bt+N^lt)
				-\Theta^{(1) ,\Omega_N}_x(\bt+N^lt)]\,\d t,\\
				\d \Theta^{(l) ,\Omega_N}_{y_m}(\bt+N^lt)&=\frac{e_m}{N^{m-l}}[\Theta^{(l) ,\Omega_N}_x(\bt+N^lt)
				-\Theta^{(l),\Omega_N}_{y_m}(\bt+N^lt)]\,\d t,\qquad m\in\N_0.
			\end{aligned}
		\end{equation}
		Therefore, in the limit as $N\to\infty$, the active population $\Theta^{(l) ,\Omega_N}_x (\bt+N^l t)$ feels a drift towards the $(l+1)$-block average $\Theta^{(l+1) ,\Omega_N}_x (\bt+N^{l+1} t)$. If $l=k$, then
		\begin{equation}
			\lim_{N \to \infty}\CL[{\bf\Theta}^{(k+1) ,\Omega_N}(\bar{t}+N^kt)]
			= \lim_{N \to \infty}\CL[{\bf\Theta}^{(k+1) ,\Omega_N}(\bar{t})],
		\end{equation}
		since the $(k+1)$-block has not yet started to move at time $\bar{t}+N^kt$. From \eqref{acin} it follows that
		\begin{equation}
			\lim_{N \to \infty}\Theta^{(k+1) ,\Omega_N}_{x}(\bar{t}+N^kt)=\vt_k\, in\, probability.
		\end{equation}
		Therefore the drift of the active population $\Theta^{(l) ,\Omega_N}_x (\bt+N^l t)$ is towards $\vt_k$. Since $\bar{t}>L(N)N^k$, the process ${\bf\Theta}^{(k) ,\Omega_N}(\bar{t}+N^kt)$ has, in the limit $N\to\infty$, already reached its equilibrium, which is denoted by $\Gamma_{\bar{\vt_k}}$, where 
		\begin{equation}
			\bar{\vt_k}=(\vartheta_k,\overbrace{\vartheta_k,\cdots, \vartheta_k}^{k+1\text{ times }},\theta_{y_{k+1}},\theta_{y_{k+2}},\cdots), 
		\end{equation}
		so that we recognise $(\bar{\vt}_k)=M^k_{-(k+1)}$. From \eqref{acin} with $l=k+1$ we see that $(\bar{\vt}_k)=M^k_{-(k+1)}$  represents the state of ${\bf\Theta}^{(k+1) ,\Omega_N}(\bar{t})$.  
		
		If we look on time scale $\bar{t}+N^{k-1}t$, then we see that the active $(k-1)$-block averages feels a drift towards the active $k$-block average. The active $k$-block does not move on time scale $N^{k-1}$, but it has already moved at time $\bar{t}$. At time $\bar{t}$ the active $k$-block has even reached its quasi-equilibrium, given by $\Gamma^{(k)}_{\vt_k}$. Thus, the drift of the active $(k-1)$-block average is towards the instantaneous state of the active $k$-block average, which has distribution $\Gamma^{(k)}_{\bar{\vt_k}}$. This explains the first step in the interaction chain.
		
		For $0\leq l<k$, the active $l$-block average feels a drift towards the $(l+1)$-block average. The latter does not evolve on time scale $N^lt$, but it has already moved at time $\bar{t}$. Therefore it is no longer in its initial state, but in a quasi-equilibrium $\Gamma^{(l+2)}_{u}$, where $u$ is the value of the active $(l+1)$-block averages determined via the interaction chain, recall Figure~\ref{fig:int}.  This explains how the different space-time scales are connected via the active block averages. The states of the different seed-bank averages is a little bit more complicated. Below we give a very short heuristic explanation of the different seed-banks in the interaction chain.
		
		For the effective process $(\boldsymbol{\Theta}^{\eff,(l),\Omega_N}(t))_{t>0}$ instead of the full process $(\boldsymbol{\Theta}^{(l),\Omega_N}(t))_{t>0}$, we can consider in \eqref{463} only the active block average and the effective seed-bank average with $m=l$ (recall that the full block average equals the active block average). According to the above explanation, we have to replace $\Gamma^{(k)}_{\bar{\vt}_k}$ by $\Gamma^{\eff,(k)}_{\vt_k}$ and $\Gamma^{(l)}_{u}$ by $ \Gamma^{\eff, (l)}_u$. Hence we find the effective interaction chain defined in \eqref{fintchaineff} and depicted in Figure~\ref{fig:effint}. 
		
		\paragraph{Slow seed-banks.}
		
		From \eqref{acin} we see that if $l\geq k$ and we use the scaling $\bar{t}$, then all $l$-blocks of seed-banks with colour $0 \leq m \leq k$ equal $\vt_k$, and all $l$-blocks of seed-banks with colour $m>k$ equal their initial values $\theta_{y_m}$. Something interesting happens when we choose $0\leq l< k$ and use the scaling $\bar{t}+N^lt$. The single colonies of seed-banks with colour $0\leq m<l$ on time scale $\bar{t}+N^lt$ follow the active population, and hence their $l$-block averages equal the $l$-block average of the active population. The $l$-block average of the seed-bank with colour $l$ has a non-trivial interaction with the active $l$-block. The $l$-blocks of seed-banks with colour $m>k$ have not yet moved and hence are still in their initial states $(\theta_{y_m})_{m=k+1}^\infty$. However, \eqref{463b} implies that the single colonies of the seed-banks with colour $k\geq m>l$ are not moving on time scale $\bar{t}+N^l t$, even though they had already moved at time $\bar{t}$. Therefore the $l$-blocks averages of seed-banks with colour  $l<m\leq k$ are no longer in their initial state at time $\bar{t}$. Note that they are also not in the state $\vt_k$, since this is the state of their $k$-block averages and not of their $l$-block averages. The single colony seed-banks with colour $l< m\leq k$ are in the state given by
		\begin{equation}
			y_{m,0}(\bar{t}\,)=\int_0^{\bar{t}} \d s\,\frac{e_m}{N^m}[x_0(s)-y_{m,0}(s)].
		\end{equation}
		Hence, for large $N$,
		\begin{equation}
			y_{m,0}(\bar{t})\approx\int_0^{L(N)N^k+N^kt_k+\cdots+N^mt_m} \d s\,\frac{e_m}{N^m}[x_0(s)-y_{m,0}(s)].
		\end{equation}
		Similarly, for the $l$-block average with colour $m$ we have
		\begin{equation}
			\begin{aligned}
				\Theta^{(l),\Omega_N}_{y_m}(\bar{t}\,)
				&=\int_0^{\bar{t}} \d s\,\frac{e_m}{N^m}[\Theta^{(l),
					\Omega_N}_{x}(s)-\Theta^{(l),\Omega_N}_{y_m}(s)]\\\
				&\approx\int_0^{L(N)N^k+N^kt_k+\cdots+N^mt_m} \d s\,\frac{e_m}{N^m}[\Theta^{(l),\Omega_N}_{x}(s)
				-\Theta^{(l),\Omega_N}_{y_m}(s)].
			\end{aligned}
		\end{equation}
		Thus, we see that the state of $\Theta^{(l),\Omega_N}_{y_m}(\bar{t})$ is completely determined at time $L(N)N^k+N^kt_k+\cdots+N^mt_m$, i.e., the last time before $\bar{t}+N^lt$ that the single colony seed-banks of colour $m$ had an opportunity to move. Up to time $L(N)N^k+N^kt_k+\cdots+N^mt_m$, the single colony seed-banks with colour $m$ interact at a very slow rate with the active single colonies, and similarly for the $l$-blocks. Therefore effectively the colour-$m$ seed-bank interacts with a ``time-average on scale $N^mt_m$" of the active population. On time scale $N^mt_m$, a single active colony migrates very fast in its $(m-1)$-block. As a consequence at time $\bar{t}+N^mt_m$ individuals that start from a particular colony, e.g. site $0$, are spread uniformly over the $m$-block containing this site. Hence the interaction of a single $m$-dormant colony with the active population can be intuitively interpreted as an interaction with the active $m$-block, and similarly for an $m$-dormant $l$-block. Once we move to lower time scales, the $m$-dormant single colonies do not interact with the active colony anymore. In the detailed proofs we show that one consequence of this is that, for $l<m$, 
		\begin{equation}
			\begin{aligned}
				\Theta^{(l),\Omega_N}_{y_m}(\bar{t}+N^lt)&\approx\Theta^{(l),\Omega_N}_{y_m}(L(N)N^k+N^kt_k+\cdots+N^mt_m)\\
				&\approx \Theta^{(m),\Omega_N}_{y_m}(L(N)N^k+N^kt_k+\cdots+N^mt_m).
			\end{aligned}
		\end{equation} 
		Thus, the $l$-block averages of colours $l\leq m\leq k$ equal the state of the corresponding $m$-block. This is the $(m+2)$-th component of the interaction chain at level $l$.
		
		\paragraph{Conclusion.}
		
		Combining the intuitive descriptions in Sections \ref{sss.blockavh}-\ref{sss.formin}, we see how Theorems~\ref{T.multiscalehiereff} and Theorems~\ref{T.multiscalehier} come about. Their proofs will rely on coupling techniques and a detailed analysis of the SSDEs.  This analysis will be done in several steps. In Sections~\ref{ss.IntroMeanfield}--\ref{ss.pabstracts} we first deal with a simplified system, the mean-field model, for which we derive the McKean-Vlasov limit and the mean-field finite-systems scheme. In Sections~\ref{s.finlevel}--\ref{s.multilevel} we extend our analysis to finitely many hierarchical levels. In particular, we go through the following list of systems of increasing complexity, each being a simplified version of the system defined by \eqref{XYdef} and each capturing a key feature:    
		\begin{enumerate}
			\item Two-colour mean-field finite-systems scheme (Section~\ref{ss.tcmfs}).
			\item Two-level hierarchical mean-field system (Section~\ref{ss.tlhmfs*}).
			\item Finite-level mean-field system (Section~\ref{ss.multilevelmf}).
		\end{enumerate}
		In Section \ref{s.multilevel} we put the pieces together to prove the multi-scaling for the infinite-level system given in Theorems~\ref{T.multiscalehiereff} and \ref{T.multiscalehier}.
		
		\section{Main results $N\to\infty$: Orbit and cluster formation}
		\label{s.orbit}
		
		In the hierarchical mean-field limit we say that the system clusters when the colonies gradually form larger and larger mono-type blocks. In Section~\ref{ss.dichotomy} we determine whether, in the hierarchical mean-field limit, the system clusters along successive space-time scales. How this happens is captured by the interaction chain. We introduce a sequence of scaling factors $(A_k)_{k\in\N_0}$, where $A_k$ is defined in terms of the rates $(c_k)_{k\in\N_0}$, $(e_k)_{k\in\N_0}$, $(K_k)_{k\in\N_0}$ and the factors $(E_k)_{k\in\N_0}$. Using these scaling factors, we analyse the orbit of the renormalisation transformation and establish \emph{universality}: $A_k\CF^{(k)}g$ converges as $k\to\infty$ to the Fisher-Wright diffusion function, irrespective of the choice of $g$. In Section~\ref{ss.growclus} we show how the scaling factors $A_k$ are connected to the \emph{growth of mono-type clusters}. In Section~\ref{ss.scalrate} we identify the asymptotics of $A_k$ as $k\to\infty$ in terms of the model parameters. 
		
		\subsection{Orbit of renormalisation transformations}
		\label{ss.dichotomy}
		
		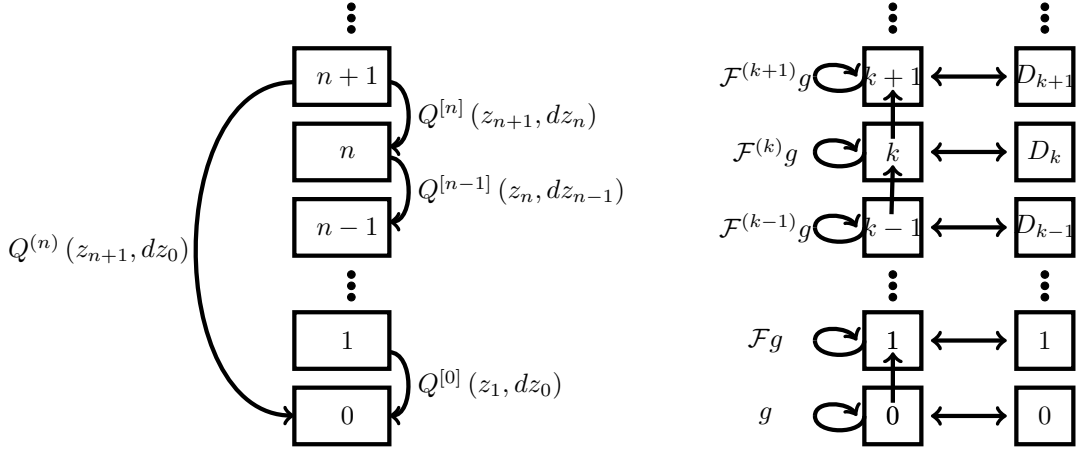
\begin{figure}[htbp]
			\vspace{0.3cm}
			\centering
			\begin{tikzpicture}
				\foreach \x in {1,2,3}
				\draw [ultra thick] (-.5-2,\x) rectangle (0.75-2,\x+0.75);
				\node at (0.2-2,1.375) {$n-1$};
				\node at (0.2-2,2.375) {$n$};
				\node at (0.2-2,3.375) {$n+1$};
				\foreach \x in {4.15,4,4.30,0.75,0.6,0.45}
				\draw[fill] (.25-2,\x) circle [radius=0.05];
				\draw [ultra thick] (-0.5-2,-.75) rectangle (0.75-2,-1.5);
				\node at (0.2-2, -1.125){0};
				\draw [ultra thick] (-0.5-2,.25) rectangle (0.75-2,-.5);
				\node at (0.2-2, -.125){1};
				\draw [ultra thick,->] (.75-2,2.3)to [out=0,in=0](.75-2, 1.45);
				\draw [ultra thick,->] (.75-2,3.3)to [out=0,in=0](.75-2, 2.45);
				\draw [ultra thick,->] (.75-2,-.28)to [out=0,in=0](.75-2, -1.12);
				\draw [ultra thick,->] (-.5-2,3.3)to [out=180,in=180](-0.5-2,-1.12 );
				\node[left] at (-1.75-2,1.09){$Q^{(n)}\left(z_{n+1},d z_0\right)$};
				\node[right] at (1-2,2.875){$Q^{[n]}\left(z_{n+1},d z_n\right)$};
				\node[right] at (1-2,1.875){$Q^{[n-1]}\left(z_{n},d z_{n-1}\right)$};
				\node[right] at (1-2,-.7){$Q^{[0]}\left(z_{1},d z_{0}\right)$};
				
				\foreach \x in {1,2,3}
				\draw [ultra thick] (5,\x) rectangle (5.75,\x+0.75);
				\foreach \x in {1,2,3}
				\draw [ultra thick] (7,\x) rectangle (7.75,\x+0.75);
				\node at (5.375,1.375) {$k-1$};
				\node at (5.375,2.375) {$k$};
				\node at (5.375,3.375) {$k+1$};
				\node at (7.375,1.375) {$D_{k-1}$};
				\node at (7.375,2.375) {$D_k$};
				\node at (7.375,3.375) {$D_{k+1}$};
				\foreach \x in {4.15,4,4.30,0.75,0.6,0.45}
				\draw[fill] (5.375,\x) circle [radius=0.05];
				\foreach \x in {4.15,4,4.30,0.75,0.6,0.45}
				\draw[fill] (7.375,\x) circle [radius=0.05];
				\draw [ultra thick] (5,-.75) rectangle (5.75,-1.5);
				\node at (5.375, -1.125){0};
				\draw [ultra thick] (5,.25) rectangle (5.75,-.5);
				\node at (5.375, -.125){1};
				\draw [ultra thick] (7,-.75) rectangle (7.75,-1.5);
				\node at (5.375, -1.125){0};
				\draw [ultra thick] (7,.25) rectangle (7.75,-.5);
				\node at (5.375, -.125){1};
				\draw [ultra thick,<-] (5.375,2.2)--(5.357, 1.55);
				\draw [ultra thick,<-] (5.375,3.2)--(5.375, 2.55);
				\draw [ultra thick,->] (5.375, -.95)--(5.375, -.28);
				\draw [ultra thick,<->] (5.875,1.375)--(6.875, 1.375);
				\draw [ultra thick,<->] (5.875,2.375)--(6.875, 2.375);
				\draw [ultra thick,<->] (5.875,3.375)--(6.875, 3.375);
				\draw [ultra thick,<->] (5.875, -.125)--(6.875, -.125);
				\draw [ultra thick,<->] (5.875, -1.125)--(6.875, -1.125);
				\node at (7.375, -.125){1};
				\node at (7.375, -1.125){0};
				\foreach \x in {-.5,-1.5,1,2,3}
				{\draw [ultra thick, -](5,\x+.375)to[out=270,in=270](4.35,\x+.375);
					\draw [ultra thick, <-](4.9,\x+.375)to[out=120,in=90](4.35,\x+.375);
				}
				\node at (3.7,1+.375){$\mathcal{F}^{(k-1)}g$};
				\node at (3.7,2+.375){$\mathcal{F}^{(k)}g$};
				\node at (3.7,3+.375){$\mathcal{F}^{(k+1)}g$};
				\node at (3.7,-.5+.375){$\mathcal{F}g$};
				\node at (3.7,-1.5+.375){$g$};
			\end{tikzpicture}
			\vspace{0.3cm}
			\caption{\emph{Left:} The interaction chain that connects successive hierarchical levels downwards. The arrows on the right correspond to \eqref{ker}, the arrow on the left corresponds to \eqref{conn2}. \emph{Right:} The renormalisation transformation that connects successive hierarchical levels upwards. The vertical arrows correspond to \eqref{frenormit}. The horizontal arrows represent the interaction with the effective seed-bank. The arrows on the left represent the resampling driven by the renormalised diffusion function.}
			\label{fig-intchain}
		\end{figure}
		
		To determine whether clustering occurs, we start from larger and larger time scales and use the interaction chain to see whether mono-type clusters are formed in the single colonies. Recall the kernels introduced in \eqref{ker} that describe the connection between subsequent hierarchical levels in the interaction chain. Define the following composition of kernels (see Fig.~\ref{fig-intchain}):
		\begin{equation} 
			\label{conn2}
			\gls{Qn}= Q^{[n]} \circ\cdots\circ Q^{[0]}, \qquad n\in\N.
		\end{equation}
		In words, $Q^{(n)}(z_n,\d z_0)$ is the probability density to see the population of a single colony in state $z_0$ given that the $(n+1)$-block average equals $z_n$. 
		
		In Section~\ref{s.mainfinite} we identified the clustering regime for fixed $N<\infty$. In this section we identify the clustering regime in the hierarchical mean-field limit. In the clustering regime, in the hierarchical mean-field limit, an interesting question is to determine how $\CF^{(n)}g$ (recall \eqref{frenormit}) scales with $n$. We identify the scaling and show that it does \emph{not} depend on $g$ (see Fig.~\ref{fig-flow}). 
		
		To state the clustering result, abbreviate
		\begin{equation}
			\label{seqtheta}
			\gls{thetan}=(\vartheta_n,\overbrace{\vartheta_n,\cdots, \vartheta_n}^{n+1\text{ times }},
			\theta_{y_{n+1}},\theta_{y_{n+2}},\cdots).
		\end{equation}
		
		\begin{theorem}{{\bf [Renormalised scaling]}}
			\label{T.scalren} 
			Let $c_k$ be as in \eqref{738}, $e_k$ and $K_k$ as in \eqref{defKem} and $E_k$ as in \eqref{Ekdef}. Define 
			\begin{equation} 
				\label{modefA}
				\gls{An} =\frac{1}{2}\sum_{k=0}^{n-1}\frac{E_k}{c_k}
				\frac{(E_kc_k+e_k)}{(E_kc_k+e_k)+E_k K_k e_k}, \qquad n\in\N.
			\end{equation}
			Then 
			\begin{equation}
				\label{limKnnu}
				\lim_{n\to\infty} Q^{(n)}\bigl(\bar{\vt}^{(n)},\,\cdot\,\bigr) 
				= (1-\theta)\,\delta_{(0,0^{\N_0})} + \theta\,\delta_{(1,1^{\N_0})}
			\end{equation}
			if and only if 
			\begin{equation}
				\label{cluscond}
				\lim_{n\to\infty} A_n=\infty.
			\end{equation}
			Moreover, if \eqref{cluscond} holds, then for all $g\in\CG$, 
			\begin{equation}
				\label{limcf}
				\lim_{n\to\infty} A_n\CF^{(n)}g=g_{\mathrm{FW}} \quad \text{pointwise},
			\end{equation}
			with $g_{\mathrm{FW}}(x)=x(1-x)$, $x \in [0,1]$. 
		\end{theorem} 
		
		\noindent
		The proof of Theorem~\ref{T.scalren} is given in Section~\ref{s.renormasym}. The scaling factors $A_n$ can be interpreted as clustering coefficients: in Section~\ref{ss.growclus} we will show that the faster the $A_n$ grow to infinity, the faster we expect to see clusters grow. The property in \eqref{limKnnu} corresponds to the \emph{clustering regime}. According to \eqref{limcf}, even though $A_n$ depends on the choice of the sequences $\underline{K},\underline{e},\underline{c}$ in \eqref{738} and \eqref{defKem}, the limit $A_n\CF^{(n)}g$ as $n\to\infty$ is \emph{universal}: irrespective of the choice of $g\in\CG$, the limit is the standard Fisher-Wright diffusion function $g_{\text{FW}}$. Thus, $g_{\text{FW}}$ is the \emph{global attractor of the renormalisation transformation} (see Fig.~\ref{fig-flow}). 
		
		\begin{figure}[htbp]
			\vspace{-.5cm}
			\begin{center}
				\setlength{\unitlength}{0.6cm}
				\begin{picture}(10,8)(0,0)
					{\thicklines
						\qbezier(0,0)(0.5,2)(2,4)
						\qbezier(2,4)(3,5.3)(4,6)
						\qbezier(4,6)(5,6.5)(6,6)
						\qbezier(6,6)(7,5.3)(8,4)
						\qbezier(8,4)(9,2)(6,1)
					}
					\qbezier[10](6,1)(5.5,0.8)(4.8,0.88)
					\qbezier[5](4.8,0.88)(4.9,0.92)(5.0,0.96)
					\qbezier[5](4.8,0.88)(4.9,0.84)(5.0,0.80)
					\put(0,0){\circle*{.25}}
					\put(2,4){\circle*{.25}}
					\put(4,6){\circle*{.25}}
					\put(6,6){\circle*{.25}}
					\put(8,4){\circle*{.25}}
					\put(4,1){\circle*{.25}}
					\put(-.7,0){$g$}
					\put(-.5,4){$A_1\CF^{(1)}g$}
					\put(1.5,6){$A_2\CF^{(2)}g$}
					\put(6.4,6){$A_3\CF^{(3)}g$}
					\put(8.4,4){$A_4\CF^{(4)}g$}
					\put(2.5,1){$g_{\text{FW}}$}
				\end{picture}
			\end{center}
			\caption{Flow of the iterates $\CF^{(n)}g$, $n\in\N_0$, of the renormalisation transformation acting 
				on the class $\CG$. After multiplication by $A_n$, the flow is globally attracted by $g_{\text{FW}}$.} 
			\label{fig-flow}
		\end{figure}
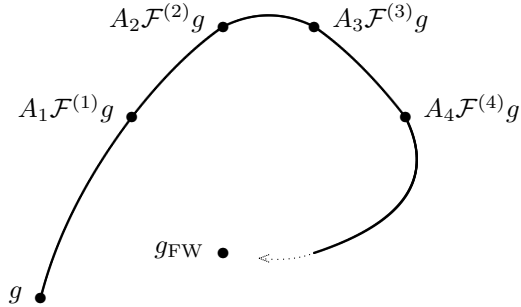

		\subsection{Growth of mono-type clusters}
		\label{ss.growclus}
		
		In the \emph{clustering regime} we are interested in how fast mono-type clusters grow in space over time. For the system on $\Z^d$ and $\Omega_N$ without seed-bank the growth rate has been studied in detail. Different growth rates were found for \emph{strongly recurrent} and \emph{critically recurrent} migration. Typical examples on $\Omega_N$ are migrations with coefficients $c_k=c^k$ with $c \in (0,1)$, respectively, $c_k = C$ with $C \in (0,\infty)$. Typical examples on $\Z$, respectively, $\Z^2$ are migrations with zero mean and finite variance. For these models the following behaviour occurs.
		\begin{itemize}
			\item
			In the strongly recurrent case, mono-type clusters grow fast and cover a volume that increases at time $t$ at a rate that is given by the Green function up to time $t$ of the underlying random walk, times a certain random constant that can be determined explicitly and that is independent of the diffusion function $g \in \CG$ \cite{EF96}, \cite{K96}. The cluster growth is monitored by considering families of balls growing in time at such a speed that, starting from a translation invariant and ergodic initial state, the mean of the configuration in the ball is still close to the starting mean but begins to move. \emph{Fast clustering} means that the cluster covers multiples of a scale that eventually lies in every finite family of balls with the above property.  
			\item
			In the critically recurrent case, the volume grows only moderately fast, like $N^{(1-U) t}$ as $t\to\infty$, with $U \in [0,1]$ a random variable. In other words, the cluster sizes have random orders of magnitude, an effect known as \emph{diffusive clustering}. For $c_k = C \in (0,\infty)$, $k\in\N_0$, the distribution of $U$ can be identified by studying the fraction of active individuals of type $1$ in a ball of size $N^{(1-u) t}$, which can be shown to converge to $V(\log\tfrac{1}{1-u})$ as $t\to\infty$ with $(V(s))_{s \geq 0}$ the standard Fisher-Wright diffusion, irrespective of the choice of $g\in\CG$ \cite{FG94}, \cite{FG96}.
			\item
			For more general migration it is possible that mono-type clusters grow slower than any positive power of $t$ as $t\to\infty$. This occurs for recurrent migration in which the Green function up to time $t$ grows like $o(\log t)$. For this regime only few results are available \cite{DGV95}.    
		\end{itemize}
		
		From the perspective of explaining \emph{universality} in $g\in\CG$ in the hierarchical mean-field limit $N\to\infty$, the above type of behaviour has been studied in detail in \cite{DGV95} and \cite{GHKK14} for the Fleming-Viot model, respectively, the Cannings model without seed-bank. The renormalisation analysis for the model with seed-bank allows us to study how the seed-bank affects the cluster growth. In what follows we give a sketch of \emph{three regimes of cluster growth}. 
		
		\paragraph{Types of clustering.}
		
		If a ball in $\Omega_N$ lies in a mono-type cluster, then the block average of the active and the dormant components in this ball are all close to either $0$ or $1$. We can therefore analyse the growth rate of mono-type clusters by analysing at which hierarchical level block averages hit $0$ or $1$ in the limit as $N\to\infty$. To that end, we look at the interaction chain $M^k_{-l(k)}$ for $k\to\infty$, where the \emph{level scaling function} $l\colon\,\N_0\to\N_0$ is non-decreasing with $\lim_{k\to\infty} l(k)=\infty$ and is suitably chosen such that we obtain a \emph{non-trivial clustering limiting law}, i.e., 
		\begin{equation}
			\label{icsl}
			\lim_{k \to \infty} 
			\mathcal{L} \bigl[M^k_{-l(k)} \bigr] = \mathcal{L} \bigl[ \hat{\theta}\, \bigr],
		\end{equation}
		where the limiting sequence of random frequencies $\hat{\theta}$ satisfies
		\begin{equation}
			\label{cic}
			0 < \P\bigl(\hat{\theta} \in \{0^{\N_0},1^{\N_0}\}\bigr) < 1.
		\end{equation}
		In line with~\cite{DGV95} and~\cite{DG96}, in order to analyse the growth of mono-type clusters on multiple space-time scales in the hierarchical mean-field limit, it is natural to consider a family of non-decreasing functions $l_\chi\colon\,\N_0 \to \N_0$, $\chi \in I \subseteq [0,\infty)$, called the \emph{cluster scales}, satisfying \eqref{icsl}--\eqref{cic}: 
		\begin{itemize}
			\item[\textup{(1)}] {\bf Fast clustering:} 
			$\lim_{k\to\infty} l_\chi(k)/k=1$ for all $\chi \in I$.
			\item[\textup{(2)}] {\bf Diffusive clustering:} 
			$\lim_{k\to\infty} l_\chi(k)/k=\kappa(\chi)$ for all $\chi \in [0,1]$, where $\chi\mapsto\kappa(\chi)$ is 
			continuous and non-increasing with $\kappa(0)=1$ and $\kappa(1)=0$. 
			\item[\textup{(3)}] {\bf Slow clustering:} 
			$\lim_{k\to\infty} l_\chi(k)/k=0$ for all $\chi \in I$. (This regime borders with the regime of 
			coexistence.)
		\end{itemize}
		We write $(M^\infty_\chi)_{\chi\in I}$ with  $M^\infty_\chi= \lim_{k\to\infty} M^k_{-l_\chi(k)}$ to denote the \emph{cluster process}. 
		
		\begin{remark}
			{\rm Examples are:
				\begin{itemize}
					\item[\textup{(1)}] $I=\N_0$, $l_\chi(k)=k-\chi$. 
					\item[\textup{(2)}] $I=[0,1]$, $l_\chi(k)=\lfloor (1-\chi)k \rfloor$.
					\item[\textup{(3)}] $I=[0,1]$, $l_\chi(k)=\lfloor L(k^{1-\chi})\rfloor$ with $L(0)=0$, $L$ non-decreasing and sublinear. 
				\end{itemize}
				In words, the clusters cover blocks of level: (1) $k-\chi$; (2) $\lfloor (1-\chi)k \rfloor$; (3) $L(k^{1-\chi})$. For the model without seed-bank and with migration coefficients $c_k=c^k$ with $c \in (0,1)$, case (1) is realised with a Markov chain $(M^\infty_l)_{l\in\N_0}$ as scaling limit, while for $c_k= C$, case (2) is realised with a time-transformed Fisher-Wright diffusion in $\chi$ as scaling limit. (For finite $N$, this corresponds to the first and the second example given in the first paragraph of this section.) For the model without seed-bank, these scales have been shown to satisfy the required conditions. Case (3) also appears for the model without seed-bank, but detailed information on scales and scaling limits is lacking. As we will see below, seed-banks can slow down cluster growth, so case (3) is worthwhile to be studied in more detail.     
			} \hfill$\blacksquare$ 
		\end{remark}
		
		Recall \eqref{modefA}. Fast clustering corresponds to $A_k \gg k$, diffusive clustering to $A_k \asymp k$, and slow clustering to $A_k \ll k$ for large $k$. Theorem~\ref{T.dichotomy} below shows that, subject to \eqref{regvar} and \eqref{pureexp}, all three regimes are possible for the model with seed-bank. The regimes are the same as for the model without seed-bank when $\rho<\infty$, but different when $\rho=\infty$.    
		
		For systems without seed-bank, examples of the three types of clustering can be found in the literature: diffusive clustering in \cite{A79}, \cite{CG86} (voter model on $\Z$, respectively, $\Z^2$) and in \cite{DG93b}, \cite{FG94}, \cite{DGV95}, \cite{K96} (interacting Fleming-Viot processes on $\Omega_N$ with $N<\infty$, respectively, $N\to\infty$), all types of clustering in \cite{DG96}, \cite{K97} , \cite{W02} (interacting Feller diffusions on $\Omega_N$ with $N\to\infty$) and in \cite{GHKK14}, \cite{GHK18} (interacting Cannings processes on $\Omega_N$ in non-random and random environment with $N\to\infty$). 
		
		For the model with seed-bank we have to use the asymptotics of $(A_k)_{k\in\N}$ to identify the set $I$ and the family $(l_\chi(\cdot))_{\chi \in I}$, and show that $(M^k_{-l_\chi(k)})_{\chi \in I}$ converges as $k\to\infty$ to a Markov process, which we want to identify.     
		
		\paragraph{Computations.}
		
		In the following we demonstrate how we can carry out the above task. The key idea is to study first and second moments of the interaction chain, as well as sums of variances in order to get a handle on the quadratic variation process. To that end  we calculate 
		\begin{equation}
			V^k_{-l} = \text{ the conditional variance of the active part of } M^k_{-l} \text{ given } M^k_{-(l+1)}
		\end{equation}
		and consider the sum of random variables $A_{k,n} = \sum_{-(k+1) \leq -l \leq -n} V^k_{-l}$, $n \in \N_0$. In order for the system to cluster, we must have $\lim_{k\to\infty} A_{k,n} = 0$ for every $n \in \N_0$. The \emph{volatility profile} is given by 
		\begin{equation}
			(p_\chi(k))_{\chi \in [0,1]}, \qquad p_\chi(k) = A_{k,l_\chi(k)}/A_{k,0}.
		\end{equation} 
		This profile is a \emph{random variable} that depends on the interaction chain up to $M^k_{-l_\chi(k)}$. Since $A_{k,l_\chi(k)} = A_{k,0} - A_{l_\chi(k)-1,0}$, we have $p_\chi(k) = (A_{k,0}-A_{l_\chi(k)-1,0})/A_{k,0}$. For diffusive clustering, for instance, we want to show that
		\begin{equation}
			\lim_{k\to\infty} p_\chi(k) = 1 - \kappa(\chi), \qquad \chi \in [0,1], 
		\end{equation}
		while for fast clustering the limit is $0$ and for slow clustering the limit is $1$. From \eqref{limcf} we know that the scaled renormalised diffusion function $A_n\CF^{(n)}(g)$ tends to the standard Fisher-Wright diffusion function as $n\to\infty$. Since the latter hits the boundary $\{0,1\}$ after some finite time, the coefficients $A_n$ describe the speed at which the interaction chain hits this boundary. We next make this idea precise and show how it can be used to obtain information about the growth of mono-type clusters.
		
		The kernels defined in Section \ref{ss.dichotomy} allow us to compute the first and second moments of all the block averages, which will be done in Section \ref{ss.momrels} (Propositions \ref{P.momrel}--\ref{P.momreliterate2}). In particular, using the interaction chain starting between at $-n$ and running until $-m$ with $-n<-m \leq 0$, and considering the $m$-block averages on time scale $N^mt$ in the limit $N\to\infty$, we find that the variance of the active component $x^n_m$ of $M^n_{-m}$ equals 
		\begin{equation}
			\label{varsum}
			\mathbb{V}\mathrm{ar}(x^n_m) = \E\left[(x^n_m-\vartheta_n)^2\right] = A_m^n(\CF^{(n+1)}g)(\vartheta_n),
		\end{equation}
		where 
		\begin{equation}
			\gls{Amn} = \frac{1}{2}\sum_{k=m}^n\frac{E_k}{c_k}\frac{(E_kc_k+e_k)}{(E_kc_k+e_k)+E_kK_ke_k}.
		\end{equation}
		(Note that $A_n=A^{n-1}_0$.) On the other hand, since $x^n_m\in(0,1)$ we have $\mathbb{V}\mathrm{ar}(x^n_m) \in (0,1)$ and $A^n_m(\CF^{(n+1)})(\vartheta_n)\in(0,1)$. Taking $m=0$, we get $(\CF^{(n+1)}g)(\vartheta_n)\in(0,\frac{1}{A^n_0})$. This implies that 
		\begin{equation}
			(\CF^{(n+1)}g)(\vartheta_n) = \int_{[0,1]^2} (\CF^m g)(x_m)\,Q^{(n)}_m((\vartheta_n,\theta_{y,n}),\d z_m)
			\in \left(0,\tfrac{1}{A^n_0}\right). 
		\end{equation}
		Since $\lim_{n\to\infty}A_n(\CF^{(n+1)}g)=g_{\text{FW}}$, for large enough $m,n$ we can approximate $\CF^{(m)}g\approx g_{\text{FW}}/A^m_0$. Therefore 
		\begin{equation}
			\int \frac{g_{\text{FW}}}{A^m_0}(x_m)\,Q^{(n)}_m((\vartheta_n,\theta_{y,n}),\d z_m) \in \left(0,\tfrac{1}{A^n_0}\right),
		\end{equation}
		or, equivalently, 
		\begin{equation}
			\int_{[0,1]^2} x_m(1-x_m)\,Q^{(n)}_m((\vartheta_n,\theta_{y,n}),\d z_m) \in \left(0,\tfrac{A^m_0}{A^n_0}\right).
		\end{equation} 
		Hence, if $A^m_0/A^n_0<\epsilon$ with $\epsilon>0$ small, then we know that with high probability the system on time scale $n$ has clusters with a radius of size $m$. (Note that for the interaction chain this means that the variance is almost entirely centred between $n$ and $m$.)  Therefore the speed at which $A^m_0/A^n_0$ converges to zero as $m,n\to\infty$ says something about the speed at which monotype clusters form.
		
		To capture the cluster growth, we must decide how we let $m,n\to\infty$. For this we look for clusters of radius $l_\chi(n)$ with $\chi \in I$. Put
		\begin{equation}
			f^n(l_\chi(n))=\frac{A^{l_\chi(n)}_0}{A^n_0},
		\end{equation}
		and define, for $\epsilon>0$,
		\begin{equation}
			\CX^n_{\epsilon}=\inf\{\chi \in I\colon\,f^n(l_\chi(n))<\epsilon\}.
		\end{equation}
		Then the three types of clustering correspond to:
		\begin{itemize}
			\item[\textup{(1)}] {\bf Fast clustering:} 
			$\lim_{n\to\infty} l_{\CX^n_{\epsilon}}(n)/n=1$. 
			\item[\textup{(2)}] {\bf Diffusive clustering:} 
			$\lim_{n\to\infty} l_{\CX^n_{\epsilon}}(n)/n=R$ for some random variable $R$ taking values in $(0,1)$. 
			\item[\textup{(3)}] {\bf Slow clustering:}
			$\lim_{n\to\infty} l_{\CX^n_{\epsilon}}(n)/n=0$.
		\end{itemize}
		In terms of the interaction chain starting from $-k$ with $k\to\infty$, in view of \eqref{varsum} this corresponds to the variance in the interaction chain being concentrated near the beginning, being spread out or being concentrated near the end.

		\subsection{Rates of scaling for renormalised diffusion function} 
		\label{ss.scalrate} 
		
		For the system without seed-bank, we have $K_k=e_k=0$ and $E_k=1$ for all $k\in\N_0$. Hence
		\begin{equation}
			\label{Anscalnoseed}
			A_n = \frac{1}{2} \sum_{k=0}^{n-1} \frac{1}{c_k}
		\end{equation}
		and \eqref{limKnnu} holds if and only if 
		\begin{equation}
			\label{cluscondnoseed}
			\sum_{k\in \N_0} \frac{1}{c_k} = \infty.
		\end{equation} 
		Various subcases were analysed in \cite{BCGH95}. For the system with seed-bank because $E_0=1$ and $E_k<1$ (see \eqref{Ekdef}), it follows from \eqref{modefA} that 
		\begin{equation}
			A_n < \frac{1}{2} \sum_{k=0}^{n-1} \frac{1}{c_k}.
		\end{equation}
		Thus we see that the seed-bank \emph{weakens clustering}, i.e., enhances genetic diversity, even in the hierarchical mean-field limit. 
		
		We identify the clustering regime in the setting where the coefficients are \emph{asymptotically polynomial}, as in \eqref{regvar}, or are \emph{pure exponential}, as in \eqref{pureexp}. It turns out that there is a delicate interplay between the migration and the seed-bank, resulting in 4 different scalings for asymptotically polynomial coefficients and 8 different scalings for pure exponential coefficients.
		
		\begin{theorem}{\bf [Rates of scaling for diffusion function]} Let $\rho$ be as defined in \eqref{rhodef}.
			\label{T.dichotomy}
			\begin{itemize}
				\item[{\rm (I)}] 
				If $\rho<\infty$, then \eqref{cluscond} holds if and only \eqref{cluscondnoseed} hold, and
				\begin{equation}
					A_n \sim \frac{1}{2(1+\rho)} \sum_{m=0}^{k-1} \frac{1}{c_k}. 
				\end{equation} 
				\item[{\rm (II)}] 
				If $\rho=\infty$, then \eqref{cluscond} holds in the following cases:
				\begin{itemize}
					\item[$\bullet$] 
					Subject to \eqref{regvar} if and only if $-\phi \leq \alpha \leq 1$, with
					\begin{equation}
						\begin{aligned}
							&-\phi < \alpha < 1\colon  &&A_n \sim C_1\,n^{\alpha+\phi},\\
							&-\phi = \alpha < 1\colon  &&A_n \sim C_2\,\log n,\\ 
							&-\phi < \alpha = 1\colon  &&A_n \sim C_3\,\frac{n^{1+\phi}}{\log n},\\
							&-\phi = \alpha = 1\colon  &&A_n \sim C_4\,\log\log n,
						\end{aligned}
					\end{equation}
					where
					\begin{equation}
						\begin{array}{ll}
							&C_1 = \frac{1}{2AF} \frac{1-\alpha}{\alpha+\phi}, \quad C_2 = \frac{1}{2AF} (1-\alpha), 
							\quad C_3 = \frac{1}{2AF} \frac{1}{1+\phi}, \quad C_4 = \frac{1}{2AF}.\\
							&
						\end{array}
					\end{equation}
					The values of $B,\beta$ play no role for the clustering, nor for the asymptotics.
					\item[$\bullet$] 
					Subject to \eqref{pureexp} if and only if $Kc \leq 1 \leq K$, with
					\begin{equation}
						\begin{aligned}
							&c<Ke,\,Kc<1\colon           &&A_n \sim \hat{C}_1\,(Kc)^{-(n-1)},\\
							&c<Ke,\,Kc=1\colon           &&A_n \sim \bar{C}_1\,n,\\
							&c=Ke,\,Kc<1\colon           &&A_n \sim\hat{C}_2\,(Kc)^{-(n-1)},\\
							&c=Ke,\,Kc=1\colon           &&A_n \sim \bar{C}_2\,n,\\
							&c>Ke,\,Kc<1\colon           &&A_n \sim \hat{C}_3\,(Kc)^{-(n-1)},\\
							&c>Ke,\,Kc=1\colon           &&A_n \sim \bar{C}_3\,n,\\
							&c<1=K\colon                    &&A_n \sim \tilde{C}_1\,n^{-1}\,c^{-(n-1)},\\
							&c=1=K\colon                    &&A_n \sim \tilde{C}_2 \log n,
						\end{aligned}
					\end{equation}
					where
					\begin{equation}
						\begin{array}{llll}
							&\hat{C}_1 = \frac{K-1}{2K(1-Kc)}, &\hat{C}_2 = \frac{(K-1)^2}{2(2K-1)(1-Kc)}, 
							&\hat{C}_3 =  \frac{K-1}{2(1-Kc)},\\[0.3cm]
							&\bar{C}_1 = \frac{K-1}{2K}, &\bar{C}_2 = \frac{(K-1)^2}{2(2K-1)}, 
							&\bar{C}_3 =  \frac{K-1}{2},\\[0.3cm]
							&\tilde{C}_1 = \frac{1}{2(1-c)}, &\tilde{C_2} = \frac{1}{2}. &
						\end{array}
					\end{equation}
					The value of $e$ plays no role for the clustering, but does for the asymptotics. 
				\end{itemize}
			\end{itemize} 
		\end{theorem}
		
		\noindent
		The proof of Theorem~\ref{T.dichotomy} is given in Section~\ref{s.renormasym}. Part (I) shows that for $\rho<\infty$ the clustering regime is the same as for the system without seed-bank. The scaling of $A_n$ is controlled by the migration and is reduced by a factor $1/(1+\rho)$ with respect to the seed-bank. Part (II) shows that for $\rho=\infty$ the clustering regime is different from that for the system without seed-bank. Clustering is harder to achieve: since $\lim_{k\to\infty} E_k = 0$ the growth rate of $A_n$ is \emph{strictly smaller} than without seed-bank.
		
		Furthermore, subject to \eqref{regvar}, if $-\phi<\alpha<1$, then the growth rate of $A_n$ drops down from $\asymp n^{1+\phi}$ without seed-bank to $\asymp n^{\alpha+\phi}$ with seed-bank, while if $-\phi=\alpha=1$, then it drops down from $\asymp \log n$ to $\asymp \log \log n$. Similarly, subject to \eqref{pureexp}, if $Kc<1<K$, then the growth rate of $A_n$ drops down from $\asymp c^{-n}$ to $\asymp (Kc)^{-n}$, while if $c=K=1$, then it drops down from $\asymp n$ to $\asymp \log n$.  
		
		Returning to the observations made in Section~\ref{ss.growclus}, we see that the three clustering regimes also appear in the model with seed-bank, both for $\rho<\infty$ and $\rho=\infty$, and in the latter case are accompanied by different migration coefficients. The scaling results mentioned in Section~\ref{ss.growclus} can in principle be deduced from the asymptotics of $A_n$ as $n\to\infty$ in Theorem~\ref{T.dichotomy}. It would be interesting to work out the details and to identify the limiting processes that control the cluster growth.

		\part{PREPARATIONS AND PROOFS}

		\section{Proofs: $N<\infty$, identification of clustering regime}
		\label{s.clusreg}
		
		In this section we prove Theorems~\ref{T.scalcoeff}--\ref{T.cluscritreg}. The integral criterion for $\rho=\infty$ in \eqref{cluscritseed-b} is explained in Section~\ref{ss.intcrit}. Theorem~\ref{T.scalcoeff} is proved in Section~\ref{ss.scal} and Theorem~\ref{T.cluscritreg} in Section~\ref{ss.cluspr}. 
		
		\subsection{Explanation of clustering criterion for infinite seed-bank}
		\label{ss.intcrit} 
		
		Recall Fig.~\ref{fig:per}. Suppose that $g=dg_{\mathrm{FW}}$, so that we have a dual. We will show that the integral criterion in \eqref{cluscritseed-b} determines whether or not two dual lineages coalesce with probability 1. Since two lineages in the dual can only coalesce when they are active at the same site, we need to keep track of the probabilities that the lineages are active at a given time. Because the lineages can only migrate when they are active, we also need to keep track of the total time they are active up to a given time.
		
		Recall the renewal interpretation of the dual process (see Remark~\ref{wakeup}). We argue heuristically as follows. If $\rho=\infty$, then the activity times $\sigma_k$ are much smaller than the sleeping times $\tau_k$, and we may assume that $\tau_k+\sigma_k \asymp \tau_k$, $k\to\infty$. Discretising time, we can use the results from \cite{AB16} for the intersection of two independent renewal processes. Then the integral criterion in \eqref{cluscritseed-b} can be interpreted as follows:
		\begin{itemize}
			\item
			If $\gamma \in (0,1)$, then the probability for each of the lineages to be active at time $s$ decays like $\asymp \varphi(s)^{-1} s^{-(1-\gamma)}$ \cite{AB16}. Hence the total time they are active up to time $s$ is $\asymp \varphi(s)^{-1} s^\gamma$. Because the lineages only move when they are active, the probability that the two lineages meet at time $s$ is $\asymp a^{(N)}_{\varphi(s)^{-1} s^\gamma}(0,0)$. Hence the total hazard is $\asymp \int_1^\infty \d s\, [\varphi(s)^{-1}s^{-(1-\gamma)}]^2\,a^{(N)}_{\varphi(s)^{-1} s^\gamma}(0,0)$. After the transformation $t=t(s)=\varphi(s)^{-1} s^\gamma$, the latter turns into the integral in \eqref{cluscritseed-b}, modulo a constant. When carrying out this transformation, we need that $s\varphi'(s)/\varphi(s) \to 0$, which follows from \eqref{hatphirepr}, and $\varphi(t(s))/\varphi(s) \asymp 1$, which follows from the bound we imposed on $\psi$ in \eqref{hatphirepr} together with the fact that $\log \varphi(s)/\log s \to 0$. This computation is spelled out in Appendix~\ref{app.comp}.  
			\item
			If $\gamma = 1$, then the probability for each of the lineages to be active at time $s$ decays like $\hat\varphi(s)^{-1}$ \cite{AB16}, and so the total time they are active up to time $s$ is $\asymp s \hat\varphi(s)^{-1}$. Recall from \eqref{hatphidefplus} that $\hat{\varphi}(t)=\E[\tau\wedge t]$ is also slowly varying.) Hence the total hazard is $\asymp \int_1^\infty \d s\, [\hat\varphi(s)^{-1}]^2 \,a^{(N)}_{\hat\varphi(s)^{-1}s}(0,0)$. After the transformation $t=t(s)=\hat\varphi(s)^{-1}s$ (for which we can use the same type of computation as in Appendix~\ref{app.comp}), the latter turns into the integral in \eqref{cluscritseed-b}, modulo a constant.   
		\end{itemize}
		
		\subsection{Scaling of wake-up time and migration kernel for infinite seed-bank}
		\label{ss.scal}
		
		We can prove Theorem~\ref{T.scalcoeff} by direct computation via assumptions \eqref{regvar}--\eqref{pureexp}. We start by computing $\gamma$. Afterwards we compute $\hat{\varphi}(t)$ and $a_t^{\Omega_N}(0,0)$.
		
		\paragraph{Computation of $\gamma$.}
		
		Recall \eqref{sigtau}, which reads 
		\begin{equation}
			\label{probtau}
			\P(\tau>t) = \frac{1}{\chi} \sum_{m\in\N_0} K_m \frac{e_m}{N^m}\,\e^{-(e_m/N^m) t}.
		\end{equation} 
		Since we are interested in the asymptotic behaviour of $\P(\tau>t) $ as $t\to\infty$, we need to consider only large values of $t$. For large values of $t$, only large values of $m$ (for which $\frac{e_m}{N^m}$ is small) contribute to the sum in \eqref{probtau}. Hence we can estimate the latter by an integral and insert the assumptions made in \eqref{regvar}--\eqref{pureexp}. Subsequently, using the change of variable $s=\frac{e_m}{N^m}$ and taking the logarithm to express $m$ in terms of $s$, we obtain the following values of $\gamma$ after extracting the $t$-dependence:    
		\begin{equation}
			\label{gammaid}
			\begin{aligned}
				&\eqref{regvar} \quad \Longrightarrow \quad \gamma = 1,\quad \varphi(t) \asymp (\log t)^{-\alpha}, \\[0.2cm]
				&\eqref{pureexp} \quad \Longrightarrow \quad \gamma = \gamma_{N,K,e} = \frac{\log(N/Ke)}{\log(N/e)},
				\quad \varphi(t) \asymp 1.
			\end{aligned}
		\end{equation}
		
		In order to guarantee that $\rho=\infty$, we must require that $\alpha \in (-\infty,1]$, respectively, $K \in [1,\infty)$ (while $\beta$, respectively, $e$ play no role). Subject to \eqref{regvar},
		\begin{equation}
			\label{hatphiscal1}
			\hat\varphi(t) \asymp \left\{\begin{array}{ll}
				(\log t)^{1-\alpha}, &\alpha \in (-\infty,1),\\[0.2cm]
				\log\log t, &\alpha=1,
			\end{array}
			\right.
		\end{equation}
		while subject to \eqref{pureexp},
		\begin{equation}
			\label{hatphiscal2}
			\hat\varphi(t) \asymp \left\{\begin{array}{ll}
				1, &K \in (1,\infty),\\[0.2cm]
				\log t, &K = 1.
			\end{array}
			\right. 
		\end{equation}
		
		\paragraph{Computation of $a^{\Omega_N}_t(0,0)$.} 
		
		To compute $a^{\Omega_N}_t(0,0)$, we first rewrite the migration kernel $a^{\Omega_N}(\cdot,\cdot)$ in \eqref{739} as
		\begin{equation}
			\label{mo4}
			a^{\Omega_N}(0,\eta)=\frac{r_{\|\eta\|}}{N^{\|\eta\|-1}(N-1)}
		\end{equation}
		with
		\begin{equation}
			\label{ak:106} 
			r_{\|\eta\|}=\frac{1}{D(N)}\frac{N-1}{N}\sum_{l \geq \|\eta\|}\frac{c_{l-1}}{N^{l-1}}\frac{1}{N^{l-\|\eta\|}},
		\end{equation}
		where $D(N)$ is a renormalisation constant such that $\sum_{j\in\N} r_j=1$. For transition kernels of the form \eqref{mo4}, the time-$t$ transition kernel $a^{\Omega_N}_t(\cdot,\cdot)$ was computed in \cite{FG94} with the help of Fourier analysis, see also \cite{DGW05}. Namely,
		\begin{equation}
			\label{ak:dgw-asympt}
			a^{\Omega_N}_t(0,\eta) = \sum_{j \geq k} K_{jk}(N)\,\frac{\exp[-h_j(N) t]}{N^j},
			\qquad t \geq 0,\quad \eta \in \Omega_N\colon\, d_{\Omega_N}(0,\eta)=k \in \N_0,
		\end{equation}
		where
		\begin{equation}
			\label{Kjkdef}
			K_{jk}(N) = \left\{\begin{array}{ll}
				0, &j=k=0,\\
				-1, &j=k>0,\\
				N-1, &\mbox{otherwise},
			\end{array}
			\right.
			\qquad j,k \in \N_0,
		\end{equation}
		and
		\begin{equation}
			\label{hjrjrel}
			h_j(N) = \frac{N}{N-1}\,r_j(N) + \sum_{i>j} r_i(N), \qquad j \in \N.
		\end{equation}
		
		The expressions in \eqref{ak:106}--\eqref{hjrjrel} simplify considerably in the limit as $N\to\infty$, namely, the term with $i=j$ dominates and
		\begin{equation}
			\label{Ninfsimp}
			h_j(N) \sim r_j(N) \sim \frac{c_{j-1}}{D(N) N^{j-1}}, \quad j \in \N, 
			\qquad D(N) \sim c_0.
		\end{equation}
		We show why this is true for $h_j(N)$ (the argument for $r_j(N)$ and $D(N)$ is similar). Write
		\begin{equation}
			\begin{aligned}
				h_j(N) &= \frac{N}{N-1}\,r_j(N) + \sum_{i>j} r_i(N)\\
				&=\frac{1}{D(N)} \left(\sum_{l\geq j}\frac{c_{l-1}}{N^{l-1}}\frac{1}{N^{l-j}}
				+ \frac{N-1}{N}\sum_{l>j}\frac{c_{l-1}}{N^{l-1}}\sum_{i < j \leq l} \frac{1}{N^{l-i}}\right)\\
				&= \frac{1}{D(N)}\frac{c_{j-1}}{N^{j-1}}\left(1+\left[1+O\left(\frac{1}{N}\right)\right]\,
				\left(\frac{c_{j-1}}{N^{j-1}}\right)^{-1}\sum_{l>j} \frac{c_{l-1}}{N^{l-1}}\right).	 
			\end{aligned}
		\end{equation}
		Hence it suffices to show that
		\begin{equation}
			\label{520}
			\limsup_{N\to\infty}\left(\frac{c_{j-1}}{N^{j-1}}\right)^{-1}\sum_{l>j}\frac{c_{l-1}}{N^{l-1}}=0, \qquad j \in\N.
		\end{equation}
		To do so, note that, since $\limsup_{k\to\infty}\frac{1}{k}\log c_k<\log N$ by \eqref{740}, for $N$ large enough we have
		\begin{equation}
			\sup_{k\in\N_0} c_k^{1/k}<N.
		\end{equation}
		Let $\bar{N}=\inf\{N\in\N\colon\,\sup_{k\in\N_0} c_k^{1/k}<N\}$. Then  
		\begin{equation}
			\begin{aligned}
				&\limsup_{N\to\infty}\left(\frac{c_{j-1}}{N^{j-1}}\right)^{-1}\sum_{l>j}\frac{c_{l-1}}{N^{l-1}}
				\leq\limsup_{N\to\infty}\frac{1}{c_{j-1}}\sum_{l>j}\frac{\bar{N}^{l-1}}{N^{l-1}}N^{j-1}\\
				&=\frac{\bar{N}^{j-1}}{c_{j-1}} \limsup_{N\to\infty} \frac{\frac{\bar{N}}{N}}{1-\frac{\bar{N}}{N}}=0, \qquad j\in\N,
			\end{aligned}
		\end{equation}
		which settles \eqref{520}. 
		
		To understand what \eqref{hjrjrel} gives for finite $N$, note that for asymptotically polynomial coefficients (recall \eqref{regvar})
		\begin{equation}
			\label{521}
			\begin{aligned}
				&\left(\frac{c_{j-1}}{N^{j-1}}\right)^{-1}\sum_{l>j}\frac{c_{l-1}}{N^{l-1}}
				= [1+o(1)]\, \frac{N^{j-1}}{F(j-1)^{-\phi}}\sum_{l>j}\frac{{F(l-1)^{-\phi}}}{N^{l-1}}\\
				&=  [1+o(1)]\, \sum_{l>j}\frac{{(l-1)^{-\phi}}}{(j-1)^{-\phi}}N^{-(l-j)}
				=  [1+o(1)]\, \sum_{k \geq 1} \left(1+\frac{{k}}{j-1}\right)^{-\phi}N^{-k}, \qquad j \in \N.
			\end{aligned}
		\end{equation}
		For $\phi\geq0$ the right-hand side is bounded from above by $\sum_{k \geq 1} N^{-k} = \frac{1}{N-1}$ and for $\phi<0$ by $N^{-1} \sum_{k\geq 1} (1+k)^{-\phi}N^{-(k-1)} \leq N^{-1} C_\phi$. On the other hand, for pure exponential coefficients (recall \eqref{pureexp}),
		\begin{equation}
			\label{525}
			\left(\frac{c_{j-1}}{N^{j-1}}\right)^{-1}\sum_{l>j} \frac{c_{l-1}}{N^{l-1}} 
			= \sum_{k \geq 1} \left(\frac{c}{N}\right)^{-k} = \frac{c}{N-c}.
		\end{equation}  
		Hence, for both choices of coefficients we have the following:
		\begin{equation}
			\begin{tabular}{ll}
				&\text{For $N\to\infty$ the quantities $h_j(N),r_j(N)$ are bounded from above}\\ 
				&\text{and below by positive finite constants times the right-hand side of}\\
				&\text{\eqref{Ninfsimp} uniformly in $j \in \N$.}
			\end{tabular} 
			\label{fdh:Ncomp}
		\end{equation} 
		
		Picking $\eta=0$ ($k=0$) in \eqref{ak:dgw-asympt}, we obtain
		\begin{equation}
			\label{523}
			a_t^{\Omega_N}(0,0) = \sum_{j\in\N} (N-1) \frac{\exp[-h_j(N)t]}{N^j}. 
		\end{equation}
		Since we are interested in the asymptotic behaviour of $a_t^{\Omega_N}(0,0)$, only large values of $j$ are relevant and we can estimate the sum in \eqref{523} by an integral. To do so, we change variables by putting $s=h_j(N)$ and exploit \eqref{fdh:Ncomp}. Take the logarithm to express $j$ in terms of $s$, compute $\d s/\d j$, and extract the $t$-dependence. This gives
		\begin{equation}
			\label{atscal}
			\begin{aligned}
				&\eqref{regvar} \quad \Longrightarrow \quad a_t^{\Omega_N}(0,0) \asymp t^{-1}\log^\phi t,\\[0.2cm]
				&\eqref{pureexp} \quad \Longrightarrow \quad a_t^{\Omega_N}(0,0) \asymp t^{-1-\delta_{N,c}},
			\end{aligned}
		\end{equation}
		where
		\begin{equation}
			\label{deltaNcalt}
			\delta_{N,c} = \frac{\log c}{\log (N/c)}.
		\end{equation}
		
		\subsection{Hierarchical clustering}
		\label{ss.cluspr}
		
		In this section we prove Theorem~\ref{T.cluscritreg} by substituting the results of Theorem~\ref{T.scalcoeff} into the clustering criterion in \eqref{cluscritseed-b}.
		
		Combining \eqref{cluscrit}, \eqref{gammaid}--\eqref{hatphiscal2} and \eqref{atscal}--\eqref{deltaNcalt}, we find the following clustering criterion for \emph{fixed} $N$ and infinite seed-bank:
		\begin{itemize}
			\item
			Subject to \eqref{regvar}, clustering prevails if and only if 
			\begin{equation}
				\label{phialphacond}
				- \phi \leq \alpha \leq 1.
			\end{equation}  
			\item
			Subject to \eqref{pureexp}, clustering prevails if and only
			\begin{equation}
				\label{crc2}
				\delta_{N,c} \leq -\frac{1-\gamma_{N,K,e}}{\gamma_{N,K,e}}.
			\end{equation}
		\end{itemize}
		In view of \eqref{gammaid} and \eqref{deltaNc}, the condition in \eqref{crc2} amounts to 
		\begin{equation}
			\label{crc2expl}
			\log N \times \log (Kc) \leq \log c \times \log (K^2e),
		\end{equation}
		where we use that $c<N$ and $Ke<N$ (recall \eqref{740} and \eqref{740alt}). The condition in \eqref{crc2expl} holds for all $N$ when 
		\begin{equation}
			Kc=1 \text{ with } \left\{
			\begin{array}{ll}
				c=1, &K^2e \in (0,\infty),\\
				c>1, &K^2e \geq 1,\\
				c<1, &K^2e \leq 1.
			\end{array}
			\right.
		\end{equation}
		It also holds for $N$ large enough when $Kc<1$ and fails for $N$ large enough when $Kc>1$. Thus, for infinite seed-bank, clustering prevails for $N$ large enough if and only if
		\begin{equation}
			\label{Kccond}
			Kc \leq 1 \leq K,
		\end{equation}  
		which is the analogue of \eqref{phialphacond}.

		\section{Preparation: $N\to\infty$, McKean-Vlasov process and mean-field system}
		\label{ss.IntroMeanfield}
		
		To analyse the scaling of our hierarchical system in the hierarchichal mean-field limit $N\to\infty$, we first need to understand simpler systems. In this section we consider the mean-field system consisting of a \emph{single hierarchy}, and introduce the following:
		\begin{enumerate}
			\item McKean-Vlasov process (Section~\ref{ss.Mckeanvlasov}).
			\item Mean-field system and McKean-Vlasov limit (Section~\ref{mfsmkv}).
		\end{enumerate}
		For each we derive a key proposition that will play a crucial role in our analysis of the truncated system with \emph{finitely many hierarchies} in Sections~\ref{s.finlevel}--\ref{s.multilevel} and the full system with \emph{infinitely many hierarchies} in Section~\ref{s.multilevel}. The proofs of the propositions stated in this section will be given in Sections~\ref{sec:equergod} and \ref{sec:mkvlim}. 
		
		\subsection{McKean-Vlasov process}
		\label{ss.Mckeanvlasov}
		In this section we introduce the McKean-Vlasov process, which will play an important role in our analysis of the mean-field system to be introduced in Sections~\ref{mfsmkv}--\ref{mffss}. (In the full system the effective process introduced in \eqref{92} will be seen to be an example of a McKean-Vlasov process.)
		
		For $g\in\CG$ and $c,K,e \in (0,\infty)$, consider the single-colony process 
		\begin{equation}
			\label{SC}
			z(t)=(x(t),y(t))_{t \geq 0},
		\end{equation} 
		taking values in $[0,1]^2$, with initial law $\CL[(x(0),y(0))]=\mu$ and with components evolving according to
		\begin{eqnarray}
			\label{gh5}
			&&\d x(t) = c\,[\E[x(t)] - x(t)]\, \d t + \sqrt{g(x(t))}\, \d w (t) + Ke\,[y(t)-x(t)]\,\d t,\\
			&&\d y(t) = e\, [x(t)-y(t)]\, \d t, \nonumber
		\end{eqnarray}
		where $\E$ denotes expectation with respect to $\mu$. With the help of It\^o-calculus we can compute the expectation $\E[x(t)]$. Indeed, from \eqref{gh5} we get
		\begin{equation}
			\label{359alt}
			\begin{aligned}
				\frac{\d}{\d t}\E[x(t)] &= Ke\,\big[\E[y(t)]-\E[x(t)]\big],\\
				\frac{\d}{\d t}\E[y(t)] &= e\,\big[\E[x(t)]-\E[y(t)]\big].
			\end{aligned}
		\end{equation}
		Define
		\begin{equation}
			\label{inc}
			\theta_x=\E^{\mu}[x(0)],\qquad \theta_y=\E^{\mu}[y(0)],\qquad
			\theta=\E^\mu\left[\frac{x(0)+Ky(0)}{1+K}\right].
		\end{equation}
		Note that \eqref{359alt} implies that $\theta$ is a preserved quantity, i.e.,
		\begin{equation}
			\E^\mu\left[\frac{x(0)+Ky(0)}{1+K}\right]=\E^\mu\left[\frac{x(t)+Ky(t)}{1+K}\right]=\theta,
			\qquad t \geq 0.
		\end{equation} 
		Solving \eqref{359alt}, we find 
		\begin{equation}
			\label{expzalt}
			\begin{aligned}
				\E[x(t)] &= \theta + \frac{K}{1+K} (\theta_x-\theta_y)\, \e^{-(K+1)et},\\
				\E[y(t)] &= \theta - \frac{1}{1+K} (\theta_x-\theta_y)\, \e^{-(K+1)et}.
			\end{aligned}
		\end{equation}
		In particular, from \eqref{expz} we see that
		\begin{equation}
			\lim_{t \to \infty} (\E[x(t)],\E[y(t)]) =  (\theta,\theta). 
		\end{equation}
		Hence, in equilibrium we can replace $\E[x(t)]$ in \eqref{gh5} by $\theta$. After inserting \eqref{expzalt} into \eqref{gh5}, we can use \cite[Theorem 1, Remark on p.156]{YW71} to show that for every deterministic initial state $(x(0),y(0)) \in [0,1]^2$ the SSDE in \eqref{gh5} has a unique strong solution. We will refer to this solution as the \emph{McKean-Vlasov process}.
		
		\begin{remark}{\bf [Self-consistency]}
			\label{R.dh12}
			{\rm To prove uniqueness of the solution to \eqref{gh5} we can also use \cite{G88}, where self-consistent mean-field dynamics are treated in detail. The solution has the Feller property.} \hfill$\blacksquare$ 
		\end{remark}
		
		\begin{proposition}{{\bf [Single-colony McKean-Vlasov process: equilibrium]}}
			\label{P.equergod}
			For every initial law $\mu\in\CP([0,1]^2)$ satisfying
			\begin{equation}
				\label{e587alt}
				\E^\mu\left[\frac{x(0)+Ky(0)}{1+K}\right]=\theta, \qquad \theta\in[0,1],
			\end{equation}  
			the process in \eqref{SC} converges to a unique equilibrium,
			\begin{equation}
				\label{e587}
				\lim_{t \to \infty} \CL[(x(t),y(t))] = \Gamma_\theta, 
			\end{equation}
			and 
			\begin{equation}
				\label{ag20}
				\Gamma_\theta \in \CP([0,1]^2), 
			\end{equation}
			satisfies
			\begin{equation}
				\label{e588}
				\theta = \int_{[0,1]^2} x\,\Gamma_\theta (\d x,\d y)=\int_{[0,1]^2} y\,\Gamma_\theta (\d x,\d y). 
			\end{equation}
		\end{proposition}
		
		\noindent
		The proof of Proposition~\ref{P.equergod} is given in Section~\ref{sec:equergod}. Note that $\Gamma_\theta = \Gamma_\theta^{g,c,K,e}$ depends on all the parameters appearing in \eqref{gh5}. In Section \ref{ss.ProofsMeanfield} we will see that $\Gamma_\theta$ is continuous as a function of $\theta$. 
		
		\begin{remark}{\bf [Non-linear Markov process]}
			{\rm Note that \eqref{SC} is a \emph{non-linear} Markov process: the evolution not only depends on the current state $z(t)$, but also on the current law $\CL[z(t)]$ via the expectation $\E[x(t)]$ appearing in the SSDE \eqref{gh5}. This is different from the model without seed-bank, where the non-linearity is replaced by a drift towards $\theta$ that is constant in time. In equilibrium we can replace $\E[x(t)]$ by $\theta$ in \eqref{gh5}, but before equilibrium is reached we cannot, because $t \mapsto \E[x(t)]$ is not constant, as is clear from \eqref{expz}. Note that $\E[x(t)]$ is a linear functional of $z(0)$. This fact will play an important role in the renormalisation analysis in Section~\ref{s.renormasym}.} 
			\hfill $\blacksquare$
		\end{remark}

		\subsection{Mean-field system and McKean-Vlasov limit}
		\label{mfsmkv}
		
		In this section we consider a simplified version of the SSDE in \eqref{moSDE}, namely, we restrict to the finite geographic space
		\begin{equation}
			\gls{spaceN}=\{0,1,\ldots,N-1\}, \qquad N \in \N.
		\end{equation} 
		In this simplified version, the migration kernel $a^{\Omega_N}(\cdot,\cdot)$ is replaced by $a^{[N]}(\xi,\eta) =c N^{-1}$ for all $(\xi,\eta)\in [N]$, where $c\in (0,\infty)$ is a constant. The seed-bank consists of only \emph{one colour} and the exchange rates between active and dormant are given by $Ke,e$. The state space is
		\begin{equation}
			S = \mathfrak{s}^{[N]}, \qquad \mathfrak{s} = [0,1]^2,
		\end{equation}
		the system is denoted by 
		\begin{equation}
			\label{e715}
			Z^{[N]}(t)=\big(X^{[N]}(t),Y^{[N]}(t)\big)_{t \geq 0}, \qquad 
			\big(X^{[N]}(t),Y^{[N]}(t)\big) = \big(x^{[N]}_i(t), y^{[N]}_i(t)\big)_{i \in [N]},
		\end{equation}
		and its components evolve according to the SSDE 
		\begin{equation}
			\label{gh45a}
			\begin{aligned}
				&\d x^{[N]}_i(t) = \frac{c}{N} \sum_{j \in [N]} \big[x^{[N]}_j(t) - x^{[N]}_i(t)\big]\, \d t 
				+ \sqrt{g\big(x^{[N]}_i(t)\big)}\, \d w_i (t) + K e\, \big[y^{[N]}_i(t)-x^{[N]}_i(t)\big]\,\d t,\\
				&\d y^{[N]}_i(t) = e\,\big[x^{[N]}_i(t)-y^{[N]}_i(t)\big]\, \d t, \qquad i \in [N],
			\end{aligned}
		\end{equation}
		which is the special case of \eqref{moSDE} obtained by setting $a^{\Omega_N}(\eta,\xi)=0$ if $d (\eta,\xi)>1$ and $K_m=e_m=0$ for $m\geq1$. It is natural to take an \emph{exchangeable random initial state}, because the evolution preserves exchangeability. According to De Finetti's theorem, there is no loss of generality in taking an i.i.d.\ initial state, i.e.,
		\begin{equation}
			\CL\big[X^{[N]}(0),Y^{[N]}(0)\big] = \mu^{\otimes [N]}, \qquad\qquad \mu\in\CP\left([0,1]^2\right).
		\end{equation}  
		By \cite[Theorem 3.1]{SS80}, the SSDE in \eqref{gh45a} is the unique weak solution of a well-posed martingale problem. By \cite[Theorem 3.2]{SS80}, for every deterministic initial state $(X^{[N]}(0),Y^{[N]}(0))$, \eqref{gh45a} has a unique strong solution. We are interested in the limit $N\to\infty$. For the limiting process we define
		\begin{equation}
			\label{gh51}
			(Z(t))_{t\geq 0}=(X(t),Y(t))_{t\geq 0}=\big((x_i(t),y_i(t))_{i\in\N_0}\big)_{t\geq 0}
		\end{equation}
		with components evolving according to \eqref{gh5}, i.e.,
		\begin{equation}
			\label{gh52}
			\begin{aligned}
				&\d x_i(t) = c\,[\E[x_i(t)] - x_i(t)]\, \d t + \sqrt{g(x_i(t))}\, \d w (t) + Ke\,[y_i(t)-x_i(t)]\,\d t,\\
				&\d y_i(t) = e\, [x_i(t)-y_i(t)]\, \d t,\qquad i\in\N_0,
			\end{aligned}
		\end{equation}
		with $\CL\left[(X(0),Y(0))\right]=\mu$ for some exchangeable $\mu\in\CP(([0,1]^2)^{\otimes[\N_0]}$. Note that \eqref{gh52} consists of i.i.d.\ copies of the single-colony McKean-Vlasov process in \eqref{SC}, labelled by $i\in\N_0$.
		
		\begin{proposition}{{\bf [Infinite-system McKean-Vlasov limit: convergence]}}
			\label{P.mkvlim}
			$\mbox{}$\\
			Suppose that $\CL[(X^{[N]}(0),Y^{[N]}(0))] = \mu^{[N]}$ is exchangeable and 
			\begin{equation}
				\label{instate}
				\theta = \E^{\mu^{[N]}}\left[\frac{x(0)+Ky(0)}{1+K}\right].
			\end{equation}
			Then 
			\begin{equation}
				\label{gh29}
				\lim_{N\to\infty} \CL\Big[\big(X^{[N]}(t), Y^{[N]}(t)\big)_{t \geq 0}\Big] = \CL\big[(X(t),Y(t))_{t \geq 0}\big]
			\end{equation}
			with
			\begin{equation}
				\CL\big[(X(0),Y(0))_{t \geq 0}\big]=\mu,\qquad \mu=\lim_{N\to \infty}\mu^{[N]},
			\end{equation}
			where the limit is the McKean-Vlasov process in \eqref{SC}--\eqref{gh5}. 
		\end{proposition}
		
		\noindent
		The proof of Proposition~\ref{P.mkvlim} is given in Section~\ref{sec:mkvlim}. For the system without seed-bank the McKean-Vlasov limit was proved in \cite{G88}. The fact that the components decouple is a property referred to as \emph{propagation of chaos}.
		
		\subsection{Proof of equilibrium and ergodicity}
		\label{sec:equergod}
		
		In this section we prove Proposition~\ref{P.equergod}.
		
		\begin{proof} 
			Note that, by \eqref{expz}, we can rewrite \eqref{gh5} as
			\begin{equation}
				\begin{aligned}
					\label{gh554}
					&\d x(t) = c\,\left[\theta + \frac{K}{1+K} (\theta_x-\theta_y)\, 
					\e^{-(K+1)et} - x(t)\right]\, \d t + \sqrt{g(x(t))}\, \d w (t) \\
					&\qquad \qquad \qquad + Ke\,[y(t)-x(t)]\,\d t,\\
					&\d y(t) = e\, [x(t)-y(t)]\, \d t.
				\end{aligned} 
			\end{equation} 
			Existence and uniqueness of a strong solution is again standard (see e.g.\ \cite[Theorem 1]{YW71} and recall Remark~\ref{R.dh12}). We start by proving existence and uniqueness of the equilibrium. Afterwards we show that the solution converges to this equilibrium.  
			
			Consider two copies $(x_1,y_1)$ and $(x_2,y_2)$ of the system defined in \eqref{gh554}, with $\CL[(x_1(0),y_1(0))]=\mu_1$ and $\CL[(x_2(0),y_2(0))]=\mu_2$, where $\mu_1$ and $\mu_2$ satisfy 
			\begin{equation}
				\E^{\mu_1}\left[\frac{x_1(0)+Ky_1(0)}{1+K}\right] = \theta =\E^{\mu_2}\left[\frac{x_2(0)+Ky_2(0)}{1+K}\right]
			\end{equation} 
			for some $\theta\in[0,1]$. Write 
			\begin{equation}
				\theta_{x_1} = \E^{\mu_1}[x_1(0)],\quad \theta_{y_1}=\E^{\mu_1}[y_1(0)],
				\quad \theta_{x_2} =\E^{\mu_2}[x_2(0)], \quad \theta_{y_2}=\E^{\mu_2}[y_2(0)].
			\end{equation}
			Couple the two systems by coupling their  Brownian motions. Denote the coupled process by
			\begin{equation}
				\label{m17}
				\begin{aligned}
					&(\bar{z}(t))_{t\geq 0}= (z_1(t),z_2(t))_{t\geq 0},\quad z_1(t)=(x_1(t),y_1(t)),\quad z_2(t)=(x_2(t),y_2(t)),\\
					&\CL(\bar{z}(0))=\mu_1\times\mu_2,
				\end{aligned}
			\end{equation}
			which has a unique strong solution. Put
			\begin{equation}
				\Delta(t)=x_1(t)-x_2(t), \qquad \delta(t)=y_1(t)-y_2(t).
			\end{equation}
			To show that the equilibrium is unique, it is enough to show that
			\begin{equation}
				\lim_{t\to\infty} \E\left[|\Delta(t)|+EK|\delta(t)|\right]=0.
			\end{equation}
			
			Using a generalised form of It\^o's formula, we find
			\begin{equation}
				\label{m12}
				\begin{aligned}
					\d |\Delta(t)|
					&= (\sign\,\Delta(t))\,\d \Delta(t)+\d L_t^0\\
					&= (\sign\,\Delta(t))\,c\,\left[\frac{K}{1+K} \big((\theta_{x_1}-\theta_{x_2})-(\theta_{y_1}-\theta_{y_2})\big)\, 
					\e^{-(K+1)et} - \Delta(t)\right]\, \d t\\
					&\qquad + (\sign\,\Delta(t))\,\left(\sqrt{g(x_1(t))}-\sqrt{g(x_2(t))}\,\right) \d w (t) \\
					& \qquad + (\sign\,\Delta(t))\,Ke\,[\delta(t)-\Delta(t)]\,\d t,
				\end{aligned}
			\end{equation}
			where we use that the local time $L_t^0$ (see \cite[Section IV.43]{RoWi00}) of $\Delta(t)$ at $0$ equals $0$, since $g$ is Lipschitz (see \cite[Proposition V.39.3]{RoWi00}). Again using It\^o's formula, we also find 
			\begin{equation}
				\label{m13}
				\d |\delta(t)| = (\sign\,\delta(t))\,\d\delta(t) = (\sign\,\delta(t))\,e\,[\Delta(t)-\delta(t)]\,\d t.
			\end{equation} 
			Taking expectations in \eqref{m12}--\eqref{m13}, we get
			\begin{equation}
				\label{m14}
				\begin{aligned}
					&\frac{\d}{\d t} \E[|\Delta(t)|+K|\delta(t)|]\\
					&=\E\left[c\,\left[(\sign\,\Delta(t))\frac{K}{1+K} \big((\theta_{x_1}-\theta_{x_2})-(\theta_{y_1}-\theta_{y_2})\big)\, 
					\e^{-(K+1)et} - \left|\Delta(t)\right|\right]\right]\\
					&\quad + Ke\,\E\Big[\left(\sign\,\Delta(t)-\sign\,\delta(t)\right)(\delta(t)-\Delta(t))\Big]\\
					&=\E\left[c\,(\sign\,\Delta(t))\frac{K}{1+K} \big((\theta_{x_1}-\theta_{x_2})-(\theta_{y_1}-\theta_{y_2})\big)\, 
					\e^{-(K+1)et}\right]\\
					&\quad-c\,\E[\left|\Delta(t)\right|]\\
					&\quad-2Ke\,\E\left[1_{\{\sign\,\delta(t)\neq\sign\Delta(t)\}}\left(|\delta(t)|+|\Delta(t)|\right)\right].
				\end{aligned}
			\end{equation}
			Define
			\begin{equation}
				h(t) = c\,\E[\left|\Delta(t)\right|] + 2Ke\,
				\E\left[1_{\{\sign\,\delta(t))\neq\sign\,\Delta(t)\}}\left(|\delta(t)|+|\Delta(t)|\right)\right].
			\end{equation}
			Then $h(t)$ satisfies
			\begin{enumerate}
				\item $h(t)>0$.
				\item $0\leq \int_0^\infty \d t\,h(t) \leq 1+K
				+c\left|(\theta_{x_1}-\theta_{x_2})-(\theta_{y_1}-\theta_{y_2})\right|
				\frac{K}{K+1}\frac{1}{e(K+1)}\left[1-\e^{-(K+1)et}\right]$.
				\item $h$ is differentiable with $h^\prime$ bounded (see \cite[Appendix D]{GdHOpr1}).
			\end{enumerate}
			Hence it follows that $\lim_{t\to\infty}h(t)=0$, which implies that 
			\begin{equation}
				\label{m16}
				\lim_{t \to \infty}\E\left[|\Delta(t)|\right]=0.
			\end{equation}
			
			We are left to prove that $\lim_{t\to\infty} \E[|\delta(t)|]=0$. To do so, we define
			\begin{equation}
				\label{m01}
				f(t)=\E[|\delta(t)|], \qquad G(t)= e\,\E[(\sign\,\delta(t))\Delta(t)].
			\end{equation}
			Note that $G$ is bounded and continuous. Taking expectations in \eqref{m13}, we find
			\begin{equation}
				\label{m15}
				\frac{\d}{\d t}f(t)=-e f(t)+G(t),
			\end{equation}
			Solving \eqref{m15} explicitly, we find that
			\begin{equation}
				f(t)=f(r)\,\e^{-e(t-r)}+\int_{r}^t \d s\,\e^{-e(t-s)}G(s), \qquad r,t\in\R,\,t>r\geq 0.
			\end{equation} 
			By \eqref{m16}, for each $\epsilon>0$ we can find an $r\in\R$ such that $\E[|\Delta(s)|]<\epsilon$ for all $s>r$, and hence $\sup_{t>r}|G(t)|<\epsilon$. Therefore
			\begin{equation}
				f(t) \leq f(r)\,\e^{-e(t-r)}+\epsilon
			\end{equation}
			and, since $|f|<1$, we find, for each $\epsilon>0$,
			\begin{equation}
				\label{m02}
				\lim_{t\to\infty} f(t)<\epsilon.
			\end{equation}
			Therefore $\lim_{t\to\infty}\E[|\delta(t)|]=0$, which completes the proof of uniqueness of the equilibrium for given $\theta$. 
			
			To prove existence of the equilibrium, let $(t_n)_{n\in\N}$ be any increasing sequenc eof times such that $\lim_{n\to\infty} t_n=\infty$. Let $\mu=\CL[(x(0),y(0))]$ be any initial measure of the system in \eqref{gh554} with $\E^\mu[\frac{x(0)+Ky(0)}{1+K}]=\theta$, and let $\mu(t_n)=\CL[(x(t_n),y(t_n))]$. Since the state space is compact, the sequence $(\mu(t_{n}))_{n\in\N}$ is tight, and by Prohorov's theorem we can find a converging subsequence $(\mu(t_{n_k}))_{k\in\N}$. Put $\nu=\lim_{k\to\infty}\mu(t_{n_k})$. We will show that $\nu$ is invariant. To that end, recall from Section~\ref{ss.Mckeanvlasov} that 
			\begin{equation}
				\label{ba}
				\E^\mu\left[\frac{x(t)+Ky(t)}{1+K}\right] = \theta, \qquad t \geq 0.
			\end{equation}
			Hence we can use the coupling in \eqref{m17} to show that the system starting in $\mu$ and the system starting $\mu(t)$ converge to the same law as $t\to\infty$, from which it follows that $\lim_{k \to \infty} \mu(t+t_{n_k})=\nu$. Let $(S_t)_{t \geq 0}$ denote the semigroup of the system in \eqref{gh554}. By the Feller property for semigroups,
			\begin{equation}
				S_t\nu = \lim_{k\to\infty} S_t\mu(t_{n_k}) = \lim_{k\to\infty} S_{t_{n_k}}\left(S_t\mu\right)=\nu,
			\end{equation} 
			where in the last equality we use the uniqueness of the equilibrium given $\theta$. Thus, $\nu$ is an invariant measure. To exhibit its dependence on $\theta$ we write $\nu_\theta$. Using the same coupling as in \eqref{m17}, and starting from $\mu\times\nu_\theta$ with $\nu_\theta$ the invariant measure just obtained, we see that for every $\theta$ the system in \eqref{gh5} converges to a unique equilibrium measure $\nu_\theta$, and so \eqref{e588} is immediate from \eqref{ba}.
		\end{proof}

		\subsection{Proof of McKean-Vlasov limit}
		\label{sec:mkvlim}
		
		In this section we give a sketch of the proof of Proposition~\ref{P.mkvlim}. In Chapters~\ref{ss.ProofsMeanfield}-\ref{s.multilevel} we encounter more difficult versions of Proposition~\ref{P.mkvlim}. There we will give the proofs in full detail.
		
		\begin{proof}
			Since we start from a distribution $\mu(0)$ that is exchangeable, Aldous's ergodic theorem gives
			\begin{equation}
				\lim_{N\to\infty} \frac{1}{N} \sum_{j\in[N]} x_j(0) = \E^{\mu(0)}[x_0] \quad \P\text{-a.s.}
			\end{equation} 
			By Ioffe's theorem \cite[Eqs.\ (1.1)--(1.2)]{DGV95}, tightness of the associated sequence of processes (uniformly on the state space) follows from boundedness of the generator as an operator. To apply the generator criterion in \cite{JM86} we must show propagation of chaos and prove the weak law of large numbers
			\be{}
			\lim_{N\to\infty} \frac{1}{N} \sum_{j\in[N]} x_j(t) = \E[x_0(t)].
			\ee
			The propagation of chaos and the weak law of large numbers for $t>0$ therefore follows from \cite[Section 4]{G88}. Since the martingale problem is well-posed \cite[Section 2]{G88}, the limiting process exists and is unique. 
		\end{proof}

		\section{Proofs: $N\to\infty$, mean-field finite-systems scheme}
		\label{ss.ProofsMeanfield}
		
		In Sections~\ref{mffss} we introduce the so called mean-field finite-systems scheme for the mean-field system introduced in Section~\ref{mfsmkv}. In Section~\ref{sec:finsysmf} we outline the \emph{abstract scheme} behind the proof behind the mean-field finite-systems scheme. The computations in the proof of the abstract scheme are long and technical, and are deferred to Section~\ref{ss.pabstracts}.
		
		\subsection{Mean-field finite-systems scheme}
		\label{mffss}
		
		In this section we describe the limiting dynamics of the finite system in \eqref{e715} from a \emph{multiple space-time scale} viewpoint. To do so, we need the following limiting SSDE for the infinite system $Z(t)=(z_i(t))_{i\in\N_0}=\left(x_i(t),y_i(t)\right)_{i\in \N_0}$, with initial law $\CL[Z(0)]=\mu^{\otimes\N_0}$, evolving according to
		\begin{eqnarray}
			\label{gh5inf}
			&&\d x_i(t) = c\,[\theta - x_i(t)]\, \d t + \sqrt{g(x_i(t))}\, \d w_i (t) + Ke\,[y_i(t)-x_i(t)]\,\d t,\\
			&&\d y_i(t) = e\, [x_i(t)-y_i(t)]\, \d t,\ i\in\N_0, \nonumber
		\end{eqnarray}
		where $\theta$ is defined in \eqref{e588}. Note that each component of \eqref{gh5inf} is an autonomous copy of the McKean-Vlasov process in \eqref{gh5} in equilibrium. 
		
		For the multiscale analysis we will need the following ingredients:
		\begin{enumerate}
			\item 
			The \emph{estimator} for the finite system is defined by 
			\begin{equation}
				\label{slovar}
				\bar{\Theta}^{[N]}(t)=\bar{\Theta}^{[N]}\big(Z^{[N]}(t)\big)=\frac{1}{N}\sum_{i\in[N]}\frac{x^{[N]}_i(t)+Ky_i^{[N]}(t)}{1+K}
			\end{equation}
			and its active and dormant counterparts
			\begin{equation}
				\label{avxy}
				\begin{aligned}
					\bar{\Theta}_{x}^{[N]}(t)&=\frac{1}{N}\sum_{i\in[N]}x^{[N]}_i(t),\\
					\bar{\Theta}_{y}^{[N]}(t)&=\frac{1}{N}\sum_{i\in[N]}y^{[N]}_i(t).
				\end{aligned}
			\end{equation}
			
			\item 
			The \emph{time scale} $N$, on which $\lim_{N\to\infty} \CL[\bar{\Theta}^{[N]} (L(N))-\bar{\Theta}^{[N]}(0)]=\delta_0$ for all $L(N)$ such that $\lim_{N\to\infty} L(N)=\infty$ and $\lim_{N\to\infty} L(N)/N=0$, but not for $L(N)=N$. In words, $N$ is the time scale on which $\bar{\Theta}^{[N]}(\cdot)$ starts evolving, i.e., $\left(\bar{\Theta}^{[N]}(Ns)\right)_{s>0}$ is not a fixed process. When we scale time by $N$, putting $t=Ns$, we view $s$ as the ``fast time scale" and $t$ as the ``slow time scale".
			\item 
			The \emph{invariant measure}, i.e., the equilibrium measure of a single component in \eqref{gh5inf} written  
			\begin{equation}
				\label{singcoleq}
				\Gamma_\theta,
			\end{equation}
			and the \emph{invariant measure} of the infinite system in \eqref{gh5inf}, written $\nu_\theta=\Gamma_\theta^{\otimes\N_0}$, with $\theta \in [0,1]$ controlled by the initial measure (recall \eqref{inc}--\eqref{expz}). 
			\item 
			The \emph{renormalisation transformation} $\CF\colon\,\CG\to\CG$,
			\begin{equation}
				\label{gh4}
				(\CF g)(\theta) = \int_{[0,1]^2} g(x)\,\nu_\theta(\d x, \d y_0), \quad \theta \in [0,1],
			\end{equation} 
			where $\nu_\theta$ is the equilibrium measure of \ref{gh5inf}. Note that $\CF$ is the same transformation as defined in \eqref{renor}, but for the truncated system. Note that we can also write
			\begin{equation}
				\label{gh44}
				(\CF g)(\theta) = \int_{[0,1]^2} g(x)\,\Gamma_\theta(\d x, \d y_0), \quad \theta \in [0,1],
			\end{equation} 
			where $\Gamma_\theta$ is as defined in \eqref{singcoleq}.
			\item
			The \emph{macroscopic observable} $\left(\bar{\Theta}(s)\right)_{s>0}$ satisfying the SSDE
			\begin{equation}
				\label{52}
				\d \bar{\Theta}(s)=\frac{1}{1+K}\sqrt{\E^{\Gamma_{\bar{\Theta}(s)}}[g(u)]}\,\d w(s)
				=\frac{1}{1+K}\sqrt{(\CF g)(\bar{\Theta}(s))}\,\d w(s),
			\end{equation}
		\end{enumerate}
		
		To obtain the multi-scale limit dynamics for the system in \eqref{e715}, we speed up time by a factor $N$ and define the process
		\begin{equation}
			\label{gh41}
			\left(x_1^{[N]}(s), y_1^{[N]}(s)\right)_{s > 0}
			=\left({\Theta}^{[N]}_x(Ns),{\Theta}^{[N]}_y(Ns)\right)_{s>0},
		\end{equation}
		which is the analogue of the $1$-block average in \eqref{blockav}. We use the lower index $1$ to indicate that the average is taken over $[N]$ components. Using \eqref{gh45a}, we see that the dynamics of  \eqref{gh41} is given by the SSDE
		\begin{equation}
			\label{mfevolve}
			\begin{aligned}
				\d x_1^{[N]}(s)&=\sqrt{\frac{1}{N}\sum_{i\in[N]}g\big(x_i(Ns)\big)}\,\d w(s)+NKe\left[y_1^{[N]}(s)-x_1^{[N]}(s)\right]\d s,\\
				\d y_1^{[N]}(s)&=Ne\left[x_1^{[N]}(s)-y_1^{[N]}(s)\right]\d s.
			\end{aligned}
		\end{equation}
		
		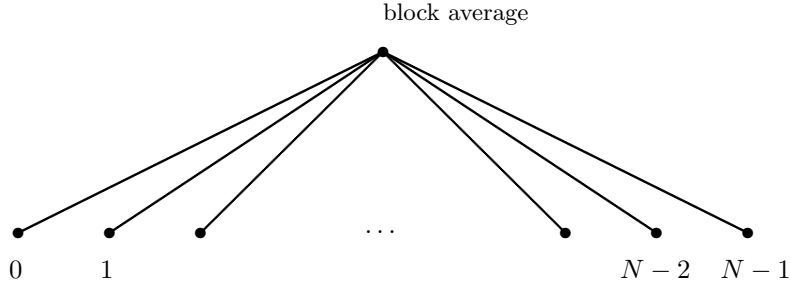
\begin{figure}[htbp]
			\begin{center}
				\setlength{\unitlength}{0.6cm}
				\begin{picture}(16,6)(-7,-1)
					{\thicklines
						\qbezier(-8,0)(-4,2)(0,4)
						\qbezier(-6,0)(-3,2)(0,4)
						\qbezier(-4,0)(-2,2)(0,4)
						\qbezier(4,0)(2,2)(0,4)
						\qbezier(6,0)(3,2)(0,4)
						\qbezier(8,0)(4,2)(0,4)
					} 
					\put(-8,0){\circle*{.25}}
					\put(-6,0){\circle*{.25}}
					\put(-4,0){\circle*{.25}}
					\put(4,0){\circle*{.25}}
					\put(6,0){\circle*{.25}}
					\put(8,0){\circle*{.25}}
					\put(0,4){\circle*{.25}}
					\put(-.4,0){$\dots$}
					\put(-8.2,-1){$0$}
					\put(-6.2,-1){$1$}
					\put(5.2,-1){$N-2$}
					\put(7.4,-1){$N-1$}
					\put(0,4.7){\small block average}
				\end{picture}
			\end{center}
			\caption{Given the value of the block average, the $N \gg 1$ constituent components
				equilibrate on a time scale that is fast with respect to the time scale on which the block 
				average fluctuates. Consequently, the volatility of the block average is the expectation
				of the volatility of the constituent components under the conditional quasi-equilibrium.}
			\label{fig-renorm}
		\end{figure}
		
		In \eqref{mfevolve}, in the limit as $N\to\infty$ infinite rates appear in the exchange between the active and the dormant population. However, looking at the process
		\begin{equation}
			\label{gh1}
			\left(\frac{x_1^{[N]}(s)+Ky_1^{[N]}(s)}{1+K}\right)_{s>0}=\left(\bar{\Theta}^{[N]}(Ns)\right)_{s>0}
		\end{equation}
		we see that the terms carrying a factor $N$ in front cancel out. Consequently, for the process in \eqref{gh1} we can use ideas from \cite{DG93b} to prove tightness as $N\to\infty$ in the \emph{classical topology} of continuum path processes. We will show in Section \ref{step2} that 
		\begin{equation}
			\label{moha44}
			\begin{aligned}
				&\lim_{N\to\infty} \CL\left[\left(\left[x_1^{[N]}(s)-y_1^{[N]}(s)\right]\right)_{s \geq 0}\right] 
				= \CL\left[(0)_{s \geq 0}\right] \\
				&\mbox{ in the \emph{Meyer-Zheng topology}}.
			\end{aligned}
		\end{equation}
		Combining \eqref{gh1} and \eqref{moha44}, we obtain the multiple space-time scaling behaviour of the system in \eqref{e715}. 
		
		\begin{proposition}{{\bf [Mean-field: finite-systems scheme]}}
			\label{P.finsysmf}
			Suppose that the SSDE in \eqref{gh45a} has initial measure $\CL[Z^{[N]}(0)]=\mu^{\otimes [N]}$ for some $\mu\in \CP\left([0,1]^2\right)$. Let 
			\begin{equation}
				\theta =\E^{\mu}\left[ \frac{x+Ky_0}{1+K}\right].
			\end{equation} 
			\begin{itemize}
				\item[(a)]
				For the averages in \eqref{gh41}, 
				\begin{equation}
					\label{gh42alt}
					\begin{aligned}
						&\lim_{N\to\infty} \CL \left[\left(x_1^{[N]}(s),  y_{0,1}^{[N]}(s)\right)_{s > 0}\right] 
						= \CL \left[\left(x_1^{\N_0}(s),y_{0,1}^{\N_0}(s)\right)_{s > 0}\right]\\
						&\text{in the Meyer-Zheng topology},
					\end{aligned}
				\end{equation}
				where the limit process is the unique solution of the SSDE
				\begin{equation}
					\label{gh43}
					\begin{aligned}
						\d x_1^{\N_0}(s) &=  \frac{1}{1+K} \sqrt{(\CF g)\big(x_1^{\N_0}(s)\big)}\, \d w(s),\\
						y_{0,1}^{\N_0}(s) &=  x_1^{\N_0} (s),
					\end{aligned}	
				\end{equation}
				with initial state
				\begin{equation}
					\label{e865}
					\left( x_1^{\N_0} (0),  y_{0,1}^{\N_0}(0)\right) = \left(\theta,\theta\right). 
				\end{equation}
				\item[(b)]
				For the weighted sum of the averages in \eqref{gh1}, 
				\begin{equation}
					\label{434}
					\lim_{N\to\infty} \CL \left[\left(\bar{\Theta}^{[N]}(Ns)\right)_{s > 0}\right] 
					= \CL \left[\left(\bar{\Theta}(s)\right)_{s > 0}\right],
				\end{equation}
				where the limit is the macroscopic observable in \eqref{52} with initial state
				\begin{equation}
					\bar{\Theta}(0) = \theta.
				\end{equation}
				\item[(c)]
				Define
				\begin{equation}
					\label{ma4}
					\nu_\theta(s)= \int_{[0,1]} Q_s\bigl(\theta,\d \theta^\prime\bigr)\,\nu_{\theta^\prime} \in \CP([0,1]^2),
				\end{equation}
				where $Q_s(\theta,\cdot)$ is the time-$s$ marginal law of the process $(\bar{\Theta}(s))_{s > 0}$ starting from $\theta \in [0,1]$ (note that $\nu_\theta(0)=\nu_\theta$).  Then, for every $s \in (0,\infty)$,
				\begin{equation}
					\begin{aligned}
						&\lim_{N\to\infty} \CL\left[\left(X^{[N]}(Ns+t),Y^{[N]}(Ns+t)\right)_{t>0}\right]
						= \CL \left[(Z^{\nu_\theta(s)}(t))_{t > 0}\right]\\ 
					\end{aligned}
				\end{equation}
				where, conditional on $\bar{\Theta}(s)=\theta$, $(z^{\nu_\theta(s)}(t))_{t \geq 0}$ is the random process in \eqref{gh5inf} and $z^{\nu_\theta(s)}(0)$ is drawn according to $\nu_\theta(s)$ (which is a mixture of random processes in equilibrium).
			\end{itemize}
		\end{proposition}
		
		\noindent
		The proof of Proposition~\ref{P.finsysmf} is given in Section~\ref{sec:finsysmf}. 
		
		The result in Part (a) shows that the limit dynamics of the averages follows a similar type of diffusion as a single colony, but with four important changes:
		\begin{itemize}
			\item 
			For the limit of the time-scaled average in \eqref{gh1} the diffusion function $g$ is replaced by a \emph{renormalised diffusion function} $\CF g$, defined by \eqref{gh4} (recall Fig.~\ref{fig-renorm}). In section \ref{sec:finsysmf} we will show that $\CF \CG\subset\CG$, i.e., $\CF$ preserves the class of diffusion functions defined in \eqref{setG}.   
			\item
			The average of the dormant population is the same as the average of the active population, and hence the term that accounts for the exchange between the active and the dormant population \emph{vanishes}. This happens because when time is speeded up by a factor $N$ also the rates of exchange between active and dormant are speeded up by a factor $N$ (see \eqref{mfevolve}). Hence the exchange rates become infinitely large, which implies that the active and the dormant population equilibrate instantly in the Meyer-Zheng topology. 
			\item 
			Since we take the average over all the components, the migration terms in \eqref{gh45a} cancel out against each other. 
			\item 
			Comparing the system in \eqref{e715} with the system of interacting Fisher-Wright diffusions in the mean-field limit studied in \cite{DG93a}, we see from \eqref{gh43} that the single-colour seed-bank \emph{slows down the average} by a factor $1/(1+K)$, but does not change the system qualitatively. This is a direct consequence of the fact that the averages of the active and the dormant population equilibrate (due to the infinite rates), while only individuals in the active part of the population resample.
		\end{itemize}
		The result in Part (b) shows that the limit dynamics of the averages in \eqref{gh1} follows an autonomous SDE, with convergence in the classical topology, i.e., in $C_b([0,\infty),[0,1])$. The Brownian motion in \eqref{SC} is taken to be independent of the initial state. The result in Part (c) says that, on time scale $1$ and starting from time $Ns$ with $N\to\infty$, the system has a \emph{McKean-Vlasov limit}, i.e., exhibits propagation of chaos, with components that are versions of a McKean-Vlasov process with a random initial state whose law depends on $s$. So, in particular, the components become independent, and we see \emph{decoupling} The proof of Part (c) will use Part (b). The proof of Part (a) will follow from Part (b) after we use the Meyer-Zheng topology.
		
		\begin{remark}{\bf[Basic multi-scale]}
			{\rm Note that Proposition~\ref{P.finsysmf} already reveals several phenomena that we encountered in Theorems~\ref{T.multiscalehiereff} and \ref{T.multiscalehier}, capturing the hierarchical multiscale behaviour. Even for the one-layer mean-field system we find decoupling of components, the occurrence of a renormalisation transformation, equalisation of the seed-bank with the active population, and the need for the Meyer-Zheng topology. Later we will see that the role of the macroscopic observable $\bar{\Theta}$ is the same as that of the effective process.} \hfill$\blacksquare$
		\end{remark}
		
		\begin{remark}{\bf [Interchange of limits]}
			{\rm The notation $x_1^{\N_0},y_{0,1}^{\N_0}$ indicates that the limit arises from taking averages over $[N]$ and letting $N\to\infty$. Note that, for i.i.d.\ initial states,
				\begin{equation}
					\label{lim1}
					x_1^{\N_0}(0) = \lim_{N\to\infty} x_1^{[N]}(0)
					= \lim_{N\to\infty}\frac{1}{N}\sum_{i \in [N]} x_i(0) =\theta_{\bar{\Theta}} 
					\qquad \P\text{-a.s.}
				\end{equation} 
				On the other hand, picking any sequence of times $L(N)$ such that $\lim_{N\to\infty} L(N)=\infty$ and $\lim_{N\to\infty} L(N)/N=0$, we get
				\begin{equation}
					\label{lim2}
					x_1^{\N_0}(0+) = \lim_{N\to\infty} x_1^{[N]}\big(\tfrac{L(N)}{N} t\big) 
					= \lim_{N\to\infty}\frac{1}{N} \sum_{i \in [N]} x_i(L(N)t) = \theta
					\qquad \P\text{-a.s.}
				\end{equation}
				The mismatch between \eqref{lim1} and \eqref{lim2} indicates that we must be careful with interchanging the limits $N\to\infty$ and $s \downarrow 0$. This is why \eqref{gh42alt}, which lives on the fast time scale, is restricted to $s>0$. } \hfill $\blacksquare$
		\end{remark} 
		
		\subsection{Abstract scheme behind finite-systems scheme}
		\label{sec:finsysmf}
		
		To prove Proposition~\ref{P.finsysmf}, we follow the abstract scheme outlined in \cite[p.\ 2314--2315]{DGV95} and based on \cite{DG93a}, \cite{DG93b}. Below we state the abstract scheme for our model. The scheme consists of 4 steps, each of the steps consists of a series of propositions and lemmas. The proofs of these are given in Section \ref{ss.pabstracts}. 
		
		\paragraph{Step 1. Equilibrium of the single components.}
		
		This step fixes the one-dimensional distributions of the single components when $t,N \to \infty$ in a combined way, and is the equivalent of \cite[Proposition 1]{DG93a}. Recall that $\bar{\Theta}^{[N]}$ is defined in \eqref{slovar}.
		
		\begin{proposition}{\bf [Equilibrium for the infinite system]}
			\label{prop1}
			Let $(N_k)_{k\in\N}$ be a sequence in $\N$. Fix $s>0$. Let $L(N)$ be such that $\lim_{N\to\infty} L(N)=\infty$ and $\lim_{N\to\infty} L(N)/N=0$, and suppose that
			\begin{equation}
				\label{823}
				\begin{aligned}
					&\lim_{k\to\infty}\CL\left[\bar{\Theta}^{[N_k]}(N_ks)\right] = P_s,\\ 
					&\lim_{k\to\infty }\CL\left[\sup_{0\leq t\leq L(N_k)}\left|\bar{\Theta}^{[N_k]}(N_k s)-\bar{\Theta}^{[N_k]}(N_ks-t)
					\right|\right]=\delta_{0},\\[0.2cm]
					&\lim_{k\to\infty} \CL\bigl(X^{[N_k]}(N_ks),Y^{[N_k]}(N_ks)\bigr) = \nu(s).
				\end{aligned}
			\end{equation}
			Then $\nu(s)$ is of the form
			\begin{equation}
				\nu(s) = \int_{[0,1]} P_s(\d \theta)\,\nu_{\theta},
			\end{equation}
			where $\nu_\theta$ is the equilibrium measure of the process defined in \eqref{gh5inf}.
		\end{proposition}
		
		Proposition \ref{prop1} follows from the following seven lemmas, which are the analogues of the five lemmas used in \cite[p.\ 477--478]{DG93a} for the system without seed-bank. 
		
		The first lemma establishes convergence of the infinite system in \eqref{gh5inf} to its equilibrium. 
		
		\begin{lemma}{{\bf [Convergence for the infinite system]}} 
			\label{lemerg}
			Let $\mu$ be an exchangeable probability measure on $([0,1]^2)^{\N_0}$. Then for the system $(Z(t))_{t \geq 0}$ given by \eqref{gh5inf} with $\CL(Z(0))=\mu$,
			\begin{equation}
				\lim_{t\to\infty }\CL[Z(t)]= \nu_\theta,
			\end{equation} 
			where $\nu_\theta$ is of the form
			\begin{equation}
				\nu_\theta=\Gamma_\theta^{\otimes\N_0},
			\end{equation}
			with $\Gamma_\theta$ the equilibrium of the single-colony process defined in \eqref{singcoleq}. Moreover, $\nu_\theta$ is ergodic.
		\end{lemma}
		
		The second lemma establishes the continuity of the equilibrium with respect to its center of drift $\theta$. 
		
		\begin{lemma}{{\bf [Continuity of the equilibrium]}} 
			\label{lemlip} 
			Let $\CP([0,1]^{\N_0})$ denote the space of probability measures on $[0,1]^{\N_0}$. The mapping $[0,1] \to \CP([0,1]^{\N_0})$ given by
			\begin{equation}
				\theta \mapsto \nu_\theta
			\end{equation}
			is continuous. Furthermore, if $h$ is a Lipschitz function on $[0,1]$, then also $\CF h$ defined by 
			\begin{equation}
				(\CF h)(\theta) =\E^{\nu_\theta}[h(\cdot)] = \int_{([0,1]^2)^{\N_0}} \nu_\theta(\d z)\,h(x_0) 
			\end{equation}
			is a Lipschitz function on $[0,1]$. 
		\end{lemma}
		
		The third lemma characterises the speed at which the estimators ${{\Theta}}^{[N]}_x$ and ${{\Theta}}^{[N]}_y$ converge to each other. 
		
		\begin{lemma}{{\bf [Comparison of empirical averages]}}
			\label{lemav}
			Let $({{\Theta}}^{[N]}_x(t))_{t\geq 0}$ and $({\bar{\Theta}}^{[N]}_{y}(t))_{t\geq 0}$ be defined as in \eqref{avxy}, and define
			\begin{equation}
			\gls{difad} = {{\Theta}}^{[N]}_x(t)-{{\Theta}}^{[N]}_y(t).
			\end{equation} 
			Then 
			\begin{equation}
				\label{pr22}
				\E\left[\left|{\Delta}^{[N]}_{\bar{\Theta}}(t)\right|\right]\leq\sqrt{\E\left[\left({\Delta}^{[N]}_{\bar{\Theta}}(0)\right)^2\right]}\,
				\e^{-(Ke+e)t}+ \sqrt{\frac{\|g\|}{N(Ke+e)}}.
			\end{equation}
		\end{lemma}
		\begin{remark}{\bf[Key estimate for Meyer-Zheng convergence]}
			{\rm The estimate in \eqref{pr22} in Lemma~\ref{lemav} will be the key estimate to show convergence of the active and dormant $1$-block in Meyer-Zheng topology. Note that if we look at times $Ns$ for $s>0$, then \eqref{pr22} shows that $\E[|{\Delta}^{[N]}_{\bar{\Theta}}(Ns)|]$ is $\CO(\sqrt{1/N})$.}
		\end{remark}
		The fourth lemma compares the finite system with an infinite system. To that end we construct both the finite and the infinite system on the same state-space by considering the finite system $(X^{[N]}(t),Y^{[N]}(t))$ as an element of $([0,1]^2)^{\N_0}$ via periodic continuation. Let $L(N)$ be such that $\lim_{N\to\infty} L(N)=\infty$ and $\lim_{N\to\infty} L(N)/N=0$, and define the distribution \gls{mun} by continuing the configuration of $(X^{[N]}(Ns-L(N)),Y^{[N]}(Ns-L(N)))$ periodically to $([0,1]^2)^{\N_0}$. Define 
		\begin{equation}
			\label{defthetaN}
			\bar{\Theta}^{[N]} = \bar{\Theta}^{[N]}(Ns-L(N)).
		\end{equation} 
		Note that
		\begin{equation}
			\label{833}
			\begin{aligned}
				\bar{\Theta}^{[N]}
				&=\frac{\frac{1}{N}\sum_{j\in[N]}x^{[N]}_j(Ns-L(N))+\frac{K}{N} \sum_{j\in[N]}y^{[N]}_j(Ns-L(N))}{1+K}\\
				&=\frac{1}{N}\sum_{j\in[N]}\frac{x_j^{\mu_N}(0)+K y_j^{\mu_N}(0)}{1+K}.
			\end{aligned}
		\end{equation} 
		Thus, $\bar{\Theta}^{[N]}$ is a random variable whose law depends on $\CL\left[X^{[N]}(Ns-L(N)),Y^{[N]}(Ns-L(N))\right]$ $=\mu_N$. The infinite system  with initial law $\mu_N$ is denoted by 
		\begin{equation}
			\gls{ifp} = \big(x^{\mu_N}_i(t),y^{\mu_N}_i(t)\big)_{i\in\N_0,\ t\geq 0}
		\end{equation}
		and evolves according to
		\begin{equation}
			\begin{aligned}
				\label{gh45ainf}
				&\d x^{\mu_N}_i(t) = c\big[\bar{\Theta}^{[N]} -x^{\mu_N}_i(t)\big]\,\d t + \sqrt{g(x^{\mu_N}_i(t))}\, \d w_i (t) 
				+ K e\, [y^{\mu_N}_i(t)-x^{\mu_N}_i(t)]\,\d t,\\
				&\d y^{\mu_N}_i(t) = e\,[x^{\mu_N}_i(t)-y^{\mu_N}_i(t)]\, \d t, \qquad i \in \N_0,
			\end{aligned}
		\end{equation}
		where $\{w_i\}_{i\in\N_0}$ is a collection of independent Brownian motions.
		
		\begin{lemma}{{\bf [Comparison of finite and infinite systems]}}
			\label{l.comp}
			Fix $s > 0$ and assume that, for any $L(N)$ satisfying $\lim_{N\to\infty} L(N) = \infty$ and $\lim_{N\to\infty} L(N)/N=0$,
			\begin{equation}
				\label{m08}
				\begin{aligned}
					&\lim_{N\to\infty} \sup_{0 \leq t \leq L(N)} \big|\bar{\Theta}^{[N]}(Ns)
					-\bar{\Theta}^{[N]}(Ns-t)\big| = 0\ \text{ in probability}.
				\end{aligned}
			\end{equation}
			Let
			\begin{equation}
				\bigl(X^{\mu_N}(t),Y^{\mu_N}(t)\bigr)_{t \geq 0}
			\end{equation}
			be the infinite system defined in \eqref{gh45ainf} starting in the distribution $\mu_N$, where $\mu_N$ is defined by continuing the configuration of $(X^{[N]}(Ns-L(N)),Y^{[N]}(Ns-L(N)))$ periodically to $([0,1]^2)^{\N_0}$. Similarly, view $(X^{[N]}(t),Y^{[N]}(t))$ as an element of $([0,1]^2)^{\N_0}$ by periodic continuation. Then, for all $t\geq 0$,
			\begin{equation}
				\label{m32}
				\begin{aligned}
					&\lim_{N\to\infty} \bigl|\E\bigl[f\bigl(X^{\mu_{N}}(t),Y^{\mu_{N}}(t)\bigr)
					-f\bigl(X^{[N]}(Ns-L(N)+t),Y^{[N]}(Ns-L(N)+t)\bigr)\bigr]\bigr| = 0\\
					&\qquad \forall\, f\in\CC\bigl(([0,1]^2)^{\N_0},\R\bigr).
				\end{aligned}
			\end{equation}
		\end{lemma}
		
		Before we can prove that the infinite system $(X^{\mu_N}(t),Y^{\mu_N}(t))_{t \geq 0}$ converges to some limiting system as $N\to\infty$, we need the following regularity property for the estimator $\bar{\Theta}^{[N]}$. This is stated in our fifth lemma. 
		
		\begin{lemma}{{\bf [Stability of the estimator for the conserved quantity]}} 
			\label{stabest} 
			Define $\mu_N$ as in Lemma~\ref{l.comp}. Let $(x_i, y_i)_{i\in [N]}$ be distributed according to the exchangeable probability measure $\mu_N$ on $([0,1]^2)^{\N_0}$ restricted to $([0,1]^2)^{[N]}$. Suppose that $\lim_{N \to \infty}\mu_N=\mu$ for some exchangeable probability measure $\mu$ on $([0,1]^2)^{\N_0}$. Define a random variable $\phi$ on $(\mu,([0,1]^2)^{\N_0})$ by putting
			\begin{equation}
				\label{phi}
				\phi = \lim_{n\to\infty} \frac{1}{n} \sum_{i \in [n]} \frac{x_i+Ky_i}{1+K},
			\end{equation}
			and a random variable $\phi_N$ on $(\mu_N,([0,1]^2)^{\N_0})$ by putting
			\begin{equation}
				\phi_{N} = \frac{1}{N}\sum_{i\in[N]} \frac{x_i+Ky_i}{1+K}.
			\end{equation}
			Then 
			\begin{equation}
				\lim_{N\to\infty}\CL[\phi_{N}] = \CL[\phi].
			\end{equation}
		\end{lemma}
		
		In the sixth lemma we state the convergence of the law $\CL[(X^{\mu_N}(t),Y^{\mu_N}(t))]$ to the law of a limiting system as $N\to\infty$.   
		
		\begin{lemma}{\bf [Uniformity of the ergodic theorem for the infinite system]} 
			\label{unifergod} 
			Let $\mu_N$ be defined as in Lemma~\ref{l.comp}. Since $(\mu_N)_{N\in\N}$ is tight, it has convergent subsequences. Let $(N_k)_{k\in\N}$ be a subsequence such that $\mu=\lim_{k\to\infty} \mu_{N_k}$. Define 
			\begin{equation}
				\label{111}
				\bar{\Theta}=\lim_{N\to\infty}\frac{1}{N}\sum_{i\in[N]}\frac{x_i^\mu+Ky_i^{\mu}}{1+K}\, \qquad \text{ in }L_2(\mu),
			\end{equation}
			and let $(X^{\mu}(t),Y^{\mu}(t))_{t\geq 0}$ be the infinite system evolving according to 
			\begin{equation}
				\begin{aligned}
					\label{binfb}
					&\d x^{\mu}_i(t) = c\left[\bar{\Theta} -x^{\mu}_i(t)\right]\d t + \sqrt{g(x^{\mu}_i(t))}\, \d w_i (t) 
					+ K e\, [y^{\mu}_i(t)-x^{\mu}_i(t)]\,\d t,\\
					&\d y^{\mu}_i(t) = e\,[x^{\mu}_i(t)-y^{\mu}_i(t)]\, \d t, \qquad i \in \N_0.
				\end{aligned}
			\end{equation}
			Then	
			\begin{enumerate}
				\item For all $t \geq 0$,
				\begin{equation}
					\begin{aligned}
						&\lim_{k\to\infty} \bigl|\E\bigl[f\bigl(X^{\mu_{N_k}}(t),Y^{\mu_{N_k}}(t)\bigr)\bigr]
						-\E\bigl[f\bigl(X^{\mu}(t),Y^{\mu}(t)\bigr)\bigr]\bigr| = 0,\\
						&\qquad\forall\, f\in\CC\bigl(([0,1]^2)^{\N_0},\R\bigr).
					\end{aligned}
				\end{equation}
				\item 
				There exists a sequence $(\bar{L}(N))_{N\in\N}$ satisfying $\lim_{N \to \infty}\bar{L}(N)=\infty$ and $\lim_{N \to \infty}\bar{L}(N)/N=0$ such that
				\begin{equation}
					\label{m32c}
					\begin{aligned}
						\lim_{k\to\infty}& \bigl|\E\bigl[f\bigl(X^{[N_k]}(N_ks-L(N_k)+\bar{L}(N_k)),Y^{[N_k]}(N_ks-L(N_k)+\bar{L}(N_k))\bigr)\\
						&\qquad - f\bigl(X^{\mu_{N_k}}(\bar{L}(N_k)),Y^{\mu_{N_k}}(\bar{L}(N_k))\bigr)\bigr|\bigr]\\
						& +\bigl|\E\bigl[f\bigl(X^{\mu_{N_k}}(\bar{L}(N_k)),Y^{\mu_{N_k}}(\bar{L}(N_k))\bigr)\bigr]
						-\E\bigl[f\bigl(X^{\mu}(\bar{L}(N_k)),Y^{\mu}(\bar{L}(N_k))\bigr)\bigr]\bigr| =  0\\
						&\qquad \forall\, f\in\CC\bigl(([0,1]^2)^{\N_0},\R\bigr).
					\end{aligned}
				\end{equation}
			\end{enumerate}
		\end{lemma}
		
		\begin{remark}{\bf [Existence of $\bar{\Theta}$]}
			{\rm Note that the limit in \eqref{111} is well-defined by the ergodic theorem in $L_2$, since $\mu$ is the limit of translation invariant measures and hence is itself translation invariant.}
		\end{remark}
		
		In the seventh lemma we provide a coupling of two copies of the finite system starting from different measures.  
		
		\begin{lemma}{\bf[Coupling of finite systems]}
			\label{lem:12}
			Let $(X^{[N],1},Y^{[N],1})$ be a finite system evolving according to \eqref{gh45a} and starting from some exchangeable measure. Let $\mu^{[N],1}$ be the measure obtain by periodic continuation of the configuration of $(X^{[N],1}(0),Y^{[N],1}(0))$. Similarly, let $(X^{[N],2},Y^{[N],2})$ be a finite system evolving according to \eqref{gh45a} and starting from some exchangeable measure. Let $\mu^{[N],2}$ be the measure obtain by periodic continuation of the configuration of $(X^{[N],2}(0),Y^{[N],2}(0))$. Let $\tilde{\mu}$ be any weak limit point of the sequence of measures $\{\mu^{[N],1}\times\mu^{[N],2}\}_{N\in\N}$. Define random variables $\bar{\Theta}^{[N],1}$ on $(\mu^{[N],1},([0,1]^2)^{\N_0})$, $\bar{\Theta}^{[N],2}$ on $(\mu^{[N],1},([0,1]^2)^{\N_0})$ and $\bar{\Theta}_1$ and $\bar{\Theta}_2$ on $(\mu,([0,1]^2)^{\N_0})$ by
			\begin{equation}
				\label{843}
				\begin{aligned}
					&\bar{\Theta}^{[N],1} = \frac{1}{N} \sum_{i \in [N]} \frac{x^{[N],1}_{i}+Ky^{[N],1}_{i}}{1+K},
					\qquad \bar{\Theta}^{[N],2} =  \frac{1}{N} \sum_{i \in [N]} \frac{x^{[N],2}_{i}+Ky^{[N],2}_{i}}{1+K},\\
					&\bar{\Theta}_1 = \lim_{n\to\infty} \frac{1}{n} \sum_{i \in [n]} \frac{x^1_{i}+Ky^1_{i}}{1+K},
					\qquad \bar{\Theta}_2 
					= \lim_{n\to\infty} \frac{1}{n} \sum_{i \in [n]} \frac{x^2_{i}+Ky^2_{i}}{1+K},
				\end{aligned}
			\end{equation}
			and let $(\bar{\Theta}^{[N],1}(t))_{t\geq 0}$ and $(\bar{\Theta}^{[N],2}(t))_{t\geq 0}$ be defined according to  \eqref{slovar} for $(X^{[N]}_1,Y^{[N]}_1)$, respectively, $(X^{[N]}_2,Y^{[N]}_2)$. Assume that   
			\begin{equation}
				\label{m08alt}
				\begin{aligned}
					&\lim_{N\to\infty} \sup_{0 \leq t \leq L(N)} \big|\bar{\Theta}^{[N],k}(0)-\bar{\Theta}^{[N],k}(t)\big| 
					= 0\ \text{ in probability}, \quad k \in\{1,2\},
				\end{aligned}
			\end{equation}
			and suppose that  $\tilde{\mu}(\{\bar{\Theta}_1=\bar{\Theta}_2\})=1$. Then, for any sequence $t(N)\to\infty$, 
			\begin{equation}
				\lim_{N\to\infty}\E\left[\big|x^{[N],1}_{i}(t(N))-x^{[N],2}_{i}(t(N))\big|+K\big|y^{[N],1}_{i}(t(N))-y^{[N],2}_{i}(t(N))\big|\right]=0.
			\end{equation}  
		\end{lemma}
		
		\paragraph{Step 2. Convergence of the estimator.} 
		
		This step is the equivalent of \cite[Proposition 2]{DG93a}. We first prove the tightness of the estimator $\bar{\Theta}^{[N]}$ in path space. After that we settle convergence of the finite-dimensional distributions and identify the limit. 
		
		\begin{proposition}{\bf [Convergence of average sum process]}
			\label{p.esti}
			\begin{equation}
				\lim_{N\to\infty} \CL\left[\left(\bar{\Theta}^{[N]}(Ns)\right)_{s > 0}\right] = \CL[(\bar{\Theta}(s))_{s > 0}],
			\end{equation}
			where $(\bar{\Theta}(s))_{s > 0}$ evolves according to 
			\begin{equation}
				\d\bar{\Theta}(s) = \frac{1}{(1+K)}\sqrt{(\CF g) (\bar{\Theta}(s))}\,\d w(s).
			\end{equation}
		\end{proposition}
		
		Proposition \ref{p.esti} follows from the following three lemmas, which are the equivalent of the three lemmas used in \cite[p.\ 488--493]{DG93a} for the system without seed-bank.
		
		\begin{lemma}{\bf [Martingale property of average sum process]} 
			\label{martav}
			$\mbox{}$
			\begin{enumerate}
				\item[{\rm (1)}] 
				The process $(\bar{\Theta}^{[N]}(Ns))_{s > 0}$ is a square-integrable martingale with continuous paths and increasing process
				\begin{equation}
					\left\langle \bar{\Theta}^{[N]}(Ns)\right\rangle_{s> 0} 
					= \frac{1}{(1+K)^2} \int_{0}^{s} \d r\,\frac{1}{N} \sum_{i\in[N]}{g\big(x_i^{[N]}(Nr)\big)}.
				\end{equation}
				\item[{\rm (2)}] 
				Let $L(N)$ be such that $\lim_{N\to \infty} L(N)=\infty$ and $\lim_{N\to \infty} L(N)/N=0$. Then 
				\begin{equation}
					\label{a7}
					\lim_{N\to\infty} \sup_{0 \leq t \leq L(N)}\big|\bar{\Theta}^{[N]}(Ns)-\bar{\Theta}^{[N]}(Ns-t)\big | = 0 
					\text{ in probability.}
				\end{equation} 
				\item[{\rm (3)}] 
				$ \left(\CL[(\bar{\Theta}^{[N]}(Ns))_{s > 0}]\right)_{N\in\N}$ is tight as a sequence of probability measures on $\mathcal{C}([0,\infty),[0,1])$.
			\end{enumerate}
		\end{lemma}
		
		\begin{lemma}{\bf [Martingale property of limit process]}
			\label{lemmart}
			Let $(N_k)_{k\in\N}$ be any subsequence such that
			\begin{equation}
				\lim_{k\to\infty} \CL\left[\big(\bar{\Theta}^{[N_k]}(N_ks)\big)_{s > 0}\right] 
				= \CL\bigl[(\bar{\Theta}(s))_{s > 0}\bigr].
			\end{equation}
			Then $(\bar{\Theta}(s))_{s > 0}$ is a square-integrable martingale with continuous paths, and
			\begin{equation}
				\label{a1}
				\left(\bar{\Theta}^2(s)-\int_{0}^{s} \d r\,\frac{1}{(1+K)^2}\,\E^{\nu_{\bar{\Theta}(r)}}[g(x_0)]\right)_{s > 0}
			\end{equation} 
			is a martingale.
		\end{lemma}
		
		\begin{lemma}{\bf [Uniqueness]}\label{lem:uni}
			The following martingale problem has a unique solution:
			\begin{equation}
				\begin{aligned}
					&(\bar{\Theta}_s)_{s> 0} \text{ is a continuous martingale with values in } [0,1],
					\label{martprob}\\
					&\left(\bar{\Theta}^2(s)-\frac{1}{(1+K)^2} \int_{0}^{s} \d r\,\E^{\nu_{\bar{\Theta}(r)}}[g(x_0)]\right)_{s > 0} 
					\text{ is a martingale.}
				\end{aligned}
			\end{equation}
			The solution of \eqref{martprob} is given by the diffusion generated by $\E^{\nu_{u}}[g(\cdot)]\frac{\partial^2}{\partial u^2}$.
		\end{lemma}
		
		\paragraph{Step 3. Convergence of the averages in the Meyer-Zheng topology.}
		
		Recall the definition of the Meyer-Zheng topology in Section~\ref{MeyerZheng}. We have to prove the following proposition.
		
		\begin{proposition}{\bf [Convergence in Meyer-Zheng topology]}
			\label{p.estim}
			If
			\begin{equation}
				\lim_{N \to \infty}\CL\left[\left(\bar{\Theta}(Ns)\right)_{s>0}\right]=\CL\left[\left({ \bar{\Theta}}(s)\right)_{s>0}\right],
			\end{equation}
			then
			\begin{equation}
				\begin{aligned}
					&\lim_{N\to\infty}\CL\Bigl[\bigl(x_1^{[N]}(t),y_1^{[N]}(t)\bigr)_{t \geq 0}\Bigr] 
					= \CL\Bigl[\bigl(x_1^{\N_0}(t),y_1^{\N_0}(t)\bigr)_{t \geq 0}\Bigr] \\
					&\mbox{ in the Meyer-Zheng topology},
				\end{aligned}
			\end{equation}
			where $(x_1^{\N_0}(t),y_1^{\N_0}(t))_{t \geq 0})$ evolves according to \eqref{gh43}.
		\end{proposition}
		
		\noindent
		To prove Proposition~\ref{p.estim} we will use Lemma~\ref{lemav} in combination with the following three general lemmas about the Meyer-Zheng topology, which are proven in Appendix~\ref{apb3}. 
		
		\begin{lemma}{\bf[Path convergence in probability in the Meyer-Zheng topology]}
			\label{lem91} 
			Let $((Z_n(t))_{t\geq 0} )_{n\in\N}$ and $(Z(t))_{t\geq 0}$ be stochastic processes on the Polish space $(E,d)$. If, for all $t\geq 0$, 
			\begin{equation}
				\label{t}
				\lim_{n\to\infty}\E\left[d(Z_n(t),Z(t))\right]=0,
			\end{equation}
			then, 
			\begin{equation}
				\lim_{n\to\infty} (Z_n(t))_{t\geq 0}=(Z(t))_{t\geq 0} \text{ in probability in the Meyer-Zheng topology}.
			\end{equation}	
		\end{lemma}
		
		\begin{lemma}{\bf[Convergence of the joint law of paths in the Meyer-Zheng topology]}
			\label{lem93} 
			Let $(X_n)_{n\in\N},(Y_n)_{n\in\N}, X$ be stochastic processes on a metric space $(E, d)$ and let $c\in E$ be a constant. If $\lim_{n\to \infty}\CL[X_n]=\CL[X]$ in the Meyer-Zheng topology and for all $t\geq0$,  $\lim_{n\to \infty}\E[d(Y_n(t),c)]=0$, then $\lim_{n\to \infty}\CL[(X_n,Y_n)]=\CL[(X,c)]$ in the Meyer-Zheng topology.
		\end{lemma}
		
		\begin{lemma}{\bf[Continuous mapping theorem in the Meyer-Zheng topology]}
			\label{lem94}
			Let $f\colon\, E\to E$ be a continuous function and $x\in M_E[0,\infty)$.  
			\begin{itemize}
				\item[(a)] The function 
				\begin{equation}
					h\colon\, \Psi\to\Psi,\qquad \psi_{x}\to\psi_{f(x)},
				\end{equation}
				is continuous. 
				\item[(b)] If the stochastic processes $(X_n)_{n\in\N}, X$ on state space $(E, d)$ satisfy 
				\begin{equation}
					\lim_{n\to \infty}\CL[X_n]=\CL[X]\text{ in the  Meyer-Zheng topology},
				\end{equation}
				then 
				\begin{equation}
					\lim_{n\to \infty}\CL[f(X_n)]=\CL[f(X)]\text{ in the Meyer-Zheng topology.}
				\end{equation}
			\end{itemize}
		\end{lemma}
		
		Note that Lemma~\ref{lem94} allows us to use the continuous mapping theorem in the Meyer-Zheng topology.

		\paragraph{Step 4. Mean-field finite-systems scheme.}
		
		Use Steps 1-- 4 to prove Proposition \ref{P.finsysmf}. 
		
		\medskip
		Having completed the abstract scheme of steps 1--4, we set out to prove the constituent propositions and lemmas.

		
		\section{Proofs: $N\to\infty$, mean-field, proof of abstract scheme}
		\label{ss.pabstracts}
		
		In Sections~\ref{step1}--\ref{step4} we prove the propositions and the lemmas stated in Steps 1--4 in Section~\ref{sec:finsysmf}.    
		
		\subsection {Proof of step 1. Equilibrium of the single components}
		\label{step1}
		
		We start by proving Proposition \ref{prop1} with the help of the seven lemmas stated in Step 1 of Section~\ref{sec:finsysmf}. Afterwards we prove each of the lemmas.
		
		\paragraph{$\bullet$ Proof of Proposition \ref{prop1}}
		
		\begin{proof}
			We use an argument similar to the one used in \cite[Section 2 (i)]{DG93a}. Let $(L(N))_{N\in\N}$ be any sequence satisfying $\lim_{N \to \infty}L(N)=\infty$ and $\lim_{N\to\infty} L(N)/N=0$. Let $\mu_{N}$ be the measure on $([0,1]^2)^{\N_0}$ obtained by periodic continuation of $\CL[X^{[N]}(Ns-L(N)),Y^{[N]}(Ns-L(N))]$. Note that $([0,1]^2)^{\N_0}$ is compact. Hence, letting $(N_k)_{k\in\N}$ be the subsequence in Proposition~\ref{prop1}, we can pass to a further subsequence and obtain 
			\begin{equation}
				\lim_{k\to\infty} \mu_{N_k} = \mu.
			\end{equation}
			Since we assumed that $\CL[X^{[N]}(0),Y^{[N]}(0)]$ is exchangeable and the dynamics preserves exchangeability, the measures $\mu_{N_k}$ are exchangeable and also the limiting law $\mu$ is exchangeable. Define $\phi$ as in \eqref{phi} in Lemma \ref{stabest}. Then we can condition on $\phi$ and write
			\begin{equation}
				\mu = \int_{[0,1]} \mu_\rho\, \d \Lambda(\rho), 
			\end{equation}
			where $\Lambda(\cdot)=\CL[\phi]$. By assumption we know that 
			\begin{equation}
				\lim_{k\to\infty} \CL\big[\bar{\Theta}^{[N_k]}(N_ks)\big] = P_s
			\end{equation}
			and
			\begin{equation}
				\lim_{k\to\infty} \CL\left[\sup_{0\leq t\leq L(N_k)}\left|\bar{\Theta}^{[N_k]}(N_ks)
				-\bar{\Theta}^{[N_k]}(N_ks-t)\right|\right]=\delta_0.
			\end{equation}
			Hence
			\begin{equation}
				\begin{aligned}
					\lim_{k\to\infty} \CL\left[\bar{\Theta}^{[N_k]}(N_ks-L(N_k))\right] = P_s.
				\end{aligned}
			\end{equation}
			
			Recall that
			\begin{equation}
				\Lambda =\CL[\phi] =\CL \left[\lim_{n\to\infty}\frac{1}{n}\sum_{i\in[n]} 
				\frac{x_i+Ky_i}{1+K}\right] \quad \text{ on } (\mu,([0,1]^2)^{\N_0}).
			\end{equation}
			By Lemma \ref{stabest}, if $\phi_{N_k}=\frac{1}{N_k} \sum_{i\in[N_k]} \frac{x_i+Ky_i}{1+K}$  on $(\mu_{N_k},([0,1]^2)^{\N_0})$, then $\lim_{k\to\infty} \CL[\phi_{N_k}]=\CL[\phi]$. Taking the subsequence $(\mu_{N_k})_{k\in\N}$, we get $\Lambda(\cdot)=P_s(\cdot)$, and hence 
			\begin{equation}
				\label{97}
				\mu=\int_{[0,1]} \mu_\rho\, \d P_s(\rho).
			\end{equation}
			
			Let $\bar{L}(N)$ be the sequence constructed in Lemma~\ref{unifergod}[b]. We can require that $\bar{L}(N) \leq L(N)$ for all $N\in\N$. Write
			\begin{equation}
				\begin{aligned}
					&\CL\bigl[X^{[N_k]}(N_ks-L(N_k)+\bar{L}(N_k)),Y^{[N_k]}(N_ks-L(N_k)+\bar{L}(N_k))\bigr]\\
					&= \CL\bigl[X^{[N_k]}(N_ks-L(N_k)+\bar{L}(N_k)),Y^{N_k}(N_ks-L(N_k)+\bar{L}(N_k))\bigr]\\
					&\quad -\CL\bigl[X^{\mu_{N_k}}(\bar{L}(N_k)),Y^{\mu_{N_k}}(\bar{L}(N_k))\bigr],\\
					&\label{triangle}
					\quad +\CL\bigl[X^{\mu_{N_k}}(\bar{L}(N_k)),Y^{\mu_{N_k}}(\bar{L}(N_k))\bigr]
					-\CL\bigl[X^{\mu}(\bar{L}(N_k)),Y^{\mu}(\bar{L}(N_k))\bigr]\\
					&\quad +\CL\bigl[X^{\mu}(\bar{L}(N_k)),Y^{\mu}(\bar{L}(N_k))\bigr].
				\end{aligned}
			\end{equation}
			By Lemma \ref{unifergod}, the first and the second term tend to zero as $k\to\infty$. Hence
			\begin{equation}
				\CL\bigl[X^{[N_k]}(N_ks-L(N_k)+\bar{L}(N_k)),Y^{[N_k]}(N_ks-L(N_k)+\bar{L}(N_k))\bigr]
			\end{equation}
			tends to $ \CL\bigl[X^{\mu}(L(N_k)),Y^{\mu}(L(N_k))\bigr]$ as $k\to\infty$. By \eqref{97},
			\begin{equation}
				\CL\bigl[X^{\mu}(\bar{L}(N_k)),Y^{\mu}(\bar{L}(N_k))\bigr]
				=\int_{[0,1]} \CL\bigl[X^{\mu_\rho}(\bar{L}(N_k)),Y^{\mu_\rho}(\bar{L}(N_k))\bigr]\,\d P_s (\rho).
			\end{equation}
			Since $\lim_{k \to \infty}\bar{L}(N_k)=\infty$, by Lemma \ref{lemerg} we have
			\begin{equation}
				\lim_{k\to\infty} \CL\bigl[X^{\mu_\rho}(\bar{L}(N_k)),Y^{\mu_\rho}(\bar{L}(N_k))\bigr] = \nu_\rho.
			\end{equation}
			Therefore, by \eqref{triangle} and Lemma \ref{lemlip},
			\begin{equation}
				\label{911}
				\begin{aligned}
					&\lim_{k\to\infty} \CL\bigl[X^{[N_k]}(N_ks-L(N_k)+\bar{L}(N_k)),Y^{[N_k]}(N_ks-L(N_k)+\bar{L}(N_k))\bigr]\\
					& \qquad = \int_{[0,1]} \nu_\rho\, \d P_s(\rho).
				\end{aligned}
			\end{equation}
			
			To show that
			\begin{equation}
				\lim_{k\to\infty} \CL\bigl[X^{[N_k]}(N_ks),Y^{[N_k]}(N_ks)\bigr] = \int_{[0,1]} \nu_\rho\, \d P_s(\rho).
			\end{equation}
			we invoke Lemma~\ref{lem:12}. Let $(X^{[N],1},Y^{[N],1})$ be the finite system starting from 
			\begin{equation}
				\CL\big[X^{[N]}(Ns-L(N)),Y^{[N]}(Ns-L(N))\big],
			\end{equation}
			and let $(\bar{L}(N))_{N\in\N}$ be the sequence such that \eqref{911} holds. Let $(X^{[N],2}.Y^{[N],2})$ be the finite system starting from 
			\begin{equation}
				\CL\big[X^{[N]}(Ns-\bar{L}(N)),Y^{[N]}(Ns-\bar{L}(N))\big].
			\end{equation}
			Choose for the sequence $t(N)$ in Lemma~\ref{lem:12} the sequence $\bar{L}(N)$. Let $\mu^{[N],1}$ be defined by periodic continuation of $(X^{[N]}(Ns-L(N)),Y^{[N]}(Ns-L(N)))$, and $\mu^{[N],2}$ by periodic continuation of $(X^{[N]}(Ns-\bar{L}(N)),Y^{[N]}(Ns-\bar{L}(N)))$. Defining $\bar{\Theta}_1$ and $\bar{\Theta}_2$ according to \eqref{843}, where for $\mu^{[N],2}$ we replace $L(N)$ by $\bar{L}(N)$, we get 
			\begin{equation}
				\lim_{k\to\infty}|\bar{\Theta}^{N_k}_{1}-\bar{\Theta}^{N_k}_2|
				= \lim_{k\to\infty} \big|\bar{\Theta}^{N_k}(N_ks-L(N_k))-\bar{\Theta}^{N_k}(N_ks-\bar{L}(N_k))\big |=0
				\text{ in probability} 
			\end{equation}
			by the assumptions in \eqref{823}. Hence, if $\mu$ is any weak limit point of the sequence $(\mu^{[N_k],1}\times\mu^{[N_k],2})_{k\in\N}$, then
			\begin{equation}
				\mu(\bar{\Theta}_1=\bar{\Theta}_2)=1. 
			\end{equation}
			By passing to a further subsequence, we can now apply Lemma~\ref{lem:12}, to obtain
			\begin{equation}
				\label{ca1}
				\lim_{k\to\infty}\E\left[|x^{N_k}_{i,1}(\bar{L}(N_k))-x^{N_k}_{i,2}(\bar{L}(N_k))|
				+K|y^{N_k}_{i,1}(\bar{L}(N_k))-y^{N_k}_{i,2}(\bar{L}(N_k))|\right]=0.
			\end{equation}
			Note that 
			\begin{equation}
				\label{c2}
				\begin{aligned}
					&\CL\big[X_1(\bar{L}(N_k),Y_1(\bar{L}(N_k)))\big]\\
					&=\CL\big[X^{[N_k]}(N_ks-L(N_k)+\bar{L}(N_k)),Y^{[N_k]}(N_ks-L(N_k)+\bar{L}(N_k))\big],\\
					&\CL\big[X_2(\bar{L}(N_k),Y_2(\bar{L}(N_k)))\big] = \CL\big[X^{[N_k]}(N_ks),Y^{[N_k]}(N_ks)\big].
				\end{aligned}
			\end{equation}
			Moreover, we know from \eqref{911} that
			\begin{equation}
				\label{c3}
				\lim_{k\to\infty}\CL\left[X^{[N_k]}(N_ks-L(N_k)+\bar{L}(N_k)),Y^{[N_k]}(N_ks-L(N_k)+\bar{L}(N_k))\right]
				=\int_{[0,1]}\nu_\rho P_s(\d \rho).
			\end{equation}
			Combining \eqref{ca1}--\eqref{c3}, we find that
			\begin{equation}
				\CL\big[X^{[N_k]}(N_ks),Y^{[N_k]}(N_ks)\big]= \int_{[0,1]}\nu_\rho P_s(\d \rho).
			\end{equation}
		\end{proof}
		
		In the remainder of this section we prove Lemmas~\ref{lemerg}--\ref{unifergod} and \ref{lem:12}.
		
		\paragraph{$\bullet$ Proof of Lemma \ref{lemerg}}
		\label{L.21}
		
		\begin{proof}
			Since the components of the infinite system in \eqref{gh5inf} evolve independently, it is enough to show that each component converges to $\Gamma_\theta$. This convergence follows from the proof of Proposition \ref{P.equergod} (see Section \ref{sec:equergod}). Hence the infinite system defined by \eqref{gh52} converges to $\nu_\theta=\Gamma_\theta^{\otimes\N_0}$. Ergodicity of $\nu_\theta$ with respect to translations follows from Kolmogorov's zero-one law. 
		\end{proof}

		\paragraph{$\bullet$ Proof of Lemma~\ref{lemav}}
		
		\begin{proof}
			Using the definition of ${{\Theta}}^{[N]}_x(t)$, ${{\Theta}}^{[N]}_y(t)$ in \eqref{avxy} and the SSDE in \eqref{gh45a}, we find the following evolution for the averages:
			\begin{equation}
				\begin{aligned}
					\d {{\Theta}}^{[N]}_x(t) &=\frac{1}{N}\sum_{i\in[N]} \sqrt{g(x_i(t))}\,\d w_i(t)
					+Ke\,[{{\Theta}}^{[N]}_y(t)-{{\Theta}}^{[N]}_x(t)]\,\d t,\\
					\d {{\Theta}}^{[N]}_y(t) &=e\,[{{\Theta}}^{[N]}_x(t)-{{\Theta}}^{[N]}_y(t)]\,\d t.
				\end{aligned}
			\end{equation}
			Consequently,
			\begin{equation}
				\begin{aligned}
					\d\big({\Delta}^{[N]}_{\bar{\Theta}}(t)\big)^2
					&=2{\Delta}^{[N]}_{\bar{\Theta}}(t)\, \d {\Delta}^{[N]}_{\bar{\Theta}}(t)+2\d \langle{\Delta}^{[N]}_{\bar{\Theta}}\rangle(t)\\
					&= {\Delta}^{[N]}_{\bar{\Theta}}(t) \frac{1}{N}\sum_{i\in[N]} \sqrt{g(x_i(t))}\,\d w_i(t)
					- (Ke+e)\,\big({\Delta}^{[N]}_{\bar{\Theta}}(t)\big)^2\,\d t\\
					&\qquad+2\frac{1}{N^2}\sum_{i\in[N]}g(x_i(t))\ \d t,
				\end{aligned}
			\end{equation}
			and hence
			\begin{equation}
				\frac{\d}{\d t}\,\E\left[\big({\Delta}^{[N]}_{\bar{\Theta}}(t)\big)^2\right]
				= -2(Ke+e)\,\E\left[\big({\Delta}^{[N]}_{\bar{\Theta}}(t)\big)^2\right]+\frac{2}{N^2} \sum_{i\in[N]} g(x_i(t))
			\end{equation}
			and
			\begin{equation}
				\label{pr12}
				\E\left[\big({\Delta}^{[N]}_{\bar{\Theta}}(t)\big)^2\right]=\E\left[\big({\Delta}^{[N]}_{\bar{\Theta}}(0)\big)^2\right]
				\e^{-2(Ke+e)t}+\int_0^t \d r\,\e^{-2(Ke+e)(t-r)}\frac{2}{N^2} \sum_{i\in[N]} g(x_i(r)).
			\end{equation}
			Therefore we get the bound
			\begin{equation}
				\label{pr2}
				\E\left[\big|{\Delta}^{[N]}_{\bar{\Theta}}(t)\big|\right]\leq\sqrt{\E\left[\big({\Delta}^{[N]}_{\bar{\Theta}}(0)\big)^2\right]}
				\e^{-(Ke+e)t}+ \sqrt{\frac{2\|g\|}{N(Ke+e)}}.
			\end{equation}
		\end{proof}
		
		
		\paragraph{$\bullet$ Proof of Lemma~\ref{l.comp}}
		\label{os}
		
		\begin{proof}
			To compare the systems in \eqref{gh45a} and \eqref{gh45ainf}, we couple them via their Brownian motions. Therefore for all $i\in[N]$ we assume that the evolution in \eqref{gh45a} and \eqref{gh45ainf} is driven by the same Brownian motion, $\tilde{w}_i=w_i$. If $i\notin [N]$, then we set $w_i=w_j$ for $j=i \mod N$. We denote the coupled process by $\tilde{z}(t)=(\tilde{z}_{i}(t))_{i\in\N_0} = (\tilde{z}^{[N]}_{i}(t),\tilde{z}^{\mu_N}_{i}(t))_{i\in\N_0}$, where $\tilde{z}^{[N]}_{i}(t)=(\tilde x^{[N]}_i(t),\tilde y^{[N]}_i(t))$ and $\tilde{z}^{\mu_N}_i(t)=(\tilde x^{\mu_N}_i(t),\tilde y^{\mu_N}_i(t))$. The tilde indicates that we are considering the coupled process, and 
			\be\label{bl}
			\begin{aligned}
				\CL[\tilde{z}(0)]
				&= \CL\big[X^{[N]}(Ns-L(N)),Y^{[N]}(Ns-L(N))\big]\times\mu_N\\
				&= \CL\big[X^{[N]}(Ns-L(N)),Y^{[N]}(Ns-L(N))\big]^2.
			\end{aligned}
			\ee 
			Define 
			\begin{equation}
				\label{mo19}
				\Delta^N_i(t)=\tilde x_i^{[N]}(t)-\tilde x_i^{\mu_N}(t), \qquad \delta^N_i(t)=\tilde y_i^{[N]}(t)-\tilde y_i^{\mu_N}(t).
			\end{equation}
			To prove that the coupling is successful, we show that, for all $t\geq 0$,
			\begin{equation}
				\label{mo20}
				\lim_{N\to\infty} \E\left[|\Delta_i^N(t)|+K|\delta_i^N(t)|\right] = 0 \quad \forall\, i\in\N_0. 
			\end{equation}
			From now on we will only consider sites $i\in [0,N]$ for which both infinite systems have the same Brownian motion.
			
			From \eqref{gh45a} and \eqref{gh45ainf} it follows that 
			\begin{equation}
				\label{mo23}
				\begin{aligned}
					&\d \left[|\Delta^N_i(t)|+K |\delta^N_i(t)|\right]\\  
					&\quad =(\sign\, \Delta^N_i(t))\ \d\Delta^N_i(t)+ \d L_t^0+K\,\sign\, \delta^N_i(t)\ \d\delta^N_i(t)\\
					&\quad =-c\ (\sign\, \Delta^N_i(t))\ \Delta^N_i(t)\,\d t
					+c\, (\sign\, \Delta^N_i(t))\ \left[\bar{\Theta}^{[N]}(t) - \bar{\Theta}^{[N]}\right]\,\d t\\
					&\qquad+c\,(\sign\, \Delta^N_i(t))\ \left[{\Theta}^{[N]}_x(t)-\bar{\Theta}^{[N]}(t)\right]\, \d t\\
					&\qquad + (\sign\, \Delta^N_i(t))\ \left(\sqrt{g(x^{[N]}_i(t))}\,
					- \sqrt{g(x^{\mu_N}_i(t))}\, \right)\d w_i (t) \\
					&\qquad +(\sign\, \Delta^N_i(t)) \,K e\, [\delta^N_i(t)-\Delta^N_i(t)]\,\d t\\			
					&\qquad +(\sign\,\delta^N_i(t))\,Ke\,[\Delta^N_i(t)-\delta^N_i(t)]\,\d t,
				\end{aligned}
			\end{equation}
			where we use that the local time \gls{localtime} is zero, since $g$ is Lipschitz (see \cite[Proposition V.39.3]{RoWi00}).
			
			Taking expectations in \eqref{mo23}, we find
			\begin{equation}
				\label{mo23a}
				\begin{aligned}
					\frac{\d}{\d t}\E[ |\Delta^N_i(t)|+ K|\delta^N_i(t)|] 
					&= - c\,\E\left[|\Delta^N_i(t)|\right]\\
					&\quad+c\ \E\left[(\sign\, \Delta^N_i(t))\ \big[\bar{\Theta}^{[N]}(t) - \bar{\Theta}^{[N]}\big]\right]\\
					&\quad+c\,\E\left[(\sign\, \Delta^N_i(t))\ \big[{\Theta}^{[N]}_x(t)-\bar{\Theta}^{[N]}(t)\big]\right]\\
					&\quad+Ke\, \E\left[\big(\sign\, \Delta^N_i(t)-\sign\,\delta^N_i(t)\big) \big[\delta^N_i(t)-\Delta^N_i(t)\big]\right].			
				\end{aligned}
			\end{equation}
			Note that we can rewrite \eqref{mo23a} as
			\begin{equation}
				\label{mo23a2}
				\begin{aligned}
					\frac{\d}{\d t}\E[ |\Delta^N_i(t)|+ K|\delta^N_i(t)|] 
					&= -c\,\E\left[|\Delta^N_i(t)|\right]\\
					&\quad-2Ke\, \E\left[1_{\sign\, \Delta^N_i(t)\neq\sign\,\delta^N_i(t)} \,\, [|\delta^N_i(t)|+|\Delta^N_i(t)|]\right]\\
					&\quad+c\ \E\left[(\sign\, \Delta^N_i(t))\ \big[\bar{\Theta}^{[N]}(t) - \bar{\Theta}^{[N]}\big]\right]\\
					&\quad+c\,\E\left[(\sign\, \Delta^N_i(t))\ \big[{\Theta}^{[N]}_x(t)-\bar{\Theta}^{[N]}(t)\big]\right].
				\end{aligned}
			\end{equation}
			It therefore follows that
			\begin{equation}
				\label{mo3a2}
				\begin{aligned}
					\E[ |\Delta^N_i(t)|+ K|\delta^N_i(t)|]
					&=\E[ |\Delta^N_i(0)|+ K|\delta^N_i(0)|]\\
					&\quad-c\int_0^t \d r\,\E\left[|\Delta^N_i(r)|\right]\\ 
					&\quad-2Ke \int_0^t \d r\,\E\left[1_{\sign\, \Delta^N_i(r)\neq\sign\,\delta^N_i(r)} \,\, 
					[|\delta^N_i(r)|+|\Delta^N_i(r)|]\right]\\
					&\quad+\int_0^t \d r\,c\, \E\left[(\sign\, \Delta^N_i(r))\ \big[\bar{\Theta}^{[N]}(r) - \bar{\Theta}^{[N]}\big]\right]\\
					&\quad+\int_0^t \d r\,c\,\E\left[(\sign\, \Delta^N_i(r))\ \big[{\Theta}^{[N]}_x(r)-\bar{\Theta}^{[N]}(r)\big]\right].
				\end{aligned}
			\end{equation}
			Note that, by the choice of initial distribution for the coupling, we have
			\begin{equation}
				\E[ |\Delta^N_i(0)|+ K|\delta^N_i(0)|]=0.
			\end{equation}
			Therefore we get
			\begin{equation}
				\label{mo132}
				\begin{aligned}
					0&\leq \E[ |\Delta^N_i(t)|+ K|\delta^N_i(t)|]\\
					&\leq -c\int_0^t\d r\,\E\left[|\Delta^N_i(r)|\right]\\ 
					&\quad-2Ke \int_0^t \d r\,\E\left[1_{\sign\, \Delta^N_i(r)\neq\sign\,\delta^N_i(r)} \,\, [|\delta^N_i(r)|+|\Delta^N_i(r)|]\right]\\
					&\quad+\int_0^t \d r\,c\, \E\left[ \big|\bar{\Theta}^{[N]}(r) - \bar{\Theta}^{[N]}\big|\right]\\
					&\quad+\int_0^t \d r\, c\,\E\left[ \big|{\Theta}^{[N]}_x(r)-\bar{\Theta}^{[N]}(r)\big|\right]\\
					&\leq t \left(\sup_{0\leq r \leq t}c\ \E\left[ \big|\bar{\Theta}^{[N]}(r) - \bar{\Theta}^{[N]}\big|\right]
					+c\,\E\left[ \big|{\Theta}^{[N]}_x(r)-\bar{\Theta}^{[N]}(r)\big|\right]\right).
				\end{aligned}
			\end{equation}
			Hence, by the assumption in \eqref{m08} and Lemma~\ref{lemav} (recall \eqref{bl}) , we see that, for all $t>0$, 
			\begin{equation}
				\label{35}
				\lim_{N\to\infty} \E[ |\Delta^N_i(t)|+ K|\delta^N_i(t)|]=0.
			\end{equation}
			Therefore, for every Lipschitz function $f\in\CC\bigl(([0,1]),\R\bigr)$ of $x_i(t)$,
			\begin{equation}
				\lim_{n\to\infty}\left|\E[f(x_i^{[N]}(t))-f(x_i^{\mu_N}(t))]\right|
				\leq \lim_{n\to\infty} \text{Lip} f\, \E[||\Delta^N_i(L(N))||] = 0,
			\end{equation}
			and the same holds for Lipschitz functions of $y_i$. Using that the Lipschitz functions are dense in $\CC\bigl(([0,1]),\R\bigr)$, we obtain that the result actually holds for all $f\in\CC\bigl(([0,1]^2)^{\N_0},\R\bigr)$ depending on finitely many components. This in turn implies that the result holds for all $f\in\CC\bigl(([0,1]^2)^{\N_0},\R\bigr)$.
		\end{proof}

		\paragraph{$\bullet$ Proof of Lemma \ref{stabest}}
		
		\begin{proof}
			Define
			\begin{equation}
				D^N(Z)=\frac{1}{N}\sum_{j\in[N]}\frac{x_j+Ky_j}{1+K}, \qquad D(Z) = \lim_{N\to\infty}D^N(Z) \text{ in } L_2(\mu).
			\end{equation}
			Since $\mu$ is translation invariant with $\int_{[0,1]^2} \frac{x_0+Ky_0}{1+K}\,\d\mu<1$, the $L_2(\mu)$-limit $D(Z)$ exists by the ergodic theorem. Since, by assumption, $\mu_N\to\mu$ as $N\to\infty$ for all fixed $M\in\N_0$, we have
			\begin{equation}
				\label{799}
				\lim_{N\to\infty} \CL_{\mu_N}[D^M(Z)]=\CL_\mu[D^M(Z)].
			\end{equation}
			Therefore, in order to prove Lemma \ref{stabest}, we are left to show
			\begin{equation}
				\label{m4}
				\lim_{M\to\infty} \sup_{N\geq M} \|D^M(Z)-D^N(Z)\|_{L_2(\mu_N)} = 0.
			\end{equation}
			This can be done by using Fourier transforms and spectral densities, and to do so we follow the same strategy as in \cite[Lemma 2.5]{DG93a}. 
			
			Define 
			\be
			\bar\theta^N =\E^{\mu_N}\left[\frac{x_0+Ky_0}{1+K}\right].
			\ee
			Since $\mu_N$ is translation invariant on $\N_0$, by Herglotz's theorem there exists a unique measure $\lambda_N$ such that, for all $j,k\in\N_0$,
			\begin{equation}
				\label{h}
				\E^{\mu_N}\left[\left(\frac{x_j+Ky_j}{1+K}-\bar\theta^N\right)\left(\frac{x_k+Ky_k}{1+K}-\bar\theta^N\right)\right]
				=\int_{(-\pi,\pi]} \lambda_N(\d u)\,\e^{\mathrm{i} (j-k)u}.
			\end{equation}
			For $N\in\N_0$, define 
			\begin{equation}
				D^N(u) =\frac{1}{N} \sum_{j\in[N]} \e^{\mathrm{i} ju}.
			\end{equation}
			By \eqref{h}, it follows that
			\begin{equation}
				\left\|D^M(Z)-D^N(Z)\right\|_{L_2(\mu_N)} = \left\|D^M(u)-D^N(u)\right\|_{L_2(\lambda_N)}.
			\end{equation}
			Polynomials of the type $D^N(u)$ are called trigonometric polynomials and satisfy:
			\begin{enumerate}
				\item 
				$\lim_{N\to\infty}D^N(u)=1_{\{0\}}(u)$.
				\item 
				For $\delta>0$ and $M<\infty$ there exists an $\epsilon(M,\delta)$ such that, for all $N\geq M$,
				\be
				|D^N(u)-1_{\{0\}}(u)| \leq 1_{(-\delta,\delta)\backslash\{0\}}+\epsilon(M,\delta) \text{ with } 
				\epsilon(M,\delta)\to 0 \text{ as } M\to\infty.
				\ee 
			\end{enumerate}
			Hence it follows that 
			\begin{equation}
				\left\|D^M(u)-D^N(u)\right\|^2_{L_2(\lambda_N)}
				\leq 2\lambda_N((-\delta,\delta)\backslash\{0\})+2\epsilon(M,\delta).
			\end{equation}
			Now let $M\to\infty$, to obtain
			\begin{equation}
				\sup_{N\geq M}\|D^M(u)-D^N(u)\|_{L_2(\lambda_N)}\leq 2\lambda_N((-\delta,\delta)\backslash\{0\}).
			\end{equation}
			Subsequently let $\delta\to 0$, so that $(-\delta,\delta)\backslash\{0\}\to\emptyset$ and
			\begin{equation}
				\lim_{M\to\infty} \sup_{N\geq M}\|D^M(u)-D^N(u)\|^2_{L_2(\lambda_N)}=0.
			\end{equation}
		\end{proof}
		
		
		\paragraph{$\bullet$ Proof of Lemma~\ref{unifergod}}
		\label{s81}
		
		\begin{proof}
			We first prove Lemma~\ref{unifergod}[1]. Afterwards we construct $(\bar{L}(N))_{N\in \N}$ to prove Lemma~\ref{unifergod}[2].
			
			Since $\lim_{k\to\infty}\mu_{N_k}=\mu$, Lemma~\ref{stabest} implies that $\lim_{k\to\infty}\CL[\bar{\Theta}^{[N_k]}]=\CL[\bar{\Theta}]$. For ease of notation we drop the subsequence notation in the remainder of this proof. By Skohorod's theorem we can construct the random variables $(z_i^{\mu_N})_{N\in\N}$ and $z_i^\mu$ on one probability space such that $\lim_{N\to\infty}z_i^{\mu_N}=z_i$ a.s. Then, as in the proof of Lemma~\ref{stabest}, we obtain
			\begin{equation}
				\label{t0}
				\lim_{N\to\infty}\E[|\bar{\Theta}^{[N]}-\bar{\Theta}|]=0.
			\end{equation} 
			To prove the claim we couple the two infinite systems, namely,
			\begin{equation}
				\begin{aligned}
					\label{ainf}
					&\d x^{\mu_N}_i(t) = c\left[\bar{\Theta}^{[N]} -x^{\mu_N}_i(t)\right]\d t + \sqrt{g(x^{\mu_N}_i(t))}\, \d w_i (t) 
					+ K e\, [y^{\mu_N}_i(t)-x^{\mu_N}_i(t)]\,\d t,\\
					&\d y^{\mu_N}_i(t) = e\,[x^{\mu_N}_i(t)-y^{\mu_N}_i(t)]\, \d t, \qquad i \in \N_0,
				\end{aligned}
			\end{equation}
			and
			\begin{equation}
				\begin{aligned}
					\label{binf}
					&\d x^{\mu}_i(t) = c\left[\bar{\Theta} -x^{\mu}_i(t)\right]\d t + \sqrt{g(x^{\mu}_i(t))}\, \d w_i (t) 
					+ K e\, [y^{\mu}_i(t)-x^{\mu}_i(t)]\,\d t,\\
					&\d y^{\mu}_i(t) = e\,[x^{\mu}_i(t)-y^{\mu}_i(t)]\, \d t, \qquad i \in \N_0,
				\end{aligned}
			\end{equation}
			are coupled by using the same Brownian motions in \eqref{ainf} and \eqref{binf}. Like before, define $\Delta^{\mu_N}_i=x_i^{\mu_N}-x_i^{\mu}$ and $\delta^{\mu_N}_i=y_i^{\mu_N}-y_i^{\mu}$. By the construction with Skohorod's theorem, we have that 
			\begin{equation}\label{o2}
				\lim_{N\to\infty}\E\big[ |\Delta^{\mu_N}_i(0)|+ K|\delta^{\mu_N}_i(0)|\big]=0.
			\end{equation} 
			To prove that, for all $t\geq 0$,
			\begin{equation}
				\label{423}
				\lim_{N\to\infty}\E\big[ |\Delta^{\mu_N}_i(t)|+ K|\delta^{\mu_N}_i(t)|\big]=0,
			\end{equation}
			we proceed as in the proof of Lemma~\ref{l.comp}. By It\^o-calculus, we find that
			\begin{equation}
				\label{o1}
				\begin{aligned}
					\E\big[ |\Delta^{\mu_N}_i(t)|+ K|\delta^{\mu_N}_i(t)|\big]
					&=\E\big[ |\Delta^{\mu_N}_i(0)|+ K|\delta^{\mu_N}_i(0)|\big]\\
					&\quad-c\int_0^t \d r\,\E\left[|\Delta^{\mu_N}_i(r)|\right]\\ 
					&\quad-2Ke \int_0^t \d r\,\E\left[1_{\sign\, \Delta^{\mu_N}_i(r)\neq\sign\,\delta^{\mu_N}_i(r)} \,\, 
					[|\delta^{\mu_N}_i(r)|+|\Delta^{\mu_N}_i(r)|]\right]\\
					&\quad+\int_0^t \d r\,c\ \E\left[ \big|\bar{\Theta}^{[N]} - \bar{\Theta}\big|\right]\\
					&\leq\E\big[ |\Delta^{\mu_N}_i(0)|+ K|\delta^{\mu_N}_i(0)|\big]
					+ tc\,\E\big[ \big|\bar{\Theta}^{[N]} - \bar{\Theta}\big|\big]. 
				\end{aligned}
			\end{equation}
			From \eqref{o1} it the follows that, for every $t\geq0$,
			\begin{equation}
				\label{23}
				\lim_{N\to\infty} \E\big[|\Delta^{\mu_N}_i(t)|+ K|\delta^{\mu_N}_i(t)|\big]=0.
			\end{equation}
			
			We next construct the sequence $(\bar{L}(N))_{N\in\N}$. From \eqref{35} and \eqref{23},we have
			\begin{equation}
				\label{955}
				\lim_{N\to\infty} \E\big[ |\Delta^N_i(t)|+ K|\delta^N_i(t)|\big]+\E\big[ |\Delta^{\mu_N}_i(t)|+ K|\delta^{\mu_N}_i(t)|\big]=0.
			\end{equation}
			Let $(t_k)_{k\in\N}$ be an increasing sequence such that $\lim_{k \to \infty} t_k=\infty$ and $\lim_{k\to\infty}
			t_k/k = 0$. By \eqref{955}, we can for each $k$ find an $N_k\in\N$ such that, for all $N\geq N_k$, 
			\begin{equation}
				\E\big[ |\Delta^N_i(t_k)|+ K|\delta^N_i(t_k)|\big]
				+\E\big[ |\Delta^{\mu_N}_i(t_k)|+ K|\delta^{\mu_N}_i(t_k)|\big]<\frac{1}{k}.
			\end{equation}
			Requiring that $N_{k+1}>N_k$, we obtain a strictly increasing sequence $(N_k)_{k\in\N}$ that partitions $\N$. Set
			\begin{equation}
				\label{957}
				\gls{barln}=\sum_{k\in\N}t_k1_{\{N_k,\cdots, N_{k+1}-1\}}(N).
			\end{equation}
			
			We show that $\bar{L}(N)$ satisfies the required properties:
			\begin{itemize}
				\item 
				$\lim_{N\to\infty}\E[ |\Delta^{N}_i(\bar{L}(N))|+ K|\delta^{N}_i(\bar{L}(N))|]+\E[ |\Delta^{\mu_N}_i(\bar{L}(N))|+ K|\delta^{\mu_N}_i(\bar{L}(N))|]=0$. To proof this, we fix $\epsilon>0$ and let $K$ be such that $\frac{1}{K}<\epsilon$. Then, for all $N\geq N_K$,
				\begin{equation}
					\begin{aligned}
						&\E\big[ |\Delta^{N}_i(\bar{L}(N))|+ K|\delta^{N}_i(\bar{L}(N))|\big]
						+\E\big[ |\Delta^{\mu_N}_i(\bar{L}(N))|+ K|\delta^{\mu_N}_i(\bar{L}(N))|\big]\\
						&\quad =\sum_{k\in\N}\E\big[ |\Delta^N_i(t_k)|+ K|\delta^N_i(t_k)|\big]
						+\E\big[ |\Delta^{\mu_N}_i(t_k)| + K|\delta^{\mu_N}_i(t_k)|\big]\,1_{\{N_k,\cdots,N_{k-1}\}}(N)\\
						&\quad <\frac{1}{K}<\epsilon.
					\end{aligned}
				\end{equation}
				We conclude that
				\begin{equation}
					\label{959}
					\lim_{N\to\infty}\E\big[ |\Delta^{N}_i(\bar{L}(N))|+ K|\delta^{N}_i(\bar{L}(N))|\big]
					+\E\big[ |\Delta^{\mu_N}_i(\bar{L}(N))| v+ K|\delta^{\mu_N}_i(\bar{L}(N))|\big]=0.
				\end{equation}
				\item 
				$\lim_{N\to\infty} \bar{L}(N)=\infty$. By \eqref{957}, for each $k\in\N$ there exists an $N_k\in\N$ such that, for all $N\geq N_k$, $\bar{L}(N) \geq t_k$ and $t_k\to\infty$. We conclude that $\lim_{N\to\infty} \bar{L}(N)=\infty$.
				\item 
				$\lim_{N\to\infty} \bar{L}(N)/N=0$. Recall that $\lim_{k\to\infty} t_k/k=0$ and $N_k\geq k$ by construction. Hence $\lim_{N\to\infty} \bar{L}(N)/N \leq \lim_{N\to\infty} \sum_{k\in\N} (t_k/k)]\,1_{\{N_k,\cdots,N_{k-1}\}}(N)=0$.
			\end{itemize}
			Choosing $(t_k)_{k\in\N}=(L(N))_{N\in\N}$, we complete the proof of Lemma~\ref{unifergod}.
		\end{proof}
		
		
		\paragraph{$\bullet$ Proof  of Lemma~\ref{lemlip}}
		
		\begin{proof}
			The goal is to prove that $\nu_\theta$ is continuous in $\theta$. To do so, let $(\theta_n)_{n\in\N}$ be a sequence in $[0,1]$ (note that $\theta_n$ is not a random variable) such that $\lim_{n\to\infty}\theta_n=\theta$. Couple the two infinite systems
			\begin{equation}
				\begin{aligned}
					\label{dainf}
					&\d x^{n}_i(t) = c\left[\theta_n -x^{n}_i(t)\right]\d t + \sqrt{g(x^{n}_i(t))}\, \d w_i (t) 
					+ K e\, [y^{n}_i(t)-x^{n}_i(t)]\,\d t,\\
					&\d y^{n}_i(t) = e\,[x^{n}_i(t)-y^{n}_i(t)]\, \d t, \qquad i \in \N_0,
				\end{aligned}
			\end{equation}
			and
			\begin{equation}
				\begin{aligned}
					\label{dbinf}
					&\d x^{}_i(t) = c\left[\theta -x^{}_i(t)\right]\d t + \sqrt{g(x^{}_i(t))}\, \d w_i (t) 
					+ K e\, [y^{}_i(t)-x^{}_i(t)]\,\d t,\\
					&\d y^{}_i(t) = e\,[x^{}_i(t)-y^{}_i(t)]\, \d t, \qquad i \in \N_0,
				\end{aligned}
			\end{equation}
			via their Brownian motions, like in the proof of Lemma~\ref{unifergod}. Let $\CL[(x_i^n(0),y_i^n(0))]=\delta_{(\theta_n,\theta_n)}$ and $\CL[(x_i(0),y_i(0))]=\delta_{(\theta,\theta)}$. As before, define  $\Delta^{n}_i=x_i^{n}-x_i^{}$ and $\delta^{n}_i=y_i^{n}-y_i^{}$. Note that
			\begin{equation}\label{o3}
				\lim_{n\to\infty}\E\big[ |\Delta^{n}_i(0)|+ K|\delta^{n}_i(0)|\big]=0.
			\end{equation} 
			By a similar argument as in the proof of Lemma~\ref{unifergod}, we obtain that, for all $t\geq0$,
			\begin{equation}
				\label{o33}
				\lim_{n\to\infty}\E\big[ |\Delta^{n}_i(t)|+ K|\delta^{n}_i(t)|\big]=0.
			\end{equation} 
			Hence we can construct a sequence $(L(n))_{n\in \N}$ satisfying $\lim_{n\to\infty} L(n) = \infty$ and $\lim_{n\to\infty} L(n)/n = 0$ such that  
			\begin{equation}
				\label{312}
				\lim_{n\to\infty}\E\big[ |\Delta^{n}_i(L(n))|+ K|\delta^{n}_i(L(n))|\big]=0.
			\end{equation}
			To prove the continuity of the equilibrium $\nu_\theta$ in $\theta$, we reason as follows. First note that we can couple the system in \eqref{dainf} starting from $\delta_{(\theta_n,\theta_n)}$ with the system in \eqref{dainf} starting from $\nu_{\theta_n}$. By the uniqueness and convergence to equilibrium (see Lemma~\ref{lemerg}), we see that this coupling is successful. Similarly, we can couple the system in \eqref{dbinf} starting from $\delta_{(\theta, \theta)}$ with the system in \eqref{dbinf} starting from $\nu_{\theta_n}$, and see the coupling succesful. Finally, we use \eqref{312} to obtain
			\begin{equation}
				\begin{aligned}
					& \lim_{n\to\infty} \E^{\nu_{\theta_n}\times\nu_\theta}\big[ |\Delta^{n}_i(L(n))|+ K|\delta^{n}_i(L(n))|\big]\\
					&\leq \lim_{n\to\infty} \E^{\nu_{\theta_n}\times\delta_{(\theta_n,\theta_n)}}\big[ |\Delta^{n}_i(L(n))|
					+ K|\delta^{n}_i(L(n))|\big]\\
					&\quad +\lim_{n\to\infty}\E^{\delta_{(\theta,\theta)}\times\delta_{(\theta_n,\theta_n)}}\big[ |\Delta^{n}_i(L(n))|+
					K|\delta^{n}_i(L(n))|\big]\\
					&\quad +\lim_{n\to\infty}\E^{\nu_\theta\times\delta_{(\theta,\theta)}}\big[ |\Delta^{n}_i(L(n))|+ K|\delta^{n}_i(L(n))|\big]
					=0.
				\end{aligned}
			\end{equation} 
			Let $f$ be a Lipschitz function. Then, by the equilibrium property of $\nu_{\theta_n}$ and $\nu_\theta$,
			\begin{equation}
				\begin{aligned}
					&\lim_{n\to\infty}\left|\E^{\nu_{\theta_n}}\big[f(x^n(0))\big]-\E^{\nu_\theta}\big[f(x(0))\big]\right|\\
					&\quad =\lim_{n\to\infty}\left|\E^{\nu_{\theta_n}}\big[f(x^n(L(n)))\big]-\E^{\nu_\theta}\big[f(x(L(n)))\big]\right|\\
					&\quad =\lim_{n\to\infty}\left|\E^{\nu_{\theta_n}\times\nu_\theta}\big[f(x^n(L(n)))-f(x(L(n)))\big]\right|\\
					&\quad \leq \lim_{n\to\infty} (\mathrm{Lip} f)\, \E^{\nu_{\theta_n}\times\nu_\theta}\big[|(x^n(L(n)))-(x(L(n)))|\big]
					=0.
				\end{aligned}
			\end{equation}
			We can also show this if $f$ is a Lipschitz function of the $y$ component. Hence $\nu_\theta$ is indeed continuous as a function of $\theta$.
		\end{proof}
		
		
		\paragraph{$\bullet$ Proof of Lemma~\ref{lem:12}}
		
		\begin{proof}
			Note that for all $N\in\N$ fixed, by It\^o-calculus we find from \eqref{gh45a} that
			\begin{equation}
				\begin{aligned}
					& \frac{\d}{\d t} \E\left[|x^{[N],1}_{i}(t)-x^{[N],2}_{i}(t)|+K|y^{[N],1}_{i}(t)-y^{[N],2}_{i}(t)|\right]\\
					&=-\frac{2c}{N}\sum_{j \in [N]}\E\left[|x^{[N]}_{j,1}(t)-x^{[N]}_{j,2}(t)|\,1_{\big\{\sign\big(x^{[N]}_{j,1}(t)-x^{[N]}_{j,2}(t)\big)
						\neq \sign\big(x^{[N],1}_{i}(t)-x^{[N],2}_{i}(t)\big)\big\}}\right]\\
					&\qquad-2Ke\,\E\Big[|x^{[N],1}_{i}(t)-x^{[N],2}_{i}(t)|\\
					&\qquad\qquad +|y^{[N],1}_{i}(t)-y^{[N],2}_{i}(t)|\,
					1_{\big\{\sign\big(x^{[N],1}_{i}(t)-x^{[N],2}_{i}(t)\big) \neq \sign\big(y^{[N],1}_{i}(t)-y^{[N],2}_{i}(t)\big)\big\}}\Big].
				\end{aligned}
			\end{equation}
			Therefore, for each $N\in\N$,
			\begin{equation}
				\label{mono}
				t \mapsto \E\left[|x^{[N],1}_{i}(t)-x^{[N],2}_{i}(t)|+K|y^{[N],1}_{i}(t)-y^{[N],2}_{i}(t)|\right]
				\text{ is decreasing}.
			\end{equation}
			
			Fix any $t(N)\to\infty$. By the assumption in Lemma~\ref{lem:12}, the proofs of Lemma~\ref{l.comp} and Lemma~\ref{unifergod} imply that \eqref{955} holds for both $(X^{[N],1},Y^{[N],1})$ and $(X^{[N],2},Y^{[N],2}) $. Using the construction in the proof of Lemma~\ref{unifergod}, we can construct \emph{one} sequence $(l(N))_{N\in\N}$, satisfying $l(N)\leq t(N)$, $\lim_{N\to\infty }l(N)=\infty$ and $\lim_{N\to\infty} l(N)/N=0$, such that \eqref{959} with $\bar{L}(N)$ replaced by $l(N)$ holds for both the systems arising from $(X^{[N],1},Y^{[N],1}) $ and $(X^{[N],2},Y^{[N],2})$. 
			
			Write
			\begin{equation}
				\label{p1}
				\begin{aligned}
					&\E\left[|x^{[N],1}_{i}(l(N))-x^{[N],2}_{i}(l(N))|+K|y^{[N],1}_{i}(l(N))-y^{[N],2}_{i}(l(N))|\right]\\
					\leq&\E\left[|x^{[N],1}_{i}(l(N))-x^{\mu,1}_{i}(l(N))|+K|y^{[N],1}_{i}(l(N))-y^{\mu,1}_{i}(l(N))|\right]\\
					&+\E\left[|x^{\mu,1}_{i}(l(N))-x^{\mu,2}_{i}(l(N))|+K|y^{\mu,1}_{i}(l(N))-y^{\mu,2}_{i}(l(N))|\right]\\
					&+\E\left[|x^{\mu,2}_{i}(l(N))-x^{[N],2}_{i}(l(N))|+K|y^{\mu,2}_{i}(l(N))-y^{[N],2}_{i}(l(N))|\right].
				\end{aligned}
			\end{equation} 
			Note that in the right-hand side of the inequality the first and the third term tend to zero by \eqref{959}. The second term tends to zero because $\mu\{\bar{\Theta}_1=\bar{\Theta}_2\}=1$, and hence Lemma~\ref{lemerg} can be applied. Therefore 
			\begin{equation}
				\lim_{N\to\infty}\E\left[|x^{[N],1}_{i}(l(N))-x^{[N],2}_{i}(l(N))|+K|y^{[N],1}_{i}(l(N))-y^{[N],2}_{i}(l(N))|\right]=0.
			\end{equation}
			Using the monotonicity in \eqref{mono}, we get
			\begin{equation}
				\begin{aligned}
					&\lim_{N\to\infty} \E\left[|x^{[N],1}_{i}(t(N))-x^{[N],2}_{i}(t(N))|+K|y^{[N],1}_{i}(t(N))-y^{[N],2}_{i}(t(N))|\right]\\
					&\leq \lim_{N\to\infty}\E\left[|x^{[N],1}_{i}(l(N))-x^{[N],2}_{i}(l(N))|+K|y^{[N],1}_{i}(l(N))-y^{[N],2}_{i}(l(N))|\right]
					=0.
				\end{aligned}
			\end{equation}
		\end{proof}
		Combining the proofs of Proposition~\ref{prop1}, Lemma~\ref{unifergod} and Lemma~\ref{lem:12}, we obtain the following corollary. This corollary turns out to be important in Section~\ref{step3} in the proof of Lemma~\ref{lemmart} to obtain the limiting evolution of the $1$-blocks.
		
		\begin{corollary}
			\label{cor1}
			Fix $s>0$. Let $\mu_{N}$ be the measure obtained by periodic configuration of 
			\begin{equation}
				\CL\big[X^{[N]}(Ns-L(N)),Y^{[N]}(Ns-L(N))\big],
			\end{equation} 
			and let $\mu$ be a weak limit point of the sequence $(\mu_N)_{N\in\N}$. Let
			\begin{equation}
				\label{1a1}
				\bar{\Theta} = \lim_{N\to\infty}\frac{1}{N}\sum_{i\in[N]}\frac{x_i^\mu+Ky_i^{\mu}}{1+K} \text{ in } L^2(\mu).
			\end{equation}
			and let $(X^{\nu_{\bar{\Theta}}},Y^{\nu_{\bar{\Theta}}})$ be the infinite system evolving according to \eqref{binfb} and starting from its equilibrium measure.
			Consider the finite system $(X^{[N]},Y^{[N]})$ as a system on $([0,1]\times[0,1])^{\N_0}$ obtained by periodic continuation.  Construct $(X^{[N]},Y^{[N]})$ and $(X^{\nu_{\bar{\Theta}}},Y^{\nu_{\bar{\Theta}}})$ on one probability space. Then, for all $t\geq0$, 
			\begin{equation}
				\label{937}
				\begin{aligned}
					\lim_{N \to \infty}\E\left[\big|x_i^{[N]}(Ns+t)-x_i^{\nu_{\bar{\Theta}}}(t)\big|\right]
					+K\,\E\left[\big|y_i^{[N]}(Ns+t)-y_i^{\nu_{\bar{\Theta}}}(t)\big|\right]=0,\qquad\forall\, i\in[N].
				\end{aligned}
			\end{equation}
		\end{corollary}
		
		\begin{proof}
			By Proposition~\ref{prop1} we have that $\lim_{k\to\infty}\CL[X^{[N_k]}(N_ks)+Y^{[N_k]}(N_ks)]=\nu(s)=\nu_{\bar{\Theta}}$. Let $\nu_{N_k}$ be defined by periodic continuation of the configuration of $(X^{[N_k]}(N_ks),Y^{[N_k]}(N_ks))$ and let $\nu=\lim_{k\to\infty}\nu_{N_k}$, so  $\nu=\nu_{\bar{\Theta}}$. Construct the process $(X^{[N]}(t),Y^{[N]}(t))_{t\geq0}$, $(X^{\nu_{N_k}}(t),Y^{\nu_{N_k}}(t))_{t\geq 0}$ and $(X^\nu(t),Y^\nu(t))_{t\geq 0}$ on one probability space and use for all processes the same Brownian motions. Then the couplings in the proofs of Lemma~\ref{l.comp} and Lemma~\ref{unifergod} imply \eqref{937}.
		\end{proof}
		
		\subsection{Proof of step 2. Convergence of the estimator}
		\label{step3}
		
		In this section we prove the three lemmas stated in Step 2 of Section~\ref{sec:finsysmf}. Afterwards we prove Proposition \ref{p.esti} with the help of these lemmas.
		
		\paragraph{$\bullet$ Proof of Lemma \ref{martav}}
		
		\begin{proof}
			Recall the definition of $\bar{\Theta}^{[N]}(t)$ in \eqref{slovar}. It follows from the SSDE in \eqref{gh45a} that
			\begin{equation}
				\label{3}
				\d \bar{\Theta}^{[N]}(t)= \frac{1}{1+K}\frac{1}{N}\sum_{i\in[N]}\sqrt{g\big(x_i^{[N]}(t)\big)}\,\d w_i(t).
			\end{equation} 
			Hence we see that $t \mapsto \bar{\Theta}^{[N]}(t)$ is a continuous martingale. By It\^o's formula we have
			\begin{equation}
				\label{1}
				\begin{aligned}
					\E\big[\big(\bar{\Theta}^{[N]}(t)\big)^2\big]
					&=\E\big[\big(\bar{\Theta}^{[N]}(0)\big)^2\big]
					+\frac{1}{(1+K)^2}\int_0^t \d r\,\frac{1}{N^2}\sum_{i\in[N]} g\big(x_i^{[N]}(r)\big)\\
					&\leq 1+\frac{1}{N}\frac{\|g\|}{(1+K)^2}\, t.
				\end{aligned}
			\end{equation}
			Since $g$ is a bounded function, we get that $t \mapsto \bar{\Theta}^{[N]}(t)$ is square integrable. It follows that, 
			\begin{equation}
				\left(\big(\bar{\Theta}^{[N]}(Ns+t)-\bar{\Theta}^{[N]}(Ns)\big)^2\right)_{t\geq 0}
			\end{equation}
			is a sub-martingale. Therefore, defining the stopping time 
			\begin{equation}
				S^N_\epsilon=\inf\left\{t\geq 0: \big(\bar{\Theta}^{[N]}(Ns+t)
				-\bar{\Theta}^{[N]}(Ns)\big)^2\geq\epsilon\right\}\wedge L(N),
			\end{equation}
			we find, by the continuity of $t \mapsto \bar{\Theta}^{[N]}(Ns+t)$ and the optional sampling theorem, that
			\begin{equation}
				\label{2}
				\P\left(S^N_\epsilon\in[Ns,Ns+L(N))\right)\leq \frac{1}{\epsilon^2}
				\E\left[\big(\bar{\Theta}^{[N]}(Ns+L(N))-\bar{\Theta}^{[N]}(Ns)\big)^2\right].
			\end{equation}
			Combining \eqref{1} and \eqref{2}, we find
			\begin{equation}
				\lim_{N\to\infty} \sup_{0 \leq t \leq L(N)}\big|\bar{\Theta}^{[N]}(Ns+t)-\bar{\Theta}^{[N]}(Ns)\big|
				= 0 \text{ in probability.}
			\end{equation}
			To obtain the increasing process, note that by It\^o-calculus it follows from \eqref{3} that
			\begin{equation}
				\label{4}
				\big\langle \bar{\Theta}^{[N]}(Ns)\big\rangle_{s\geq 0} 
				= \frac{1}{(1+K)^2}\int_{0}^{s} \d r\,\frac{1}{N} \sum_{i\in[N]}{g\big(x_i^{[N]}(Nr)\big)}.
			\end{equation}
			
			We are left to show that the sequence of processes $((\bar{\Theta}^{[N]}(Ns))_{s\geq 0})_{N\in\N}$ is tight. Note that \eqref{4} implies that, for all $N\in\N$ and $s\geq 0$,
			\begin{equation}
				\big\langle\bar{\Theta}^{[N]}(Ns)\big\rangle\leq \frac{\|g\|}{(1+K)^2}\,s 
			\end{equation} 
			and 
			\begin{equation}
				\frac{\d}{\d s}\big\langle \bar{\Theta}^{[N]}(Ns)\big\rangle\leq \frac{\|g\|}{(1+K)^2}.
			\end{equation}
			Hence the sequence $(\langle\bar{\Theta}^{[N]}(Ns)\rangle_{s\geq 0})_{N\in\N}$ is equicontinuous. Therefore, by the Arzela-Ascoli theorem (see e.g.\ \cite[Theorem 7.2]{Bi99}), for each $T\geq0$ the set $(\langle\bar{\Theta}^{[N]}(Ns)\rangle_{0\leq s\leq T})_{N\in\N}$ is relatively compact in $C([0,T],\R)$, the space of continuous functions from $[0,T]\to\R$. Therefore the set of laws $(\CL\langle\bar{\Theta}^{[N]}(Ns)\rangle_{0\leq s\leq T})_{N\in\N}$ is tight in $\CP(C([0,T],\R))$. Hence it follows that $(\CL\langle\bar{\Theta}^{[N]}(Ns)\rangle_{s\geq 0})_{N\in\N}$ is tight in $\CP(C([0,\infty),\R))$, the set of probability laws on $C([0,\infty),\R)$. 
			
			Since $(\bar{\Theta}^{[N]}(Ns)-\bar{\Theta}^{[N]}(0))_{s\geq 0}$ is a stochastic integral, we can represent it as a time-transformed Brownian motion (see e.g. \cite[Chapter 5]{RY99}):
			\begin{equation}
				\label{m89}
				\big(\bar{\Theta}^{[N]}(Ns)-\bar{\Theta}^{[N]}(0)\big)_{s\geq 0}
				= w\big(\langle\bar{\Theta}^{[N]}(Ns)\rangle\big)_{s\geq 0}.
			\end{equation} 
			Let $\chi$ be a standard normal random variable. Then
			\begin{equation}
				\begin{aligned}
					\E\left[\left(w\big(\langle\bar{\Theta}^{[N]}(Ns)\rangle\big)
					-w\big(\langle\bar{\Theta}^{[N]}(Nr)\rangle\big)\right)^2\right]
					&\leq\E\left[\left(\langle\bar{\Theta}^{[N]}(Ns)\rangle
					-\langle\bar{\Theta}^{[N]}(Nr)\rangle\right)^2\right]\E\left[\chi^4\right]\\
					&\leq \left(s-r\right)^2\frac{\|g\|^2}{(1+K)^4}\,\E\left[\chi^4\right].
				\end{aligned}
			\end{equation}
			Hence it follows from Kolmogorov's criterion for weak compactness (see e.g.\ \cite[Chapter XIII, Theorem 1.8]{RY99}) that the sequence $(\CL[(\bar{\Theta}^{[N]}(Ns))_{s\geq 0}])_{N\in\N}$ is tight in $\CP(C([0,\infty),\R))$.
		\end{proof}
		
		
		\paragraph{$\bullet$ Proof of Lemma~\ref{lemmart}}
		
		\begin{proof}
			For ease of notation we will suppress the subsequence notation and assume that
			\begin{equation}
				\lim_{N\to\infty} \CL\left[\big(\bar{\Theta}^{[N]}(Ns)\big)_{s > 0}\right] 
				= \CL\bigl[(\bar{\Theta}(s))_{s > 0}\bigr].
			\end{equation}
			
			The processes $(\bar{\Theta}^{[N]}(Ns))_{s \geq 0}$ are martingales, see \eqref{3}, measurable w.r.t.\ the canonical filtration $(\CF_s)_{s \geq 0}$ and so are there weak limit points. Therefore also the weak limit point $(\bar{\Theta}(s))_{s>0}$ is a martingale, (see \cite[Section 3]{DG93a}). To obtain \eqref{a1}, we use the following strategy. Recall from the proof of Lemma~\ref{martav} that the sequence $\{\langle\bar{\Theta}^{[N]}(Ns)\rangle_{s>0}\}_{N\in\N}$ is tight. Hence, in order to prove that 
			\begin{equation}
				\label{m90}
				\lim_{N \to \infty}\CL\left[\langle\bar{\Theta}^{[N]}(Ns)\rangle_{s>0}\right]
				=\CL\left[\left(\int_0^s \d r \frac{1}{(1+K)^2}\,\E^{\nu_{\bar{\Theta}(r)}}[g(x_0)]\right)_{s>0}\right],
			\end{equation}
			it is enough to prove that the finite-dimensional distributions of $(\langle\bar{\Theta}^{[N]}(Ns)\rangle_{s>0})_{N\in\N}$ converge to the finite-dimensional distribution of $(\int_0^s \d r\,\frac{1}{(1+K)^2}\E^{\nu_{\bar{\Theta}(r)}}[g(x_0)])_{s>0}$. We will prove a slightly stronger result, namely,
			\begin{equation}
				\label{a2}
				\lim_{N\to \infty}\E\left[\left|\langle\bar{\Theta}^{[N]}(Ns)\rangle
				-\int_{0}^{s} \d r\,\frac{1}{(1+K)^2}\,\E^{\nu_{\bar{\Theta}(r)}}[g(x_0)]\right|\right] = 0.
			\end{equation}
			Once we have obtained \eqref{a2} and hence \eqref{m90}, by Skohorod's theorem we can construct the processes $(\langle\bar{\Theta}^{[N]}(Ns)\rangle_{s>0})_{N\in\N}$ on a single probability space, to obtain
			\begin{equation}
				\label{m91}
				\lim_{N \to \infty}\langle\bar{\Theta}^{[N]}(Ns)\rangle_{s\geq0}=\left(\int_0^s \d r\, 
				\frac{1}{(1+K)^2}\E^{\nu_{\bar{\Theta}(r)}}[g(x_0)]\right)_{s\geq 0} \quad a.s. 
			\end{equation}
			Using the continuity of Brownian motion, we get that (recall \eqref{m89})
			\begin{equation}
				\begin{aligned}
					\lim_{N \to \infty}\big(\bar{\Theta}^{[N]}(Ns)\big)_{s> 0}
					&= \lim_{N \to \infty} \left[w\big(\langle\bar{\Theta}^{[N]}(Ns)\rangle\big)_{s> 0}+\bar{\Theta}^{[N]}(0)\right]\\
					&=w\left(\int_0^s \d r\, \frac{1}{(1+K)^2}\,\E^{\nu_{\bar{\Theta}(r)}}[g(x_0)]\right)_{s>0}+\vartheta_0
					\quad a.s.
				\end{aligned}
			\end{equation}
			Therefore we can choose a version of $(\bar{\Theta}(s))_{s>0}$ such that
			\begin{equation}
				\lim_{N \to \infty}\big(\bar{\Theta}^{[N]}(Ns)\big)_{s>0}=\lim_{N \to \infty}\big(\bar{\Theta}(s)\big)_{s>0} \quad a.s.
			\end{equation}
			and
			\begin{equation}
				\lim_{N \to \infty} \big(\bar{\Theta}^{[N]}(Ns),\langle\bar{\Theta}^{[N]}(Ns)\rangle\big)_{s>0}
				=\left(\bar{\Theta}^{}(s),\langle\bar{\Theta}^{}(s)\rangle\right)_{s>0} \quad a.s.
			\end{equation}
			By the continuous mapping theorem, \eqref{a1} follows. The martingale property follows from the fact that $\big(\bar{\Theta}^{[N]}(Ns)^2-\langle\bar{\Theta}^{[N]}(Ns)\rangle\big)_{s>0}$ are martingales. Therefore, to finish the proof of Lemma~\ref{lemmart} we are left to prove \eqref{a2}.
			
			To prove \eqref{a2}, define the empirical measures on $[0,1]$ by 
			\begin{equation}
				U^{[N]}(Ns) = \frac{1}{N}\sum_{i\in[N]} \delta_{ x_i(Ns)}.
			\end{equation}
			Note that we can write
			\begin{equation}
				\label{a3}
				\begin{aligned}
					&\E\left[\left|\langle\bar{\Theta}^{[N]}(Ns)\rangle
					-\int_{0}^{s} \d r\,\frac{1}{(1+K)^2}\,\E^{\nu_{\bar{\Theta}(r)}}[g(x_0)]\right|\right]\\
					&=\frac{1}{(1+K)^2}\,\E\left[\left|\int_0^s \d r\,\,\E^{U^{[N]}(Nr)}[g(x_0)]
					-\int_{0}^{s} \d r\,\,\E^{\nu_{\bar{\Theta}(r)}}[g(x_0)]\right|\right]\\
					&\leq \frac{1}{(1+K)^2}\int_0^s \d r\,\E\left[\left|\E^{U^{[N]}(Nr)}[g(x_0)]- \E^{\nu_{\bar{\Theta}(r)}}[g(x_0)]\right|\right].
				\end{aligned}
			\end{equation}
			Hence, to prove \eqref{a2} it is enough to prove that, for all $r>0$,
			\begin{equation}
				\label{a4}
				\lim_{N\to \infty}\E\left[\left|\E^{U^{[N]}(Nr)}[g(x_0)]- \E^{\nu_{\bar{\Theta}(r)}}[g(x_0)]\right|\right]=0
			\end{equation}
			and apply the dominated convergence theorem. 
			
			To prove \eqref{a4}, we will use the coupling arguments from Section~\ref{step1}, as well as ergodicity and invariance under the evolution of $\nu_{\bar{\Theta}}$. As before, let $z^{[N]}(t)$ denote the $[N]$-component system $(x_i^{[N]}(t),y_i^{[N]}(t))_{t\geq 0}$ evolving according to \eqref{gh45a}, viewed as a system on $\N_0$ obtained by periodic continuation and with initial law $\CL[z^{[N]}(0)]=\CL[x_i^{[N]}(Nr-L(N)),y_i^{[N]}(Nr-L(N))]$. Let $(z^{\mu_N}(t))_{t \geq 0}$ denote the infinite system $(x_i^{\mu_N}(t),y_i^{\mu_N}(t))_{t\geq 0}$ evolving according to \eqref{gh45ainf} with initial law $\mu_N$ obtained by periodic continuation of the configuration of $(x_i^{[N]}(Nr-L(N)),y_i^{[N]}(Nr-L(N)))$, and let $\mu$ be a weak limit point of the sequence $\mu_N$. Note that, for all $r>0$, by Lemma~\ref{stabest} we have that $\lim_{N \to \infty}\bar{\Theta}^{[N]}(Nr)=\bar{\Theta}(r)$ for $\bar{\Theta}(r)=\lim_{n\to \infty}\frac{1}{n}\sum_{i\in[n]}\frac{x_i+Ky_i}{1+K}\text{ in }L_2(\mu)$. Let $\bar{L}(N)$ be the sequence constructed in Corollary~\ref{cor1}. Then we can write    
			\begin{equation}
				\label{a5}
				\begin{aligned}
					&\E\left[\left|\E^{U^{[N]}(Nr)}[g(x_0)]- \E^{\nu_{\bar{\Theta}}(r)}[g(x_0)]\right|\right]\\
					&\leq\E\left[\left|\frac{1}{N}\sum_{i\in[N]}g(x_i^{[N]}(Nr))
					- \frac{1}{N}\sum_{i\in[N]} g\big(x_i^{\nu_{\bar{\Theta}}(r)}(\bar{L}(N))\big)\right|\right]\\
					&\qquad+\E\left[\left|\frac{1}{N}\sum_{i\in[N]}g\big(x_i^{\nu_{\bar{\Theta}(r)}}(\bar{L}(N))\big)
					- \E^{\nu_{\bar{\Theta}(r)}}[g(x_0)]   \right|\right]\\
					&\leq (\text{Lip}\,g)\,\E\left[\left|x_0^{[N]}(Nr) - x_0^{\nu_{\bar{\Theta}}(r)}(\bar{L}(N))\right|\right]\\
					&\qquad+\E\left[\left|\frac{1}{N}\sum_{i\in[N]}g\big(x_i^{\nu_{\bar{\Theta}(r)}}(\bar{L}(N))\big)
					- \E^{\nu_{\bar{\Theta}(r)}}[g(x_0)]\right|\right],
				\end{aligned}
			\end{equation}
			where in the second inequality we use the Lipschitz property of $g$ and the translation invariance of the system. Note that the first term tends to $0$ as $N\to\infty$ by Corollary~\ref{cor1}. Finally, note that by Lemma~\ref{lemerg} $(x_i)_{i\in\N_0}$ is a sequence of bounded i.i.d. random variables under $\nu_{\bar{\Theta}(r)}$. Hence the last term tends to zero by the law of large numbers.
		\end{proof}
		
		
		\paragraph{$\bullet$ Proof of Lemma~\ref{lem:uni}}
		
		\begin{proof}
			Note that, since $g$ is Lipschitz, the function
			\begin{equation}
				\theta\mapsto\E^{\nu_\theta}[g],
			\end{equation}
			is Lipschitz by Lemma~\ref{lemlip}. Hence, by \cite[Theorem 1]{YW71}, the SDE
			\begin{equation}
				\label{m28}
				\d \Phi(s)=\frac{1}{(1+K)}\sqrt{\E^{\nu_\Phi(s)}[g]}\,\d w(s)
			\end{equation}
			has a pathwise unique solution (see \c ite[p315--317]{KS91}) and a unique solution in law (see \cite[Proposition 1]{YW71}). This implies that the martingale problem with generator
			\begin{equation}
				\frac{1}{(1+K)^2}\,\E^{\nu_\Phi}[g]\frac{\d}{\d \Phi^2}
			\end{equation}
			has a unique solution. In particular, choosing $f(\Phi)=\Phi^2$, we see that the martingale problem implies that
			\begin{equation}
				\left(\Phi^2(s)-\frac{1}{(1+K)^2}\int_{0}^{s} \d u\,\E^{\nu_{\Phi(u)}}[g(x_0)]\right)_{s > 0} 
			\end{equation}
			is a martingale. 
			
			Since $(\bar{\Theta}(s))_{s>0}$ is a continuous bounded martingale, it could be written as a time transformed Brownian motion. The uniqueness of the martingale problem in \eqref{martprob} now follows from the fact that the quadratic variation of a martingale is unique.
		\end{proof}   
		
		
		\paragraph{$\bullet$ Proof of Proposition \ref{p.esti}}
		
		\begin{proof}
			Combining Lemma~\ref{lemmart}--\ref{lem:uni}, all converging subsequences of $(\CL[(\bar{\Theta}^{[N]}(Ns))_{s\geq 0}])_{N\in\N}$ converge to the same limit, which is the unique process satisfying the martingale problem in \eqref{martprob}. 
		\end{proof}

		\subsection{Proof of step 3. Convergence of the $1$-blocks in the Meyer-Zheng topology}
		\label{step2}
		
		In this section we prove  Proposition~\ref{p.estim}  stated in Step 3 of Section~\ref{sec:finsysmf}. The Lemmas~\ref{lem91}, \ref{lem93} and \ref{lem94} are proven in Appendix~\ref{apb}.
		
		\paragraph{$\bullet$ Proof of Proposition~\ref{p.estim}} 
		
		\begin{proof}
			By Proposition~\ref{p.esti} we have that
			\begin{equation}
				\lim_{N\to\infty}\CL[\bar{\Theta}^{[N]}(Ns)_{s>0}]=\CL[(\bar{\Theta}(s))_{s>0}]
			\end{equation}
			in the normal topology and therefore (see \ref{apb} Lemma~\ref{lem90}) 
			\begin{equation}
				\lim_{N\to\infty}\CL[\bar{\Theta}^{[N]}(Ns)_{s>0}]=\CL[(\bar{\Theta}(s))_{s>0}]\, \text{ in Meyer-Zheng topology}.
			\end{equation}
			By Lemma~\ref{lemav}, for $s>0$,
			\begin{equation}
				\lim_{N\to\infty}\E\left[ \left|\bar{\Theta}^{[N]}(Ns)-x_1^{[N]}(s)\right|\right]=0 
			\end{equation}
			and therefore, by Lemma~\ref{lem91}, 
			\begin{equation}
				\lim_{N\to\infty} d_P(\psi_{\bar{\Theta}^{[N]}},\psi_{x_1^{[N]}})=0  \text{ in probability}.
			\end{equation}
			To apply the above results to our model, we recall the following basic result (see \cite[Chapter 1]{Bi99}), which also holds for the Meyer-Zheng topology.
			
			\begin{lemma}
				\label{lem92}
				Let $X_n$, $Y_n$ be random variables. If 
				\begin{equation}
					\lim_{n\to\infty} \CL[X_n]=\CL[X]
				\end{equation}
				and 
				\begin{equation}
					\lim_{n\to\infty}d(X_n,Y_n)=0 \text{ in probability},
				\end{equation}
				then 
				\begin{equation}
					\lim_{n\to\infty} \CL[Y_n]=\CL(X).
				\end{equation}
			\end{lemma}
			
			\noindent 
			Applying Lemma~\eqref{lem92} to our case, we obtain
			\begin{equation}
				\lim_{N\to\infty}	\CL[(x_1^{[N]}(s))_{s>0}]=\CL[(\bar{\Theta}(s))_{s>0}]\, \text{ in the Meyer-Zheng topology}.
			\end{equation}
			The argument for 
			\begin{equation}
				\lim_{N\to\infty}\CL[(y_1^{[N]}(s))_{s>0}]=\CL[(\bar{\Theta}(s))_{s>0}]\, \text{ in the Meyer-Zheng topology}
			\end{equation}
			follows in the same way. By Lemma~\ref{lem93}, we obtain 
			\begin{equation}
				\lim_{N\to\infty}\CL[(x_1^{[N]}(s),y_1^{[N]}(s)-x_1^{[N]}(s))_{s>0}]=\CL[(\bar{\Theta}(s),0)_{s>0}]\, 
				\text{ in the Meyer-Zheng topology}.
			\end{equation}
			Applying Lemma~\ref{lem94} with $f(x,y)=f(x,y+x)$ and the continuous mapping theorem, we obtain
			\begin{equation}
				\lim_{N\to \infty}\CL\left[\left(x_1^{[N]}(s),y_1^{[N]}(s)\right)_{s>0}\right]=\CL[(\bar{\Theta}(s),\bar{\Theta}(s))_{s>0}]\, 
				\text{ in the Meyer-Zheng topology.}
			\end{equation}
		\end{proof}

		\subsection{Proof of step 4. Mean-field finite-systems scheme}
		\label{step4}
		
		\paragraph{$\bullet$ Proof of Proposition~\ref{P.finsysmf}}
		
		\begin{proof}
			Proposition~\ref{P.finsysmf}(b) follows directly from Propostion~\ref{p.esti}. The proof of Proposition~\ref{P.finsysmf}(a) follows from Proposition~\ref{P.finsysmf}(b) and Proposition~\ref{p.estim}. 
			
			To prove Proposition~\ref{P.finsysmf}(c) fix $t>0$. Consider the processes $(X^{[N]}(sN+t),Y^{[N]}(sN+t))_{t\geq 0}$ as processes on $([0,1]^2)^{\N_0}$ by periodic continuation. Since $([0,1]^2)^{\N_0}$ is compact, the sequence $(X^{[N]}(sN+t),Y^{[N]}(sN+t))_{N\in\N}$ is tight and hence has a converging subsequence. Let $(\bar{\Theta}(s))_{s>0}$ be the limiting process obtained in Proposition \eqref{p.esti}. This has continuous paths and is the unique solution of a well-posed martingale problem, and hence is a Markov process. Denote by $Q_s$ the time-$s$ semigroup corresponding to $(\bar{\Theta}(s))_{s\geq 0}$. Combining Proposition~\ref{prop1}, Proposition~\ref{p.esti} and Lemma~\ref{martav}, we get that, for each converging subsequence, 
			\begin{equation}
				\label{ma1}
				\lim_{k\to \infty} \CL\left[X^{[N_k]}(sN_k+t),Y^{[N_k]}(sN_k+t)\right]
				=\int Q_s\left(\theta,\d \theta^\prime\right)\nu_{\theta^\prime} =\nu(s),
			\end{equation}
			and hence it follows that, for all $t\geq 0$,
			\begin{equation}
				\label{ma2}
				\lim_{N\to \infty}\CL\left[X^{[N]}(sN+t),Y^{[N]}(sN+t)\right]
				= \int Q_s\left(\theta,\d \theta^\prime\right)\nu_{\theta^\prime} =\nu(s).
			\end{equation}
			Let $(X^{\nu(s)}(t),Y^{\nu(s)}(t))_{t\geq 0}$ be the infinite system defined in \eqref{gh52}, starting from initial measure $\nu(s)$. Then it follows from Corollary~\ref{cor1} that we can construct the processes $(X^{[N]}(sN+t),Y^{[N]}(sN+t))_{t\geq 0}$ and $(X^{\nu(s)}(t),Y^{\nu(s)}(t))_{t\geq 0}$ on one probability space such that, for all $t\geq 0$,  
			\begin{equation}
				\lim_{N\to \infty}\E\left[\big|z^{\nu(s)}_{(i,R_i)}(t)-z^{[N]}_{(i,R_i)}(sN+t)\big|\right]=0 
				\qquad \forall\,(i,R_i)\in \Z\times\{A,D\}.
			\end{equation}
			Hence we see that the finite-dimensional distributions of the process $(X^{[N]}(sN+t),Y^{[N]}(sN+t))_{t\geq 0}$ converge to the finite-dimensional distributions of te process $(X^{\nu(s)}(t),Y^{\nu(s)}(t))_{t\geq 0}$. 
			
			Since we want that
			\begin{equation}
				\label{g}
				\lim_{N\to \infty}\CL\left[\big(X^{[N]}(sN+t),Y^{[N]}(sN+t)\big)_{t\geq 0}\right]
				= \CL\left[\big(X^{\nu(s)}(t),Y^{\nu(s)}(t)\big)_{t\geq 0}\right],
			\end{equation} 
			we are left to show the tightness of $\CL[(X^{[N]}(sN+t),Y^{[N]}(sN+t))_{t\geq 0}]_{N\in\N}$ in the path space $\CC([0,\infty), \left([0,1]\times[0,1] \right)^{\N_0})$. Since $([0,1]^2)^{\N_0}$ is endowed with the product topology, it is enough to show for all sequence components $(x^{[N]}_i(t))_{t\geq 0}$ and $(y^{[N]}_i(t))_{t\geq 0}$ that they are tight in path space (see \cite[Theorem 7.3]{Bi99}). 
			
			To prove that the components are tight, we use a tightness criterion for semi-martingales by Joffe and Metivier, \cite[Proposition 3.2.3]{JM86}. To use this criterion, we have to show that for all $i\in [N]$ the components $(x^{[N]}_i(t),y^{[N]}_i(t))$ are $\CD$-semi-martingales as defined in \cite[Definition 3.1.1]{JM86}. To do this, let $\CC^*\subset \CC_b\left([0,1]^2\right)$ be the set of polynomials on $[0,1]^2$, and define the operator $\gls{dsmop}\colon\,(\CC^*\times[0,1]^2\times[0,\infty),\Omega)\to\R$ by
			\begin{equation}
				G_\dagger^{[N]}(f,(x,y),t,\omega)
				=\left[\frac{c}{N}\sum_{j\in[N]}[x^{[N]}_j(t,\omega)-x]+K\,[y-x]\right]\frac{\partial f}{\partial x}\\
				+ \frac{1}{2}g(x)\frac{\partial^2 f}{\partial x^2} + e\,[x-y]\frac{\partial f}{\partial y}.
			\end{equation}
			We use the subscript $\dagger$ to emphasize that $G^{[N]}_\dagger$ is the operator of a $\CD$-semi-martingale and not a generator.
			Below we check in 4 steps that the component processes $(x_i^{[N]}(t),y_i^{[N]}(t))_{t\geq 0}$ are indeed $\CD$-semi-martingales.
			\begin{enumerate}
				\item 
				The functions 
				\begin{equation}
					\label{c1}
					f_1(x_i,y_i) =x_i, \quad f_2(x_i,y_i) =y_i,
				\end{equation}  
				are in $\CC^*$, and so are $f_1^2$, $f_1f_2$, $f_2^2$.
				\item 
				For every $((x,y),t,\omega)\in \left([0,1]^2\times [0,\infty) \times\Omega\right)$, the mapping $f\mapsto G_\dagger^{[N]}(f,(x,y),t,\omega)$ is linear on $\CC^*$ and $G_\dagger^{[N]}(f,\cdot,t,\omega)\in \CC^*$.
				\item 
				Let $(\CF_s)_{s\geq 0}$ be the filtration generated by the Brownian motions $((w_i(s))_{s\geq 0})_{i\in[N]}$, and let $\CP$ be the $\sigma$-algebra generated by the predictable sets, i.e., sets of the form $(s,t]\times F$ for $F\in\CF_s$. Since the component processes $(x_j^{[N]}(t))_{t\geq 0}$ are continuous, $((x,y),t,\omega)\mapsto G^{[N]}_\dagger(f,(x,y),t,\omega)$ is $\CB([0,1]^2) \otimes \CP$ measurable for every $f\in \CC^*$, where $\CP$ is the $\sigma$-algebra generated by the sets of the form $(s,t]\times F$ for $F\in\CF_s$.
				\item 
				Applying It\^o's formula to the SSDE in \eqref{gh45a}, we obtain, for every $f\in\CC^*$,
				\begin{equation}
					\begin{aligned}
						f\big(x_i^{[N]}(t),y_i^{[N]}(t)\big)
						&=f\big(x_i^{[N]}(0),y_i^{[N]}(0)\big)\\
						&\qquad +\int_0^t \d s\,\frac{c}{N}\sum_{j\in[N]} \big[x^{[N]}_j(s,\omega)-x^{[N]}_i(s)\big]\, 
						\frac{\partial f}{\partial x}\big(x_i^{[N]}(t),y_i^{[N]}(t)\big)\\
						&\qquad +\frac{1}{2}\int_0^t \d w_i(s)\,\sqrt{g(x_i^{[N]}(s))}\,\,\frac{\partial f}{\partial x}\big(x_i^{[N]}(t),y_i^{[N]}(t)\big)\\
						&\qquad +\int_0^t \d s\,Ke\big[y_i^{[N]}(s)-x_i^{[N]}(s)\big]\,\frac{\partial f}{\partial x}\big(x_i^{[N]}(t),y_i^{[N]}(t)\big)\\
						&\qquad +\int_0^t \d s\,e \big[x_i^{[N]}(s)-y_i^{[N]}(s)\big]\,\frac{\partial f}{\partial y}\big(x_i^{[N]}(t),y_i^{[N]}(t)\big)\\
						&\qquad +\int_0^t \d s\, g(x_i^{[N]}(s))\,\,\frac{\partial^2 f}{\partial x^2}\big(x_i^{[N]}(t),y_i^{[N]}(t)\big).
					\end{aligned}
				\end{equation}
				Therefore 
				\begin{equation}
					\begin{aligned}
						M^{[N],f}(t,\omega)
						&=f\big(x_i^{[N]}(t,\omega),y_i^{[N]}(t,\omega)\big)-f\big(x_i^{[N]}(0,\omega),y_i^{[N]}(0,\omega)\big)\\
						&\qquad-\int_0^t \d s\,G_\dagger^{[N]}\big(f(x_i^{[N]}(s,\omega),y_i^{[N]}(s,\omega)),s,\omega\big)
					\end{aligned}
				\end{equation}
				is a square-integrable martingale on $\left(\Omega, (\CF_s)_{s\geq 0},\P\right)$.
			\end{enumerate}
			
			To check that the sequence of component processes $((x_i^{[N]}(t),y_i^{[N]}(t)))_{N\in\N}$ is tight, we need the local characteristics of the $\CD$-semi-martingale, which are defined in \cite[Definition 3.1.2]{JM86} as (recall \eqref{c1})
			\begin{equation}
				\begin{aligned}
					b^{[N]}_1((x,y),t,\omega) 
					&=G^{[N]}_\dagger(f_1,(x,y),t,\omega),\\
					b^{[N]}_2((x,y),t,\omega)
					&=G^{[N]}_\dagger(f_2,(x,y),t,\omega),\\
					a^{[N]}_{(1,1)}((x,y),t,\omega)
					&=G^{[N]}_\dagger(f_1 f_1,(x,y),t,\omega)- 2x\, b_1((x,y),t,\omega),\\
					a^{[N]}_{(2,1)}((x,y),t,\omega)  
					&=G^{[N]}_\dagger(f_1 f_2,(x,y), t,\omega)-  x\, b_2((x,y),t,\omega)-  y b_1((x,y),t,\omega),\\
					a^{[N]}_{(1,2)}((x,y),t,\omega)
					&=a_{(2,1)}((x,y),t,\omega),\\
					a^{[N]}_{(2,2)}((x,y),t,\omega)
					&=G^{[N]}_\dagger(f_2 f_2,(x,y),t,\omega)- 2y\, b_2((x,y),t,\omega).
				\end{aligned}
			\end{equation}
			Hence
			\begin{equation}
				\begin{aligned}
					b^{[N]}_1((x,y),t,\omega)
					&=\frac{c}{N}\sum_{j\in[N]}[x_j^{[N]}(t,\omega)-x]+Ke[y-x],\\
					b^{[N]}_2((x,y),t,\omega)
					&=e[x-y],\\
					a^{[N]}_{(1,1)}((x,y),t,\omega)
					&= 2 g(x),\\
					a^{[N]}_{(1,2)}((x,y),t,\omega)
					&=a_{(2,1)}((x,y),t,\omega)=0,\\
					a^{[N]}_{(2,2)}((x,y),t,\omega)
					&=0.
				\end{aligned}
			\end{equation}
			Here, $b^{[N]}_i$ and $a^{[N]}_{i,j}$, $i,j\in\{1,2\}$, are called the local coefficients of first and second order. We check that the hypotheses \cite[H1, H2, H3 in Section 3.2.1]{JM86} are satisfied.
			\begin{itemize}
				\item[$H_1$:] 
				For all $N\in\N$, 
				\begin{equation}
					\sum_{i\in\{1,2\}}
					|b^{[N]}_i((x,y),t,\omega)|^2+\sum_{i,j\in\{1,2\}}|a^{[N]}_{(i,j)}((x,y),t,\omega)|^2
					\leq 4(c+Ke+e)^2+2||g||^2.
				\end{equation} 
				Hence, choosing as positive adapted process the constant process $1$ and letting the constant be equal to $4(c+Ke+e)^2+2||g||^2$, we see that $H_1$ is satisfied.
				\item[$H_2$:] 
				Since the component processes are bounded by $1$, also $H_2$ is satisfied. 
				\item[$H_3$:] 
				Since the increasing c\`adl\`ag function $(A^{[N]}(t))_{t\geq 0}$ in \cite[Definition 3.1.1]{JM86} is in our case $A^{[N]}(t)=t$, also $H_3$ is satisfied.   
			\end{itemize} 
			Since $H_1$, $H_2$, $H_3$ are met, \cite[Proposition 3.2.3]{JM86} implies that $((x_i^{[N]}(t),y_i^{[N]}(t))_{t>0})_{N\in\N}$ are tight in the space of c\`adl\`ag paths $\CD((0,\infty],[0,1]^2)$. Hence \eqref{g} indeed holds. 
		\end{proof}
		
		\section{Preparation: $N\to\infty$, two-colour mean-field}
		\label{s.finlevel}
		
		In this section we extend the results obtained in Section~\ref{mffss} to a mean-field system where the seed-bank consists of two colours, one colour that interacts on the slow time scale and one colour that interacts on the fast time scale. To do so we follow the set-up used in Chapter~\ref{s.intromultscallim}. In particular, we highlight the role of the second colour. Section~\ref{ss.tcmfs} builds up the setting and states the main scaling result: Proposition~\ref{P.finsysmf2}. Section~\ref{pmfs2} provides the proof of this proposition based on a series of lemmas, which are stated and proved first. 
		
		\subsection{Two-colour mean-field finite-systems scheme}
		\label{ss.tcmfs}
		
		
		\paragraph{Setup.}
		
		In this section we consider a simplified version of our SSDE in \eqref{moSDE} on the finite geographic space
		\begin{equation}
			[N]=\{0,1,\ldots,N-1\}, \qquad N \in \N.
		\end{equation} 
		The migration kernel $a^{\Omega_N}(\cdot,\cdot)$ is replaced by the migration kernel $a^{[N]}(i,j)=c_0 N^{-1}$ for all $i,j\in [N]$, where $c_0\in (0,\infty)$ is a constant. The seed-bank consists of \emph{two colours}, labeled $0$ and $1$. The exchange rates between the active and the colour-$0$ dormant population are given by $K_0e_0,e_0$. The exchange rates between active and the colour-$1$ dormant population are given by $\frac{K_1e_1}{N},\frac{e_1}{N}$. The state space is
		\begin{equation}
			S = \mathfrak{s}^{[N]}, \qquad \mathfrak{s} = [0,1]\times [0,1]^2,
		\end{equation}
		and the system, consisting of three components, is denoted by
		\begin{equation}
			\label{e7152}
			\begin{aligned}
				&Z^{[N]}(t)=\big(X^{[N]}(t),(Y_0^{[N]}(t),Y_1^{[N]}(t))\big)_{t \geq 0},\\
				&\big(X^{[N]}(t),(Y_0^{[N]}(t),Y_1^{[N]}(t))\big) = \big(x_i(t), (y_{i,0}(t),y_{i,1}(t))\big)_{i \in [N]}.
			\end{aligned}
		\end{equation}
		The components of $(Z^{[N]}(t))_{t\geq 0}$ evolve according to the SSDE
		\begin{equation}
			\label{gh45a2}
			\begin{aligned}
				&\d x^{[N]}_i(t) = \frac{c_0}{N} \sum_{j \in [N]} [x^{[N]}_j(t) - x^{[N]}_i(t)]\, \d t 
				+ \sqrt{g(x^{[N]}_i(t))}\, \d w_i (t)\\
				&\qquad\qquad + K_0 e_0\, [y^{[N]}_{i,0}(t)-x^{[N]}_{i}(t)]\,\d t+ \frac{K_1 e_1}{N}\, [y^{[N]}_{i,1}(t)-x^{[N]}_{i}(t)]\,\d t,\\
				&\d y^{[N]}_{i,0}(t) = e_0\,[x^{[N]}_i(t)-y^{[N]}_{i,0}(t)]\, \d t,\\
				&\d y^{[N]}_{i,1}(t) = \frac{e_1}{N}\,[x^{[N]}_i(t)-y^{[N]}_{i,1}(t)]\, \d t, \qquad i \in [N],
			\end{aligned}
		\end{equation}
		which is a special case of \eqref{moSDE}. The initial state is $\mu(0)=\mu^{\otimes[N]}$ for some $\mu\in\CP([0,1]^3)$. The SSDE in \eqref{gh45a2} has a unique weak solution coming from a well-posed martingale problem \cite[Theorem 3.1]{SS80}. By \cite[Theorem 3.2]{SS80}, \eqref{gh45a2} has a unique strong solution for every deterministic initial state $Z^{[N]}(0)$. Therefore  the solutino of \eqref{gh45a2} is Feller and Markov for any initial law. The SSDE in \eqref{gh45a2} can alternatively be written as
		\begin{equation}
			\label{gh45a2b}
			\begin{aligned}
				&\d x^{[N]}_i(t) =  c_0\left[\frac{1}{N} \sum_{j \in [N]}x^{[N]}_j(t) - x^{[N]}_i(t)\right]\, \d t 
				+ \sqrt{g\big(x^{[N]}_i(t))\big)}\, \d w_i (t)\\
				&\qquad\qquad\qquad + K_0 e_0\, [y^{[N]}_{i,0}(t)-x^{[N]}_{i}(t)]\,\d t
				+ \frac{K_1 e_1}{N}\, [y^{[N]}_{i,1}(t)-x^{[N]}_{i}(t)]\,\d t,\\
				&\d y^{[N]}_{i,0}(t) = e_0\,[x^{[N]}_i(t)-y^{[N]}_{i,0}(t)]\, \d t,\\
				&\d y^{[N]}_{i,1}(t) = \frac{e_1}{N}\,[x^{[N]}_i(t)-y^{[N]}_{i,1}(t)]\, \d t, \qquad i \in [N].
			\end{aligned}
		\end{equation}
		So the migration term for a single colony can be interpreted as a drift towards the average of the active population. We are interested in $\CL[(Z^{[N]}(t)))_{t \geq 0}]$ in the limit as $N\to\infty$, on time scales $t$ and $Ns$. Heuristically, analysing the SSDE in \eqref{gh45a2b}, we can foresee the following results, which are made precise in Proposition \ref{P.finsysmf2}.
		
		\medskip\noindent 
		$\bullet$ \textbf{On time scale $1=N^0$} (space-time scale $0$), in the limit as $N\to\infty$ the colour-$1$ dormant population $(Y^{[N]}_1(t))_{t\geq0}$  in \eqref{gh45a2} converges to a constant process, since the single components $y_{i,1}$ do not move on time scale $t$. The components of $(X^{[N]}(t),Y_0^{[N]}(t))_{t \geq 0}$ converge to i.i.d.\ copies of the single-colony McKean-Vlasov process in \eqref{SC}, where in the corresponding SSDE the parameters $c,e,K$ are replaced by $c_0,e_0,K_0$ and $E=1$. So on time scale $t$ we only see the colour-$0$ dormant population interacting with the active population, and the colour-$1$ dormant population is not yet coming into play. Therefore the colour-$0$ dormant population is the \textit{effective seed-bank} on time scale $1$, and the process
		\begin{equation}
			z_0^{\eff,[N]}(t)=(x^{[N]}_0(t),y^{[N]}_{0,0}(t))_{t\geq0}
		\end{equation}
		is called the effective process on level $0$. Note that the active population has a drift towards $\frac{1}{N}\sum_{j\in[N]}x_j(t)$, which in the McKean-Vlasov limit is replaced by $\E[x(t)]$ given by \eqref{expz}.
		
		\medskip\noindent
		$\bullet$ \textbf{On time scale $N$} (space-time scale $1$), we look at the averages 
		\begin{equation}
			\label{gh412}
			\begin{aligned}
				(z_1^{[N]}(s))_{s>0}&=\left(x_1^{[N]}(s),( y_{0,1}^{[N]}(s),y_{1,1}^{[N]}(s))\right)_{s > 0}\\
				&=\left( \frac{1}{N} \sum_{i \in [N]} x^{[N]}_i(Ns), \left( \frac{1}{N} 
				\sum_{i \in [N]} y^{[N]}_{i,0}(Ns), \frac{1}{N} \sum_{i \in [N]} y^{[N]}_{i,1}(Ns)\right)\right)_{s> 0}.
			\end{aligned}
		\end{equation} 
		Again the lower index $1$ indicates that the average is the analogue of the 1-block average defined in \eqref{blockav}. Using \eqref{gh45a2}, we see that the dynamics of the system in \eqref{gh412} is given by the SSDE
		\begin{equation}
			\label{mfevolve2}
			\begin{aligned}
				\d x_1^{[N]}(s)&=\sqrt{\frac{1}{N}\sum_{i\in[N]}g(x^{[N]}_i(Ns))}\,\d w(s)
				+NK_0e_0\left[y_{0,1}^{[N]}(s)-x_1^{[N]}(s)\right]\d s\\
				&\qquad +K_1e_1\left[y_{1,1}^{[N]}(s)-x_1^{[N]}(s)\right]\d s,\\
				\d y_{0,1}^{[N]}(s)&=Ne_0\left[x_1^{[N]}(s)-y_{0,1}^{[N]}(s)\right]\d s,\\
				\d y_{1,1}^{[N]}(s)&=e_1\left[x_1^{[N]}(s)-y_{1,1}^{[N]}(s)\right]\d s.
			\end{aligned}
		\end{equation}
		Thus, as in the mean-field system with one-colour, on time scale $N$ infinite rates appear in the interaction of the active population with the colour-$0$ dormant population. Therefore in the limit as $N\to\infty$ the path becomes rougher and rougher at rarer and rarer times. Using the \emph{Meyer-Zheng topology} we can prove that $\lim_{N\to\infty} y_{0,1}^{[N]}(s)=\lim_{N\to\infty }x_1^{[N]}(s)$ most of the time. On the other hand, on time scale $N$, $x_1^{[N]}(s)$ has a non-trivial interaction with $y_{1,1}^{[N]}(s)$, and therefore we say that on time scale $N$ the colour-$1$ dormant population is the \textit{effective seed-bank}. Note that for the evolution of the average $\frac{x_1^{[N]}(s)+K_0y_{0,1}^{[N]}(s)}{1+K_0}$ the rates with a factor $N$ in front cancel out. We will use the quantity $\frac{x_1^{[N]}(s)+K_0y_{0,1}^{[N]}(s)}{1+K_0}$ to obtain results in the classical path-space topology. We call 
		\begin{equation}
			(z_1^{[N],\eff}(s))_{s>0} =	\left(\frac{x_1^{[N]}(s)+K_0y_{0,1}^{[N]}(s)}{1+K_0},y_{1,1}(s)\right)_{s>0}
		\end{equation}
		the effective process on space-time scale 1. We will call space-time scale $1$ also level $1$.
		
		\paragraph{Scaling limit.}
		
		To describe the limiting dynamics of the system in \eqref{gh45a2}, we need the infinite-dimensional process 
		\begin{equation}\label{1010}
			(Z(t))_{t\geq 0} = \big((z_i(t))_{t\geq 0}\big)_{i\in\N_0}= \big((x_i(t),(y_{i,0}(t),y_{i,1}(t)))_{t\geq 0}\big)_{i\in\N_0}
		\end{equation}
		with state space $([0,1]^3 )^{\N_0}$ that evolves according to
		\begin{equation}
			\label{gh45a2binf}
			\begin{aligned}
				\d x_i(t) &=  c_0 [\theta - x_i(t)]\, \d t 
				+ \sqrt{g(x_i(t))}\, \d w_i (t) + K_0 e_0\, [y_{i,0}(t)-x_{i}(t)]\,\d t,\\
				\d y_{i,0}(t) &= e_0\,[x_i(t)-y_{i,0}(t)]\, \d t,\\
				y_{i,1}(t) &= y_{i,1} , \qquad i \in \N_0.
			\end{aligned}
		\end{equation}
		Here, $\theta\in[0,1]$ and $y_{i,1}\in[0,1]$ for all $i\in\N_0$. We will also need the limiting effective process
		\begin{equation}
			(Z^\eff(t))_{t\geq 0} = \big((z^\eff_i(t))_{t\geq 0}\big)_{i\in\N_0} = \big((x_i^\eff(t),y^\eff_{i,0}(t))_{t\geq 0}\big)_{i\in\N_0}
		\end{equation}
		with state space $([0,1]^2)^{\N_0}$ that evolves according to
		\begin{equation}
			\label{m12h}
			\begin{aligned}
				\d x^{\eff}_i(t) &=  c_0 [\theta - x^{\eff}_i(t)]\, \d t 
				+ \sqrt{g\big(x^{\eff}_i(t)\big)}\, \d w (t) + K_0 e_0\, [y^{\eff}_{i,0}(t)-x^{\eff}_{i}(t)]\,\d t,\\
				\d y^{\eff}_{i,0}(t) &= e_0\,[x^{\eff}_i(t)-y^{\eff}_{i,0}(t)]\, \d t,\qquad i\in\N_0.
			\end{aligned}
		\end{equation}
		
		Like for the one-colour mean-field finite-systems scheme, we need the following list of ingredients to formally state our multi-scaling properties:
		\begin{enumerate}
			\item 
			For positive times $t>0$, we define the so-called \emph{estimators} for the finite system by:
			\begin{equation}
				\label{ma6}
				\begin{aligned}
					\bar{\Theta}^{(1),[N]}(t) &=\frac{1}{N}\sum_{i\in[N]}\frac{x^{[N]}_i(t)+K_0y^{[N]}_{i,0}(t)}{1+K_0},\\
					\Theta^{(1),[N]}_x(t) &=\frac{1}{N}\sum_{i\in[N]}x^{[N]}_{i}(t),\\
					\Theta^{(1),[N]}_{y_0}(t) &=\frac{1}{N}\sum_{i\in[N]}y^{[N]}_{i,0}(t),\\
					\Theta^{(1),[N]}_{y_1}(t) &=\frac{1}{N}\sum_{i\in[N]}y^{[N]}_{i,1}(t).
				\end{aligned}
			\end{equation}
			We abbreviate
			\begin{equation}
				\label{1013}
				\begin{aligned}
					{\bf\Theta}^{(1),[N]}(t)&=\left({\Theta}_x^{(1),[N]}(t),\Theta^{(1),[N]}_{y_0}(t),\Theta^{(1),[N]}_{y_1}(t)\right),\\
					{\bf\Theta}^{\eff,(1),[N]}(t)&=\left(\bar{\Theta}^{(1),[N]}(t),\Theta^{(1),[N]}_{y_1}(t)\right).
				\end{aligned}
			\end{equation}
			We refer to $({\bf\Theta}^{\eff,(1),[N]}(t))_{t\geq0}$ as the \emph{effective  estimator process} and to $({\bf\Theta}^{(1),[N]}(t))_{t\geq 0}$ as the \emph{ estimator process}.  
			\item 
			The \emph{time scale} $Ns$ is such that $\CL[\bar\Theta^{[N]} (Ns-L(N))-\bar\Theta^{[N]}(Ns)]=\delta_0$ for all $L(N)$ satisfying $\lim_{N\to\infty}$ $L(N) = \infty$ and $\lim_{N\to\infty} L(N)/N=0$, but not for $L(N)=N$. In words, $Ns$ is the time scale on which $\bar\Theta^{[N]}(\cdot)$ starts evolving, i.e., $(\bar\Theta^{[N]}(Ns))_{s>0}$ is no longer a fixed process. When we scale time by $Ns$, we will use $s$ as a time index, which indicates the ``fast time scale". The ``slow time scale" will be indicated by $t$. Thus, the time scales for the two-colour mean-field system are the same as the time scales for the one-colour mean-field system.
			
			\begin{remark}{\bf{[Notation]}}
				{\rm The upper index $1$ in $\bar\Theta^{(1)}$ and $\Theta_{y_1}^{(1)}$ is used to indicate that we are working with a system of level 1, so the system that lives on space-time scale 1. This can later be easily generalized to levels $2$ and $k$.} \hfill$\blacksquare$
			\end{remark}
			
			\item 
			The \emph{invariant measure} (i.e., the equilibrium measure) for the evolution of a single colony in \eqref{gh45a2binf}, written 
			\begin{equation}
				\label{singcoleq2}
				\Gamma_{\theta,\theta,y_1},
			\end{equation}
			and the \emph{invariant measure} of the infinite system in \eqref{gh45a2binf}, written $\nu_{\theta,\theta,{\bf y_1}}=\Gamma_{\theta,\theta,y_1}^{\otimes\N_0}$ with $\theta \in [0,1]$ and ${\bf y}_1\in[0,1]^{\N_0}$ a random variable. The existence of the invariant measure $\nu_\theta$ and the convergence of $\CL[Z(t)_{t\geq 0}]$ towards $\nu_\theta$ will be shown in the proof of Proposition~\ref{P.finsysmf2}. 
			
			\item
			The invariant measure of the effective process in \eqref{m12h},
			\begin{equation}
				\label{singcoleq2ef}
				\Gamma^\eff_{\theta},
			\end{equation}	
			and the invariant measure for the full process, $\nu^\eff_\theta=(\Gamma^{\eff}_{\theta})^{\otimes\N_0}$.
			
			\item 
			The renormalisation transformation $\CF\colon\,\CG\to\CG$,
			\begin{equation}
				\label{gh42b}
				(\CF g)(\theta) = \int_{([0,1]^2)^{\N_0}} g(x_0)\,\nu^\eff_{\theta}(\d x_0, \d y_{0,0}), \quad \theta \in [0,1],
			\end{equation} 
			where $\Gamma_{\theta}^\eff$ is the equilibrium measure of \eqref{singcoleq2}. Note that this is the same transformation as defined in \eqref{renor}, but for the truncated system. Since $\nu^\eff_{\theta}$ is a product measure, we can write
			\begin{equation}
				\label{gh42b2}
				(\CF g)(\theta) = \int_{[0,1]^2} g(x)\,\Gamma^\eff_{\theta}(\d x, \d y_0), \quad \theta \in [0,1],
			\end{equation} 
			
			\item  
			The limiting 1-block process $(z_1(s))_{s>0}=(x_1(s),(y_{0,1}(s),y_{1,1}(s)))_{s>0}$ evolving according to
			\begin{equation}
				\label{gh432}
				\begin{aligned}
					\d x_1(s) &=  \frac{1}{1+K_0} \left[\sqrt{(\CF g)\big(x_1(s)\big)}\, \d w(s)
					+ K_1 e_1\left[y_{1,1}(s)-x_1(s)\right]\d s\right],\\
					y_{0,1}(s) &=  x_1 (s),\\
					\d y_{1,1}(s) &= e_1\left[x_1 (s)-y_{1,1}(s)\right]\,\d s,
				\end{aligned}
			\end{equation}
			where $\CF g$ is defined in \eqref{gh432}.  The effective process $(z^{\eff}_1(s))_{s>0}=(x^{\eff}_1(s),y^{\eff}_{1,1}(s))_{s>0}$ on space-time scale $1$,
			\begin{equation}
				\label{m64}
				\begin{aligned}
					\d x^{\eff}_1(s) &=  \frac{1}{1+K_0}\left[
					\sqrt{(\CF g)(x^{\eff}_1(s))}\, \d w (s) + K_1 e_1\, [y^{\eff}_{1,1}(s)-x^{\eff}_{1}(s)]\,\d s\right] ,\\
					\d y^{\eff}_{1,1}(s) &= e_1\,[x^{\eff}_1(s)-y^{\eff}_{1,1}(s)]\, \d s.
				\end{aligned}
			\end{equation}
		\end{enumerate}
		We are now ready to state the scaling limit for the evolution of the averages in \eqref{gh412}, which we refer to as the  \emph{mean-field finite-systems scheme with two colours}. 
		
		\begin{proposition}{{\bf [Mean-field: two-colour finite-systems scheme]}}
			\label{P.finsysmf2}
			$\mbox{}$\\
			Suppose that $\CL[Z^{[N]}(0)]=\mu^{\otimes [N]}$ for some $\mu\in \CP\left([0,1]\times[0,1]^2\right)$. Let 
			\begin{equation}
				\vartheta_0 =\E^{\mu}\left[ \frac{x+K_0y_0}{1+K_0}\right],\qquad\qquad \theta_{y_1}=\E^{\mu}\left[ y_1\right].
			\end{equation} 
			\begin{itemize}
				\item[(a)]
				For the effective estimator process defined in \eqref{1013},
				\begin{equation}
					\label{m1a}
					\lim_{N\to\infty} \CL \left[\left({\bf\Theta}^{\eff,(1),[N]}(Ns)\right)_{s > 0}\right] 
					= \CL \left[\left(z_1^\eff(s)\right)_{s > 0}\right],
				\end{equation}
				where the limit is determined by the unique solution of the SSDE \eqref{m64}, with initial state
				\begin{equation}
					\begin{aligned}
						z_1^\eff(0)=\left(x_1^\eff(0),y^\eff_1(0)\right)= \left(\vartheta_0,\theta_{y_1}\right).
					\end{aligned}
				\end{equation}
				\item[(b)] Assume for the $1$-dormant single components that
				\begin{equation}
					\label{as}
					\lim_{N\to\infty}\CL\left[Y_1^{[N]}(Ns)\Big|{\bf\Theta}^{(1),[N]}(Ns)\right] = P^{z_1(s)}_{Y_1(s)}. 
				\end{equation}
				Define
				\begin{equation}
					\label{ma3ef}
					\Gamma^\eff_{(\vt_0,\theta_{y_1})}(s)= \int_{[0,1]^2} S_s\bigl((\vartheta_0,\theta_{y_1}),\d (u_x,u_y)\bigr)\,
					\Gamma^\eff_{u_x} \in \CP([0,1]^2),
				\end{equation}
				where $S_s((\vartheta_0,\theta_{y_1}),\cdot)$ is the time-$s$ marginal law of the process $(z_1^\eff(s))_{s > 0}$  starting from $(\theta_0,\theta_{y_1}) \in [0,1]^2$ and $\Gamma^\eff_{u_x}$ is the equilibrium distribution of the system in \eqref{m12h} with $\theta=u_x$ (note that $\Gamma^\eff_{\vt_0,\theta_{y_1}}(0)=\Gamma^\eff_{\vt_0}$). Let $(z^{\eff,\Gamma_{(\vartheta_0,\theta_{y_1})}(s)}(t))_{t \geq 0}$ be the process  with initial law $z^{\eff,\Gamma_{(\vartheta_0,\theta_{y_1})}(s)}(0)$ drawn according to $\Gamma^\eff_{  (\vartheta_0,\theta_{y_1})}(s)$ (which is a mixture of random processes in equilibrium) that, conditional on $x_1^\eff(s)=\theta$, evolves according to \eqref{m12h}. Then, for every $s \in (0,\infty)$,
				\begin{equation}
					\lim_{N\to\infty} \CL\left[\left(z^{\eff,[N]}_0(Ns+t)\right)_{t \geq 0}\right]
					= \CL \left[(z^{\Gamma^\eff_{  (\vartheta_0,\theta_{y_1})}(s)}(t))_{t \geq 0}\right].
				\end{equation}
				\item[(c)]
				For the averages in \eqref{gh412}, 
				\begin{equation}
					\label{gh42}
					\begin{aligned}
						&\lim_{N\to\infty} \CL \left[\left(z_1^{[N]}(s)\right)_{s > 0}\right] 
						= \CL \left[\left(z_1^{}(s)\right)_{s > 0}\right] \\
						&\text{in the Meyer-Zheng topology},
					\end{aligned}
				\end{equation}
				where the limit process is the unique solution of the SSDE in \eqref{gh432} with initial state
				\begin{equation}
					\label{e8652}
					z_1 (0)=\left( x_1 (0), y_{0,1}(0),y_{1,1}(0)\right) 
					= \left(\vartheta_0,\vartheta_0, \theta_{y_1}\right). 
				\end{equation}
				\item[(d)] 
				Assume \ref{as}	and define 
				\begin{equation}
					\label{ma3}
					\nu(s) = \int_{[0,1]^3} S_s\bigl((\vartheta_0,\vt_0,\theta_{y_1}),\d (u_x,u_x,u_{y_1})\bigr)\,
					\int_{[0,1]^{\N_0}}P^{(u_x,u_x,u_{y_1})}_{Y_1(s)}(\d {\bf y_1})\,\nu_{u_x,{\bf y_1}},
				\end{equation}
				where $S_s((\vartheta_0,\vt_0,\theta_{y_1}),\cdot)$ is the time-$s$ marginal law of the process $\left(z_1(s)\right)_{s > 0}$ in \eqref{gh432},  starting from $(\vartheta_0,\vt_0,\theta_{y_1}) \in [0,1]^3$, and $\nu_{u_x,{\bf y_1}}$ is the equilibrium distribution of the system in \eqref{gh45a2binf} with $\theta=u_x$ and $(y_{i,1})_{i\in\N_0}=\bf{y_1}$, (note that $\nu(0)=\nu_{\vt_0, (y_{i,1}(0))_{i\in\N_0}}$). Let $(z^{\nu(s)}(t))_{t \geq 0}$ be the process on $([0,1]^3 )^{\N_0}$ with initial measure $z^{\nu(s)}(0)$ drawn according to $\nu(s)$ (which is a mixture of random processes in equilibrium) that conditional on $x_1(s)=\theta$ and $Y_1(s)={\bf y_1}$ evolves according to \eqref{gh45a2binf} with $\theta=u_x$ and $(y_{i,1})_{i\in\N_0}={\bf y_1}$. Then, for every $s \in (0,\infty)$,
				\begin{equation}
					\lim_{N\to\infty} \CL\left[\left(Z^{[N]}(Ns+t)\right)_{t \geq 0}\right]
					= \CL \left[(z^{\nu(s)}(t))_{t \geq 0}\right].
				\end{equation}
			\end{itemize}
		\end{proposition}
		
		\begin{remark}[{\bf Law of $1$-dormant single components}]
			{\rm Note that 
				\begin{equation}
					\label{y1}
					\left(\CL\left[Y_1^{[N]}(Ns)\Big|\left(\bar{\Theta}^{[N]}(Ns),\Theta_{y_1}^{[N]}(Ns)\right)
					\right]\right)_{N\in\N_0}
				\end{equation}
				is a tight sequence of measures. Hence there exist weak limit points. In Section~\ref{ss.tlhmfs} we will see that if there is a higher layer in the hierarchy, then we can show that all weak limit points of \eqref{y1} are the same and we can identify the limit. For Theorems~\ref{T.multiscalehiereff} and \ref{T.multiscalehier} we do not need this assumption, since there will alwyas be multiple higher levels.}\hfill$\blacksquare$
		\end{remark}

		\subsection{Proof of the two-colour mean-field finite-systems scheme}
		\label{pmfs2}
		
		The proof of Proposition~\ref{P.finsysmf2}, the finite-systems scheme with one level and two colours, follows the strategy used in Section~\ref{ss.pabstracts} for the proof of Proposition~\ref{P.finsysmf}. Like for the one-colour finite-systems scheme, we denote the slow time scale by $t$ and the fast time scale by $s$. The proof consists of the following 6 steps:
		\begin{enumerate}
			\item 
			Tightness of the effective estimator processes defined in \eqref{1013}. 
			\begin{equation}
				\big(({\bf\Theta}^{\eff,(1),[N]}(Ns))_{s>0}\big)_{N\in\N}
			\end{equation}
			\item 
			Stability property of $({\bf\Theta}^{\eff,(1),[N]}(Ns+t))_{t>0}$, i.e., for $L(N)$ satisfying $\lim_{N\to \infty}L(N)=\infty$ and $\lim_{N\to \infty} L(N)/N=0$, and all $\epsilon>0$,
			\begin{equation}
				\label{ma22}
				\lim_{N\to\infty}\P\left[\sup_{0 \leq t\leq L(N)}\left|\bar{\Theta}^{(1),[N]}(Ns)
				-\bar{\Theta}^{(1),[N]}(Ns-t)\right|>\epsilon\right]=0.
			\end{equation}
			and
			\begin{equation}
				\label{ma22y}
				\lim_{N\to\infty}\P\left[\sup_{0 \leq t\leq L(N)}\left|\Theta_{y_1}^{(1),[N]}(Ns)
				-\Theta_{y_1}^{(1),[N]}(Ns-t)\right|>\epsilon\right]=0.
			\end{equation}
			\item 
			Equilibrium of the infinite system and the one-dimensional distribution of the effective single components $(Z(Ns+t))_{t>0}$, analogous to Proposition \ref{prop1}.
			\item 
			Limiting evolution of the effective processes $(({\bf\Theta}^{\eff,(1),[N]}(Ns))_{s>0})_{N\in\N}$.
			\item 
			Evolution of the $1$-blocks in the Meyer-Zheng topology.
			\item 
			Proof of Proposition~\ref{P.finsysmf2}.
		\end{enumerate}
		
		\paragraph{Step 1: Tightness of the 1-block estimators.}
		
		\begin{lemma}{\bf[Tightness of the 1-block estimator]}
			Let 
			\begin{equation}
				({\bf\Theta}^{\eff,(1),[N]}(Ns))_{s>0}
			\end{equation} 
			be defined as in \eqref{ma6}. Then $(\CL[({\bf\Theta}^{\eff,(1),[N]}(Ns))_{s>0}])_{N\in\N}$ is a tight sequence of probability measures on $\CC((0,\infty), [0,1]^2)$.
		\end{lemma}
		
		\begin{proof}
			To prove tightness of $(({\bf\Theta}^{\eff,(1),[N]}(Ns))_{s>0})_{N\in\N}$, we will prove for all $\epsilon>0$ that the set of measures $(({\bf\Theta}^{\eff,(1),[N]}(Ns))_{s\geq\epsilon})_{N\in\N}$ is tight. To do so, fix $\epsilon>0$. We will again use \cite[Proposition 3.2.3]{JM86}. From \eqref{gh45a2} we find that the 1-block averages $({\bf\Theta}^{\eff,(1),[N]}(Ns))_{s>0}$ evolve according to 
			\begin{equation}
				\label{ma21}
				\begin{aligned}
					\d \bar{\Theta}^{(1),[N]}(Ns)&=\frac{1}{1+K_0}\Bigg[\sqrt{\frac{1}{N}\sum_{i\in[N]}g\big(x^{[N]}_i(Ns)\big)}\,\d w_i(s)\\
					&\qquad+K_1e_1\left[\Theta_{y_1}^{(1),[N]}-\frac{1}{N}\sum_{i\in[N]} x^{[N]}_i(Ns)\right]\,\d s\Bigg],\\
					\d \Theta^{(1),[N]}_{y_1}(Ns)&=e_1\left[\frac{1}{N}\sum_{i\in[N]} x^{[N]}_i(Ns)-\Theta^{(1),[N]}_{y_1}(Ns)\right]\,\d s.
				\end{aligned}
			\end{equation}
			To use \cite[Proposition 3.2.3]{JM86}, we define $\CC^*$ as the set of polynomials on $([0,1]^2)$. Note that $({\bf\Theta}^{\eff,(1),[N]}(Ns))_{s\geq \epsilon}$ is a semi-martingale. Applying Itô's formula, we get
			\begin{equation}
				\begin{aligned}
					&f\left({\bf\Theta}^{\eff,(1),[N]}(Ns)\right)\\
					&=f\left({\bf\Theta}^{\eff,(1),[N]}(N\epsilon)\right)\\
					&\quad +\int_{\epsilon}^{s} \d w_i(r)\,\frac{1}{1+K_0}\sqrt{\frac{1}{N}\sum_{i\in[N]}g\big(x^{[N]}_i(Nr)\big)}\,
					\frac{\partial f}{\partial x}\left({\bf\Theta}^{\eff,(1),[N]}(Nr)\right)\\
					&\quad +\int_{\epsilon}^{s} \d r\,\frac{K_1e_1}{1+K_0}\left[\Theta_{y_1}^{(1),[N]}(Nr)
					-\frac{1}{N}\sum_{i\in[N]} x^{[N]}_i(Nr)\right]\frac{\partial f}{\partial x}\left({\bf\Theta}^{\eff,(1),[N]}(Nr)\right)\\
					&\quad +\int_{\epsilon }^{s} \d r\,e_1\left[\frac{1}{N}\sum_{i\in[N]} x^{[N]}_i(Nr)-\Theta_{y_1}^{(1),[N]}(Nr)\right]
					\,\frac{\partial f}{\partial y}\left({\bf\Theta}^{\eff,(1),[N]}(Nr)\right)\\
					&\quad +\int_{\epsilon }^{s} \d r\,\frac{1}{2(1+K_0)^2} \frac{1}{N}\sum_{i\in[N]}g\big(x^{[N]}_i(Nr)\big)\,
					\frac{\partial^2 f}{\partial x^2}\left({\bf\Theta}^{\eff,(1),[N]}(Nr)\right)
				\end{aligned}
			\end{equation}
			for all $f\in\CC^*$. Hence, if we define the operator 
			\begin{equation}
				\label{m45a}
				\begin{aligned}
					G^{(1),[N]}_\dagger&\colon\, (\CC^*,[0,1]^2,[\epsilon,\infty),\Omega)\to \R,\\
					G^{(1),[N]}_\dagger(f,(x,y),s,\omega)&=
					\frac{K_1e_1}{1+K_0}\left[y-\frac{1}{N}\sum_{i\in[N]} x^{[N]}_i(Ns,\omega)\right]\frac{\partial f}{\partial x}\\
					&\qquad+e_1\left[\frac{1}{N}\sum_{i\in[N]} x^{[N]}_i(Ns,\omega)-y\right]\frac{\partial f}{\partial y}\\
					&\qquad+\frac{1}{2(1+K_0)^2} \frac{1}{N}\sum_{i\in[N]}g(x^{[N]}_i(Ns,\omega))\, \frac{\partial^2 f}{\partial x^2},
				\end{aligned}
			\end{equation}
			then we see that the process $({\bf\Theta}^{\eff,(1),[N]}(Ns))_{s\geq \epsilon}$ is a $\CD$-semi-martingale for all $\epsilon>0$. For all $\epsilon>0$ the conditions $H_1,\ H_2,\ H_3$ are satisfied as before. Therefore we conclude from \cite[Proposition 3.2.3]{JM86} that the sequence $(({\bf\Theta}^{\eff,(1),[N]}(Ns))_{s\geq \epsilon})_{N\in\N}$ is tight. Since this is true for all $\epsilon>0$, we conclude that $(\CL[({\bf\Theta}^{\eff,(1),[N]}(Ns))_{s>0}])_{N\in\N}$ is tight.
		\end{proof}
		
		\paragraph{Step 2: Stability of the $1$-block estimators.}
		
		\begin{lemma}{\bf[Stability property of the 1-block estimator]}
			\label{stab}
			Let  ${\bf\Theta}^{\eff,(1),[N]}(t)$ be defined as in \eqref{ma6}. For any $L(N)$ satisfying $\lim_{N\to \infty}L(N)=\infty$ and $\lim_{N\to \infty} L(N)/N=0$,
			\begin{equation}
				\label{ma22b}
				\lim_{N\to\infty}\sup_{0 \leq t\leq L(N)}\left|\bar{\Theta}^{(1),[N]}(Ns)-\bar{\Theta}^{(1),[N]}(Ns-t)\right|=0
				\text{ in probability}
			\end{equation}
			and
			\begin{equation}
				\label{may}
				\lim_{N\to\infty}\sup_{0 \leq t\leq L(N)}\left|\Theta_{y_1}^{(1),[N]}(Ns)
				-\Theta_{y_1}^{(1),[N]}(Ns-t)\right|=0 \text{ in probability. }
			\end{equation}
		\end{lemma}
		
		\begin{proof}
			Fix $\epsilon>0$. From the SSDE \eqref{gh45a2} we obtain that, for $N$ large enough,
			\begin{equation}
				\begin{aligned}
					&\P\Bigg(\sup_{0\leq t\leq L(N)}\Bigg|\bar{\Theta}^{(1),[N]}(Ns)-\bar{\Theta}^{(1),[N]}(Ns-t)\Bigg|>\epsilon\Bigg)\\
					&=\P\Bigg(\sup_{0 \leq t\leq L(N)}\frac{1}{1+K_0}\Bigg|\int_{Ns-t}^{Ns}\d r\,
					\frac{K_1e_1}{N}\left[\Theta^{(1),[N]}_{y_1}(r)-\frac{1}{N}\sum_{i \in [N]}x^{[N]}_i(r)\right]\\
					&\qquad\qquad\qquad+\int_{Ns-t}^{Ns} \d w_i(r)\,\frac{1}{N}\sum_{i \in [N]}
					\sqrt{g\big(x^{[N]}_i(r)\big)}\,\Bigg|>\epsilon\Bigg]\\
					&\leq\P\Bigg(\Bigg|\frac{L(N)2K_1e_1}{N(1+K_0)}\Bigg|
					+\sup_{0\leq t\leq L(N)}\Bigg|\frac{1}{1+K_0}\int_{Ns-t}^{Ns} 
					\d w_i(r)\,\frac{1}{N}\sum_{i \in [N]}\sqrt{g(x^{[N]}_i(r))}\,
					\Bigg|>\epsilon\Bigg)\\
					&=\P\Bigg(\sup_{0\leq t\leq L(N)}\,\Bigg|\frac{1}{1+K_0}\int_{Ns-t}^{Ns} \d w_i(r)\,\frac{1}{N}
					\sum_{i \in [N]}\sqrt{g\big(x^{[N]}_i(r)\big)}\,\Bigg|>\epsilon-\frac{L(N)2K_1e_1}{N(1+K_0)}\Bigg)\\
					&\leq\P\Bigg(\sup_{0\leq t\leq L(N)}\Bigg|\frac{1}{1+K_0}\int_{Ns-t}^{Ns} \d w_i(r)\,\frac{1}{N}
					\sum_{i \in [N]}\sqrt{g\big(x^{[N]}_i(r)\big)}\,\Bigg|>\frac{\epsilon}{2}\Bigg).
				\end{aligned}
			\end{equation} 
			Applying the same optional stopping argument as used in the proof of Lemma~\ref{martav}, we find \eqref{ma22b}. For \eqref{may}, note that
			\begin{equation}
				\begin{aligned}
					\P&\Bigg(\sup_{0\leq t\leq L(N)}\Bigg|\Theta_{y_1}^{(1),[N]}(Ns)-\Theta_{y_2}^{(1),[N]}(Ns-t)\Bigg|>\epsilon\Bigg)\\
					&=\P\Bigg(\sup_{0 \leq t\leq L(N)}\frac{1}{1+K_0}\Bigg|\int_{Ns-t}^{Ns}\d r\,
					\frac{e_1}{N}\left[\Theta^{(1),[N]}_{y_1}(r)-\frac{1}{N}\sum_{i \in [N]}x_i^{[N]}(r)\right]\Bigg|>\epsilon\Bigg)\\
					&\leq\P\Bigg(\frac{2e_1L(N)}{(1+K_0)N}>\epsilon\Bigg). 
				\end{aligned}
			\end{equation} 
			Let $N\to\infty$ to obtain \eqref{may}.
		\end{proof}

		\paragraph{Step 3: Equilibrium for the infinite system.}
		
		To derive the equilibrium of the single components in the infinite system, we derive the following analoque of Proposition~\ref{prop1}. Recall that the finite system is denoted by $Z^{[N_k]}$ in \eqref{e7152}, and recall the list of ingredients in Section~\ref{ss.tcmfs}.
		
		\begin{proposition}{\bf [Equilibrium for the infinite 2-colour system]}
			\label{prop2}
			Let $(N_k)_{k\in\N}$ be a sequence in $\N$. Fix $s>0$. Let $L(N)$ satisfy $\lim_{N\to\infty} L(N)=\infty$ and $\lim_{N\to\infty} L(N)/N=0$, and suppose that
			\begin{equation}
				\label{112}
				\begin{aligned}
					&\lim_{k\to\infty}\CL\left[{\bf\Theta}^{\eff,(1),[N_k]}(N_ks)\right] = P_{{\bf\Theta^\eff}(s)},\\			
					&\lim_{k\to\infty}\CL\left[Y_1^{[N_k]}(N_ks)\Big|{\bf\Theta}^{\eff,(1),[N_k]}(N_ks)\right] 
					= P^{{\bf\Theta}^{\eff,(1)}(s)}_{Y_1(s)},\\ 
					&\lim_{k\to\infty }\CL\left[\sup_{0\leq t\leq L(N_k)}\left|\bar{\Theta}^{[N_k]}(N_k s)-\bar{\Theta}^{[N_k]}(N_ks-t)\right|
					+\left|{\Theta}_{y_1}^{[N_k]}(N_ks)-{\Theta_{y_1}}^{[N_k]}(N_ks-t)\right|\right]\\
					&\qquad \qquad \qquad \qquad \,\,\, =\delta_{0},\\
					&\lim_{k\to\infty} \CL\bigl[Z^{[N_k]}(N_ks)\bigr] = \nu(s).
				\end{aligned}
			\end{equation}
			Then $\nu(s)$ is of the form
			\begin{equation}
				\nu(s) = \int_{[0,1]^2} P_{{\bf\Theta^\eff}(s)}(\d \theta,\d \theta_y)\,\int_{[0,1]^{\N_0}}
				P^{(\theta,\theta_y)}_{Y_1(s)}(\d {\bf y_1})\,\nu_{\theta,{\bf y_1}},
			\end{equation}
			where ${\bf y_1}=(y_{i,1})_{i\in\N_0}$  is a sequence with elements in $[0,1]$, and $\nu_{\theta,{\bf y_1}}$ is the equilibrium measure of the process in \eqref{1010} evolving according to \eqref{gh45a2binf} with $(y_{i,1})_{i\in\N_0}$ given by the sequence ${\bf y_1}=$.
		\end{proposition}
		
		
		\paragraph{Preparation for the proof of Proposition~\ref{prop2}.}
		
		The proof of Proposition~\ref{prop2} follows the same line of argument as used in the proof of Proposition~\ref{prop1}. We need lemmas that are similar to Lemmas~\ref{lemerg}-\ref{unifergod}, but this time in the setting of the two-colour hierarchical mean-field finite-systems scheme. Afterwards we prove Proposition~\ref{prop2}.
		
		\begin{lemma}{\bf [Convergence for the infinite system]} 
			\label{lemerg2}
			Let $\mu$ be an exchangeable probability measure on $([0,1]^3)^{\N_0}$. Then for the system $(Z(t))_{t \geq 0}$ given by \eqref{1010} with $\CL[Z(0)]=\mu$,
			\begin{equation}
				\lim_{t\to\infty }\CL[Z(t)]= \nu_{\theta,{\bf y_1}},
			\end{equation} 
			where $\nu_{\theta,{\bf y_1}}$ is of the form
			\begin{equation}
				\nu_{\theta,{\bf y_1}}=\prod_{i\in\N_0}\Gamma_{\theta,y_{i,1}}
			\end{equation}
			with $\Gamma_{\theta,y_{i,1}}$ the equilibrium of the $i$th single-component process in \eqref{gh45a2binf}. 
		\end{lemma}
		
		\begin{proof}
			For each component of the infinite system in \eqref{1010} the $1$-dormant single component process $(y_{i,1}(t))_{t\geq 0}$ does not move on time scale $t$. Hence, given the states of $1$-dormant single components, we can use a similar argument as in the proof of Proposition~\ref{P.equergod} (see Section~\ref{sec:equergod}) to show that the single components converge to an equilibrium measure $\Gamma_{\theta,y_{i,1}}$. Since the single components do not interact, the claim in Lemma~\ref{lemerg2} follows.
		\end{proof}
		
		The second lemma establishes the continuity of the equilibrium with respect to $\theta$, its center of drift.
		
		\begin{lemma}{\bf [Continuity of the equilibrium]} 
			\label{lemlip2} 
			Let $\CP(([0,1]^3)^{\N_0})$ denote the space of probability measures on $([0,1]^3)^{\N_0}$. The mapping $[0,1]\times[0,1]^{\N_0} \to \CP(([0,1]^3)^{\N_0})$ given by
			\begin{equation}
				(\theta,{\bf y_1}) \mapsto \nu_{\theta,{\bf y_1}}
			\end{equation}
			is continuous. Furthermore, if $h$ is a Lipschitz function on $[0,1]$, then also $\CF h$ defined by 
			\begin{equation}
				(\CF h)(\theta) =\E^{\nu_{\theta,{\bf y_1}}}[h(\cdot)] = \int_{([0,1]^3)^{\N_0}} \nu_{\theta,{\bf y_1}}(\d z)\,h(x_0) 
			\end{equation}
			is a Lipschitz function on $[0,1]$, whose values are independent of ${\bf y_1}$. 
		\end{lemma}
		
		\begin{proof}
			Lemma~\ref{lemlip2} follows from the proof of Lemma~\ref{unifergod2}.
		\end{proof}
		
		The third lemma characterises the speed at which the estimators $({\Theta}^{[N]}_x(t))_{t\geq 0}$ and $({\Theta}^{[N]}_y(t) )_{t\geq 0}$ converge to each other when $N\to\infty$ and $t\to\infty$.  
		
		\begin{lemma}{\bf [Comparison of empirical averages]}
			\label{lemlev1a}
			Let $(\Theta^{(1),[N]}_x(t))_{t\geq 0}$ and $(\Theta^{(1),[N]}_{y_0}(t))_{t\geq 0}$ be defined as in \eqref{ma6}. Then 
			\begin{equation}
				\label{bnd}
				\begin{aligned}
					\E\left[\left|\Theta^{(1),[N]}_{x}(t)-\Theta^{(1),[N]}_{y_0}(t)\right|\right]
					&\leq \sqrt{\E\left[\left(\Theta^{(1),[N]}_{x}(0)-\Theta^{(1),[N]}_{y_0}(0)\right)^2\right]}\e^{-(K_0e_0+e_0)t}\\ 
					&\quad + \sqrt{\frac{2}{K_0e_0+e_0}\left[\frac{||g||}{N}+\frac{4 K_1e_1}{N}\right]}.
				\end{aligned}
			\end{equation}	
		\end{lemma}
		
		\begin{proof} 
			From \eqref{gh45a2} it follows via Itô-calculus that
			\begin{equation}
				\begin{aligned}
					\frac{\d }{\d t} \E\left[\left(\Theta^{(1),[N]}_{x}(t)-\Theta^{(1),[N]}_{y_0}(t)\right)^2\right]
					&=-2(K_0e_0+e_0)\,\E\left[\left(\Theta^{(1),[N]}_{x}(t)-\Theta^{(1),[N]}_{y_0}(t)\right)^2\right]\\
					&\qquad +h^{[N]}(t),
				\end{aligned}
			\end{equation}
			where
			\begin{equation}
				\begin{aligned}
					h^{[N]}(t) &=\E\left[\frac{2K_1e_1}{N}\left(\Theta^{(1),[N]}_{x}(t)-\Theta^{(1),[N]}_{y_0}(t)\right)
					\left[\Theta_{y_1}^{(1),[N]}(t)-\Theta^{(1),[N]}_{x}(t)\right]\right]\\
					&\quad +\frac{2}{N^2}\sum_{i\in[N]}\E\big[g\big(x^{[N]}_i(r)\big)\big].
				\end{aligned}
			\end{equation}
			Hence
			\begin{equation}
				\begin{aligned}
					\E\left[\left(\Theta^{(1),[N]}_{x}(t)-\Theta^{(1),[N]}_{y_0}(t)\right)^2\right]
					&=\E\left[\left(\Theta^{(1),[N]}_{x}(0)-\Theta^{(1),[N]}_{y_0}(0)\right)^2\right]\e^{-2(K_0e_0+e_0)t}\\
					&\qquad+\int_0^t \d r\,\e^{-2(K_0e_0+e_0)(t-r)} h^{[N]}(r).
				\end{aligned}
			\end{equation}
			Take the square root on both sides and use Jensen's inequality to get \eqref{bnd}.
		\end{proof}
		
		Like for the mean-field system with one colour, we need to compare the finite system in \eqref{e7152} with an infinite system. To derive the analogue of Lemma~\ref{l.comp}, let $L(N)$ satisfy $\lim_{N\to\infty} L(N)=\infty$ and $\lim_{N \to \infty} L(N)/N=0$. Define the measure $\mu_N$ on $([0,1]^3)^{\N_0}$ by continuing the configuration of 
		\begin{equation}
			Z^{[N]}(Ns-L(N))=\left(X^{[N]}(Ns-L(N)),\left(Y_0^{[N]}(Ns-L(N)),Y_1^{[N]}(Ns-L(N))\right)\right)
		\end{equation} 
		periodically to $([0,1]^3)^{\N_0}$. Let  
		\begin{equation}
			\bar{\Theta}^{(1),[N]} =\frac{1}{N}\sum_{i\in[N]}\frac{x^{[N]}_i(Ns-L(N))+K_0y^{[N]}_{i,0}(Ns-L(N))}{1+K_0}.
		\end{equation}
		Let
		\begin{equation}
			\label{t5}
			(Z^{\mu_N}(t))_{t\geq 0}=\bigl(X^{\mu_N}(t),\left(Y_0^{\mu_N}(t),Y_1^{\mu_N}(t)\right)\bigr)_{t \geq 0}
		\end{equation}
		be the infinite system evolving according to
		\begin{equation}
			\label{mgh52b2}
			\begin{aligned}
				\d x^{\mu_N}_i(t) &= c_0\,[\bar{\Theta}^{(1),[N]} - x^{\mu_N}_i(t)]\, \d t + \sqrt{g\big(x^{\mu_N}_i(t)\big)}\, \d w_i (t) 
				+ K_0e_0\,[y^{\mu_N}_{i,0}(t)-x^{\mu_N}_i(t)]\,\d t,\\
				\d y^{\mu_N}_{i,0}(t) &= e_0\, [x^{\mu_N}_i(t)-y^{\mu_N}_{i,0}(t)]\, \d t,\\
				y^{\mu_N}_{i,1}(t) &= y^{\mu_N}_{i,1}(0), \qquad i\in\N_0,
			\end{aligned}
		\end{equation} 
		starting from initial distribution $\mu_N$. Then the following lemma is the equivalent of Lemma~\ref{l.comp} for the two-colour mean-field system.
		
		\begin{lemma}{\bf [Comparison of finite and infinite systems]}
			\label{l.comp2b2}
			Fix $s > 0$, and let $L(N)$ satisfy $\lim_{N\to\infty} L(N)=\infty$ and $\lim_{N\to\infty} L(N)/N$. Suppose that 
			\begin{equation}
				\begin{aligned}
					&\lim_{N\to\infty} \sup_{0 \leq t \leq L(N)} \left|\bar{\Theta}^{(1),[N]}(Ns)
					-\bar{\Theta}^{(1),[N]}(Ns-t)\right| = 0\ \text{ in probability}.
				\end{aligned}
			\end{equation}
			Then, for all $t\geq 0$,
			\begin{equation}
				\label{m322b2}
				\lim_{k\to\infty} \left|\E\left[f\big(Z^{\mu_{N}}(t)\big) -f\big(Z^{[N]}(Ns-L(N)+t)\big)\right]\right| = 0
				\qquad \forall\, f\in\CC\bigl(([0,1]^3)^{\N_0},\R\bigr).
			\end{equation}
		\end{lemma}
		
		\begin{proof}
			We proceed as in the proof of Lemma~\ref{l.comp}. We rewrite the SSDE in \eqref{gh45a2} as
			\begin{equation}
				\begin{aligned}
					\label{m24b}
					&\d x^{[N]}_i(t) = c_0\,\big[\Theta^{(1),[N]}-x^{[N]}_i(t)\big]\,\d t\\
					&\qquad\qquad\qquad+c_0\,\big[\bar{\Theta}^{(1),[N]}(t)  - \Theta^{(1),[N]}\big]\, \d t
					+ c_0\,\big[\Theta^{(1),[N]}_x(t) -\bar{\Theta}^{(1),[N]}(t)\big]\, \d t\\
					&\qquad\qquad\qquad + \sqrt{g\big(x^{[N]}_i(t)\big)}\, \d  w_i (t)\\
					&\qquad\qquad\qquad + K_0 e_0\, \big[y^{[N]}_{i,0}(t)-x^{[N]}_i(t)\big]\d t
					+\frac{K_1e_1}{N} \big[y^{[N]}_{i,1}(t)-x_i^{[N]}(t)\big]\d t,\\
					&\d y^{[N]}_{i,0}(t) = e_0\,\big[x^{[N]}_i(t)-y^{[N]}_{i,0}(t)\big]\, \d t, \\
					&\d y^{[N]}_{i,1}(t) = \frac{e_1}{N}\,\big[x^{[N]}_i(t)-y^{[N]}_{i,1}(t)\big]\, \d t, \qquad i \in [N].
				\end{aligned}
			\end{equation}
			As before, we consider the finite system in \eqref{m24b} as a system on $([0,1]^3)^{\N_0}$ by periodic continuation, and we couple the finite system in \eqref{m24b} and the infinite system in \eqref{m322b2} via there Brownian motions. We denote the coupled process by $\tilde{z}(t)=(\tilde{z}_{i}(t))_{i\in\N_0} = (\tilde{z}^{[N]}_{i}(t),\tilde{z}^{\mu_N}_{i}(t))_{i\in\N_0}$, where $\tilde{z}^{[N]}_{i}(t)=(\tilde x^{[N]}_i(t),\tilde y^{[N]}_{i,0}(t),\tilde y^{[N]}_{i,1}(t))$ and $\tilde{z}^{\mu_N}_i(t)=(\tilde x^{\mu_N}_i(t),\tilde y^{\mu_N}_{i,0}(t),\tilde y^{\mu_N}_{i,1}(t))$. We define
			\begin{equation}
				\begin{aligned}
					\Delta_{i,0}^{[N]}(t)&=\tilde{x}_{i}^{[N]}(t)-\tilde{x}_{i}^{\mu_N}(t),\\
					\delta_{i,0}^{[N]}(t)&=\tilde{y}_{i,0}^{[N]}(t)-\tilde{y}_{i,0}^{\mu_N}(t),\\
					\delta_{i,1}^{[N]}(t)&=\tilde{y}_{i,1}^{[N]}(t)-\tilde{y}_{i,1}^{\mu_N}(t).
				\end{aligned}
			\end{equation}
			As in the proof of Lemma~\ref{l.comp}, we have to show that, for all $t\geq0$,
			\begin{equation}
				\label{m113}
				\lim_{N\to \infty}\E[|\Delta^{[N]}_i(t)|]=0,\qquad \lim_{N\to \infty}\E[|\delta^{[N]}_{i,0}(t)|]=0, 
				\qquad \lim_{N\to \infty}\E[|\delta^{[N]}_{i,1}(t)|]=0.
			\end{equation}
			
			To prove the third limit in \eqref{m112}, note that, by \eqref{mgh52b2}, \eqref{m24b} and the choice of the initial measure in the coupling,
			\begin{equation}
				\label{m112}
				y_{i,1}^{[N]}(t)=y_{i,1}^{[N]}(0)+\frac{e_1}{N}\int_0^t \d r\,\big[x_i^{[N]}(r)-y_{i,1}^{[N]}(r)\big]
				=y_{i,1}^{\mu_N}(t)+\frac{e_1}{N}\int_0^t \d r\,\big[x_i^{[N]}(r)-y_{i,1}^{[N]}(r)\big]. 
			\end{equation}
			Hence
			\begin{equation}
				\label{m1123}
				\lim_{N\to \infty}\E[|\delta^{[N]}_{i,1}(L(N))|]=0.
			\end{equation}
			
			To prove the first two limits in \eqref{m112}, we argue as in the proof of Lemma~\ref{l.comp}, but we need to add extra drift terms towards the first seed-bank. Using It\^o-calculus, we obtain
			\begin{equation}
				\begin{aligned}
					\frac{\d}{\d t} \E[|\Delta_i^{[N]}(t)|+K|\delta^{[N]}_{i,0}(t)|]
					&=-c\,\E[\Delta_i^{[N]}(t)]\\
					&\quad -2K_0e_0\,\E\left[[|\Delta_i^{[N]}(t)|+|\delta^{[N]}_i(t)|]\,
					1_{\{\sign \Delta_i^{[N]}(t)\neq \sign \delta_{i,0}^{[N]}(t)\}}\right]\\
					&\quad + c\, \sign\Delta_i^{[N]}(t) \big[\bar{\Theta}^{(1),[N]}(t)-\bar{\Theta}^{(1),[N]}\big]\\
					&\quad + c\, \sign\Delta_i^{[N]}(t) \big[\bar{\Theta}_x^{(1),[N]}(t)-\bar{\Theta}^{(1),[N]}(t)\big]\\
					&\quad +\frac{K_1e_1}{N}\,\sign\Delta_i^{[N]}(t) \big[\delta_{i,1}^{[N]}(t)-\Delta_i^{[N]}(t)\big].
				\end{aligned}
			\end{equation}
			This can be rewritten as 
			\begin{equation}
				\begin{aligned}
					&0\leq \E[|\Delta_i^{[N]}(t)|+K_0|\delta^{[N]}_{i,0}(t)|]\\
					&\leq \E[|\Delta_i^{[N]}(0)|+K|\delta^{[N]}_{i,0}(0)|] - c \int_0^t \d r\,\E[\Delta_i^{[N]}(r)]\\
					&\quad -2K_0e_0 \int_0^t \d r\,\E\left[[|\Delta_i^{[N]}(r)|+|\delta^{[N]}_{i,0}(r)|]\,
					1_{\{\sign \Delta_i^{[N]}(t)\neq \sign \delta_{i,0}^{[N]}(t)\}}\right]\\
					&\quad + c\, \int_0^t \d r\,|\bar{\Theta}^{(1),[N]}(r)-{\Theta}^{(1),[N]}|\\
					&\quad + c\,\int_0^t \d r\,|\bar{\Theta}_x^{(1),[N]}(r)-\bar{\Theta}^{(1),[N]}(r)|\\
					&\quad +\frac{K_1e_1}{N} \int_0^t \d r\,\left|\delta_{i,1}^{[N]}(r)-\Delta_i^{[N]}(r)\right|.
				\end{aligned}
			\end{equation}
			By the construction of the measure $\mu_N$, we have 
			\begin{equation}
				\lim_{N \to \infty}\E[|\Delta_i^{[N]}(0)|+K_0|\delta^{[N]}_i(0)|]=0.
			\end{equation}
			Therefore, for all $t\geq 0$,
			\begin{equation}
				\label{1097}
				\lim_{N \to \infty}\E[|\Delta_i^{[N]}(t)|+K_0|\delta^{[N]}_{i,0}(t)|]=0.
			\end{equation}
			Combine this with \eqref{m1123} and use that Lipschitz functions are dense in the set of bounded continuous functions. Then, as in the proof of Lemma~\ref{l.comp}, we get the claim in \eqref{m322b2}.
		\end{proof}
		
		Before we can prove that the infinite system $(X^{\mu_N}(t),Y_0^{\mu_N}(t),Y_1^{\mu_N}(t))_{t \geq 0}$ converges to a limiting system as $N\to\infty$, we need the following regularity property for the estimators $(\bar{\Theta}^{[N]},\Theta_{y_1}^{[N]})$.
		
		\begin{lemma}{\bf [Stability of the estimator for the conserved quantity]}
			\label{stabest2} 
			Define $\mu_N$ as in Lemma~\ref{l.comp2b2}. Let $(x^N_i, y^N_{i,0},y^N_{i,1})_{i\in [N]}$ be distributed according to the exchangeable probability measure $\mu_N$ on $([0,1]^3 )^{\N_0}$ restricted to $([0,1]^3)^{[N]}$. Suppose that $\lim_{N\to\infty} \mu_N=\mu$ for some exchangeable probability measure $\mu$ on $([0,1]^3)^{\N_0}$. Define the random variable $\phi$ on $(\mu,([0,1]^3)^{\N_0})$ by putting
			\begin{equation}
				\label{phi2}
				\begin{aligned}
					\phi&=(\phi_1,\phi_2),\\
					\phi_1 &= \lim_{n\to\infty} \frac{1}{n} \sum_{i \in [n]} \frac{x_i+Ky_{i,0}}{1+K},\qquad\phi_2 
					= \lim_{n\to\infty} \frac{1}{n} \sum_{i \in [n]}y_{i,1,} 
				\end{aligned}
			\end{equation}
			and the random variable $\phi^{[N]}$on $(\mu_N,([0,1]^3)^{\N_0})$ by putting
			\begin{equation}
				\begin{aligned}
					\phi^{[N]} &=(\phi_1^{[N]},\phi_2^{[N]})\\
					\phi^{[N]}_1 &= \frac{1}{N}\sum_{i\in[N]} \frac{x^N_i+Ky^N_{i,0}}{1+K},\qquad \phi^{[N]}_2&=\frac{1}{N}\sum_{i\in[N]} y^N_{i,1}.
				\end{aligned}
			\end{equation}
			Then 
			\begin{equation}
				\lim_{N\to\infty}\CL[\phi^{[N]}] = \CL[\phi].
			\end{equation}
		\end{lemma}
		
		\begin{proof}
			We can use a similar argument as in the proof of Lemma~\ref{stabest}. Define 
			\begin{equation}
				D^{[N]}(Z)=\left(\frac{1}{N}\sum_{j\in[N]}\frac{x_j+K_0y_{j,0}}{1+K_0},\frac{1}{N}\sum_{j\in[N]}y_{i,1}\right).
			\end{equation}
			Then we can proceed as in the proof of Lemma~\ref{stabest}, using Fourier analysis for both components of $D^{[N]}(Z)$ separately.
		\end{proof}
		
		In the fifth and final lemma we state the convergence of $\CL[(X^{\mu_N}(t),Y_0^{\mu_N}(t),Y_1^{\mu_N}(t))]$ to the law of a limiting system as $N\to\infty$.  
		
		\begin{lemma}{\bf [Uniformity of the ergodic theorem for the infinite system]} 
			\label{unifergod2} 
			Let $\mu_N$ be defined as in \eqref{t5}. Since $(\mu_N)_{N\in\N}$ is tight, it has convergent subsequences. Let $(N_k)_{k\in\N}$ be a subsequence such that $\mu=\lim_{k\to\infty} \mu_{N_k}$. Define 
			\begin{equation}
				\label{111b}
				\Theta=\lim_{N\to\infty}\frac{1}{N}\sum_{i\in[N]}\frac{x_i^\mu+Ky_{i,0}^{\mu}}{1+K}\, \qquad \text{ in }L_2(\mu).
			\end{equation}
			Let $Z^{\mu}(t)=\bigl(X^{\mu}(t),Y_0^{\mu}(t),Y_1^{\mu}(t)\bigr)_{t\geq 0}$ be the infinite system evolving according to 
			\begin{equation}
				\begin{aligned}
					\label{binfb2}
					&\d x^{\mu}_i(t) = c\left[\Theta -x^{\mu}_i(t)\right]\d t + \sqrt{g(x^{\mu}_i(t))}\, \d w_i (t) 
					+ K e\, [y^{\mu}_{i,1}(t)-x^{\mu}_i(t)]\,\d t,\\
					&\d y^{\mu}_{i,0}(t) = e\,[x^{\mu}_i(t)-y^{\mu}_{i,1}(t)]\, \d t,\\	
					&\d y^{\mu}_{i,1}(t) = y^{\mu}_{i,1}(0), \qquad i \in \N_0.
				\end{aligned}
			\end{equation}
			and let $Z^{\mu_{N_k}}(t)=(X^{\mu_{N_k}}(t),Y_0^{\mu_{N_k}}(t),Y_1^{\mu_{N_k}}(t))_{t\geq 0}$ be the infinite system defined in \eqref{t5}. Then	
			\begin{enumerate}
				\item 
				For all $t \geq 0$,
				\begin{equation}
					\label{t12}
					\begin{aligned}
						\lim_{k\to\infty} \bigl|\E\bigl[f\bigl(Z^{\mu_{N_k}}(t)\bigr)\bigr]-\E\bigl[f\bigl(Z^{\mu}(t)\bigr)\bigr]\bigr| = 0
						\qquad\forall\, f\in\CC\bigl(([0,1]^2)^{\N_0},\R\bigr).
					\end{aligned}
				\end{equation}
				\item 
				There exists a sequence $\bar{L}(N)$ satisfying $\lim_{N \to \infty} \bar{L}(N)=\infty$ and $\lim_{N \to \infty} \bar{L}(N)/N=0$ such that
				\begin{equation}
					\label{m32cb}
					\begin{aligned}
						&\lim_{k\to\infty} \bigl|\E\bigl[f\bigl(Z^{[N_k]}(N_ks-L(N_k)+\bar{L}(N_k))\bigr)
						-f\bigl(Z^{\mu_{N_k}}(\bar{L}(N_k))\bigr)\bigr|\bigr]\\
						&\quad +\bigl|\E\bigl[f\bigl(Z^{\mu_{N_k}}(\bar{L}(N_k))\bigr)\bigr]
						-\E\bigl[f\bigl(Z^{\mu}(\bar{L}(N_k))\bigr)\bigr]\bigr| =  0\qquad \forall\, f\in\CC\bigl(([0,1]^2)^{\N_0},\R\bigr).
					\end{aligned}
				\end{equation}
			\end{enumerate}
		\end{lemma}
		
		\begin{proof}
			As in the proof of Lemma~\ref{unifergod}, we can construct $(z_i^{\mu_N})_{i\in\N_0}$ and $(z_i^{\mu})_{i\in\N_0}$ on one probability space. Then
			\begin{equation}
				\label{t19}
				\lim_{N \to \infty}y^{\mu_N}_{i,1}(0)=y^{\mu}_{i,1}(0)\text{ a.s. }
			\end{equation}
			and 
			\begin{equation}
				\label{t20}
				\lim_{N\to\infty}\E[|\bar{\Theta}^{[N]}-\Theta|]=0.
			\end{equation} 
			Via a similar coupling as in Lemma~\eqref{l.comp2b2}, it follows via It\^o-calculus that \eqref{t12} holds. Combining \eqref{m1123}, \eqref{1097}, \eqref{t19} and \eqref{t20}, we obtain, via a similar construction as in the proof of Lemma~\ref{unifergod}, a sequence $\bar{L}(N)$ such that
			\begin{equation}
				\label{9592}
				\begin{aligned}
					&\lim_{N\to\infty}\E[ |\Delta^{N}_i(\bar{L}(N))|+ K_0|\delta^{N}_{i,0}(\bar{L}(N))|]+ K_1|\delta^{N}_{i,1}(\bar{L}(N))|]\\
					&+\E[ |\Delta^{\mu_N}_i(\bar{L}(N))|+ K_0|\delta^{\mu_N}_{i,0}(\bar{L}(N))|]+K_1|\delta^{\mu_N}_{i,1}(\bar{L}(N))|]=0.
				\end{aligned}
			\end{equation}
			As in the proof of Lemma~\ref{unifergod}, we can again use Lipschitz functions to conclude \eqref{m32cb}.
		\end{proof}
		
		\begin{lemma}{\bf[Coupling of finite systems]}
			\label{lem:12a}
			Let 
			\begin{equation}
				Z^{[N],1}=(X^{[N],1},Y_0^{[N],1},Y_1^{[N],1})
			\end{equation}
			be the finite system evolving according to \eqref{gh45a2} starting from an exchangeable initial measure. Let $\mu^{[N],1}$ be the measure obtained by periodic continuation of the configuration of $Z^{[N],1}(0)$. Similarly, let  
			\begin{equation}
				Z^{[N],2}=(X^{[N],2},Y_0^{[N],2},Y_1^{[N],2})
			\end{equation} 
			be the finite system evolving according to \eqref{gh45a2} starting from an exchangeable initial measure. Let $\mu^{[N],2}$ be the measure obtained by periodic continuation of the configuration of $Z^{[N],2}(0)$. Let $\tilde{\mu}$ be any weak limit point of the sequence of measures $(\mu^{[N],1}\times\mu^{[N],2})_{N\in\N}$. Define the random variables $\bar{\Theta}^{[N],1}$ and $\bar{\Theta}^{[N],2}$ on $(([0,1]^3)^{\N_0}\times ([0,1]^3)^{\N_0},\mu^{[N],1}\times\mu^{[N],2})$ and $\bar{\Theta}_1$ and $\bar{\Theta}_2$ on $(([0,1]^3)^{\N_0}\times([0,1]^3)^{\N_0},\bar{\mu})$ by
			\begin{equation}
				\label{843a}
				\begin{aligned}
					&\bar{\Theta}^{[N],1} = \frac{1}{N} \sum_{i \in [N]} \frac{x^{[N],1}_{i}+K_0y^{[N],1}_{i,0}}{1+K_0},
					\qquad \bar{\Theta}^{[N],2} =  \frac{1}{N} \sum_{i \in [N]} \frac{x^{[N],2}_{i}+K_0y^{[N],2}_{i,0}}{1+K_0},\\
					&\bar{\Theta}^1 = \lim_{n\to\infty} \frac{1}{n} \sum_{i \in [n]} \frac{x^1_{i}+K_0y^1_{i,0}}{1+K_0},
					\qquad \bar{\Theta}^2 = \lim_{n\to\infty} \frac{1}{n} \sum_{i \in [n]} \frac{x^2_{i}+K_0y^2_{i,0}}{1+K_0},
				\end{aligned}
			\end{equation}
			and let $(\bar{\Theta}^{(1),[N],1}(t))_{t\geq 0}$ and $(\bar{\Theta}^{(1),[N],2}(t))_{t\geq 0}$ be defined  as in  \eqref{ma6} for $Z^{[N],1}$, respectively, $Z^{[N],2}$. Suppose that   
			\begin{equation}
				\label{m08a}
				\begin{aligned}
					&\lim_{N\to\infty} \sup_{0 \leq t \leq L(N)} \left(\left|\bar{\Theta}^{[N],k}(0)-\bar{\Theta}^{[N],k}(t)\right|\right) 
					= 0\ \text{ in probability}, \quad k \in\{1,2\},
				\end{aligned}
			\end{equation}
			and suppose that  $\tilde{\mu}(\{\bar{\Theta}_1=\bar{\Theta}_2,\, Y^1_{1}=Y^2_1\})=1$. Then, for any $t(N)\to\infty$,
			\begin{equation}
				\begin{aligned}
					&\lim_{N\to\infty}\E\bigl[|x^{[N],1}_{i}(t(N))-x^{[N],2}_{i}(t(N))|+K_0|y^{[N],1}_{i,0}(t(N))-y^{[N],2}_{i,0}(t(N))|\\
					&\qquad +K_1|y^{[N],1}_{i,1}(t(N))-y^{[N],2}_{i,1}(t(N))|\bigr]=0.
				\end{aligned}
			\end{equation}  
		\end{lemma}
		
		\begin{proof}
			Via standard It\^o-calculus we obtain from \eqref{gh45a2} that
			\begin{equation} 
				\begin{aligned}
					&\frac{\d}{\d t} \E\left[|x^{[N],1}_{i}(t)-x^{[N],2}_{i}(t)|+K_0|y^{[N],1}_{i,0}(t)-y^{[N],2}_{i,0}(t)|
					+K_1|y^{[N],1}_{i,1}(t)-y^{[N],2}_{i,1}(t)|\right]\\
					&=-\frac{2c}{N}\sum_{j \in [N]}\E\left[|x^{[N],1}_{j}(t)-x^{[N],2}_{j}(t)|
					1_{\{\sign(x^{[N],1}_{j}(t)-x^{[N],2}_{j}(t))\neq \sign(x^{[N],1}_{i}(t)-x^{[N],2}_{i}(t))\}}\right]\\
					&\qquad-2K_0e_0\,\E\bigg[|x^{[N],1}_{i}(t)-x^{[N],2}_{i}(t)|+K|y^{[N],1}_{i,0}(t)-y^{[N],2}_{i,0}(t)|\\
					&\qquad\qquad\times1_{\{\sign(x^{[N],1}_{i}(t)-x^{[N],2}_{i}(t))\neq \sign(y^{[N],1}_{i,0}(t)-y^{[N],2}_{i,0}(t))\}}\bigg]\\
					&\qquad-2 \frac{K_1e_1}{N}\,\E\bigg[|x^{[N],1}_{i}(t)-x^{[N],2}_{i}(t)|+K_1|y^{[N],1}_{i,0}(t)-y^{[N],2}_{i,0}(t)|\\
					&\qquad\qquad\times1_{\{\sign(x^{[N],1}_{i}(t)-x^{[N],2}_{i}(t))\neq \sign(y^{[N],1}_{i,1}(t)-y^{[N],2}_{i,1}(t))\}}\bigg].
				\end{aligned}
			\end{equation}	
			Therefore, for all $N\in\N$, 
			\begin{equation}
				\begin{aligned}
					t\mapsto\E\left[|x^{[N],1}_{i}(t)-x^{[N],2}_{i}(t)|+K_0|y^{[N],1}_{i,0}(t)-y^{[N],2}_{i,0}(t)|
					+K_1|y^{[N],1}_{i,1}(t)-y^{[N],2}_{i,1}(t)|\right]
				\end{aligned}
			\end{equation}
			is a decreasing function. Hence we can use the same strategy as in the proof of Lemma~\ref{lem:12} to finish the proof.
		\end{proof}
		
		\paragraph{$\bullet$ Proof of Proposition~\ref{prop2}}
		
		\begin{proof}
			We follow a similar argument as in the proof of Proposition~\ref{prop1}. Let $L(N)$ satisfy $\lim_{N \to \infty}$ 
			$L(N)=\infty$ and $\lim_{N\to\infty} L(N)/N=0$. Let $\mu_{N}$ be the measure on $([0,1]^3)^{\N_0}$  obtained by periodic continuation of $\CL[Z^{[N]}(Ns-L(N))]$.  Note that $([0,1]^3)^{\N_0}$ is compact. Hence, letting $(N_k)_{k\in\N}$ be the subsequence in Proposition~\ref{prop2}, we can pass to a possibly further subsequence and obtain 
			\begin{equation}
				\lim_{k\to\infty} \mu_{N_k} = \mu.
			\end{equation}
			Since we assumed that $\CL[Z^{[N]}(0)]$ is exchangeable and the dynamics preserve exchangeability, the measures $\mu_{N_k}$ are translation invariant and also the limiting law $\mu$ is translation invariant. 
			
			Let $\phi=(\phi_1,\phi_2)$ be defined as in \eqref{phi2} in Lemma~\ref{stabest2}. Then we can condition on $\phi=(\phi_1,\phi_2)$ and write
			\begin{equation}
				\label{1078}
				\mu = \int_{[0,1]^2} \mu_\rho\, \d \Lambda(\rho), 
			\end{equation}
			where $\Lambda(\cdot)=\CL[\phi]=\CL[(\phi_1,\phi_2)]$ and $\rho=(\rho_1,\rho_2)$. By assumption we know that 
			\begin{equation}
				\lim_{k\to\infty} \CL\left[{\bf\Theta}^{\eff,(1),[N_k]}(N_ks)\right] = P_{{\bf\Theta^\eff}(s)}(\cdot)
			\end{equation}
			and
			\begin{equation}
				\begin{aligned}
					&\lim_{k\to\infty} \CL\left[\sup_{0\leq t\leq L(N_k)}\left|\bar{\Theta}^{[N_k]}(N_ks)-\bar{\Theta}^{[N_k]}(N_ks-t)\right|
					+\left|{\Theta}_{y_1}^{[N_k]}(N_ks)-{\Theta_{y_1}}^{[N_k]}(N_ks-t)\right|\right]\\
					&\qquad \qquad \qquad =\delta_0.
				\end{aligned}
			\end{equation}
			Hence
			\begin{equation}
				\begin{aligned}
					\lim_{k\to\infty} \CL\left[{\bf\Theta}^{\eff,(1),[N_k]}(N_ks-L(N_k))\right] = P_{{\bf\Theta^\eff}(s)}(\cdot).
				\end{aligned}
			\end{equation}	
			Recall that
			\begin{equation}
				\Lambda(\cdot) = \CL \left[\lim_{n\to\infty}\left(\frac{1}{n}\sum_{i\in[n]} 
				\frac{x_i+Ky_{i,0}}{1+K},\frac{1}{n}\sum_{i\in[n]}y_{i,1}\right)\right] \quad \text{ on } (\mu,([0,1]^2)^{\N_0}).
			\end{equation}
			By Lemma \ref{stabest}, if 
			\begin{equation}
				\phi^{N_k}=(\phi_1^{N_k},\phi_2^{N_k})=\left(\frac{1}{N_k} \sum_{i\in[N_k]} \frac{x_i+Ky_{i,0}}{1+K},\frac{1}{N_k} \sum_{i\in[N_k]} y^{[N_k]}_{i,1}\right)  \text{ on } (\mu_{N_k},([0,1]^3)^{\N_0}),
			\end{equation} 
			then $\lim_{k\to\infty} \CL[\phi^{N_k}]=\CL[\phi]$. Taking the subsequence $(\mu_{N_k})_{k\in\N}$, we get $\Lambda(\cdot)=P_{{\bf\Theta^\eff}(s)}(\cdot)$, and hence 
			\begin{equation}
				\mu=\int_{[0,1]} \mu_\rho\, \d P_s(\rho).
			\end{equation}
			Let $\bar{L}(N)$ be the sequence constructed in Lemma~\ref{unifergod2}[b]. By construction we can require that $\bar{L}(N)\leq L(N)$ for all $N\in\N$. Write
			\begin{equation}
				\label{triangle2}
				\begin{aligned}
					&\CL\bigl[Z^{[N_k]}(N_ks-L(N_k)+\bar{L}(N_k))\bigr]\\
					&= \CL\bigl[Z^{[N_k]}(N_ks-L(N_k)+\bar{L}(N_k))\bigr]-\CL\bigl[Z^{\mu_{N_k}}(\bar{L}(N_k))\bigr],\\
					&\quad +\CL\bigl[Z^{\mu_{N_k}}(\bar{L}(N_k))\bigr]-\CL\bigl[Z^{\mu}(\bar{L}(N_k))\bigr]\\
					&\quad +\CL\bigl[Z^{\mu}(\bar{L}(N_k))\bigr].
				\end{aligned}
			\end{equation}
			By Lemma \ref{unifergod2} the first and second differences tend to zero as $k\to\infty$. Hence
			\begin{equation}
				\lim_{k\to\infty} \CL\bigl[Z^{[N_k]}(N_ks-L(N_k)+\bar{L}(N_k))\bigr]
				= \CL\bigl[Z^{\mu}(\bar{L}(N_k))\bigr].
			\end{equation}
			By \eqref{1078},
			\begin{equation}
				\label{115}
				\CL\bigl[Z^{\mu}(\bar{L}(N_k))\bigr]
				=\int_{[0,1]^2} \CL\bigl[Z^{\mu_\rho}(\bar{L}(N_k))\bigr]\, P_{{\bf\Theta^\eff}(s)} (\d\rho).
			\end{equation}
			
			For the infinite system $(Z^{\mu_\rho}(t))_{t\geq 0}=\left(X^{\mu_\rho}(t),Y_0^{\mu_\rho}(t),Y_1^{\mu_\rho}(t)\right)_{t\geq 0}$ we have
			\begin{equation}
				Y^\mu_1(t)=Y^\mu_1(0)\, a.s.
			\end{equation}
			and hence, since $\lim_{k\to\infty} \bar{L}(N_k)/N_k = 0$ by \eqref{112},
			\begin{equation}
				\lim_{k\to\infty}\CL[Y_1^{\mu_\rho}(\bar{L}(N_k))]=\CL[Y_1^{\mu_\rho}(0)]\,\, \forall\, \rho\in[0,1].
			\end{equation}
			Therefore 
			\begin{equation}
				\lim_{k\to\infty}\CL[Y_1^{\mu_\rho}(\bar{L}(N_k))]=P_{Y_1(s)}^\rho(\cdot)
			\end{equation}	
			and
			\begin{equation}
				\begin{aligned}
					&\CL\bigl[X^{\mu_\rho}(\bar{L}(N_k)),Y_0^{\mu_\rho}(\bar{L}(N_k)),Y_1^{\mu_\rho}(\bar{L}(N_k))\bigr]\\ 
					&\qquad = \int \CL\bigl[X_1^{\mu_\rho}(\bar{L}(N_k)),Y_0^{\mu_\rho}(\bar{L}(N_k)),{\bf y_1}\bigr]
					\d P_{Y_1(s)}^\rho(\d {\bf y_1}).
				\end{aligned} 
			\end{equation}
			Hence, since $\lim_{k \to \infty}\bar{L}(N_k)=\infty$, by Lemma \ref{lemerg} we have
			\begin{equation}
				\begin{aligned}
					\lim_{k\to\infty} \CL\bigl[Z^{\mu_\rho}(\bar{L}(N_k))\bigr] 
					&= \lim_{k\to\infty} \CL\bigl[X_1^{\mu_\rho}(\bar{L}(N_k)),Y_0^{\mu_\rho}(\bar{L}(N_k)),
					Y_1^{\mu_\rho}(\bar{L}(N_k))\bigr]\\
					& = \int \nu_{\rho,{\bf y_1}} P_{Y_1(s)}^\rho(\d {\bf y_1}).
				\end{aligned}
			\end{equation}
			Therefore, by \eqref{triangle}, \eqref{115} and Lemma \ref{lemlip},
			\begin{equation}
				\begin{aligned}
					&\lim_{k\to\infty} \CL\bigl[Z^{[N_k]}(N_ks-L(N_k)+\bar{L}(N_k))\bigr]
					= \int_{[0,1]} P_{{\bf\Theta^\eff}(s)} (\d \rho)\int \nu_{\rho,{\bf y_1}} P_{Y_1(s)}^\rho(\d {\bf y_1}).
				\end{aligned}
			\end{equation}
			To finish the proof, we proceed as in the proof of Proposition~\ref{prop1} and invoke Lemma~\ref{lem:12a}. Let  $Z^{[N],1}=(X^{[N],1},Y_0^{[N],1},Y_1^{[N],1})$ be the finite system starting from 
			\begin{equation}
				\CL\bigl[Z^{[N]}(Ns-L(N))\bigr]=\CL\bigl[X^{[N]}(Ns-L(N)),Y_0^{[N]}(Ns-L(N)),Y_1^{[N]}(Ns-L(N))\bigr].
			\end{equation}
			Let $(\bar{L}(N))_{N\in\N}$ be the sequence constructed in Lemma~\ref{unifergod2}. Let $Z^{[N],2}=(X^{[N],2},Y_0^{[N],2},Y_1^{[N],2})$ be the finite system starting from 
			\begin{equation}
				\CL\bigl[X^{[N]}(Ns-\bar{L}(N))\bigr]=\CL\left[X^{[N]}(Ns-\bar{L}(N)),Y_0^{[N]}(Ns-\bar{L}(N)),
				Y_1^{[N]}(Ns-\bar{L}(N))\right].
			\end{equation}
			Choose for $t(N)$ in Lemma~\ref{lem:12a} the sequence $\bar{L}(N)$. Let $\mu^{[N],1}$ be defined by the periodic continuation of the configuration of $Z^{[N]}(Ns-L(N))$ and $\mu^{[N],2}$ be defined by periodic continuation of the configuration of $Z^{[N]}(Ns-\bar{L}(N))$. Define $\Theta_1$ and $\Theta_2$ according to \eqref{843}, where under $\mu^{[N],2}$ we replace $L(N)$ by $\bar{L}(N)$. Then, by the assumptions in \eqref{112}, 
			\begin{equation}
				\begin{aligned}
					\lim_{k\to\infty}|\Theta^{(1),[N_k],1}-\Theta^{(1),[N_k],2}|
					&=\lim_{k\to\infty}|\Theta^{N_k}(N_ks-L(N_k))-\Theta^{N_k}(N_ks-\bar{L}(N_k))|\\
					&=0\text{ in probability}. 
				\end{aligned}
			\end{equation}
			Using \ref{m112} we see that also, for all $i\in[N]$,
			\begin{equation}
				\begin{aligned}
					\lim_{k\to\infty}|y^{[N_k],1}_{i,1}(0)-y^{[N_k],2}_{i,1}(0)|
					&=\lim_{k\to\infty}|y_{i,1}^{[N_k]}(N_ks-L(N_k))-y_{i,1}^{[N_k]}(N_ks-\bar{L}(N_k))|\\
					&=0\text{ in probability}.
				\end{aligned} 
			\end{equation}
			Therefore, if $\mu$ is any weak limit point of the sequence $\left(\mu^{[N_k],1}\times\mu^{[N_k],2}\right)_{k\in\N}$, then
			\begin{equation}
				\mu(\{\Theta_1=\Theta_2,\, Y^1_1=Y^2_1\})=1. 
			\end{equation}
			Hence, by possibly passing to a further subsequence, we can now apply Lemma~\ref{lem:12a} to obtain, for all $i$,
			\begin{equation}
				\label{ca12}
				\begin{aligned}
					\lim_{k\to\infty}\E\Bigl[&|x^{[N_k],1}_{i}(\bar{L}(N_k))-x^{[N_k],2}_{i}(\bar{L}(N_k))|\\
					&\qquad +K_0\,|y_{i,0}^{[N_k],1}(\bar{L}(N_k))-y_{i,0}^{[N_k],2}(\bar{L}(N_k))|\\
					&\qquad +K_1\,|y_{i,1}^{[N_k],1}(\bar{L}(N_k))-y_{i,1}^{[N_k],2}(\bar{L}(N_k))|\Bigr]=0.
				\end{aligned}
			\end{equation}
			Hence 
			\begin{equation}
				\lim_{N\to\infty} \Big(\CL[Z^{[N],1}(\bar{L}(N_k))]-\CL[Z^{[N],2}(\bar{L}(N_k))]\Big) = \delta_0
			\end{equation}
			and therefore
			\begin{equation}
				\lim_{k\to\infty} \CL\bigl(Z^{[N_k]}(N_ks)\bigr) =  \int_{[0,1]} P_{{\bf\Theta^\eff}(s)} (\d \rho)\int \nu_{\rho,{\bf y_1}} P_{Y_1(s)}^\rho(\d {\bf y_1}).
			\end{equation}
			This concludes the proof of Proposition~\ref{prop2}.
		\end{proof}
		
		Like for the one-colour mean-field system, Proposition~\ref{prop2} and Lemmas~\ref{lemerg2}--\ref{lem:12a} give rise to the following corollary, which will be important to derive the evolution of the $1$-blocks on time scale $Ns$.
		
		\begin{corollary}
			\label{cor2}
			Fix $s>0$. Let $\mu_{N}$ be the measure obtained by periodic continuation of 
			\begin{equation}
				Z^{[N]}(Ns-L(N))=(X^{[N]}(Ns-L(N)),Y_0^{[N]}(Ns-L(N)),Y_1^{[N]}(Ns-L(N))),
			\end{equation} 
			and let $\mu$ be a weak limit point of the sequence $(\mu_N)_{N\in\N}$. Let
			\begin{equation}
				\label{1a2}
				\Theta=\lim_{N\to\infty}\frac{1}{N}\sum_{i\in[N]}\frac{x_i^\mu+Ky_i^{\mu}}{1+K}\, \qquad \text{ in }L_2(\mu),
			\end{equation}
			and let $(Z^{\nu_\Theta}(t))_{t>0}=(X^{\nu_\Theta}(t),Y_0^{\nu_\Theta}(t),Y_1^{\nu_\Theta}(t))_{t>0}$ be the infinite system evolving according to \eqref{binfb2} starting from its equilibrium measure. Consider the finite system $Z^{[N]}$ as a system on $([0,1]^3)^{\N_0}$ by periodic continuation. Construct $(Z^{[N]}(t))_{t> 0}$ and  $(Z^{\nu_\Theta}(t))_{t>0}$ on one probability space. Then, for all $t\geq 0$,
			\begin{equation}
				\label{932s}
				\begin{aligned}
					\lim_{N \to \infty}&\E\left[\left|x_i^{[N]}(Ns+t)-x_i^{\nu_\Theta}(t)\right|\right]
					+K_0\,\E\left[\left|y_{i,0}^{[N]}(Ns+t)-y_{i,0}^{\nu_\Theta}(t)\right|\right]\\
					&\qquad+K_1\,\E\left[\left|y_{i,1}^{[N]}(Ns+t)-y_{i,1}^{\nu_\Theta}(t)\right|\right]=0 \qquad\forall\, i\in[N].
				\end{aligned}
			\end{equation}
		\end{corollary}
		
		\begin{proof}
			Proceed as in the proof of Corollary~\ref{cor1}, but use the setup of the two-colour mean-field system and therefore replace Proposition~\ref{prop1}, Lemma~\ref{unifergod} and Lemma~\ref{lem:12} by, respectively Proposition~\ref{prop2}, Lemma~\ref{unifergod2} and Lemma~\ref{lem:12a}.
		\end{proof}
		
		
		\paragraph{Step 4: Limiting evolution of the $1$-blocks.}
		
		\begin{lemma}{{\bf [Limiting evolution of the 1-blocks]}}
			\label{lemlimev}
			Let $(z_1^\eff(s))_{s>0}$ be the process defined in \eqref{m64} with initial state
			\begin{equation}
				z_1^\eff(0)=(\vt_0,\theta_{y_1}).
			\end{equation}
			Then 
			\begin{equation}
				\label{mgh6b}
				\lim_{N\to \infty}\CL\left[\big({\bf\Theta}^{\eff,(1),[N]}(Ns)\big)_{s>0}\right]=\CL\left[(z_1^\eff(s))_{s>0}\right].
			\end{equation}	
		\end{lemma}
		
		\begin{proof}
			By \cite{YW71}, the SSDE in \eqref{m64} has a unique strong solution. Therefore the process $(z^\eff_1(s))_{s>0}$ is Markov. Its generator $G$ is given by
			\begin{equation}
				G =\frac{K_1e_1}{1+K_0}\,(y-x)\,\frac{\partial }{\partial x}+e_1(x-y)\,\frac{\partial }{\partial y} 
				+\frac{1}{(1+K_0)^2}\,(\CF g)(x)\,\frac{\partial^2 }{\partial x^2},
			\end{equation}
			and hence $(z_1^\eff(s))_{s\geq0}$ solves the martingale problem for $G$. We will use \cite[Theorem 3.3.1]{JM86}, to prove that \eqref{mgh6b} holds. 
			
			Define
			\begin{equation}
				(\vt_0^N,\vt_{y_1}^N)=\left(\bar{\Theta}^{(1),[N]}(0),\Theta^{(1),[N]}_{y_1}(0)\right).
			\end{equation}	
			Since we start from an i.i.d.\ law, by the law of large numbers we have that
			\begin{equation}
				\lim_{N\to\infty} {\bf\Theta}^{\eff, (1),[N]}(0) = \lim_{N \to \infty} (\vt_0^N,\vt_{y_1}^N) 
				= \left(\vartheta_0,\theta_{y_1}\right)\qquad a.s.
			\end{equation}
			By the SSDE in \eqref{ma21} and an optional sampling argument, we have, for all $N\in\N$,
			\begin{equation}
				\lim_{s\downarrow 0}\left(\bar{\Theta}^{(1),[N]}(Ns),\Theta^{(1),[N]}(Ns)\right)=(\vt_0^N,\vt_{y_1}^N) \quad  \text{ a.s.}
			\end{equation} 
			Therefore we can continuously extend the process $({\bf \Theta}^{\eff,(1),[N]}(Ns))_{s>0}$ to $0$ and, in particular,
			\begin{equation}
				\lim_{N\to \infty}\CL\left[{\bf\Theta}^{\eff,(1),[N]}(0)\right]=\CL\left[z_1^\eff(0)\right].
			\end{equation}
			Since we already showed that the processes
			\begin{equation}
				\big({\bf \Theta}^{\eff,(1),[N]}(Ns)\big)_{s>0}
			\end{equation} 
			are $\CD$-semimartingales, and are trivially bounded, we are left to show that 
			\begin{equation}
				\begin{aligned}
					\lim_{N\to \infty}\int_0^s \d r\,
					\E\Big[\Big|G^{(1),[N]}_\dagger\big(f,{\bf\Theta}^{\eff,(1),[N]}(Nr),r,\cdot\big) 
					-(Gf)\big({\bf\Theta}^{\eff,(1),[N]}(Nr)\big)\Big|\Big]=0.
				\end{aligned}
			\end{equation}
			Here, $G^{(1),[N]}_\dagger$ is the operator defined in \eqref{m45a}. Since we are working on the space $\CC^*$ of polynomials on $[0,1]^2$, all derivatives of $f\in\CC^*$ are bounded. Hence, by dominated convergence, it is enough to prove that, for all $s>0$,
			\begin{equation}
				\label{m50b}
				\begin{aligned}
					\lim_{N\to \infty} \E^{[N]}\Big[\Big|G^{(1),[N]}_\dagger\big(f,{\bf\Theta}^{\eff,(1),[N]}(Ns),s,\cdot\big) 
					-(Gf)\big({\bf\Theta}^{\eff,(1),[N]}(Ns)\big)\Big|\Big] = 0.
				\end{aligned}
			\end{equation}
			Note that
			\begin{equation}
				\label{m51b}
				\begin{aligned}
					&\E\Big[\Big|G^{(1),[N]}_\dagger\big(f,{\bf\Theta}^{\eff,(1),[N]}(Ns),s,\cdot\big)
					-(Gf)\big({\bf\Theta}^{\eff,(1),[N]}(Ns)\big)\Big|\Big]\\
					&= \E\Bigg[\Bigg|\frac{K_1e_1}{1+K_0}\,\Big[\Theta^{(1),[N]}_{y_1}(Ns)
					-\frac{1}{N}\sum_{i\in[N]} x_i(Ns,\omega)\Big]\,\frac{\partial f}{\partial x}
					\big({\bf\Theta}^{\eff,(1),[N]}(Ns)\big)\\
					&\qquad+e_1\Big[\frac{1}{N}\sum_{i\in[N]} x_i(Ns,\omega)-\Theta^{(1),[N]}_{y_1}(Ns)\Big]
					\frac{\partial f}{\partial y}\big({\bf\Theta}^{\eff,(1),[N]}(Ns)\big)\\
					&\qquad+\frac{1}{(1+K_0)^2} \frac{1}{N}\sum_{i\in[N]} g(x_i(Ns,\omega))\,\frac{\partial^2 f}{\partial x^2}
					\big({\bf\Theta}^{\eff,(1),[N]}(Ns)\big)\\
					&\qquad-\frac{K_1e_1}{1+K_0} \big[\Theta^{(1),[N]}_{y_1}(Ns)-\bar{\Theta}^{(1),[N]}(Ns)\big]\,
					\frac{\partial f}{\partial x}\big({\bf\Theta}^{\eff,(1),[N]}(Ns)\big)\\
					&\qquad-e_1\big[\bar{\Theta}^{(1),[N]}(Ns)-\Theta^{(1),[N]}_{y_1}(Ns)\big]\,\frac{\partial f}{\partial y}
					\big({\bf\Theta}^{\eff,(1),[N]}(Ns)\big)\\
					&\qquad-\frac{1}{(1+K_0)^2} (\CF g)(\bar{\Theta}^{(1),[N]}(Ns))\,\frac{\partial^2 f}{\partial x^2}
					\big({\bf\Theta}^{\eff,(1),[N]}(Ns)\big)\Bigg|\Bigg].
				\end{aligned}
			\end{equation}
			Hence
			\small
			\begin{equation}
				\label{m52b}
				\begin{aligned}
					&\lim_{N\to \infty}\E\left[\Big|G^{(1),[N]}_\dagger\big(f,{\bf\Theta}^{\eff,(1),[N]}(Ns),s,\cdot\big)
					-(Gf)\big({\bf\Theta}^{\eff,(1),[N]}(Ns)\big)\Big|\right]\\
					&\leq \lim_{N\to \infty}\E\left[\frac{K_1e_1}{1+K_0}\,\left|\bar{\Theta}^{(1),[N]}(Ns)
					-\frac{1}{N}\sum_{i\in[N]} x_i(Ns,\omega)\right|\,\left|\frac{\partial f}{\partial x}\big({\bf\Theta}^{\eff,(1),[N]}(Ns)\big)\right|\right]\\
					&+\lim_{N\to \infty}\E\left[e_1\left|\frac{1}{N}\sum_{i\in[N]} x_i(Ns,\omega)
					-\bar{\Theta}^{(1),[N]}(Ns)\right|\,\left|\frac{\partial f}{\partial y}\big({\bf\Theta}^{\eff,(1),[N]}(Ns)\big)\right|\right]\\
					&+\lim_{N\to \infty}\E\left[\frac{1}{(1+K_0)^2}\left| \frac{1}{N}\sum_{i\in[N]}g(x_i(Ns,\omega))
					-(\CF g)\big(\bar{\Theta}^{(1),[N]}(Ns)\big)\right|\, \left|\frac{\partial^2 f}{\partial x^2}
					\big({\bf\Theta}^{\eff,(1),[N]}(Ns)\big)\right|\right].
				\end{aligned}
			\end{equation}
			\normalsize
			Note that each of the derivatives is bounded by aconstant because we work on $\CC^*$. The first and the second term tend to zero by Lemma~\ref{lemlev1a}. For the third term we can use a similar argument as used in \eqref{a5}, since we showed Lemmas~\ref{lemerg2}--\ref{lem:12a} for the single components in the mean-field system with two colours.
		\end{proof}
		
		
		\paragraph{Step 5: Evolution of the averages in the Meyer-Zheng topology.}
		
		In this section we prove the following proposition
		
		\begin{proposition}{\bf [Convergence in the Meyer-Zheng topology]}
			\label{p.estima}
			Suppose that the effective estimator process defined in \eqref{1013} satisfies
			\begin{equation}
				\label{m1a2}
				\lim_{N\to\infty} \CL \left[\left({\bf\Theta}^{\eff,(1),[N]}(Ns)\right)_{s > 0}\right] 
				= \CL \left[\left(z_1^\eff(s)\right)_{s > 0}\right].
			\end{equation}
			Then for the averages in \eqref{gh412}, 
			\begin{equation}
				\label{gh422}
				\begin{aligned}
					&\lim_{N\to\infty} \CL \left[\left(z_1^{[N]}(s)\right)_{s > 0}\right] 
					= \CL \left[\left(z_1^{}(s)\right)_{s > 0}\right] \\
					&\text{in the Meyer-Zheng topology},
				\end{aligned}
			\end{equation}
			where the limiting process $\left(z_1^{}(s)\right)_{s > 0}$ is defined as in \eqref{gh432}.
		\end{proposition}
		
		To prove Proposition~\ref{p.estima} we need the following characterisation of continuous functions in the Meyer-Zheng topology
		
		\begin{lemma}{\bf[Convergence of marginals in the Meyer-Zheng topology]}
			\label{lem95} 
			Let $(E,d)$ be a Polish space with metric $d$. Suppose that $(X_n(s),Y_n(s))_{s>0}$ is a stochastic process with state space $E^2$. If 
			\begin{equation}
				\lim_{n\to \infty}\CL\left[(X_n(s),Y_n(s))_{s>0}\right]=\CL\left[(X(s),Y(s))_{s>0}\right]\text{ in the Meyer-Zheng topology},
			\end{equation} 
			then the marginals also converge in the Meyer-Zheng topology, i.e.,
			\begin{equation}
				\begin{aligned}
					&\lim_{n\to \infty}\CL\left[(X_n(s))_{s>0}\right]=\CL\left[(X(s))_{s>0}\right]\text{ in the Meyer-Zheng topology},\\
					&\lim_{n\to \infty}\CL\left[(Y_n(s))_{s>0}\right]=\CL\left[(Y(s))_{s>0}\right]\text{ in the Meyer-Zheng topology}.
				\end{aligned}
			\end{equation}
		\end{lemma}
		
		\noindent
		The proof of Lemma~\ref{lem95} is given in Appendix~\ref{apb3}.
		
		\begin{proof}[Proof of Proposition~\ref{p.estima}]
			By Lemma~\ref{lemlev1a}, we have that, for all $s>0$,
			\begin{equation}
				\lim_{n\to\infty}\E\left[\left|\bar{ \Theta}^{[N]}(Ns)-x_1^{[N]}(s)\right|\right]=0
			\end{equation}
			and
			\begin{equation}
				\lim_{n\to\infty}\E\left[\left|\bar{ \Theta}^{[N]}(Ns)-y_{0,1}^{[N]}(s)\right|\right]=0.
			\end{equation}
			Applying Lemmas~\ref{lem91}, \ref{lem93} and \ref{lem94}, like in the proof of Proposition~\ref{p.estim}, we obtain
			\begin{equation}
				\begin{aligned}
					&\lim_{N\to \infty}\CL\left[\left(x_1^{[N]}(s),y_{0,1}^{[N]}(s),\bar{ \Theta}^{[N]}(Ns),
					\Theta^{[N]}_{y_{1,1}}(Ns)\right)_{s>0}\right]\\
					&=\CL\left[\left(x^\eff_1(s),x^\eff_1(s),x^\eff_1(s),y_1^\eff(s)\right)_{s>0}\right]
					\text{ in the Meyer-Zheng topology.}
				\end{aligned}
			\end{equation}
			Applying Lemma~\ref{lem95}, we get the claim.
		\end{proof}
		
		\paragraph{Step 6: Proof of the two-colour mean-field finite-systems scheme.}
		
		\begin{proof}
			The proof of Proposition~\ref{P.finsysmf2}(a) follows directly from Lemma~\ref{lemlimev}. The proof of Proposition~\ref{P.finsysmf2}(b) is a consequence of Proposition~\ref{P.finsysmf2}(d). The proof of Prosition~\ref{P.finsysmf2}(c) follows from Proposition~\ref{P.finsysmf2}(a) by applying Proposition~\ref{p.estima}. The proof of Proposition \eqref{P.finsysmf2}(d) follows by the same argument as used in the proof Proposition~\ref{P.finsysmf}(c) in Section~\ref{step4}. In this argument we have to replace the two-component system $Z^{[N]}(Ns+t)=(X^{[N]}(Ns+t),Y^{[N]}(Ns+t))_{t\geq 0}$ by the three-component system $Z^{[N]}(Ns+t)=(X^{[N]}(Ns+t),Y_0^{[N]}(Ns+t),Y_1^{[N]}(Ns+t))_{t\geq 0}$ and use the infinite system defined in \ref{gh45a2binf} instead of the infinite system defined in \eqref{gh5inf}. We now use the two-dimensional transition kernel in \eqref{ma3}, which controls the transition probabilities of the two-dimensional process $(\bar{\Theta}^{(1)}(s), \Theta_{y_1}^{(1)}(s))_{s>0}$, instead of the one-dimensional transition kernel in \eqref{ma4}.
		\end{proof}
		
		\section{Preparation: $N\to\infty$, two-level three-colour}
		\label{ss.tlhmfs}
		
		To get a proper understanding of how the migration comes into play on different space-time scales, we next look at a two-level mean-field system where the geographic space consists of two layers and the seed-bank consist of three layers, corresponding three colours $0$, $1$, $2$. In Section~\ref{ss.tlhmfs*} we give the set-up of the two-level three-colour mean-field model.
		In Section~\ref{ss.org22} we give a scheme to prove the analysis of the two-level three-colour mean-field model. Finally, in Section~\ref{pmfs23} we prove the steps of the scheme given in Section~\ref{ss.org22}.
		
		\subsection{Two-level three-colour mean-field finite-systems scheme}
		\label{ss.tlhmfs*}
		
		We consider a restricted version of the SSDE in \eqref{moSDE} on the finite geographic space
		\begin{equation}
			[N^2]=\{0,1,\ldots,N^2-1\}, \qquad N \in \N.
		\end{equation} 
		This space should be interpreted as grouping the $N$-blocks consisting of $N$ colonies together, i.e.
		\begin{equation}
			[N^2]=\bigcup_{l=0}^{N-1}\{Nl,Nl+1,\cdots,Nl+N-1\}.
		\end{equation} 
		With this interpretation we can use the metric $d_{[N^2]}$ that is induced by the metric $d_{\Omega_N}$ on the hierarchical group $\Omega_N$ (recall \eqref{ultra}). The migration kernel $a^{\Omega_N}(\cdot,\cdot)$ is restricted to $[N^2]$ by setting all migration rates outside the 2-block equal to  $0$, i.e., $c_k=0$ for all $k\geq 2$. Hence the migration kernel is given by
		\begin{equation}
			a^{[N^2]}(i,j)= 1_{\{d_{[N^2]}(i,j)\leq 1\}}\frac{c_0}{N}+\frac{c_1}{N^3},
		\end{equation}
		where $c_0,c_1\in (0,\infty)$ are constants. The seed-bank of the restricted system consists of \emph{three colours}, labeled $0$   $1$ and $2$, with exchange rates given by $K_0e_0,e_0$,  $\frac{K_1e_1}{N},\frac{e_1}{N}$ and  $\frac{K_2e_2}{N^2},\frac{e_2}{N^2}$ respectively. The state space of the restricted system is
		\begin{equation}
			S = \mathfrak{s}^{[N^2]}, \qquad \mathfrak{s} = [0,1]\times [0,1]^3,
		\end{equation}
		and the restricted system is denoted by
		\begin{equation}
			\label{e7152twee}
			\begin{aligned}
				(Z^{[N^2]}(t))_{t\geq0}&=\Big(X^{[N^2]}(t),\left(Y_0^{[N^2]}(t),Y_1^{[N^2]}(t),Y_2^{[N^2]}(t)\right)\Big)_{t \geq 0},\\
				\Big(X^{[N^2]}(t),\left(Y_0^{[N^2]}(t),Y_1^{[N^2]}(t),Y_2^{[N^2]}(t)\right)\Big) 
				&= \Big(x^{[N^2]}_i(t), \left(y^{[N^2]}_{i,0}(t),y^{[N^2]}_{i,1}(t),y^{[N^2]}_{i,2}(t)\right)\Big)_{i \in [N^2]}.
			\end{aligned}
		\end{equation}
		The components of the restricted system $(Z^{[N^2]}(t))_{t\geq0}$ evolve according to the SSDE
		\begin{equation}
			\label{gh45a2twee}
			\begin{aligned}
				&\d x^{[N^2]}_i(t) = \frac{c_0}{N} \sum_{j \in [N^2]} 1_{\{d_{[N^2]}(i,j)\leq 1\}}[x^{[N^2]}_j(t) - x^{[N^2]}_i(t)]\, \d t\\
				&\qquad\qquad\qquad + \frac{c_1}{N^3} \sum_{j \in [N^2]} [x^{[N^2]}_j(t) - x^{[N^2]}_i(t)]\, \d t 
				+ \sqrt{g(x^{[N^2]}_i(t))}\, \d w_i (t)\\
				&\qquad\qquad\qquad + K_0 e_0\, [y^{[N^2]}_{i,0}(t)-x^{[N^2]}_{i}(t)]\,\d t\\
				&\qquad\qquad\qquad + \frac{K_1 e_1}{N}\, [y^{[N^2]}_{i,1}(t)-x^{[N^2]}_{i}(t)]\,\d t\\
				&\qquad\qquad\qquad + \frac{K_2 e_2}{N^2}\, [y^{[N^2]}_{i,2}(t)-x^{[N^2]}_{i}(t)]\,\d t,\\
				&\d y^{[N^2]}_{i,0}(t) = e_0\,[x^{[N^2]}_i(t)-y^{[N^2]}_{i,0}(t)]\, \d t,\\
				&\d y^{[N^2]}_{i,1}(t) = \frac{e_1}{N}\,[x^{[N^2]}_i(t)-y^{[N^2]}_{i,1}(t)]\, \d t, \\
				&\d y^{[N^2]}_{i,2}(t) = \frac{e_2}{N^2}\,[x^{[N^2]}_i(t)-y^{[N^2]}_{i,2}(t)]\, \d t, \qquad i \in [N^2],
			\end{aligned}
		\end{equation}
		which is a special case of \eqref{moSDE}. By \cite[Theorem 3.1]{SS80}, the SSDE in \eqref{gh45a2twee} is the unique solution. It is important to note that we can write the SSDE also
		\begin{equation}
			\label{gh45b2twee}
			\begin{aligned}
				&\d x^{[N^2]}_i(t) = c_0\left[\frac{1}{N} \sum_{j \in [N]_i}x^{[N^2]}_j(t) - x^{[N^2]}_i(t)\right]\, \d t\\
				&\qquad\qquad\qquad + \frac{c_1}{N}\left[\frac{1}{N^2} \sum_{j \in [N^2]}x^{[N^2]}_j(t) - x^{[N^2]}_i(t)\right]\,\d t 
				+ \sqrt{g(x^{[N^2]}_i(t))}\, \d w_i (t)\\
				&\qquad\qquad\qquad + K_0 e_0\, [y^{[N^2]}_{i,0}(t)-x^{[N^2]}_{i}(t)]\,\d t
				+ \frac{K_1 e_1}{N}\, [y^{[N^2]}_{i,1}(t)-x^{[N^2]}_{i}(t)]\,\d t\\
				&\qquad\qquad\qquad
				+ \frac{K_2 e_2}{N^2}\, [y^{[N^2]}_{i,2}(t)-x^{[N^2]}_{i}(t)]\,\d t,\\
				&\d y^{[N^2]}_{i,0}(t) = e_0\,[x^{[N^2]}_i(t)-y^{[N^2]}_{i,0}(t)]\, \d t,\\
				&\d y^{[N^2]}_{i,1}(t) = \frac{e_1}{N}\,[x^{[N^2]}_i(t)-y^{[N^2]}_{i,1}(t)]\, \d t,\\
				&\d y^{[N^2]}_{i,2}(t) = \frac{e_2}{N^2}\,[x^{[N^2]}_i(t)-y^{[N^2]}_{i,2}(t)]\, \d t, \qquad i \in [N^2],
			\end{aligned}
		\end{equation}
		where $[N]_i$ denotes the set of colonies in the 1-block around site $i\in[N^2]$. Therefore the migration term for a single colony in the two-level mean-field system can be interpreted as a drift towards the $1$-block average of the active population at rate $c_0$ and a drift towards the $2$-block average of the active population at rate $\frac{c_1}{N}$. We are interested in \eqref{gh45b2twee} on time scales $N^0$, $N$ and $N^2$. On time scale $N^0$ we will look at the single colonies, i.e., space-time scale $0$. On time scale $N$  we will look at the $1$-block averages, i.e., space-time scale $1$ and on time scale $N^2$ we will look at the $2$-block averages, i.e., space-time scale $2$.  In the sequel we will focus on site $0$, the 1-block around site $0$ and the 2-block around site $0$. We will suppress this site from the notation, but instead use subscripts $0$, $1$, $2$ to indicate when we look at a single colony, a 1-block average or a 2-block average. We will use the convention that in the subscript of a dormant population the first subscript denotes the colour and the second subscript denotes the level of the block, so $y_{0,1}$ is the $1$-block average around site $0$ of the dormant population with colour $0$, while $y_{1,0}$ is the $1$-dormant single colony at site $0$. Heuristically, we can read off the following results from the SSDE in \eqref{gh45b2twee}.
		
		\medskip\noindent
		$\bullet$ \textbf{On time scale $1=N^0$} (i.e., space-time scale $0$) in the limit as $N\to\infty$, the  colour-$1$ dormant population and the colour-$2$ dormant population do not yet move. Hence
		\begin{equation}
			\big(y^{[N^2]}_{1,0}(t_0),y^{[N^2]}_{2,0}(t_0)\big)_{t_0\geq0},
		\end{equation}
		converges as $N\to\infty$ to the constant processes on time scale $t_0$. Therefore the colour $1$-dormant population and the colour $2$-dormant population are both slow seed-banks on space-time scale $0$. The components $((x_0^{[N^2]}(t_0),y_{0,0}^{[N^2]}(t_0)))_{t_0 \geq 0}$ converge to i.i.d.\ copies of the single-colony McKean-Vlasov process in \eqref{SC}, where in the corresponding SSDE the parameters $e,K,c$ are replaced by $c_0,e_0,K_0$ and $E=1$. So, on time scale $1$ we only see the colour $0$-dormant population evolve. Therefore the colour-$0$ dormant population is the \emph{effective seed-bank} on time scale $t_0$. The process
		\begin{equation}
			\label{ef}
			(z^{[N^2]}_0(t_0))_{t_0\geq 0}=(x^{[N^2]}_0(t_0),y^{[N^2]}_{0,0}(t_0))_{t_0\geq 0},
		\end{equation}
		will be called \emph{the single colony effective process.}
		
		\medskip\noindent
		$\bullet$ \textbf{On time scale $N^1$} (i.e., space-time scale $1$), we look at the averages 
		\small
		\begin{equation}
			\label{gh412twee}
			\begin{aligned}
				&(z_1^{[N^2]}(t_1))_{t_1>0}=\left(x_1^{[N^2]}(t_1),\left( y_{0,1}^{[N^2]}(t_1),y_{1,1}^{[N^2]}(t_1),
				y_{2,1}^{[N^2]}(t_1)\right)\right)_{t_1 > 0}\\
				&=\left( \frac{1}{N} \sum_{i \in [N]} x^{[N^2]}_i(Nt_1), \left( \frac{1}{N} \sum_{i \in [N]} y^{[N^2]}_{i,0}(Nt_1), 
				\frac{1}{N} \sum_{i \in [N]} y^{[N^2]}_{i,1}(Nt_1),\frac{1}{N} \sum_{i \in [N]} y^{[N^2]}_{i,2}(Nt_1)\right)\right)_{t_1 > 0}.
			\end{aligned}
		\end{equation} 
		\normalsize
		(Recall Remark~\ref{rem:not2} to appreciate the notation.) We use the lower index $1$ to indicate that the average is the analogue of the 1-block average defined in \eqref{blockav}. Using \eqref{gh45a2twee}, we see that the dynamics of the system in \eqref{gh412twee} is given by the SSDE
		\begin{equation}
			\label{mfevolve3}
			\begin{aligned}
				\d x_1^{[N^2]}(t_1)&=c_1\left[\frac{1}{N^2}\sum_{j\in[N^2]}x_j(Nt_1)-x_1(t_1)\right]\d t_1
				+\sqrt{\frac{1}{N}\sum_{i\in[N]}g(x_i(Nt_1))}\,\d w(t_1)\\
				&\qquad\qquad+NK_0e_0\left[y_{0,1}^{[N^2]}(t_1)-x_1^{[N^2]}(t_1)\right]\d t_1\\
				&\qquad\qquad+K_1e_1\left[y_{1,1}^{[N^2]}(t_1)-x_1^{[N^2]}(t_1)\right]\d t_1\\
				&\qquad\qquad+\frac{K_2e_2}{N}\left[y_{2,1}^{[N^2]}(t_1)-x_1^{[N^2]}(t_1)\right]\d t_1,\\
				\d y_{0,1}^{[N^2]}(t_1)&=Ne_0\left[x_1^{[N^2]}(t_1)-y_{0,1}^{[N^2]}(t_1)\right]\d t_1,\\
				\d y_{1,1}^{[N^2]}(t_1)&=e_1\left[x_1^{[N^2]}(t_1)-y_{1,1}^{[N^2]}(t_1)\right]\d t_1,\\
				\d y_{2,1}^{[N^2]}(t_1)&=\frac{e_2}{N}\left[x_1^{[N^2]}(t_1)-y_{1,1}^{[N^2]}(t_1)\right]\d t_1.
			\end{aligned}
		\end{equation} 
		In the limit $N\to\infty$ we expect that the colour $2$-dormant population does not move, since it only interacts with the active population at rate $\frac{e_2}{N}$. Therefore we expect $(y_{2,1}^{[N^2]}(t))_{t>0}$ to converge to a constant process and hence we say that the colour $2$-dormant population behaves like a \emph{slow seed-bank}. The colour $1$-dormant population, however, has a non-trivial interaction with the active population and therefore is the \emph{effective seed-bank} on space-time scale $1$. The colour $0$-dormant population has, in the limit as $N\to\infty$, an infinitely strong interaction with the active population. Therefore we expect that, in the limit as $N\to\infty$, its path becomes rougher and rougher at rarer and rarer times.  We will need to use the \emph{Meyer-Zheng topology} to prove that 
		\begin{equation}
			\lim_{N\to\infty} y_{0,1}^{[N^2]}(t_1) = \lim_{N\to\infty}x_1^{[N^2]}(t_1) \text{ for most } t_1.
		\end{equation}
		Therefore the colour $0$-dormant population equalizes with the active population, due to its infinitely strong interaction with the active population. Hence at space-time scale $1$, the colour $0$-dormant population behaves like a \emph{fast seed-bank.} If we look at the active population, then we see that it feels a drift towards the $2$-block average of the active population, and resamples at a rate that is the 1-block average of the resampling rates in the single colonies. Furthermore, in the limit as $N\to\infty$, it feels an infinitely fast drift towards the colour $0$-dormant population, has a non-trivial interaction with the colour $1$-dormant population, and its interaction with the colour $2$-dormant population cancels out. As long as we focus on the combination 
		$$
		\frac{x_1^{[N^2]}(t_1)+K_0y_{0,1}^{[N^2]}(t_1)}{1+K_0},
		$$
		we see that the colour-$0$ terms with the factor $N$ in front cancel out. This will allow us to do most of the analysis in the path space topology, without using the Meyer-Zheng topology. The process 
		$$
		\left(\frac{x_1^{[N^2]}(t_1)+K_0y_{0,1}^{[N^2]}(t_1)}{1+K_0},y^{[N^2]}_{1,1}(t_1)\right)_{t_1>0}
		$$ 
		will therefore be called the\emph{ effective process}.
		
		\medskip\noindent
		$\bullet$ \textbf{On time scale $N^2$} (i.e., space-time scale $2$) we look at the equivalent of the $2$-block averages in \eqref{blockav}, 
		\begin{equation}
			\label{gh412drie}
			\begin{aligned}
				&\left(x_2^{[N^2]}(t_2),(y_{0,2}^{[N^2]}(t_2),y_{1,2}^{[N^2]}(t_2),y_{2,2}^{[N^2]}(t_2))\right)_{t_2 > 0}\\
				&= \Bigg(\frac{1}{N^2} \sum_{i \in [N^2]} x^{[N^2]}_i(N^2t_2),\\
				&\qquad\qquad \Bigg( \frac{1}{N^2} \sum_{i \in [N^2]} y^{[N^2]}_{i,0}(N^2t_2), 
				\frac{1}{N^2} \sum_{i \in [N^2]} y^{[N^2]}_{i,1}(N^2t_2),\frac{1}{N^2} 
				\sum_{i \in [N^2]} y^{[N^2]}_{i,2}(N^2t_2)\Bigg)\Bigg)_{t_2 > 0},
			\end{aligned}
		\end{equation}
		which evolves according to the SSDE
		\begin{equation}
			\label{mfevolveba2}
			\begin{aligned}
				\d x_2^{[N^2]}(t_2)&=\sqrt{\frac{1}{N^2}\sum_{i\in[N^2]}g(x_i(N^2t_2))}\,\d w(t_2)\\
				&\qquad\qquad+N^2K_0e_0\left[y_{0,2}^{[N^2]}(t_2)-x_2^{[N^2]}(t_2)\right]\d t_2\\
				&\qquad\qquad +NK_1e_1\left[y_{1,2}^{[N^2]}(t_2)-x_2^{[N^2]}(t_2)\right]\d t_2\\
				&\qquad\qquad
				+K_2e_2\left[y_{2,2}^{[N^2]}(t_2)-x_2^{[N^2]}(t_2)\right]\d t_2,\\
				\d y_{0,2}^{[N^2]}(t_2)&=N^2e_0\left[x_2^{[N^2]}(t_2)-y_{0,2}^{[N^2]}(t_2)\right]\d t_2,\\
				\d y_{1,2}^{[N^2]}(t_2)&=Ne_1\left[x_2^{[N^2]}(t_2)-y_{1,2}^{[N^2]}(t_2)\right]\d t_2,\\
				\d y_{2,2}^{[N^2]}(t_2)&=e_2\left[x_2^{[N^2]}(t_2)-y_{2,2}^{[N^2]}(t_2)\right]\d t_2.
			\end{aligned}
		\end{equation}	
		In this case we see that migration in the active component cancels out and the resampling rate is given by the average over the complete population. In the limit as $N\to\infty$, we see that the active population interacts at an infinitely fast rate with the $0$-dormant population as well as with the colour $1$-dormant population. Hence both the colour $0$ and the colour $1$ seed-banks are fast seed-banks and we expect equalisation of the active population and the colour $0$-dormant population and the colour $1$-dormant population in Meyer-Zheng topology. The active population, in the limit as $N\to\infty$, has a non-trivial interaction with the colour $2$-dormant population, and hence the colour $2$-dormant population is the \emph{effective seed-bank} on time scale $N^2$.  Looking at the quantity 
		\be
		\label{mo}
		\frac{x_2^{[N^2]}(t_2)+K_0y_{0,2}^{[N^2]}(t_2)+K_1y_{1,2}^{[N^2]}(t_2)}{1+K_0+K_1},
		\ee 
		for which we find\
		\begin{equation}
			\begin{aligned}
				&\d\left[ \frac{x_2^{[N^2]}(t_2)+K_0y_{0,2}^{[N^2]}(t_2)+K_1y_{1,2}^{[N^2]}(t_2)}{1+K_0+K_1}\right]\\
				&=\frac{1}{1+K_0+K_1}\sqrt{\frac{1}{N^2}\sum_{i\in[N^2]}g(x_i(N^2t_2))}\,\d w(t_2)
				+K_2e_2\left[y_{2,2}^{[N^2]}(t_2)-x_2^{[N^2]}(t_2)\right]\d t_2,
			\end{aligned}
		\end{equation}
		we see that the infinite rates cancel out. We will call
		\begin{equation}
			\left(\frac{x_2^{[N^2]}(t_2)+K_0y_{0,2}^{[N^2]}(t_2)+K_1y_{1,2}^{[N^2]}(t_2)}{1+K_0+K_1},y_{2,2}^{[N^2]}(t_2)\right)_{t_2>0}
		\end{equation} 
		\emph{the effective process.} Using the effective process we can analyse our system in path space.

		\paragraph{$\blacktriangleright$ Scaling limit.}
		
		Let $(z_0(t))_{t\geq 0}=(x_0(t),(y_{0,0}(t),y_{1,0}(t),y_{2,0}(t)))_{t\geq 0}$ be the process evolving according to
		\begin{equation}
			\label{z0}
			\begin{aligned}
				&\d x_0(t) =  c_0\, [\theta - x_0(t)]\, \d t  + \sqrt{g(x_0(t))}\, \d w (t)\\
				&\qquad\qquad + K_0 e_0\, [y_{0,0}(t)-x_{0}(t)]\,\d t,\\
				&\d y_{0,0}(t) = e_0\,[x_0(t)-y_{0,0}(t)]\, \d t,\\
				& y_{1,0}(t) =y_{1,0},\\
				& y_{2,0}(t) =y_{2,0},
			\end{aligned}
		\end{equation}
		where $\theta\in[0,1]$, $y_{1,0}\in[0,1]$ and $y_{2,0}\in[0,1]$. The process $(z_0(t))_{t\geq 0})$ will be the limiting process for the single colonies. The corresponding single colony effective processes are given by
		\begin{equation}
			\label{m12ch}
			\begin{aligned}
				\d x^{\eff}_0(t) &=  c_0 [\theta - x^{\eff}_0(t)]\, \d t 
				+ \sqrt{g(x^{\eff}_0(t))}\, \d w (t) + K_0 e_0\, [y^{\eff}_{0,0}(t)-x^{\eff}_0(t)]\,\d t,\\
				\d y^{\eff}_{0,0}(t) &= e_0\,[x^{\eff}_0(t)-y^{\eff}_{0,0}(t)]\, \d t,\qquad i\in\N_0,
			\end{aligned}
		\end{equation}
		where $\theta\in[0,1]$. By \cite{YW71}, \eqref{z0} and \eqref{m12ch} have a unique strong solution. Like for the one-colour mean-field finite-systems scheme, we need the following list of ingredients to formally state the multi-scale analysis:
		\begin{enumerate}
			\item 
			For positive times $t>0$, we define the following \emph{$1$-block estimators} for the finite system:
			\begin{equation}
				\label{am6}
				\begin{aligned}
					\bar{\Theta}^{(1) ,[N^2]}(t)&=\frac{1}{N}\sum_{i\in[N]}\frac{ x^{[N^2]}_i(t)+K_0 y^{[N^2]}_{i,0}(t)}{1+K_0},\\
					\Theta^{(1) ,[N^2]}_x(t)&=\frac{1}{N}\sum_{i\in[N]} x^{[N^2]}_{i}(t),\\
					\Theta^{(1) ,[N^2]}_{y_0}(t)&=\frac{1}{N}\sum_{i\in[N]} y^{[N^2]}_{i,0}(t),\\
					\Theta^{(1) ,[N^2]}_{y_1}(t)&=\frac{1}{N}\sum_{i\in[N]} y^{[N^2]}_{i,1}(t),\\
					\Theta^{(1) ,[N^2]}_{y_2}(t)&=\frac{1}{N}\sum_{i\in[N]} y^{[N^2]}_{i,2}(t).
				\end{aligned}
			\end{equation}
			We abbreviate
			\begin{equation}
				\label{1113}
				\begin{aligned}
					{\bf\Theta}^{(1),[N^2]}(t)
					&=\left({\Theta}_x^{(1),[N^2]}(t),\left(\Theta^{(1),[N^2]}_{y_0}(t),
					\Theta^{(1),[N^2]}_{y_1}(t),\Theta^{(1),[N^2]}_{y_2}(t)\right)\right),\\
					{\bf\Theta}^{\aux,(1),[N^2]}(t)
					&=\left(\bar{\Theta}^{(1),[N^2]}(t),\Theta^{(1),[N^2]}_{y_1}(t),
					\Theta^{(1),[N^2]}_{y_2}(t)\right),\\
					{\bf\Theta}^{\eff,(1),[N^2]}(t)
					&=\left(\bar{\Theta}^{(1),[N^2]}(t),\Theta^{(1),[N^2]}_{y_1}(t)\right).
				\end{aligned}
			\end{equation}
			We call  $({\bf\Theta}^{(1),[N^2]}(t))_{t>0}$ the \emph{$1$-block estimator process}, $({\bf\Theta}^{\aux,(1),[N^2]}(t))_{t>0}$ the \emph{auxiliary $1$-block estimator process} and $({\bf\Theta}^{\eff,(1),[N^2]}(t))_{t>0}$ the \emph{effective $1$-block estimator process}. The auxiliary $1$-block estimator will be useful in the proofs. For $t>0$, we define the following \emph{$2$-block estimators} for the finite system:
			\begin{equation}
				\label{am43}
				\begin{aligned}
					\bar\Theta^{(2) ,[N^2]}(t) 
					&=\frac{1}{N^2}\sum_{i\in[N^2]}\frac{ x^{[N^2]}_i(t)+K_0 y^{[N^2]}_{i,0}(t)
						+K_1 y^{[N^2]}_{i,1}(t)}{1+K_0+K_1},\\
					\Theta^{(2) ,[N^2]}_x(t)
					&=\frac{1}{N^2}\sum_{i\in[N^2]} x^{[N^2]}_{i}(t),\\
					\Theta^{(2) ,[N^2]}_{y_0}(t)
					&=\frac{1}{N^2}\sum_{i\in[N^2]} y^{[N^2]}_{i,0}(t),\\
					\Theta^{(2) ,[N^2]}_{y_1}(t)
					&=\frac{1}{N^2}\sum_{i\in[N^2]} y^{[N^2]}_{i,1}(t),\\
					\Theta^{(2) ,[N^2]}_{ y_2}(t)&=\frac{1}{N^2}\sum_{i\in[N^2]} y^{[N^2]}_{i,2}(t).
				\end{aligned}
			\end{equation}
			
			We abbreviate
			\begin{equation}
				\label{1114}
				\begin{aligned}
					{\bf\Theta}^{(2),[N^2]}(t)
					&=\left({\Theta}_x^{(2),[N^2]}(t),\left(\Theta^{(2),[N^2]}_{y_0}(t),\Theta^{(2),[N^2]}_{y_1}(t),
					\Theta^{(2),[N^2]}_{y_2}(t)\right)\right),\\
					{\bf\Theta}^{\eff,(2),[N^2]}(t)
					&=\left(\bar{\Theta}^{(2),[N^2]}(t),\Theta^{(2),[N^2]}_{y_1}(t)\right).
				\end{aligned}
			\end{equation}
			We call $({\bf\Theta}^{\eff,(2),[N^2]}(t))_{t>0}$ the \emph{effective $2$-block estimator process} and $({\bf\Theta}^{(2),[N^2]}(t))_{t>0}$ as the \emph{$2$-block estimator process}.
			
			\item 
			The \emph{time scale} $N$ for which $\CL[\bar\Theta^{(1),[N^2]} (Nt_1-L(N))-\bar\Theta^{(1),[N^2]}(Nt_1)]=\delta_0$ for all $L(N)$ such that $L(N)\to\infty$ and $\lim_{N \to \infty} L(N)/N=0$, but not for $L(N)=N$. In words, $N$ is the time scale on which $\bar\Theta^{(1),[N^2]}(\cdot)$ starts evolving, i.e., $(\bar\Theta^{(1),[N^2]}(Nt_1))_{t_1>0}$ is no longer a fixed process. When we use time scale  $N$, we will use $t_1$ as a time index, which indicates the ``faster time scale". For the ``slow time scale" we use $t_0$ as time index.
			
			The \emph{time scale} $N^2$ for which $\CL[\bar\Theta^{(2),[N^2]} (N^2t_2-L(N)N)-\bar\Theta^{(2),[N^2]}(N^2t_2)]=\delta_0$ for all $L(N)$ such that $L(N)\to\infty$ and $\lim_{N \to \infty} L(N)/N=0$, but not for $L(N)=N$. In words, $N^2$ is the time scale on which $\bar\Theta^{(2),[N^2]}(\cdot)$ starts evolving, i.e., $(\bar\Theta^{(2),[N^2]}(N^2t_2))_{t_2>0},$ is no longer a fixed process. When we use time scale  $N^2$, we will use $t_2$ as a time index, which indicates the ``fastest time scale". 
			
			\item 
			The \emph{invariant measure} for the evolution of a single colony in \eqref{z0}, written   
			\begin{equation}
				\label{singcoleq22}
				\Gamma_{\theta,{\bf y_0}}^{(0)},\qquad {\bf y_0}=(\theta,y_{1,0},y_{2,0}),
			\end{equation}
			and the invariant measure of the level-$0$ effective process evolving according to \eqref{m12ch}, written
			\begin{equation}
				\label{001}
				\Gamma_\theta^{\eff,(0)}.
			\end{equation}
			
			\item 
			The renormalisation transformation $\CF\colon\,\CG\to\CG$,
			\begin{equation}
				\label{gh42bb}
				(\CF g)(\theta) = \int_{[0,1]^2} g(x)\,\Gamma^{\eff,(0)}_{\theta}(\d x, \d y_0), \quad \theta \in [0,1],
			\end{equation} 
			where $\Gamma_{\theta}^{\eff, (0)}$ is the equilibrium measure in \eqref{001}. Note that this is the same transformation as defined in \eqref{renor}, but defined for the truncated system. Later we will study iterates of the renormalisation transformation. Therefore we will write $\CF^{(1)} g=\CF g$, to indicate that we apply the renormalisation transformation only once.
			\item 
			The limiting $1$-block process is given by $(z_1(t))_{t> 0})=(x_1(t),(y_{0,1}(t),y_{1,1}(t),y_{2,1}(t)))_{t> 0}$ and evolves according to
			\begin{equation}
				\label{z1}
				\begin{aligned}
					\d x_1(t) &=  \frac{1}{1+K_0}\Bigg[c_1 [\theta - x_1(t)]\, \d t 
					+ \sqrt{(\CF^{(1)} g)(x_1(t))}\, \d w (t)\\ 
					&\qquad + K_1 e_1\, [y_{1,1}(t)-x_{1}(t)]\,\d t\Bigg],\\
					y_{0,1}(t) &= x_1(t),\\
					\d y_{1,1}(t) &= e_1\,[x_1(t)-y_{1,1}(t)]\, \d t,\\
					y_{2,1}(t) &= y_{2,1},
				\end{aligned}
			\end{equation}
			where $\theta\in[0,1]$, and $y_{2,1}\in[0,1]$, and $\CF^{(1)}$ is the renormalisation transformation defined in \eqref{gh42bb}. The limiting $1$-block process for the auxiliary estimator process is given by $(z_1^\aux(t))_{t> 0}=(x^\aux_1(t),y^\aux_{1,1}(t),y^\aux_{2,1}(t))_{t> 0}$ and evolves according to
			\begin{equation}
				\label{m1chb}
				\begin{aligned}
					\d x^{\aux}_1(t) &= \frac{1}{1+K_0}\Bigg[ c_1 [\theta - x^{\aux}_1(t)]\, \d t 
					+ \sqrt{(\CF^{(1)}g)(x^{\aux}_1(t))}\, \d w (t)\\
					&\qquad + K_1 e_1\, [y^{\aux}_{1,1}(t)-x^{\aux}_1(t)]\,\d t\Bigg]\,,\\
					\d y^{\aux}_{1,1}(t) &= e_1\,[x^{\aux}_1(t)-y^{\aux}_{1,1}(t)]\,\d t,\\
					y^\aux_{2,1}(t) &=y_{2,1},
				\end{aligned}
			\end{equation}
			for $\theta\in[0,1]$. The auxiliary estimator process turns out to be important in the next section. The effective limiting $1$-block process is given by $(z_1^\eff(t))_{t> 0})=(x^\eff_1(t),y^\eff_{1,1}(t))_{t> 0}$ and evolves according to
			\begin{equation}
				\label{m12chb}
				\begin{aligned}
					\d x^{\eff}_1(t) &= \frac{1}{1+K_0}\Bigg[ c_1 [\theta - x^{\eff}_1(t)]\, \d t 
					+ \sqrt{(\CF^{(1)}g)(x^{\eff}_1(t))}\, \d w (t)\\ 
					&\qquad + K_1 e_1\, [y^{\eff}_{1,1}(t)-x^{\eff}_1(t)]\,\d t\Bigg],\\
					\d y^{\eff}_{1,1}(t) &= e_1\,[x^{\eff}_1(t)-y^{\eff}_{1,1}(t)]\, \d t,
				\end{aligned}
			\end{equation}
			for $\theta\in[0,1]$. By \cite{YW71}, \eqref{z1}, \eqref{m1chb} and \eqref{m12chb} have a unique strong solution.
			\item 
			The \emph{invariant measure} of the infinite system in \eqref{z1}, written
			\begin{equation}
				\label{singcoleq2b}
				\Gamma_{\theta,y_1}^{(1)},\qquad y_1=(\theta,\theta,y_{2,1}),
			\end{equation}
			and the invariant measures of the level-$1$ limiting estimator process evolving according to \eqref{m1chb} and the level-$1$ effective process evolving according to \eqref{m12chb},
			\begin{equation}
				\label{002}
				\gls{gammaaux},\, \Gamma_\theta^{\eff,(1)}.
			\end{equation}
			\item  
			The first iteration of the renormalisation transformation,
			\begin{equation}
				\label{gh42bc}
				(\CF^{(2)} g)(\theta) = \int_{[0,1]^2} (\CF g)(x)\,\Gamma^{\eff,(1)}_{\theta}(\d x, \d y_1), \quad \theta \in [0,1].
			\end{equation}
			Hence
			\begin{equation}
				(\CF^{(2)}g)(\theta)=\int_{[0,1]^2} \Gamma^{\eff,(1)}_\theta(\d u, \d v)
				\int_{[0,1]^2} g(x)\,\Gamma^{\eff,(0)}_u(\d x, \d y).
			\end{equation} 
			\item 
			The limiting $2$-block process $(z_2(t))_{t>0}=(x_2(t),(y_{0,2}(t),y_{1,2}(t),y_{2,2}(t)))_{t> 0}$ evolves according to
			\begin{equation}
				\label{gh432c}
				\begin{aligned}
					\d x_2^{}(t) &=  \frac{1}{1+K_0+K_1} \left[\sqrt{(\CF^{(2)} g)\big(x_2^{}(t)\big)}\, \d w(t)
					+ K_2 e_2\left[y_{2,2}^{}(t)-x_2^{}(t)\right]\d t\right],\\
					y_{0,2}^{}(t) &=  x_2^{} (t),\\
					\d y_{1,2}^{}(t) &= x_2^{} (t),\\
					\d y_{2,2}^{}(t) &= e_2\left[x_2^{} (t)-y_{2,2}^{}(t)\right]\,\d t,
				\end{aligned}
			\end{equation}
			where $\CF^{(2)} g$ is defined as in \eqref{gh42bc}. The limiting effective $2$-block process on space-time scale $2$ is $(z^\eff_2(t))_{t> 0})=(x^\eff_2(t),y^\eff_{2,2}(t))_{t> 0}$ and evolves according to
			\begin{equation}
				\label{m64c}
				\begin{aligned}
					\d x^{\eff}_2(t) &=  \frac{1}{1+K_0+K_1}\left[
					\sqrt{(\CF^{(2)} g)(x^{\eff}_2(t))}\, \d w (t) + K_2 e_2\, [y^{\eff}_{2,2}(t)-x^{\eff}_{2}(t)]\,\d t\right],\\
					\d y^{\eff}_{2,2}(t) &= e_2\,[x^{\eff}_2(t)-y^{\eff}_{2,2}(t)]\, \d t.
				\end{aligned}
			\end{equation}
		\end{enumerate}
		
		We are now ready to state the scaling limit for the evolution of the averages in \eqref{gh412}.
		
		\begin{proposition}{{\bf [Mean-field: two-level three-colour finite-systems scheme]}}
			\label{P.finsysmf2lev}
			Suppose that $\mu(0)=\mu^{\otimes [N^2]}$ for some $\mu\in \CP\left([0,1]\times[0,1]^2\right)$. Let 
			\begin{equation}
				\begin{aligned}
					&\vartheta_0 = \E^{\mu}\left[\frac{x+K_0y_0}{1+K_0}\right], \quad
					\vartheta_1 = \E^{\mu}\left[\frac{x+K_0y_0+K_1y_1}{1+K_0+K_1}\right],\\
					&\theta_{y_1} = \E^{\mu}\left[y_1\right],\quad \theta_{y_2} = \E^{\mu}\left[y_2\right].
				\end{aligned}
			\end{equation} 
			and recall the limiting process $(z_2(t))_{t>0}$ in \eqref{gh432c} and the limiting process $(z_1(t))_{t>0}$ in \eqref{z1}.
			Assume for the $2$-dormant $1$-blocks that 
			\begin{equation}
				\label{as2}
				\lim_{N\to\infty}\CL\left[Y_{2,1}^{[N^2]}(Nt_2)\Big|{\bf \Theta}^{(2),[N^2]}(N^2t_2)\right] = P^{z_2(t_2)}, 
			\end{equation}
			and for the $2$-dormant $0$-blocks (= single colonies) that 
			\begin{equation}
				\label{as32}
				\lim_{N\to\infty}\CL\left[Y_{2,0}^{[N^2]}(Nt_2+Nt_1)\Big|{\bf \Theta}^{\eff,(1),[N^2]}(N^2t_2+Nt_1)\right] 
				= P^{z_1(t_1)}. 
			\end{equation}	 
			Then the following hold:
			\begin{itemize}
				\item[(a)]
				For the effective $2$-block estimator process defined in \eqref{1114},
				\begin{equation}
					\label{m12ab}
					\lim_{N\to\infty} \CL \left[\left({\bf\Theta}^{\eff,(2),[N^2]}(N^2t_2)\right)_{t_2 > 0}\right] 
					= \CL \left[\left(z_2^\eff(t_2)\right)_{t_2 > 0}\right],
				\end{equation}
				where the limit is determined by the unique solution of the SSDE \eqref{m64c} with initial state
				\begin{equation}
					\begin{aligned}
						z_2^\eff(0)=\left(x_2^\eff(0),y^\eff_2(0)\right)= \left(\vartheta_1,\theta_{y_2}\right).
					\end{aligned}
				\end{equation}
				\item[(b)]	
				For the effective $1$-block estimator process defined in \eqref{1113}, 
				\begin{equation}
					\label{m5}
					\begin{aligned}
						\lim_{N\to\infty} \CL\left[\left({\bf\Theta}^{\eff,(1) ,[N^2]}(N^2t_2+Nt_1)\right)_{t_1 > 0}\,\right]
						= \CL\left[(z_1^{\eff}(t_1))_{t_1>0}\right], 
					\end{aligned}
				\end{equation}
				where, conditional on $x^{\eff}_2(t_2)=u$, the limit process is the unique solution of the SSDE in \eqref{m12chb} with $\theta$ replaced by $u$ and with initial measure $\Gamma_u^{\eff,(1)}$.
				\item[(c)]  
				For the single colony effective process defined in \eqref{ef}, 
				\begin{equation}
					\label{mef}
					\begin{aligned}
						\lim_{N\to\infty} \CL\left[\left(z_0^{\eff ,[N^2]}(N^2t_2+Nt_1+t_0)\right)_{t_0 \geq 0}\,\right]
						= \CL\left[(z_0^{\eff}(t_0))_{t_0\geq0}\right], 
					\end{aligned}
				\end{equation}
				where, conditional on $x_1^{\eff}(t_1)=v$, the limit process is the unique solution of the SSDE in \eqref{m12ch} with $\theta$ replaced by $v$ and with initial measure $\Gamma_v^{\eff,(0)}$.
				\item[(d)] 
				For the $2$-block estimator process defined in \eqref{1114}, 
				\begin{equation}
					\label{gh423}
					\begin{aligned}
						&\lim_{N\to\infty} \CL \left[\left({\bf \Theta}^{(2),[N^2]}(N^2t_2)\right)_{t_2 > 0}\right]
						= \CL \left[\left(z_2^{}(t_2)\right)_{t_2 > 0}\right]\\
						&\text{in the Meyer-Zheng topology},
					\end{aligned}
				\end{equation}
				where the limit process is the unique solution of the SSDE in \eqref{gh432c} with initial state
				\begin{equation}
					\label{e8653}
					z_2 (0) = \left(\vartheta_1,(\vartheta_1,\vartheta_1,\theta_{y_2})\right). 
				\end{equation}		
				\item[(e)]
				Fix $t_2>0$. Assume~\eqref{as2}.	 
				Define
				\begin{equation}
					\label{mfkern1}
					\begin{aligned}
						\Gamma^{(1)}(t_2)=& \int_{[0,1]^4} S^{[2]}_{t_2}\bigl((\vartheta_1,(\vartheta_1,\vartheta_1,\theta_{y_2})), 
						\d (u_x,u_x,u_x,u_{y_{2,2}})\bigr)\\
						&\int_{[0,1]} P^{(u_x,u_x,u_x,u_{y_{2,2}})}(\d y_{2,1})\,\Gamma^{(1)}_{(u_x,(u_x,u_x,y_{2,1}))} \in \CP([0,1]^4),
					\end{aligned}
				\end{equation}
				where $\Gamma^{(1)}_{(u_x,(u_x,u_x,y_{2,1}))}$ is the equilibrium measure in \eqref{singcoleq2b} and $S^{[2]}_{t_2}((\vartheta_1,(\vartheta_1,\vartheta_1,\theta_{y_2})), \cdot)$ is the time-$t_2$  law of the limiting process $(z_2(t_2))_{t_2>0 }$ in \eqref{gh423}  starting from $(\vartheta_1,(\vartheta_1,\vartheta_1,\theta_{y_2}))  \in [0,1]\times [0,1]^3$. 
				
				Let $(z^{\Gamma^{(1)}(t_2)}(t_1))_{t_1 \geq 0}$ be the random process that conditioned on $z_2^{}(t_2)=(\theta,(\theta, \theta, y_{2,2}))$ moves according to \eqref{z1} with $\theta=\theta$ and $y_{2,1}(0)=y_{2,1}$ and with $z^{\Gamma^{(1)}(t_2)}(0)$ be drawn according to $\Gamma^{(1)}(t_2)$ (which is a mixture of random processes in equilibrium). Then for the $1$-block estimator process defined in \eqref{1113},
				\begin{equation}
					\label{1136}
					\begin{aligned}
						&\hspace{-0.3cm} \lim_{N\to\infty} \CL\left[\left({\bf \Theta}^{(1),[N^2]}(N^2t_2+Nt_1)\right)_{t_1 > 0}\right]
						= \CL \left[(z^{\Gamma^{(1)}(t_2)}(t_1))_{t_1 > 0}\right]\\
						&\qquad\qquad\qquad \text{ in the Meyer-Zheng topology}.
					\end{aligned}
				\end{equation}		
				\item[(f)]
				Let $z_1(t_1)$ be the limiting process obtained in (e). Assume~\eqref{as32}. Define, for $t_2 \in (0,\infty)$,
				\begin{equation}
					\label{mfkern2}
					\begin{aligned}
						\Gamma^{(0)}(t_2)
						&= \int_{[0,1]^4} \Gamma^{(1)}(t_2)(\d z_1)\,\int_{[0,1]} P^{z_1}(\d y_{2,0})\,\Gamma^{(0)}_{(x_1,(x_1,y_{1,1}.y_{2,0}))},
					\end{aligned}
				\end{equation}
				where $\Gamma^{(1)}(t_2)$ is as defined in \eqref{mfkern1}. Let $(z^{\Gamma^{(0)}(t_2)}(t_0))_{t_0 \geq 0}$ be the random process in \eqref{z0} with $z^{\Gamma^{(0)}(t_2)}(0)$ drawn according to $\Gamma^{(0)}(t_2)$ which is a mixture of random processes in equilibrium. Then
				\begin{equation}
					\label{1145}
					\begin{aligned}
						\lim_{N\to\infty} \CL\Big[\Big(z_0^{[N^2]}(N^2t_2+Nt_1+t_0)\Big)_{t_0 \geq 0}\Big] 
						= \CL \left[(z^{\Gamma^{(0)}(t_2)}(t_0))_{t_0 \geq 0}\right].
					\end{aligned}
				\end{equation}
			\end{itemize}
		\end{proposition}
		
		\begin{remark}
			{\rm Note that Proposition~\ref{P.finsysmf2lev}(f) does not depend on the choice of $t_1$, because $\Gamma^{(1)}(t_2)$ is already a mixture of equilibrium measures of the $1$-block process.} \hfill$\blacksquare$
		\end{remark}
		
		\begin{remark}
			{\rm Note that in Propostion~\ref{P.finsysmf2lev}(f) $\Gamma^{(0)}_{(x_1,(x_1,y_{1,1},y_{2,0}))}$ is the equilibrium measure of \eqref{z0} (see also \eqref{singcoleq22}), where $y_{1,0}=y_{1,1}$. This means that all colour $1$-dormant single colonies equal the current state of the colour $1$-dormant $1$-block. We say that given the state of the $1$-dormant $1$-block, the $1$-dormant single colonies become deterministic. This effect occurs once a slow seed-bank, in this case the colour $1$ seed-bank, is already in equilbrium on the space-time scale where it is effective, in this case space-time-scale $1$. Since we start at times $N^2t_2$, the $1$-dormant $1$-blocks are already in equilibrium. This will turn out to be the reason that the single colour $1$-dormant colonies are equal to the current value of the $1$-dormant $1$-block averages. Note that at time $N^2t_2$ the $2$-dormant $2$-blocks do not yet have reached equilibrium. Hence the colour $2$-dormant $1$-blocks and the colour $2$-dormant single colonies do not equal the instantaneous value of the $2$-dormant $2$-block averages. In the Section~\ref{ss.ssb} we will treat this effect in detail.}\hfill$\blacksquare$
		\end{remark}

		\subsection{Scheme for the two-level three-colour mean-field analysis.}
		\label{ss.org22}
		
		In this section we give a scheme to prove Proposition~\ref{P.finsysmf2lev}. The proof of the steps in the scheme will be written in Section~\ref{pmfs23}. To analyse the two-level hierarchical mean-field system we use the results obtained in Sections~\ref{sec:finsysmf}, \ref{ss.pabstracts} and \ref{pmfs2}. 
		
		The scheme for the two-level three-colour hierarchical mean-field system comes in 11 steps. Recall the estimators defined in \eqref{am6} and \eqref{am43}.
		\begin{enumerate}
			\item 
			Tightness of the effective $2$-block estimator processes
			\begin{equation}
				\left(\left({\bf\Theta}^{\eff,(2),[N^2]}(N^2t_2)\right)_{t_2 > 0}\right)_{N\in\N}
				=\left(\left(\bar\Theta^{(2),[N^2]}(N^2t_2),\Theta_{y_2}^{(2),[N^2]}(N^2t_2)\right)_{t_2 > 0}\right)_{N\in\N}.
			\end{equation} 
			\item 
			Stability property of the 2-block estimators, i.e., for  $L(N)$ such that $\lim_{N\to \infty}L(N)=\infty$ and $\lim_{N\to \infty} L(N)/N=0$,
			\begin{equation}
				\label{mb}
				\lim_{N\to\infty}\sup_{0 \leq t\leq L(N)}\left|\bar{\Theta}^{(2),[N^2]}(N^2t_2)
				-\bar{\Theta}^{(2),[N^2]}(N^2t_2-Nt)\right|=0\text{ in probability}
			\end{equation}
			and
			\begin{equation}
				\label{m3}
				\lim_{N\to\infty}\sup_{0 \leq t\leq L(N)}\left|\Theta_{y_2}^{(2),[N^2]}(N^2t_2)
				-\Theta_{y_2}^{(2),[N^2]}(N^2t_2-Nt)\right|=0 \text{ in probability. }
			\end{equation}
			\item 
			Tightness of the effective $1$-block estimator process (recall \eqref{1113}), 
			\begin{equation}
				\left(\left({\bf\Theta}^{\aux, (1),[N^2]}(N^2t_2+Nt_1)\right)_{t_1>0}\right)_{N\in\N}.
			\end{equation}
			\item 
			Stability property of $({\bf\Theta}^{\aux, (1),[N^2]}(N^2t_2+Nt_1))_{t_1>0}$, i.e., for $L(N)$ such that $\lim_{N\to \infty} L(N)=\infty$ and $\lim_{N\to \infty} L(N)/N=0$, for all $\epsilon>0$,
			\begin{equation}
				\label{m22}
				\begin{aligned}
					&\lim_{N\to\infty}\sup_{0 \leq t\leq L(N)}\left|\bar{\Theta}^{(1),[N^2]}(N^2t_2+Nt_1)
					-\bar{\Theta}^{(1),[N^2]}(N^2t_2+Nt_1-t)\right|=0\\
					&\qquad \text{ in probability},
				\end{aligned}
			\end{equation}
			\begin{equation}
				\label{mac}
				\begin{aligned}
					&\lim_{N\to\infty}\sup_{0 \leq t\leq L(N)}\left|\Theta_{y_1}^{(1),[N^2]}(N^2t_2+Nt_1)
					-\Theta_{y_1}^{(1),[N]}(N^2t_2+Nt_1-t)\right| = 0\\ 
					&\qquad \text{ in probability},
				\end{aligned}
			\end{equation}
			\begin{equation}
				\label{masc}
				\begin{aligned}
					&\lim_{N\to\infty}\sup_{0 \leq t\leq L(N)}\left|\Theta_{y_2}^{(1),[N^2]}(N^2t_2+Nt_1)
					-\Theta_{y_2}^{(1),[N]}(N^2t_2+Nt_1-t)\right| = 0\\ 
					&\qquad \text{ in probability}.
				\end{aligned}
			\end{equation}
			\item  
			Recall that there are $N$ 1-blocks in $[N^2]$. Since tightness of components implies tightness of the process, step 3 implies that the full $1$-block process 
			\begin{equation}
				\left(\left({\bf\Theta}_i^{\aux, (1),[N^2]}(N^2t_2+Nt_1)\right)_{t_1>0,\, i\in[N]}\right)_{N\in\N}
			\end{equation}
			is tight. From the tightness in steps 1 and 3 we can construct a subsequence $(N_k)_{k\in\N}$ along which 
			\begin{equation}
				\begin{aligned}
					&\lim_{k\to\infty}\CL\left[\left({\bf\Theta}^{\eff,(2),[N_k^2]}(N_k^2t_2)\right)_{t_2 > 0}\right],\\ &\lim_{k\to\infty}\CL\left[\left({\bf\Theta}_i^{\aux,(1),[N_k^2]}(N_k^2t_2+N_kt_1)\right)_{t_1 > 0,\, i\in[N_k]}\right]
				\end{aligned}
			\end{equation}
			both exists. Define the measure
			\begin{equation}
				\nu^{(0)}(t_2)=\prod_{i\in\N_0}\Gamma^{(0)}_{i}(t_2).
			\end{equation}
			Show that along the same subsequence the single components converge to the infinite system, i.e.,
			\begin{equation}
				\begin{aligned}
					\lim_{k\to\infty} \CL\Big[\Big(Z^{[N_k^2]}(N_k^2t_2+N_kt_1+t_0)\Big)_{t_0 \geq 0}\Big] 
					= \CL \left[(Z^{\nu^{(0)}(t_2)}(t_0))_{t_0 \geq 0}\right].
				\end{aligned}
			\end{equation}
			Here, $(Z^{\nu^{(0)}(t_2)}(t_0))_{t_0 \geq 0}$ is the process starting from $\nu^{(0)}(t_2)$ with components evolving according to \eqref{z0}, where $\theta$ is now a random variable that inherits its law from 
			\be{}
			\lim_{k\to\infty} \CL[({\bf\Theta}^{\aux,(1),[N_k^2]}(N_k^2t_2+N_kt_1))_{i\in[N_k^2]}],
			\ee 
			and, similarly, the laws of $y_{1,0}$ and $y_{2,0}$ in the limiting process $(Z^{\nu^{(0)}(t_2)}(t_0))_{t_0 \geq 0}$ are determined by \be{}
			\lim_{k\to\infty} \CL[({\bf\Theta}^{\aux,(1),[N_k^2]}(N_k^2t_2+N_kt_1))_{i\in[N_k^2]}].
			\ee   
			\item 
			Use the limiting evolution of the single colonies obtained in step 5 to identify the limiting $1$-block process along the same subsequence, i.e., identify the limit
			\begin{equation}
				\lim_{k\to\infty}\CL\left[\left({\bf\Theta}^{\aux,(1),[N_k^2]}(N_k^2t_2+N_kt_1)\right)_{t_1 > 0, i\in[N_k]}\right].
			\end{equation}
			\item  
			Identify the limit $\lim_{k\to\infty}\CL[({\bf\Theta}^{\eff,(2),[N_k^2]}(N_k^2t_2))_{t_2 > 0}]$ with the help of the limiting evolution of the single colonies obtained in step 5 and the limiting evolution of the full $1$-block process obtained in step 6.
			\item 
			Prove that the $1$-dormant single colonies at time $N^2t_2+Nt_1$ equal, in the limit as $N\to\infty$, the $1$-dormant $1$-block averages. The proof of this step shows how the evolution of the slow seed-banks must be analysed.
			\item 
			Show that the convergence in step 8, step 7 and step 5 actually holds along each subsequence. Therefore we obtain the limiting evolution of the single colonies, the auxiliary $1$-block process and the effective $2$-block process.
			\item
			Use the Meyer-Zheng topology to describe the limiting evolution of 
			\begin{equation}
				\begin{aligned}
					&\Big(\Theta_x^{(1),[N^2]}(N^2t_2+Nt_1),\Theta^{(1),[N^2]}_{y_0}(N^2t_2+Nt_1),
					\Theta^{(1),[N^2]}_{y_1}(N^2t_2+Nt_1),\\
					&\qquad \Theta^{(1),[N^2]}_{y_2}(N^2t_2+Nt_1)\Big)_{t_1>0}
				\end{aligned}
			\end{equation}
			and
			\begin{equation}
				\Big(\Theta_x^{(2),[N^2]}(N^2t_2),\Theta^{(2),[N^2]}_{y_0}(N^2t_2),\Theta^{(2),[N^2]}_{y_1}(N^2t_2),
				\Theta^{(2),[N^2]}_{y_2}(N^2t_2)\Big)_{t_2>0}.
			\end{equation}
			\item 
			Combine the above steps to complete the proof of Proposition~\ref{P.finsysmf2lev}.
		\end{enumerate} 
		
		\subsection{Proof of two-level three-colour mean-field finite-systems scheme}
		\label{pmfs23}
		
		In this section we prove the steps in the scheme given in Section~\ref{ss.org22}.
		
		\subsubsection{Tightness of the $2$-block estimators}
		
		In this section we prove step 1 of the scheme.
		
		\begin{lemma}{\bf[Tightness of the 2-block estimator]}
			\label{lem:t3}
			Let 
			\begin{equation}
				{\bf\Theta}^{\eff,(2),[N^2]}(N^2t_2)=(\bar\Theta^{(2),[N^2]}(N^2t_2),\Theta^{(2),[N^2]}_{y_2}(N^2t_2))
			\end{equation}
			be defined as in \eqref{am43}. Then $(\CL[({\bf\Theta}^{\eff,(2),[N^2]}(N^2t_2))_{t_2>0}])_{N\in\N}$ is a tight sequence of probability measures on $\CC((0,\infty), [0,1]^2)$.
		\end{lemma}
		
		\begin{proof}To prove the tightness of the $1$-blocks, we use \cite[Proposition 3.2.3]{JM86}.
			From \eqref{gh45a2twee} we find that $({\bf\Theta}^{\eff,(2),[N^2]}(t))_{t>0}$ evolves according to 
			\begin{equation}
				\label{1184}
				\begin{aligned}
					\d \bar\Theta^{(2),[N^2]}(t)
					&=\frac{1}{1+K_0+K_1}\frac{1}{N^2}\sum_{i\in[N^2]}\sqrt{g(x_i^{[N^2]}(t))}\,\d w_i(t)\\
					&+\frac{1}{1+K_0+K_1}\frac{K_2e_2}{N^2}\left[\frac{1}{N^2}
					\sum_{i\in[N^2]}y_{i,2}^{[N^2]}(t)-\frac{1}{N^2}\sum_{i\in[N^2]}x_i^{[N^2]}(t)\right]\,\d t,\\
					\d \Theta^{(2),[N^2]}_{y_2}(t)
					&= \frac{e_2}{N^2}\left[\frac{1}{N^2}\sum_{i\in[N^2]}x_i^{[N^2]}(t)
					-\frac{1}{N^2}\sum_{i\in[N^2]}y_{i,2}^{[N^2]}(t)\right]\,\d t.
				\end{aligned}
			\end{equation}
			Therefore the process $({\bf\Theta}^{\eff,(2),[N^2]}(N^2t_2))_{t_2>0}$ evolves according to
			\begin{equation}
				\begin{aligned}
					&\d \bar\Theta^{(2),[N^2]}(N^2t_2)
					= \frac{1}{1+K_0+K_1}\sqrt{\frac{1}{N^2}\sum_{i\in[N^2]} g(x_i^{[N^2]}(N^2t_2))}\,\d w_i(t_2)\\
					&\qquad +\frac{1}{1+K_0+K_1}K_2e_2\left[\frac{1}{N^2}\sum_{i\in[N^2]}y_{i,2}^{[N^2]}(N^2t_2)
					-\frac{1}{N^2}\sum_{i\in[N^2]} x_i^{[N^2]}(N^2t_2)\right]\,\d t_2,\\
					&\d \Theta^{(2),[N^2]}_{y_2}(N^2t_2)
					= e_2\left[\frac{1}{N^2}\sum_{i\in[N^2]}x_i^{[N^2]}(N^2t_2)
					-\frac{1}{N^2}\sum_{i\in[N^2]}y_{i,2}^{[N^2]}(N^2t_2)\right]\,\d t_2.
				\end{aligned}
			\end{equation}
			To use \cite[Proposition 3.2.3]{JM86}, we define $\CC^*$ as the set of polynomials on $([0,1]^2)$. Since $({\bf \Theta}^{\eff,(2),[N^2]}(N^2t_2))_{t_2>0}$ is a semi-martingale, by applying Itô's formula we obtain that \newline \mbox{$({\bf \Theta}^{\eff,(2),[N^2]}(N^2t_2))_{t_2>0}$} is a $\CD$-semi-martingale with corresponding operator 
			
			\begin{equation}
				\label{m455}
				\begin{aligned}			
					G_\dagger^{(2),[N^2]}&\colon\, (\CC^*,[0,1]^2,(0,\infty),\Omega)\to \R,\\
					G_\dagger^{(2),[N^2]}(f,(x,y),t,\omega)
					&=\frac{K_2e_2}{1+K_0+K_1}\left[y-\frac{1}{N^2}\sum_{i\in[N^2]} x^{[N^2]}_i(N^2t,\omega)\right]\,
					\frac{\partial f}{\partial x}\\
					&\qquad+e_2\left[\frac{1}{N^2}\sum_{i\in[N^2]} x^{[N^2]}_i(N^2t,\omega)-y\right]\,\frac{\partial f}{\partial y}\\
					&\qquad+\frac{1}{2(1+K_0+K_1)^2} \frac{1}{N^2}
					\sum_{i\in[N^2]}g(x^{[N^2]}_i(N^2t,\omega))\,\frac{\partial^2 f}{\partial x^2}.
				\end{aligned}
			\end{equation} 
			The conditions $H_1,\ H_2,\ H_3$ in \cite[Proposition 3.2.3]{JM86} are satisfied.
			Hence tightness follows from \cite[Proposition 3.2.3]{JM86}.	
		\end{proof}
		
		\subsubsection{Stability of the $2$-block estimators}
		
		\begin{lemma}{\bf[Stability property of the $2$-block estimator]}
			\label{lem:sta3}
			Let  $({\bf\Theta}^{\eff,(2),[N^2]}(t))_{t>0}$ be defined as in \eqref{1114}. For any $L(N)$ such that $\lim_{N\to \infty}L(N)=\infty$ and $\lim_{N\to \infty} L(N)/N=0$,
			\begin{equation}
				\label{ma23bb}
				\lim_{N\to\infty}\sup_{0 \leq t\leq L(N)}\left|\bar{\Theta}^{(2),[N^2]}(N^2t_2)
				-\bar{\Theta}^{(2),[N^2]}(N^2t_2-Nt)\right|=0\text{ in probability}
			\end{equation}
			and
			\begin{equation}
				\label{macy3}
				\lim_{N\to\infty}\sup_{0 \leq t\leq L(N)}\left|\Theta_{y_2}^{(2),[N^2]}(N^2t_2)
				-\Theta_{y_2}^{(2),[N^2]}(N^2t_2-Nt)\right|=0 \text{ in probability. }
			\end{equation}
		\end{lemma}
		
		\begin{proof}
			Fix $\epsilon>0$. From the SSDE in \eqref{1184} we obtain that, for $N$ large enough,
			\begin{equation}
				\begin{aligned}
					&\P\Bigg[\sup_{0\leq t\leq L(N)}\Bigg|\bar{\Theta}^{(2),[N^2]}(N^2t_2)-\bar{\Theta}^{(2),[N^2]}(N^2t_2-Nt)\Bigg|>\epsilon\Bigg]\\
					&=\P\Bigg[\sup_{0 \leq t\leq L(N)}\frac{1}{1+K_0+K_1}\Bigg|\int_{N^2t_2-Nt}^{N^2t_2} 
					\d w_i(r)\,\frac{1}{N^2}\sum_{i \in [N^2]}\sqrt{g(x^{[N^2]}_i(r))}\\
					&\qquad\qquad\qquad+\int_{N^2t_2-Nt}^{N^2t_2}\d r\,\frac{K_2e_2}{N^2}\left[\Theta^{(2),[N^2]}_{y_2}(r)
					-\frac{1}{N^2}\sum_{i \in [N^2]}x^{[N^2]}_i(r)\right]\Bigg|>\epsilon\Bigg]\\
					&\leq\P\Bigg[\sup_{0\leq t\leq L(N)}\,\frac{1}{1+K_0+K_1}\Bigg|\int_{N^2t_2-Nt}^{N^2t_2}\d w_i(r)\,\frac{1}{N^2}
					\sum_{i \in [N^2]}\sqrt{g(x^{[N^2]}_i(r))}\,\Bigg|\\
					&\qquad\qquad\qquad\qquad > \epsilon-\frac{K_2e_2}{1+K_0+K_1}\frac{L(N)N}{N^2}\Bigg]\\
					&\leq\P\Bigg[\sup_{0\leq t\leq L(N)}\frac{1}{1+K_0+K_1}\Bigg|\int_{N^2t_2-Nt}^{N^2t_2} \d w_i(r)\frac{1}{N^2}
					\sum_{i \in [N^2]}\sqrt{g(x^{[N^2]}_i(r))}\,\Bigg|>\frac{\epsilon}{2}\Bigg].
				\end{aligned}
			\end{equation} 
			By a similar optional stopping time argument as in the proof of Lemma~\ref{martav}, the above computation shows that \eqref{ma23bb} holds. Equation \eqref{macy3} holds by a similar argument as given in the proof of Lemma~\ref{stab}.	
		\end{proof}
		
		\subsubsection{Tightness of the 1-block estimators}\label{s.t1bl}
		
		\begin{lemma}{\bf [Tightness of the 1-block estimator]}
			\label{tonebl}
			Let
			\begin{equation}
				\begin{aligned}
					&{\bf \Theta}^{\aux,(1),[N^2]}(N^2t_2+Nt_1)\\
					&=(\bar\Theta^{(1),[N^2]}(N^2t_2+Nt_1),
					\Theta^{(1),[N^2]}_{y_1}(N^2t_2+Nt_1),\Theta^{(1),[N^2]}_{y_2}(N^2t_2+Nt_1))
				\end{aligned}
			\end{equation}
			be defined as in \eqref{am6}. Then $(\CL[({\bf \Theta}^{\aux,(1),[N]}(N^2t_2+Nt_1))_{t_1>0}])_{N\in\N}$ is a tight sequence of probability measures on $\CC((0,\infty), [0,1]^3)$.
		\end{lemma}
		
		\begin{proof}
			To prove the tightness of the $1$-blocks, we again use \cite[Proposition 3.2.3]{JM86}. From \eqref{mfevolve3} we find that the effective process  $({\bf \Theta}^{\aux,(1),[N^2]}(N^2t_2+Nt_1))_{t_1>0}$
			evolves according to 
			\begin{equation}
				\label{m21}
				\begin{aligned}
					\d \bar{\Theta}^{(1),[N^2]}(Nt_1)&=\frac{1}{1+K_0}c_1\left[\frac{1}{N^2}\sum_{j\in[N^2]} x^{[N^2]}_j(Nt_1)
					-\frac{1}{N}\sum_{i\in[N]} x^{[N^2]}_i(Nt_1)\right]\,\d t_1\\
					&\qquad+\frac{1}{1+K_0}\sqrt{\frac{1}{N}\sum_{i\in[N]}g(x^{[N^2]}_i(Nt_1))}\,\d w_i(t_1)\\
					&\qquad+\frac{K_1e_1}{1+K_0}\left[\Theta_{y_1}^{(1),[N^2]}(Nt_1)
					-\frac{1}{N}\sum_{i\in[N]} x^{[N^2]}_i(Nt_1)\right]\,\d t_1\\
					&\qquad+\frac{K_2e_2}{N(1+K_0)}\left[\Theta^{(1),[N^2]}_{y_2}(Nt_1)-\frac{1}{N}
					\sum_{i\in[N]} x^{[N^2]}_i(Nt_1)\right]\,\d t_1,\\
					\d \Theta^{(1),[N^2]}_{y_1}(Nt_1)&=e_1\left[\frac{1}{N}\sum_{i\in[N]} x^{[N^2]}_i(Nt_1)
					-\Theta^{(1),[N^2]}_{y_1}(Nt_1)\right]\,\d t_1,\\
					\d \Theta^{(1),[N^2]}_{y_2}(Nt_1)&=\frac{e_2}{N}\left[\frac{1}{N}\sum_{i\in[N]} x^{[N^2]}_i(Nt_1)
					-\Theta^{(1),[N^2]}_{y_2}(Nt_1)\right]\,\d t_1.
				\end{aligned}
			\end{equation}
			To use \cite[Proposition 3.2.3]{JM86}, we define $\CC^*$ as the set of polynomials on $([0,1]^2)$. Since $({\bf \Theta}^{\aux,(1),[N^2]}(N^2t_2+Nt_1))_{t_1>0}$ is a semi-martingale, by applying Itô's formula we obtain that $({\bf \Theta}^{\aux,(1),[N^2]}(N^2t_2+Nt_1))_{t_1>0}$ is a $\CD$-semi-martingale with corresponding operator
			\begin{equation}
				\label{m45}
				\begin{aligned}
					G_\dagger^{(1),[N^2]}&\colon\, (\CC^*,[0,1]^3,(0,\infty),\Omega)\to \R,\\
					G_\dagger^{(1),[N^2]}(f,(x,y_1,y_2),t,\omega)&=
					\frac{c_1}{1+K_0}\left[\frac{1}{N^2}\sum_{j\in[N^2]} x^{[N^2]}_j(Nt,\omega)-\frac{1}{N}\sum_{i\in[N]}
					x^{[N^2]}_i(Nt,\omega)\right]\,\frac{\partial f}{\partial x}\\
					&\qquad+\frac{K_1e_1}{1+K_0}\left[y_1-\frac{1}{N}\sum_{i\in[N]} x^{[N^2]}_i(Nt,\omega)\right]\,
					\frac{\partial f}{\partial x}\\
					&\qquad+ \frac{K_2e_2}{N(1+K_0)}\left[y_2(Nt,\omega)-\frac{1}{N}
					\sum_{i\in[N]} x^{[N^2]}_i(Nt)\right]\,\frac{\partial f}{\partial x}\\
					&\qquad+e_1\left[\frac{1}{N}\sum_{i\in[N]} x^{[N^2]}_i(Nt,\omega)-y_1\right]\,\frac{\partial f}{\partial y_1}\\
					&\qquad+\frac{e_2}{N}\left[\frac{1}{N}\sum_{i\in[N]} x^{[N^2]}_i(Nt,\omega)-y_2(Nt,\omega)\right]\,\frac{\partial f}{\partial y_2}\\
					&\qquad+\frac{1}{2(1+K_0)^2} \frac{1}{N}\sum_{i\in[N]}g(x^{[N^2]}_i(Nt,\omega))\,\frac{\partial^2 f}{\partial x^2}.
				\end{aligned}
			\end{equation}
			The conditions $H_1,\ H_2,\ H_3$ in \cite[Proposition 3.2.3]{JM86} are satisfied as before. Hence we conclude  that the sequence $(\CL[({\bf \Theta}^{\aux,(1),[N^2]}(N^2t_2+Nt_1))_{t_1>0}])_{N\in\N}$ is tight.
		\end{proof}
		
		\subsubsection{Stability of the $1$-block estimators}
		
		\begin{lemma}{\bf[Stability property of the $1$-block estimator]}
			\label{lemsta2}
			Let  ${\bf\Theta}^{\aux,(1),[N^2]}(t)$ be defined as in \eqref{ma6}. For any $L(N)$ such that $\lim_{N\to \infty}L(N)=\infty$ and $\lim_{N\to \infty} L(N)/N=0$,
			\begin{equation}
				\label{ma22bb}
				\lim_{N\to\infty}\sup_{0 \leq t\leq L(N)}\left|\bar{\Theta}^{(1),[N^2]}(N^2t_2+Nt_1)
				-\bar{\Theta}^{(1),[N^2]}(N^2t_2+Nt_1-t)\right|=0\text{ in probability},
			\end{equation}
			\begin{equation}
				\label{macy}
				\lim_{N\to\infty}\sup_{0 \leq t\leq L(N)}\left|\Theta_{y_1}^{(1),[N^2]}(N^2t_2+Nt_1)
				-\Theta_{y_1}^{(1),[N^2]}(N^2t_2+Nt_1-t)\right|=0 \text{ in probability},
			\end{equation}
			\begin{equation}
				\label{mascy}
				\lim_{N\to\infty}\sup_{0 \leq t\leq L(N)}\left|\Theta_{y_2}^{(1),[N^2]}(N^2t_2+Nt_1)
				-\Theta_{y_2}^{(1),[N^2]}(N^2t_2+Nt_1-t)\right|=0 \text{ in probability}.
			\end{equation}
		\end{lemma}
		
		\begin{proof}
			Define 
			\begin{equation}
				u=N^2t_2+Nt_1.
			\end{equation}
			From the SSDE in \eqref{gh45b2twee} we obtain that
			\begin{equation}
				\begin{aligned}
					&\P\Bigg[\sup_{0\leq t\leq L(N)}\Bigg|\bar{\Theta}^{(1),[N^2]}(u)-\bar{\Theta}^{(1),[N^2]}(u-t)\Bigg|>\epsilon\Bigg]\\
					&=\P\Bigg[\sup_{0 \leq t\leq L(N)}\frac{1}{1+K_0}\Bigg|\int_{u-t}^{u} \d r\,\frac{c_1}{N}\left[\frac{1}{N^2}
					\sum_{j\in[N^2]}x^{[N^2]}_j(r)-\frac{1}{N}\sum_{i \in [N]}x^{[N^2]}_i(r)\right]\\
					&\qquad\qquad\qquad+\int_{u-t}^{u}\d r\,\frac{K_1e_1}{N}\left[\Theta^{(1),[N^2]}_{y_1}(r)
					-\frac{1}{N}\sum_{i \in [N]}x^{[N^2]}_i(r)\right]\\
					&\qquad\qquad\qquad+\frac{K_2e_2}{N^2}\left[\frac{1}{N}\sum_{i \in [N]}y^{[N^2]}_{i,2}(r)
					-\frac{1}{N}\sum_{i \in [N]}x^{[N^2]}_i(r)\right]\\
					&\qquad\qquad\qquad+\int_{u-t}^{u} \d w_i(r)\,\frac{1}{N}\sum_{i \in [N]}\sqrt{g(x^{[N^2]}_i(r))}\,\Bigg|>\epsilon\Bigg]\\
					&\leq\P\Bigg[\sup_{0\leq t\leq L(N)}\,\frac{1}{1+K_0}\Bigg|\int_{u-t}^{u}\d w_i(r)\,\frac{1}{N}
					\sum_{i \in [N]}\sqrt{g(x^{[N^2]}_i(r))}\,\Bigg|\\
					&\qquad\qquad > \epsilon-\frac{L(N)2(c_1+K_1e_1+\frac{K_2e_2}{N})}{N(1+K_0)}\Bigg]\\
					&\leq\P\Bigg[\sup_{0\leq t\leq L(N)}\frac{1}{1+K_0}\Bigg|\int_{u-t}^{u} \d w_i(r)\frac{1}{N}
					\sum_{i \in [N]}\sqrt{g(x^{[N^2]}_i(r))}\,\Bigg|>\frac{\epsilon}{2}\Bigg].
				\end{aligned}
			\end{equation} 
			Via the same optional stopping time argument as in the proof of Lemma~\ref{martav}, the above computation shows that \eqref{ma22bb} holds. Note that the extra drift term $\frac{K_2e_2}{N}$ does not have any influence. Equations \eqref{macy}--\eqref{mascy} hold by a similar argument as given in the proof of Lemma~\ref{stab}.
		\end{proof}

		\subsubsection{Limiting evolution for the single components}
		
		\begin{proposition}{\bf[Equilibrium for the infinite system]}
			\label{ma77}
			Fix $t_2,t_1>0$. Let $(N_k)_{k\in\N}\subset \N$ and let $L(N)$ be any sequence satisfying  $\lim_{N\to\infty} L(N)=\infty$ and $\lim_{N\to\infty} L(N)/N=0$ such that  
			\begin{equation}
				\label{z3}
				\begin{aligned}
					&\lim_{k\to\infty}\CL\left[{\bf\Theta}^{\aux,(1),[N_k^2]}(N^2_kt_2+N_kt_1)\right] = P_{t_1,t_2},\\ 
					&\lim_{k\to\infty}\CL\left[\left(Y_{1,0}^{[N_k^2]}(N^2_kt_2+N_kt_1),Y_{2,0}^{[N_k^2]}(N^2_kt_2)\right)
					\Big|{\bf\Theta}^{\aux,(1),[N_k^2]}(N^2_kt_2+N_kt_1)\right]\\
					&\qquad\qquad= P^{{z}_1^{\eff}(t_1)},\\ 
					&\lim_{k\to\infty }\CL\Bigl[\sup_{0\leq t\leq L(N_k)}\left|\bar{\Theta}^{[N^2_k]}(N^2_kt_2+N_kt_1)
					-\bar{\Theta}^{[N^2_k]}(N^2_kt_2+N_kt_1-t)\right|\\
					&\qquad\qquad\qquad+\left|{\Theta}_{y_1}^{[N^2_k]}(N^2_kt_2+N_kt_1)
					-{\Theta_{y_1}}^{[N^2_k]}(N^2_kt_2+N_kt_1-t)\right|\\
					&\qquad\qquad\qquad+\left|{\Theta}_{y_2}^{[N^2_k]}(N^2_kt_2+N_kt_1)
					-{\Theta_{y_2}}^{[N^2_k]}(N^2_kt_2+N_kt_1-t)\right|
					\Bigr] =\delta_{0},\\
					&\lim_{k\to\infty} \CL\Big(Z^{[N^2_k]}(N^2_kt_2+N_kt_1),\Big) = \nu(t_1,t_2).
				\end{aligned}
			\end{equation}
			
			Then $\nu(t_1,t_2)$ is of the form 
			\begin{equation}
				\label{m2a3b}
				\nu(t_1,t_2) = \int_{[0,1]^2} P_{t_1,t_2}(\d \theta^{(1)},\d \theta^{(1)}_y)\,\int_{[0,1]^{\N_0}}
				P^{( \theta^{(1)}, \theta^{(1)}_y)}(\d {\bf y})\,\nu_{\theta,{\bf y}},
			\end{equation}
			where
			\begin{equation}
				\nu_{\theta,{\bf y_0}}=\prod_{i\in\N_0}\Gamma_{(\theta,{\bf y_{0,i}})}
			\end{equation}
			with $\Gamma_{(\theta,{\bf y_{0,i}})}$ the equilibrium measure for the $i$'th single colony defined in \eqref{singcoleq22}.
		\end{proposition}
		
		Note that by step 1 and step 3 we can find a subsequence $(N_k)_{k\in\N}$ such that the first and third line in \eqref{z3} hold. The second  line in \eqref{z3} follows from assumptions \eqref{as2} and \eqref{as32}. To prove Proposition~\ref{ma77} we proceed as in the proof of Proposition \ref{prop2}, but with the finite system in \eqref{gh45a2} replaced by the system in \eqref{gh45a2twee} and the infinite system in \eqref{gh45a2binf} replaced by the system in \eqref{z0}. Note that Lemma~\ref{lemerg2} holds also for the system in \eqref{z0}, after adding the non-interacting component $y_{2,0}$ to the equilibrium. The equivalent of Lemma~\ref{lemlip2} will again follow from the equivalent of Lemma~\ref{unifergod2}. We will derive the analogue of Lemmas~\ref{lemlev1a} and ~\ref{l.comp2b2} (see Lemma's~\ref{lemlev1} and \ref{l.comp2} below). Lemma~\ref{stabest2} can be extended with an extra colour-$2$ seed-bank estimator by using the same proof. Since the infinite system for the single colonies in the two-layer three-colour mean-field system (see \eqref{mgh52}) equals the one for the one-layer two-colour mean-field system, up to a non-interacting component, we obtain an equivalent of Lemma~\ref{unifergod2}. Finally, also the equivalent of Lemma~\ref{lem:12a} holds under an additional assumption, see Lemma~\ref{lem:2}. Finally Corollary~\ref{cor3} states the equivalent of Corollary~\ref{cor2}.  With the help of the lemma's and the corollary, the proof of Proposition~\ref{ma77} follows from the same argument as used in the proof of Proposition~\ref{prop2}.
		
		\begin{lemma}{{\bf [Comparison of empirical averages]}}
			\label{lemlev1}
			$\mbox{}$\\
			Let $(\Theta^{(1),[N^2]}_x(t_0))_{t_0\geq 0}$ and $(\Theta^{(1),[N^2]}_{y_0}(t_0))_{t_0\geq 0}$ be defined as in \eqref{am6}. Then 
			\begin{equation}
				\label{bnd2}
				\begin{aligned}
					\E\left[\left|\Theta^{(1),[N]}_{x}(t)-\Theta^{(1),[N]}_{y_0}(t)\right|\right]
					&\leq \sqrt{\E\left[\left(\Theta^{(1),[N]}_{x}(0)-\Theta^{(1),[N]}_{y_0}(0)\right)^2\right]}\e^{-(K_0e_0+e_0)t}\\ 
					&\quad + \sqrt{\frac{1}{K_0e_0+e_0}\left[\frac{c_1}{N}+\frac{||g||}{N}+\frac{ K_1e_1}{N}+\frac{K_2e_2}{N^2}\right]}.
				\end{aligned}
			\end{equation}	
		\end{lemma}
		
		\begin{proof} 
			The result follows by Itô-calculus on the SSDE in \eqref{gh45a2twee} and the same type of argument as used in the proof of Lemma~\ref{lemlev1a}.
		\end{proof}
		
		Like for the mean-field system with one colour, we need to compare the finite system in \eqref{gh45a2twee} with an infinite system. To derive the analogue of Lemma~\ref{l.comp2b2}, let $L(N)$ satisfy $\lim_{N\to\infty} L(N)=\infty$ and $\lim_{N \to \infty} L(N)/N=0$. Define $[N]_i$ to be the $1$-block that contains site $i\in[N^2]$. Since we start our system in an exchangeable measure and the dynamics are exchangeable, we will only consider the single colonies in $[N]_0$, the $1$-block containing the site $0\in[N^2]$. In the rest of the prove, we will suppress the $0$ from the notation i.e., $[N]_0=[N]$ and  $\bar{ \Theta}^{(1),[N^2]}_0=\bar{ \Theta}^{(1),[N^2]}$. Define 
		\begin{equation}
			u=N^2t_2+Nt_1
		\end{equation} 
		and let  $\mu_N$ be the measure on $([0,1]^3)^{\N_0}$ by continuing the configuration of 
		\begin{equation}
			\begin{aligned}
				&\left(Z^{[N^2]}(u-L(N))\right)\\
				&=\left(X^{[N^2]}(u-L(N)),\left(Y_{0}^{[N^2]}(u-L(N)),Y_{1}^{[N^2]}
				(u-L(N)),Y_{2}^{[N^2]}(u-L(N))\right)\right)
			\end{aligned}
		\end{equation} 
		periodically to $([0,1]^4)^{\N_0}$, i.e., we continue the configuration of the single colonies in the first block to $([0,1]^4)^{\N_0}$. Let  
		\begin{equation}
			\bar{\Theta}^{(1),[N^2]} =\frac{1}{N}\sum_{i\in[N]}\frac{x^{[N^2]}_i(u-L(N))+K_0y^{[N^2]}_{i,0}(u-L(N))}{1+K_0}.
		\end{equation}
		Let
		\begin{equation}
			\label{t5c}
			(Z^{\mu_N}(t))_{t\geq 0}=\bigl(X^{\mu_N}(t),\left(Y_0^{\mu_N}(t),Y_1^{\mu_N}(t),
			Y_2^{\mu_N}(t)\right)\bigr)_{t \geq 0}
		\end{equation}
		be the infinite system evolving according to
		\begin{equation}
			\label{mgh52}
			\begin{aligned}
				\d x^{\mu_N}_i(t) &= c_0\,[\bar{\Theta}^{(1),[N^2]} - x^{\mu_N}_i(t)]\, \d t + \sqrt{g\big(x^{\mu_N}_i(t)\big)}\, \d w_i (t) 
				+ K_0e_0\,[y^{\mu_N}_{i,0}(t)-x^{\mu_N}_i(t)]\,\d t,\\
				\d y^{\mu_N}_{i,0}(t) &= e_0\, [x^{\mu_N}_i(t)-y^{\mu_N}_{i,0}(t)]\, \d t,\\
				y^{\mu_N}_{i,1}(t) &= y^{\mu_N}_{i,1}(0),\\
				y^{\mu_N}_{i,2}(t) &= y^{\mu_N}_{i,2}(0), \qquad i\in\N_0,
			\end{aligned}
		\end{equation} 
		starting from initial distribution $\mu_N$. Then the following Lemma~\ref{l.comp2} is the equivalent of Lemma~\ref{l.comp2b2} for the three-colour two-layer mean-field system. In particular, the infinite system considered in Lemma~\ref{l.comp2} is similar to the infinite system in Lemma~\ref{l.comp2b2}. The only difference is that there is one more non-interacting component added in \eqref{mgh52}.
		
		\begin{lemma}{\bf [Comparison of finite and infinite systems]}
			\label{l.comp2}
			Fix $t_1,t_2 > 0$, and let $u=N^2t_2+Nt_1$. Let $L(N)$ satisfy $\lim_{N\to\infty} L(N)=\infty$ and $\lim_{N\to\infty} L(N)/N$. Suppose that 
			\begin{equation}
				\begin{aligned}
					&\lim_{N\to\infty} \sup_{0 \leq t \leq L(N)} \left|\bar{\Theta}^{(1),[N]}(u)
					-\bar{\Theta}^{(1),[N]}(u-t)\right| = 0\ \text{ in probability}.
				\end{aligned}
			\end{equation}
			Then, for all $t\geq 0$,
			\begin{equation}
				\label{m322}
				\lim_{k\to\infty} \left|\E\left[f\big(Z^{\mu_{N}}(t)\big) -f\big(Z^{[N^2]}(u-L(N)+t)\big)\right]\right| = 0
				\qquad \forall\, f\in\CC\bigl(([0,1]^3)^{\N_0},\R\bigr).
			\end{equation}
		\end{lemma}
		
		\begin{proof}
			We proceed as in the proof of Lemma~\ref{l.comp2b2} and couple the finite and infinite systems by their Brownian motion, exactly as was done there. The single components in the block around site $0$ of the finite process $(Z^{[N^2]}(t))$ are evolving according to
			\begin{equation}
				\begin{aligned}
					\label{m24}
					&\d x^{[N^2]}_i(t) = c_0\,\left[\Theta^{(1),[N^2]}-x^{[N^2]}_i(t)\right]\,\d t+c_0\,\left[\bar{\Theta}^{(1),[N^2]}(t) 
					- \Theta^{(1),[N^2]}\right]\, \d t\\
					&\qquad\qquad+ c_0\,\left[\Theta^{(1),[N^2]}_x(t) -\bar{\Theta}^{(1),[N^2]}(t)\right]\, \d t 
					+ \frac{c_1}{N}\left[\frac{1}{N^2}\sum_{i\in[N^2]}x_j^{[N^2]}(t)-x_i^{[N^2]}(t)  \right]\d t\\
					&\qquad\qquad + \sqrt{g(x^{[N^2]}_i(t))}\, \d  w_i (t)+ K_0 e_0\, [y^{[N^2]}_{i,0}(t)-x^{[N^2]}_i(t)]\,\d t\\
					&\qquad\qquad+\frac{K_1e_1}{N}\left[y^{[N^2]}_{i,1}(t)-x_i^{[N^2]}(t)\right]\d t
					+ \frac{K_2 e_2}{N^2}\, [y^{[N^2]}_{i,0}(t)-x^{[N^2]}_i(t)]\,\d t,\\
					&\d y^{[N^2]}_{i,0}(t) = e_0\,[x^{[N^2]}_i(t)-y^{[N^2]}_{i,0}(t)]\, \d t, \\
					&\d y^{[N^2]}_{i,1}(t) = \frac{e_1}{N}\,[x^{[N^2]}_i(t)-y^{[N^2]}_{i,1}(t)]\, \d t, \\
					&\d y^{[N^2]}_{i,2}(t) = \frac{e_2}{N^2}\,[x^{[N^2]}_i(t)-y^{[N^2]}_{i,2}(t)]\, \d t, \qquad i \in [N].
				\end{aligned}
			\end{equation}
			Using this SSDE we can exactly proceed as in the proof of Lemma~\ref{l.comp2b2} to obtain the result. Note that the colour-$2$ seed-bank can be treated just in the same way as the colour-$1$ seed-bank in the proof of Lemma~\ref{l.comp2b2}, since its rate of interaction with the active population is even slower than the rate of interaction of the colour-$1$ seed-bank. 
		\end{proof}
		
		Finally, we state the equivalent of Lemma~\ref{lem:12a} for the three-colour two-layer mean-field system.
		
		\begin{lemma}{\bf[Coupling of finite systems]}
			\label{lem:2}
			Let 
			\begin{equation}
				Z^{[N^2],1}=(X^{[N^2],1},Y_0^{[N^2],1},Y_1^{[N^2],1},Y_2^{[N^2],1})
			\end{equation}
			be the finite system evolving according to \eqref{gh45a2twee} starting from an exchangeable initial measure. Let $\mu^{[N],1}$ be the measure obtain by periodic continuation of the configuration of $Z^{[N^2],1}(0)$ in the $1$-block around $0$. Similarly, let  
			\begin{equation}
				Z^{[N^2],2}=(X^{[N^2],2},Y_0^{[N^2],2},Y_1^{[N^2],2},Y_2^{[N^2],2})
			\end{equation} 
			be the finite system evolving according to \eqref{gh45a2twee} starting from an exchangeable initial measure. Let $\mu^{[N],2}$ be the measure obtained by periodic continuation of the configuration of $Z^{[N^2],1}(0)$ in the $1$-block around $0$. Let $\tilde{\mu}$ be any weak limit point of the sequence of measures $(\mu^{[N],1}\times\mu^{[N],2})_{N\in\N}$. Define the variables $\bar{\Theta}^{[N],1}$ on $(([0,1]^4,\mu^{[N],1})^{\N_0})$, $\bar{\Theta}^{[N],2}$ on $(([0,1]^4)^{\N_0},\mu^{[N],2})$ and $\bar{\Theta}_1$ and $\bar{\Theta}_2$ on $(([0,1]^4)^{\N_0}\times([0,1]^4)^{\N_0},\mu)$ by
			\begin{equation}
				\label{83a}
				\begin{aligned}
					&\bar{\Theta}^{[N],1} = \frac{1}{N} \sum_{i \in [N]} \frac{x^{[N^2],1}_{i}+K_0y^{[N^2],1}_{i,0}}{1+K_0},
					\qquad \bar{\Theta}^{[N],2} =  \frac{1}{N} \sum_{i \in [N]} \frac{x^{[N^2],2}_{i}+K_0y^{[N^2],2}_{i,0}}{1+K_0},\\
					&\bar{\Theta}^1 = \lim_{n\to\infty} \frac{1}{n} \sum_{i \in [n]} \frac{x^1_{i}+K_0y^1_{i,0}}{1+K_0},
					\qquad \bar{\Theta}^2 = \lim_{n\to\infty} \frac{1}{n} \sum_{i \in [n]} \frac{x^2_{i}+K_0y^2_{i,0}}{1+K_0},
				\end{aligned}
			\end{equation}
			and let $(\bar{\Theta}^{(1),[N],1}(t))_{t\geq 0}$ and $(\bar{\Theta}^{(1),[N],2}(t))_{t\geq 0}$ be defined  as in  \eqref{ma6} for $Z^{[N^2],1}$, respectively, $Z^{[N^2],2}$. Suppose that   
			\begin{equation}
				\label{m8a}
				\begin{aligned}
					&\lim_{N\to\infty} \sup_{0 \leq t \leq L(N)} \left(\left|\bar{\Theta}^{[N],k}(0)-\bar{\Theta}^{[N],k}(t)\right|\right) 
					= 0\ \text{ in probability}, \quad k \in\{1,2\},
				\end{aligned}
			\end{equation}
			and suppose that  $\tilde{\mu}(\{\bar{\Theta}_1=\bar{\Theta}_2,\, Y^1_{1}=Y^2_1, Y^1_2=Y^2_2\})=1$. Then, for any  sequence $(t(N))_{N\in\N}$ with  $\lim_{N \to \infty}t(N)=\infty$,
			\begin{equation}
				\label{m30}
				\begin{aligned}
					&\lim_{N\to\infty}\E\bigl[|x^{[N],1}_{i}(t(N))-x^{[N],2}_{i}(t(N))|+K_0|y^{[N],1}_{i,0}(t(N))-y^{[N],2}_{i,0}(t(N))|\\
					&\qquad +K_1|y^{[N],1}_{i,1}(t(N))-y^{[N],2}_{i,1}(t(N))|+K_2|y^{[N],1}_{i,2}(t(N))-y^{[N],2}_{i,2}(t(N))|\bigr]=0.
				\end{aligned}
			\end{equation}  
		\end{lemma}
		
		\begin{proof}
			Like in the proof of Lemma~\ref{lem:12a}, we can show with It\^o calculus that the function
			\begin{equation}
				\begin{aligned}
					t\to&\E\bigl[|x^{[N],1}_{i}(t(N))-x^{[N],2}_{i}(t(N))|+K_0|y^{[N],1}_{i,0}(t(N))-y^{[N],2}_{i,0}(t(N))|\\
					&\qquad +K_1|y^{[N],1}_{i,1}(t(N))-y^{[N],2}_{i,1}(t(N))|+K_2|y^{[N],1}_{i,2}(t(N))-y^{[N],2}_{i,2}(t(N))|\bigr]
				\end{aligned}
			\end{equation}
			is monotonically decreasing. Hence we can proceed as in the proof of Lemma~\ref{lem:12} to show that  \eqref{m30} is true.
		\end{proof}
		
		From the above couplings we can derive the following corollary, which is the analogue of Corollary~\ref{cor2} for the two-level three-colour mean-field system.
		
		\begin{corollary}
			\label{cor3}
			Fix $t_1,t_2>0$ and set $u=N^2t_2+Nt_1$. Let $\mu_{N}$ be the measure obtained by periodic continuation of 
			\begin{equation}
				Z^{[N^2]}(u-L(N)) = \big(X^{[N^2]}(u-L(N)),Y_0^{[N^2]}(u-L(N)),Y_1^{[N^2]}(u-L(N)),Y_2^{[N^2]}(u-L(N))\big),
			\end{equation} 
			and let $\mu$ be a weak limit point of the sequence $(\mu_N)_{N\in\N}$. Let
			\begin{equation}
				\label{1a2alt}
				\Theta=\lim_{N\to\infty}\frac{1}{N}\sum_{i\in[N]}\frac{x_i^\mu+Ky_i^{\mu}}{1+K}\, \qquad \text{ in }L^2(\mu),
			\end{equation}
			and let $(Z^{\nu_\Theta}(t))_{t>0}=(X^{\nu_\Theta}(t),Y_0^{\nu_\Theta}(t),Y_1^{\nu_\Theta}(t),Y_2^{\nu_\Theta}(t))_{t>0}$ be the infinite system with components evolving according to \eqref{z0} with $\theta=\Theta$ and $y_{i,1,0}$ and $y_{i,2,0}$ determined by assumption \eqref{z3} and starting from its equilibrium measure. Extend the finite system $Z^{[N^2]}$ as a system on $([0,1]^4)^{\N_0}$ by periodic continuation. Construct $(Z^{[N^2]}(t))_{t> 0}$ and  $(Z^{\nu_\Theta}(t))_{t>0}$ on one probability space. Then there exists a sequence $(\bar{L}(N))_{N\in\N}$ such that $\lim_{N \to \infty}\bar{L}(N)=\infty,$ $\lim_{N \to \infty}\frac{\bar{L(N)}}{N}=0$ and
			\begin{equation}
				\label{932}
				\begin{aligned}
					&\lim_{N \to \infty} \E\left[\left|x_i^{[N^2]}(Ns)-x_i^{\nu_\Theta}(\bar{L}(N))\right|\right]
					+K_0\,\E\left[\left|y_{i,0}^{[N^2]}(Ns)-y_{i,0}^{\nu_\Theta}(\bar{L}(N))\right|\right]\\
					&+K_1\,\E\left[\left|y_{i,1}^{[N^2]}(Ns)-y_{i,1}^{\nu_\Theta}(\bar{L}(N))\right|\right]+K_2\,\E\left[\left|y_{i,2}^{[N^2]}(Ns)-y_{i,2}^{\nu_\Theta}(\bar{L}(N))\right|\right]=0,\,i\in[N].
				\end{aligned}
			\end{equation}
		\end{corollary}
		
		Note that Lemmas~\ref{l.comp2}, \ref{lem:2} and Corollary~\ref{cor3} do not only hold for sites $i$ in the $1$-block around $0$, but hold for for all sites $i\in[N^2]$, after we replace $\Theta^{[N]_0}_0$ by $\Theta^{[N]_i}_i$ .
		
		\subsubsection{Limiting evolution of the $1$-block estimator process}
		\label{s.1bl}
		
		\begin{proposition}{\bf [Limiting evolution of the 1-blocks]}
			\label{lem:1blev}
			Fix $t_2>0$.  Let  $(L(N))_{N\in\N}$ satisfy $\lim_{N \to \infty}L(N)=\infty$ and $\lim_{N \to \infty} L(N)/N=0$. Let $(N_k)_{k\in\N}$ be a subsequence such that
			\begin{equation}
				\label{as23}
				\begin{aligned}
					&\lim_{k \to \infty}\CL\left[\left({\bf\Theta}^{\eff, (2),[N_k^2]}(N_k^2t_2)\right)\right]=P_{t_2}(\cdot),\\
					&\lim_{k\to\infty}\CL\left[y_{2,1}^{[N_k^2]}(N_kt_2)\Big|{\bf \Theta}^{(2),[N_k^2]}(N_k^2t_2)\right] = P^{z_2(t_2)}, \\ 
					&\lim_{k\to\infty}\CL\left[\left(Y_{1,0}^{[N_k^2]}(N^2_kt_2+N_kt_1),Y_{2,0}^{[N_k^2]}(N^2_kt_2)\right)
					\Big|{\bf\Theta}^{\aux,(1),[N_k^2]}(N^2_kt_2+N_kt_1)\right] = P^{{z}_1^{\eff}(t_1)},\\
					&\lim_{k\to\infty }\CL\Bigg[\sup_{0\leq t\leq L(N_k)}\left|\bar{\Theta}^{(2),[N^2_k]}(N^2_k t_2)
					-\bar{\Theta}^{(2),[N^2_k]}(N_k^2t_2-N_kt)\right|\\
					&\qquad\qquad  +\left|{\Theta}_{y_2}^{(2),[N^2_k]}(N_k^2t_2)-{\Theta_{y_{2}}}^{(2),[N^2_k]}(N^2_kt_2-Nt)\right|\Bigg]
					=\delta_{0}.
				\end{aligned}
			\end{equation}
			Then, for the $1$-block around $0$,
			\begin{equation}
				\label{mgh6}
				\begin{aligned}
					\lim_{k \to \infty}\CL\left[{\bf\Theta}^{\aux, (1),[N_k^2]}(N_k^2t_2)\right]
					&=\int_{[0,1]^2}\int_{[0,1]} \Gamma^{\aux,(1)}_{u,y_{2,1}}\,P^{(u,v)}(\d y_{2,1})\,P_{t_2}(\d u,\d v),
				\end{aligned}
			\end{equation}	
			where $\Gamma^{\aux,(1)}_{u,y_{2,1}} $ is the equilibrium measure of \eqref{m1chb} with $\theta$ replaced by $u$, and
			\begin{equation}
				\label{mgh6e}
				\lim_{k \to \infty}\CL\left[\left({\bf\Theta}^{\aux, (1),[N_k^2]}(N_k^2t_2+N_kt_1)\right)_{t_1>0}\right]
				= \CL\left[(z_1^\aux(t_1))_{t_1> 0})\right],
			\end{equation}
			where $(z_1^\aux(t_1))_{t_1> 0}$ is the process evolving according to \eqref{m1chb} with $\theta$ replaced by the random variable $\bar{\Theta}^{(2)}(t_2)$ and with initial measure $\int_{[0,1]^2}\int_{[0,1]} \Gamma^{\aux,(1)}_{u,y_{2,1}}\,P^{(u,v)}(\d y_{2,1})\,P_{t_2}(\d u,\d v)$, and $y_{2,1}$ is a random variable.
		\end{proposition}
		Note that by tightness of the $2$-blocks and the assumptions in Proposition~\ref{P.finsysmf2lev}, we can always find a subsequence $(N_k)_{k\in\N}$ such that \eqref{as23} holds and also \eqref{z3} holds. To prepare for the proof of  Proposition~\ref{lem:1blev}, we prove four lemmas: Lemma~\ref{lem15} shows that the limiting $1$-block system has a unique equilibrium, Lemma~\ref{lemav3} implies convergence of the active $2$-block estimator and the combined $2$-block estimator, Lemma~\ref{stabest3} gives a regularity property for the $2$-block estimator, and Lemma~\ref{lem:aux} shows the limiting evolution of the auxiliary $1$-block estimator process. Lemma~\ref{cgam1} proves equation~\eqref{mgh6}. After we derive these lemmas we prove Proposition~\ref{lem:1blev}.
		
		\begin{lemma}{\bf [1-block equilibrium]}
			\label{lem15}
			For any initial distribution $\mu\in([0,1]^3)$, the process $(z_1^\aux(t_1))_{t_1> 0}$  evolving according to \eqref{m1chb} is well defined and converges to a unique equilibrium measure
			\begin{equation}
				\label{0}
				\lim_{t_1\to\infty}\CL[z^\aux_1(t_1)]=\Gamma_{\theta,y_{2,1}}^{\aux,(1)}.
			\end{equation} 
		\end{lemma}
		
		\begin{proof}
			By \cite{YW71}, the SSDE in \eqref{m1chb}  has a unique strong solution. By a similar argument as in the proof of Lemma~\ref{lemerg2}, the SSDE in \eqref{m1chb} converges to a unique equilibrium measure $\Gamma^{\aux,(1)}_{\theta,y_{2,1}}$. 
		\end{proof}
		
		\begin{remark}{\bf [Equilibrium measure]}
			{\rm Note that Lemma~\ref{lem15} still holds when we allow $\theta$ and $y_{2,1}$ to be the random variables $\bar{\Theta}(t_2)$ and $y_{2,1}$. Assuming \eqref{as23}, we can derive the distributions of $\bar{\Theta}(t_2)$ and $y_{2,1}$, and we can write the equilibrium as $\int_{[0,1]^2}\int_{[0,1]}\Gamma^{\aux,(1)}_{u,y_{2,1}}\,P^{(u,v)}(\d y_{2,1})\,P_{t_2}(\d u,\d v)$.  In what follows we abbreviate
				\begin{equation}
					\Gamma^{(1)}_{\bar{ \Theta}(t_2),y_{2,1,i}}=\int_{[0,1]^2}\int_{[0,1]}
					\Gamma^{\aux,(1)}_{u,y_{2,1,i}}\,P^{(u,v)}(\d y_{2,1,i})\,P_{t_2}(\d u,\d v).
				\end{equation}
			}\hfill$\blacksquare$
		\end{remark}
		
		\begin{lemma}{\bf [$2$-block averages]}
			\label{lemav3}
			Define
			\begin{equation}
				\begin{aligned}
					&\Delta^{(2),[N^2]}_\Sigma (Nt_1)=\frac{\Theta^{(2),[N^2]}_x(Nt_1)
						+K_0\Theta^{(2),[N^2]}_{y_0}(Nt_1)}{1+K_0}-\Theta^{(2),[N^2]}_{y_1}(Nt_1).
				\end{aligned}
			\end{equation}
			Then
			\small
			\begin{equation}
				\begin{aligned}
					&\E\left[\left|\Delta^{(2),[N^2]}_\Sigma (Nt_1)\right|\right]\\
					&\leq \sqrt{\E\left[\left(\Delta^{(2),[N^2]}_\Sigma (0)\right)^2\right]}
					\e^{-e_1\left(\frac{1+K_0+K_1}{1+K_0}\right)t_1}\\
					&+\sqrt{\int_0^{t_1} \d s\,2e_1\left(\frac{1+K_0+K_1}{1+K_0}\right)\e^{-2e_1\left(\frac{1+K_0+K_1}{1+K_0}\right)(t_1-s)}
						\E \left[\left|\bar{\Theta}^{(1),[N^2]}(Ns)-\Theta^{(1),[N^2]}_x(Ns)\right|\right]}\\
					&+\sqrt{\frac{1}{e_1}\left[\frac{K_2e_2}{N(1+K_0+K_1)}+\frac{||g||}{2N(1+K_0+K_1)}\right]}.
				\end{aligned}
			\end{equation}
			\normalsize
		\end{lemma}
		
		\begin{proof} 
			For the two-level mean-field system we have the following SSDE for the $2$-block averages:
			\small
			\begin{equation}
				\label{m8}
				\begin{aligned}
					\d\Theta^{(2),[N^2]}_x(Nt_1)
					&=\sqrt{\frac{1}{N^3}\sum_{i\in[N^2]}g(x^{[N^2]}_i(Nt_1))}\, \d \tilde{w}(t_1)\\
					&\qquad\qquad + NK_0 e_0\, \left[\Theta^{(2),[N^2]}_{y_0}(Nt_1)-\Theta^{(2),[N^2]}_x(Nt_1)\right]\,\d t_1\\
					&\qquad\qquad+ K_1 e_1\, \left[\Theta^{(2),[N^2]}_{y_1}(Nt_1)-\Theta^{(2),[N^2]}_x(Nt_1)\right]\,\d t_1\\
					&\qquad\qquad+ \frac{K_2 e_2}{N}\, \left[\Theta^{(2),[N^2]}_{y_2}(Nt_1)-\Theta^{(2),[N^2]}_x(Nt_1)\right]\,\d t_1,\\
					\d\Theta^{(2),[N^2]}_{y_0}(Nt_1)&=Ne_0\,\left[\Theta^{(2),[N^2]}_x(Nt_1)-\Theta^{(2),[N^2]}_{y_0}(Nt_1)\right]\, \d t_1,\\
					\d\Theta^{(2),[N^2]}_{y_1}(Nt_1)&=e_1\left[\Theta^{(2),[N^2]}_x(Nt_1)-\Theta^{(2),[N^2]}_{y_1}(Nt_1)\right]\, \d t_1,\\
					\d\Theta^{(2),[N^2]}_{y_2}(Nt_1)&=\frac{e_2}{N}\left[\Theta^{(2),[N^2]}_x(Nt_1)-\Theta^{(2),[N^2]}_{y_2}(Nt_1)\right]\, \d t_1.
				\end{aligned}
			\end{equation} 
			\normalsize
			Therefore
			\begin{equation}
				\begin{aligned}
					&\d \left(\Delta^{(2),[N^2]}_\Sigma (Nt_1)\right)^2
					=2 \Delta^{(2),[N^2]}_\Sigma (Nt_1)\, \d \Delta^{(2),[N^2]}_\Sigma (Nt_1)+\d \left<\Delta^{(2),[N^2]}_\Sigma (Nt_1)\right>\\
					&= 2 \Delta^{(2),[N^2]}_\Sigma (Nt_1)\,\frac{1}{1+K_0}\sqrt{\frac{1}{N^3}\sum_{i\in[N^2]}g(x^{[N^2]}_i(Nt_1))}\, \d \tilde{w}(t_1)\\
					&\qquad+2 \Delta^{(2),[N^2]}_\Sigma (Nt_1)\,\frac{K_1 e_1}{(1+K_0)}\, \left[\Theta^{(2),[N^2]}_{y_1}(Nt_1)
					-\Theta^{(2),[N^2]}_x(Nt_1)\right]\,\d t_1,\\	
					&\qquad+2 \Delta^{(2),[N^2]}_\Sigma (Nt_1)\, \frac{K_2 e_2}{N(1+K_0)}\, \left[\Theta^{(2),[N^2]}_{y_2}(Nt_1)
					-\Theta^{(2),[N^2]}_x(Nt_1)\right]\,\d t_1,\\
					&\qquad-2 \Delta^{(2),[N^2]}_\Sigma (Nt_1)\,e_1\left[\Theta^{(2),[N^2]}_x(Nt_1)-\Theta^{(2),[N^2]}_{y_1}(Nt_1)\right]\, \d t_1\\
					&\qquad+\frac{1}{(1+K_0)^2}\frac{1}{N^3}\sum_{i\in[N^2]} g(x^{[N^2]}_i(Nt_1))\,\d t_1.
				\end{aligned}
			\end{equation}
			Hence
			\small
			\begin{equation}
				\begin{aligned}
					&\frac{\d}{\d t} \E\left[\left(\Delta^{(2),[N^2]}_\Sigma (Nt_1)\right)^2\right]\\
					&=  -2e_1\left(\frac{1+K_0+K_1}{1+K_0}\right)\E\left[\left(\Delta^{(2),[N^2]}_\Sigma (Nt_1)\right)^2\right]\\
					&\qquad+2 e_1\left(\frac{1+K_0+K_1}{1+K_0}\right)\\
					&\qquad \qquad \times \E \left[\Delta^{(2),[N^2]}_\Sigma (Nt_1)
					\left(\frac{\Theta^{(2),[N^2]}_x(Nt_1)+K_0\Theta^{(2),[N^2]}_{y_0}(Nt_1)}
					{1+K_0}-\Theta^{(2),[N^2]}_x(Nt_1)\right)\right]\,\\
					&\qquad+\frac{K_2 e_2}{N(1+K_0)}\,2 \E\left[\Delta^{(2),[N^2]}_\Sigma (Nt_1)\,
					\left[\Theta^{(2),[N^2]}_{y_2}(Nt_1)-\Theta^{(2),[N^2]}_x(Nt_1)\right]\right]\,\\
					&\qquad+\frac{1}{(1+K_0)^2}\E\left[\frac{1}{N^3}\sum_{i\in[N^2]} g(x^{[N^2]}_i(Nt_1))\right],	
				\end{aligned}
			\end{equation}
			\normalsize
			and therefore
			\begin{equation}
				\begin{aligned}
					&\E\left[\left(\Delta^{(2),[N^2]}_\Sigma (Nt_1)\right)^2\right]\\
					&= \E\left[\left(\Delta^{(2),[N^2]}_\Sigma (0)\right)^2\right]
					\e^{-2e_1\left(\frac{1+K_0+K_1}{1+K_0}\right)t_1}+\int_0^{t_1} \d s\,
					\e^{-2e_1\left(\frac{1+K_0+K_1}{1+K_0}\right)(t_1-s)}h^{[N]}(s),
				\end{aligned}
			\end{equation}
			where 
			\begin{equation}
				\begin{aligned}
					&h^{[N]}( s) = 2e_1\left(\frac{1+K_0+K_1}{1+K_0}\right)\\
					&\qquad\qquad \times \E \left[ \Delta^{(2),[N^2]}_\Sigma (Ns)\, 
					\left(\frac{\Theta^{(2),[N^2]}_x(Ns)+K_0\Theta^{(2),[N^2]}_{y_0}(Ns)}{1+K_0}-\Theta^{(2),[N^2]}_x(Ns)\right)\right]\\
					&\qquad\qquad+\frac{2K_2 e_2}{N(1+K_0)}\,\E\left[ \Delta^{(2),[N^2]}_\Sigma (Ns)\,  
					\left[\Theta^{(2),[N^2]}_{y_2}(Ns)-\Theta^{(2),[N^2]}_x(Ns)\right]\right]\\
					&\qquad\qquad+\frac{1}{(1+K_0)^2 }\E\left[\frac{1}{N^3}\sum_{i\in[N^2]} g(x^{[N^2]}_i(Ns))\right].
				\end{aligned}
			\end{equation}
			Therefore
			\small
			\begin{equation}
				\begin{aligned}
					&\E\left[\left|\Delta^{(2),[N^2]}_\Sigma (Nt_1)\right|\right]\\
					&\leq \sqrt{\E\left[\left(\Delta^{(2),[N^2]}_\Sigma (0)\right)^2\right]}
					\e^{-e_1\left(\frac{1+K_0+K_1}{1+K_0}\right)t_1}\\
					&+\sqrt{\int_0^{t_1} \d s\,2e_1\left(\frac{1+K_0+K_1}{1+K_0}\right)\e^{-2e_1\left(\frac{1+K_0+K_1}{1+K_0}\right)(t_1-s)}
						\E \left[\left|\bar{\Theta}^{(1),[N^2]}(Ns)-\Theta^{(1),[N^2]}_x(Ns)\right|\right]}\\
					&+\sqrt{\frac{1}{e_1}\left[\frac{K_2e_2}{N(1+K_0+K_1)}+\frac{||g||}{2N(1+K_0+K_1)}\right]}.
				\end{aligned}
			\end{equation}
			\normalsize
		\end{proof}
		
		Let $\mu_{N_k}$ be the measure obtained by periodic continuation of the configuration 
		\begin{equation}
			\begin{aligned}
				Z^{[N^2]}(N_k^2t_2).
			\end{aligned}
		\end{equation}  
		Since the state space $([0,1]\times[0,1]^3)^{\N_0}$ is compact, we can pass to a further subsequence, to obtain 
		\begin{equation}
			\label{11106}
			\mu=\lim_{k \to \infty}\mu_{N_k}.
		\end{equation} 
		
		\begin{lemma}{\bf [Regularity for 2-block estimator]}
			\label{stabest3} 
			Let $\mu $ and $\mu_N$ be as defined above. Let $(x_i,y_{1,i},y_{2,i})_{i\in\N_0}$ be distributed according to $\mu$. Define the random variable 
			\begin{equation}
				\label{11107}
				\begin{aligned}
					&\phi=(\phi_1,\phi_2),\\
					&\phi_1=\lim_{n\to \infty}\frac{1}{n^2}\sum_{i\in[n^2]}
					\frac{x_i^{}+K_0y_{i,0}+K_1y_{i,1}^{}}{1+K_0+K_1},\qquad \phi_2
					= \lim_{n\to \infty}\frac{1}{n^2}\sum_{i\in[n^2]}y_{i,2},
				\end{aligned}
			\end{equation}
			and the random variable $\phi^{[N]}$ on $(\mu_N,([0,1]^3)^{\N_0})$ by putting
			\begin{equation}
				\label{1107}
				\begin{aligned}
					&\phi^{[N^2]}=(\phi^{[N^2]}_1,\phi^{[N^2]}_2),\\
					&\phi_1^{[N^2]}=\frac{1}{N^2}\sum_{i\in[N^2]}
					\frac{x_i^{[N^2]}+K_0y_{i,0}^{[N^2]}+K_1y_{i,1}^{[N^2]}}{1+K_0+K_1},\qquad \phi_2
					= \lim_{N\to \infty}\frac{1}{N^2}\sum_{i\in[N^2]}y^{[N^2]}_{i,2}.
				\end{aligned}
			\end{equation}
			Then 
			\begin{equation}
				\label{123}
				\lim_{N\to\infty}\CL[\phi^{[N^2]}]=\CL[\phi].
			\end{equation}
		\end{lemma}
		
		\begin{proof}
			Use a similar argument as in the proof of Lemma~\ref{stabest2}.
		\end{proof}
		
		We will first determine the limiting evolution of $({\bf\Theta}^{\aux,(1),[N_k^2]}(N_k^2t_2+Nt_1))_{t_1>0}$. To do so we consider all the $N_k$ $1$-blocks in $[N_k^2]$. After that we show that
		\be{}
		\lim_{k \to \infty}\CL\left[\left({\bf\Theta}_i^{\aux,(1),[N_k^2]}(N_k^2t_2)\right)_{i\in[N_k]}\right]
		=\prod_{i\in\N_0}\Gamma^{(1)}_{\bar{ \Theta}(t_2),y_{2,1,i}},
		\ee
		
		The limiting $1$-block process for the auxiliary estimator process (recall \eqref{m1chb}) is given by 
		\begin{equation}
			\begin{aligned}
				&({\bf z}_1^\aux(t))_{t> 0}=({\bf x}^\aux_1(t),{\bf y}^\aux_{1,1}(t),{\bf y}^\aux_{2,1}(t))_{t> 0},\\
				&{\bf z}_1^\aux(t)=({z}_{1,i}^\aux(t))_{i\in\N_0},\qquad{\bf x}_1^\aux(t)=({x}_{1,i}^\aux(t))_{i\in\N_0},\\
				&{\bf y}_{1,1}^\aux(t)=({y}_{1,1,i}^\aux(t))_{i\in\N_0}\qquad {\bf y}_{2,1}^\aux(t)=({y}_{2,i}^\aux(t))_{i\in\N_0}
			\end{aligned}
		\end{equation}
		and its components evolve according to
		\begin{equation}
			\label{m1chbs}
			\begin{aligned}
				\d x^{\aux}_{1,i}(t) &= \frac{1}{1+K_0}\Bigg[ c_{1} [\bar{\Theta}^{(2)}(t_2) - x^{\aux}_{1,i}(t)]\, \d t 
				+ \sqrt{(\CF^{(1)}g)(x^{\aux}_{1,i}(t))}\, \d w (t)\\
				&\qquad + K_{1} e_{1}\, [y^{\aux}_{1,1,i}(t)-x^{\aux}_{1,i}(t)]\,\d t\Bigg]\,,\\
				\d y^{\aux}_{1,1,i}(t) &= e_{1}\,[x^{\aux}_{1,i}(t)-y^{\aux}_{1,1,i}(t)]\,\d t,\\
				y^\aux_{2,1,i}(t) &=y_{2,1,i},\qquad i\in\N_0,
			\end{aligned}
		\end{equation}
		where 
		\be{}
		\bar{\Theta}^{(2)}(t_2)=\lim_{N \to \infty}\sum_{i\in[N^2]}\frac{x_i^{[N^2]}+K_0y_{i,0}^{[N^2]}
			+K_1y_{i,1}^{[N^2]}}{1+K_0+K_1}\text{ in }L_2(\mu).
		\ee
		Let $\mu^{(1)}_{N_k}$ be the law obtained by periodic continuation of $({\bf\Theta}_i^{\aux,(1),[N_k^2]}(N_k^2t_2))_{i\in[N_k]}$, and let $\mu^{(1)}=\lim_{k \to \infty}\mu^{(1)}_{N_k}$ be any weak limit point of the sequence $(\mu^{(1)}_{N_k})_{k\in\N}$. 
		
		\begin{lemma}{\bf [Limiting evolution of auxiliary $1$-block estimator]}
			\label{lem:aux}
			Let $\CL[({\bf z}_1^\aux(0))]=\mu^{(1)}$. Then the following hold.
			\begin{enumerate}
				\item  
				For all $t_1>0$ and $i\in[N_k]$, 
				\begin{equation}
					\begin{aligned}
						\lim_{k \to \infty}&\E\Bigg[(1+K_0)\left(x_{1,i}^\aux(t_1)-\bar{\Theta}_i^{\aux,(1),[N_k^2]}(N_k^2t_2+N_kt_1)\right)^2\\
						&\qquad+K_1\left(y_{1,1,i}(t_1)-\Theta_{y_1,i}^{\aux,(1),[N_k^2]}(N_k^2t_2+N_kt_1)\right)^2\\
						&\qquad+K_2\left(y_{2,1,i}^{\aux,(1),[N_k^2]}(t_1)-\Theta_{y_2,i}^{\aux,(1),[N_k]}(N_k^2t_2+N_kt_1)\right)^2\Bigg]=0.
					\end{aligned}
				\end{equation}
				\item 
				For all $t_2>0$,
				\begin{equation}
					\label{s12}
					\lim_{k \to \infty}\CL\left[({\bf\Theta}^{\aux,(1),[N_k^2]}(N_k^2t_2+Nt_1))_{t_1>0}\right]
					=\CL[({\bf z}_1^\aux(t_1))_{t_1> 0}].
				\end{equation}
			\end{enumerate}
		\end{lemma}
		
		\begin{proof}
			Abbreviate
			\begin{equation}
				\begin{aligned}
					\Delta_i^{(1),[N_k^2]}(N_kt_1)&=x_{1,i}^\aux(t_1)-\Theta_i^{\aux,(1),[N_k^2]}(N_k^2t_2+N_kt_1),\\
					\delta_{y_1,i}^{(1),[N_k^2]}(N_kt_1)&=y_{1,1,i}(t_1)-\Theta_{y_1,i}^{\aux,(1),[N_k^2]}(N_k^2t_2+N_kt_1),\\
					\delta_{y_2,i}^{(1),[N_k^2]}(N_kt_1)&=y_{2,1,i}^{\aux}(t_1)-\Theta_{y_2,i}^{\aux,(1),[N_k^2]}(N_k^2t_2+N_kt_1).
				\end{aligned}
			\end{equation}
			Extending $({\bf\Theta}^{\aux,(1),[N_k^2]}(N_kt_1))_{t_1>0}$ as a process on $\N_0$ by periodic continuation, we can construct $({\bf z}_1^\aux(t_1))_{t_1>0}$ and $ ({\bf\Theta}^{\aux,(1)}(N_kt_1))_{t_1>0}$ on one probability space such that 
			\begin{equation}
				\label{124}
				\lim_{k \to \infty}{\bf\Theta}^{\aux,(1),[N_k^2]}(N_k^2t_2)={\bf z}_1^{\aux,(1)}(0) \quad a.s.
			\end{equation}
			We couple the processes $({\bf z}_1^\aux(t_1))_{t_1>0}$ and $({\bf\Theta}^{\aux,(1),[N_k^2]}(N_k^2t_2+N_kt_1))_{t_1>0}$ by using the same Brownian motions for both processes. By It\^o-calculus we obtain for the coupled process (recall \eqref{m21}) 
			\begin{equation}
				\label{325}
				\begin{aligned}
					&\E\left[(1+K_0)\left(\Delta^{(1),[N_k^2]}_i(N_kt_1)\right)^2+K_1\left(\delta^{(1),[N_k^2]}_{y_1,i}(N_kt_1)\right)^2
					+K_2\left(\delta^{(1),[N_k^2]}_{y_2,i}(N_kt_1)\right)^2\right]\\
					= &\E\left[(1+K_0)\left(\Delta^{(1),[N_k^2]}_i(0)\right)^2+K_1\left(\delta^{(1),[N_k^2]}_{y_1,i}(0)\right)^2
					+K_2\left(\delta^{(1),[N_k^2]}_{y_2,i}(0)\right)^2\right]\\
					&-2c_1\int_0^{t_1}\E\left[\left(\Delta^{(1),[N_k^2]}_i(N_ks)\right)^2\right]\d s\\
					&-2K_1e_1\int_0^{t_1}\E\left[\left(\Delta^{(1),[N_k^2]}_i(N_ks)-\delta^{(1),[N_k^2]}_{y_1,i}(N_ks)\right)^2\right]\d s\\
					&+2c_1\int_0^{t_1}\E\left[\Delta^{(1),[N_k^2]}_i(N_ks)\left(\Theta^{(2)}(t_2)
					-\frac{1}{N_k^2}\sum_{j\in[N_k^2]}x^{[N_k^2]}_j(N^2_kt_2+N_ks)\right)\right]\d s\\
					&+\left(K_1e_1+c_1\right)\int_0^{t_1}\E\left[\Delta^{(1),[N_k^2]}_i(N_ks)\right. \\
					&\qquad\qquad\qquad\left.\times\left[\frac{1}{N_k}
					\sum_{j\in[N_k]_i}x^{[N_k^2]}_j(N_k^2t_2+N_ks)-\bar{ \Theta}^{\aux,(1),[N_k^2]}_i(N_k^2t_2+N_ks)\right]\right]\d s\\
					&+2K_1e_1\int_0^{t_1}\E\left[\delta^{(1),[N_k^2]}_{y_1,i}(N_ks)\right.\\
					&\qquad\qquad\qquad\times\left.\left[\bar{ \Theta}_i^{\aux,(1),[N_k^2]}(N^2_kt_2+N_ks)
					-\frac{1}{N_k}\sum_{j\in[N_k]_i}x^{[N_k^2]}_j(N^2_kt_2+N_ks)\right]\right]\d s\\
					&+2 \frac{K_2e_2}{N_k}\int_0^{t_1}\E\Bigg[\left[\delta^{(1),[N_k^2]}_{y_1,i}(N_ks)
					-\Delta^{(1),[N_k^2]}_i(N^2_kt_2+N_ks)\right]\\
					&\qquad\qquad\qquad\times\left[\frac{1}{N_k}\sum_{j\in[N_k]_i}x^{[N_k^2]}_j(N^2_ks)
					-\Theta^{\aux,(1),[N_k^2]}_{y_2,i}(N^2_kt_2+Nt_1)\right]\Bigg]\d s\\
					&+(1+K_0)^2\int_0^{t_1}\E\left[\left(\sqrt{\left(\CF g\right) (x_1^{\aux}(s))}-\sqrt{ \frac{1}{N}
						\sum_{i\in[N]} g (x^{[N_k^2]}_i(N^2_kt_2+N_ks))}\right)^2\right]\d s.			
				\end{aligned}
			\end{equation}
			
			Note that $|\Delta^{(1),[N_k^2]}_i|\leq 1$ and $|\delta^{(1),[N_k^2]}_{y_1,i}|\leq1 $. Note that the first term tends to $0$ by \eqref{124}. We show by dominated convergence that also all other positive terms in the right-hand side of \eqref{325} tend to $0$ as $k\to\infty$. 
			
			For the third term, we can estimate
			\small
			\begin{equation}
				\label{m10}
				\begin{aligned}
					&\lim_{k\to \infty}\E\Bigg[\frac{c_1}{1+K_0}\Bigg|\frac{1}{N_k^2}
					\sum_{i\in[N_k^2]}x^{[N_k^2]}_j(N_k^2t_2+N_ks)-\bar\Theta^{(2)}(t_2)\Bigg|\Bigg]\\
					&\leq \lim_{k\to \infty} \E\Bigg[\frac{c_1}{1+K_0}\Bigg|\Theta^{(2),[N_k^2]}_x(N_k^2t_2+N_ks)\\
					&\qquad\qquad\qquad\qquad\qquad-\frac{\Theta^{(2),[N_k^2]}_x(N_k^2t_2+N_ks)
						+K_0 \Theta^{(2),[N_k^2]}_{y_0}(N_k^2t_2+N_ks)}{1+K_0}\Bigg|\Bigg]\\
					&\quad+\E\Bigg[\frac{c_1}{1+K_0}\Bigg|\frac{\Theta^{(2),[N_k^2]}_x(N_k^2t_2+N_ks)
						+K_0 \Theta^{(2),[N_k^2]}_{y_0}(N_k^2t_2+N_ks)}{1+K_0}\\
					&\qquad\, -\frac{\Theta^{(2),[N_k^2]}_x(N_k^2t_2+N_ks)+K_0 \Theta^{(2),[N_k^2]}_{y_0}(N_k^2t_2+N_ks)
						+K_1 \Theta^{(2),[N_k^2]}_{y_1}(N_k^2t_2+N_ks)}{1+K_0+K_1}\Bigg|\Bigg]\\
					&\quad+\E\Bigg[\frac{c_1}{1+K_0}\Bigg|\frac{\Theta^{(2),[N_k^2]}_x(N_k^2t_2+N_ks)+K_0 \Theta^{(2),[N_k^2]}_{y_0}(N_k^2t_2+N_ks)
						+K_1 \Theta^{(2),[N_k^2]}_{y_1}(N_k^2t_2+N_ks)}{1+K_0+K_1}\\
					&\qquad\qquad\qquad\qquad-\bar\Theta^{(2)}(t_2)\Bigg|\Bigg].
				\end{aligned}
			\end{equation}
			\normalsize
			The first term in \eqref{m10} tends to zero by Lemma~\ref{lemlev1}, the second term tends to zero by Lemma~\ref{lemav3}, while the third term tends to zero by Lemma~\ref{stabest3} and Lemma~\ref{lem:sta3}, which is the third assumption in \eqref{as23}. Hence the third term in \eqref{325} tends to zero by dominated convergence as $k\to\infty$. 
			
			The fourth and fifth term in \eqref{325} tend to zero by Lemma~\ref{lemlev1} and dominated convergence. The sixth term in \eqref{325} tends to zero because the integral is bounded by $t_1$ and there is a factor $\frac{1}{N_k}$ in front. To see that the last term in the right-hand side in \eqref{325} tends to zero, recall that the subsequence $N_k$ is chosen such that 	
			\be{}
			\lim_{k \to \infty}\CL\left[({\bf\Theta}^{\aux,(1),[N_k^2]}(N_k^2t_2+Nt_1))_{t_1>0}\right]
			\ee 
			exists. Note that
			\begin{equation}
				\label{s5}
				\begin{aligned}
					&\E\left[\left(\sqrt{(\CF g) (x_1^{\aux}(s))}
					-\sqrt{ \frac{1}{N}\sum_{i\in[N]} g (x^{[N_k^2]}_i(N^2_kt_2+N_ks))}\right)^2\right]\\
					&\leq \E\left[\left|(\CF g) (x_1^{\aux}(s))- \frac{1}{N}\sum_{i\in[N]} g (x^{[N_k^2]}_i(N^2_kt_2+N_ks))\right|\right],
				\end{aligned}
			\end{equation}
			and hence we can apply a similar reasoning as in \eqref{a5} to see that \eqref{s5} tends to zero as $k\to\infty$. Therefore we obtain 
			\begin{equation}
				\label{s23}
				\begin{aligned}
					&\lim_{k \to \infty}\E\left[(1+K_0)(\Delta^{(1),[N_k^2]}_i(N_kt_1))^2+K_1(\delta^{(1),[N_k^2]}_{y_1,i}(N_kt_1))^2
					+K_2(\delta^{(1),[N_k^2]}_{y_2,i}(N_kt_1))^2\right]\\
					&\qquad\qquad=0.
				\end{aligned}
			\end{equation}
			To prove \eqref{s12}, note that \eqref{s23} implies convergence of the finite-dimensional distributions of $({\bf\Theta}^{\aux,(1),[N_k^2]}(N_k^2t_2+N_kt_1))_{t_1>0}$ by a similar argument as given below \eqref{35}. By Lemma~\ref{tonebl} we see that the laws of the processes 
			\be{}
			\left(\CL\left[({\bf\Theta}^{\aux,(1),[N_k^2]}(N_k^2t_2+N_kt_1))_{t_1>0}\right]\right)_{k\in\N_0}
			\ee 
			are tight. Therefore \eqref{s12} indeed holds. 
		\end{proof} 
		
		\begin{remark}
			{\rm Note that in the proof of Lemma~\ref{lemlimev} we could have proceeded as in the proof of Lemma~\ref{lem:aux}, instead of using the criterion in \cite[Theorem 3.3.1]{JM86}.}\hfill$\blacksquare$
		\end{remark}
		
		\begin{lemma}{\bf [Proof of \eqref{mgh6}]}
			\label{cgam1}
			Under the assumptions in Proposition~\ref{lem:1blev},
			\begin{equation}
				\label{mgh46}
				\begin{aligned}
					\lim_{k \to \infty}\CL\left[{\bf\Theta}^{\aux, (1),[N_k^2]}(N_k^2t_2)\right]
					&=\int_{[0,1]^2}\int_{[0,1]} \Gamma^{\aux,(1)}_{u,y_{2,1}}\,P^{(u,v)}(\d y_{2,1})\,P_{t_2}(\d u,\d v).
				\end{aligned}
			\end{equation}
		\end{lemma}
		
		\begin{proof} 
			For ease of notation, we drop the subsequence notation in this proof. Let $(t_n)_{n\in\N_0}$ be any sequence satisfying $\lim_{n\to\infty} t_n = \infty$ and $\lim_{n\to\infty} t_n/n = 0$. For each $t_n$, let $\mu^{(1)}_{N,t_n}$ be the law obtained by periodic continuation of the configuration of $({\bf \Theta}_i^{\aux,(1),[N^2]}(N^2t_2-Nt_n))_{i\in[N]}$. Recall that, since our state space is compact, the sequence $(\mu^{(1)}_{N,t_n})_{N\in\N}$ is tight. Let $\mu^{(1)}_{t_n}$ be any weak limit point of the sequence $(\mu^{(1)}_{N,t_n})_{N\in\N}$.
			
			Let $\CL[{\bf z}_1^{\aux, n}(0)]$ be the law obtained by periodic continuation of ${\bf \Theta}^{\aux,(1),[N^2]}(N^2t_2-Nt_n)$. By Lemma~\ref{tonebl} we know that the sequence 
			\be{}
			\left(\CL\left[\left({\bf \Theta}_i^{\aux,(1),[N^2]}(N^2t_2-Nt_n+Nt_1)\right)_{t_1>0,i\in[N]}\right]\right)_{N\in\N}
			\ee 
			is tight and hence for each $t_n$ we can pass to a subsequence such that 
			\begin{equation}
				\lim_{k \to \infty}\CL\left[\left({\bf \Theta}_i^{\aux,(1),[N_k^2]}(N_k^2t_2-N_kt_{n}+N_kt_1)\right)_{t_1>0, i\in[N]}\right]
			\end{equation}
			exists. By Lemmas~\ref{lem:sta3}--\ref{stabest3}, we obtain for all $t_n$ that
			\begin{equation}
				\lim_{N \to \infty}\bar{ \Theta}^{(2),[N_k^2]}(N_k^2t_2-Nt_n)=\bar{ \Theta}^{(2)}(t_2)\text{ in probability}.
			\end{equation}
			Then, by \eqref{325} in the proof of Lemma~\ref{lem:aux}, for fixed $t_n$ and all $i\in[N]$,
			\begin{equation}
				\label{7175}
				\begin{aligned}
					\lim_{k \to \infty}
					&\E\Bigg[(1+K_0)\left(x_{1,i}^{\aux,n}(t_n)-\bar{\Theta}_i^{\aux,(1),[N_k^2]}(N_k^2t_2-N_kt_n+N_kt_n)\right)^2\\
					&\qquad+K_1\left(y_{1,1,i}^{\aux,n}(t_n)-\Theta_{y_1,i}^{\aux,(1),[N_k^2]}(N_k^2t_2-N_kt_n+N_kt_n)\right)^2\\
					&\qquad+K_2\left(y_{2,1,i}^{\aux,n}(t_n)-\Theta_{y_2,i}^{\aux,(1),[N_k^2]}(N_k^2t_2-N_kt_n+N_kt_n)\right)^2\Bigg]=0.
				\end{aligned}
			\end{equation}
			By contradiction we can argue that 
			\begin{equation}
				\label{7184}
				\begin{aligned}
					\lim_{N \to \infty}
					&\E\Bigg[(1+K_0)\left(x_{1,i}^{\aux,n}(t_n)-\bar{\Theta}_i^{\aux,(1),[N^2]}(N^2t_2-Nt_n+Nt_n)\right)^2\\
					&\qquad+K_1\left(y_{1,1,i}^{\aux,n}(t_n)-\Theta_{y_1,i}^{\aux,(1),[N^2]}(N^2t_2-Nt_n+Nt_n)\right)^2\\
					&\qquad+K_2\left(y_{2,1,i}^{\aux,n}(t_n)-\Theta_{y_2,i}^{\aux,(1),[N^2]}(N^2t_2-Nt_n+Nt_n)\right)^2\Bigg]=0.
				\end{aligned}
			\end{equation}
			To see why, suppose that \eqref{7184} does not hold. Then for any $\delta>0$ we can construct a sequence $(N_l)_{l>0}$ such that, for $l\in\N$,
			\begin{equation}
				\label{7193}
				\begin{aligned}
					&\E\Bigg[(1+K_0)\left(x_{1,i}^{\aux,n}(t_n)-\bar{\Theta}_i^{\aux,(1),[N_l^2]}(N_l^2t_2-N_lt_n+N_lt_n)\right)^2\\
					&\qquad+K_1\left(y_{1,1,i}^{\aux,n}(t_n)-\Theta_{y_1,i}^{\aux,(1),[N_l^2]}(N_l^2t_2-N_lt_n+N_lt_n)\right)^2\\
					&\qquad+K_2\left(y_{2,1,i}^{\aux,n}(t_n)-\Theta_{y_2,i}^{\aux,(1),[N_l^2]}(N_l^2t_2-N_lt_n+N_lt_n)\right)^2\Bigg]>\delta.
				\end{aligned}
			\end{equation}
			However, also the sequence
			\begin{equation}
				\left(\CL\left[\left({\bf \Theta}_i^{\aux,(1),[N_l^2]}(N_l^2t_2-N_lt_n+N_lt_1)\right)_{t_1>0,i\in[N]}\right]\right)_{l\in\N}
			\end{equation}
			is tight. Hence we can pass to a further subsequence $(N_{\tilde l})_{\tilde{l}\in \N}$ for which \eqref{7175} holds. But this contradicts \eqref{7193}. We conclude that \eqref{7184} indeed holds. Moreover the argument holds for all $t_n$, so that \eqref{7184} holds for all $t_n$.
			
			Hence for every $t_n$ there exists a $N_n$ such that, for all $N\geq N_n$, 
			\begin{equation}
				\begin{aligned}
					&\E\Bigg[(1+K_0)\left(x_{1,i}^{\aux,n}(t_n)-\bar{\Theta}_i^{\aux,(1),[N_k^2]}(N^2t_2)\right)^2\\
					&\qquad+K_1\left(y_{1,1,i}^{\aux,n}(t_n)-\Theta_{y_1,i}^{\aux,(1),[N_k^2]}(N^2t_2)\right)^2\\
					&\qquad+K_2\left(y_{2,1,i}^{\aux,n}(t_n)-\Theta_{y_2,i}^{\aux,(1),[N_k^2]}(N^2t_2)\right)^2\Bigg]<\frac{1}{n}.
				\end{aligned}
			\end{equation}
			In particular, we may require that $N_n>N_{n-1}$. Setting $N=N_n$, we obtain
			\begin{equation}\label{7215}
				\begin{aligned}
					\lim_{n\to\infty}&\E\Bigg[(1+K_0)\left(x_{1,i}^{\aux,n}(t_n)-\bar{\Theta}_i^{\aux,(1),[N^2_{n}]}(N_{n}^2t_2)\right)^2\\
					&\qquad+K_1\left(y_{1,1,i}^{\aux,n}(t_n)-\Theta_{y_1,i}^{\aux,(1),[N_{n}^2]}(N_{n}^2t_2)\right)^2\\
					&\qquad+K_2\left(y_{2,1,i}^{\aux,n}(t_n)-\Theta_{y_2,i}^{\aux,(1),[N_{n}^2]}(N_{n}^2t_2)\right)^2\Bigg]=0.
				\end{aligned}
			\end{equation}
			If we can prove that 
			\begin{equation}
				\label{claim}
				\lim_{n\to\infty}\CL[z_i^{\aux, n}(t_n)]=\Gamma^{(1)}_{\bar{ \Theta}(t_2)},
			\end{equation}
			then we are done. To see why, note that, for all $f\in\CC_b([0,1]\times[0,1]^2)$, $f$ Lipschitz continuous
			\begin{equation}
				\begin{aligned}
					&\left|\E[f({\bf \Theta}_i^{\aux,(1),[N^2]}(N^2t_2))]-\E^{\Gamma^{(1)}_{\bar{ \Theta}(t_2)}}[f]\right|\\
					\leq&\left|\E[f({\bf \Theta}_i^{\aux,(1),[N^2]}(N^2t_2))-f({ z_{1,i}}^{\aux,n}(t_n))]\right|
					+\left|\E[f({ z_{1,i}}^{\aux,n}(t_n))]-\E^{\Gamma^{(1)}_{\bar{ \Theta}(t_2)}}[f]\right|.
				\end{aligned}
			\end{equation}
			Therefore if \eqref{claim} holds, then for all $\epsilon>0$ we can choose $\bar{n}$ such that, for all $n>\bar{n}$,
			\begin{equation}
				\left|\E[f({ z_{1,i}}^{\aux,n}(t_n))]-\E^{\Gamma^{(1)}_{\bar{ \Theta}(t_2)}}[f]\right|<\frac{\epsilon}{2}.
			\end{equation} 
			By \eqref{7215} we can find a $\hat{n}>\bar{n}$ such that for all $n>\hat{n}$
			\begin{equation}
				\left|\E[f({\bf \Theta}_i^\aux(N_n^2t_2))-f({ z_{1,i}}^{\aux,n}(t_{\hat{n}}))]\right|<\frac{\epsilon}{2}.
			\end{equation}
			Using \eqref{claim} and the fact that the Lipschitz functions are dense in $f\in\CC_b([0,1]\times[0,1]^2)$, we obtain \eqref{mgh46}.
			
			\paragraph{Proof of \eqref{claim}.}
			
			We use that any two systems $({\bf z}^{\aux,1}(t_1))_{t_1>0}$ and $({\bf z}^{\aux,2}(t_1))_{t_1>0}$ evolving according to \eqref{m1chb}, and having the same $y_{2,1}$-components and the same $\bar{ \Theta}(t_2)$, can be constructed on one probability space and can coupled by their Brownian motions. We obtain, for a component $i\in\N_0$, 
			\begin{equation}
				\begin{aligned}
					&\E[|x_{1,i}^{\aux,1}(t_n)-x_{1,i}^{\aux,2}(t_n)|  +K_1|y_{1,1,i}^{\aux,1}(t_n)-y_{1,1,i}^{\aux,2}(t_n)|
					+K_2|y_{2,1,i}^{\aux,1}(t_n)-y_{2,1,i}^{\aux,2}(t_n)| ]\\
					&=\E[|x_{1,i}^{\aux,1}(0)-x_{1,i}^{\aux,2}(0)|+K_1|y_{1,1,i}^{\aux,1}(0)-y_{1,1,i}^{\aux,2}(0)|+K_2|y_{2,1,i}^{\aux,1}(0)
					-y_{2,1,i}^{\aux,2}(0)|  ] \\
					&-c\int_0^{t_n}\E[|x_{1,i}^{\aux,1}(s)-x_{1,i}^{\aux,2}(s)|]\d s\\
					&-2K_1e_1\int_0^{t_n}\E\big[[|x_{1,i}^{\aux,1}(s)-x_{1,i}^{\aux,2}(s)|+K_1|y_{1,1,i}^{\aux,1}(s)-y_{1,1,i}^{\aux,2}(s)|]\\
					&\qquad\qquad\qquad\qquad\times 1_{\{\sign(x_{1,i}^{\aux,1}(s)-x_{1,i}^{\aux,2}(s))
						\neq\sign(y_{1,1,i}^{\aux,1}(s)-y_{1,1,i}^{\aux,2}(s))\}}\big]\d s. 
				\end{aligned}
			\end{equation}
			Therefore the difference between these two systems monotonically decreases.
			
			Since the state space $[0,1]\times[0,1]^2$ is compact, the sequence of laws 
			\be{}
			(\CL[{z}_i^{\aux, n}(0)])_{n\in\N}
			\ee 
			is tight. Therefore we can find converging subsequences such that 
			\begin{equation}
				\lim_{k \to \infty}\CL[{z_i}^{\aux, n_k}(0)]=\mu
			\end{equation}
			for some probability measure $\mu$ on $[0,1]\times[0,1]^2$.
			
			Let $(z^{\aux,0}(t_1))_{t_1>0}$ be the limiting system evolving according to \eqref{m1chb} and starting from initial distribution $\mu$. By Skorohod's theorem, we can construct the sequence of limiting systems $((z^{\aux,n_k}(t_1))_{t_1>0})_{k\in\N}$ and $(z^{\aux,0}(t_1))_{t_1>0}$ on one probability space such that 
			\begin{equation}
				\lim_{k \to \infty}{z}^{\aux, n_k}(0)=z^{\aux,0}(0) \quad  a.s.
			\end{equation}
			Use the coupling of Brownian motions to obtain
			\begin{equation}
				\label{s1}
				\begin{aligned}
					&\E[|x_{1,i}^{\aux,n_k}(t_{n_k})-x_{1,i}^{\aux,0}(t_{n_k})|  +K_1|y_{1,1,i}^{\aux,n_k}(t_{n_k})-y_{1,1,i}^{\aux,0}(t_{n_k})|\\
					&\qquad +K_2|y_{2,1,i}^{\aux,n_k}(t_{n_k})-y_{2,1,i}^{\aux,0}(t_{n_k})| ]\\
					&=\E[|x_{1,i}^{\aux,n_k}(0)-x_{1,i}^{\aux,0}(0)|+K_1|y_{1,1,i}^{\aux,n_k}(0)-y_{1,1,i}^{\aux,0}(0)|
					+K_2|y_{2,1,i}^{\aux,n_k}(0)-y_{2,1,i}^{\aux,0}(0)|  ] \\
					&-c\int_0^{t_{n_k}}\E[|x_{1,i}^{\aux,n_k}(s)-x_{1,i}^{\aux,0}(s)|]\d s\\
					&-2K_1e_1\int_0^{t_{n_k}}\E\big[[|x_{1,i}^{\aux,n_k}(s)-x_{1,i}^{\aux,0}(s)|+K_1|y_{1,1,i}^{\aux,n_k}(s)-y_{1,1,i}^{\aux,0}(s)|]\\
					&\qquad\qquad\qquad\qquad\times 1_{\{\sign(x_{1,i}^{\aux,n_k}(s)-x_{1,i}^{\aux,0}(s))\neq\sign(y_{1,1,i}^{\aux,n_k}(s)
						-y_{1,1,i}^{\aux,0}(s))\}}\big]\d s. 
				\end{aligned}
			\end{equation}
			Taking the limit $k\to\infty$ on both sides of \eqref{s1}, we obtain
			\begin{equation}
				\begin{aligned}
					&\lim_{k \to \infty}\E[|x_{1,i}^{\aux,n_k}(t_{n_k})-x_{1,i}^{\aux,0}(t_{n_k})| + K_1|y_{1,1,i}^{\aux,n_k}(t_{n_k})-y_{1,1,i}^{\aux,0}(t_{n_k})|\\
					&\qquad +K_2|y_{2,1,i}^{\aux,n_k}(t_{n_k})-y_{2,1,i}^{\aux,0}(t_{n_k})| ]=0.
				\end{aligned}
			\end{equation}
			Note that $\lim_{n\to\infty} t_n=\infty$ implies that $\lim_{k\to\infty}t_{n_k}=\infty$, so $z_i^{\aux,0}$ is the limiting system in \eqref{m1chb} with $\theta$ replaced by the random variable $\bar{ \Theta}(t_2)$ and $y_{2,1,i}$. Therefore we can condition on $\bar{\Theta}(t_2)$ and $y_{2,1,i}$, and use the assumption in \eqref{z3}, to obtain
			\begin{equation}
				\lim_{k\to\infty}\CL\left[z_i^{\aux,0}(t_{n_k})\right]=\int_{[0,1]^2}\int_{[0,1]} \Gamma^{\aux,(1)}_{u,y_{2,1}}\,P^{(u,v)}(\d y_{2,1})\,P_{t_2}(\d u,\d v).
			\end{equation}
			Hence we conclude that
			\begin{equation}
				\label{s7}
				\lim_{k\to\infty}\CL[z_i^{\aux, n_k}(t_{n_k})]=\int_{[0,1]^2}\int_{[0,1]} \Gamma^{\aux,(1)}_{u,y_{2,1}}\,P^{(u,v)}(\d y_{2,1})\,P_{t_2}(\d u,\d v).
			\end{equation}
			Equation~\eqref{s7} holds for all subsequences along which the initial distribution converges, 
			\begin{equation}
				\lim_{k \to \infty}{z}_i^{\aux, n_k}(0)=z_i^{\aux,0}(0) \quad  a.s.
			\end{equation}
			We will show that this implies \eqref{claim}. 
			
			Suppose that 
			\begin{equation}
				\lim_{n\to\infty}\CL[z_i^{\aux, n}(t_n)]\neq \int_{[0,1]^2}\int_{[0,1]} \Gamma^{\aux,(1)}_{u,y_{2,1}}\,P^{(u,v)}(\d y_{2,1})\,P_{t_2}(\d u,\d v).
			\end{equation}
			Then there exist $f\in\CC_b([0,1]\times[0,1]^3)$ and $\delta>0$ such that for all $N\in\N$ there exists an $n\in\N$, $n>N$ such that 
			\begin{equation}
				\label{star}
				\left|\E[f(z_i^{\aux,n}(t_n))]-\E^{\Gamma^{(1)}_{\bar{ \Theta}(t_2)}}[f]\right|>\delta.
			\end{equation}
			Hence we can construct a subsequence  $(z_i^{\aux, n_k}(t_1))_{t_1>0,\, k\in\N}$ such that \eqref{star} holds for each $k\in\N$. However, also for this sequence $(\CL[z_i^{\aux, n_k}(0)])_{ k\in\N}$ is tight. Passing to a possibly further subsequence of converging initial distributions, we argue like before to obtain that along this subsequence 
			\begin{equation}
				\label{starretje}
				\lim_{k\to\infty}\left|\E[f(z_i^{\aux,n_k}(t_{n_k}))]-\E^{\Gamma^{(1)}_{\bar{ \Theta}(t_2)}}[f]\right|=0.
			\end{equation}
			This contradicts \eqref{star} and so \eqref{claim} is indeed true.
		\end{proof}
		
		\paragraph{Proof of Proposition~\ref{lem:1blev}}
		
		\begin{proof}
			Lemma~\ref{cgam1} implies \eqref{mgh6}. Therefore Lemma~\ref{lem:aux} implies \eqref{mgh6e}. 
		\end{proof}
		
		\subsubsection{Convergence of 2-block process}
		\label{sss.2bl}
		
		In this section we derive the limiting evolution of the effective $2$-block process.
		
		\begin{lemma}{\bf [Convergence of the 2-block averages]}
			\label{lem2blcon}
			Assume that $(N_k)_{k\in\N}\subset\N$ is a subsequence satisfying
			\begin{equation}
				\begin{aligned}
					&\lim_{k\to\infty}\CL\left[y_{2,1}^{[N_k^2]}(N_kt_2)\Big|{\bf \Theta}^{(2),[N_k^2]}(N_k^2t_2)\right] = P^{z_2(t_2)}, \\ 
					&\lim_{k\to\infty}\CL\left[\left(Y_{1,0}^{[N_k^2]}(N^2_kt_2+N_kt_1),Y_{2,0}^{[N_k^2]}(N^2_kt_2)\right)
					\Big|{\bf\Theta}^{\aux,(1),[N_k^2]}(N^2_kt_2+N_kt_1)\right]\\
					&\qquad\qquad= P^{{z}_1^{\eff}(t_1)}.
				\end{aligned}
			\end{equation}
			Then, for the effective $2$-block estimator process defined in \eqref{1114},
			\begin{equation}
				\label{m12a}
				\lim_{k\to\infty} \CL \left[\left({\bf\Theta}^{\eff,(2),[N_k^2]}(N_k^2t_2)\right)_{t_2 > 0}\right] 
				= \CL \left[\left(z_2^\eff(t_2)\right)_{t_2 > 0}\right],
			\end{equation}
			where the limit is determined by the unique solution of the SSDE \eqref{m64c} with initial state
			\begin{equation}
				\begin{aligned}
					z_2^\eff(0)=\left(x_2^\eff(0),y^\eff_2(0)\right)= \left(\vartheta_1,\theta_{y_2}\right).
				\end{aligned}
			\end{equation}
		\end{lemma}
		
		\begin{proof}
			Again we use \cite[Theorem 3.3.1]{JM86}. By a similar argument as used in the proof of Lemma~\ref{lemlimev} we can show that
			\begin{equation}
				\lim_{t_2 \downarrow 0}\CL \left[{\bf\Theta}^{\eff,(2),[N_k^2]}(N_k^2t_2)\right]
				=\delta_{(\vt_1,\theta_{y_2})}.
			\end{equation}
			Note that by steps 1-4 of the scheme in Section~\ref{ss.org22} we can choose the subsequence $(N_k)_{k\in\N}$ such that both \eqref{z3} and \eqref{as23} hold. Since we already established the tightness of the $2$-block in Lemma~\ref{lem:t3}, we are left to show that, for all $t_2>0$, 
			\begin{equation}
				\label{ga}
				\begin{aligned}
					\lim_{N \to \infty}&\E^{}\Big[\Big|G_\dagger^{(2),[N_k^2]}(f,{\bf\Theta}^{\eff,(2),[N_k^2]}(N_k^2t_2),t_2,\omega)\\
					&\qquad\qquad\qquad-G^{(2)}f\left({\bf\Theta}^{\eff,(2),[N_k^2]}(N_k^2t_2)\right)\Big|\Big]=0,
				\end{aligned}
			\end{equation}
			where $G_\dagger^{(2),[N_k^2]}$ is the $\CD$-semi-martingale operator defined in \eqref{m455}, $G^{(2)}$ is the generator of the process $(z_2^\eff(t_2))_{t_2>0}$ defined in \eqref{m64c}, and both generators work on a probability space driven by one set of Brownian motions. Note that, for all $t_2>0$,
			\begin{equation}
				\begin{aligned}
					&\E^{}\Big[\Big|G_\dagger^{(2),[N_k^2]}\left(f,\left(\bar{\Theta}^{(2),[N_k^2]}(N_k^2t_2),
					\Theta^{(2),[N_k^2]}_{y_2}(N_k^2t_2)\right),t_2,\omega\right)\\
					&\qquad\qquad -G^{(2)}f\left(\bar{\Theta}^{(2),[N_k^2]}(N_k^2t_2),\Theta^{(2),[N_k^2]}_{y_2}(N_k^2t_2)\right)\Big|\Big]\\
					&\leq  \frac{K_2e_2}{1+K_0+K_1}\E\left[\left|\bar{\Theta}^{(2),[N_k^2]}(N_k^2t_2)
					-\Theta_x^{(2),[N_k^2]}(N_k^2t_2,\omega)\right|\left|\,\frac{\partial f}{\partial x}\right|\right]\\
					&\qquad+e_2\E^{}\left[\left|\Theta_x^{(2),[N_k^2]}(N_k^2t_2,\omega)
					-\bar{\Theta}^{(2),[N_k^2]}(N_k^2t_2)\right|\left|\,\frac{\partial f}{\partial y}\right|\right]\\
					&\qquad+\frac{1}{(1+K_0+K_1)^2}\\
					&\qquad\qquad \times \E^{}\left[\left|\frac{1}{N^2_k}\sum_{i\in[N_k^2]}g(x_i^{[N^2_k]}(N_k^2t_2,\omega))
					-(\CF^{(2)}g)(\bar{\Theta}^{(2),[N_k^2]}(N_k^2t_2))\right|\left|\,\frac{\partial^2 f}{\partial x^2}\right|\right].
				\end{aligned}
			\end{equation}	
			The first and second term on the right-hand side tend to $0$ as $k\to\infty$ by a similar argument as used in \eqref{m10} and below. For the third let $[N]_i$ denote the $1$-block that contains site $i$ and let $(z^{\nu_{\bar{\Theta}_i^{(1)}}}(t))_{t>0}$ be the limiting single colony system, with drift towards the random variable $\bar{\Theta}_i$ and starting from the equilibrium measure $\nu_{\bar{\Theta}_i^{(1)}}$. We construct the single colony system $Z^{[N^2_k]}(N^2_kt_2-L(N)+t)_{t\geq 0}$ and the limiting system $(z^{\nu_{\bar{\Theta}_i^{(1)}}}(t))_{t>0}$ on one probability space, such that by Skohorod's theorem, we can assume that the convergence is almost surely. Note that $\bar{\Theta}_i$ is the limiting one block. Then we can write 
			\begin{equation}
				\label{k1}
				\begin{aligned}
					&\E^{}\left[\left|\frac{1}{N^2_k}\sum_{i\in[N_k^2]}g(x_i^{[N_k^2]}(N_k^2t_2,\omega))
					-(\CF^{(2)}g)(\bar{\Theta}^{(2),[N_k^2]}(N_k^2t_2))\right|\left|\frac{\partial^2 f}{\partial x^2}\right|\right]\\
					&\qquad\leq \frac{1}{N_k}\sum_{i\in[N]}\E^{}\left[\left|\frac{1}{N_k}\sum_{j\in[N]_i}g(x_j^{[N_k^2]}(N_k^2t_2,\omega))
					-\frac{1}{N_k}\sum_{j\in[N]_i}g(x_j^{\nu_{\bar{\Theta}_i^{(1)}}}(L(N))\right|\left|\frac{\partial^2 f}{\partial x^2}\right|\right]\\
					&\qquad+ \frac{1}{N_k}\sum_{i\in[N]}\E^{}\left[\left|\frac{1}{N_k}\sum_{j\in[N]_i}g(x_j^{\nu_{\bar{\Theta}_i^{(1)}}}(L(N))
					-(\CF^{(1)}g)(\bar{\Theta}_i^{(1)})\right|\left|\frac{\partial^2 f}{\partial x^2}\right|\right]\\
					&\qquad+ \E^{}\left[\left|\frac{1}{N_k}\sum_{i\in[N]}(\CF^{(1)}g)(\bar{\Theta}_i^{(1)})
					-(\CF^{(2)}g)(\bar{\Theta}_i^{(2)})\right|\left|\frac{\partial^2 f}{\partial x^2}\right|\right]\\ 
					&\qquad+\E^{}\left[\left|(\CF^{(2)}g)(\bar{\Theta}_i^{(2)})
					-(\CF^{(2)}g)(\bar{\Theta}_i^{(2)}(N_k^2t_2))\right|\left|\frac{\partial^2 f}{\partial x^2}\right|\right].
				\end{aligned}
			\end{equation}	
			The first term on the right-hand side tends to zero by Lipschitz continuity for $g$ and Corollary~\ref{cor3}. The second term tends to zero by the law of large numbers, since the limiting single colonies are i.i.d. given the value of the random variable $\bar{\Theta}_i^{(1)}$. The third term tends to zero since by Proposition~\ref{lem:1blev} also the limiting $1$-blocks become independent given the value of the $2$-block. Hence we can again apply the law of large numbers. Finally, for the last term, note that since we construct the single components and the limiting process on one probability space, we can argue like in the proof of Lemma~\ref{stabest2} that
			\begin{equation}
				\lim_{k\to\infty}\E\left[\left|\bar{\Theta}_i^{(2)}-\bar{\Theta}_i^{(2)}(N_k^2t_2)\right|\right]=0.
			\end{equation}
			Hence the last term tends to zero by the Lipschitz property of $\CF^{(2)}g$.
		\end{proof}
		
		\begin{remark}
			{\rm Instead of \cite[Theorem 3.3.1]{JM86} we could have used a similar strategy as in the proof of Lemma~\ref{lem:aux} to obtain Lemma~\ref{lem2blcon}.}\hfill$\blacksquare$
		\end{remark}
		
		\subsubsection{State of the slow seed-banks}
		\label{ss.ssb}
		
		On time scale $t_0$, i.e., space-time scale $0$, the colour-$1$ seed-bank is a ``slow seed-bank," since it does not move on this time scale. Because we study the two-layer three-colour mean-field system from time $N^2t_2$ onwards, the $1$-block averages of the colour $1$-dormant population are already in equilibrium. As a consequence we can exactly describe the single $1$-dormant colonies, which turn out to be in a state that equals the current $1$-block average of the dormant population of colour $1$. To obtain the formal result we will first prove the following lemma.
		
		\begin{lemma}{\bf [Slow seed-banks]}
			\label{lem:ds}
			Fix $t_2,t_1>0$, for $i\in[N^2]$ and all $t_0\geq 0$,
			\begin{equation}
				\label{14c}
				\lim_{N\to\infty}\left[ y^{[N^2]}_{i,1}(N^2t_2+Nt_1+t_0)-\Theta_{y_1,i}^{(1),[N^2]}(N^2t_2+Nt_1+t_0)\right]=0 \qquad \text{a.s.},
			\end{equation}
			where $\Theta_{y_1,i}^{(1),[N^2]}$ is the $1$-block average to which $y^{[N^2]}_{i,1}$ contributes.
		\end{lemma}
		
		To prove Lemma~\ref{lem:ds}, we need the kernel $b^{[N^2]}(\cdot,\cdot)$ defined in \ref{mrw}, which becomes in the current setting 
		\small
		\begin{equation}
			\label{mrwN1}
			b^{[N^2]}((i,R_i), (j, R_j)) = \left\{ \begin{array}{ll}
				1_{\{d_{[N^2]}(i,j)\leq 1\}}\frac{c_0}{N}+\frac{c_1}{N^3},
				&\text{ if } R_i = R_j =A,\\
				K_m\frac{e_m}{N^m}, &\text{ if } i = j,\ R_i=A,\ R_j = D_m,\,m\in\{0,1,2\}, \\
				\frac{e_m}{N^m}, &\text{ if } i = j,\ R_i=D_m,\ R_j = A,\,m\in\{0,1,2\},\\
				0, &\mbox{ otherwise}.
			\end{array}
			\right.
		\end{equation}
		\normalsize
		The corresponding semigroup of the kernel $b^{[N^2]}(\cdot,\cdot)$ is denoted by $b^{[N^2]}_t(\cdot,\cdot)$.
		
		To prove Lemma~\ref{lem:ds} we will use the following lemma, which was proved in \cite{GdHOpr1}[Lemma 6.1] and for our setting reads as follows.
		
		\begin{lemma}{\bf [First and second moment]}
			\label{lem1cg9422}
			$\mbox{}$\\
			Let $\E_{z^{[N^2]}}$ the expectation if the process start from some state $z^{[N^2]}\in ([0,1]\times[0,1]^3)^{[N^2]}$. For $z^{[N^2]}\in ([0,1]\times[0,1]^2)^{[N^2]}$, $t\geq 0$ and $(i,R_i),(j,R_j) \in [N^2]\times\{A,D_0,D_1,D_2\}$,
			\begin{equation}
				\label{eqexprw21}
				\E_{z^{[N^2]}}[z^{[N^2]}_{(i,R_i)}(t)]=\sum_{\substack{(k,R_k)\in\\ \Omega_N\times\{A,D_0,D_1,D_2\}}} 
				b^{[N^2]}_t\big((i,R_i),(k,R_k)\big)\,z^{[N^2]}_{(k,R_k)}
			\end{equation}
			and 
			\begin{equation}
				\label{eqexprw22}
				\begin{aligned}
					&\E_{z^{[N^2]}}[z^{[N^2]}_{(i,R_i)}(t)z^{[N^2]}_{(j,R_j)}(t)]\\
					&= \sum_{\substack{(k,R_k),(l,R_l)\in\\  \Omega_N\times\{A,D_0,D_1,D_2\}}} b^{[N^2]}_t\big((i,R_i),(k,R_k)\big)\,
					b^{[N^2]}_t\big((j,R_j),(l,R_l)\big)\,z^{[N^2]}_{(k,R_k)}z^{[N^2]}_{(l,R_l)}\\ 
					&\quad +\,2\int_0^{{t}} \d s \sum_{k\in \Omega_N} 
					b^{[N^2]}_{(t-s)}((i,R_i),(k,A))\,b^{[N^2]}_{(t-s)}((j,R_j),(k,A))\,\E_z^{[N^2]}[g(x^{[N^2]}_k(s))].
				\end{aligned}
			\end{equation}
		\end{lemma}
		
		\begin{proof}[Proof of Lemma~\ref{lem:ds}] 
			The argument is given in such a way that it can easily be generalised to more complicated systems, which we treat later. Let $\bar{t}(N)=N^2 t_2+Nt_1+t_0$. We will show that if $i,j\in[N]_i$, i.e., $i$ and $j$ belong to the same $1$-block, then
			\begin{equation}
				\label{14b}
				\lim_{N\to\infty}\E\left[\left(y^{[N^2]}_{i,1}(\bar{t}(N))-y^{[N^2]}_{j,1}(\bar{t}(N))\right)^2\right]=0.
			\end{equation}
			This implies \eqref{14c}. By Lemma~\ref{lem1cg9422}, we can write
			\begin{equation}
				\label{1039}
				\begin{aligned}
					&\E\left[\left(y^{[N^2]}_{i,1}(\bar{t}(N))-y^{[N^2]}_{j,1}(\bar{t}(N))\right)^2\right]\\
					&= \sum_{\substack{(k,R_k),(l,R_l)\in\\  [N^2]\times\{A,D_0,D_1,D_2\}}} \left(b^{[N^2]}_{\bar{t}(N)}
					\left((i,D_1),(k,R_k)\right)-b^{[N^2]}_{\bar{t}(N)}\left((j,D_1),(k,R_k)\right)\right)\\
					&\,\times\left(b^{[N^2]}_{\bar{t}(N)}\left((i,D_1),(l,R_l)\right)-b^{[N^2]}_{\bar{t}(N)}
					\left((j,D_1),(l,R_l)\right)\right)\E[z^{[N^2]}_{(k,R_k)}z^{[N^2]}_{(l,R_l)}]\\
					&+\,2\int_0^{\bar{t}(N)} \d s \sum_{k\in [N^2]} 
					\left(b^{[N^2]}_{(\bar{t}(N)-s)}((i,D_1),(k,A))-b^{[N^2]}_{(\bar{t}(N)-s)}((j,D_1),(k,A))\right)^2\\
					&\,\times\E[g(x^{[N^2]}_k(s))].
				\end{aligned}
			\end{equation}
			Using a coupling argument, we show that both terms in \eqref{1039} tend to $0$ as $N\to\infty$. To prove that the first term tends to $0$, we will show that 
			\begin{equation}
				\label{1040}
				\lim_{N\to\infty}\sum_{\substack{(k,R_k)\in\\  [N^2]\times\{A,D_0,D_1,D_2\}}} 
				\left|b^{[N^2]}_{\bar{t}(N)}\left((i,D_1),(k,R_k)\right)-b^{[N^2]}_{\bar{t}(N)}\left((j,D_1),(k,R_k)\right)\right|=0.
			\end{equation}
			
			To do so, let $(RW^{[N^2]}(t))_{t\geq 0}$ and $(RW^{\prime[N^2]}(t))_{t\geq 0}$ be two independent random walks, starting from $RW^{[N^2]}(0)=(i,D_1)$ and $RW^{\prime[N^2]}(0)=(j,D_1)$, where $i$ and $j$ are in the same $1$-block. Let $RW^{[N^2]}$ and $RW^{\prime[N^2]}$  both evolve according to the kernel $b^{[N^2]}(\cdot,\cdot)$, so $b^{[N^2]}_t(\cdot,\cdot)$ is their corresponding semigroup. Since $RW^{[N^2]}$ and $RW^{\prime[N^2]}$ both start from the colour $1$-seed-bank, we can perfectly couple their switches between $A$, $D_0$, $D_1$  and $D_2$. Since this implies that both $RW^{[N^2]}$ and $RW^{\prime[N^2]}$ are always simultaneously active, we can also couple the times when they jump due to migration and the distance over which they migrate. However, we do not couple their migrations, i.e. $RW^{[N^2]}$ and $RW^{\prime[N^2]}$ jump at the same time and over the same distance, but they can jump to different sites. This implies that the coupled process $(RW^{[N^2]}(t),RW^{\prime[N^2]}(t))_{t\geq 0}$ has transition rates
			\begin{equation}
				((i,R_i),(j,R_i))\rightarrow\begin{cases}
					((k,A),(l,A))&\text{ if }R_i=R_j=A\text{ and }d_{[N^2]}(i,k)=d_{[N^2]}(j,l)\\
					&\text{ at rate }1_{\{d_{[N^2]}(i,j)\leq 1\}}\frac{c_0}{N}+\frac{c_1}{N^3},\\
					((i,D_m),(j,D_m))&\text{ if }R_i=R_j=A\text{ at rate }\frac{K_me_m}{N^m},\, m\in\{0,1,2\},\\
					((i,A),(j,A))&\text{ if }R_i=R_j=D_m\text{ at rate }\frac{e_m}{N^m},\, m\in\{0,1,2\}.	
				\end{cases}
			\end{equation}
			Define the event
			\begin{equation}
				\label{1041}
				H_t^{[N^2]}=\{RW^{[N^2]} \text{ has migrated at least once up to time } t\}.
			\end{equation} 
			Note that if $H_t^{[N^2]}$ has happened, then also $RW^{\prime[N^2]}$ has migrated. Hence
			\begin{equation}
				\begin{aligned}
					b^{[N^2]}_{\bar{t}(N)}&\left((i,D_1),(k,R_k)\right)-b^{[N^2]}_{\bar{t}(N)}\left((j,D_1),(k,R_k)\right)\\
					&= \P_{(i,D_1)}\left(RW^{[N^2]}(\bar{t}(N))=(k,R_k)\right)-\P_{(j,D_1)}
					\left(RW^{\prime[N^2]}(\bar{t}(N))=(k,R_k)\right)\\
					&= \tilde{\P}_{(i,D_1),(j,D_1)}\left(RW^{[N^2]}(\bar{t}(N))=(k,R_k),\, H_{\bar{t}(N)}^{[N^2]}\right)\\
					&\qquad +\tilde{\P}_{(i,D_1),(j,D_1)}\left(RW^{[N^2]}(\bar{t}(N))=(k,R_k),\, (H_{\bar{t}(N)}^{[N^2]})^c\right)\\
					&\qquad -\tilde{\P}_{(i,D_1),(j,D_1)}\left(RW^{\prime[N^2]}(\bar{t}(N))=(k,R_k),\, 
					H_{\bar{t}(N)}^{[N^2]}\right)\\
					&\qquad -\tilde{\P}_{(i,D_1),(j,D_1)}\left(RW^{\prime[N^2]}(\bar{t}(N))
					=(k,R_k),\, (H_{\bar{t}(N)}^{[N^2]})^c\right)\\
					&= \tilde{\P}_{(i,D_1),(j,D_1)}\left(RW^{[N^2]}(\bar{t}(N))=(k,R_k),\, 
					(H_{\bar{t}(N)}^{[N^2]})^c\right)\\
					&\qquad -\tilde{\P}_{(i,D_1),(j,D_1)}
					\left(RW^{\prime[N^2]}(\bar{t}(N))=(k,R_k),\, (H_{\bar{t}(N)}^{[N^2]})^c\right),
				\end{aligned}
			\end{equation}
			where the last equality follows because, once the random walks have just jumped once, $RW^{[N^2]}$ and $RW^{\prime[N^2]}$ are uniformly distributed over $[N]\times A$ if their jump horizon was $1$ and they are uniformly distributed of $[N^2]\times A$ if they jumped over distance $2$. Hence if $H_t^{[N^2]}$ has occurred, then $RW^{[N^2]}$ and $RW^{\prime[N^2]}$ have the same distribution. Therefore
			\begin{equation}
				\sum_{\substack{(k,R_k)\in\\  [N^2]\times\{A,D_0,D_1,D_2\}}} \left|b^{[N^2]}_{\bar{t}(N)}\left((i,D_1),(k,R_k)\right)
				-b^{[N^2]}_{\bar{t}(N)}\left((j,D_1),(k,R_k)\right)\right|\leq 2 \tilde{\P}^{}((H_{\bar{t}(N)}^{[N^2]})^c)
			\end{equation}
			and we are left to show that
			\begin{equation}
				\label{1045}
				\lim_{N\to\infty}\tilde{\P}((H_{\bar{t}(N)}^{[N^2]})^c)=0.
			\end{equation}
			The event $(H_{\bar{t}(N)}^{[N^2]})^c$ occurs either when the random walks do not wake up before time $\bar{t}(N)$ or when the random walks wake up before time $\bar{t}(N)$ but do not migrate. By the coupling we only have to consider one of the random walks. Therefore the probability that $RW^{[N^2]}$ and $RW^{\prime[N^2]}$ do not wake up before time $\bar{t}(N)$ is given by
			\begin{equation}
				\tilde{\P}_{(i,D_1),(j,D_1)}\left(RW^{[N^2]} \text{ does not wake up before }\bar{t}(N)\right)
				=\e^{-\frac{e_1}{N}\bar{t}(N)}=\e^{-\frac{e_1(N^2t_2+Nt_1+t_0)}{N}}
			\end{equation} 
			and hence
			\begin{equation}
				\label{1047}
				\lim_{N\to\infty}\P_{(i,D_1),(j,D_1)}\left(RW^{[N^2]} \text{ does not wake up before }\bar{t}(N)\right)=0.
			\end{equation}
			The probability that the random walks do wake up, but do not migrate is a little more complicated, since each time they wake up with positive probability they go to sleep before they migrate. Define
			\begin{equation}
				\label{1048}
				\begin{aligned}
					C^{[N^2]}(t)&=\{\#\, \text{times $RW^{[N^2]}$ gets active before time }t\},\\
					T_A^{[N^2]}(t)&=\{\text{total time $RW^{[N^2]}$ is active up to time }t\},\\
					T_D^{[N^2]}(t)&=\{\text{total time $RW^{[N^2]}$ is dormant up to time }t\}.
				\end{aligned}
			\end{equation}
			Thus, $C^{[N^2]}(t)$ counts the number of active/dormant cycles. Define $T^{[N^2]}_{A,n},\ T^{[N^2]}_{D,n}$ as the active respectively, dormant time during the $n$th cycle. Define
			\begin{equation}
				\chi=K_0e_0+\frac{K_1e_1}{N}+\frac{K_2e_2}{N^2},
			\end{equation}
			so $\chi$ is the total rate at which $RW$ and $RW^\prime$ become dormant when they are active. Define
			\begin{equation}
				c=c_0+\frac{c_1}{N},
			\end{equation} 
			so $c$ is the total rate at which $RW$ and $RW^\prime$ migrate when they are active. Then 
			\begin{equation}
				T^{[N^2]}_A(t)=\sum_{n=1}^{C^{[N^2]}(t)}T^{[N^2]}_{A,n},\qquad T^{[N^2]}_D(t)
				=\sum_{n=1}^{C^{[N^2]}(t)}T^{[N^2]}_{D,n}, 
			\end{equation}
			where $T^{[N^2]}_{A,n} \,{\buildrel d \over =}\, \exp(\chi)$ and $T^{[N^2]}_{D,n}\,{\buildrel d \over =}\, \frac{1}{\chi}K_0e_0\,\exp(e_0)+\frac{1}{\chi}\frac{K_1e_1}{N}\,\exp(\frac{e_1}{N})+\frac{1}{\chi}\,\frac{K_2e_2}{N^2}\,\exp(\frac{e_2}{N^2})$. Once awake, $RW^{[N^2]}$ migrates at rate $c$ and hence the probability to migrate before time $\bar{t}(N)$ is given by  $1-\e^{-c T^{[N^2]}_A(\bar{t}(N))}$. Therefore we are left to show that
			\begin{equation}
				\lim_{N\to\infty}cT_A^{[N^2]}(\bar{t}(N))=\lim_{N\to\infty}c\sum_{n=1}^{C^{[N^2]}(\bar{t}(N))}T^{[N^2]}_{A,n}=\infty,\qquad a.s.
			\end{equation}
			Since $T^{[N^2]}_{A,n} \,{\buildrel d \over =}\, \text{exp}(\chi)$, it is enough to show that
			\begin{equation}
				\label{1051}
				\lim_{N\to\infty}C^{[N^2]}(\bar{t}(N))=\infty\qquad a.s. 
			\end{equation}
			
			To do so, we assume the contrary, i.e., there exists an $R\in \N$ such that for all $\bar{N}\in\N$ there exists an $N>\bar{N}$ such that
			\begin{equation}
				\P_{(i,D_1)} ( C^{[N^2]}(\bar{t}(N))\leq R)>0.
			\end{equation} 
			Let $L(N)$ be such that  $\lim_{N \to \infty}L(N)=\infty$ and $\lim_{N \to \infty} L(N)/N=0$. Note that, by \eqref{1047}, we can condition on the first wake-up time and estimate
			\begin{equation}
				\label{007}
				\begin{aligned}
					\P_{(i,D_1)} (C^{[N^2]}(\bar{t}(N))\leq R)
					&= \int_0^{\bar{t}(N)} \d s\,\P_{(i,A)} (C^{[N^2]}(\bar{t}(N)-s)\leq R)\ \frac{e_1}{N} \e^{-\frac{e_1}{N}s}\\
					&= \int_0^{\bar{t}(N)-L(N)} \d s\,\P_{(i,A)} (C^{[N^2]}(\bar{t}(N)-s)\leq R)\ \frac{e_1}{N} \e^{-\frac{e_1}{N}s}\\
					&\quad +\int_{\bar{t}(N)-L(N)}^{\bar{t}(N)} \d s\,
					\P_{(i,A)} (C^{[N^2]}(\bar{t}(N)-s)\leq R)\ \frac{e_1}{N} \e^{-\frac{e_1}{N}s}\\
					&\leq \P_{(i,A)} (C^{[N^2]}(L(N))\leq R)+\e^{-\frac{e_1}{N}\bar{t}(N)}\left[\e^{\frac{e_1}{N}L(N)}-1\right].
				\end{aligned}
			\end{equation}
			Note that the second term in the last inequality tends to $0$ as $N\to\infty$. For the first term, note that we are now looking at time $L(N)$, i.e., time scale $N^0$. Since 
			\begin{equation}
				\label{08}
				\begin{aligned}
					&\lim_{N\to\infty}\P_{i,A}\left( RW^{[N^2]} \text{ jumps to $D_1$ or $D_2$ before time }L(N)\right)\\
					&\qquad = \lim_{N\to\infty} 1-\e^{-\left(\frac{K_1e_1}{N}+\frac{K_2e_2}{N^2}\right)L(N)}=0.
				\end{aligned}
			\end{equation}
			we have $\lim_{N\to\infty}\P_{i,A}\left( \{RW^{[N^2]}(s)\in\{A,D_0\}\text{ for }s\in[0,L(N)]\}\right)=1$. Hence, conditioned on the event $\{RW^{[N^2]}\in\{A,D_0\}\}$, $T_{A,n}^{[N^2]} \,{\buildrel d \over =}\, \exp(K_0e_0)$ and $T_{D,n}^{[N^2]} \,{\buildrel d \over =}\, \exp(e_0)$. We therefore obtain
			\begin{equation}
				\label{008}
				\begin{aligned}
					\P_{(i,A)} (C^{[N^2]}(L(N))\leq R)
					&=\P_{(i,A)}\left(\sum_{n=1}^R (T_{A,n}^{[N^2]}+ T_{D,n}^{[N^2]})\geq  L(N)\right)\\
					&\leq \frac{R}{L(N)}\E_{(i,A)}\left[T_{A,n}^{[N^2]}+ T_{D,n}^{[N^2]}\right]\\
					&= \frac{R}{L(N)}\left[\frac{1}{K_0e_0}+\frac{1}{e_0}\right].
				\end{aligned}
			\end{equation} 
			Taking the limit $N\to\infty$ in \eqref{008} and combining this with \eqref{007}, we conclude that \eqref{1051} indeed holds. Hence also \eqref{1045} and \eqref{1040} hold.
			
			We are left to show that 
			\begin{equation}
				\label{1056}
				\begin{aligned}
					&\lim_{N\to\infty} 2\int_0^{\bar{t}(N)}\, \d s \sum_{k\in [N^2]} 
					\left(b^{[N^2]}_{(\bar{t}(N)-s)}((i,D_1),(k,A))-b^{[N^2]}_{(\bar{t}(N)-s)}((j,D_1),(k,A))\right)^2\\ 
					&\qquad \qquad \qquad\qquad\qquad \times \E[g(x_k(s))]=0.
				\end{aligned}
			\end{equation}
			Also here the idea is to make a similar coupling. As soon as the random walks migrate, they are equally distributed. On time scale $N$, after waking up from the colour $1$ seed-bank they will almost immediately migrate, since migration happens on time scale $1$, i.e., by time $L(N)$ they have migrated with probability tending to $1$. This will again be the key to show that \eqref{1056} tends to $0$ as $N\to\infty$.
			
			Note that, by \eqref{1039}, 
			\begin{equation}
				\label{1057}
				\begin{aligned}
					&\Bigg|2\int_0^{\bar{t}(N)} \d s \sum_{k\in [N^2]} 
					\left(b^{[N^2]}_{(\bar{t}(N)-s)}((i,D_1),(k,A))-b^{[N^2]}_{(\bar{t}(N)-s)}((j,D_1),(k,A))\right)^2\\ 
					&\qquad \qquad \qquad \qquad \times \E[g(x_k^{[N^2]}(s))]\Bigg|\leq 2.
				\end{aligned}
			\end{equation}
			We will again use the coupling in \eqref{1041}. Define
			\begin{equation}
				\label{1058}
				\tau^{[N^2]}=\inf\{t\geq 0: RW^{[N^2]}(t)=(k,A)\text{ for some }k\in[N^2]\}.
			\end{equation}
			Then, for all $l\in[N^2]$,
			\begin{equation}
				\label{1059a}
				\P_{(l,D_1)}(\tau^{[N^2]}\leq t)=1-\e^{-\frac{e_1}{N}t}.
			\end{equation}
			Setting $s=\bar{t}(N)-s$, we can rewrite the integral in \eqref{1056} as
			\begin{equation}
				\label{1059}
				\begin{aligned}
					&2\int_0^{\bar{t}(N)} \d s \sum_{k\in [N^2]} 
					\Big(\tilde{\P}^{[N^2]}_{(i,D_1),(j,D_1)}\left(RW^{[N^2]}(s)=(k,A)\right)\\
					&\qquad-\tilde{\P}^{[N^2]}_{(i,D_1),(j,D_1)}\left(RW^{\prime[N^2]}(s)=(k,A),\right)\Big)^2 
					\E[g(x_k^{[N^2]}(\bar{t}(N)-s))]\\
					&= 2\int_0^{\bar{t}(N)} \d s
					\sum_{k\in [N^2]}\Bigg[\int_0^s\d r\,\tilde{\P}(\tau^{[N^2]}=r)\\ 
					&\qquad \left(\tilde{\P}^{[N^2]}_{(i,A),(j,A)}\left(RW^{[N^2]}(s-r)=(k,A)\right)-\tilde{\P}^{[N^2]}_{(i,A),(j,A)}
					\left(RW^{\prime[N^2]}(s-r)=(k,A)\right)\right)\Bigg]\\
					&\qquad\times
					\Big(\tilde{\P}^{[N^2]}_{(i,D_1),(j,D_1)}\left(RW^{[N^2]}(s)=(k,A)\right)-\tilde{\P}^{[N^2]}_{(i,D_1),(j,D_1)}
					\left(RW^{\prime[N^2]}(s)=(k,A)\right)\Big)\\
					&\qquad\times \E[g(x_k^{[N^2]}(\bar{t}(N)-s))].
				\end{aligned}
			\end{equation}
			In what follows we will abbreviate
			\begin{equation}
				\begin{aligned}
					&P_{\Delta^{[N^2]}_{(i,A),(j,A)}}(t,((l,R_l),(j,R_j)))\\
					&=\tilde{\P}^{}_{(i,A),(j,A)}\left(RW^{[N^2]}(t)=(l,R_l)\right)-\tilde{\P}^{}_{(i,A),(j,A)}
					\left(RW^{\prime[N^2]}(t)=(j,R_j)\right),
				\end{aligned}
			\end{equation}
			and similarly, for $m\in\{0,1,2\}$,
			\begin{equation}
				\begin{aligned}
					&P_{\Delta^{[N^2]}_{(i,D_m),(j,D_m)}}(t,((l,R_l),(j,R_j)))\\
					&=\tilde{\P}^{}_{(i,D_m),(j,D_m)}\left(RW^{[N^2]}(t)=(l,R_l)\right)-\tilde{\P}^{}_{(i,D_m),(j,D_m)}
					\left(RW^{\prime[N^2]}(t)=(j,R_j)\right).
				\end{aligned}
			\end{equation}
			By \eqref{1057}, we can use Fubini to swap the order of integration and subsequently substitute $v=s-r$, to obtain
			\small
			\begin{equation}
				\label{1060}
				\begin{aligned}
					&2\int_0^{\bar{t}(N)}\, \d r\,\tilde{\P}(\tau^{[N^2]}=r) \int_0^{\bar{t}(N)-r}\d v
					\sum_{k\in [N^2]} P_{\Delta^{[N^2]}_{(i,A),(j,A)}}(v,((k,A),(k,A)))\\ 
					&\qquad \times P_{\Delta^{[N^2]}_{(i,D_1),(j,D_1)}}(v+r,((k,A),(k,A))) 
					\E[g(x_k^{[N^2]}(\bar{t}(N)-r-v))]\\
					&= 2\int_0^{\bar{t}(N)} \d r\,\tilde{\P}(\tau^{[N^2]}=r)\\ 
					&\sum_{\substack{(l,R_l)\in\\ [N^2]\times\{A,D_0,D_1,D_2\}}}
					\Big({\P}_{(i,D_1)}\left(RW^{[N^2]}(r)=(l,R_l)\right)-{\P}^{}_{(j,D_1)}\left(RW^{\prime[N^2]}(r)=(l,R_l)\right)\Big)\\
					&\quad \times \int_0^{\bar{t}(N)-r} \d v\, 
					\sum_{k\in [N^2]}P_{\Delta^{[N^2]}_{(i,A),(j,A)}}(v,((k,A),(k,A)))\\
					&\quad \times\P^{}_{(l,R_l)}\left(RW^{[N^2]}(v)=(k,A)\right)\E[g(x_k^{[N^2]}(\bar{t}(N)-r-v))],
				\end{aligned}
			\end{equation}
			\normalsize
			where in the last equality we use that the random walks move according to the same kernel $b(\cdot,\cdot)$. We can continue by writing
			\begin{equation}
				\label{1061}
				\begin{aligned}
					&2\int_0^{\bar{t}(N)} \d r\,\tilde{\P}(\tau^{[N^2]}=r)\\
					&\quad \sum_{\substack{(l,R_l)\in\\ [N^2]\times\{A,D_0,D_1,D_2\}}}
					\Big({\P}^{[N^2]}_{(i,D_1)}\left(RW^{[N^2]}(r)
					=(l,R_l)\right)-{\P}^{[N^2]}_{(j,D_1)}\left(RW^{\prime[N^2]}(r)=(l,R_l)\right)\Big)\\
					&\quad \times \int_0^{\bar{t}(N)-r}\d v\, \sum_{k\in [N^2]}
					\bigg(\tilde{\P}^{[N^2]}_{(i,A),(j,A)}\left(RW^{[N^2]}(v)=(k,A)\right)\\
					&\qquad \qquad \qquad -\tilde{\P}^{[N^2]}_{(i,A),(j,A)}\left(RW^{\prime[N^2]}(v)=(k,A)\right)\bigg)\\
					&\quad \times \P^{[N^2]}_{(l,R_l)}\left(RW^{[N^2]}(v)=(k,A)\right)\E[g(x_k^{[N^2]}(\bar{t}(N)-r-v))]\\
					&=2 \int_0^{\bar{t}(N)} \d r\,\tilde{\P}(\tau^{[N^2]}=r) \int_0^r \d u\, \tilde{\P}(\tau^{[N^2]}=u) 
					\sum_{\substack{(l,R_l)\in\\ [N^2]\times\{A,D_0,D_1,D_2\}}}\\
					&\qquad \Big(\tilde{\P}^{[N^2]}_{(i,A),(j,A)}\left(RW^{[N^2]}(r-u)=(l,R_l)\right)-\tilde{\P}^{[N^2]}_{(i,A),(j,A)}
					\left(RW^{\prime[N^2]}(r-u)=(l,R_l)\right)\Big)\\
					&\quad \times \int_0^{\bar{t}(N)-r}\d v\, \sum_{k\in [N^2]}
					\bigg(\tilde{\P}^{[N^2]}_{(i,A),(j,A)}\left(RW^{[N^2]}(v)=(k,A)\right)\\
					&\qquad \qquad \qquad -\tilde{\P}^{[N^2]}_{(i,A),(j,A)}\left(RW^{\prime[N^2]}(v)=(k,A)\right)\bigg)\\
					&\quad \times \P^{[N^2]}_{(l,R_l)}\left(RW^{[N^2]}(v)=(k,A)\right)\E[g(x_k^{[N^2]}(\bar{t}(N)-r-v))]\\
					&+ 2\int_0^{\bar{t}(N)} \d r\,\tilde{\P}(\tau^{[N^2]}=r)\tilde{\P}(\tau^{[N^2]}\geq r)\\
					&\quad \times \int_0^{\bar{t}(N)-r}\d v\, \sum_{k\in [N^2]}
					\bigg(\tilde{\P}^{[N^2]}_{(i,A),(j,A)}\left(RW^{[N^2]}(v)=(k,A)\right)\\
					&\qquad \qquad \qquad -\tilde{\P}^{[N^2]}_{(i,A),(j,A)}\left(RW^{\prime[N^2]}(v)=(k,A)\right)\bigg)\\
					&\quad \times \Big(\P^{[N^2]}_{(i,D_1)}\left(RW^{[N^2]}(v)=(k,A)\right)-\P^{[N^2]}_{(j,D_1)}
					\left(RW^{\prime[N^2]}(v)=(k,A)\right)\Big)\\
					&\quad \times \E[g(x_k^{[N^2]}(\bar{t}(N)-r-v))].
				\end{aligned}
			\end{equation}
			We will show that both terms in the last equality of \eqref{1061} tends to $0$ as $N\to\infty$. 
			
			For the first term note that, by \eqref{eqexprw22}, \eqref{1039} and \eqref{1059a}, we have
			\begin{equation}
				\begin{aligned}
					&2\int_0^{\bar{t}(N)} \d r\,\tilde{\P}(\tau^{[N^2]}=r)\int_0^r \d u\, \tilde{\P}(\tau^{[N^2]}=u) \\
					&\quad \times\Bigg[\sum_{\substack{(l,R_l)\in\\ [N^2]\times\{A,D_0,D_1,D_2\}}}
					\tilde{\P}^{[N^2]}_{(i,A),(j,A)}\left(RW^{[N^2]}(r-u)=(l,R_l)\right)\\
					&\qquad \qquad \qquad  -\tilde{\P}^{[N^2]}_{(i,A),(j,A)}\left(RW^{\prime[N^2]}(r-u)=(l,R_l)\right)\Bigg]\\
					&\quad \times \int_0^{\bar{t}(N)-r}\d v\, \sum_{k\in [N^2]}
					\bigg(\tilde{\P}^{[N^2]}_{(i,A),(j,A)}\left(RW^{[N^2]}(v)=(k,A)\right)\\
					&\qquad \qquad \qquad - \tilde{\P}^{[N^2]}_{(i,A),(j,A)}\left(RW^{\prime[N^2]}(v)=(k,A)\right)\bigg)\\
					&\quad \times \P^{[N^2]}_{(l,R_l)}\left(RW^{[N^2]}(v)=(k,A)\right)\E[g(x_k(\bar{t}(N)-r-v))]\\
					&\leq 4\int_0^{\bar{t}(N)} \d r\,\tilde{\P}(\tau^{[N^2]}=r)\int_0^r \d u\, \tilde{\P}(\tau^{[N^2]}=u) \\
					&\quad \times\Bigg|\sum_{\substack{(l,R_l)\in\\ [N^2]\times\{A,D_0,D_1,D_2\}}}
					\tilde{\P}^{[N^2]}_{(i,A),(j,A)}\left(RW^{[N^2]}(r-u)=(l,R_l)\right)\\
					&\qquad \qquad \qquad -\tilde{\P}^{[N^2]}_{(i,A),(j,A)}\left(RW^{\prime[N^2]}(r-u)=(l,R_l)\right)\Bigg|\\
					&\leq 4\int_0^{\bar{t}(N)} \d r\,\tilde{\P}(\tau^{[N^2]}=r)\int_0^{r-L(N)} \d u\, \tilde{\P}(\tau^{[N^2]}=u) \\
					&\quad \times\Bigg|\sum_{\substack{(l,R_l)\in\\ [N^2]\times\{A,D_0,D_1,D_2\}}}
					\tilde{\P}^{[N^2]}_{(i,A),(j,A)}\left(RW^{[N^2]}(r-u)=(l,R_l)\right)\\
					&\qquad \qquad \qquad -\tilde{\P}^{[N^2]}_{(i,A),(j,A)}\left(RW^{\prime[N^2]}(r-u)=(l,R_l)\right)\Bigg|\\
					&+ 8\int_{L(N)}^{\bar{t}(N)} \d r\,\tilde{\P}(\tau^{[N^2]}=r)\tilde{\P}(\tau^{[N^2]}\in [r-L(N),r])+8\P[\tau^{[N^2]}\in[0,L(N)]]\\
					&\leq 4\int_0^{\bar{t}(N)} \d r\,\tilde{\P}(\tau^{[N^2]}=r)\int_0^{r-L(N)} \d u\, \tilde{\P}(\tau^{[N^2]}=u) \\
					&\quad \times\Bigg|\sum_{\substack{(l,R_l)\in\\ [N^2]\times\{A,D_0,D_1,D_2\}}}
					\tilde{\P}^{[N^2]}_{(i,A),(j,A)}\left(RW^{[N^2]}(r-u)=(l,R_l)\right)\\
					&\qquad \qquad \qquad -\tilde{\P}^{[N^2]}_{(i,A),(j,A)}\left(RW^{\prime[N^2]}(r-u)=(l,R_l)\right)\Bigg|\\
					&+ 16 [1-\e^{-\frac{e_1}{N} L(N)}].
				\end{aligned}
			\end{equation}
			Hence the last term in the last inequality tends to $0$ as $N\to\infty$.
			\vspace{0.2cm} 
			
			To show that the first term in the last inequality tends to $0$ we use the coupling again. Recall the definition of $H_t^{[N^2]}$ in \eqref{1041}. Note that we can rewrite the sum as
			\footnotesize
			\begin{equation}
				\label{1064}
				\begin{aligned}
					&\sum_{\substack{(l,R_l)\in\\ [N^2]\times\{A,D_0,D_1,D_2\}}}
					\tilde{\P}^{[N^2]}_{(i,A),(j,A)}\left(RW^{[N^2]}(r-u)=(l,R_l)\right)
					-\tilde{\P}^{[N^2]}_{(i,A),(j,A)}\left(RW^{\prime[N^2]}(r-u)=(l,R_l)\right)\\
					&\leq 
					\sum_{\substack{(l,R_{l})\in\\ [N^2]\times\{A,D_0,D_1,D_2\}}}
					\sum_{\substack{(l^\prime,R_{l^\prime})\in\\ [N^2]\times\{A,D_0,D_1,D_2\}}}
					\Big({\P}^{[N^2]}_{(i,A)}\left(RW^{[N^2]}(L(N))=(l^\prime,R_l^\prime)\right)
					&\qquad\qquad\qquad\qquad\qquad\qquad\qquad
					-{\P}^{[N^2]}_{(i,A),(j,A)}\left(RW^{\prime[N^2]}(L(N))=(l^\prime,R^\prime_l)\right)\Big)\\
					&\qquad\times \P_{(l^\prime,R_l^\prime)}^{[N^2]}\left(RW^{[N^2]}(r-u-L(N))=(l,R_l)\right)\\
					&\leq \sum_{\substack{(l^\prime,R_{l^\prime})\in\\ [N^2]\times\{A,D_0,D_1,D_2\}}}
					\left({\P}^{[N^2]}_{(i,A)}\left(RW^{[N^2]}(L(N))=(l^\prime,R_l^\prime)\right)
					-{\P}^{[N^2]}_{(j,A)}\left(RW^{\prime[N^2]}(L(N))=(l^\prime,R^\prime_l)\right)\right)\\
					&=\sum_{\substack{(l^\prime,R_{l^\prime})\in\\ [N^2]\times\{A,D_0,D_1,D_2\}}}
					\Big(\tilde{\P}^{[N^2]}_{(i,A),(j,A)}\left(RW^{[N^2]}(L(N))=(l^\prime,R_l^\prime),(H_t^{[N^2]})^c\right)\\
					&\qquad\qquad\qquad\qquad\qquad\qquad\qquad
					-\tilde{\P}^{[N^2]}_{(i,A),(j,A)}\left(RW^{\prime[N^2]}(L(N))=(l^\prime,R^\prime_l),(H_t^{[N^2]})^c\right)\Big)\\
					&\leq 2\tilde{\P}^{[N^2]}_{(i,A),(j,A)}\left((H_t^{[N^2]})^c\right).
				\end{aligned}
			\end{equation}
			\normalsize
			To show that
			\begin{equation}
				\label{1065}
				\lim_{n\to\infty}\tilde{\P}^{[N^2]}_{(i,A),(j,A)}\left((H_t^{[N^2]})^c\right)=0,
			\end{equation}
			we can use a similar strategy as between \eqref{007} and \eqref{008}, but note that we now start from two active sites instead of two $1$-dormant sites. Therefore \eqref{1051} directly follows from \eqref{08} and \eqref{008}.
			
			To show  the second term in \eqref{1061} tends to $0$, we write it as
			\small
			\begin{equation}
				\begin{aligned}
					&2 \int_0^{\bar{t}(N)} \d r\,\tilde{\P}(\tau^{[N^2]}=r)\tilde{\P}(\tau^{[N^2]}\geq r)\\
					&\times \int_0^{\bar{t}(N)-r}\d v\, \sum_{k\in [N^2]}
					\left(\tilde{\P}^{[N^2]}_{(i,A),(j,A)}\left(RW^{[N^2]}(v)=(k,A)\right)
					-\tilde{\P}^{[N^2]}_{(i,A),(j,A)}\left(RW^{\prime[N^2]}(v)=(k,A)\right)\right)\\
					&\times \left[\P^{[N^2]}_{(i,D_1)}\left(RW^{[N^2]}(v)=(k,A)\right)
					-\P^{[N^2]}_{(j,D_1)}\left(RW^{\prime[N^2]}(v)=(k,A)\right)\right]\E[g(x_k^{[N^2]}(\bar{t}(N)-r-v))]\\
					&=2\int_0^{\bar{t}(N)} \d r\,\tilde{\P}(\tau^{[N^2]}=r)\tilde{\P}(\tau^{[N^2]}\geq r)\\
					&\times \int_0^{\bar{t}(N)-r}\d v\, \sum_{k\in [N^2]}
					\left(\tilde{\P}^{[N^2]}_{(i,A),(j,A)}\left(RW^{[N^2]}(v)=(k,A)\right)
					-\tilde{\P}^{[N^2]}_{(i,A),(j,A)}\left(RW^{\prime[N^2]}(v)=(k,A)\right)\right)\\
					&\times \int_0^v\, \d u\,\tilde{\P}^{[N^2]}(\tau^{[N^2]}= u)\\
					&\quad \Big(\tilde{\P}^{[N^2]}_{(i,A),(j,A)}\left(RW^{[N^2]}(v-u)=(k,A)\right)
					-\tilde{\P}^{[N^2]}_{(i,A),(j,A)}\left(RW^{\prime[N^2]}(v-u)=(k,A)\right)\Big)\\
					&\times\E[g(x_k^{[N^2]}(\bar{t}(N)-r-v))].
				\end{aligned}
			\end{equation}
			\normalsize
			Changing the order of integration and setting $w=v-u$, we obtain
			\begin{equation}
				\begin{aligned}
					&2\int_0^{\bar{t}(N)} \d r\,\tilde{\P}(\tau^{[N^2]}=r)\tilde{\P}(\tau^{[N^2]}\geq r)\\
					&\times \int_0^{\bar{t}(N)-r}\d u\, \tilde{\P}^{[N^2]}(\tau^{[N^2]}= u)\int_u^{\bar{t}(N)-r}\, \d v\,\\
					&\sum_{k\in [N^2]}\left[\tilde{\P}^{[N^2]}_{(i,A),(j,A)}\left(RW^{[N^2]}(v)=(k,A)\right)
					-\tilde{\P}^{[N^2]}_{(i,A),(j,A)}\left(RW^{\prime[N^2]}(v)=(k,A)\right)\right]\\
					&\times \left[\tilde{\P}^{[N^2]}_{(i,A),(j,A)}\left(RW^{[N^2]}(v-u)=(k,A)\right)
					-\tilde{\P}^{[N^2]}_{(i,A),(j,A)}\left(RW^{\prime[N^2]}(v-u)=(k,A)\right)\right]\\
					&\times\E[g(x_k^{[N^2]}(\bar{t}(N)-r-v))]\\
					&=2 \int_0^{\bar{t}(N)} \d r\,\tilde{\P}(\tau^{[N^2]}=r)\tilde{\P}(\tau^{[N^2]}\geq r)\\
					&\times \int_0^{\bar{t}(N)-r}\d u\, \tilde{\P}^{[N^2]}(\tau^{[N^2]}= u)\int_0^{\bar{t}(N)-r-u}\, \d w\,\\
					&\sum_{k\in [N^2]}\left[\tilde{\P}^{[N^2]}_{(i,A),(j,A)}\left(RW^{[N^2]}(w+u)=(k,A)\right)
					-\tilde{\P}^{[N^2]}_{(i,A),(j,A)}\left(RW^{\prime[N^2]}(w+u)=(k,A)\right)\right]\\
					&\times \left[\tilde{\P}^{[N^2]}_{(i,A),(j,A)}\left(RW^{[N^2]}(w)=(k,A)\right)
					-\tilde{\P}^{[N^2]}_{(i,A),(j,A)}\left(RW^{\prime[N^2]}(w)=(k,A)\right)\right]\\
					&\times\E[g(x_k^{[N^2]}(\bar{t}(N)-r-u-w))].
				\end{aligned}
			\end{equation}
			This can be rewritten as
			\small
			\begin{equation}
				\begin{aligned}
					&2\int_0^{\bar{t}(N)} \d r\,\tilde{\P}(\tau^{[N^2]}=r)\tilde{\P}(\tau^{[N^2]}\geq r)\\
					&\int_0^{\bar{t}(N)-r}\d u\, \tilde{\P}^{[N^2]}(\tau^{[N^2]}= u)\\
					&\quad \times \sum_{(l,R_l)\in [N^2]}
					\left[\tilde{\P}^{[N^2]}_{(i,A),(j,A)}\left(RW^{[N^2]}(u)=(l,R_l)\right)
					-\tilde{\P}^{[N^2]}_{(i,A),(j,A)}\left(RW^{\prime[N^2]}(u)=(l,R_l)\right)\right]\\
					&\quad \times \int_0^{\bar{t}(N)-r-u}\, \d w\,\sum_{k\in [N^2]}\left[\P^{[N^2]}_{(l,R_l)}
					\left(RW^{[N^2]}(w)=(k,A)\right)\right]\\
					&\quad \times \left[\tilde{\P}^{[N^2]}_{(i,A),(j,A)}\left(RW^{[N^2]}(w)=(k,A)\right)
					-\tilde{\P}^{[N^2]}_{(i,A),(j,A)}\left(RW^{\prime[N^2]}(w)=(k,A)\right)\right]\\
					&\quad \times\E[g(x_k^{[N^2]}(\bar{t}(N)-r-u-w))]\\
					&\leq 8\int_{L(N)}^{\bar{t}(N)} \d r\,\tilde{\P}(\tau^{[N^2]}=r)\tilde{\P}(\tau^{[N^2]}\geq r)\\
					&\int_{L(N)}^{\bar{t}(N)-r}\d u\, \tilde{\P}^{[N^2]}(\tau^{[N^2]}= u)\\
					&\quad \times \sum_{(l,R_l)\in [N^2]}
					\left[\tilde{\P}^{[N^2]}_{(i,A),(j,A)}\left(RW^{[N^2]}(u)=(l,R_l)\right)
					-\tilde{\P}^{[N^2]}_{(i,A),(j,A)}\left(RW^{\prime[N^2]}(u)=(l,R_l)\right)\right]\\
					&+8\int_{L(N)}^{\bar{t}(N)} \d r\,\tilde{\P}(\tau^{[N^2]}=r)\tilde{\P}(\tau^{[N^2]}\geq r)\\
					&\int_0^{L(N)}\d u\, \tilde{\P}^{[N^2]}(\tau^{[N^2]}= u)\\
					&\quad \times \sum_{(l,R_l)\in [N^2]}
					\left[\tilde{\P}^{[N^2]}_{(i,A),(j,A)}\left(RW^{[N^2]}(u)=(l,R_l)\right)
					-\tilde{\P}^{[N^2]}_{(i,A),(j,A)}\left(RW^{\prime[N^2]}(u)=(l,R_l)\right)\right]\\
					&+16\int_0^{L(N)} \d r\,\tilde{\P}(\tau^{[N^2]}=r)\tilde{\P}(\tau^{[N^2]}\geq r).
				\end{aligned}
			\end{equation}
			\normalsize 
			This tends to $0$ by \eqref{1064} and the reasoning below \eqref{1065}.
		\end{proof}
		
		\subsubsection{Limiting evolution of the estimator processes}
		\label{ss.leo}
		
		In this section we show that the results along the subsequences used in steps 5-8 of the scheme for the two-level three-colour mean-field system actually hold for all subsequences. Therefore the limiting evolution holds for $N\to\infty$. Recall that Lemma~\ref{lem:ds} tells us that all single $1$-dormant colonies equal the value of the $1$-dormant $1$-block average. Therefore  the second assumption in \eqref{z3} in Proposition~\eqref{ma77} can be replaced by
		\begin{equation}
			\begin{aligned}
				\lim_{k\to\infty}\CL\left[\left(Y_{2,0}^{[N_k^2]}(N^2_kt_2)\right)
				\Big|{\bf\Theta}^{\aux,(1),[N_k^2]}(N^2_kt_2+N_kt_1)\right]= P^{{z}_1^{\eff}(t_1)},
			\end{aligned}
		\end{equation}
		since, by Lemma~\ref{lem:ds}, the limiting law
		\begin{equation}
			\lim_{k \to \infty}\CL\left[\left(Y_{1,0}^{[N_k^2]}(N^2_kt_2)\right)\right] 
		\end{equation}
		is completely determined by the first line in \eqref{z3}. Hence, the assumptions in Proposition~\ref{lem:1blev} and Lemma~\ref{lem2blcon} can be weakened in the same way. Using that in Proposition~\ref{P.finsysmf2lev} we assume \eqref{as2} and \eqref{as32}, we find that the $2$-block convergence stated in Lemma~\ref{lem2blcon} holds along all subsequences we choose in Step 5. We conclude that Proposition~\ref{P.finsysmf2lev}(a) is indeed true. Combining Proposition~\ref{P.finsysmf2lev}(a) with steps 1-4 of the scheme and Lemma~\ref{lem:ds}, we find that the assumptions in Proposition~\ref{lem:1blev} are true for all subsequences, and we obtain the limiting evolution of the $1$-block estimator process. Projecting this limiting evolution onto the active $1$-block average and the $1$-dormant $1$-block average, we obtain Proposition~\ref{P.finsysmf2lev}(b). Finally, combining Proposition~\ref{P.finsysmf2lev}(a), steps1-4, and the fact that Proposition~\ref{lem:1blev} is true along all subsequences, we obtain Proposition~\ref{P.finsysmf2lev}(c) and (f).   
		
		\subsubsection{Convergence in the Meyer-Zheng topology}\label{ss.11mz}
		
		In this section we show how the results on the effective and estimator processes can be use to show convergence of the full $1$- and $2$-block processes.
		
		\begin{lemma}[\bf{[Convergence of $1$ process in  the Meyer-Zheng topology]}]\label{lem1mz2}
			Assume that for the $1$-block estimator process defined in \eqref{1113} 
			\begin{equation}
				\label{ma55}
				\begin{aligned}
					\lim_{N\to\infty} \CL\left[\left({\bf\Theta}^{\aux,(1) ,[N^2]}(N^2t_2+Nt_1)\right)_{t_1 > 0}\,\right]
					= \CL\left[(z_1^{\aux}(t_1))_{t_1>0}\right], 
				\end{aligned}
			\end{equation}
			where, conditional on $x^{\eff}_2(t_2)=u$, the limit process is the unique solution of the SSDE in \eqref{m12chb} with $\theta$ replaced by $u$ and with initial measure $\Gamma_u^{\eff,(1)}$. Then 
			\begin{equation}
				\label{1136s}
				\begin{aligned}
					&\hspace{-0.3cm} \lim_{N\to\infty} \CL\left[\left({\bf \Theta}^{(1),[N^2]}(N^2t_2+Nt_1)\right)_{t_1 > 0}\right]
					= \CL \left[(z_1^{\Gamma^{(1)}(t_2)}(t_1))_{t_1 > 0}\right]\\
					&\qquad\qquad\qquad \text{ in the Meyer-Zheng topology},
				\end{aligned}
			\end{equation}	
			where $\Gamma^{(1)}(t_2)$ is defined as in \eqref{mfkern1} and $(z_1^{\Gamma^{(1)}(t_2)}(t_1))_{t_1 > 0}$ is the process moving according to \eqref{z1} with initial measure $\Gamma^{(1)}(t_2)$. 	
		\end{lemma}
		
		\begin{proof}
			By assumption~\ref{ma55} and Lemma~\ref{lemlev1}, we can proceed as in the proof of Proposition~\ref{p.estima} to find \eqref{1136s}.
		\end{proof}
		
		\begin{lemma}[\bf{[Convergence of $2$-block process in  the Meyer-Zheng topology]}]
			\label{lem2mz2}
			$\mbox{}$\\
			Assume that for the effective $2$-block process defined in \eqref{1114} 
			\begin{equation}
				\label{ma5}
				\begin{aligned}
					\lim_{N\to\infty} \CL\left[\left({\bf\Theta}^{\eff,(2) ,[N^2]}(N^2t_2)\right)_{t_2 > 0}\,\right]
					= \CL\left[(z_2^{\aux}(t_2))_{t_2>0}\right], 
				\end{aligned}
			\end{equation}
			where $(z^{\eff}_2(t_2))_{t_2>0}$ is the process evolving according to \eqref{m64c} and starting from $(\vt_1,\theta_{y_2})$.  Then for
			the $2$-block estimator process defined in \eqref{1114} 
			\begin{equation}
				\label{sgh423}
				\begin{aligned}
					&\lim_{N\to\infty} \CL \left[\left({\bf \Theta}^{(2),[N^2]}(N^2t_2)\right)_{t_2 > 0}\right]
					= \CL \left[\left(z_2^{}(t_2)\right)_{t_2 > 0}\right]\\
					&\text{in the Meyer-Zheng topology},
				\end{aligned}
			\end{equation}
			where $\left(z_2^{}(t_2)\right)_{t_2 > 0}$ is the process evolving according to \eqref{gh432c} and starting in state $(\vt_1,\vt_1,\vt_1,\theta_{y_2})$.
		\end{lemma}
		
		\begin{proof}
			Combining Lemmas~\ref{lemlev1} and \ref{lemav3}, we find for $t_2>0$
			\begin{equation}
				\begin{aligned}
					\lim_{N \to \infty}\E\left[\left|\bar{ \Theta}^{(2),[N^2]}(N^2t_2)-\Theta_x^{(2),[N^2]}(N^2t_2)\right|\right]&=0,\\
					\lim_{N \to \infty}\E\left[\left|\bar{ \Theta}^{(2),[N^2]}(N^2t_2)-\Theta_{y_0}^{(2),[N^2]}(N^2t_2)\right|\right]&=0,\\ 
					\lim_{N \to \infty}\E\left[\left|\bar{ \Theta}^{(2),[N^2]}(N^2t_2)-\Theta_{y_1}^{(2),[N^2]}(N^2t_2)\right|\right]&=0.\\
				\end{aligned}
			\end{equation}
			Therefore we can again proceed as in the proof of Proposition~\ref{p.estima} to find \eqref{sgh423}.
		\end{proof}

		\subsubsection{Proof of the two-level three-colour mean-field finite-systems scheme}
		
		In Section~\ref{ss.leo} we already proved Proposition~\ref{P.finsysmf2lev}(a),(b),(c) and (f).  The proof of Proposition~\ref{P.finsysmf2lev}(d) follows from Proposition~\ref{P.finsysmf2lev}(a) by applying Lemma~\ref{lem2mz2}. The proof of Proposition~\ref{P.finsysmf2lev}(e) follows from Proposition~\ref{P.finsysmf2lev}(b) by applying Lemma~\ref{lem1mz2}. This completes the proof of Proposition~\ref{P.finsysmf2lev}.

		\section{Proofs: $N\to\infty$, hierarchical multi-scale limit theorems}
		\label{s.multilevel}
		
		In this section we prove the hierarchical multi-scale limit theorems stated in Theorem~\ref{T.multiscalehiereff} and Theorem~\ref{T.multiscalehier}. In Section~\ref{ss.multilevelmf} we first introduce the finite-level mean-field finite-systems scheme. In Section~\ref{ss.flmfss} we given an outline of how to prove the finite-level mean-field finite-systems scheme. In Section~\ref{ss.pmslt} we show how Theorems~\ref{T.multiscalehiereff}--\ref{T.multiscalehier} can be obtained by a simple generalisation of the finite-level mean-field finite-systems scheme. The proof of the finite-level mean-field finite-systems scheme follows a similar lin3 of argument as in Section~\ref{ss.org22} once we incorporate more levels. Since the proofs for the finite-level mean-field systems scheme are similar as the proofs in Section~\ref{pmfs23}, we will not write out the full proof, but only give an outline and a sketch.

		\subsection{Finite-level mean-field finite-systems scheme and interaction chain}
		\label{ss.multilevelmf}
		
		In this section we extend the two-level three-colour system to a $k$-level $(k+1)$-colour system with an ``outside world" for any $k\in\N$. This outside world allows also the highest level, the $k$-block, to start from equilibrium. It is also needed to generalize the results in this subsection to the infinite hierarchical group. 
		
		\paragraph{$\blacktriangleright$ Definitions.}
		
		To set up the system, fix $k \in \N$ and consider the geographic space \gls{omegatrunc} obtained by truncating the hierarchical group $\Omega_N$ (recall \eqref{hiergroup}) after hierarchical level $k+1$, i.e., $\Omega_N^{k+1} = B_{k+1}(0)$ the $(k+1)$-block centred at the origin (recall \eqref{hiergroup}, \eqref{ultra} and Fig.~\ref{fig-hierargr}). Note that the $k+1$-block consists of  $N$ $k$-blocks i.e., $B_{k+1}(0)=\bigcup_{i=0}^N B_k (i)$ and $B_{k+1}(0)=[N^{k+1}]$.  The seed-bank in this model consists of the $k+2$ layers corresponding to colours $\{0,\cdots,k\}\cup \{k+1\}.$  On this space we again consider a restricted version of the SSDE in \eqref{moSDE} to the geographic space $\Omega_N^{k+1}$. The migration kernel $a^{\Omega_N}(\cdot,\cdot)$ is restricted to $\Omega_N^{k+1}$ by setting all migration outside $B_{k+1}(0)$ equal to 0, i.e., 
		\begin{equation}
			a^{[\Omega_N^{k+1}]}(\xi,\eta)=\sum_{l=1}^{k+1}1_{\big\{d_{\Omega_N^{k+1}}(\xi,\eta)\leq l\big\}}
			\frac{c_l}{N^{l-1}}\frac{1}{N^l},
		\end{equation}
		where $d_{\Omega_N^{k+1}}$ is the hierarchical distance $d_{\Omega_N}$ restricted to the space $\Omega_N^{k+1}$. The colour-$l$ dormant population exchanges individuals with the active population at rates $\frac{e_l}{N^l}$, $\frac{K_le_l}{N^l}$ for all $0\leq l\leq k$. We set the interaction of the active population with the colour $(k+1)$-dormant population equal to $0$. This seed-bank is only needed later, namely for the ``outside world". 
		
		The state space of the finite-level mean-field system is
		\begin{equation}
			S = (\mathfrak{s}^{k+1})^{\Omega_N^{k+1}}, \quad \mathfrak{s}^{k+1} = [0,1] \times [0,1]^{k+2},
		\end{equation}
		and the system is denoted by
		\begin{equation}
			\label{klevhierar}
			\begin{aligned}
				&Z^{\Omega_N^{k+1}} = (Z^{\Omega_N^{k+1}}(t))_{t \geq 0}, 
				\qquad Z^{\Omega_N^{k+1}}(t) = (z^{\Omega_N^{k+1}}_\xi(t))_{\xi \in \Omega_N^{k+1}},\\
				&z^{\Omega_N^{k+1}}_\xi(t) = (x^{\Omega_N^{k+1}}_\xi(t),(y^{\Omega_N^{k+1}}_{\xi,m}(t))_{m=0}^{k+1}).  
			\end{aligned}
		\end{equation}
		The components evolve according to the SSDE
		\begin{equation}
			\label{klevel}
			\begin{aligned}
				\d x^{\Omega_N^{k+1}}_\xi(t) &= \sum_{l=1}^{k} \,\frac{c_{l-1}}{N^{l-1}}\,\frac{1}{N^l}\sum_{\eta \in B_l(\xi)} \,
				\left[x^{\Omega_N^{k+1}}_\eta(t)-x^{\Omega_N^{k+1}}_\xi(t)\right]\,\d t\\
				&\qquad + \sqrt{g(x^{\Omega_N^{k+1}}_\xi(t))}\,\d w_\xi(t)
				+ \sum_{m=0}^{k} \frac{K_m e_m}{N^m}\, \left[y^{\Omega_N^{k+1}}_{\xi,m}(t) - x^{\Omega_N^{k+1}}_\xi(t)\right]\,\d t, \\
				\d y^{\Omega_N^{k+1}}_{\xi,m}(t) &= \frac{e_m}{N^m} \,\left[x^{\Omega_N^{k+1}}_\xi(t) 
				- y^{\Omega_N^{k+1}}_{\xi,m}(t)\right]\,\d t, \quad 0 \leq m \leq k, \\
				\d y^{\Omega_N^{k+1}}_{\xi,m}(t) &= 0, \qquad \xi \in \Omega_N^{k+1},
			\end{aligned}
		\end{equation}
		with $B_l(\xi)$ the ball of radius $l$ around $\xi \in \Omega_N^{k+1}$. 
		
		Note that this system is the hierarchical SSDE in \eqref{moSDE} with all interactions at distance $>k$ switched off (i.e., $c_l = 0$ for $l > k+1$), and also the exchange with dormant populations of colour $m>k$ is switched off. As before, by \cite{SS80} the martingale problem associated with \eqref{klevel} is well-posed, and for every initial state in $S$ the SSDE has a unique strong solution. We will analyse \eqref{klevel} on time scales $1, N, N^2,\cdots,N^k$. If for $0\leq l\leq k$ time runs on time scale $N^l$, then we write $N^lt_l$ with $t_l>0$. 
		
		To study the $k$-level mean-field system, we analyse the equivalent of the block averages defined in \eqref{defblockav}. For the $k$-level mean-field system these are given by
		\begin{equation}
			\label{blockavlk}
			\begin{aligned}
				x^{\Omega^{k+1}_N}_{l}(t) &= \frac{1}{N^l} \sum_{\eta \in B_l(0)} x^{\Omega_N^{k+1}}_\eta(N^lt),\\
				y^{\Omega^{k+1}_N}_{m,l}(t) &= \frac{1}{N^l} \sum_{\eta \in B_l(0)} y^{\Omega_N^{k+1}}_{\eta,m}(N^lt),
				\qquad  0\leq m \leq k+1,\qquad 0\leq l\leq k.
			\end{aligned}
		\end{equation}
		For $0\leq l\leq k$ these block averages evolve according to the SSDE
		\begin{eqnarray}
			\label{rblockavxmultik}
			\d{x}_{l}^{\Omega^{k+1}_N}(t)
			&=& \sum_{n=1}^{k-(l-1 )}
			\frac{c_{l+n-1}}{N^{n-1}}\big[{x}_{l+n}^{\Omega^{k+1}_N}(N^{-n}t)-{x}_{l}^{\Omega^{k+1}_N}(t)\big]\,\d t
			\nonumber\\
			&&+\sqrt{\frac{1}{N^l} \sum_{i\in B_l(0)} g(x_i(N^lt))}\,\,\d w_l(t)\\ \nonumber
			&&\qquad +\sum_{m=0}^{k} N^l \frac{K_m e_m}{N^m}
			\big[{y}_{m,l}^{\Omega^{k+1}_N}(t)-{x}_{l}^{\Omega^{k+1}_N}(t)\big]\,\d t,\\
			\label{rblockavzmultik}
			\d{y}_{m,l}^{\Omega^{k+1}_N}(t)
			&=&N^l \frac{e_m}{N^{m}}\big[{x}_{l}^{\Omega^{k+1}_N}(t)-{y}_{m,l}^{\Omega^{k+1}_N}(t)\big]\,\d t,
			\qquad 0\leq m\leq k,\\
			\d y^{\Omega_N^{k+1}}_{k+1,l}(t) &=& 0,
		\end{eqnarray}
		In the limit as $N\to\infty$, the active $l$-block average feels a drift towards the active $(l+1)$-block average, which is not moving on time scale $N^l$, at rate $c_l$. The diffusion term for the $l$-block average becomes the average diffusion over the $l$-block. The drift of the active $l$-block average towards the $l$-block average  of $m$-dormant populations $y_{m,l}^{\Omega_N^{k+1}}$ with $m>l$ vanishes in the limit as $N\to\infty$. Therefore, the $m>l$ $m$-dormant populations are \emph{slow seed-banks} on space-time scale $l$. The $l$-block average of the colour-$l$ dormant population $y_{l,l}^{\Omega_N^{k+1}}$ has a non-trivial drift towards the active $l$-block average, written $x_l^{\Omega_N^{k+1}}$. Therefore the $l$-dormant population is the \emph{effective seed-bank} on space-time scale $l$. For the colour $m$-dormant populations $y_{m,l}^{\Omega_N^{k+1}}$ with $m<l$, we see that infinite rates appear. Therefore the $m$-dormant populations with $m<l$ are \emph{fast seed-banks} on space-time scale $l$.  We again need the Meyer-Zheng topology to show that $\lim_{N\to\infty}y_{m,l}^{\Omega_N^{k+1}}=\lim_{N\to \infty}x_{l}^{\Omega_N^{k+1}}$. On space-time scale $l$, the colour-$l$ dormant population is the effective seed-bank. To get rid of the infinite rates we again look at combinations. From the above discussion and the SSDE in \eqref{rblockavxmultik}--\eqref{rblockavxmultik}, we see that if we consider the quantity
		\begin{equation}
			\frac{x_l^{\Omega_N^{k+1}}(t) + \sum_{m=0}^{l-1} K_m y_{l,m}^{\Omega_N^{k+1}}(t)}{1+\sum_{m=0}^{l-1} K_m},
		\end{equation}
		then all infinite rates cancel out. Therefore
		\begin{equation}
			\left(\frac{x_l^{\Omega_N^{k+1}}(t) + \sum_{m=0}^{l-1} K_m y_{l,m}^{\Omega_N^{k+1}}(t)}
			{1+\sum_{m=0}^{l-1} K_m}, y_{l,l}^{\Omega_N^{k+1}}(t)\right)_{t > 0}
		\end{equation}
		is called the \emph{effective process} on space-time scale $l$. Like in the simpler mean-field finite-systems scheme, the effective process allows us to analyse our system in path space.
		
		An important difference between the finite-level mean-field system in \eqref{rblockavxmultik}--\eqref{rblockavzmultik} and the two-level three-colour mean-field system in Section \eqref{ss.tlhmfs*} is that in the finite-level mean-field system also the highest level $k$ has a drift towards the outside world. This outside world is the active $k+1$-block average, which does not evolve on time scale $N^k$. This drift allows the finite-level mean-field system to equilibrate to a non-trivial equilibrium. In the two-level mean-field system, the highest level, i.e., the active $2$-block average, does not feel a drift due to migration. Consequently, the $2$-block averages will eventually cluster.
		
		\paragraph{$\blacktriangleright$ Scaling limit.}
		
		To state and prove the finite-level multi-scale limit, we need the following three limiting processes. Recall \eqref{Ekdef} and \eqref{deftheta}. For $0\leq l\leq k$, let 
		\begin{equation}
			(z_{l,(\theta,(y_{m,l})_{m=0}^{k+1})}(t))_{t\geq 0})=(x_l(t),(y_{m,l}(t))_{m=0}^{k+1})_{t\geq 0})
		\end{equation}
		be the process evolving according to
		\begin{equation}
			\label{z11}
			\begin{aligned}
				&\d x_l(t) =  E_l\Bigg[c_l [\theta - x_l(t)]\, \d t 
				+ \sqrt{\CF^{(l)} g(x_1(t))}\, \d w (t) +  K_l e_l\, [y_{l,l}(t)-x_{l}(t)]\,\d t\Bigg],\\
				& y_{m,l}(t) = x_l(t),\qquad \text{ for }0\leq m<l\\
				&\d y_{l,l}(t) = e_l\,[x_l(t)-y_{l,l}(t)]\, \d t,\\
				& y_{m,l}(t) =y_{m,l} ,\qquad\text{ for }l<m\leq k+1,
			\end{aligned}
		\end{equation}
		where $\theta\in[0,1]$ and $y_{m,l}\in[0,1]$ for $l<m\leq k+1$. 
		
		For $0\leq l\leq k$, let 
		\begin{equation}
			(z_{l,(\theta,(y_{m,l})_{m=l+1}^{k+1})}^\aux(t))_{t\geq 0}=(x^{\aux}_l(t),(y^{\aux}_{m,l}(t))_{m=l+1}^{k+1})_{t\geq 0}
		\end{equation}
		be the process evolving according to
		\begin{equation}
			\label{z11s}
			\begin{aligned}
				&\d x^{\aux}_l(t) =  E_l\Bigg[c_l [\theta - x^{\aux}_l(t)]\, \d t 
				+ \sqrt{\CF^{(l)} g(x^{\aux}_1(t))}\, \d w (t) +  K_l e_l\, [y^{\aux}_{l,l}(t)-x^{\aux}_{l}(t)]\,\d t\Bigg],\\
				&\d y^{\aux}_{l,l}(t) = e_l\,[x^{\aux}_l(t)-y^{\aux}_{l,l}(t)]\, \d t,\\
				& y^{\aux}_{m,l}(t) =y^{}_{m,l} ,\qquad\text{ for }l<m\leq k+1,
			\end{aligned}
		\end{equation}
		where $\theta\in[0,1]$ and $y^{\aux}_{m,l}\in[0,1]$, for $l<m\leq k+1$. 
		
		For $0\leq l \leq k$, let  
		\begin{equation}
			\label{92fl}
			(z^{\eff}_{l,\theta}(t))_{t\geq 0}=\left(x^\eff_l(t),y^\eff_{l,l}(t)\right)_{t\geq 0}
		\end{equation}
		be the \emph{effective process} evolving according to 
		\begin{equation}
			\label{927fl}
			\begin{aligned}
				&\d x^{\eff}_l(t) = E_l\left[c_l\,[\theta - x^{\eff}_l(t)]\, \d t + \sqrt{(\CF^{(l)}g)(x^{\eff}_l(t))}\, \d w (t) 
				+ K_le_l\,[y^{\eff}_{l,l}(t)-x^{\eff}_l(t)]\,\d t\right],\\
				&\d y^{\eff}_{l,l}(t) = e_l\, [x^{\eff}_l(t)-y^{\eff}_{l,l}(t)]\, \d t.
			\end{aligned}
		\end{equation}
		Comparing \eqref{z11} with \eqref{927fl}, we see that the effective process looks at the non-trivial components of the full process. The auxiliary process in \eqref{z11s} looks at the active population, the effective seed-bank and the slow seed-banks. 
		
		To state and prove the finite-level multi-scale limit, we need the following list of ingredients:
		\begin{enumerate}
			\item 
			For $t>0$ and for $0\leq l\leq k$, define the \emph{$l$-block estimators}
			\begin{equation}
				\label{tam6}
				\begin{aligned}
					\bar{\Theta}^{(l) ,\Omega_N^{k+1}}(t)
					&=\frac{1}{N^l}\sum_{i\in B_l}\frac{ x^{\Omega_N^{k+1}}_i(t)
						+\sum_{m=0}^{l-1}K_m y^{\Omega_N^{k+1}}_{i,0}(t)}{1+K_0},\\
					\Theta^{(l) ,\Omega_N^{k+1}}_x(t)&=\frac{1}{N^l}\sum_{i\in B_l} x^{\Omega_N^{k+1}}_{i}(t),\\
					\Theta^{(l) ,\Omega_N^{k+1}}_{y_m}(t)&=\frac{1}{N^l}\sum_{i\in B_l} y^{\Omega_N^{k+1}}_{i,m}(t),\qquad 0\leq m\leq k+1,
				\end{aligned}
			\end{equation}
			and put
			\begin{equation}
				\label{412fl}
				\begin{aligned}
					\boldsymbol{\Theta}^{(l),\Omega_N^{k+1}}(t),
					&=\left(\Theta^{(l) ,\Omega_N^{k+1}}_x(t),\big(\Theta^{(l) ,\Omega_N^{k+1}}_{y_m}(t)\big)_{m=0}^{k+1}\right),\\
					\gls{aux}
					&=\left(\bar{\Theta}^{(l) ,\Omega_N^{k+1}}(t),\left(\Theta^{(l) ,\Omega_N^{k+1}}_{y_{l}}(t)\right)_{m=l}^{k+1}\right),\\
					\boldsymbol{\Theta}^{\eff,(l),\Omega_N^{k+1}}(t)
					&=\left(\bar{\Theta}^{(l) ,\Omega_N^{k+1}}(t),\Theta^{(l) ,\Omega_N^{k+1}}_{y_{l}}(t)\right).
				\end{aligned}
			\end{equation}
			We call $(\boldsymbol{\Theta}^{(l),\Omega_N^{k+1}}(t))_{t>0}$ the \emph{$l$-block estimator process}, $(\boldsymbol{\Theta}^{\aux,(l),\Omega_N^{k+1}}(t))_{t>0}$ the \emph{auxiliary $l$-block estimator process} and $(\boldsymbol{\Theta}^{\eff,(l),\Omega_N^{k+1}}(t))_{t>0}$ the \emph{effective $l$-block estimator process}.  
			\item  
			For $0\leq l\leq k$, define the \emph{time scales} $N^l$ such that
			\begin{equation}
				\CL[\bar{\Theta}^{(l),\Omega_N^{k+1}} (N^{l}t_l-L(N)N^{l-1})-\bar\Theta^{(l),\Omega_N^{k+1}}( N^{l}t_l)]=\delta_0
			\end{equation}
			for all $L(N)$ such that $\lim_{N\to\infty} L(N) = infty$ and $\lim_{N\to\infty} L(N)/N = 0$, but not for $L(N)=N$. In words, $N$ is the time scale on which $\bar{\Theta}^{(l),\Omega_N^{k+1}}(\cdot)$ starts evolving, i.e., $\left(\bar{ \Theta}^{(l),\Omega_N^{k+1}}(N^lt_l)\right)_{t_l>0},$ is no longer a fixed process. 
			\item 
			The \emph{invariant measure} for the evolution of the $l$-block average in \eqref{z11}, denoted by   
			\begin{equation}
				\label{singcoleq322}
				\Gamma_{\theta, y_l}^{(l)},\qquad y_l=(y_{m,l})_{m=0}^{k+1}.
			\end{equation}
			The \emph{invariant measures} of the auxiliary $l$-block process in \eqref{92fl} and the effective $l$-block process in \eqref{927fl}, denoted by, respectively,  
			\begin{equation}
				\label{fleqau}
				\Gamma_{\theta, y_l}^{(l),\aux},\qquad y_l=(y_{m,l})_{m=l+1}^{k+1}
			\end{equation}
			and
			\begin{equation}
				\label{fleqeff}
				\Gamma_\theta^{(l),\eff}.
			\end{equation}
			\item 
			For $0\leq l\leq k$, let $\CF^{E_l,c_l,K_l,e_l}$ denote the \emph{renormalisation transformation} acting on $\CG$ defined by
			\begin{equation}
				(\CF^{E_l,c_l,K_l,e_l}g)(\theta) = \int_{[0,1]^2} g(x)\,\Gamma_\theta^{(l)}(\d x,(\d y)_{m=0}^{k+1}), \qquad \theta \in [0,1],
			\end{equation}
			and define the \emph{iterates} $\CF^{(n)}$, $0\leq n\leq k$, of the renormalisation transformation as the compositions
			\begin{equation}
				\label{frenormitalt}
				\CF^{(l)} = \CF^{E_{n-1},c_{n-1},K_{n-1},e_{n-1}} \circ \cdots \circ \CF^{E_0,c_0,K_0,e_0},\qquad 0\leq l\leq k.
			\end{equation}
			(Recall \eqref{frenormit}.)
			\item 
			To give a detailed description of the multi-scale behaviour of the SSDE in \eqref{moSDE}, define the \emph{interaction chain} 
			\begin{equation}
				\label{fintchainalt}
				(M^k_{-l})_{-l = -(k+1),-k,\ldots,0}
			\end{equation}
			as the \emph{time-inhomogeneous} Markov chain on $[0,1]\times[0,1]^{k+1}$ with initial state 
			\begin{equation}
				M^k_{-(k+1)} =(\vartheta_k,\overbrace{\vartheta_k,\cdots, \vartheta_k}^{k+1 \text{ times}},\theta_{y,k+1})
			\end{equation} 
			that evolves according to the transition kernel $Q^{[l]}$ from time $-(l+1)$ to time $-l$ given by 
			\begin{equation}
				Q^{[l]}(u,\d v) =  \Gamma_{u}^{(l)}(\d v), \qquad 0\leq l\leq k.
			\end{equation}
			(Recall \eqref{fintchain}.)
		\end{enumerate}
		
		We are now ready to state the scaling limit for the evolution of the averages in \eqref{gh412}. 
		
		\begin{proposition}{\bf [Finite-level mean-field: finite-systems scheme]}
			\label{P.finitesysmfmulti} 
			Suppose that the initial state of the system in \eqref{klevel} is given by $\mu(0)=\mu^{\otimes[\Omega_N^{k+1}]}$ for some $\mu\in\CP([0,1]\times[0,1]^{k+2})$. Let $L(N)$ be such that $\lim_{N \to \infty} L(N)=\infty$ and $\lim_{N \to \infty} L(N)/N=0$, and for $t_k,\ldots,t_0 \in (0,\infty)$ set $\bar{t}=L(N)N^k + \sum_{n=0}^k t_n N^n$.
			\begin{itemize}
				\item[(a)]
				For every $t_k,\ldots,t_0 \in (0,\infty)$, 
				\begin{equation}
					\begin{aligned}
						\lim_{N\to\infty} \CL\Bigg[\Bigg({\bf \Theta}^{(l),\Omega_N^{k+1}}\left(\bar{t}\right)\Bigg)_{l=k+1,k,\ldots,0}\Bigg]
						= \CL\left[(M^k_{-l})_{-l = -(k+1),-k,\ldots,0}\right],
					\end{aligned}
				\end{equation}
				where $(M^k_{-l})_{-l=-(k+1),-k,\ldots,0}$ is the interaction chain in \eqref{fintchainalt} starting from
				\begin{equation}
					M^k_{-(k+1)} =(\vartheta_k,\overbrace{\vartheta_k,\cdots, \vartheta_k}^{k+1 \text{ times}},\theta_{y,k+1}).
				\end{equation}
				\item[(b)]
				For all $0\leq l\leq k$,
				\begin{equation}
					\label{multiconvpre}
					\begin{aligned}
						& \lim_{N\to\infty} 
						\CL\left[\left(\boldsymbol{\Theta}^{(l),\Omega_N^{k+1}}(\bar{t}+ t_l N^l)\right)_{t_l > 0}\right] 
						= \CL\left[\left(z_{l,M^k_{-(l+1)}}(t_l)\right)_{t_l > 0}\right],\\[0.3cm]
						&\text{in the Meyer-Zheng topology}, 
					\end{aligned}
				\end{equation}
				where
				\begin{equation}
					(z_{l,M^k_{-(l+1)}}(t_l))_{t_l>0}
				\end{equation}
				is the processes defined \eqref{z11} with $\theta$, $(y_m,l)_{m=l+1}^{k+1}$ given by the corresponding components in $M^k_{-(l+1)}$ and with initial measure
				\begin{equation}
					\begin{aligned}
						&\CL\left[z_{l,M^k_{-(l+1)}}(0)\right]=\Gamma^{(l)}_{M^k_{-(l+1)}}\\
						&\Gamma^{(l)}_{M^k_{-(l+1)}}=\int_{\mathfrak{s}^{k+1}}\cdots\int_{\mathfrak{s}^{k+1}}\int_{\mathfrak{s}^{k+1}}
						\Gamma^{(k)}_{M^k_{-(k+1)}}(\d u_k)\Gamma^{(k-1)}_{u_k}(\d u_{k-1})
						\cdots\Gamma^{(l+1)}_{u_2}(\d u_{l+2}) \Gamma^{(l)}_{u_{l+1}}. 
					\end{aligned}
				\end{equation}  
			\end{itemize}
		\end{proposition}
		
		\medskip\noindent
		In Part (a), the limit does not depend on the choice of the times $t_k,\ldots,t_0$, since we let time start from a time larger than $L(N)N^k$, so that in the limit as $N\to\infty$ all the $l$-block averages with $l \leq k$ are already in quasi-equilibrium. In Part (b), for $l<k$ the center of the drift for the active population is \emph{random} and is determined by the first component of the interaction chain. Also the states of the $m$-dormant populations with $l<m\leq k+1$ are determined by the interaction chain. 
		
		\begin{remark}
			{\rm In contrast to Propositions~\ref{P.finsysmf2}--\ref{P.finsysmf2lev}, there are no assumptions on the seed-bank behaviour in Proposition~\ref{P.finitesysmfmulti}. This is because all the block-averages that we consider are in equilibrium at time $\bar{t}$. Consequently on space-time scales $l<m$ the $m$-dormant $l$-block average will equal the state of the $m$-dormant $m$-block average at time $\bar{t}$. Therefore we say that the state of the \emph{slow seed-banks} is determined by the space-time scaleon which this seed-bank is effective. Hence the state of the slow seed-banks is determined by the interaction chain.}\hfill$\blacksquare$
		\end{remark}
		
		The proof of Proposition~\ref{P.finitesysmfmulti} will be given in Section~\ref{ss.flmfss}.
		
		\subsection{Proof of the mean-field finite-systems scheme: finite-level}
		\label{ss.flmfss}
		
		We give a sketch of the proof Proposition ~\ref{P.finitesysmfmulti}. The proof uses a similar scheme as the proof of Proposition~\ref{P.finsysmf2lev}. We state the scheme and indicate at each step how it can be proved. 
		
		\begin{enumerate}
			\item 
			Tightness of the auxiliary $l$-block estimator processes, for $0\leq l\leq k$,
			\begin{equation}
				\left(\left({\bf\Theta}^{\aux,(l),\Omega_N^{k+1}}(N^lt_l)\right)_{t_l > 0}\right)_{N\in\N}.
			\end{equation} 
			
			\begin{proof}
				For each $0\leq l\leq k$ we use the tightness criterion in \cite[Proposition 3.2.3.]{JM86}. 
			\end{proof}
			
			\item 
			Stability property of the 2-block estimators, i.e., for  $L(N)$ such that $\lim_{N\to \infty} L(N)=\infty$ and $\lim_{N\to \infty} L(N)/N=0$,
			\begin{equation}
				\label{flmb}
				\lim_{N\to\infty}\sup_{0 \leq t\leq L(N)}\left|\bar{\Theta}^{(l),\Omega_N^{k+1}}(N^lt_l)
				-\bar{\Theta}^{(l),\Omega_N^{k+1}}(N^lt_l-N^{l-1}t)\right|=0\text{ in probability}
			\end{equation}
			and, for all $l\leq m\leq k+1$,
			\begin{equation}
				\label{flm3}
				\lim_{N\to\infty}\sup_{0 \leq t\leq L(N)}\left|\Theta_{y_m}^{(l),\Omega_N^{k+1}}(N^lt_l)
				-\Theta_{y_m}^{(l),\Omega_N^{k+1}}(N^lt_l-N^{l-1}t_l)\right|=0 \text{ in probability.}
			\end{equation}
			
			\begin{proof}
				Use a similar computation as in the proof of Lemma~\ref{lemsta2}.
			\end{proof}
			
			\item 
			We analyse the behaviour of the slow seed-banks by proving the following lemma.
			
			\begin{lemma}{\bf [Slow seed-banks in the multi-level system]}
				\label{lemflss}
				Let $\Theta^{(l)}_{y_m,i}$ denote the $m$-dormant $l$-block average containing colony $i\in \Omega_N^{k+1}$. Then for all $i\in \Omega_N^{k+1}$, $m<k+1$, $l<m$ and $t_l>0$,
				\begin{equation}
					\lim_{N \to \infty}\left[y_{i,m}^{\Omega_N^{k+1}}(\bar{t}+N^lt_l)-\Theta_{y_m,i}^{(m),\Omega_N^{k+1}}(\bar{t}+N^lt_l)\right]
					=0\quad a.s.
				\end{equation}
				and hence
				\begin{equation}
					\lim_{N \to \infty}\left[\Theta_{y_m,i}^{(l),\Omega_N^{k+1}}(\bar{t}+N^lt_l)
					-\Theta_{y_m,i}^{(m),\Omega_N^{k+1}}(\bar{t}+N^lt_l)\right]=0\quad a.s.
				\end{equation}
			\end{lemma}
			
			\begin{proof}
				We can proceed as in the proof of Lemma~\ref{lem:ds}, after adapting the kernel $b^{[N^2]}(\cdot,\cdot)$ to the kernel $b^{\Omega_N^{k+1}}(\cdot,\cdot)$. Then we can use that, from each of the $m<k+1$ $m$-dormant populations, individuals wake up before time $\bar{t}$ with probability $1$. For individuals starting from an $m$-dormant state, we define the coupling event 
				\begin{equation}
					H_t^{m,\Omega_N^{k+1}}=\{RW^{\Omega_N^{k+1}} \text{ has migrated over distance $m$ at least once up to time } t\}.
				\end{equation}  
				The migration over distance $m$ is needed because we need $m$-dormant individuals to be uniformly distributed over the $m$-block in order to almost surely equal the state of the $m$-block. 
			\end{proof}
			
			\item 
			We prove the convergence of the single components. Recall that there are $N^{k+1-l}$ $l$-blocks in $\Omega_N^{k+1}$. Since tightness of components implies tightness of the process, step 1 implies that for $0\leq l\leq k$ the full $l$-block processes 
			\begin{equation}
				\left(\left({\bf\Theta}_i^{\aux, (l),\Omega_N^{k+1}}(\bar{t}+N^lt_l)\right)_{t_l>0,\, i\in[N^{k+1-l}]}\right)_{N\in\N}
			\end{equation}
			are tight. From the tightness in steps 1 we can construct a subsequence $(N_n)_{n\in\N}$ along which, for all $0\leq l\leq k$, 
			\begin{equation}
				\begin{aligned} 
					\lim_{n\to\infty}\CL\left[\left({\bf\Theta}_i^{\aux,(1),\Omega_{N_n}^{k+1}}(\bar{t}+N^l_nt_l)\right)_{t_l > 0,\, i\in[N_n^{k+1-l}]}\right]
				\end{aligned}
			\end{equation}
			exists. Note that $\bar{t}$ depends on the subsequence. For example, along the subsequence $(N_{\tilde{n}})_{\tilde{n}\in\N}$, 
			\begin{equation}
				\bar{t}=L(N)N^k + \sum_{n=0}^k t_n N_{\tilde{n}}^n.
			\end{equation}
			
			We define the measure
			\begin{equation}
				\nu^{(0)}_{\bf \Theta}=\prod_{i\in\N_0}\Gamma^{(0)}_{\Theta_i}(\bar{t}),
			\end{equation}
			where 
			\begin{equation}
				\Theta_i\in\mathfrak{s}^{k+1}.
			\end{equation}
			In this step we show that along the same subsequence the single components converge to the infinite system. We show that if 
			\begin{equation}
				\label{1244}
				\lim_{n\to\infty}\CL[({\bf\Theta}^{\aux,(1),\Omega_{N_n}^{k+1}}(\bar{t}))_{i\in[N_n^{k}]}]=P^{(1)},
			\end{equation}
			then
			\begin{equation}
				\begin{aligned}
					\lim_{n\to\infty} \CL\Big[\Big(Z^{\Omega_{N_n}^{k+1}}(\bar{t}+t_0)\Big)_{t_0 \geq 0}\Big] 
					= \CL \left[(Z^{\nu^{(0)}(\bar{t})}(t_0))_{t_0 \geq 0}\right],
				\end{aligned}
			\end{equation}
			where
			\begin{equation}
				\nu^{(0)}(\bar{t})=\int \nu^{(0)}_{u} P^{(1)}(\d u).
			\end{equation}
			Here, $(Z^{\nu^{(0)}(\bar{t})}(t_0))_{t_0 \geq 0}$ is the process starting from $\nu^{(0)}(\bar{t})$ with components evolving according to \eqref{z0}, where $\theta$ is now a random variable that inherits its law from
			\be{}
			\lim_{n\to\infty}\CL[({\bf\Theta}^{\aux,(1),\Omega_{N_n}^{k+1}}(\bar{t}))_{i\in[N_n^{k}]}]
			\ee
			and, similarly, the laws of $y_{m,0}$, $1\leq l\leq k+1$ in the limiting process $(Z^{\nu^{(0)}(t_2)}(t_0))_{t_0 \geq 0}$ are determined by
			\be{}
			\lim_{n\to\infty}\CL[({\bf\Theta}^{\aux,(1),\Omega_{N_n}^{k+1}}(\bar{t}))_{i\in[N_n^{k}]}]. 
			\ee
			Note that we choose the subsequence $(N_n)_{n\in\N}$ in such a way that we know that the law $P^{(1)}$ in \eqref{1244} exists.
			
			\begin{proof}
				Proceed as in the proof of Proposition~\ref{ma77}. Note that the assumptions on the seed-banks in Proposition~\ref{ma77} follow from the choice of the subsequence and Lemma~\ref{lemflss}.
			\end{proof}   
			
			\item 
			Using the limiting evolution of the single colonies obtained in step 4, we can identify the limiting $l$-block process along the same subsequence. For $1\leq l\leq k$, we show that if 
			\begin{equation}
				\label{ll}
				\lim_{n\to\infty}\CL[({\bf\Theta}^{\aux,(l+1),\Omega_{N_n}^{k+1}}(\bar{t}))_{i\in[N_n^{k}]}]=P^{(l+1)},
			\end{equation}
			then
			\begin{equation}
				\begin{aligned} \lim_{n\to\infty}\CL\left[\left({\bf\Theta}^{\aux,(l),\Omega_{N_n}^{k+1}}(\bar{t}+N^l_nt_l)\right)_{t_l > 0,
						\, i\in[N_n^{k+1-l}]}\right]=\CL\left[(z_{l,{\bf\Theta}^{(l+1)}}^\aux(t))_{t\geq 0}\right],
				\end{aligned}
			\end{equation}
			where ${\bf\Theta}^{(l+1)}=({\bf \Theta}^{(l+1)}_x,({\bf \Theta}_{y_m,l}^{(l+1)})_{m=l+1}^{k+1})\in[0,1]\times[0,1]^{k+2-(l+1)}$,
			\begin{equation}
				\begin{aligned}
					\CL\left[z_{l,{\bf\Theta}^{(l)}}^\aux(0)\right]&=\Gamma_{{\bf\Theta}^{(l+1)}}^{(l),\aux},\\
					\Gamma_{{\bf\Theta}^{(l+1)}}^{(l),\aux}&=\int_{[0,1]\times[0,1]^{k+2-(l+1)}}\Gamma_{u}^{(l),\aux}P^{(l+1)}(\d u)
				\end{aligned}
			\end{equation}
			and $(z_{l,{\bf\Theta}^{(l+1)}}^\aux(t))_{t\geq 0}$ is the process evolving according to \eqref{z11s} with $\theta,(y_{m,l})_{m=l}^{k+1}$ replaced by the random variables ${\bf \Theta}^{(l+1)}_x,({\bf \Theta}^{(l+1)}_{y_m,l})_{m=l}^{k+1}$. Note that by the choice of the subsequence $(N_n)_{n\in\N}$ we know that for $1\leq l\leq k$ the limiting laws in \eqref{ll} exist.
			
			\begin{proof}
				The proof goes by induction. Using the convergence of the single components, we can proceed as in the proof of Proposition~\ref{lem:1blev} to prove the convergence of the $1$-blocks averages
				\begin{equation}
					\begin{aligned} 
						\lim_{n\to\infty}\CL\left[\left({\bf\Theta}_i^{\aux,(1),\Omega_{N_n}^{k+1}}
						(\bar{t}+N_nt_1)\right)_{t_1 > 0,\, i\in[N_n^{k+1-l}]}\right].
					\end{aligned}
				\end{equation} 
				Then, assuming that we have the convergence for all $0\leq l\leq L$, we get
				\begin{equation}
					\begin{aligned} 
						\lim_{n\to\infty}\CL\left[\left({\bf\Theta}_i^{\aux,(l),\Omega_{N_n}^{k+1}}(\bar{t}
						+N^{l}_nt_{l})\right)_{t_{l} > 0,\, i\in[N_n^{k+1-l}]}\right],
					\end{aligned}
				\end{equation}
				and we prove the convergence of
				\begin{equation}
					\begin{aligned} 
						\lim_{n\to\infty}\CL\left[\left({\bf\Theta}_i^{\aux,(L+1),\Omega_{N_n}^{k+1}}(\bar{t}
						+N^{(L+1)}_nt_{(L+1)})\right)_{t_{(L+1)} > 0,\, i\in[N_n^{k+1-(L+1)}]}\right].
					\end{aligned}
				\end{equation} 
				This is done using a similar proof strategy as in the proof of Proposition~\ref{lem:1blev}. In particular, we need to derive the $l$-level equivalent of Lemma~\ref{lemav3}. Since this lemma is also key to proving convergence in the Meyer-Zheng topology, we state it explicitly below.
			\end{proof}
			
			\begin{lemma}{\bf [$l$-block averages]}
				\label{fllemav3}
				Define
				\begin{equation}
					\begin{aligned}
						&\Delta^{(l),\Omega_N^{k+1}}_\Sigma (N^{l-1}t_{l-1})\\
						&\,=\frac{\Theta^{(l),\Omega_N^{k+1}}_x(N^{l-1}t_{l-1})
							+\sum_{m=0}^{l-2}K_m\Theta^{(l),\Omega_N^{k+1}}_{y_m}(N^{l-1}t_{l-1})}
						{1+\sum_{m=0}^{l-2}K_m}-\Theta^{(l),\Omega_N^{k+1}}_{y_{l-1}}(N^{l-1}t_{l-1})
					\end{aligned}
				\end{equation}
				and
				\begin{equation}
					R_l=\frac{1+\sum_{m=0}^{l-1}K_m}{1+\sum_{m=0}^{l-2}K_m}.
				\end{equation}
				For $t\geq 0$ set $\bar{ \Theta}^{(0)}(t)=\Theta_x^{(0)}(t)=x_0(t)$. Then, for $1\leq l\leq k$, 
				\begin{equation}
					\begin{aligned}
						&\E\left[\left|\Delta^{(l),\Omega_N^{k+1}}_\Sigma (N^{l-1}t_{l-1})\right|\right]\\
						&\leq \sqrt{\E\left[\left(\Delta^{(l),\Omega_N^{k+1}}_\Sigma (0)\right)^2\right]} \e^{-e_lR_lt_{l-1}}\\
						&\quad+\sqrt{\int_0^{t_1} \d s\,2e_lR_l\e^{-2e_lR_l(t_1-s)}
							\E \left[\left|\bar{\Theta}^{(l-1),\Omega_N^{k+1}}(N^{l-1}s)-\Theta^{(l-1),\Omega_N^{k+1}}_x(N^{l-1}s)\right|\right]}\\
						&\quad+\sqrt{\frac{1}{N}\frac{1}{2e_l(1+\sum_{m=0}^{l-1}K_m)}
							\left[\sum_{n=l+1}^k\frac{c_{n-1}}{N^{n-(l+1)}}+\sum_{m=l}^k\frac{K_me_m}{N^{m-l}}+E_{l-1}||g||\right]}.
					\end{aligned}
				\end{equation}
			\end{lemma}
			
			\begin{proof}
				Proceed like in the proof of Lemma~\ref{lemav3}, using the SSDE in \eqref{klevel} instead of the SSDE in \eqref{gh45a2twee}.
			\end{proof}
			
			We obtain the following useful corollary from Lemma~\ref{fllemav3}. 
			
			\begin{corollary}
				\label{cor124}
				For all $1\leq l\leq k$, $s>0$ and $\tilde{l}\geq l$,
				\begin{equation}
					\lim_{N \to \infty}\E \left[\left|\bar{\Theta}^{(l-1),\Omega_N^{k+1}}(N^{\tilde{l}-1}s)
					-\Theta^{(l-1),\Omega_N^{k+1}}_x(N^{\tilde{l}-1}s)\right|\right]=0.
				\end{equation}
			\end{corollary}
			
			\begin{proof}
				We proceed by induction. The result for $l=1$ is trivial. Suppose that the result holds for $l=L$. Then for $l=L+1$ we obtain
				\begin{equation}
					\label{8347}
					\begin{aligned}
						&\E \left[\left|\bar{\Theta}^{(L),\Omega_N^{k+1}}(N^{L}s)-\Theta^{(L),\Omega_N^{k+1}}_x(N^{L}s)\right|\right]\\
						&\leq\E \left[\left|\bar{\Theta}^{(L),\Omega_N^{k+1}}(N^{L}s)-\frac{1}{N}\sum_{i=0}^{N-1}\bar{\Theta}^{(L-1),
							\Omega_N^{k+1}}_{i}(N^{L}s)\right|\right]\\
						&\qquad+\frac{1}{N}\sum_{i=0}^{N-1}\E \left[\left|\bar{\Theta}^{(L-1),
							\Omega_N^{k+1}}_{i}(N^{L}s)-{ \Theta}^{(L-1),\Omega_N^{k+1}}_{x,i}(N^{L}s)\right|\right].
					\end{aligned}
				\end{equation}
				Note that the second term in the right-hand side of \eqref{8347} tends to $0$ as $N\to\infty$ by the induction hypothesis. For the first term in the right-hand side of \eqref{8347}, note that
				\begin{equation}
					\begin{aligned}
						&\E\left[\left|\bar{\Theta}^{(L),\Omega_N^{k+1}}(N^{L}s)-\frac{1}{N}\sum_{i=0}^{N-1}
						\bar{\Theta}^{(L-1),\Omega_N^{k+1}}_{i}(N^{L}s)\right|\right]\\
						&=\E\left[\left|\frac{\Theta^{(L),\Omega_N^{k+1}}_{x}(N^{L}s)+\sum_{m=0}^{L-1}K_m\Theta^{(L),
								\Omega_N^{k+1}}_{y_m}(N^{L}s)}{1+\sum_{m=0}^{L-1}K_m}\right.\right.\\
						&\qquad\qquad-\left.\left.\frac{\Theta^{(L),\Omega_N^{k+1}}_{x}(N^{L}s)+\sum_{m=0}^{L-2}K_m
							\Theta^{(L),\Omega_N^{k+1}}_{y_m}(N^{L}s)}{1+\sum_{m=0}^{L-2}K_m}\right|\right]\\
						&=\frac{K_{L-1}}{1+\sum_{m=0}^{L-1}K_m}\E\left[\left|\Theta^{(L),\Omega_N^{k+1}}_{y_{L-1}}(N^{L}s)
						-\bar{ \Theta}^{(L),\Omega_N^{k+1}}_{}(N^{L}s)\right|\right].	
					\end{aligned}
				\end{equation}
				Invoking Lemma~\ref{fllemav3} and using the induction hypothesis, we see that for $s>0$ and $\tilde{l}\geq L$ indeed
				\begin{equation}
					\lim_{N\to\infty}\E \left[\left|\bar{\Theta}^{(L),\Omega_N^{k+1}}(N^{\tilde{l}}s)
					-\Theta^{(L),\Omega_N^{k+1}}_x(N^{\tilde{l}}s)\right|\right]=0.
				\end{equation}
			\end{proof}
			
			\item 
			Show that the convergence in step 4 and step 5 actually holds along each subsequence. Therefore we obtain the limiting evolution of the single colonies, the auxiliary $1$-block process and the effective $2$-block process. This follows from the fact that the auxiliary $k$-estimator process converges to the same limit along every subsequence. Consequently, the same holds for the auxiliary $k-1$-estimator process. In this way we can traverse back through the levels to obtain that all $l$-estimator process converges as $N\to\infty$.
			
			Define, for $0\leq l\leq k$,
			\begin{equation}
				\mathfrak{s}^{k+1}_l=[0,1]\times[0,1]^{k+2-l}.
			\end{equation}
			We obtain, for $0\leq l\leq k-1$, 
			\begin{equation}
				\label{ll2}
				\begin{aligned}
					&\lim_{N\to\infty}\CL[({\bf\Theta}^{\aux,(l+1),\Omega_{N}^{k+1}}(\bar{t}))]]=\Gamma_{{\bf\Theta}^{(l+2)}}^{(l+1),\aux},\\
					&\Gamma_{{\bf\Theta}^{(l+2)}}^{(l+1),\aux}=\int_{\mathfrak{s}^{k+1}_{l+2}}\cdots\int_{\mathfrak{s}^{k+1}_{k}}
					\Gamma^{(k),\aux}_{\vt_k}(\d u_k)\cdots\Gamma^{(l+2),\aux}_{u_{l+3}}(\d u_{l+2}) \Gamma^{(l+1),\aux}_{u_{l+2}}.
				\end{aligned}
			\end{equation}
			Therefore by step 5
			\begin{equation}
				\label{1262}
				\begin{aligned} 
					\lim_{N\to\infty}\CL\left[\left({\bf\Theta}^{\aux,(l),\Omega_{N}^{k+1}}(\bar{t}+N^l_nt_l)\right)_{t_l > 0,\, i\in[N^{k+1-l}]}\right]
					=\CL\left[(z_{l,{\bf\Theta}^{(l+1)}}^\aux(t))_{t\geq 0}\right],
				\end{aligned}
			\end{equation}
			where ${\bf\Theta}^{(l+1)}=({\bf \Theta}^{(l+1)}_x,({\bf \Theta}_{y_m,l}^{(l+1)})_{m=l}^{k+1})\in\mathfrak{s}^{(k+1)}_l$ are random variables with law
			\begin{equation}
				\CL\left[{\bf\Theta}^{(l+1)}\right]=\Gamma_{{\bf\Theta}^{(l+2)}}^{(l+1),\aux}.
			\end{equation}
			The initial state of the limiting process in \eqref{1262} is given by
			\begin{equation}
				\begin{aligned}
					\CL\left[z_{l,{\bf\Theta}^{(l)}}^\aux(0)\right]&=\Gamma_{{\bf\Theta}^{(l+1)}}^{(l),\aux},\\
					\Gamma_{{\bf\Theta}^{(l+1)}}^{(l),\aux}&=\int_{\mathfrak{s}^{k+1}_{l+2}}\cdots
					\int_{\mathfrak{s}^{k+1}_{k}}\Gamma^{(k),\aux}_{\vt+k}(\d u_k)\cdots\Gamma^{(l+2),\aux}_{u_{l+2}}(\d u_{l+1}) 
					\Gamma^{(l+1),\aux}_{u_{l+1}}
				\end{aligned}
			\end{equation}
			and $(z_{l,{\bf\Theta}^{(l+1)}}^\aux(t))_{t\geq 0}$ is the process evolving according to \eqref{z11s} with $\theta,(y_{m,l})_{m=l+1}^{k+1}$ replaced by the random variables ${\bf \Theta}^{(l+1)}_x,({\bf \Theta}^{(l+1)}_{y_m})_{m=l+1}^{k+1}$. Recall that, by Lemma~\ref{lemflss}, we have, for $l+1\leq m\leq k+1$,
			\begin{equation}
				{\bf \Theta}^{(l+1)}_{y_m}={\bf \Theta}^{(m)}_{y_m}.
			\end{equation}
			\item
			Use the Meyer-Zheng topology to obtain Proposition~\ref{P.finitesysmfmulti}(b).
			
			\begin{proof}
				Note that Lemma~\ref{fllemav3} and Corollary~\ref{cor124} together imply
				\begin{equation}
					\begin{aligned}
						\lim_{N \to \infty}\E\left[\left|\bar{ \Theta}^{(l),\Omega_N^{k+1}}(N^lt_l)-\Theta_x^{(l),\Omega_N^{k+1}}(N^lt_l)\right|\right]&=0,\\
						\lim_{N \to \infty}\E\left[\left|\bar{ \Theta}^{(l),\Omega_N^{k+1}}(N^lt_l)-\Theta_{y_m}^{(l),
							\Omega_N^{k+1}}(N^lt_l)\right|\right]&=0,\text{ for }0\leq m\leq l-1.
					\end{aligned}
				\end{equation}
				Combining the result obtained in step 6 with the proof strategy followed in Section~\ref{ss.11mz}, we get the claim.
			\end{proof}
			
			\item 
			Finally, we prove Proposition~\ref{P.finitesysmfmulti}(a). 
			
			\begin{proof}
				Step 6 and step 7 yield the laws of the components $\CL[M_l^k]$ of the interaction chain $(M_{-l}^{k})_{-l=-(k+1)}^{0}$. Note that the state space $([0,1]\times[0,1]^{k+2})^{k+2}$ is compact, and therefore the sequence of random variables
				\begin{equation}
					\left(\Bigg({\bf \Theta}^{(l),\Omega_N^{k+1}}\left(\bar{t}\right)\Bigg)_{l=k+1,k,\ldots,0}\right)_{N\in\N}
				\end{equation}
				is tight. For any
				\begin{equation}
					\label{1267}
					\begin{aligned}
						&f\colon\,([0,1]\times[0,1]^{k+2})^{n+2}\to \R,\\
						&f(x)=\prod_{i=1}^nf_i(x_i),\\
						&f_i\in\CC_b([0,1]\to \R),
					\end{aligned}
				\end{equation} 
				we can use conditioning on the previous block average to obtain
				\begin{equation}
					\lim_{N \to \infty}\E\left[f\left(\Bigg({\bf \Theta}^{(l),\Omega_N^{k+1}}\left(\bar{t}\right)\Bigg)_{l=k+1,k,\ldots,0}\right)\right]
					=\E\left[f\left((M_{-l}^{k})_{-l=-(k+1)}^{0}\right)\right].
				\end{equation}
				Using that the set of functions of the form \eqref{1267} is separating, we obtain the claim.
			\end{proof}
			
		\end{enumerate} 
		
		\subsection{Proof: of the hierarchical multi-scale limit theorems.}
		\label{ss.pmslt}
		
		In this section we prove Theorems~\ref{T.multiscalehiereff} and \ref{T.multiscalehier}. We start by proving Theorem~\ref{T.multiscalehier}. Theorem~\ref{T.multiscalehiereff} will follow from Theorem~\ref{T.multiscalehier} by projection onto the effective components.
		
		\paragraph{Proof of Theorem~\ref{T.multiscalehiereff}}
		\begin{proof}
			Recall the estimators in \eqref{tam6alt}. Like for the finite-level hierarchical mean-field system, we can define the auxiliary estimator process by
			\begin{equation}
				\boldsymbol{\Theta}^{(l),\aux,\Omega_N}(t)
				=\big(\bar{\Theta}^{(l) ,\Omega_N}(t),\big(\Theta^{(l) ,\Omega_N}_{y_m}(t)\big)_{m=l}^{\infty}\big).
			\end{equation}
			For $l,k\in\N$ the processes $({\bf\Theta}^{(l) ,\aux, \Omega_N}(\bar{t}+N^kt))_{t>0}$ evolve according to (recall, \eqref{464}) 
			\begin{equation}
				\label{464pr}
				\begin{aligned}
					\d \bar{\Theta}^{(l) ,\Omega_N}(N^kt)&=E_{l}\sum_{n=l+1}^\infty\frac{c_{n-1}}{N^{n-1-k}}
					\left[\Theta_x^{(n),\Omega_N}(N^kt)-\Theta_x^{(l),\Omega_N}(N^kt)\right]\,\d t\\
					&\qquad+E_l\sqrt{\frac{N^k}{N^{2l}}\sum_{\xi\in B_l}g\big(x_\xi(N^kt)\big)}\,\d w(t)\\
					&\qquad+E_l\sum_{m=l}^\infty \frac{K_me_m}{N^{m-k}}
					\left[\Theta^{(l) ,\Omega_N}_{y_m}(N^{k}t)-\Theta^{(l) ,\Omega_N}_{x}(N^{k}t)\right]\,\d t,\\[0.2cm]
					\d \Theta^{(l) ,\Omega_N}_{y_m}(N^{k}t)&=\frac{e_m}{N^{m-k}}
					\left[\Theta^{(l) ,\Omega_N}_{x}(N^{k}t)-\Theta^{(l) ,\Omega_N}_{y_m}(N^{k}t)\right]\,\d t, \qquad l\leq m\leq \infty.
				\end{aligned}
			\end{equation}
			Therefore, for $l>k$ and all $\epsilon>0$,
			\begin{equation}
				\label{1272}
				\begin{aligned}
					&\P\Bigg[\sup_{0\leq t\leq L(N)}\Bigg|\bar{\Theta}^{(l),\Omega_N}(\bar{t})-\bar{\Theta}^{(l),
						\Omega_N}(\bar{t}+N^kt)\Bigg|>\epsilon\Bigg]\\
					&=\P\Bigg[\sup_{0 \leq t\leq L(N)}E_l\Bigg|\int_{\bar{t}}^{\bar{t}+N^kt} \d r\,\sum_{n=l+1}^\infty\frac{c_{n-1}}{N^{n-1}}
					\left[\Theta_x^{(n),\Omega_N}(r)-\Theta_x^{(l),\Omega_N}(r)\right]\\
					&\qquad\qquad\qquad+\int_{\bar{t}}^{\bar{t}+N^kt}\d r\,\sum_{m=l}^\infty \frac{K_me_m}{N^{m}}
					\left[\Theta^{(l) ,\Omega_N}_{y_m}(r)-\Theta^{(l) ,\Omega_N}_{x}(r)\right]\\
					&\qquad\qquad\qquad+\int_{\bar{t}}^{\bar{t}+N^kt} \d w_i(r)\,\sqrt{\frac{1}{N^{2l}}\sum_{\xi\in B_l}g
						\big(x_\xi(r)\big)}\,\Bigg|>\epsilon\Bigg]\\
					&\leq\P\Bigg[\sup_{0\leq t\leq L(N)}\,E_l\Bigg|\int_{\bar{t}}^{\bar{t}+N^kt} \d w_i(r)\,\sqrt{\frac{1}{N^{2l}}
						\sum_{\xi\in B_l}g\big(x_\xi(r)\big)}\,\,\Bigg|\\
					&\qquad\qquad > \epsilon-t\left[\sum_{n=l+1}^\infty\frac{c_{n-1}}{N^{n-1-k}}+\sum_{m=l}^\infty 
					\frac{K_me_m}{N^{m-k}}\right]\Bigg].
				\end{aligned}
			\end{equation} 
			Note that, since $l>k$,
			\begin{equation}
				\lim_{N \to \infty}t\left[\sum_{n=l+1}^\infty\frac{c_{n-1}}{N^{n-1-k}}+\sum_{m=l}^\infty \frac{K_me_m}{N^{m-k}}\right]=0.
			\end{equation}
			Hence, like in the proof of Lemma~\ref{lemsta2}, we can use an optional stopping argument to obtain
			\begin{eqnarray}
				\lim_{N \to \infty}\sup_{0\leq t\leq L(N)}\Bigg|\bar{\Theta}^{(l),\Omega_N}(\bar{t})-\bar{\Theta}^{(l),
					\Omega_N}(\bar{t}+N^kt)\Bigg|=0\qquad\text{ in probability}.
			\end{eqnarray} 
			Using a similar computation as in \eqref{1272}, we can show
			\begin{eqnarray}
				\lim_{N \to \infty}\sup_{0\leq t\leq L(N)}\Bigg|\bar{\Theta}_{y_m}^{(l),\Omega_N}(\bar{t})-{\Theta}_{y_m}^{(l),
					\Omega_N}(\bar{t}+N^kt)\Bigg|=0\qquad\text{ in probability}.
			\end{eqnarray} 
			Hence we obtain that, on time scale $N^k$ as $N\to\infty$, the process $({\bf\Theta}^{(l) ,\aux, \Omega_N}(\bar{t}+N^kt))_{t>0}$ does not evolve and therefore is still in its initial state$({\bf\Theta}^{(l) ,\aux, \Omega_N}(\bar{t}))$.
			
			Using that the $l$-auxiliary estimator processes do not move for $l>k$, they function like the ``outside world" for the finite-level mean-field system in Section~\ref{ss.multilevelmf}. Therefore we can proceed as in the proof of Proposition~\ref{P.finitesysmfmulti}
			to prove the second and third line in \eqref{multiconvalt} in Theorem~\ref{T.multiscalehiereff}. The $l$-block estimator process $({\bf\Theta}^{(l),\aux ,\Omega_N}(\bar{t}+N^lt))_{t>0}$ evolves according to \eqref{464pr} with $l=k$. Note that the extra interactions due to migration over larger blocks $l>k$ and exchange with deeper seed-banks $m>k$ in \eqref{463b} are of order $\CO(1/N)$. Therefore these terms vanish as $N\to\infty$, and we can just proceed as in the scheme of Section~\ref{ss.flmfss}, to obtain the second and third line in \eqref{multiconvalt} in Theorem~\ref{T.multiscalehiereff}. Using these results, we obtain that, for $l>k$, 
			\begin{equation}
				{\bf\Theta}^{(l),\aux ,\Omega_N}(\bar{t})=\delta_{M^k_{-(k+1)}}.
			\end{equation} 
		\end{proof}

		\section{Proofs: $N\to\infty$, orbit of renormalisation transformation}
		\label{s.renormasym}
		
		In this section we analyse the orbit of the renormalisation transformation and show that it has the Fisher-Wright diffusion as a global attractor. In Section~\ref{ss.momrels} we write down moment relations for the equilibrium defined in \eqref{z11a} for single colonies (Proposition~\ref{P.momrel}) and for block averages (Proposition~\ref{P.momrelalt}). In Section~\ref{ss.momrelsit} we derive the iterates of these moment relations for single colonies (Proposition~\ref{P.momreliterate}) and for blocks (Proposition~\ref{P.momreliterate2}). In Section~\ref{ss.clustering} we prove clustering (Propositions~\ref{limQ}--\ref{limQ2}). In Section~\ref{ss.orbit} we prove Theorems~\ref{T.scalren} and \ref{T.dichotomy}, and work out the dichotomy of a finite seed-bank ($\rho<\infty$) versus infinite seed-bank ($\rho=\infty$).
		
		\subsection{Moment relations}
		\label{ss.momrels}
		
		We use It\^o-calculus to compute the mixed moments. Recall $\theta_x, (\theta_{y_m})$ as defined in \eqref{initialvalue}, $\vt_k$ as defined in \eqref{deftheta} and $\bar{\vt}^{(k)}$ \eqref{seqtheta}.
		Also recall $E_k$ as defined in \eqref{Ekdef}. Abbreviate
		\begin{equation} 
			\label{modefA0}
			A_0^n = \frac{1}{2}\sum_{k=0}^{n}\frac{E_k}{c_k}
			\frac{(E_kc_k+e_k)}{(E_kc_k+e_k)+E_k K_ke_k }, \qquad n\in\N,
		\end{equation}
		and
		\begin{equation}
			B_0 = \frac{1}{2}\frac{E_0^2}{(E_0c_0+e_0)+E_0K_0e_0}.
		\end{equation}
		
		In the following proposition, the first five equations are first and second moment relations, while the last equation is the definition of the renormalisation transformation. Later we will see that this set of equations can be iterated.   
		
		\begin{proposition}{\bf [Moment relations: single colonies]}
			\label{P.momrel}
			$\mbox{}$\\
			Let $\vt_0$ be as defined in \eqref{deftheta}, and let $\Gamma_{(\vt_0,y_l)}^{(0)}=\Gamma_{(\vt_0,y_l)}^{g,c_0,E_0,K_0,e_0}$ be the equilibrium of \eqref{z11a} measure defined in \eqref{singcoleq322alt} with $k=0$, with $g\in\CG$, $c_0 \in (0,\infty)$, $E_0\in[0,1]$ and $K_0,e_0 \in(0,\infty)$. Then the following moment relations hold: 
			\small
			\begin{eqnarray}
				\label{e590a0}
				\int_{[0,1]\times[0,1]^{\N_0}} x_0 \; \Gamma_{(\vt_0,y_l)}^{g,E_0,c_0,K_0,e_0}(\d z_0)  
				&=&  \vt_0, \\
				\label{e590a'0}
				\int_{[0,1]\times[0,1]^{\N_0}} y_{0,0} \; \Gamma_{(\vt_0,y_l)}^{g,E_0,c_0,K_0,e_0}(\d z_0) 
				&=&  \vt_0, \\
				\label{e590b0}
				\int_{[0,1]\times[0,1]^{\N_0}} x_0y_{0,0} \; \Gamma_{(\vt_0,y_l)}^{g,E_0,c_0,K_0,e_0}(\d z_0) 
				&=& \int_{[0,1]\times[0,1]^{\N_0}} y_{0,0}^2 \; \Gamma_{(\vt_0,y_l)}^{g,E_0,c_0, K_0,e_0}(\d z_0),\\
				\label{e590d0} 
				\int_{[0,1]\times[0,1]^{\N_0}} x_0^2 \; \Gamma_{(\vt_0,y_l)}^{g,E_0,c_0,K_0,e_0}(\d z_0) 
				&=&  \vt_0^2+ A_0^0(\CF g)(\vt_0),\label{e590c0} \\
				\int_{[0,1]\times[0,1]^{\N_0}} y_{0,0}^2 \; \Gamma_{(\vt_0,y_l)}^{g,E_0,c_0,K_0,e_0}(\d z_0)  
				&=&  \vt_0^2 + (A_0^0-B_{0})(\CF g)( \vt_0),\\
				\label{e590e0}
				\int_{[0,1]\times[0,1]^{\N_0}} g(x_0) \; \Gamma_{(\vt_0,y_l)}^{g,E_0,c_0,K_0,e_0}(\d z_0)  
				&=& (\CF g)( \vt_0).
			\end{eqnarray}
			\normalsize
		\end{proposition}
		
		\begin{proof} 
			For ease of notation we write $x,y_0$ instead of $x_0,y_{0,0}$ for the single colonies. We use It\^o's formula to compute the first and second moments, and invoke the equilibrium condition to get the above formulas, except for the last formula, which is the definition of $\CF$ in \eqref{renor}. 
			
			\medskip\noindent
			{\bf 1.}	
			We begin with the first moments of $x$ and $y_0$. For $k=0$, \eqref{z11a} becomes
			\begin{eqnarray}
				\label{mocs1}
				\d x(t) &=& E_0\left[c_0\big[\vt_0 - x(t)\big]\, \d t + \sqrt{g(x(t))}\, \d w_0(t) 
				+ K_0e_0\,[y_0(t)-x(t)]\,\d t\right],\\
				\d y_0(t) &=& e_0\,[x(t)-y_0(t)]\, \d t,\label{mocs2}\\
				\d y_m(t)&=&0. \nonumber
			\end{eqnarray}
			In equilibrium the distribution of $x(t)$ is constant in time, and so $\frac{\d}{\d t}\mathbb{E}[x(t)]=0$, where $\mathbb{E}$ denotes expectation w.r.t.\ $\Gamma_\theta^{g,c_0,E_0,K_0,e_0}$. Integrating \eqref{mocs1} and taking the expectation, we get
			\begin{eqnarray}
				\mathbb{E}[x(t)-x(0)] 
				\label{ma}&=&E_0\left[\mathbb{E}\left[\int_{0}^{t} \d s\,c_0\big[\vt_0-x(s)\big]
				+K_0e_0 \int_{0}^{t} \d s\,[y_0(s)-x(s)]\right]\right],
				\nonumber \\
				& = &E_0 \left\{ \int_{0}^{t} \d s\,\mathbb{E}\Big[c_0\big[\vt_0-x(s)\big]
				+K_0e_0 \big[y_0(s)-x(s)\big]\Big]\right\},
				\label{mc}
			\end{eqnarray}
			where in the second equation we use Fubini. Turning back to differential notation, we see from
			\eqref{mc} that
			\begin{equation}
				\frac{\d}{\d t}\,\mathbb{E}[x(0)] = 0 = E_0\left\{\mathbb{E}\Big[c_0\big[\vt_0-x(t)\big]
				+K_0e_0\big[y_0(t)-x(t)\big]\Big]\right\},
				\label{md}
			\end{equation} 
			and it follows that 
			\begin{equation}
				\mathbb{E}\Big[c_0\big[\vt_0-x\big]+K_0e_0\big[y_0-x\big]\Big]=0.
				\label{m2}
			\end{equation}
			In the same way it follows from \eqref{mocs2} that
			\begin{equation}
				\mathbb{E}\Big[e_0 \big[x-y_0\big]\Big]=0.
				\label{m1}
			\end{equation}
			Therefore we obtain from \eqref{m2}--\eqref{m1} that
			\begin{equation}
				\mathbb{E}[x]=\mathbb{E}[y_0]=\vt_0.
				\label{m3*}
			\end{equation}
			
			\medskip\noindent
			{\bf 2.}
			We next compute the second moments. By It\^o's formula,
			\begin{align}
				& \d (x(t))^2 = 2x(t)\,\d x(t) +(\d x(t))^2 \\ \nonumber  
				&= 2c_0 x(t)E_0 [\vt_0-x(t)]\,\d t + 2x(t)E_0\sqrt{g(x(t))}\,\d w_0(t)\\ \nonumber
				&\qquad +E_0 \big[2K_0e_0\,x(t)y_0(t)-2K_0e_0\, x^2_t\big]\,\d t +E_0^2g(x(t))\,\d t.
			\end{align}
			Taking expectations and using that we are in equilibrium, we get
			\begin{eqnarray}
				0=2c_0\vt_0^2-2c_0\mathbb{E}[x^2]+2K_0e_0 \mathbb{E}[x y_0]
				-2K_0e_0 \mathbb{E}[x^2]+E_0\mathbb{E}[g(x)].
			\end{eqnarray}
			Using $\mathbb{E}[g(x)]=(\CF g)( \vt_0)$, we find
			\begin{equation}
				\label{xeq50}
				\mathbb{E}[x^2] = \frac{c_0}{(c_0+K_0e_0 )} \vt_0^2+\frac{K_0e_0 }{(c_0+K_0e_0 )}
				\mathbb{E}[xy_0]+\frac{E_0}{2(c_0+K_0e_0 )}(\CF g) ( \vt_0).
			\end{equation}
			In the same way we find
			\begin{equation}
				\label{e595*}
				\mathbb{E}[y_0^2] = \mathbb{E}[xy_0],
			\end{equation}	
			and for the mixed second moment
			\begin{equation}
				\label{e593*}
				\mathbb{E}[xy_0] = \frac{E_0c_0}{(E_0c_0+e_0)}  \vt_0^2 
				+ \frac{e_0}{(E_0c_0+e_0)} \mathbb{E}[x^2].
			\end{equation}
			Substituting \eqref{e593*} into \eqref{xeq50}, we find $\E[x^2]$ and hence also $\E[y_0^2]$  and $\E[x_0y_0]$. This finishes the proof of Proposition~\ref{P.momrel}.
		\end{proof}
		
		Similar moment relations can be derived for the equilibrium measures of the block averages. Define
		\begin{equation} 
			\label{modefA0alt}
			A_{m}^{n} =\frac{1}{2}\sum_{k=m}^{n}\frac{E_k}{c_k}
			\frac{(E_kc_k+e_k)}{(E_kc_k+e_k)+E_k K_k e_k}, \qquad m\in\N_0,\,n\in\N,
		\end{equation}
		and
		\begin{equation} 
			\label{modefB0alt}
			B_{m} =\frac{1}{2} \frac{E_m^2}{(E_mc_m+e_m)+E_m K_m e_m}, \qquad m\in\N_0.
		\end{equation}
		Recall the definition of $\CF^{(n)}$ in \eqref{frenormit}.
		
		\begin{proposition}{\bf [Moment relations: blocks]}
			\label{P.momrelalt}
			$\mbox{}$\\
			Let $\vt_m$ be as defined in \eqref{deftheta}, and let $\Gamma_{(\vt_m,y_m)}^{(m)}=\Gamma_{(\vt_m,y_m)}^{\CF^{(m)}g,c_m,E_m,K_m,e_m}$ be the equilibrium measure of \eqref{z11a} with $k=m$, with $g\in\CG$, $c_0 \in (0,\infty)$, $E_0\in[0,1]$ and $K_0,e_0 \in(0,\infty)$. Then the following moment relations hold:
			\small
			\begin{eqnarray}
				\label{e590a0alt}
				\int_{[0,1]\times[0,1]^{\N_0}} x_m \; \Gamma_{(\vt_m,y_m)}^{\CF^{(m)}g,E_m,c_m,K_m,e_m}(\d z_m)  
				&=& \vt_m, \\
				\label{e590a'0alt}
				\int_{[0,1]\times[0,1]^{\N_0}} y_{m,m} \; \Gamma_{(\vt_m,y_m)}^{\CF^{(m)}g,E_m,c_m,K_m,e_m}(\d z_m)
				&=& \vt_m, \\
				\label{e590b0alt}
				\int_{[0,1]\times[0,1]^{\N_0}} x_my_{m,m} \; \Gamma_{(\vt_m,y_m)}^{\CF^{(m)}g,E_m,c_m,K_m,e_m}(\d z_m)
				&& \\ \nonumber
				= \int_{[0,1]\times[0,1]^{\N_0}} y_{m,m}^2\; \Gamma_{(\vt_m,y_m)}^{\CF^{(m)}g,E_m,c_m,K_m,e_m}(\d z_m),\\
				\label{e590d0alt} 
				\int_{[0,1]\times[0,1]^{\N_0}} x_m^2 \; \Gamma_{(\vt_m,y_m)}^{\CF^{(m)}g,E_m,c_m,K_m,e_m}(\d z_m) 
				&=&  \vt_m^2+ A_m^m(\CF^{(m+1)} g)(\vt_m),\label{e590c0alt} \\
				\int_{[0,1]\times[0,1]^{\N_0}} y_{m,m}^2 \; \Gamma_{(\vt_m,y_m)}^{\CF^{(m)}g,E_m,c_m,K_m,e_m}(\d z_m) 
				&=&  \vt_m^2 + (A_m^m-B_m)(\CF^{(m+1)}g)( \vt_m),\\
				\label{e590e0alt}
				\int_{[0,1]\times[0,1]^{\N_0}} (\CF^{(m)}g)(x_m) \; \d\Gamma_{(\vt_m,y_m)}^{\CF^{(m)}g,E_m,c_m,K_m,e_m}(\d z_m)  
				&=& (\CF^{(m+1)} g)( \vt_m).
			\end{eqnarray}
			\normalsize
		\end{proposition}
		
		\begin{proof}
			The proof follows the same line of argument as the proof of Proposition~\ref{P.momrel}.  
		\end{proof}

		\subsection{Iterate moment relations}
		\label{ss.momrelsit}
		Recall the kernels defined in \eqref{ker}, the iterates of the kernels defined in \eqref{conn2} and $\bar{\vt}^{(n)}$.
		Recall that $Q^{(n)}(\bar{\vt}^{(n)},\d z_0)$ is the probability density to see the population of a single colony in state $z_0$ given that the $(n+1)$-block averages equal $\bar{\vt}^{(n)}$.
		
		\begin{proposition}\label{P.momreliterate}{{\bf [Iterated moment relations: single components]}}
			For $n\in\N_0$,
			\small
			\begin{eqnarray}
				\label{e590a*}
				\int_{[0,1]\times[0,1]^{\N_0}} x_0 \; Q^{(n)}(\bar{\vt}^{(n)},\d z_0)
				&=&  \vartheta_n, \\
				\label{e590a'*}
				\int_{[0,1]\times[0,1]^{\N_0}} y_{0,0} \; Q^{(n)}(\bar{\vt}^{(n)},\d z_0)
				&=&  \vartheta_n, \\
				\label{e590b*}
				\int_{[0,1]\times[0,1]^{\N_0}} x_0^2 \; Q^{(n)}(\bar{\vt}^{(n)},\d z_0)
				&=&  \;  \vartheta_n^2 + A_0^n \; (\CF^{(n+1)} g)(\vartheta_n), \\
				\label{e590c*} 
				\int_{[0,1]\times[0,1]^{\N_0}} y_{0,0}^2 \; Q^{(n)}(\bar{\vt}^{(n)},\d z_0) 
				&=&  \;  \vartheta_n^2 +  (A_0^n-B_0)(\CF^{(n+1)} g)(\vartheta_n), \\
				\label{e590d*}
				\int_{[0,1]\times[0,1]^{\N_0}} x_0y_{0,0} \; Q^{(n)}(\bar{\vt}^{(n)},\d z_0) 
				&=&  \int_{[0,1]\times[0,1]^{\N_0}} y^2_{0,0} \; Q^{(n)}(\vartheta_n,\d z_0),\\
				\label{e590e*}
				\int_{[0,1]\times[0,1]^{\N_0}} g(x) \; Q^{(n)}(\bar{\vt}^{(n)},\d z_0)
				&=& (\CF^{(n+1)} g)(\vartheta_n).
			\end{eqnarray}
			\normalsize
		\end{proposition}
		
		\begin{proof}
			We prove the claim for $x_0^2$ only. The other relations follow in a similar way. The proof proceeds by induction. The result for $n=0$ follows directly from Proposition~\ref{P.momrel}.  Assume the result holds true for $n=n$, for $n=n+1$ we write
			\begin{eqnarray}
				&&\int_{[0,1]\times[0,1]^{\N_0}} x_0^2 \; Q^{(n+1)}(\bar{\vt}^{(n+1)},\d z_0)\\  \nonumber
				&&\quad = \int_{[0,1]\times[0,1]^{\N_0}} x_0^2\,(Q^{[n+1]}\circ Q^{(n)})(\bar{\vt}^{(n+1)},\d z_0)\\ \nonumber
				&&\quad = \int_{[0,1]\times[0,1]^{\N_0}} \int_{[0,1]\times[0,1]^{\N_0}} x_0^2\,Q^{[n+1]}(\bar{\vt}^{(n+1)}, \d z_{n+1})\, 
				Q^{(n)}(z_{n+1},\d z_0)\\  \nonumber
				&&\quad = \int_{[0,1]\times[0,1]^{\N_0}} \left[\int_{[0,1]\times[0,1]^{\N_0}} x_0^2\,Q^{(n)}(z_{n+1},\d z_0)\right]\,
				Q^{[n+1]}(\bar{\vt}^{(n+1)}, \d z_{n+1})\\  \nonumber
				&&\quad = \int_{[0,1]\times[0,1]^{\N_0}} \Big[x_{n+1}^2 + A_{0}^{n}\;(\CF^{(n+1)} g)(x_{n+1})\Big]
				\Gamma_{\bar{\vt}^{(n+1)}}( \d z_{n+1} )\\  \nonumber
				&&\quad= \vt_{n+1}^2 + A_{0}^n\; (\CF^{(n+2)} g)(\vt_{n+1}) 
				+ A_{n+1}^{n+1} \; (\CF^{(n+2)} g) (\vt_{n+1})\nonumber\\
				&&\quad = \theta_{n+1}^2 + A_0^{n+1} \; (\CF^{(n+2)} g)(\vt_{n+1}).
				\nonumber
			\end{eqnarray}
			The first and second equality use the definition in \eqref{conn2}, the third equality uses Fubini, the fourth equality is the induction step, the fifth equality uses Proposition~\ref{P.momrelalt}, in particular, \eqref{e590c0alt} and \eqref{e590e0alt}.
		\end{proof}
		
		Similar iterate moment relations hold for blocks. Define, for $m,n\in\N_0$ with $n \geq m$,
		\be{}
		Q_m^{(n)}=Q^{[n]} \circ \cdots \circ Q^{[m]}.
		\ee 
		
		\begin{proposition}\label{P.momreliterate2}{{\bf [Iterated moment relations: blocks of components]}}
			For $n,m\in\N_0$ with $n \geq m$,
			\small
			\begin{eqnarray}
				\label{e590a*alt}
				\int_{[0,1]\times[0,1]^{\N_0}} x_m \; Q_m^{(n)}(\bar{\vt}^{(n)},\d z_m)
				&=&  \vartheta_n, \\
				\label{e590a'*alt}
				\int_{[0,1]\times[0,1]^{\N_0}} y_{m,m} \; Q_m^{(n)}(\bar{\vt}^{(n)},\d z_m)
				&=&  \vartheta_n, \\
				\label{e590b*alt}
				\int_{[0,1]\times[0,1]^{\N_0}} x_m^2 \; Q_m^{(n)}(\bar{\vt}^{(n)},\d z_m)
				&=&  \;  \vartheta_n^2 + A_m^n \; (\CF^{(n+1)} g)(\vartheta_n), \\
				\label{e590c} 
				\int_{[0,1]\times[0,1]^{\N_0}} y_{m,m}^2 \; Q_m^{(n)}(\bar{\vt}^{(n)},\d z_m) 
				&=&  \;  \vartheta_n^2+  (A_m^n-B_m)(\CF^{(n+1)} g)(\vartheta_n), \\
				\label{e590d*alt}
				\int_{[0,1]\times[0,1]^{\N_0}} x_my_{m,m} \; Q_m^{(n)}(\bar{\vt}^{(n)},\d z_m) 
				&=&  \int_{[0,1]\times[0,1]^{\N_0}} y^2_{m,m} \; Q_m^{(n)}(\bar{\vt}^{(n)},\d z_m),\\
				\label{e590e*alt}
				\int_{[0,1]\times[0,1]^{\N_0}} (\CF^{(m)}g)(x_m) \; Q_m^{(n)}(\bar{\vt}^{(n)},\d z_m)
				&=& (\CF^{(n+1)} g)(\vartheta_n).
			\end{eqnarray}
			\normalsize
		\end{proposition}
		
		\begin{proof}
			Follow a similar induction argument as in the proof of Proposition~\ref{P.momreliterate}.
		\end{proof}

		\subsection{Clustering}
		\label{ss.clustering} 
		
		To prove Theorem~\ref{T.scalren}, we proceed as in \cite{BCGH95}. The following clustering property holds for the kernels associated with single colonies. 
		
		\begin{proposition}
			\label{limQ}{{\bf [Clustering: single colonies]}} 
			Assume
			\begin{equation}
				\lim_{n\to\infty}\vt_n=\theta.
			\end{equation}
			Then
			\be{}
			\begin{aligned}\label{ifc}
				&\lim_{n\to\infty}Q^{(n)}\big(\bar{\vt}^{(n)},\{(x_0,y_{(0,0)})=(0,0)\}\big)=1-\theta,\\
				&\lim_{n\to\infty}Q^{(n)}\big(\bar{\vt}^{(n)},\{(x_0,y_{(0,0)})=(1,1)\}\big)=\theta,
			\end{aligned}
			\ee
			if and only if 
			\begin{equation}
				\label{ifa}
				\lim_{n\to\infty}A_0^n=\infty.
			\end{equation}
			Consequently,
			\be{}
			\lim_{n\to\infty}Q^{(n)}\big(\bar{\vt}^{(n)},\{(x_0,y_{0,0})\}\notin\{(0,0),(1,1)\}\big)=0.
			\ee
		\end{proposition}
		
		\begin{proof}
			The proof exploits the iterated moment relations. First assume \eqref{ifa}
			
			\medskip\noindent
			{\bf 1.}
			By Proposition~\ref{P.momreliterate}
			\begin{equation}
				\label{limKn}
				\int_{[0,1]\times[0,1]^{\N_0}}x_0(1-x_0)\,Q^{(n)}(\bar{\vt}^{(n)},\d z_0)
				= \vartheta_n(1 -   \vartheta_n) - A_0^n(\CF^{(n+1)}g)(\vartheta_n).
			\end{equation}
			Because $x_0(1-x_0) \geq 0$ for $x \in [0,1]$, we have
			\begin{equation}
				\label{ineqKn}
				\vartheta_n(1-\vartheta_n) \geq A_0^n(\CF^{(n+1)}g)(\vartheta_n) \qquad \forall\,n\in\N.
			\end{equation}
			Since $\lim_{n\to\infty} A_0^{n}=\infty$, it follows that $\lim_{n\to\infty}(\CF^{(n+1)}g)(\vartheta_n)=0$. On the other hand,
			\begin{equation}
				\label{klm}
				(\CF^{(n+1)} g)(\vartheta_n)= \int_{[0,1]\times[0,1]^{\N_0}} g(x_0)\,Q^{(n)}(\bar{\vt}^{(n)},\d z_0)
			\end{equation}
			and, because $g(x)>0$ for $x\in(0,1)$, $Q^{(n)}(\vartheta_n,\d z_0)$ puts all its mass on $x_0=0$ and $x_0=1$ in the limit as $n\to\infty$. Let 
			\be{}
			Q^{(n)}\big(\bar{\vt}^{(n)},\{x_0=0 \text{ or } x_0=1\}\big) 
			= \int_{[0,1]\times[0,1]^{\N_0}} \mathbf{1}_{\{x_0=0 \text{ or } x_0=1\}}(z_0)\,Q^{(n)}(\bar{\vt}^{(n)},\d z_0),
			\ee
			then
			\be{}
			\lim_{n\to\infty}Q^{(n)}\big(\bar{\vt}^{(n)},\{x_0=0 \text{ or } x_1=1\}\big)=1.
			\ee
			The first moment of $x_0$ converges to (recall \eqref{thetantotheta})
			\begin{equation}
				\lim_{n\to\infty}\int_{[0,1]\times[0,1]^{\N_0}} x_0\, Q^{(n)}(\bar{\vt}^{(n)},\d z_0)=\lim_{n\to\infty} \vartheta_n=\theta.
			\end{equation}
			Hence
			\be{}
			\begin{aligned}
				&\lim_{n\to\infty} Q^{(n)}(\bar{\vt}^{(n)},\{x_0=0 \})=1-\theta,\\
				&\lim_{n\to\infty}Q^{(n)}(\bar{\vt}^{(n)},\{x_0=1 \})=\theta,
			\end{aligned}
			\ee
			and
			\begin{equation}
				\label{eqx00}
				\begin{aligned}
					&\lim_{n\to\infty}\int_{[0,1]\times[0,1]^{\N_0}} x_0^2\, Q^{(n)}(\bar{\vt}^{(n)},\d z_0)\\
					&\quad =\lim_{n\to\infty} \int_{[0,1]\times[0,1]^{\N_0}} x_0^2 \Big(\mathbf{1}_{\{1\}}(x_0)+\mathbf{1}_{\{0\}}(x_0)
					+\mathbf{1}_{\{(0,1)\}}(x_0)\Big)\,Q^{(n)}(\bar{\vt}^{(n)},\d z_0)=\theta.
				\end{aligned}
			\end{equation}
			On the other hand,
			\begin{equation}
				\label{uxtofk}
				\lim_{n\to\infty}\int_{[0,1]\times[0,1]^{\N_0}} x_0^2 \; Q^{(n)}(\bar{\vt}^{(n)},\d z_0)
				=  \theta^2  + \lim_{n\to\infty} A_0^n(\CF^{(n+1)}g)(\theta),
			\end{equation}
			and so, combining \eqref{eqx00}--\eqref{uxtofk}, we obtain
			\begin{equation}
				\label{blabla}
				\lim_{n\to\infty} A_0^n(\CF^{(n+1)}g)(\theta) = \theta(1-\theta).
			\end{equation}
			
			\medskip\noindent
			{\bf 2.}
			We know also that
			\begin{eqnarray}
				\label{yxtfg*}
				&&\lim_{n\to\infty}\int_{[0,1]\times[0,1]^{\N_0}} x_0y_{0,0}\, Q^{(n)}(\bar{\vt}^{(n)},\d z_0)\\ \nonumber
				&&\qquad = \lim_{n\to\infty}\;  \vartheta_n^2 + \left(A_0^{n}-\frac{E_0^2}{E_0c_0+e_0+E_0 K_0 e_0}\right) 
				(\CF^{(n+1)} g)(\vartheta_n)\\ \nonumber
				&&\qquad = \theta^2 +\theta(1-\theta) = \theta
			\end{eqnarray} 
			and
			\begin{align}
				\label{grr}
				&\lim_{n\to\infty}\int_{[0,1]\times[0,1]^{\N_0}} xy_0\, Q^{(n)}(\bar{\vt}^{(n)},\d z_0)\\ \nonumber
				&\quad =\lim_{n\to\infty} \int_{[0,1]\times[0,1]^{\N_0}} x_0y_{0,0}\,\Big(\mathbf{1}_{\{1\}}(x_0)
				+\mathbf{1}_{\{0\}}(x_0)+\mathbf{1}_{\{(0,1)\}}(x_0)\Big)\, Q^{(n)}(\bar{\vt}^{(n)},\d z_0)\\
				&\quad =\lim_{n\to\infty}\int_{[0,1]\times[0,1]^{\N_0}} y_{0,0}\,\mathbf{1}_{\{1\}}(x_0)\, Q^{(n)}(\bar{\vt}^{(n)},\d z_0).
				\nonumber
			\end{align}
			Therefore 
			\begin{eqnarray}
				\lim_{n\to\infty}\int_{[0,1]\times[0,1]^{\N_0}} y_{0,0}\,\mathbf{1}_{\{1\}}(x_0)\,Q^{(n)}(\bar{\vt}^{(n)},\d z_0)
				=\lim_{n\to\infty} \vartheta_n=\theta,
			\end{eqnarray}
			and hence
			\begin{eqnarray}
				\lim_{n\to\infty}\int_{[0,1]\times[0,1]^{\N_0}} (1-y_{(0,0)})\, \mathbf{1}_{\{1\}}(x_0)\,Q^{(n)}(\bar{\vt}^{(n)},\d z_0)
				=\theta-\theta=0.
			\end{eqnarray}
			Since $1-y_{(0,0)} \geq 0$, we conclude that if $x_0=1$, then $Q^{(n)}(\bar{\vt}^{(n)},\d z_0)$ puts all its mass on $y_{0,0}=1$ in the limit as $n\to\infty$. Hence 
			\be{}
			\lim_{n\to\infty}Q^{(n)}\big(\bar{\vt}^{(n)},\{(x_0,y_{0,0})=(1,1)\}\big)=\theta.
			\ee
			From Proposition \ref{P.momreliterate2} it also follows that
			\begin{eqnarray}
				\label{tfgy}
				&&\lim_{n\to\infty}\int_{[0,1]\times[0,1]^{\N_0}}(1-x_0)(1-y_{0,0})\,Q^{(n)}(\bar{\vt}^{(n)},\d z_0)\\ \nonumber
				&&\quad = 1-\theta-\theta+ \theta^2 
				+ \lim_{n\to\infty} \left(A_0^{n}-\frac{E_0^2}{E_0c_0+e_0+E_0K_0 e_0}\right) 
				(\CF^{(n+1)} g)(\vartheta_n)\\ \nonumber
				&&\quad = 1-\theta.
			\end{eqnarray}
			On the other hand,
			\begin{equation}
				\label{yxtfg}
				\begin{aligned}
					&\lim_{n\to\infty}\int_{[0,1]\times[0,1]^{\N_0}}(1-x_0)(1-y_{0,0})\, Q^{(n)}(\bar{\vt}^{(n)},\d z_0) \\
					&= \lim_{n\to\infty}\int_{[0,1]\times[0,1]^{\N_0}} (1-x_0)(1-y_{0,0})\,\Big(\mathbf{1}_{\{1\}}(x_0)
					+\mathbf{1}_{\{0\}}(x_0)+\mathbf{1}_{\{(0,1)\}}(x_0)\Big)\\ 
					&\qquad\qquad\qquad \times Q^{(n)}(\bar{\vt}^{(n)},\d z_0)\\
					&=\lim_{n\to\infty}\int_{[0,1]\times[0,1]^{\N_0}}(1-y_{0,0})\,
					\mathbf{1}_{\{0\}}(x_0)\,Q^{(n)}(\bar{\vt}^{(n)},\d z_0)\\
					&=1-\theta.
				\end{aligned}
			\end{equation}
			Since $y\in[0,1]$ and
			\begin{eqnarray}
				\lim_{n\to\infty}\int_{[0,1]\times[0,1]^{\N_0}} \mathbf{1}_{\{0\}}(x_0)\,Q^{(n)}(\bar{\vt}^{(n)},\d z_0)=1-\theta,
			\end{eqnarray}
			it follows that 
			\begin{eqnarray}
				\lim_{n\to\infty}\int_{[0,1]\times[0,1]^{\N_0}} y_{0,0} \mathbf{1}_{\{0\}}(x_0)\,Q^{(n)}(\bar{\vt}^{(n)},\d z_0)=0.
			\end{eqnarray}
			This implies that if $x_0=0$, then $Q^{(n)}(\bar{\vt}^{(n)},\d z_0)$ puts all its mass on $y_{0,0}=0$ in the limit as $n\to\infty$. Hence
			\be{}
			\begin{aligned}
				&\lim_{n\to\infty}Q^{(n)}\big(\bar{\vt}^{(n)},\{(x_0,y_{0,0})=(0,0)\}\big)=1-\theta,\\
				&\lim_{n\to\infty}Q^{(n)}\big(\bar{\vt}^{(n)},\{(x_0,y_{0,0})=(1,1)\}\big)=\theta.
			\end{aligned}
			\ee
			Now assume \eqref{ifc}. Then \eqref{eqx00} still holds. On the other hand, also \eqref{uxtofk} still holds by Proposition~\ref{P.momreliterate}. Therefore we obtain \eqref{blabla}. On the other hand, by \eqref{ifc}
			\begin{equation}
				\lim_{n\to \infty}\CF^{(n+1)}g=\lim_{n\to \infty}\int_{[0,1]\times[0,1]^{\N_0}}g(x_0)Q^{(n)}(\bar{\vt}^{(n)},\d z_0)=0.
			\end{equation}
			Hence \eqref{ifa} holds.
		\end{proof}
		
		A similar clustering property holds for the kernels associated with blocks. 
		
		\begin{proposition}{\bf [Clustering: blocks]} 
			\label{limQ2} 
			Assume
			\begin{equation}
				\lim_{n\to\infty}\vt_n=\theta.
			\end{equation}
			Then
			\be{}
			\begin{aligned}
				\label{ifc2}
				&\lim_{n\to\infty}Q_m^{(n)}\big(\bar{\vt}^{(n)},\{(x_m,y_{(m,m)})=(0,0)\}\big)=1-\theta,\\
				&\lim_{n\to\infty}Q_m^{(n)}\big(\bar{\vt}^{(n)},\{(x_m,y_{(m,m)})=(1,1)\}\big)=\theta,
			\end{aligned}
			\ee
			if and only if 
			\begin{equation}
				\label{ifa2}
				\lim_{n\to\infty}A_m^n=\infty.
			\end{equation}
			Consequently,
			\be{}
			\lim_{n\to\infty}Q^{(n)}_m\big(\bar{\vt}^{(n)},\{(x_m,y_{m,m})\}\notin\{(0,0),(1,1)\}\big)=0.
			\ee
		\end{proposition}
		
		\begin{proof}
			We can proceed exactly as in the proof of Proposition~\ref{limQ}.
		\end{proof}
		
		Finally, we can prove Theorem~\ref{T.scalren}.
		
		\begin{proof}
			Note that \eqref{ifa} implies \eqref{ifa2}. Recall that single colonies of deep seed-banks that have already interacted and reached their quasi-equilibrium equal the block average of the level on which they interact (see Theorem~\ref{T.multiscalehier}). It follows that, for $m\in\N_0$,
			\be{}
			\begin{aligned}
				&\lim_{n\to\infty}Q^{(n)}\big(\bar{\vt}^{(n)},{(x_0,y_{0,m})=(1,1)}\big)=\theta,\\
				&\lim_{n\to\infty} Q^{(n)}\big(\bar{\vt}^{(n)},{(x_0,y_{0,m})=(0,0)}\big)=1-\theta.
			\end{aligned}
			\ee 
			Therefore, for $N\in\N_0$,
			\begin{equation}
				\begin{aligned}
					&\lim_{n\to\infty} Q^{(n)}\left(\bar{\vt}^{(n)},\bigcap_{m=0}^N\big\{(x_0,y_{0,m})=(1,1)
					\text{ or }(x_0,y_{0,m})=(0,0)\big\}\right)\\
					&\qquad =1-\lim_{n\to\infty}Q^{(n)}\left(\bar{\vt}^{(n)},
					\bigcup_{m=0}^N\big\{(x_0,y_{0,m})\in[0,1]^2\backslash\{(0,0),(1,1)\}\big\}\right)\\
					&\qquad \geq 1-\lim_{n\to\infty}\sum_{m=0}^NQ^{(n)}\Big(\bar{\vt}^{(n)},\big\{(x_0,y_{0,m})
					\in[0,1]^2\backslash\{(0,0),(1,1)\}\big\}\Big) =1-0=1.
				\end{aligned}
			\end{equation}
			Note that 
			\begin{equation}
				\begin{aligned}
					&\lim_{n\to\infty}Q^{(n)}\left(\bar{\vt}^{(n)},
					\bigcap_{m=0}^N\big\{(x_0,y_{0,m})=(1,1)\text{ or }(x_0,y_{0,m})=(0,0)\big\}\right)\\
					&\qquad =\lim_{n\to\infty}Q^{(n)}\Big(\bar{\vt}^{(n)},\big\{(x_0,(y_{0,m})_{0 \leq m \leq N})
					=(0,0^{N+1})\\
					&\qquad\qquad\qquad\qquad\qquad
					\text{ or }(x_0,(y_{0,m})_{0 \leq m \leq N})=(1,1^{N+1})\big\}\Big) =1.
				\end{aligned}
			\end{equation}
			On the other hand,
			\begin{equation}
				\begin{aligned}
					&\lim_{n\to\infty}Q^{(n)}\Big(\bar{\vt}^{(n)},\{(x_0,(y_{0,m})_{0 \leq m \leq N})=(1,1^{N+1})\}\Big)\\
					&\qquad \leq \lim_{n\to\infty}Q^{(n)}\big(\bar{\vt}^{(n)},\{(x_0,y_{0,0})=(1,1)\}\big)=\theta
				\end{aligned} 
			\end{equation}
			and
			\begin{equation}
				\begin{aligned}
					&\lim_{n\to\infty}Q^{(n)}\Big(\bar{\vt}^{(n)},\{(x_0,(y_{0,m})_{0 \leq m \leq N})
					=(0,0^{N+1})\}\Big)\\
					&\qquad \leq \lim_{n\to\infty} Q^{(n)}\big(\bar{\vt}^{(n)},\{(x_0,y_{0,0})=(0,0)\}\big)=1-\theta.
				\end{aligned}
			\end{equation} 
			Hence we conclude that
			\be{}\label{901}
			\lim_{n\to\infty}Q^{(n)}\Big(\bar{\vt}^{(n)},\{(x_0,(y_{0,m})_{0 \leq m \leq N})=(1,1^{N+1})\}\Big)
			= \theta
			\ee 
			and
			\be{}\label{902}
			\lim_{n\to\infty}Q^{(n)}\Big(\bar{\vt}^{(n)},\{(x_0,(y_{0,m})_{0 \leq m \leq N})=(0,0^{N+1})\}\Big)
			= 1-\theta.
			\ee
			We can do the same for all finite-dimensional distributions. Since $[0,1]\times[0,1]^{\N_0}$ is compact, the process $z_0=(x_0,(y_{0,m})_{m\in\N_0})$ is tight. Therefore, by  \eqref{901}--\eqref{902} we find for every converging subsequence
			\begin{equation}
				\lim_{k\to\infty}Q^{(n_k)}\Big(\bar{\vt}^{(n_k)},\cdot\Big)=(1-\theta)\,\delta_{(0,0^{\N_0})} + \theta\,\delta_{(1,1^{\N_0})}.	
			\end{equation}
			We conclude that 
			\begin{equation}
				\lim_{n\to\infty} Q^{(n)}(\bar{\vt}^{(n)},\d z_0) 
				= (1-\theta)\,\delta_{(0,0^{\N_0})} + \theta\,\delta_{(1,1^{\N_0})},
			\end{equation}
			which is the claim in \eqref{limKnnu}.
		\end{proof}
		
		\subsection{Dichotomy finite versus infinite seed-bank}
		\label{ss.orbit}
		
		In this section we prove Theorem~\ref{T.dichotomy}.
		
		\begin{proof} 
			We investigate for what choices of the sequences $\underline{c},\underline{K},\underline{e}$ defined in \eqref{738} and \eqref{defKem} we meet the \emph{clustering criterion} $\lim_{n\to\infty} A_n \to\infty$ in \eqref{cluscond}. Recall from \eqref{Ekdef} and \eqref{modefA} that
			\begin{equation} 
				A_n = \frac{1}{2}\sum_{k=0}^{n-1}\frac{E_k}{c_k}\frac{(E_kc_k+e_k)}{(E_kc_k+e_k)+E_kK_k e_k},
				\qquad 
				E_k = \frac{1}{1+\sum_{m=0}^{k-1} K_m}.
			\end{equation}
			We distinguish between three regimes as $k\to\infty$: 
			\begin{enumerate}
				\item $E_kc_k+e_k \gg E_kK_ke_k$.
				\item $E_kc_k+e_k \asymp E_kK_ke_k$.
				\item $E_kc_k+e_k \ll E_kK_ke_k$.  
			\end{enumerate}
			These regimes correspond to the following scaling for $A_n$ as $n\to\infty$:
			\begin{enumerate}
				\item $A_n \sim \frac{1}{2}\sum_{k=0}^{n-1} \frac{E_k}{c_k}$.
				\item $A_n \asymp \sum_{k=0}^{n-1} \frac{E_k}{c_k}$.
				\item $A_n \sim \frac{1}{2}\sum_{k=0}^{n-1} \frac{E_kc_k+e_k}{c_kK_ke_k}$.  
			\end{enumerate}
			Recall from that \eqref{rhodef} that
			\begin{equation}
				\rho = \sum_{m\in\N_0} K_m.
			\end{equation}
			Different behaviour shows up for finite seed-bank ($\rho<\infty$) and infinite seed-bank ($\rho=\infty$). 
			
			\paragraph{(I) $\rho<\infty$.}
			
			Note that $k \mapsto E_k$ is non-increasing and converges to $1/(1+\rho)>0$. Since $E_k \leq 1$, we have $\lim_{k\to\infty} E_kK_k=0$ and hence we are in regime 1. Therefore
			\be{}
			A_n \sim \frac{1}{2(1+\rho)} \sum_{k=0}^{n-1}\frac{1}{c_k}, \qquad n\to\infty,
			\ee
			and clustering occurs if and only if $\sum_{k\in\N_0} \frac{1}{c_k}=\infty$, which is the same criterion as for the system without seed-bank.  
			
			\paragraph{(II) $\rho=\infty$.}
			
			We focus on the settings in \eqref{regvar} and \eqref{pureexp}, which fall in regimes 1 and 2. 
			
			\medskip\noindent
			{\bf Asymptotically polynomial.}
			Suppose that
			\be{} 
			K_k \sim Ak^{-\alpha}, \quad k \to \infty,\, A \in (0,\infty),\, \alpha \in (-\infty,1).
			\ee
			Then
			\begin{equation}
				\label{scalrel1}
				E_k \sim \frac{1-\alpha}{A}k^{-(1-\alpha)}, \qquad
				E_kK_k \sim \frac{1-\alpha}{A}k^{-1}, \qquad k \to \infty.
			\end{equation}
			Hence we are in regime 1. Suppose that
			\be{}
			c_k\sim Fk^{-\phi}, \quad k\to\infty,\, F \in (0,\infty),\, \phi \in \R.
			\ee 
			Then
			\begin{equation}
				\label{asymp2}
				A_n \sim \frac{1-\alpha}{2AF} \sum_{k=1}^{n-1} k^{-1+\alpha+\phi},
				\qquad n \to \infty,
			\end{equation}
			and clustering occurs if and only if $-\phi \leq \alpha < 1$. In this case
			\begin{equation}
				\begin{aligned}
					&-\phi<\alpha\colon \quad A_n \sim \frac{1-\alpha}{2AF(\alpha+\phi)}\,n^{\alpha+\phi},\\
					&-\phi=\alpha\colon \quad A_n \sim \frac{1-\alpha}{2AF}\,\log n.
				\end{aligned}
			\end{equation}
			
			The case $\alpha=1$ can be included. Then \eqref{scalrel1} becomes
			\begin{equation}
				\label{scalrel2}
				E_k \sim \frac{1}{A \log k}, \qquad
				E_kK_k\sim \frac{1}{k\log k}, \qquad k \to \infty,
			\end{equation}
			so that we are again in regime 1. Now \eqref{asymp2} becomes
			\begin{equation}
				A_n \sim \tfrac{1}{2AF} \sum_{k=1}^{n-1} \frac{k^\phi}{\log k}, \qquad n\to\infty,
			\end{equation}
			and clustering occurs if and only if $-\phi \leq 1$. In this case
			\begin{equation}
				\begin{aligned}
					&-\phi<1\colon \quad A_n \sim \frac{1}{2AF(1+\phi)}\,\frac{n^{1+\phi}}{\log n},\\
					&-\phi=1\colon \quad A_n \sim \frac{1}{2AF}\,\log\log n.
				\end{aligned}
			\end{equation}	 
			
			\medskip\noindent
			{\bf Pure exponential.}
			Suppose that
			\be{}
			K_k=K^{k}, \quad k \in \N_0,\,K \in (1,\infty).
			\ee
			Then
			\begin{equation}
				\label{Ekrels}
				E_k = \frac{1}{1+\sum_{m=0}^{k-1} K^m} = \frac{1}{1+\frac{K^k-1}{K-1}} = \frac{K-1}{K^k+K-2}.
			\end{equation}
			Suppose that
			\begin{equation}
				e_k= e^k, \quad c_k= c^k, \quad k \in \N_0,\, e,c \in (0,\infty).
			\end{equation}
			Then
			\be{}
			E_kc_k+e_k = \frac{K-1}{K^k+K-2}c^k+e^k,
			\qquad E_kK_ke_k = \frac{K-1}{K^k+K-2}K^ke^k,
			\ee
			and so
			\begin{equation}
				\label{scalrel1*}
				E_kc_k+e_k \sim (K-1) \left(\frac{c}{K}\right)^k+e^k,
				\qquad E_kK_ke_k \sim (K-1) e^k, \qquad k \to \infty. 
			\end{equation}
			For $c \leq Ke$ we are in regime 2, and hence
			\begin{equation}
				A_n \sim \frac12 \sum_{k=0}^{n-1} \frac{K-1}{(Kc)^k}\,
				\frac{(K-1)(\frac{c}{K})^k+e^k}{(K-1)(\frac{c}{K})^k+ Ke^k},
				\qquad n\to\infty,
			\end{equation}
			which simplifies to 
			\begin{equation}
				\begin{aligned}
					&c<Ke\colon \quad A_n \sim \frac{1}{2K} \sum_{k=0}^{n-1} \frac{K-1}{(Kc)^k},\\
					&c=Ke\colon \quad A_n \sim \frac{K-1}{2(2K-1)} \sum_{k=0}^{n-1} \frac{K-1}{(Kc)^k}.
				\end{aligned}
			\end{equation}
			Clustering occurs if and only if $Kc \leq 1$. In this case
			\begin{equation}
				\label{sumasymp}
				\begin{aligned}
					&Kc<1\colon \quad  \sum_{k=0}^{n-1} \frac{1}{(Kc)^k}\sim \frac{1}{1-Kc}\,(Kc)^{-(n-1)},\\
					&Kc=1\colon \quad  \sum_{k=0}^{n-1} \frac{1}{(Kc)^k} \sim n.
				\end{aligned}
			\end{equation}
			For $c > Ke$, on the other hand, we are in regime 1, and hence
			\be{}
			A_n \sim \frac12 \sum_{k=0}^{n-1} \frac{K-1}{(Kc)^k}, \qquad n\to\infty,
			\ee
			for which we can again use \eqref{sumasymp}.
			
			The case $K=1$ can be included. Then $E_k=1/k$ and \eqref{scalrel1*} becomes
			\begin{equation}
				E_kc_k+e_k \sim \frac{1}{k}c^k+e^k,
				\qquad E_kK_ke_k \sim \frac{1}{k} e^k, \qquad k \to \infty, 
			\end{equation}
			we are again in regime 1. Hence
			\begin{equation}
				\label{asymp3}
				A_n \sim \frac12 \sum_{k=0}^{n-1} \frac{1}{kc^k}, \qquad n \to \infty,
			\end{equation}
			and clustering occurs if and only if $c \leq 1$. In that case
			\begin{equation}
				\begin{aligned}
					&c<1\colon \quad A_n \sim \frac{1}{2(1-c)}\,\frac{1}{n}\,c^{-(n-1)},\\
					&c=1\colon \quad A_n \sim \frac{1}{2}\,\log n.
				\end{aligned}
			\end{equation}
		\end{proof}
		
		In the above computations, only regimes 1 and 2 arise. Regime 3 arises, for instance, when
		\begin{equation}
			\lim_{k\to\infty} \frac{1}{k} \log K_k = \infty, \qquad \lim_{k\to\infty} K_ke_k/c_k = \infty.    
		\end{equation}
		Indeed, the first condition implies that $E_k \sim 1/K_{k-1}$ and $E_k K_k \gg 1$, while the second implies that $E_kK_ke_k \gg  E_kc_k$. There are two subcases: 
		\begin{equation}
			\begin{aligned}
				&K_{k-1} e_k \ll c_k\colon \quad A_n \sim \tfrac12 \sum_{k=0}^{n-1} \frac{1}{K_kK_{k-1}e_k},\\
				&K_{k-1} e_k \gg c_k\colon \quad A_n \sim \tfrac12 \sum_{k=0}^{n-1} \frac{1}{K_kc_k}.
			\end{aligned}
		\end{equation}
		By picking, for instance, $e_k=1/K_kK_{k-1}$, we find that $A_n \sim \tfrac12 n$ in the first subcase and $A_n \gg n$ in the second subcase.  By picking, alternatively, $c_k = 1/K_k$, we find that $A_n \gg n$ in the first subcase and $A_n \sim \tfrac12 n$ in the second subcase.    
		
		\appendix
		
		\section{Computation of scaling coefficients}
		\label{app.comp}
		
		In Appendices~\ref{app.regvar}--\ref{app.pureexp} we spell out a technical computation for the tail of the wake-up time defined in \eqref{sigtau}--\eqref{deftau} in the two parameter regimes given by \eqref{regvar}--\eqref{pureexp}.  In Appendix~\ref{app.slowvar} we carry out a computation that is needed in Section~\ref{ss.intcrit}.

		\subsection{Regularly varying coefficients}
		\label{app.regvar}
		
		In \eqref{sigtau}, note that for large $t$ in the sum over $m$ only small values of $e_m/N^m$ contribute, which means large $m$. Hence, by the Euler-MacLaurin approximation formula, we have
		\begin{equation}
			P(\tau>t) = \frac{1}{\chi} \sum_{m\in\N_0} K_m \frac{e_m}{N^m}\,\e^{-(e_m/N^m) t}
			\sim \frac{1}{\chi} \int_c^\infty \d m\,K_m \frac{e_m}{N^m}\,\e^{-(e_m/N^m) t},
		\end{equation}
		where $c$ is a constant that identifies from which value of $m$ onward terms contribute significantly. Make the change of variable $s=\frac{e_m}{N^m}$. Since $e_m\sim Bm^{-\beta}$ and $K_m\sim Am^{-\alpha}$ as $m\to\infty$, we have
		\begin{equation}
			s\sim Bm^{-\beta}N^{-m}
		\end{equation}
		and hence
		\begin{equation}
			\begin{aligned}
				\log s &\sim \log B - \beta\log m - m\log N,\\
				\log \frac{1}{s} &= m\log N\left(-\frac{B}{m\log N}+\frac{\beta\log m}{m\log N}+1\right) = [1+o(1)]\,m\log N,
			\end{aligned}
		\end{equation}
		which gives
		\begin{equation}
			m = [1+o(1)]\,\frac{\log(\frac{1}{s})}{\log N}.
		\end{equation}
		Thus,
		\begin{equation}
			\frac{1}{s}\frac{\d s}{\d m} = -\log N-\frac{\beta}{m} = -[1+o(1)]\,\log N,
		\end{equation}
		which implies
		\begin{equation}
			\frac{\d s}{\d m} = - [1+o(1)]\,s\log N,
		\end{equation}
		so that $s(m)$ is asymptotically decreasing in $m$, and
		\begin{equation}
			\frac{\d m}{\d s} = -[1+o(1)]\,\left(s\log N\right)^{-1}.
		\end{equation}
		Note that if $c\leq m<\infty$, then asymptotically $0<m^{-\beta}N^{-m}<c^{-\beta}N^{-c}=C_2$. Doing the substitution, we get
		\begin{equation}
			\begin{aligned}
				&\P(\tau>t) \sim \frac{1}{\chi} \int_0^{C_2} \d s\,K_m s \left(s\log N\right)^{-1} \,\e^{-s t}
				\sim \frac{1}{\chi} \int_0^{C_2} \d s\,Am^{-\alpha}  \left(\log N\right)^{-1} \,\e^{-s t}\\
				&\sim \frac{1}{\chi} \int_0^{C_2} \d s\,A \left(\frac{\log(\frac{1}{s})}{\log N}\right)^{-\alpha}
				\left(\log N\right)^{-1} \,\e^{-s t}
				\sim \frac{A}{\chi} \left(\frac{1}{\log N}\right)^{-\alpha+1} 
				\int_0^{C_2}  \d s\, \log\left(\tfrac{1}{s}\right)^{-\alpha} \,\e^{-s t}.
			\end{aligned}
		\end{equation}
		Next, put $st=u$, so $s=\frac{u}{t}$ and $\frac{\d s}{\d u}=\frac{1}{t}$ and $0<u<tC_2$. Then
		\begin{equation}
			\P(\tau>t) \sim  \frac{A}{\chi} \left(\frac{1}{\log N}\right)^{-\alpha+1} 
			\frac{1}{t}\int_0^{C_2 t} \d u\, \log\left(\tfrac{t}{u}\right)^{-\alpha} \,\e^{-u}.
		\end{equation}
		We will show that 
		\begin{equation}
			\frac{A}{\chi} \left(\frac{1}{\log N}\right)^{-\alpha+1} 
			\frac{1}{t}\int_0^{C_2 t} \d u\, \log\left(\tfrac{t}{u}\right)^{-\alpha} \,\e^{-u}
			\asymp \frac{A}{\chi} \left(\frac{1}{\log N}\right)^{-\alpha+1} 
			\frac{1}{t} \int_0^{C_2 t} \d u\, \log t^{-\alpha} \,\e^{-u}.
		\end{equation}
		For $\alpha=0$ this claim is immediate. For $\alpha\in(-\infty,0)$, note that $\log\left(\frac{t}{u}\right)^{-\alpha}$ is a decreasing function on $(0,C_2t)$. Therefore we can reason as follows: 
		\begin{equation}
			\begin{aligned}
				\int_0^{C_2 t} \d u\,\log\left(\tfrac{t}{u}\right)^{-\alpha}\e^{-u}
				&=\int_0^{1} \d u\,\log\left(\tfrac{t}{u}\right)^{-\alpha}\e^{-u}
				+\int_1^{C_2 t} \d u\,\log\left(\tfrac{t}{u}\right)^{-\alpha}\e^{-u}\\
				&\leq \int_0^{1} \d u\,\log\left(\tfrac{t}{u}\right)^{-\alpha} + \int_1^{C_2 t} \d u\,\log{t}^{-\alpha}\e^{-u}\\
				&\leq 2^{-\alpha}\int_0^{\frac{1}{t}} \d u\,\log\left(\tfrac{1}{u}\right)^{-\alpha}
				+2^{-\alpha}\int_{\frac{1}{t}}^{1} \d u\,\log t^{-\alpha} + \log t^{-\alpha}\left[1-\e^{-1}\right]\\
				&\leq 2^{-\alpha}\Gamma(-\alpha+1)+2^{-\alpha}\log t^{-\alpha}\left[1-\tfrac{1}{t}\right]
				+\log t^{-\alpha}\left[1-\e^{-1}\right]\\
				&=\log t^{-\alpha}\left[2^{-\alpha}\tfrac{\Gamma(-\alpha+1)}{\log t^{-\alpha}}
				+2^{-\alpha}\left[1-\tfrac{1}{t}\right]+\left[1-\e^{-1}\right]\right]
				\asymp\log t^{-\alpha}.
			\end{aligned}
		\end{equation}
		For the lower bound, note that 
		\begin{equation}
			\begin{aligned}
				\int_0^{C_2 t} \d u\,\log\left(\tfrac{t}{u}\right)^{-\alpha}\e^{-u}
				&\geq \log\left(t\right)^{-\alpha}\int_0^{1} \d u\,\e^{-u}
				+\log(\tfrac{1}{C_2})^{-\alpha}\int_1^{C_2 t} \d u\,\e^{-u}\\
				&=\log t^{-\alpha}\left[1-e^{-1}+\frac{\log(\tfrac{1}{C_2})^{-\alpha}}{\log t^{-\alpha}}\right]
				\asymp\log t^{-\alpha}.
			\end{aligned}
		\end{equation} 
		For $\alpha\in(0,1]$, note that the function $\log\left(\frac{t}{u}\right)^{-\alpha}$ is increasing in $u$. For the lower bound estimate
		\begin{equation}
			\begin{aligned}
				\int_0^{C_2 t} \d u\,\log\left(\tfrac{t}{u}\right)^{-\alpha}\e^{-u}
				&\geq \lim_{u\to 0}\log\left(\tfrac{t}{u}\right)^{-\alpha}[1-\e^{-1}]+\log t^{-\alpha}[\e^{-1}-e^{-C_2 t}]\\
				&=\log t^{-\alpha}\left[0+\e^{-1}-\e^{-C_2 t}\right]
				\asymp \log t^{-\alpha}.
			\end{aligned}
		\end{equation}
		For the upper bound estimate
		\begin{equation}
			\begin{aligned}
				&\int_0^{C_2 t} \d u\,\log\left(\tfrac{t}{u}\right)^{-\alpha}\e^{-u}\\
				&\leq \log t^{-\alpha}[1-\e^{-1}]+\log\left(\frac{t}{\sqrt{C_2 t}}\right)^{-\alpha}\int_1^{\sqrt{C_2 t}} \d u\,e^{-u}
				+\log\left(\frac{1}{C_2}\right)^{-\alpha}\int_{\sqrt{C_2t}}^{C_2t} \d u\,\e^{-u}\\
				&=\log t^{-\alpha}[1-\e^{-1}]+(\tfrac{1}{2})^{-\alpha}\log\left(\frac{t}{C_2}\right)^{-\alpha}
				\left[\e^{-1}-\e^{-\sqrt{C_2 t}}\right]
				+\log\left(\frac{1}{C_2}\right)^{-\alpha}\left[\e^{-\sqrt{C_2 t}}-\e^{-C_2 t}\right]\\
				&=\log t^{-\alpha}\Bigg[1-\e^{-1}+(\tfrac{1}{2})^{-\alpha}
				\left(\frac{\log t-\log C_2}{\log t}\right)^{-\alpha}\left[\e^{-1}-\e^{-\sqrt{C_2 t}}\,\right]\\
				&\qquad \qquad\qquad+ \log\left(\frac{1}{C_2}\right)^{-\alpha}\frac{\left[\e^{-\sqrt{C_2 t}}
					-\e^{-C_2 t}\right]}{\log t ^{-\alpha}} \Bigg]
				\asymp \log t^{-\alpha}.
			\end{aligned}
		\end{equation}
		
		\subsection{Pure exponential coefficients}
		\label{app.pureexp}
		
		In order to satisfy condition in \eqref{740alt}, we must assume that $Ke<N$. Since $K\geq1$ for $\rho=\infty$, we also have $e<N$. We again use that for large $t$ only large $m$ contribute to the sum. Hence, again by the Euler-MacLaurin approximation formula, we have
		\begin{equation}
			P(\tau>t) = \frac{1}{\chi} \sum_{m\in\N_0} K_m \frac{e_m}{N^m}\,\e^{-(e_m/N^m) t}
			\sim \int_M^\infty \d m\,K_m \frac{e_m}{N^m}\,\e^{-(e_m/N^m) t}.\\
		\end{equation} 
		Again we put $s=\frac{e^m}{N^m}$. Hence 
		\begin{equation}
			\log s= m\log\left(\frac{e}{N}\right), \qquad m=\frac{\log s}{\log \frac{e}{N}}, \qquad
			\frac{\d m}{\d s} = \frac{1}{s \log \frac{e}{N}},
		\end{equation}
		and  
		\begin{equation}
			K_m\sim K^m \sim \e^{\log s\frac{\log K}{\log \frac{e}{N}}}\sim s^{\frac{\log K}{\log \frac{e}{N}}}.
		\end{equation}
		Since $s(m)$ is decreasing in $m$, putting $C=(\frac{e}{N})^{M}$ we obtain
		\begin{equation}
			\P(\tau>t)
			\sim \int_0^C \d s\,K_m  \frac{s}{s \log \frac{e}{N}} \,\e^{-s t} 
			\sim \frac{1}{ \log \frac{e}{N}} \int_0^C \d s\,s^{\frac{\log K}{\log \frac{e}{N}}}  \,\e^{-s t}.
		\end{equation}
		Substitute $u=st$, i.e., $\frac{u}{t}=s$, to get
		\begin{equation}
			\begin{aligned}
				&\P(\tau>t) \sim \frac{1}{\log \frac{e}{N}}\, t^{-1-{\frac{\log K}{\log \frac{e}{N}}}}
				\int_0^{Ct} \d u\,u^{\frac{\log K}{\log \frac{e}{N}}} \,\e^{-u}\\
				&\sim \frac{1}{\log \frac{e}{N}}\, t^{\frac{-\log(\frac{e}{N})-\log K}{\log \frac{e}{N}}} 
				\int_0^{Ct} \d u\,u^{\frac{\log K}{\log \frac{e}{N}}} \,\e^{-u}
				\sim \frac{1}{\log \frac{e}{N}}\, t^{-\frac{\log(\frac{N}{Ke})}{\log \frac{N}{e}}}  
				\int_0^{Ct} \d u\,u^{\frac{\log K}{\log \frac{e}{N}}} \,\e^{-u}.
			\end{aligned}
		\end{equation}
		The last integral converges because $\frac{\log K}{\log (\frac{e}{N})}>-1$, and
		\begin{equation}
			\int_0^{Ct} \d u\,u^{\frac{\log K}{\log \frac{e}{N}}} \,\e^{-u}
			\leq \int_0^{\infty} \d u\,u^{\frac{\log K}{\log \frac{e}{N}}} \,\e^{-u}
			=\Gamma\left(\frac{\log K}{\log (\frac{e}{N})}+1\right).
		\end{equation}
		
		\subsection{Slowly varying functions}
		\label{app.slowvar}
		
		Return to Section~\ref{ss.intcrit}. Note that $t(s)=\varphi(s)^{-1} s^\gamma$. Since this is the total time two lineages are active up to time $s$, $t(s)$ must be smaller than $s$. By \eqref{hatphirepr}, we have
		\begin{equation}
			\begin{aligned}
				\frac{{\varphi}(t)}{{\varphi}(s)}=\exp\left[-\int_{t(s)}^{s}\frac{\d u}{u}\psi(u)\right].
			\end{aligned}
		\end{equation}
		Since we are interested in $s\to\infty$, we may assume that $s\gg1$ and $t(s)>1$, and estimate
		\begin{equation}
			\begin{aligned}
				\frac{{\varphi}(t)}{{\varphi}(s)}
				&\leq \exp\left[\int_{t(s)}^s\frac{\d u}{u}\frac{C}{\log u}\right]
				= \exp\big[C(\log\log s-\log\log t(s))\big]\\
				&= \exp\left[C\log\left(\frac{\log s}{\log \left(\varphi(s)^{-1}s^\gamma\right)}\right)\right]
				= \exp\left[-C\log\left(\frac{\gamma\log s-\log \varphi(s)}{\log s}\right)\right].
			\end{aligned}
		\end{equation}
		A similar lower bound holds with the sign reversed. Using that $\lim_{s\to\infty}\frac{\log \varphi(s)}{\log s}=0$, we get
		\begin{equation}
			\begin{aligned}
				\gamma^C \leq \liminf_{s\to\infty}\frac{{\varphi}(t)}{{\varphi}(s)} 
				\leq \limsup_{s\to\infty}\frac{{\varphi}(t)}{{\varphi}(s)} \leq \gamma^{-C}. 
			\end{aligned}
		\end{equation}
		Both bounds above are positive, so indeed $\frac{{\varphi}(t)}{{\varphi}(s)}\asymp 1$. 
		
		\section{Meyer-Zheng topology}
		\label{apb}
		
		\subsection{Basic facts about the Meyer-Zheng topology}
		
		In the Meyer-Zheng topology we assign to each real-valued Borel measurable function $(w(t))_{t\geq 0}$ a probability law on $[0,\infty]\times\bar{\R}$ that is called the pseudopath $\psi_w$. Note that the Borel-$\sigma$ algebra on $[0,\infty]\times\bar{\R}$ is generated by sets of the form $[a,b]\times B$ for $B\in\mathcal{B}$ and $0<a<b$. For $A=[a,b]\times B$, set
		\begin{equation}
			\psi_w (A)=\int 1_A(t,w(t)) \e^{-t}\d t=\int_a^b 1_B(w(t))e^{-t}\d t,
		\end{equation} 
		i.e., $\psi_w$ is the image measure of the mapping $t\to(t,w(t))$ under the measure $\lambda (\d t)=\e^{-t}\d t$. The set of all pseudopaths is denoted by $\Psi$. Note that a pseudopath corresponding to $(w(t))_{t>0}$ is simply its occupation measure. The following important facts are stated in \cite{MZ84}:
		\begin{itemize}
			\item 
			If two paths $w_1$ and $w_2$ are the same Lebesgue a.e., then $\psi_{w_1}=\psi_{w_2}$. 
			\item 
			Denote by $\mathbf{D}$ the space of c\`adl\`ag paths on $[0,\infty]\times\R$. The mapping $\psi\colon\, \mathbf{D}\to\Psi,\, w\mapsto\psi_w$ is one-to-one on $\mathbf{D}$ and hence gives an embedding of $\mathbf{D}$ into the compact space $\bar{\CP}$, the space of probability measures on $[0,\infty]\times\bar{\R}$.
			\item 
			Note if $f$ is a function on $[0,\infty]\times\R$ and $w\in \mathbf{D}$, then
			\begin{equation}
				\psi_w(f)=\int_0^\infty f(t,w(t))\,\e^{-t}\d t.
			\end{equation} 
			Therefore we say that the sequence of pseudopaths induced by $(w_n)\subset\mathbf{D}$ converges to a pseudopath $w$ if, for all continuous bounded function $f(t,w(t))$ on $[0,\infty]\times\bar{\R}$, 
			\begin{equation}
				\lim_{n\to\infty}\int_0^\infty f(t, w_n(t))\,\e^{-t}\d t=\int_0^\infty f(t,w(t))\,\e^{-t}\d t.
			\end{equation}  
			Since a pseudopath is a measure, convergence of pseudopaths is convergence of measures.
			\item 
			$\mathbf{D}$ endowed with the pseudopath topology is \emph{not} a Polish space. $\Psi$ endowed with the pseudopath topology is a Polish space.
			\item 
			According to \cite{MZ84}[Lemma 1], the pseudopath topology on $\Psi$ is convergence in Lebesgue measure on $\mathbf{D}$. 
		\end{itemize} 
		
		\subsection{Pseudopaths of stochastic processes on a general metric separable space}
		
		In \cite{K91} the results of \cite{MZ84} on state space $\R$ are generalised to a general metric separable space $E$. Let $(Z(t))_{t>0}$ be a stochastic process with state space $E$. Then we assign a random pseudopath to $(Z(t))$ as follows: for $\omega\in\Omega$ and $A=[a,b]\times B$, $0\leq a<b$ and $B\in\CB(E)$, 
		\begin{equation}
			\psi_{(Z(t,\omega))_{t\geq 0}}(A)=\int_a^b1_{B}(Z(t,\omega))\,\e^{-t}\d t.
		\end{equation} 
		Hence $\psi_{(Z(t))_{t\geq 0}}$ is a random variable with state space $\Psi$, i.e., $\psi_{(Z(t))_{t\geq 0}}\in\CM(\Psi)$, the set of probability measures on pseudopaths. Note that
		\begin{equation}
			\E\left[\psi_{(Z(t))_{t\geq 0}}f\right]=\E\left[\int_0^\infty f(t,Z(t,\omega))\,\e^{-t}\d t\right]
			= \E\left[\int_0^\infty f(t,Z(t))\,\e^{-t}\d t\right].
		\end{equation}
		
		\paragraph{Weak convergence in the Meyer-Zheng topology.} 
		
		Let $(Z_n(t))_{t\geq 0}$ and $(Z(t))_{t\geq 0}$ be stochastic processes with state-space $E$. We say that
		\begin{equation}
			\CL\left[(Z_n(t))_{t\geq 0}\right]=\CL\left[(Z(t))_{t\geq 0}\right] \text{ in the Meyer-Zheng topology}
		\end{equation}
		if, for all $f\in \CC_b(\Psi)$,
		\begin{equation}
			\lim_{n\to\infty}\E[f(\psi_{(Z^n(t))_{t\geq 0}})]=\E[f(\psi_{(Z(t))_{t\geq 0}})].
		\end{equation}
		
		Let $\CC_m([0,\infty)\times E)\subset\CC_b([0,\infty)\times E)$ be the set of functions of the form
		\begin{equation}
			\begin{aligned}
				\CC_m([0,\infty)\times E)
				&= \Big\{F\in \CC_b([0,\infty)\times E): F(t,x(t))
				=\prod_{i=1}^{m}\int_0^{T_i}f_i(t,x(t))\d t,\\
				&\qquad m\in\N,\,\forall 1\leq i\leq m,\, f_i\in\CC_b([0,\infty)\times E),\, T_i>0\Big\}.
			\end{aligned}
		\end{equation}
		Note that $\CC_m$ is an algebra. Let $M_E[0,\infty)$ be the space of measurable processes from $[0,\infty)$ to $\E$, so $\mathbf{D}\subset M_E[0,\infty)$. Note that $\CC_m$ separates points in $M_E[0,\infty)$. By \cite{K91}[Proposition 4.5], the set $\CC_m$ is separating in the set of measures on $M_E[0,\infty)$. This means that if two stochastic processes $(Z_1(t))_{t>0}$ and $(Z_2(t))_{t\geq 0}$ satisfy
		\begin{equation}
			\label{friet}
			\E[F(Z_1)]=\E[F(Z_2)] \qquad \forall F\in\CC_m,
		\end{equation}
		then $\CL[Z_1]=\CL[Z_2]$.
		
		Define
		\begin{equation}
			F(\psi)=\int \d \psi \prod_{i=1}^{m} \int_0^{T_i} f_i(t,x(t)) \,\d t.
		\end{equation}
		Recall that a pseudopath $\psi$ is associated with a path $w\in M_E[0,\infty)$. Hence this becomes
		\begin{equation}
			\label{19}
			F(\psi_w)= \prod_{i=1}^{m} \int_0^{T_i}f_i(t,w(t))\,\d t.
		\end{equation}
		Since each pseudopath $\psi\in\Psi$ is associated with a path in $M_E[0\infty)$, $\CC_m$ also separates points on $\Psi$ and hence $\CC_m$ separates measures on $\Psi$. This implies that if
		\begin{equation}
			\label{mayo}
			\E[F(\psi_{Z_1})]=\E[F(\psi_{Z_2})]\qquad \forall F\in\CC_m,
		\end{equation}
		then $\CL[\psi_{Z_1}]=\CL[\psi_{Z_2}]$. Therefore $\CL[Z_1]=\CL[Z_2]$ if and only if $\CL[\psi_{Z_1}]=\CL[\psi_{Z_2}]$.  
		
		The Meyer-Zheng topology is a weaker than the Skohorod topology. 
		
		\begin{lemma}
			\label{lem90}
			Let $(Z_n(t))_{t\geq 0}$ $n\in\N$ and $(Z(t))_{t\geq 0}$ be stochastic processes with Polish state-space $E$. If
			\begin{equation}
				\lim_{n\to\infty}\CL\left[(Z_n(t))_{t\geq 0}\right]=\CL\left[(Z(t))_{t\geq 0}\right] \text{ in the Skohorod topology},
			\end{equation} 
			then
			\begin{equation}
				\lim_{n\to\infty}\CL\left[(Z_n(t))_{t\geq 0}\right]=\CL\left[(Z(t))_{t\geq 0}\right] \text{ in the Meyer-Zheng topology}.
			\end{equation}
		\end{lemma}
		
		\begin{proof}
			Since we do not know whether $\Psi$ is compact, the set $\CC_m$ does not have to be convergence determining. Therefore, via Skorohod's theorem we construct the process $\tilde{Z}^n$ and $\tilde{Z}$ on one probability space, such that $\CL[\tilde{Z}^n]=\CL[{Z}^n]$ and $\CL[\tilde{Z}]=\CL[{Z}]$, and 
			\begin{equation}
				\lim_{n\to\infty} \tilde{Z}_n=\tilde{Z}\qquad a.s.
			\end{equation}
			This implies
			\begin{equation}
				\lim_{n\to\infty}\psi_{\tilde{Z}^n}=\psi_{\tilde{Z}}\qquad a.s. 
			\end{equation}
			Consequently, for all $f\in\CC_b(\Psi)$,
			\begin{equation}
				\label{sate}
				\lim_{n\to\infty}\E[f(\psi_{\tilde{Z}^n})]=\E[f(\psi_{\tilde{Z}})].
			\end{equation}
			Note that, since $\CL[\tilde{Z}^n]=\CL[{Z}^n]$ and $\CL[\tilde{Z}]=\CL[{Z}]$, we can use \eqref{friet} and \eqref{mayo} to see that the latter implies $\CL[\psi_{Z^n}]=\CL[\psi_{\tilde{Z}^n}]$ and $\CL[\psi_{Z}]=\CL[\psi_{\tilde{Z}}]$. Hence \eqref{sate} indeed implies that
			\begin{equation}
				\lim_{n\to\infty}\CL[\psi_{Z^n}]=\CL[\psi_Z].
			\end{equation}
		\end{proof}
		
		\paragraph{Convergence in probability in the Meyer-Zheng topology.}
		
		Let $(S,d)$ be a metric space,  $\CB(S)$ denote the Borel-$\sigma$ algebra on $S$, and $\CP(S)$ the set of probability measures on $\CB(S)$. Recall (see e.g.\ \cite[Chapter 3]{EK86}) that the Prohorov metric $d_P$ on the space $\CP(S)$ is given by 
		\begin{equation}
			d_P(\P,\Q)=\inf\left\{\epsilon>0\colon\, \P(A)\leq \Q(A^\epsilon)+\epsilon\,\, \forall\, A\in\CC\right\},
		\end{equation}
		where $\CC\subset\CB(S)$ is the set of all closed sets in $S$ and $A^\epsilon=\{x\in S\colon\, \inf_{y\in A}d(x,y)<\epsilon\}$. 
		
		\begin{theorem}{\bf \cite[Theorem 3.1.2]{EK86}}
			Let $(S,d)$ be separable and let $\P,\Q\in\CP(S)$. Define $\CM(\P,\Q)$ to be the set of all $\mu\in \CP(S\times S)$ with marginals $\P$ and $\Q$, i.e., $\mu(A\times S)=\P(A)$ and $\mu(S\times A)=\Q(A)$ for all $A\in\CB(S)$. Then
			\begin{equation}
				d_P(\P,\Q)=\inf_{\mu\in\CM(\P,\Q)}\inf\{\epsilon>0\colon\,\mu(\{(x,y)\colon\,d(x,y)\geq\epsilon\})\leq \epsilon\}.
			\end{equation} 
		\end{theorem} 
		
		Moreover, \cite[Theorem 3.3.1]{EK86} states that convergence of measures in the Prohorov distance, $\lim_{n\to\infty} d_P(\P_n, \P)=0$, is the same as weak convergence $\P_n\Rightarrow\P$. Hence, since convergence of pseudopaths is weak convergence,  we can endow the space of pseudopaths $\Psi$ with the metric $d_P$.
		
		Let $(\Psi,d_P)$ be the pseudopath space metrized by the Prohorov distance. Let $(Z^n(t))_{t>0}$, $(Z(t))_{t>0}$ be stochastic processes on the state space $E$, where $E$ is endowed with the metric $d(\cdot,\cdot)$. Note that convergence in probability in the Meyer-Zheng topology means that
		\begin{equation}
			\forall\,\delta>0\colon\quad \lim_{n\to\infty}\P\left[d_P\left(\psi_{Z^n},\psi_Z\right)>\delta\right]=0.
		\end{equation}
		
		\paragraph{Tightness.}
		
		Define the \emph{conditional variation} for an $\R$-valued process $(U(t))_{t \geq 0}$ with natural filtration $(\CF(t))_{t \geq 0}$ as follows. For a subdivision $\tau\colon\,0=t_0<t_1< \cdots<t_n=\infty$, set
		\begin{equation}
			\label{e1575}
			V_\tau(U)=\suml_{0 \leq i<n} \E\Big[\big|\E[U(t_{i+1})-U(t_i) \mid F(t_i)]\big|\Big]
		\end{equation}
		(with $U(\infty)=0)$ and
		\begin{equation}
			\label{e1578}
			V(U)=\sup\limits_\tau V_\tau(U).
		\end{equation}
		If $V(U)< \infty$, then $U$ is called a \emph{quasi-martingale}. Note that we can always stop the process at some finite time and work with compact time intervals.
		
		\begin{lemma}{\bf [Tightness in the Meyer-Zheng topology]}
			\label{l.1585}\hfil\\
			If $(P_n)_{n\in\N}$ is a sequence of probability laws on $D([0,T],\R)$ such that under $P_n$ the coordinate process $(U(t))_{t \geq 0}$ is a quasi-martingale with a conditional variation $V_n(U)$ that is bounded uniformly in $n$, then there exists a subsequence $(P_{n_k})_{k\in\N}$ that converges weakly in the Meyer-Zheng topology on $D([0,T],\R)$ to a probability law $P$, and $(U(t))_{t \geq 0}$ is a quasi-martingale under $P$.
		\end{lemma}
		
		\noindent
		(See \cite[Theorem 7]{MZ84} for the identification of the limiting semi-martingale.) 
		
		\subsection{Proof of key lemmas}
		\label{apb3}
		
		\paragraph{$\bullet$ Proof of Lemma~\ref{lem91}.}
		
		\begin{proof}
			Fix $\delta>0$. Then
			\begin{equation}
				\label{hal}
				\begin{aligned}
					&\lim_{n\to\infty}\P\left[d_P(\psi_{Z_n},\psi_Z)>\delta\right]\\
					&=\lim_{n\to\infty}\P\left[\inf_{\mu\in\CM(\psi_{Z_n},\psi_Z)}\inf\{\epsilon>0\colon\, \mu(\{(x,y)\colon\, 
					d(x,y)\geq\epsilon\})\leq \epsilon\}>\delta\right]\\			
					&=\lim_{n\to\infty}\P\left[\,\forall\mu\in\CM(\psi_{Z_n},\psi_Z),\,\inf\{\epsilon>0\colon\, \mu(\{(x,y)\colon\, 
					d(x,y)\geq\epsilon\})\leq \epsilon\}>\delta\right]\\
					&=\lim_{n\to\infty}\P\left[\,\forall\mu\in\CM(\psi_{Z_n},\psi_Z),\, \mu(\{(x,y)\colon\, 
					d(x,y)\geq\delta\})>\delta\right].
				\end{aligned}
			\end{equation}
			Let $\mu_n\in\CP(([0,\infty]\times E)^2)$ be the measure defined by
			\begin{equation}
				\mu_n(A)=\int_0^\infty 1_A\left((t,Z_n(t)),(t,Z(t))\right)\,\e^{-t}\d t, \qquad A\in\CB(([0,\infty]\times E)^2),
			\end{equation}
			such that, for $B\in\CB([0,\infty]\times E)$,
			\begin{equation}
				\mu_n(B\times S)=\int_0^\infty 1_B (t,Z_n(t)) 1_S((t,Z(t))\,\e^{-t}\d t=\psi_{Z^n}(B),
			\end{equation}
			and similarly $\mu_n(S\times B)=\psi_Z(B)$. Hence $\mu_n\in\CM(\psi_{Z_n},\psi_Z)$ for all $n\in\N$, and we obtain from \eqref{hal} that
			\begin{equation}	
				\begin{aligned}
					&\lim_{n\to\infty}\P\left[d_P(\psi_{Z_n},\psi_Z)>\delta\right]\\
					&\leq\lim_{n\to\infty}\P\left[\mu_n(\{(x,y): d(x,y)\geq\delta\})>\delta\right]\\
					&\leq\lim_{n\to\infty}\P\left[\int_0^\infty 1_{\{(x,y): d(x,y)\geq\delta\}}\left((t,Z_n(t)),(t,Z(t))\right)\e^{-t}\d t>\delta\right]\\
					&\leq\lim_{n\to\infty}\P\left[\int_0^\infty 1_{\{ d(Z_n(t),Z(t))\geq\delta\}}\,\e^{-t}\d t>\delta\right]\\
					&\leq\lim_{n\to\infty}\frac{1}{\delta}\E\left[\int_0^\infty d(Z_n(t),Z(t))\,\e^{-t}\d t\right]\\
					&=\lim_{n\to\infty}\frac{1}{\delta}\int_0^\infty\E\left[ d(Z_n(t),Z(t))\right]\,\e^{-t}\d t =0.
				\end{aligned}
			\end{equation}
		\end{proof}
		
		\paragraph{$\bullet$ Proof of Lemma~\ref{lem93}.}
		
		\begin{proof}
			We have to show that
			\begin{equation}
				\lim_{n\to \infty}\CL\left[\psi_{(X_n,Y_n)}\right]=\CL\left[\psi_{(X,c)}\right].
			\end{equation}
			Hence we must show that, for all $f\in\CC_b(\Psi)$,
			\begin{equation}
				\lim_{n\to \infty}\E[f(\psi_{(X_n,Y_n)})]=\E[f(\psi_{(X,c)})].
			\end{equation}
			We can write
			\begin{equation}
				\label{871}
				\begin{aligned}
					|\E[f(\psi_{(X_n,Y_n)})-f(\psi_{(X,c)})]|\leq |\E[f(\psi_{(X_n,Y_n)})-f(\psi_{(X_n,c)})]| +|\E[f(\psi_{(X_n,c)})-f(\psi_{(X,c)})]|.
				\end{aligned}
			\end{equation}
			Since $\lim_{n\to \infty}\E[d(Y_n(t),c)]=0$ implies $\lim_{n\to \infty}\E[d((X_n(t),Y_n(t)),(X_n(t),c))]=0$, it follows from Lemma~\ref{lem91} that, for all $\delta>0$,
			\begin{equation}
				\lim_{n\to \infty}\P\left[d_P\left(\psi_{(X_n,Y_n)},\psi_{(X_n,c)}\right)\right]=0.
			\end{equation}
			Hence, for all $f\in\CC_b(\Psi)$,
			\begin{equation}
				\lim_{n\to\infty}|\E[f(\psi_{(X_n,Y_n)})-f(\psi_{(X_n,c)})]|=0. 
			\end{equation}
			To see that the second term in the right-hand side of \eqref{871} tends to zero, note that we can define
			\begin{equation}
				\tilde{f}(\psi_{x})=f(\psi_{x,c}).
			\end{equation}
			We show that $\tilde{f}$ is continuous. 
			
			Recall that convergence in the Meyer-Zheng topology is simply convergence in Lebesgue measure. Hence, for two paths $(t,x_n(t))$ and $(t,x(t))\in M_E[0\infty)$ we have $\psi_{x_n}\to\psi_x$ if and only if, for all $\delta>0$,
			\begin{equation}
				\lim_{n\to\infty}\int_0^\infty1_{\{d(x_n(t),x(t))>\delta\}}\,\e^{-t}\d t=0.
			\end{equation}
			Therefore $\psi_{x_n}\to\psi_{x}$ implies that, for all $\delta>0$, 
			\begin{equation}
				\lim_{n\to\infty}\int_0^\infty1_{\{d((x_n(t),c),(x(t),c))>\delta\}}\,\e^{-t}\d t=0,
			\end{equation}
			and hence $\psi_{x_n,c}\to\psi_{x,c}$. Therefore
			\begin{equation}
				\lim_{n\to\infty}\tilde{f}(\psi_{x_n})=\lim_{n\to\infty}f(\psi_{(x_n,c)})=f(\psi_{(x,c)})=\tilde{f}(\psi_x)
			\end{equation}
			and we conclude that $f\in\CC_b(\Psi)$. Since $\CL[X_n]=\CL[X]$ in the Meyer-Zheng topology, we have, for all $f\in\CC_b(\Psi)$,
			\begin{equation}
				\begin{aligned}
					\lim_{n\to \infty}|\E[f(\psi_{(X_n,c)})-f(\psi_{(X,c)})]|=\lim_{n\to \infty}|\E[\tilde{f}(\psi_{(X_n)})-\tilde{f}(\psi_{(X)})]|=0.
				\end{aligned}
			\end{equation}
			Therefore also the second term on the right-hand side of \eqref{871} tends to $0$.
		\end{proof}
		
		\paragraph{$\bullet$ Proof of Lemma~\ref{lem94}.}
		
		\begin{proof}
			For part (a), suppose that $\lim_{n\to\infty}\psi_{x_n}=\psi_x$. Then, since convergence in pseudopath space is convergence in measure, we have, for all $\delta>0$,
			\begin{equation}
				\lim_{n\to\infty}\int_0^\infty1_{\{d(x_n(t),x(t))>\delta\}}\,\e^{-t}\d t=0.
			\end{equation}
			Since $f$ is a continuous function, this implies that, for all $\epsilon>0$,
			\begin{equation}
				\lim_{n\to\infty}\int_0^\infty1_{\{d(f(x_n(t)),f(x(t)))>\epsilon\}}\,\e^{-t}\d t=0.
			\end{equation}
			and we conclude that $\lim_{n\to \infty}\psi_{f(x_n)}=\psi_{f(x)}$. Hence $h$ is indeed continuous. 
			
			For part (b), recall that
			\begin{equation}
				\lim_{n\to \infty}\CL[X_n]=\CL[X]\text{ in the Meyer-Zheng topology}
			\end{equation}
			implies that, for all $g\in\CC_b(\Psi)$,
			\begin{equation}
				\lim_{n\to \infty}\E[g(\psi_{X_n})]=\E[g(\psi_X)].
			\end{equation}
			Since $h\colon\,\Psi\to\Psi$ is continuous, we have for all $g\in\CC_b(\Psi)$ that $g\circ h\in \CC_b(\Psi)$. Hence
			\begin{equation}
				\lim_{n\to \infty}\E[g(\psi_{f(X_n)}]=\lim_{n\to \infty}\E[g\circ h(\psi_{X_n})]=\E[g\circ h(\psi_X)]=\E[g (\psi_{f(X)})].
			\end{equation}
			We conclude that
			\begin{equation}
				\lim_{n\to \infty}\CL[f(X_n)]=\CL[f(X)]\text{ in the Meyer-Zheng topology.}
			\end{equation} 
		\end{proof}
		
		\paragraph{$\bullet$ Proof of Lemma~\ref{lem95}.}
		
		Supposethat  $\lim_{n\to \infty}\psi_{(x_n,y_n)}=\psi_{(x,y)}$. Since convergence of pseudopaths is convergence in Lebesgue measure, we have 
		\begin{equation}
			\lim_{n\to \infty}\int_0^\infty 1_{\{d[(x_n,y_n),(x,y)]>\delta\}}\e^{-t}\d t=0
		\end{equation}
		and, consequently, 
		\begin{equation}
			\lim_{n\to \infty}\int_0^\infty 1_{\{d[x_n,x]>\delta\}}\e^{-t}\d t=0.
		\end{equation}
		Therefore $\lim_{n\to \infty}\psi_{x_n}=\psi_x$. Suppose that $f\in\CC_b(\Psi(E))$, so $f$ is bounded continuous function on the space of pseudopaths on $[0,\infty]\times E$. Define the function $\tilde{f}$ on the space of pseudopaths on $[0,\infty]\times E^2$, i.e., $\tilde{f}$ is a function on $\Psi(E^2)$, by 
		\begin{equation}
			\label{b46}
			\tilde{f}(\psi_{(x,y)})=f(\psi_x).
		\end{equation}
		Then $\tilde{f}\in \CC_b(\Psi(E^2))$ and  
		\begin{equation}
			\lim_{n\to\infty}\tilde{f}(\psi_{(x_n,y_n)})=\lim_{n\to \infty}f(\psi_{x_n})=f(\psi_x)=\tilde{f}(\psi_{(x_n,y_n)}).
		\end{equation}
		Hence $\tilde{f}$ is indeed a continuous function on $\Psi(E^2)$. Moreover, since $f$ is bounded, it follows that $\tilde{f}$ is bounded and we conclude that $\tilde{f}\in\CC_b(\Psi(E^2))$.
		
		Therefore, if $X_n,Y_n$ are continuous-time stochastic processes on $E$ and
		\begin{equation}
			\lim_{n\to \infty}\CL\left[(X_n(s),Y_n(s))_{s>0}\right]=\CL\left[(X(s),Y(s))_{s>0}\right]\, \text{ in Meyer Zheng topology},
		\end{equation}
		then for all $f\in\CC_b(\Psi(E^2))$ we have 
		\begin{equation}
			\lim_{n\to \infty}\E[f(\psi_{(X_n,Y_n)})]=\E[f(\psi_{(X,Y)})].
		\end{equation}
		Since for each $f\in\CC_b(\Psi(E))$ we can construct a function $\tilde{f}\in \CC_b(\Psi(E^2))$ as in \eqref{b46}, we obtain for all $f\in\CC_b(\Psi(E))$ that
		\begin{equation}
			\begin{aligned}
				\lim_{n\to \infty}\E[f(\psi_{(X_n)})]=\lim_{n\to \infty}\E[\tilde{f}(\psi_{(X_n,Y_n)})]
				=\E[\tilde{f}(\psi_{(X,Y)})]=\E[\tilde{f}(\psi_{X})].
			\end{aligned}
		\end{equation} 
		We conclude that
		\begin{equation}
			\lim_{n\to \infty}\CL[(X_n(s))]=\CL[(X(s))_{s>0}]\,\text{ in Meyer-Zheng topology}
		\end{equation}
		and, similarly, 
		\begin{equation}
			\lim_{n\to \infty}\CL[(Y_n(s))]=\CL[(Y(s))_{s>0}]\,\text{ in Meyer-Zheng topology.}
		\end{equation}
		
		\bibliography{seedbank}
		\bibliographystyle{alpha}
		
		
	\end{document}